%% file: thesis.tex
\documentclass[11pt,a4paper,twoside]{amsbook}
\usepackage[top=1in, bottom=1.25in, marginparwidth=1in, inner=1.25in, outer=1in]{geometry}
\usepackage{datetime}
\usepackage[T1]{fontenc}
\usepackage[utf8]{inputenc}
\usepackage{amsmath}
\usepackage[UKenglish]{babel}
\usepackage{fancyhdr}
\usepackage{amsfonts}
\usepackage{amssymb}
\usepackage{enumerate}
\usepackage{units}

\usepackage{mathtools}
\mathtoolsset{mathic=true}

\usepackage{xcolor}
\colorlet{myblue}{black}
\colorlet{lightred}{red!30!white}
\colorlet{lightblue}{blue!30!white}
\colorlet{darkgreen}{green!70!black}
\colorlet{gold}{yellow!70!black!90!red}
\newcommand{\darkgreen}[1]{\textcolor{darkgreen}{#1}}
\newcommand{\gold}[1]{\textcolor{gold}{#1}}

\usepackage[section]{placeins}
\usepackage{rotating}

\usepackage{pstricks,pstricks-add,pst-pdf}
\usepackage{pst-node}
\usepackage{pst-func}
\usepackage{pst-3d}
\usepackage{pst-3dplot}
\usepackage{stackengine} 
\setstackEOL{\\}

\usepackage{cancel}
\ifpdf
\usepackage[bookmarks=true,pdfencoding=auto]{hyperref}
\usepackage{bookmark}
\hypersetup{colorlinks=true,citecolor=black,urlcolor=myblue,linkcolor=black}

\usepackage{tikz}
\usepackage{tikz-cd}
\else
\fi

\usepackage{graphicx}
\usepackage{caption}
\usepackage[labelfont=up,margin=5pt]{subcaption}
\usepackage{fancybox}
\usepackage{amsthm}
\usepackage{marginnote}

\usepackage{multirow}

\makeatletter
\def\citeform#1{{#1}}
\def\@cite#1#2{{\@citestyle[\citeform{#1}\if@tempswa, #2\fi]}}
\makeatother

\newtheorem{theorem}{Theorem}[section]
\newtheorem*{theorem*}{Theorem}
\newtheorem{lemma}[theorem]{Lemma}
\newtheorem{fact}[theorem]{Fact}
\newtheorem{conjecture}[theorem]{Conjecture}
\newtheorem{question}[theorem]{Question}
\newtheorem{questions}[theorem]{Questions}

\newtheorem{corollary}[theorem]{Corollary}
\newtheorem{proposition}[theorem]{Proposition}
\theoremstyle{definition}
\newtheorem{definition}[theorem]{Definition}
\newtheorem{example}[theorem]{Example}
\newtheorem{observation}[theorem]{Observation}
\newtheorem{Remark}[theorem]{Remark}

\usepackage{chngcntr}
\counterwithout{figure}{chapter}
\counterwithout{table}{chapter}

\usepackage{smartref}

\addtoreflist{chapter}

\newcommand*{\srefaux}[1]{%
    \ischapterchanged{#1}
    \ifchapterchanged
        \chapterref{#1}.
    \fi
    \ref*{#1}
}

\makeatletter
\renewcommand \theequation {\@arabic\c@equation}
\renewcommand \thetheorem {\@arabic\c@theorem}
\makeatother

\newcommand*\sref[1]{\hyperref[#1]{\srefaux{#1}}}
\newcommand*\seqref[1]{(\hyperref[#1]{\srefaux{#1}})}

\usepackage{setspace} 
\newcommand{\qp}[2]{({\blue q_{#1}}\vert\textcolor{violet}{p_{#2}})}
\newcommand{\pq}[2]{({\blue p_{#1}}\vert\textcolor{violet}{q_{#2}})}

\newcommand{\pn}[1]{({\blue p_{#1}}\vert\textcolor{violet}{-})}
\newcommand{\qn}[1]{({\blue q_{#1}}\vert\textcolor{violet}{-})}
\newcommand{\np}[1]{({\blue -}\vert\textcolor{violet}{p_{#1}})}
\newcommand{\nq}[1]{({\blue -}\vert\textcolor{violet}{q_{#1}})}
\newcommand{\nn}{({\blue -}\vert\textcolor{violet}{-})}

\DeclareMathOperator{\Mor}{Mor}

\DeclareMathOperator{\lk}{lk} 
\DeclareMathOperator{\m}{m} 
\DeclareMathOperator{\sgn}{sgn} 
\DeclareMathOperator{\Spinc}{Spin^\mathit{c}} 
\DeclareMathOperator{\rr}{r} 
\DeclareMathOperator{\id}{id} 
\DeclareMathOperator{\obj}{obj} 
\DeclareMathOperator{\pmr}{\pm}

\DeclareMathOperator{\pmt}{\textcolor{blue}{\pm}}
\DeclareMathOperator{\mpt}{\textcolor{blue}{\mp}}
\DeclareMathOperator{\CFTd}{CFT^\partial}
\DeclareMathOperator{\CFT}{\widehat{CFT}}
\DeclareMathOperator{\HFT}{\widehat{HFT}}
\DeclareMathOperator{\HF}{\widehat{HF}}
\DeclareMathOperator{\SFH}{SFH}
\DeclareMathOperator{\SFC}{SFC}
\DeclareMathOperator{\HFK}{\widehat{HFK}}
\DeclareMathOperator{\HFL}{\widehat{HFL}}
\DeclareMathOperator{\CFL}{\widehat{CFL}}
\DeclareMathOperator{\CFA}{\widehat{CFA}}
\DeclareMathOperator{\CFD}{\widehat{CFD}}
\DeclareMathOperator{\BSD}{\widehat{BSD}}
\DeclareMathOperator{\BSA}{\widehat{BSA}}
\DeclareMathOperator{\Fuk}{Fuk}
\DeclareMathOperator{\TwFuk}{TwFuk}
\DeclareMathOperator{\pqMod}{pqMod}
\DeclareMathOperator{\Com}{\mathfrak{Com}}
\DeclareMathOperator{\im}{im}
\DeclareMathOperator{\rk}{rk}

\newcommand{\Ad}{\operatorname{\mathcal{A}}^\partial}
\newcommand{\Adop}{\operatorname{\mathcal{A}}^\partial_\text{op}}
\newcommand{\Id}{\operatorname{\mathcal{I}}^\partial}
\newcommand{\ob}{\operatorname{ob}}
\newcommand{\Red}[1]{{\textcolor{red}{#1}}}

\newcommand{\Z}{\operatorname{\mathcal{Z}}}
\newcommand{\A}{\boldsymbol{\alpha}}
\newcommand{\B}{\boldsymbol{\beta}}
\newcommand{\GG}{\boldsymbol{\gamma}}
\newcommand{\D}{\boldsymbol{\delta}}
\newcommand{\Ac}{\boldsymbol{\alpha}^c}
\newcommand{\Aa}{\operatorname{\boldsymbol{\alpha}}^a}
\newcommand{\aaa}{\operatorname{\alpha}^a}
\newcommand{\x}{\boldsymbol{x}}
\newcommand{\y}{\boldsymbol{y}}
\newcommand{\z}{\boldsymbol{z}}
\newcommand{\Tw}{\operatorname{Tw}\!}
\newcommand{\mutw}{\operatorname{\mu}^\textit{tw}}

\makeatletter
\def\Mydot{\pst@object{Mydot}}
\def\Mydot@i(#1,#2,#3,#4,#5){%
  \begin@ClosedObj%
	\rput(#2,#3){
	\psdot[linecolor=#4, dotsize=5pt](0,0)
	\uput{0.3}[#1](0,0){\footnotesize #5}
	}
  \end@ClosedObj%
}
\makeatother
\makeatletter
\def\wu{\pst@object{wu}}
\def\wu@i#1(#2,#3){%
  \begin@ClosedObj%
	\rput(#2,#3){
	\psline{*-*}(0,0)(1,1)
	\rput[c](0.4,0.5){
	\psframebox[linecolor=white,fillcolor=white, fillstyle=solid,framearc=.5,framesep=0]{\scriptsize #1}}
	}
  \end@ClosedObj%
}
\makeatother
\makeatletter
\def\Wu{\pst@object{Wu}}
\def\Wu@i(#1,#2){%
  \begin@ClosedObj%
	\rput(#1,#2){
	\psline{*-*}(0,0)(1,1)
	}
  \end@ClosedObj%
}
\makeatother
\makeatletter
\def\wv{\pst@object{wv}}
\def\wv@i#1(#2,#3){%
  \begin@ClosedObj%
	\rput(#2,#3){
	\psline{*-*}(0,0)(1,-1)
	\rput[c](0.4,-0.5){
	\psframebox[linecolor=white,fillcolor=white, fillstyle=solid,framearc=.5,framesep=0]{\scriptsize #1}}
	}
  \end@ClosedObj%
}
\makeatother
\makeatletter
\def\Wv{\pst@object{Wv}}
\def\Wv@i(#1,#2){%
  \begin@ClosedObj%
	\rput(#1,#2){
	\psline{*-*}(0,0)(1,-1)
	}
  \end@ClosedObj%
}
\makeatother
\makeatletter
\def\wi{\pst@object{wi}}
\def\wi@i#1(#2,#3){%
  \begin@ClosedObj%
	\rput(#2,#3){
	\psline{*-*}(0,0)(1,0)
	\rput(0.36,0){
	\psframebox[linecolor=white,fillcolor=white, fillstyle=solid,framearc=.5,framesep=0]{\scriptsize #1}}
	}
  \end@ClosedObj%
}
\makeatother
\makeatletter
\def\Wi{\pst@object{Wi}}
\def\Wi@i(#1,#2){%
  \begin@ClosedObj%
	\rput(#1,#2){
	\psline{*-*}(0,0)(1,0)
	}
  \end@ClosedObj%
}
\makeatother

\renewcommand{\thetheorem}{\arabic{section}.\arabic{theorem}}

\newcommand{\nmathphantom}[1]{\settowidth{\dimen0}{$#1$}\hspace*{-\dimen0}}

\begin{document}
\psset{arrowsize=3pt 2}
\mainmatter

\pdfbookmark[1]{Titlepage}{Titlepage}
\thispagestyle{empty}
\topskip0pt
\vspace*{\fill}
\begin{center}
	\begin{LARGE}
		\textbf{On a Heegaard Floer theory for tangles}\medskip\\
	\end{LARGE}	
	\begin{large}
		\textsc{Claudius Bodo Zibrowius}\medskip\\
		\textsc{PhD thesis}\smallskip\\
	\end{large}	
	\vspace{11pt}
	\includegraphics[width=50pt]{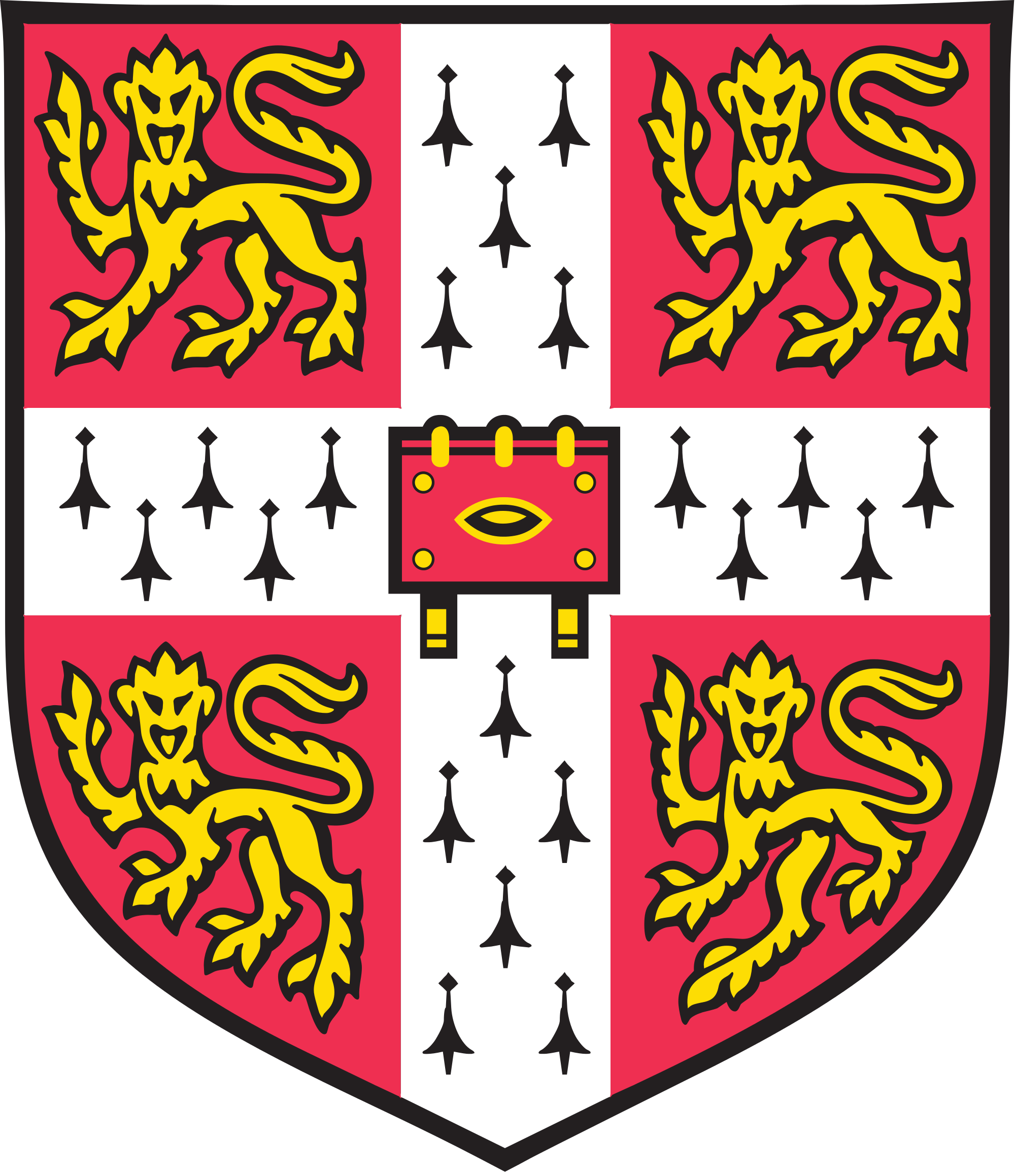}\\
	\vspace{11pt}
\end{center}

\begin{quotation}\small
	\textsc{Abstract.} The purpose of this thesis is to define a ``local'' version of Ozsváth and Szabó's Heegaard Floer homology $\HFL$ for links in the 3-sphere, i.\,e.\ a Heegaard Floer homology $\HFT$ for tangles in the 3-ball. \\
	The decategorification of $\HFL$ is the classical Alexander polynomial for links; likewise, the decategorification of $\HFT$ gives a local version $\nabla_T^s$ of the Alexander polynomial. In the first chapter of this thesis, we give a purely combinatorial definition of this polynomial invariant $\nabla_T^s$ via Kauffman states and Alexander codes and investigate some of its properties. As an application, we \linebreak[3] show that the multivariate Alexander polynomial is mutation invariant.\\
	In the second chapter, we define $\HFT$ in two slightly  different, but equivalent ways: One is via Juhász's sutured Floer homology, the other by imitating the construction of $\HFL$. We then state a glueing theorem in terms of Zarev's bordered sutured Floer homology, which endows $\HFT$ with additional  structure. As an application, we show that any two links related by mutation about a $(2,-3)$-pretzel tangle have the same $\delta$-graded link Floer homology. This result relies on a computer calculation.\\
	In the third and last chapter, we specialise to 4-ended tangles. In this case, we give a reformulation of $\HFT$ with a glueing structure in terms of (what we call) peculiar modules. Together with a glueing theorem, we can easily recover oriented and unoriented skein relations for $\HFL$. 
	Our peculiar modules also enjoy some symmetry relations, which support a conjecture about $\delta$-graded mutation invariance of $\HFL$. However, stronger symmetries would be needed to actually prove this conjecture.
	Finally, we explore the relationship between peculiar modules and twisted complexes in the wrapped Fukaya category of the 4-punctured sphere.\\
	There are four appendices, some of which might be of independent interest: In the first appendix, we describe a general construction of dg categories which unifies all algebraic structures used in this thesis, in particular type~A and type D modules from bordered theory. In the second appendix, we prove a generalised version of Kauffman's clock theorem, which plays a major role for our decategorified invariants. The last two appendices are manuals for two Mathematica programs. The first is a tool for computing the generators of $\HFT$ and the decategorified tangle invariant $\nabla_T^s$. The second allows us to compute bordered sutured Floer homology using nice diagrams. 
\end{quotation}
\vspace*{\fill}

\renewcommand{\contentsname}{Table of Contents}

\pdfbookmark{Table of Contents}{Table of Contents}
\tableofcontents

\input{sections/0_Intro}

\chapter{Alexander polynomials for tangles}\label{chapter:polynomial}

In the first chapter, we define and study the polynomial invariant $\nabla_T^s$, a generalisation of the Alexander polynomial of knots and links  to tangles. The Alexander polynomial is a classical knot and link invariant which takes the form of a Laurent polynomial in the same number of variables as there are link components \cite{Alexander}. We will use its normalised form, which is also known as the Conway potential function and denoted by~$\nabla_L$, see for example Hartley's monograph \cite{hartley}. This polynomial invariant can be defined and interpreted in many different ways, depending on one's preferred point of view. For example, this might be Fox calculus, elementary ideals, Reidemeister torsion or skein theory.\\
We start from Kauffman's combinatorial definition of the Conway potential function for knots and links. In section \ref{sec:basicdefinitions}, we adapt this definition to tangles. In general, this gives us a finite set of Laurent polynomials $\nabla_T^s$ associated with an oriented tangle diagram~$T$. This finite set of invariants is indexed by some additional input data for tangles, which we call the sites of $T$ (see definition \ref{def:basic}). In theorems \ref{thm:nablaisaninvariant} and \ref{thm:twoended}, we show the following basic result.
\begin{theorem*}
For each site \(s\) of an oriented tangle \(T\), \(\nabla_T^s\) is an invariant of \(T\). Furthermore, if \(T\) represents a link or a knot \(L\), there is exactly one site \(s\), namely \(s=\emptyset\), and \(\nabla^\emptyset_T\) is equal to the Conway potential function \(\nabla_L\) up to a certain factor.
\end{theorem*}
The definition of $\nabla_T^s$ is a very straightforward generalisation of Kauffman's construction. However, apart from a short discussion in \cite{ouka}, I am unaware of any reference in the literature where this invariant has been studied. So this will be object of the remaining part of this chapter. \\
We show that the invariants $\nabla_T^s$ satisfy a glueing formula (proposition~\ref{prop:glueing}) which generalises the connected sum formula for the knot and link case. In section \ref{sec:basicpropertiesofnabla}, we derive some further properties of our tangle invariants. As we will see, many properties of the Conway potential function for knots and links generalise. In particular, they are well-behaved under orientation reversal of the tangle strands and taking mirror images. In section \ref{sec:4endedandmutation}, we study $\nabla_T^s$ for 4-ended tangles and prove symmetry relations between different sites $s$. They imply that $\nabla^s_T$ and in particular also the \textit{multivariate} Alexander polynomial of links is mutation invariant, for a precise statement, see theorem~\sref{thm:mutation}. In section \ref{sec:geometricinterpretation}, we interpret the invariants geometrically in terms of the first homology of the maximal Abelian cover of the tangle complement relative to certain subspaces of its boundary. Because this geometric generalisation of the Alexander polynomial looks even more natural than the one using Kauffman states (but with the slight drawback of being unnormalised), this seems to be a good starting point for any comparison with other constructions of Alexander polynomials for tangles, of which there are many 
\cite{%
Archibald,
Polyak,
Bigelow,
Kennedy,
Florens,
Sartori14,
Damiani}.

\input{sections/1_Definitions}
\input{sections/1_BasicProps}
\input{sections/1_FourEndedTangles}
\input{sections/1_Geometric}

\chapter{A Heegaard Floer homology for tangles}\label{chapter:categorification}

In this chapter, we categorify the polynomial invariants $\nabla_T^s$ from the first chapter. 
We start in section~\ref{sec:HFTviaSFH} by defining a tangle Floer homology~$\HFT$ in terms of Juhász's sutured Floer homology $\SFH$ \cite{Juhasz}. 
Using our geometric interpretation of the polynomial tangle invariants~$\nabla_T^s$ from section~\ref{sec:geometricinterpretation} and a description of the decategorification of $\SFH$ due to Friedl, Juhász and Rasmussen \cite{DecatSFH}, we show that $\HFT$ categorifies~$\nabla_T^s$. 
In section~\ref{sec:HDsForTangles}, we give an independent, but equivalent construction of~$\HFT$ from Heegaard diagrams for tangles. The main advantage over the first definition via sutured Floer homology is that we naturally get relative gradings for all sites simultaneously. For the sutured approach, this would require a means of comparing $\Spinc$-structures for different sets of sutures, see remark~\ref{rem:absolutegradings}.\\
In order to obtain a glueing theorem for~$\HFT$ that categorifies the glueing formula for~$\nabla_T^s$ (proposition~\sref{prop:glueing}), we need to add some extra structure to~$\HFT$. 
This is done in section~\ref{sec:TFHviaBSFH}, using Zarev's bordered sutured Floer theory. Note that in the first two sections, we are working over the coefficient ring $\mathbb{Z}$, whereas in the third, we restrict to $\mathbb{Z}/2$-coefficients.

\input{sections/2_HFTviaSFH}
\input{sections/2_HDsForTangles}
\input{sections/2_GlueingViaBSFH}

\chapter{Peculiar invariants for 4-ended tangles}\label{chapter:HFTd}
In this chapter, we specialise to 4-ended tangles and reformulate the glueing structure from section~\ref{sec:TFHviaBSFH} in terms of algebraic invariants which we call peculiar modules. In section~\ref{sec:Pairing}, we prove a glueing formula for these invariants, based on a computer calculation of a bordered sutured type~AA bimodule, and give several applications thereof: For example, in section~\ref{sec:SkeinRelations}, we give new proofs of the oriented and the unoriented skein exact sequences for link Floer homology, due to Ozsváth and Szabó \cite{OSHFK} and Manolescu \cite{Manolescu}. We also compute the peculiar module for the $(2,-3)$-pretzel tangle. It has certain symmetries which, together with the glueing theorem, give a second proof of theorem~\sref{thm:2m3pt}. In general, we only obtain slightly weaker symmetry relations, see section~\ref{sec:SymRels}. Finally, in section~\ref{sec:LoopsAreLoops}, we explore the relationship between peculiar modules and the wrapped Fukaya category of the 4-punctured sphere. 

\input{sections/3_HFTd}
\input{sections/3_Pairing}

\input{sections/3_SkeinRelations}
\input{sections/3_MoreRelations}
\input{sections/3_LoopsAreLoops}

\setcounter{theorem}{0}
\renewcommand\thetheorem{\Alph{chapter}.\arabic{theorem}}

\appendix
\input{sections/A_AlgebraicStructures.tex}
\setcounter{theorem}{0}
\input{sections/A_GCTproof.tex}
\setcounter{theorem}{0}
\input{sections/A_APT-Manual.tex}
\setcounter{theorem}{0}
\input{sections/A_BSFH-Manual.tex}

\backmatter
\pagebreak

\end{document}

%% file: sections/0_Intro.tex
\chapter*{Introduction}\label{sec:intro}
Let $L$ be a link in the 3-sphere~$S^3$. Consider a closed 3-ball~$B^3\subset S^3$ whose boundary intersects~$L$ transversely. Then $L\cap B^3$ is essentially what we call a tangle, the main protagonist of this thesis. We define a tangle invariant $\HFT$, a Heegaard Floer homology for tangles, and study its properties.\\
Heegaard Floer homology theories were first defined by Ozsváth and Szabó in~2001 \cite{OSHF3mfds}. With an oriented, closed 3-dimensional manifold~$M$, they associated a family of homological invariants, the simplest of which is denoted by~$\widehat{\text{HF}}(M)$. Given an oriented knot in~$M$, Ozsváth and Szabó, and independently Rasmussen, then defined filtrations on the chain complexes which give rise to the respective flavours of knot Floer homology \cite{OSHFK,Jake}. This was later generalised to oriented links in~$S^3$~\cite{OSHFL}. Our tangle Floer homology should be understood as a generalisation of the hat version of link Floer homology $\HFL$ to oriented tangles.
\subsection*{$\HFT$ and an Alexander polynomial for tangles. } Like $\HFL$, our tangle Floer homology is a finitely generated Abelian group which comes with two gradings: a relative homological $\mathbb{Z}$-grading and an Alexander grading, which is an additional relative $\mathbb{Z}$-grading for each component of the tangle. However, unlike $\HFL$, our tangle Floer homology depends on some extra data, a site, associated with a tangle, see definition~\sref{def:site}. For a tangle $T$ with $n$ open strands, there are $\binom{2n}{n-1}$ such sites $s$, and for each of them, we define a bigraded chain complex
$$\CFT(T,s)=\bigoplus_{\substack{h\in\mathbb{Z}~\leftarrow \text{homological grading}\hspace{-2.96cm}\\ a\in\mathbb{Z}^{\vert T\vert}~\leftarrow \text{Alexander  grading}\hspace{-2.7cm}}}\CFT_h(T,s,a),$$
where $\vert T\vert$ denotes the number of components of $T$. 
\begin{theorem}[\sref{thm:HFTiswelldefandinvariant}]
Given a tangle \(T\) and a site \(s\) for \(T\), the bigraded chain homotopy type of \(\CFT(T,s)\) is an invariant of \(T\). We denote its homology by \(\HFT(T,s)\) and call it the \textbf{tangle Floer homology} of \(T\).
\end{theorem}
In link Floer homology, the ``Alexander'' in ``Alexander grading'' comes from the fact that, given a link $L$ in $S^3$, the graded Euler characteristic of $\HFL(L)$ recovers the Alexander polynomial of $L$, a classical polynomial link invariant, named after its discoverer~\cite{Alexander}. We say link Floer homology \textit{categorifies} the Alexander polynomial. Similarly, for tangles, we obtain polynomial invariants 
$$
\chi(\HFT(T,s))=\sum_{h, a} (-1)^h\operatorname{rk}(\HFT_h(T,s,a))\cdot t_{1\phantom{\vert}}^{a_{1\phantom{\vert}}}\!\!\cdots t_{\vert T\vert}^{a_{\vert T\vert}}\in\mathbb{Z}[t_1^{\pm1},\dots,t_{\vert  T\vert}^{\pm1}],
$$
which are well-defined up to multiplication by a unit. In chapter~\ref{chapter:polynomial}, we give a purely combinatorial definition of a normalised version $\nabla_T^s$ of these polynomial invariants in terms of Kauffman states and Alexander codes and study their properties. 
\subsection*{Mutation. }
A new invariant can already be interesting because of its simplicity or its aesthetic appeal. But its true value should be determined by its capacity to answer questions about existing theory. So the primary purpose of any tangle Floer homology should be to learn more about ``the local nature'' of knot and link Floer homology and, ultimately, of the geometric objects themselves. A prime example of an open question one might hope to address is how link Floer homology behaves under mutation. 
\begin{definition}\label{def:INTROmutation}
Let $L$ be a link. Construct a new link $L'$ by cutting out a 4-ended tangle $\Gamma$ and glueing it back in after a half-rotation, as illustrated below:
$$
\psset{unit=0.25}
\begin{pspicture}(-8.01,-3.01)(8.01,3.01)
\rput(-5,0){
\pscircle[linestyle=dotted](0,0){3}
\psline(3;45)(3;-135)
\psline(3;-45)(3;135)
\pscircle[fillcolor=white,fillstyle=solid](0,0){1.5}
\rput(0,0){$\Gamma$}
}
\rput(0,0){$\longrightarrow$}
\rput(5,0){
\pscircle[linestyle=dotted](0,0){3}
\psline(3;45)(3;-135)
\psline(3;-45)(3;135)
\pscircle[fillcolor=white,fillstyle=solid](0,0){1.5}
\psrotate(0,0){180}{\rput(0,0){$\Gamma$}}
}
\end{pspicture}
$$
We say $L'$ is obtained from $L$ by \textbf{Conway mutation}. We call~$\Gamma$ the \textbf{mutating tangle} and $L'$ a \textbf{mutant} of~$L$. If~$L$ is oriented, we define an orientation on $L'$ such that it agrees with the one on~$L$ outside~$\Gamma$. If this means that we need to reverse the orientation of the two open components of~$\Gamma$, we also reverse the orientation of any closed components of~$\Gamma$; otherwise, we do not change the orientation on~$\Gamma$.
\end{definition}
We know from~\cite{OSmutation} that knot and link Floer homology is, in general, not invariant under mutation. However, we have the following conjecture from \cite[conjecture~1.5]{BaldwinLevine}.
\begin{conjecture}\label{conj:MutInvHFL}
Let \(L\) be a link and let \(L'\) be obtained from \(L\) by Conway mutation. Then \(\HFL(L)\) and \(\HFL(L')\) agree after collapsing the bigrading to a single \(\mathbb{Z}\)-grading, known as the \(\delta\)-grading. In short: \(\delta\)-graded link Floer homology is mutation invariant.
\end{conjecture}
Let us see what the new invariants tell us on the decategorified level, i.\,e.\ on the level of the Alexander polynomial. 
\begin{theorem}[\sref{thm:mutation}]\label{thm:IntroMutationDECAT}
The multivariate Alexander polynomial is invariant under Conway mutation after identifying the variables corresponding to the two open strands of the mutating tangle.
\end{theorem}
This result has long been known for the single-variate Alexander polynomial, see for example~\cite[proposition~11]{LickorishMillett}, but I have been unable to find a result for the multivariate polynomial in the literature. The proof of theorem~\ref{thm:IntroMutationDECAT} relies on certain symmetry relations between the invariants $\nabla_T^s$ for 4-ended tangles $T$ and varying sites $s$, see proposition~\sref{prop:fourended}. 
For $\HFT$, we can prove similar symmetry relations which categorify those for $\nabla_T^s$. As conjecture~\ref{conj:MutInvHFL} suggests, in general, they only hold for $\delta$-graded tangle Floer homology, see proposition~\sref{prop:fourendedHFT} and example~\sref{exa:pretzeltangle}. However, these symmetry relations are not sufficient to prove the conjecture. This is because, unlike $\nabla_T^s$, $\HFT$ \emph{alone} is insufficient to state a glueing formula.  
\subsection*{Glueing tangle Floer homologies. }
The main tool for glueing 3-manifolds with boundary in Heegaard Floer homology is bordered Heegaard Floer homology, developed by Lipshitz, Ozsváth and Thurston \cite{LOT}. In \cite{Zarev}, Zarev generalised this theory to sutured manifolds. We interpret our tangle Floer homology $\HFT$ in terms of Zarev's theory to add a glueing structure to $\HFT$. This glueing structure essentially takes the form of extra differentials between the tangle Floer chain complexes $\CFT(T,s)$ for different sites $s$. As an accompaniment to this thesis, we provide the Mathematica package~\cite{BSFH.m} which allows us to compute the bordered sutured invariants for any bordered sutured manifold from nice diagrams, see appendix~\ref{app:manualBSFH} for a documentation of \cite{BSFH.m}. In particular, it allows us to confirm conjecture~\ref{conj:MutInvHFL} for mutation about a particular non-trivial tangle.
\begin{theorem}[\sref{thm:2m3pt}, \sref{exa:HFTdpretzeltangle}]\label{thm:23pretzelintro}
Consider the following \((2,-3)\)-pretzel tangle:
\begin{center}
\psset{unit=0.65}
\begin{pspicture}[showgrid=false](-5.2,-3.1)(3.2,3.1)
\psecurve(-2.5,1.5)(0,2)(0.75,1)(-0.75,-1)(0,-2)(0.97,-2.24)(2,-2)
\psecurve[arrowhead=3pt 1.5]{<-}(2,2)(0.97,2.24)(0,2)(-0.75,1)(0.75,-1)(0,-2)(-2.5,-1.5)(-3.25,0)(-2.5,1.5)(0,2)(0.75,1)
\psecurve[arrowhead=3pt 1.5]{<-}(-6,1.5)(-3.3,1.85)(-2.5,1.5)(-1.85,0)(-2.5,-1.5)(-3.3,-1.85)(-6,-1.5)
\pscircle*[linecolor=white](-2.5,1.5){0.2}

\psecurve(0.75,-1)(0,-2)(-2.5,-1.5)(-3.25,0)(-2.5,1.5)

\pscircle*[linecolor=white](0,2){0.2}
\pscircle*[linecolor=white](0,0){0.2}
\pscircle*[linecolor=white](0,-2){0.2}

\psecurve(0.75,1)(-0.75,-1)(0,-2)(0.97,-2.24)(2,-2)
\psecurve(0,2)(-0.75,1)(0.75,-1)(0,-2)
\psecurve(-2.5,-1.5)(-3.25,0)(-2.5,1.5)(0,2)(0.75,1)(-0.75,-1)

\pscircle*[linecolor=white](-2.5,-1.5){0.2}
\psecurve(-2.5,1.5)(-1.85,0)(-2.5,-1.5)(-3.3,-1.85)(-6,-1.5)
\pscircle[linestyle=dotted](-1,0){3.05}
\end{pspicture}
\end{center}
Two knots or links that are related by mutation of this tangle have the same bigraded knot or link Floer homologies after identifying the Alexander gradings corresponding to the two open strands. If the orientation of one of those two strands is reversed, then their \(\delta\)-graded knot or link Floer homologies agree. 
\end{theorem}
We offer two independent proofs of this result, both of which, however, rely on calculations using the program~\cite{BSFH.m}. The first proof follows from computing the glueing structure for the $(2,-3)$-pretzel tangle in two different ways corresponding to mutation and observing that the two results are homotopy equivalent, see theorem~\sref{thm:2m3pt}. However, this proof does not really tell us \textit{why} they are homotopic. The second proof answers this question more satisfactorily, replacing the previous ad hoc construction by a conceptually more refined approach, see example~\sref{exa:HFTdpretzeltangle} in conjunction with theorem~\sref{thm:CFTdGeneralGlueing}. To explain this, let us discuss bordered invariants in a little more detail.

\subsection*{Bordered theory and Fukaya categories. }
In general, the invariants and the glueing theorems in bordered and also bordered sutured Heegaard Floer homology look rather complicated. With a closed 3-manifold~$M$ split along a (parametrised) closed surface~$F$ into two components~$M_1$ and~$M_2$, one associates a differential graded algebra~$\mathcal{A}(F)$, a so-called type D module~$\CFD(M_1)$ and a type A module~$\CFA(M_2)$ over~$\mathcal{A}(F)$. Then~$\HF(M)$ can be computed as the homology of a special tensor product~$\boxtimes$ of $\CFD(M_1)$ and $\CFA(M_2)$ over $\mathcal{A}(F)$. If one wants to split $M$ into more than two pieces, there are also bimodule invariants of type AA, AD, DA and DD, depending on whether we treat the glueing surfaces of the components as type A or type D sides.\pagebreak[3]\\
Especially the algebra $\mathcal{A}(F)$ can be quite complicated, and in general, it will be so for our tangle Floer homology. 
In \cite{Auroux}, Auroux gave an interpretation of the bordered algebra $\mathcal{A}(F)$ in terms of some partially wrapped Fukaya category and outlined a conjectural reformulation of bordered Heegaard Floer theory in this framework. In~\cite{LOTMor}, Lipshitz, Ozsváth and Thurston gave further evidence for this conjectural relationship by restating their glueing theorem purely in terms of the homology of the space of morphisms between the two type D modules of $M_1$ and $M_2$. Moreover, Rasmussen, Hanselman and Watson very recently gave an elegant reformulation of the glueing theorem for a surprisingly large class of manifolds with torus boundary, which they called loop-type \cite{HRW}. For such manifolds, the type D structure can be interpreted as an object in the Fukaya category of immersed curves on a (punctured) torus. In particular, glueing corresponds to taking Lagrangian intersection homology on the torus. \\
For 4-ended tangles, a similar story seems to be true, which we explore in chapter~\ref{chapter:HFTd}. 

\subsection*{A glueing structure for 4-ended tangles. } 
\begin{theorem}[\sref{defthm:CFTd}]
Given a 4-ended tangle~\(T\), we can endow the tangle Floer homology \(\HFT\) with an additional structure of a curved complex, which we denote by \(\CFTd(T)\). (Roughly speaking, a curved complex is a type~D module for which the differential does not square to 0, see definition~\ref{exa:HighBrowDefcurvedTypeD}.) \(\CFTd(T)\) is a tangle invariant. We call it the \textbf{peculiar module} of \(T\). 
\end{theorem}
This enables us to explicitly calculate objects for 4-ended tangles in the (triangulated enlargement of the) fully wrapped Fukaya category $\TwFuk(S^2,4)$ of the 4-punctured sphere, via the $A_\infty$-functor $\mathcal{L}$ in the following theorem.
\begin{theorem}[\sref{thm:TwFukpqModEquivalent}]\label{thm:INTROTwFukpqModEquivalent}
Let \(\pqMod\) be the category of peculiar modules. There exist two non-trivial \(A_\infty\)-functors \vspace*{-0.4cm}
$$\begin{tikzcd}[column sep=2cm]
\pqMod
\arrow[in=175,out=5]{r}{\mathcal{L}}
&
\TwFuk(S^2,4).\nmathphantom{\TwFuk(S^2,4)}\phantom{\pqMod}
\arrow[in=-5,out=185]{l}{\mathcal{M}}
\end{tikzcd}$$
\end{theorem}
This allows us to interpret the sites of 4-ended tangles explicitly in terms of generators of the Fukaya category  
$\TwFuk(S^2,4)$ via the following proposition.
\begin{proposition}[\sref{prop:CFTsiteFromFukCFTd}]\label{prop:INTROCFTsiteFromFukCFTd}
Let \(T\) be a 4-ended tangle and \(s\) a site of~\(T\). Then, there exists a generator \(L_s\) of \(\TwFuk(S^2,4)\) such that \(\CFT(T,s)\) is bigraded chain homotopic to the Lagrangian intersection chain complex
$$\left(\Mor\left(L_s,\mathcal{L}\left(\CFTd(T)\right)\right),\mutw_1\right),$$
where \(\mutw_1\) denotes the differential on morphism spaces in \(\TwFuk(S^2,4)\).
\end{proposition}
We expect that one can extend the proof of the result above to show the following.
\begin{conjecture}[\sref{conj:TwFukpqModEquivalent}]
The functors \(\mathcal{M}\) and \(\mathcal{L}\) from theorem~\ref{thm:INTROTwFukpqModEquivalent} define an equivalence of \(A_\infty\)-categories.
\end{conjecture}

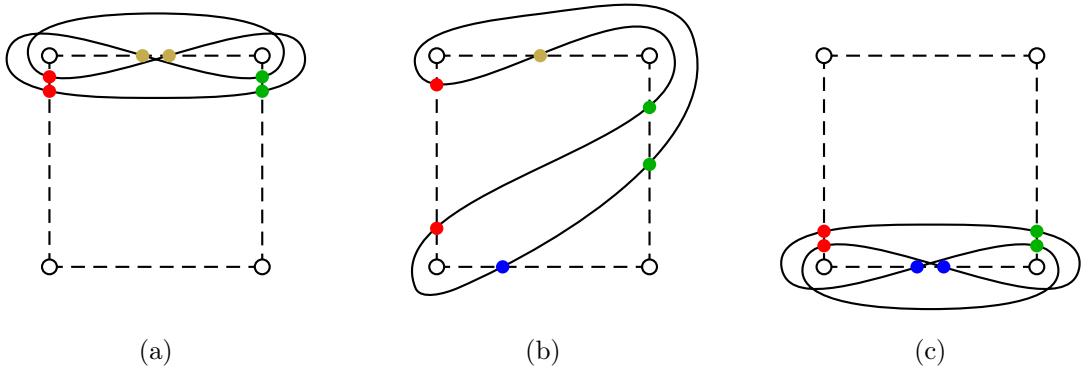
\begin{figure}[t]
\psset{unit=1.4}
\begin{subfigure}[b]{0.3\textwidth}\centering
\begin{pspicture}(-1.5,-1.5)(1.5,1.5)
\psecurve(1.2,1)(0,1.4)(-1.2,1)(1.4,1)(0,0.6)(-1.4,1)(1.2,1)(0,1.4)(-1.2,1)

\psline[linestyle=dashed](1,1)(1,-1)
\psline[linestyle=dashed](1,-1)(-1,-1)
\psline[linestyle=dashed](-1,-1)(-1,1)
\psline[linestyle=dashed](-1,1)(1,1)

\psset{dotsize=5pt}

\psdot[linecolor=red](-1,0.8)
\psdot[linecolor=red](-1,0.663)

\psdot[linecolor=gold](0.125,1)
\psdot[linecolor=gold](-0.125,1)

\psdot[linecolor=darkgreen](1,0.8)
\psdot[linecolor=darkgreen](1,0.663)

\pscircle[fillstyle=solid, fillcolor=white](1,1){0.08}
\pscircle[fillstyle=solid, fillcolor=white](-1,1){0.08}
\pscircle[fillstyle=solid, fillcolor=white](1,-1){0.08}
\pscircle[fillstyle=solid, fillcolor=white](-1,-1){0.08}

\end{pspicture}
\caption{}\label{fig:INTROloopbottom}
\end{subfigure}
\quad
\begin{subfigure}[b]{0.3\textwidth}\centering
\begin{pspicture}(-1.5,-1.5)(1.5,1.5)

\psecurve(-0.8,-1.2)(-1.2,-1.2)(1.2,1.1)(-1.2,0.9)(0,1.4)(1.4,1.2)(-0.8,-1.2)(-1.2,-1.2)(1.2,1.1)

\psline[linestyle=dashed](1,1)(1,-1)
\psline[linestyle=dashed](1,-1)(-1,-1)
\psline[linestyle=dashed](-1,-1)(-1,1)
\psline[linestyle=dashed](-1,1)(1,1)

\psset{dotsize=5pt}

\psdot[linecolor=red](-1,0.724)
\psdot[linecolor=red](-1,-0.635)

\psdot[linecolor=gold](-0.027,1)

\psdot[linecolor=darkgreen](1,0.51)
\psdot[linecolor=darkgreen](1,-0.03)

\psdot[linecolor=blue](-0.38,-1)

\pscircle[fillstyle=solid, fillcolor=white](1,1){0.08}
\pscircle[fillstyle=solid, fillcolor=white](-1,1){0.08}
\pscircle[fillstyle=solid, fillcolor=white](1,-1){0.08}
\pscircle[fillstyle=solid, fillcolor=white](-1,-1){0.08}

\end{pspicture}
\caption{}\label{fig:INTROloopmiddle}
\end{subfigure}
\quad
\begin{subfigure}[b]{0.3\textwidth}\centering
\begin{pspicture}(-1.5,-1.5)(1.5,1.5)

\psrotate(0,0){180}{
\psecurve(1.2,1)(0,1.4)(-1.2,1)(1.4,1)(0,0.6)(-1.4,1)(1.2,1)(0,1.4)(-1.2,1)

\psline[linestyle=dashed](1,1)(1,-1)
\psline[linestyle=dashed](1,-1)(-1,-1)
\psline[linestyle=dashed](-1,-1)(-1,1)
\psline[linestyle=dashed](-1,1)(1,1)

\psset{dotsize=5pt}

\psdot[linecolor=darkgreen](-1,0.8)
\psdot[linecolor=darkgreen](-1,0.663)

\psdot[linecolor=blue](0.125,1)
\psdot[linecolor=blue](-0.125,1)

\psdot[linecolor=red](1,0.8)
\psdot[linecolor=red](1,0.663)

\pscircle[fillstyle=solid, fillcolor=white](1,1){0.08}
\pscircle[fillstyle=solid, fillcolor=white](-1,1){0.08}
\pscircle[fillstyle=solid, fillcolor=white](1,-1){0.08}
\pscircle[fillstyle=solid, fillcolor=white](-1,-1){0.08}
}
\end{pspicture}
\caption{}\label{fig:INTROlooptop}
\end{subfigure}
\caption{The three loops on the 4-punctured sphere respresenting $\CFTd$ for the $(2,-3)$-pretzel tangle of theorem~\ref{thm:23pretzelintro}. The intersection points of the loops with the dashed arcs connecting the four punctures calculate the underlying non-glueable tangle Floer homology $\HFT$. The splitting of $\HFT$ according to the four different sites is indicated by the colouring of the intersection points.}\label{fig:INTROmutationexamplefinalresult}
\end{figure}
For rational tangles, but also for more complicated ones like the $(2,-3)$-pretzel tangle, the peculiar modules are loop-type in a similar sense to \cite{HRW}: the tangle Floer homology of these tangles is represented by a collection of loops in the 4-punctured 2-sphere, see figure~\ref{fig:INTROmutationexamplefinalresult} and examples~\sref{exa:CFTdRatTang} and~\sref{exa:HFTdpretzeltangle}. 
Furthermore, computations suggest that we can indeed compute link Floer homology as the Lagrangian intersection homology of such immersed curves. In fact, we have the following glueing result.

\begin{theorem}[\sref{thm:CFTdGlueingTrivial} and \sref{thm:CFTdGeneralGlueing}]\label{thm:INTROglueing}
Let \(T_1\) and \(T_2\) be two 4-ended tangles and \(L\) the link obtained by glueing them together according to the following picture.
$$
\psset{unit=0.6}
\begin{pspicture}(-2.5,-1.6)(2.5,1.6)
\pscurve(-1.5,0)(-0.7,1)(2.3,1)(1.5,0)
\pscircle*[linecolor=white](0,1.2){0.35}
\pscurve(-1.5,0)(-2.3,1)(0.7,1)(1.5,0)

\pscurve(-1.5,0)(-0.7,-1)(2.3,-1)(1.5,0)
\pscircle*[linecolor=white](0,-1.2){0.35}
\pscurve(-1.5,0)(-2.3,-1)(0.7,-1)(1.5,0)

\pscircle*[linecolor=white](-1.5,0){1}
\pscircle(-1.5,0){1}
\rput[c](-1.5,0){$T_1$}

\pscircle*[linecolor=white](1.5,0){1}
\pscircle(1.5,0){1}
\rput[c](1.5,0){$T_2$}
\end{pspicture}
$$
Then there exists a certain type~\(AA\) bimodule \(\mathcal{P}\) such that \(\CFL(L)\) is equal to 
$$
\CFTd(T_1)\boxtimes\mathcal{P}\boxtimes \CFTd(T_2)
$$
up to at most three stabilisations, i.\,e.\ tensoring with a certain 2-dimensional vector space. Furthermore, for loop-type~\(\CFTd(T_1)\) and trivial~\(T_2\), \(\CFL(L)\) agrees with the Lagrangian intersection Floer homology of the loop of~\(T_1\) with that of the trivial tangle, up to at most a single stabilisation.
\end{theorem}
The computation of the type AA bimodule~$\mathcal{P}$ is done using~\cite{BSFH.m}. The second statement follows from computing a type A structure, which can be done by hand, or alternatively, by simplifying $\mathcal{P}\boxtimes \CFTd(T_2)$. We expect the second statement to generalise to pairings of arbitrary (loop-type) tangles. Unfortunately, however, the bimodule~$\mathcal{P}$ looks rather complicated and not like the one to be expected from the Fukaya category. So there remains work to be done, see conjectures~\sref{conj:CFTdBetterGlueing} and~\sref{conj:GlueingCFTdFUK}.\pagebreak[3]\\
Nonetheless, the mere existence of a glueing theorem for $\CFTd$ allows us to infer properties of link Floer homology. For example, the peculiar invariant of the $(2,-3)$-pretzel tangle has an intrinsic symmetry (see figure~\ref{fig:INTROmutationexamplefinalresult}), which gives us the second proof of theorem~\ref{thm:23pretzelintro}. We expect other 2-stranded pretzel tangles to have the same kind of symmetry; however, a slightly more careful analysis of holomorphic curves is needed to determine the structure maps.
\begin{conjecture}
Theorem~\ref{thm:23pretzelintro} generalises to any 2-stranded pretzel-tangle.
\end{conjecture}
\pagebreak

For general 4-ended tangles, we are only able to prove slightly weaker symmetry relations for $\CFTd$. They support the mutation conjecture, see propositions \sref{prop:CFTdAll4sites} and~\sref{prop:CFTdPeculiarRanks}, but they do not seem to be sufficient to prove it. \\
As another application of the existence of a glueing theorem, we show that peculiar modules detect rational tangles:
\begin{theorem}[\sref{thm:CFTdDetectsRatTan}]
A 4-ended tangle \(T\) is rational iff \(\CFTd(T)\) is homotopic to a single loop that corresponds to an embedded loop on the 4-punctured sphere.
\end{theorem}
Furthermore, we can easily reprove the existence of an unoriented skein exact sequence \cite{Manolescu}, see theorem~\sref{thm:ResolutionExactTriangle}. Similarly, we obtain the following slight generalisation of Ozsváth and Szabó's oriented skein exact sequence \cite{OSHFK}.

\begin{theorem}[\sref{thm:nTwistSkeinRelation}, see also \sref{rem:nTwistSkeinRelation}]\label{thm:INTROnTwistSkeinRelation}
Let \(T_n\) be the positive \(n\)-twist tangle, \(T_{-n}\) the negative \(n\)-twist tangle and \(T_0\) the trivial tangle, see figure~\ref{fig:INTROOrientedSkeinRelation}. Then there is an exact triangle:
$$\begin{tikzcd}[row sep=0.9cm, column sep=-0.5cm]
\CFTd(T_{-n})
\arrow{rr}
& &
\CFTd(T_{n})
\arrow{dl}
\\
&
\CFTd(T_{0})\otimes V
\arrow{lu}
\end{tikzcd}$$
where \(V\) is some 2-dimensional vector space. If the tangles are oriented and coloured consistently, one obtains (bi)graded versions of this triangle. 
Furthermore, it gives rise to an exact triangle relating the (appropriately stabilised) link Floer homologies of links that differ in these three tangles.
\end{theorem}
\begin{figure}[t]
\centering
\psset{unit=0.35}
\begin{subfigure}[b]{0.25\textwidth}\centering
$n\left\{\raisebox{-0.85cm}{
\begin{pspicture}(-1.1,-3.1)(1.1,3.1)
\psecurve(1,5)(-1,3)(1,1)(-1,-1)
\pscircle*[linecolor=white](0,2){0.3}
\psecurve(-1,5)(1,3)(-1,1)(1,-1)
\rput(0,0.3){$\vdots$}
\psecurve(-1,-5)(1,-3)(-1,-1)(1,1)
\pscircle*[linecolor=white](0,-2){0.3}
\psecurve(1,-5)(-1,-3)(1,-1)(-1,1)
\end{pspicture}}\right.\quad
$
\caption{$T_{n}$}\label{fig:INTROOrientedSkeinRelationTn}
\end{subfigure}
\quad
\begin{subfigure}[b]{0.25\textwidth}\centering
$n\left\{\raisebox{-0.85cm}{
\begin{pspicture}(-1.1,-3.1)(1.1,3.1)
\psecurve(-1,5)(1,3)(-1,1)(1,-1)
\pscircle*[linecolor=white](0,2){0.3}
\psecurve(1,5)(-1,3)(1,1)(-1,-1)
\rput(0,0.3){$\vdots$}
\psecurve(1,-5)(-1,-3)(1,-1)(-1,1)
\pscircle*[linecolor=white](0,-2){0.3}
\psecurve(-1,-5)(1,-3)(-1,-1)(1,1)
\end{pspicture}}\right.\quad
$
\caption{$T_{-n}$}\label{fig:INTROOrientedSkeinRelationTmn}
\end{subfigure}
\quad
\begin{subfigure}[b]{0.25\textwidth}\centering
\begin{pspicture}(-5,-3.5)(5,3.5)
\psecurve(-2,6)(-1,3)(-1,-3)(-2,-6)
\psecurve(2,6)(1,3)(1,-3)(2,-6)
\end{pspicture}
\caption{$T_0$}\label{fig:INTROOrientedSkeinRelationT0}
\end{subfigure}
\caption{The basic tangles for the skein exact sequence from theorem~\ref{thm:INTROnTwistSkeinRelation}}\label{fig:INTROOrientedSkeinRelation}
\end{figure}
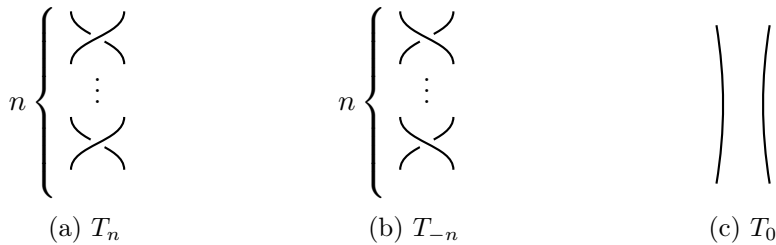

\subsection*{Parallels to Khovanov homology. }
Our peculiar invariants $\CFTd$ are also interesting from another, perhaps more philosophical perspective, namely the relationship between link Floer homology and Khovanov homology. Khovanov homology is another homology theory for knots and links, first defined by Khovanov in 1999 \cite{Khovanov}. It categorifies the Jones polynomial in the same way that $\HFL$ categorifies the Alexander polynomial. Although the two theories are defined and computed in very different ways, they look quite similar from a formal point of view, see for example \cite{JakeComparisonKhHFK}.

\begin{figure}[t]
\centering
\psset{unit=0.9}
\psset{xunit=1.5,nodesep=2pt,nrot=:U}
\begin{pspicture}(-5,-3.4)(5,3.6)
\rput(1,1){\rnode{A}{Khovanov homology}}
\rput(-3,1){\rnode{B}{knot Floer homology}}
\rput(-3,-3){\rnode{C}{Alexander polynomial}}
\rput(1,-3){\rnode{D}{Jones polynomial}}
\rput(3,3){\rnode{E}{\Centerstack{Bar-Natan's Khovanov\\ homology for tangles}}}
\rput(-1,3){\rnode{F}{?}}
\rput(-1,-1){\rnode{G}{?}}
\rput(3,-1){\rnode{H}{\Centerstack{Jones polynomial \\ for tangles}}}

\ncline[linestyle=dotted,dotsep=1pt]{F}{G}
\ncline[linestyle=dotted,dotsep=1pt]{<->}{G}{H}
\ncline[linestyle=dotted,dotsep=1pt]{->}{C}{G}
\ncline{E}{H}\ncput*{categorifies}
\ncline{->}{D}{H}
\ncline{<->}{A}{B}
\ncline{B}{C}\ncput*{categorifies}
\ncline{<->}{C}{D}
\ncline{A}{D}\ncput*{categorifies}
\ncline{->}{A}{E}
\ncline{->}{B}{F}
\ncline{<->}{E}{F}
\end{pspicture}
\caption{Another perspective on $\nabla_T^s$, $\HFT$ and especially $\CFTd$}\label{fig:introcontext}
\end{figure}
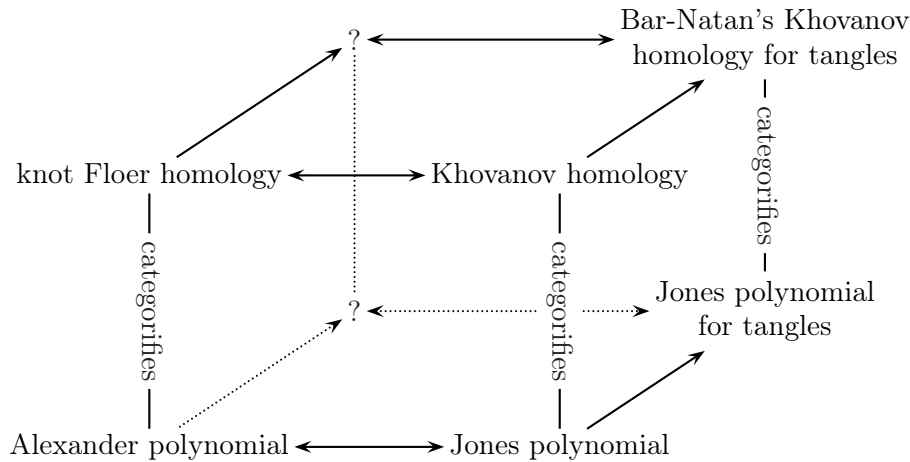

\noindent
In \cite{BarNatanKhT}, Bar-Natan gave an elegant generalisation of Khovanov homology to tangles. With a tangle diagram, he associated an (up to homotopy) invariant chain complex over a certain category which essentially (that is, up to grading) consists of only finitely many objects and morphisms. For example, for 4-ended tangles, there are just two objects and at most four morphisms in each hom-set. For $\CFTd$, we similarly get two candidates for such basic objects, and we sketch how to write the peculiar module of any tangle as a chain complex in these basic objects, up to a large tensor factor, see remark~\sref{rem:singularcrossings}. For the $(2,-3)$-pretzel tangle, this tensor factor can be removed.
\begin{questions}
Given any oriented 4-ended tangle \(T\), can we write \(\CFTd(T)\) as a chain complex in two basic objects? If so, can we describe the chain maps?
\end{questions}
An affirmative answer to these two questions would not only be aesthetically pleasing. The basic objects are symmetric under mutation, so if the chain maps between the basic complexes also have this symmetry, one might be able to prove the mutation conjecture. In Khovanov homology, such an approach has been successful: Khovanov homology with $\mathbb{Z}/2$-coefficients is known to be mutation invariant, see \cite{WehrliKhMutInv}.

\subsection*{Similar work by other people. }
There are several other groups of people working on similar ideas to those described in this thesis. \\
In 2014, Petkova and Vértesi defined a combinatorial tangle Floer homology using grid diagrams and ideas from bordered Floer homology \cite{cHFT}. They use a more general definition of tangles, namely those with a ``top'' and a ``bottom'', i.\,e.\ braids with caps and cups. In \cite{DecatCTFH}, they and Ellis show that the decategorification of their invariant agrees with Sartori's generalisation of the Alexander polynomials to top-bottom-tangles via representations of $U_q(\mathfrak{gl}(1\vert 1))$ \cite{Sartori14}. Thus, Petkova and Vértesi's theory fits very nicely into $\mathfrak{sl}_n$-homology theories arising from Khovanov homology. \\
Very recently, Ozsváth and Szabó developed a completely algebraically defined knot homology theory, which they conjecture to be equivalent to knot Floer homology \cite{OSKauffmanStates}. Like Petkova and Vértesi, they cut up a knot diagram into elementary pieces, so they automatically obtain tangle invariants, too. This theory seems to be frightfully powerful: from a computational point of view, since they can compute their homology from diagrams with over 50 crossings; but also from a more theoretical point of view, since their theory includes the hat- as well as the more sophisticated ``$-$''-version of knot Floer homology without reference to holomorphic curves or grid diagrams. Interestingly, the generators in their theory correspond to Kauffman states like in ours. \pagebreak\\
In \cite{HHK13,HHK15}, Hedden, Herald and Kirk study the Lagrangian intersection homology of immersed curves on a 4-punctured sphere (the ``pillowcase'') in the context of instanton knot Floer homology, a computationally more difficult knot Floer homology due to Kronheimer and Mrowka, which also categorifies the Alexander polynomial and is conjecturally closely related to $\HFK$ \cite{KM}. The curve they associate with the trivial tangle does not agree with ours, but it looks very similar to the curve we associate with a singular crossing, see proposition~\sref{prop:singularcrossing} and remark~\sref{rem:singularcrossings}.\\
Finally, I want to mention recent work of Lambert-Cole on conjecture~\ref{conj:MutInvHFL}. In \cite{LambertCole1}, he shows that the  bigraded knot Floer homologies of a mutant knot pair obtained by introducing a sufficiently large number of twists into a given (positive) mutant knot pair agree. This follows from a certain stabilisation property of knot Floer homology with respect to twisting that he proves in the same paper. In \cite{LambertCole2}, he investigates conjecture~\ref{conj:MutInvHFL} from a more axiomatic point of view, using basepoint maps and Manolescu's unoriented skein exact triangle as basic ingredients. He is able to confirm conjecture~\ref{conj:MutInvHFL} for mutating tangles that can be closed to an unlink by a rational tangle. In particular, this result encompasses the $\delta$-graded part of theorem~\ref{thm:23pretzelintro}.

\subsection*{Outline. }
The thesis is split into three chapters, following not only a logical, but incidentally also a roughly chronological order.\\
The first chapter is purely concerned with the decategorified story, the combinatorial definition of the polynomial tangle invariants~$\nabla_T^s$ and their properties. The chapter can be seen as a playground for testing ideas for chapters~\ref{chapter:categorification} and~\ref{chapter:HFTd}. In fact, some properties of $\nabla_{T}^s$ are immediate consequences of their categorified counterparts. However, other questions, most prominently the one concerning mutation invariance, are still unresolved in the categorified setting, whilst being relatively easy to answer for the decategorified invariants. On a first read, one can skip all sections of this chapter except the very first.\\
In chapter~\ref{chapter:categorification}, we define the tangle Floer homology~$\HFT$; first, via sutured Floer homology, then in more detail via Heegaard diagrams for tangles, imitating the definition of link Floer homology. We then interpret our invariant in terms of Zarev's bordered sutured theory.\\
In the third and final chapter, we specialise to 4-ended tangles and repackage the glueing structure in this special case into the peculiar invariant~$\CFTd$. We investigate some of its properties and discuss several applications and open questions.\\
There are four appendices. In appendix~\ref{appendix:AlgStructFromGDCats}, we describe the algebraic structures appearing in this thesis from an abstract category-theoretic point of view and derive useful tools for working with them, which form the basis of all our computations. In the second appendix, we give a proof of the generalised clock theorem, which is essential for studying the polynomial invariants~$\nabla_T^s$ in chapter~\ref{chapter:polynomial}.
The last two appendices are documentations for the programs~\cite{APT.m} and~\cite{BSFH.m}.

\subsection*{Acknowledgements}
First and foremost, I would like to thank my supervisor Jake Rasmussen for his generous support throughout the entire time of my PhD and before, during Part III. I consider myself very fortunate to have been his student.\pagebreak\\
My PhD was funded by an EPSRC scholarship covering tuition fees and a DPMMS grant for maintenance, for which I thank the then Head of Department Martin Hyland. I also gratefully acknowledge a research studentship from the Cambridge Philosophical Society for Michaelmas Term 2016.\\
I thank my examiners Ivan Smith and András Juhász for many valuable comments on and corrections to the soft-bound version of this thesis that I prepared for my viva. I also thank Mohammed Abouzaid, Guillem Cazassus, Celeste Damiani, Artem Kotelskiy, Peter Lambert-Cole, Adam Levine, Ina Petkova, Vera Vértesi and my brother Marcus Zibrowius for helpful conversations. My special thanks go to Liam Watson for his interest in my work, his support and the opportunity to speak in Glasgow twice.\\
I am very grateful to my PhD brothers Tom Brown, Tom Gillespie and Paul Wedrich, and my fellow PhD students Nina Friedrich and Christian Lund for their company and friendship. I would also like to thank Senja Barthel, Fyodor Gainullin, Tom Hockenhull and Marco Marengon for organising yearly student conferences at Imperial College London, all of which were terrific.\\
I thank Johnny Nicholson for helpful comments on an earlier draft of chapter~\ref{chapter:polynomial} and for sharing his computations with me during an undergraduate summer research project in 2016.\\
I am indebted to Tom Brown, Nina Friedrich, Paul Wedrich, Marcus Zibrowius and, especially, my father for their proof-reading services. 
\bigskip\\
No line of this thesis would have been written without the love and support of my brother and parents. This thesis is dedicated to them.

\begin{center}
	\vspace{35pt}
	\includegraphics[width=6cm]{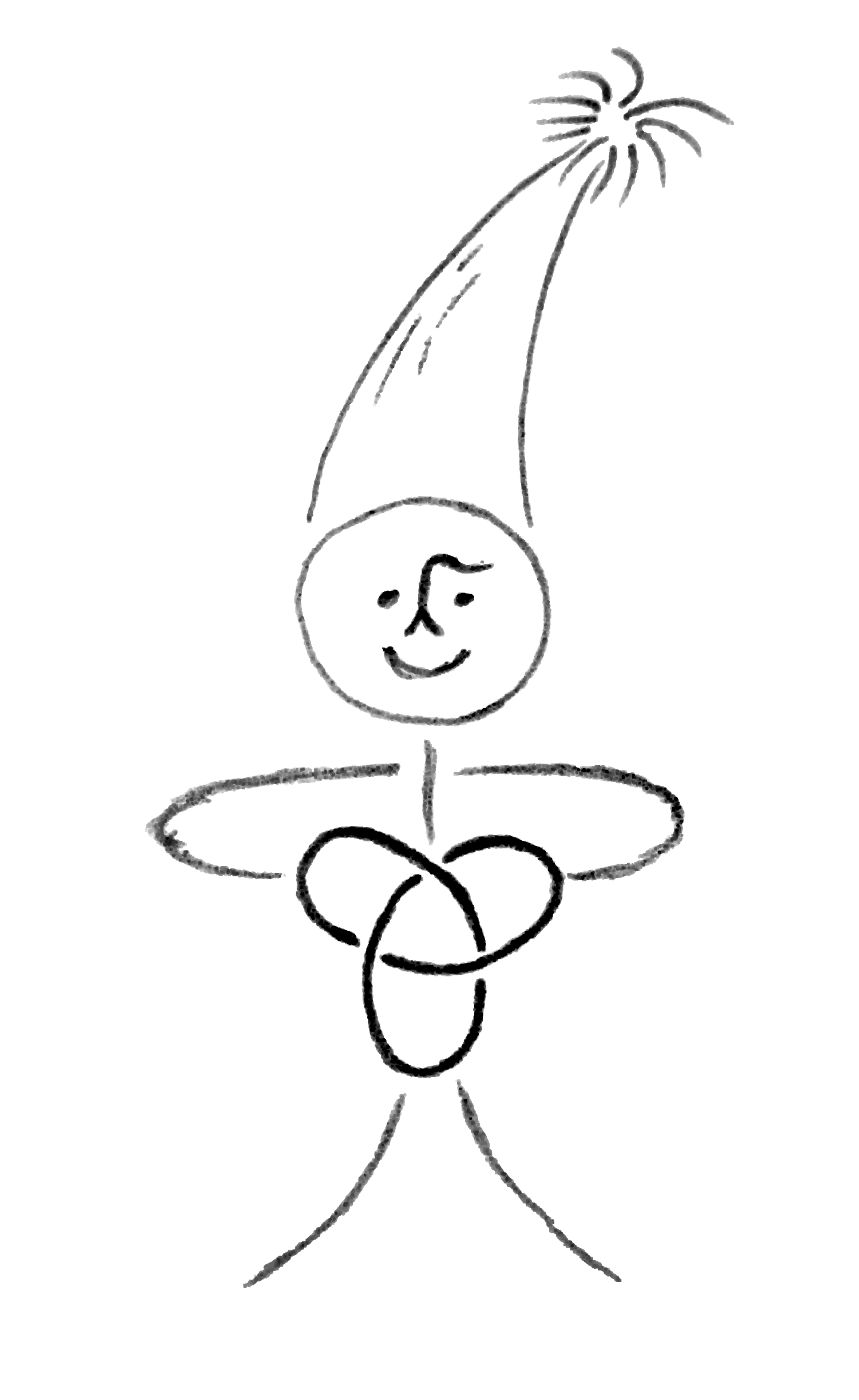}\\
\end{center}
\vspace*{\fill}

%% file: sections/1_Definitions.tex
\section{\texorpdfstring{The tangle invariant $\nabla_T^s$}{The tangle invariant ∇}}\label{sec:basicdefinitions}

First of all, we define what we mean by a tangle. Our definition is based on Conway's notion of tangles, see for example \cite[section~2.3]{Adams}.

\begin{definition}\label{def:tangle}
A \textbf{tangle} $T$ is an embedding of a disjoint union of intervals and circles into the closed 3-ball $B^3$,
$$T: \left(\coprod I \amalg \coprod S^1,\partial\right)\hookrightarrow \left(B^3,{\red S^1}\subset \partial B^3\right),$$
such that the endpoints of the intervals lie on a fixed circle ${\red S^1}$ on the boundary of~$B^3$, together with a labelling of the arcs ${\red S^1}\smallsetminus \im(T)$ by some index set $\{a,b,c,\dots\}$.
We consider tangles up to ambient isotopy which keeps track of the labelling of the arcs. If the number of intervals is $n$, we call a tangle \textbf{$\boldsymbol{2n}$-ended}. The images of the intervals are called \textbf{open components}, the images of the circles are called \textbf{closed components}. We often label these tangle components by variables $p,q,r,\dots$ or $t_1,t_2,\dots$, which we call the \textbf{colours} of~$T$. 
In analogy to link diagrams, we define a \textbf{tangle diagram} $D$ to be an immersion of intervals and circles into the closed 2-disc, 
$$D: \left(\coprod I \amalg \coprod S^1,\partial\right)\looparrowright \left(D^2,\partial D^2\right),$$
whose image is a graph with 1- and 4-valent vertices only, together with under/over information at each 4-valent vertex and a labelling of the arcs $\partial D^2\smallsetminus D(\partial I)$ by some index set $\{a,b,c,\dots\}$. Just as in the case of links, we consider tangle diagrams up to ambient isotopy and the usual Reidemeister moves, see for example~\cite{Lickorish}. Connected components of the complement of the image of $D$ are called \textbf{regions}. Those regions that meet $\partial D^2$ are called \textbf{open}, the others are called \textbf{closed}. We call a diagram \textbf{connected} if the intersection of each open region with $\partial D^2$ is connected. Unless specified otherwise, diagrams are assumed to be connected and have at least one crossing.
\end{definition}
\begin{Remark}
Regard $D^2$ as the intersection of $B^3$ with the plane $\{z=0\}$. Given a tangle diagram $D$, we can remove the singularities of the immersion $D$ by pushing the two components at each singularity into $\{z>0\}$ and $\{z<0\}$, according to the under/over information. Conversely, given a tangle $T$, we can choose an embedded disc $D^2$ bounding the fixed circle ${\red S^1}$. Then, just as in the case of links, a generic projection of $B^3$ onto this disc gives rise to a well-defined tangle diagram. So in the following, we use ``tangles'' and ``tangle diagrams'' synonymously, unless it is clear from the context that we do not.
\end{Remark} 
\begin{lemma}\label{lem:RMmovesconnectdiagrams}
Let \(T_1\) and \(T_2\) be two oriented connected tangle diagrams that represent the same tangle. Then there is a sequence of Reidemeister moves that connects \(T_1\) to \(T_2\) via connected diagrams.
\end{lemma}
\begin{proof}
By the previous remark, we can always find a sequence of Reidemeister moves connecting the two diagrams. To ensure that all diagrams are connected, we pick one open strand near the boundary and pull it once around the whole diagram, using Reidemeister II moves before we go along the sequence of Reidemeister moves and undo the first step once we have arrived at $T_2$.
\end{proof}
\begin{definition}
\textbf{Rational tangles} are 4-ended tangles without any closed components obtained from the 4-ended tangle in figure \ref{figsub:rattangle0a} by adding twists to the top and to the right.
\begin{figure}[t]
\centering
\psset{unit=0.5}
\begin{subfigure}[b]{0.18\textwidth}\centering
\begin{pspicture}(-2,-2)(2,2)
\rput(0,-0.9){
\pscircle[linestyle=dotted](0,0.9){2}
\psecurve(1.2,-1.9)(0.9,-0.9)(0.9,2.7)(1.2,3.7)
\psecurve(-1.2,-1.9)(-0.9,-0.9)(-0.9,2.7)(-1.2,3.7)
}
\rput(1.45;180){$a$}
\rput(1.45;-90){$b$}
\rput(1.45;360){$c$}
\rput(1.45;450){$d$}
\end{pspicture}
\caption{}
\label{figsub:rattangle0a}
\end{subfigure}
~
\begin{subfigure}[b]{0.18\textwidth}\centering
\begin{pspicture}(-2,-2)(2,2)
\rput(0,-0.9){
\pscircle[linestyle=dotted](0,0.9){2}
\psline(-0.9,-0.9)(0.9,0.9)(-0.9,2.7)
\pscircle*[linecolor=white](0,1.8){0.3}
\pscircle*[linecolor=white](0,0){0.3}
\psline(0.9,-0.9)(-0.9,0.9)(0.9,2.7)
}
\rput(1.45;180){$a$}
\rput(1.5;-90){$b$}
\rput(1.45;360){$c$}
\rput(1.5;450){$d$}
\end{pspicture}
\caption{}
\label{figsub:rattangle0b}
\end{subfigure}
~
\begin{subfigure}[b]{0.18\textwidth}\centering
\begin{pspicture}(-2,-2)(2,2)
\rput(0,-0.9){
\pscircle[linestyle=dotted](0,0.9){2}

\psline(0.9,0.9)(-0.9,2.7)
\pscircle*[linecolor=white](0,1.8){0.3}
\psline(0.9,-0.9)(-0.9,0.9)(0.9,2.7)
\pscircle*[linecolor=white](0,0){0.3}
\psline(-0.9,-0.9)(0.9,0.9)(0.5,1.3)
}
\rput(1.45;180){$a$}
\rput(1.5;-90){$b$}
\rput(1.45;360){$c$}
\rput(1.5;450){$d$}
\end{pspicture}
\caption{}
\label{figsub:rattangle2}
\end{subfigure}
~
\begin{subfigure}[b]{0.18\textwidth}\centering
\begin{pspicture}(-2,-2)(2,2)
\rput(-0.8,-0.65){
\pscircle[linestyle=dotted](0.8,0.65){2}
\psline(0.35,0.35)(0.65,0.65)(-0.65,1.95)
\pscircle*[linecolor=white](0,1.3){0.3}
\psline(2.65,1.35)(0.65,-0.65)(-0.65,0.65)(0.65,1.95)
\pscircle*[linecolor=white](0,0){0.3}
\psline(-0.65,-0.65)(0.65,0.65)
\pscircle*[linecolor=white](1.95,0.65){0.3}
\psline(0.35,1.65)(0.65,1.95)(2.65,-0.05)
}
\rput(1.55;200){$a$}
\rput(1.45;-70){$b$}
\rput(1.6;360){$c$}
\rput(1.45;430){$d$}
\end{pspicture}
\caption{}\label{figsub:rattangle21}
\end{subfigure}
~
\begin{subfigure}[b]{0.18\textwidth}\centering
\begin{pspicture}(-2,-2)(2,2)
\psrotate(0,0){90}{\rput(0,-0.9){
\pscircle[linestyle=dotted](0,0.9){2}
\psecurve(1.2,-1.9)(0.9,-0.9)(0.9,2.7)(1.2,3.7)
\psecurve(-1.2,-1.9)(-0.9,-0.9)(-0.9,2.7)(-1.2,3.7)
}}
\rput(1.45;180){$a$}
\rput(1.45;-90){$b$}
\rput(1.45;360){$c$}
\rput(1.45;450){$d$}
\end{pspicture}
\caption{}
\label{figsub:rattangle0e}
\end{subfigure}
\caption{Some diagrams of rational tangles. (a) and (b) represent the same tangle, but only (b) is a connected diagram. (c) and (d) show some more complicated rational tangles. (e) does not represent the same tangle as (a), since the labelling is different.}\label{fig:rattangles}
\end{figure}
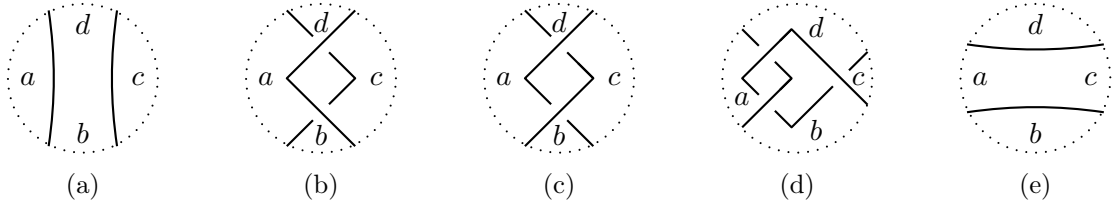
\end{definition}
\begin{Remark}
One might wonder why we have defined tangles the way we have. Alternatively, we could have simply allowed only those isotopies that fix the \textit{whole} boundary sphere. However, such a definition is too rigid, since it differentiates between far too many tangles that are essentially the same. In the other extreme, allowing just \textit{any} isotopies would mean that we do not distinguish between tangles that are related by some twists of the tangle ends. As a consequence, all rational tangles would, for example, be the same, see figure~\ref{fig:rattangles}. The introduction of a fixed circle on the boundary sphere in definition \ref{def:tangle} mediates between the two extremes. It can be viewed as a parametrisation of the punctured sphere $\partial B^3\smallsetminus\partial T$. We explore this point of view in section~\ref{sec:geometricinterpretation}, in particular, see proposition~\sref{prop:anyBisfine}. It will also play a major role in chapters~\ref{chapter:categorification} and~\ref{chapter:HFTd}.
\end{Remark}

\subsection*{Alexander polynomials of knots and links. }
Next, let us recall how the Alexander polynomial of knots and links can be computed using Kauffman states and Alexander codes following~\cite{Kauffman}. Given a diagram of a 2-ended tangle (whose closure represents a knot or link), a Kauffman state is an assignment of a marker~$\bullet$ to one of the four regions at each crossing such that each closed region is occupied by exactly one marker. 
One then applies the Alexander codes to the Kauffman states, i.\,e.\ one labels the markers by the monomials specified by the Alexander codes, as shown in figure~\ref{Kauffknot}. To get the multivariate Alexander polynomial, one just multiplies these labels, takes the sum over all Kauffman states and finally multiplies everything by some normalisation factor.\\
When trying to apply this well-known algorithm to the general case of a $2n$-ended tangle, one encounters the following problem: There are, say, $m$ crossings in the diagram, so by an Euler characteristic argument, there are at least $(m+n+1)$ regions. (We have exactly $(m+n+1)$ regions iff all regions are simply connected. Otherwise, we have a split component and so the Alexander polynomial should be zero.) Thus, there are at least $(n+1)$ regions more than there are markers, but the number of open regions in a connected diagram is $2n$. This motivates the following two definitions.
\begin{definition}\label{def:site} 
A \textbf{site} of a $2n$-ended tangle $T$ is a choice of an $(n-1)$-element subset of the set of arcs ${\red S^1}\smallsetminus \im(T)$. For connected tangle diagrams, this is equivalent to choosing $(n-1)$ open regions. The set of all sites of a tangle $T$ is denoted by $\mathbb{S}=\mathbb{S}(T)$.
\end{definition}

\begin{figure}[t]
	\centering
	\begin{subfigure}[b]{0.3\textwidth}\centering
		\psset{unit=1}
		\begin{pspicture}(-1.3,-1.3)(1.3,1.3)
		\psline[linecolor=violet]{->}(0.9,-0.9)(-0.9,0.9)
		\pscircle*[linecolor=white](0,0){0.3}
		\psline[linecolor=darkgreen]{->}(-0.9,-0.9)(0.9,0.9)
		\uput{0.5}[90](0,0){$\textcolor{violet}{u}$}
		\uput{0.4}[180](0,0){$\textcolor{violet}{u}^{-1}$}
		\uput{0.5}[-90](0,0){$-\textcolor{violet}{u}^{-1}$}
		\uput{0.7}[0](0,0){$\textcolor{violet}{u}$}
		\end{pspicture}
		\caption{The Alexander code for a positive crossing from \cite[figure~33]{Kauffman}. There is a similar one for a negative crossing. $\textcolor{violet}{u}$ is the colour of the under-strand.}
		\label{AlexCode}
	\end{subfigure}
	\quad
	\begin{subfigure}[b]{0.3\textwidth}\centering
		\psset{unit=0.59, linewidth=1.1pt}
		\begin{pspicture}[showgrid=false](-4.2,-3.1)(3.2,3.1)
		\psecurve(-2,2)(0,2)(0.75,1)(-0.75,-1)(0,-2)(0.97,-2.24)(2,-2)
		\psecurve(2,2)(0.97,2.24)(0,2)(-0.75,1)(0.75,-1)(0,-2)(-2,-2)(-3,0)(-2,2)(0,2)(0.75,1)
		
		\psecurve[linecolor=lightgray](0,-2)(0.97,-2.24)(2,-2)(2.8,-1)(2.8,1)(2,2)(0.97,2.24)(0,2)
		
		\pscircle*[linecolor=white](0,2){0.2}
		\pscircle*[linecolor=white](0,0){0.2}
		\pscircle*[linecolor=white](0,-2){0.2}
		
		\psecurve(0,2)(0.75,1)(-0.75,-1)(0,-2)
		\psecurve(2,2)(0.97,2.24)(0,2)(-0.75,1)(0.75,-1)
		\psecurve(-0.75,1)(0.75,-1)(0,-2)(-2,-2)(-3,0)
		
		\psline{->}(0.75,1.1)(0.75,1.2)
		\psline{->}(-0.75,1.11)(-0.75,1.2)
		\psline{->}(-0.75,-0.9)(-0.75,-0.8)
		\psline{->}(0.75,-0.9)(0.75,-0.8)
		
		\pscircle[linestyle=dotted](-1,0){3.05}
		
		\psline{->}(0.8,2.235)(0.9,2.25)
		\psline{->}(-0.8,2.235)(-0.9,2.25)

		\psdot(0,1.65)
		\psdot(-0.5,0)
		\psdot(0,-1.65)
		
		\uput{0.2}[-90](0,1.65){\scriptsize $-t^{-1}$}
		\uput{0.2}[180](-0.5,0){$t^{-1}$}
		\uput{0.2}[90](0,-1.65){$t$}
		
		\end{pspicture}
		\caption{A Kauffman state for a 2-ended tangle, coloured by~$t$. The closure is indicated by the grey arc.}
		\label{Kauffknot}
	\end{subfigure}
	\quad
	\begin{subfigure}[b]{0.3\textwidth}\centering
		\psset{unit=0.59, linewidth=1.1pt}
		\begin{pspicture}[showgrid=false](-4.5,-3.1)(2.5,3.1)
		
		\psecurve[linecolor=violet](-2.5,-1.5)(0,-2)(0.75,-1)(-0.75,1)(0,2)(0.97,2.24)
		\psecurve[linecolor=violet]{<-}(2,-2)(0.97,-2.24)(0,-2)(-0.75,-1)(0.75,1)(0,2)(-2.5,1.5)(-3.25,0)(-2.5,-1.5)(0,-2)(0.75,-1)
		
		\psecurve[linecolor=darkgreen]{->}(-6,-1.5)(-3.3,-1.85)(-2.5,-1.5)(-1.85,0)(-2.5,1.5)(-3.3,1.85)(-6,1.5)
		\psline[linecolor=darkgreen]{->}(-1.85,0.05)(-1.85,0.15)
		\pscircle*[linecolor=white](-2.5,1.5){0.2}
		\psecurve[linecolor=violet](0.75,1)(0,2)(-2.5,1.5)(-3.25,0)(-2.5,-1.5)
		\psecurve[linecolor=violet](0.75,-1)(0,-2)(-2.5,-1.5)(-3.25,0)(-2.5,1.5)
		
		\pscircle*[linecolor=white](0,2){0.2}
		\pscircle*[linecolor=white](0,0){0.2}
		\pscircle*[linecolor=white](0,-2){0.2}
		
		\psecurve[linecolor=violet](0.75,-1)(-0.75,1)(0,2)(0.97,2.24)(2,2)
		\psecurve[linecolor=violet](0,-2)(-0.75,-1)(0.75,1)(0,2)
		\psecurve[linecolor=violet](-2.5,1.5)(-3.25,0)(-2.5,-1.5)(0,-2)(0.75,-1)(-0.75,1)
		
		\pscircle*[linecolor=white](-2.5,-1.5){0.2}
		\psecurve[linecolor=darkgreen](-2.5,1.5)(-1.85,0)(-2.5,-1.5)(-3.3,-1.85)(-6,-1.5)

		\pscircle[linestyle=dotted](-1,0){3.05}
		
		
		
		\psline[linecolor=violet]{->}(0.75,-1.1)(0.75,-1.2)
		\psline[linecolor=violet]{->}(-0.75,-1.1)(-0.75,-1.2)
		\psline[linecolor=violet]{->}(-0.75,0.9)(-0.75,0.8)
		\psline[linecolor=violet]{->}(0.75,0.9)(0.75,0.8)

		\psline[linecolor=violet]{->}(-3.25,0.05)(-3.25,0.15)
		
		\psline[linecolor=violet]{<-}(-1.5,-2.015)(-1.4,-2.04)
		\psline[linecolor=violet]{->}(-1.5,2.015)(-1.4,2.05)
		
		\uput{0.2}[45](0.97,2.24){$\textcolor{violet}{p}$}
		\uput{0.2}[-45](0.97,-2.24){$\textcolor{violet}{p}$}
		\uput{0.2}[135](-3.3,1.85){$\textcolor{darkgreen}{q}$}
		\uput{0.2}[-135](-3.3,-1.85){$\textcolor{darkgreen}{q}$}
		
		\psdot(0,1.65)
		\psdot(0.5,0)
		\psdot(0,-1.65)
		
		\psdot(-2.5,-1)
		\psdot(-1.9,1.4)
		
		\uput{0.2}[90](0,-1.65){\footnotesize $-\textcolor{violet}{p}^{-1}$}
		\uput{0.2}[-90](0,1.65){$\textcolor{violet}{p}$}
		\uput{0.2}[0](0.5,0){$\textcolor{violet}{p}^{-1}$}
		\uput{0.2}[90](-2.5,-1){$\textcolor{violet}{p}$}
		\uput{0.2}[0](-1.9,1.4){$\textcolor{darkgreen}{q}$}
		\end{pspicture}
		\caption{A Kauffman state for a 4-ended tangle with two components coloured by $\textcolor{violet}{p}$ and $\textcolor{darkgreen}{q}$, respectively.}\label{Kaufftangle}
	\end{subfigure}
	\caption{Applying Alexander codes to Kauffman states}
\end{figure}
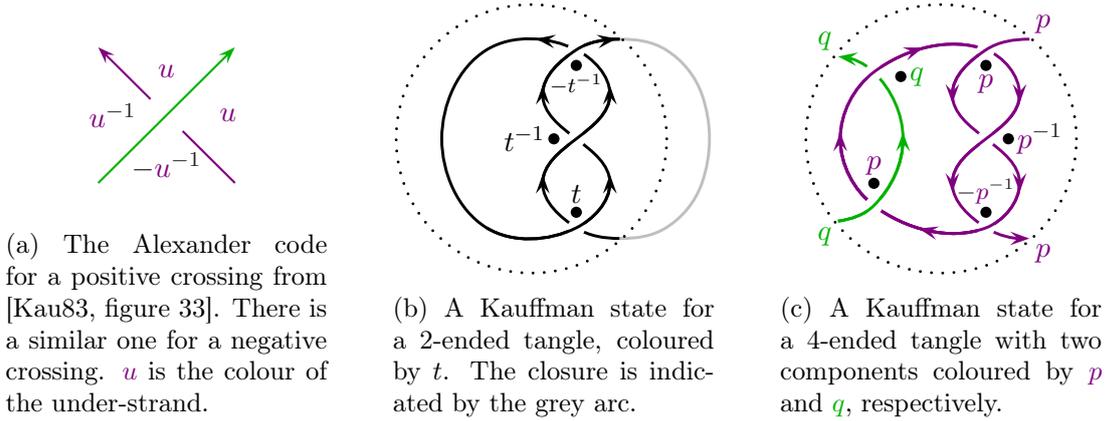

\begin{Remark}
	Later, we will give a more geometric interpretation of sites, see proposition~\sref{prop:anyBisfine}. In the categorified invariants of chapters~\ref{chapter:categorification} and~\ref{chapter:HFTd}, sites correspond to idempotents of the glueing algebra. 
\end{Remark}
\begin{definition}\label{def:basic} 
Let $T$ be a (connected) diagram of an oriented $2n$-ended tangle. A \textbf{generalised Kauffman state} of $T$ is an assignment of a marker to one of the four regions at each crossing such that each closed region is occupied by exactly one marker, with the additional condition that there be at most one marker in each open region. Furthermore,
\begin{itemize}
\item denote the set of all generalised Kauffman states of $T$ by $\mathbb{K}=\mathbb{K}(T)$.
\item For $x\in\mathbb{K}$, let $s(x)$ be the set of open regions that are occupied by a marker of $x$. (Note that $s(x)$ has $(n-1)$ elements, so $s(x)\in\mathbb{S}$.)
\item For each site $s\in\mathbb{S}$, let $\mathbb{K}(T,s):=\{x\in \mathbb{K}(T)\vert s(x)=s\}$.
\item For $x\in\mathbb{K}$, let $l(x)$ be the labelling of the markers of $x$ according to the Alexander codes in figure~\ref{figAlexCodesForNabla}. Moreover, let $c(x)$ be the product of the labels $l(x)$.
\end{itemize}
Then for each site $s\in\mathbb{S}$, let $$\hat{\nabla}_T^s:=\sum_{x\in\mathbb{K}(T,s)}c(x).$$
Furthermore, let $\nabla_T^s$ denote the function $\hat{\nabla}_T^s$ evaluated at $h=-1$. We call $\nabla_T^s$ the \textbf{Alexander polynomial of~$T$ at the site~$s$}.
\end{definition}
\begin{Remark}
The variable $h$ stands for ``homological grading''. In chapter \ref{chapter:categorification}, we will generalise the hat version of knot and link Floer homology to tangles. The generators of these homology groups will correspond to the generalised Kauffman states above.
\end{Remark}
\begin{observation}\label{ObsAlexanderCode}
In the Alexander codes of figure~\ref{figAlexCodesForNabla}, the exponents of $u$ in the two regions left of an under-strand are $-\tfrac{1}{2}$, and $+\tfrac{1}{2}$ in the regions on its right. For over-strands, it is the other way round. This Alexander code has the advantage over the one in figure \ref{AlexCode} that we do not need to multiply $\nabla^s_T$ by a normalisation factor to turn it into a tangle invariant.
\end{observation}
\begin{theorem} \label{thm:nablaisaninvariant}
For two oriented tangle diagrams \(T_1\) and \(T_2\) representing the same tangle and \(s\in\mathbb{S}(T_1)=\mathbb{S}(T_2)\), we have \(\nabla^s_{T_1}=\nabla^s_{T_2}\). So \(\nabla_T^s\) is a tangle invariant.
\end{theorem}
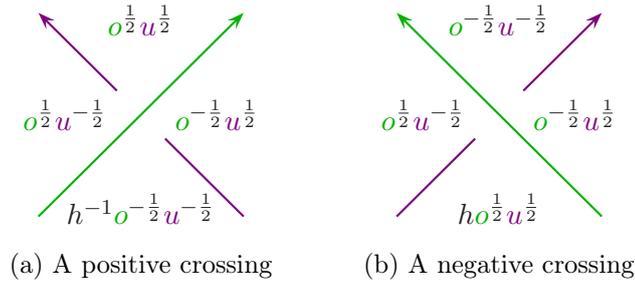
\begin{figure}[t]
	\centering
	\psset{unit=1.5}
	\begin{subfigure}[b]{0.3\textwidth}\centering
		\begin{pspicture}(-1.05,-1.05)(1.05,1.05)
		\psline[linecolor=violet]{->}(0.9,-0.9)(-0.9,0.9)
		\pscircle*[linecolor=white](0,0){0.3}
		\psline[linecolor=darkgreen]{->}(-0.9,-0.9)(0.9,0.9)
		\uput{0.7}[90](0,0){$\textcolor{darkgreen}{o}^{\frac{1}{2}} \textcolor{violet}{u}^{\frac{1}{2}}$}
		\uput{0.3}[180](0,0){$\textcolor{darkgreen}{o}^{\frac{1}{2}} \textcolor{violet}{u}^{-\frac{1}{2}}$}
		\uput{0.7}[-90](0,0){$h^{-1}\textcolor{darkgreen}{o}^{-\frac{1}{2}} \textcolor{violet}{u}^{-\frac{1}{2}}$}
		\uput{0.3}[0](0,0){$\textcolor{darkgreen}{o}^{-\frac{1}{2}} \textcolor{violet}{u}^{\frac{1}{2}}$}
		\end{pspicture}
		\caption{A positive crossing}
	\end{subfigure}
	\begin{subfigure}[b]{0.3\textwidth}\centering
		\begin{pspicture}(-1.05,-1.05)(1.05,1.05)
		\psline[linecolor=violet]{->}(-0.9,-0.9)(0.9,0.9)
		\pscircle*[linecolor=white](0,0){0.3}
		
		\psline[linecolor=darkgreen]{->}(0.9,-0.9)(-0.9,0.9)
		\uput{0.7}[90](0,0){$\textcolor{darkgreen}{o}^{-\frac{1}{2}} \textcolor{violet}{u}^{-\frac{1}{2}}$}
		\uput{0.3}[180](0,0){$\textcolor{darkgreen}{o}^{\frac{1}{2}} \textcolor{violet}{u}^{-\frac{1}{2}}$}
		\uput{0.7}[-90](0,0){$h\textcolor{darkgreen}{o}^{\frac{1}{2}} \textcolor{violet}{u}^{\frac{1}{2}}$}
		\uput{0.3}[0](0,0){$\textcolor{darkgreen}{o}^{-\frac{1}{2}} \textcolor{violet}{u}^{\frac{1}{2}}$}
		\end{pspicture}
		\caption{A negative crossing}
	\end{subfigure}
	\caption{The Alexander codes for definition \ref{def:basic}. The variable $\textcolor{darkgreen}{o}$ is the colour of the over-strand and $\textcolor{violet}{u}$ the colour of the under-strand.}\label{figAlexCodesForNabla}
\end{figure} 
\begin{proof}
By lemma \ref{lem:RMmovesconnectdiagrams}, we just need to check that the polynomials are invariant under the Reidemeister moves RM~I--III. We can check this locally, so the proof becomes exactly the same as for the usual knot and link case: We verify the theorem for the basic diagrams that appear in the Reidemeister moves, for each site separately. We only do this for RM~I and~II here; for RM~III, we refer the reader to the Mathematica notebook \cite{APT.nb} which uses the package~\cite{APT.m} for calculating $\nabla_T^s$ for any connected tangle diagram $T$ and site~$s$; see also appendix~\ref{app:manualAPT} for the corresponding manual.\\
RM~I looks as follows:
$$
\psset{unit=0.2}
\raisebox{-0.5cm}{
\begin{pspicture}(-2,-3)(2,3)
\psarc(-3,0){4.2426406871193}{-45}{45}
\end{pspicture}
}
\hspace{-0.7cm}
\tikzset{commutative diagrams/column sep/my size/.initial=1.5cm}
\begin{tikzcd}[column sep=my size]    
\phantom{X}\arrow{r}{\text{RM I}}&\arrow{l}\phantom{X}
\end{tikzcd}
\hspace{-0.7cm}
\raisebox{-0.5cm}{
\begin{pspicture}(-3,-3)(4,3)
\psline(-3,-3)(1,1)
\pscircle*[linecolor=white](0,0){0.3}
\psline(1,-1)(-3,3)
\psarcn(2,0){1.4142135623731}{135}{-135}
\end{pspicture}
}\medskip
$$
The enclosed region on the right only has one crossing. Hence, the corresponding marker has to sit in that region in every Kauffman state. For both orientations, the labelling of this marker is 1, so we might as well remove this crossing. The same holds if we reverse the crossing; we can either check this directly, or apply proposition~\ref{prop:mirrortangle}. \\
For RM~II, we only check one orientation; again, for the others, we can either check this separately or simply apply proposition~\ref{prop:reverseorientI}.
\begin{center}
\psset{unit=0.5}
\raisebox{-1.15cm}{
\begin{pspicture}(-1.5,-1.6)(1.5,3.5)

\psecurve[linecolor=darkgreen]{<-}(-1.8,3.6)(-0.9,2.7)(-0.3,0.45)(-0.9,-0.9)(-1.8,-1.8)
\psecurve[linecolor=violet]{<-}(1.8,3.6)(0.9,2.7)(0.3,0.45)(0.9,-0.9)(1.8,-1.8)

\rput(0,1.8){
\uput{1.35}[135](0,0){$\textcolor{darkgreen}{p}$}
\uput{1.35}[45](0,0){$\textcolor{violet}{q}$}}
\uput{1.35}[-135](0,0){$\textcolor{darkgreen}{p}$}
\uput{1.35}[-45](0,0){$\textcolor{violet}{q}$}

\end{pspicture}
}
\hspace{-1cm}
$
\tikzset{commutative diagrams/column sep/my size/.initial=1.5cm}
\begin{tikzcd}[column sep=my size]    
\phantom{X}\arrow{r}{\text{RM II}}&\arrow{l}\phantom{X}
\end{tikzcd}
$
\hspace{-1cm}
\raisebox{-1.15cm}{
\begin{pspicture}(-1.5,-1.6)(1.5,3.5)
\rput(0,1.8){
\psline[linecolor=violet]{->}(-0.9,-0.9)(0.9,0.9)
\pscircle*[linecolor=white](0,0){0.3}
\psline[linecolor=darkgreen]{->}(0.9,-0.9)(-0.9,0.9)
}
\psline[linecolor=violet]{->}(0.9,-0.9)(-0.9,0.9)
\pscircle*[linecolor=white](0,0){0.3}
\psline[linecolor=darkgreen]{->}(-0.9,-0.9)(0.9,0.9)

\rput(0,1.8){
\uput{1.35}[135](0,0){$\textcolor{darkgreen}{p}$}
\uput{1.35}[45](0,0){$\textcolor{violet}{q}$}}
\uput{1.35}[-135](0,0){$\textcolor{darkgreen}{p}$}
\uput{1.35}[-45](0,0){$\textcolor{violet}{q}$}

\end{pspicture}
}
\end{center}
In the diagram on the right, there are exactly two Kauffman states that occupy the open region on the left; they contribute $p$ and $hp$, so after setting $h=-1$, they cancel. The same is true for the open region on the right; the contribution there is $q$ and $hq$. Finally, for each of the open regions at the top and the bottom, there is exactly one Kauffman state and it contributes 1.
\end{proof}

We have chosen the letter $\nabla$ for a reason:

\begin{theorem}\label{thm:twoended}
Let \(T\) be a diagram of an oriented $2$-ended tangle representing a link~\(L\). Note that in this case there is only one site of~\(T\), namely the empty set~\(\emptyset\). Let the colour of the open component be~\(c\). Then the Conway potential function \(\nabla_L\) is equal to
$$\frac{1}{c-c^{-1}}\nabla^\emptyset_T.$$
\end{theorem}
\begin{proof}
Verify that $\nabla:=\frac{1}{c-c^{-1}}\nabla^\emptyset_T$ satisfies the axioms in \cite{jiang14}, see \cite{APT.nb}.
\end{proof}

\begin{Remark}\label{rem:conwaypotfunctionandHFL}
Recall that the Conway potential function of an $n$-component oriented link $L$
is a rational function $\nabla_L(t_1,\dots,t_n)$
which is related to the multivariate Alexander polynomial $\Delta_L$ in the following way: (see for example \cite{hartley} or \cite{jiang14})
$$
\nabla_L(t_1,\dots,t_n)=
\begin{cases}
\displaystyle\frac{\Delta_L(t_1^2)}{t_1-t_1^{-1}}, & \text{if $n=1$};
\\
\Delta_L(t_1^2,\dots,t_n^2), & \text{if $n>1$}.
\end{cases}
$$
Hence, using the notation of the theorem above
$$
\nabla^\emptyset_T(t_1,\dots,t_n)=
\begin{cases}
\Delta_L(t_1^2), & \text{if $n=1$};
\\
(c-c^{-1}) \Delta_L(t_1^2,\dots,t_n^2), & \text{if $n>1$}.
\end{cases}
$$
We also note that $\nabla^\emptyset_T$ of a 2-ended tangle $T$, multiplied by a factor of $(c-c^{-1})$ for each \textit{closed} component of $T$, is equal to the Euler characteristic of Ozsváth and Szabó's link Floer homology from~\cite{OSHFL}.
\end{Remark}

We collect some properties of the Conway potential function in the following theorem.

\begin{theorem}\label{thm:CPFprops}
{\normalfont\cite[propositions 5.6, 5.5, 5.7 and 5.3]{hartley}}
The Conway potential function of an oriented link \(L\) satisfies the following properties:
\begin{enumerate}[(i)]
\item If \(\m(L)\) denotes the mirror image of \(L\), then \label{CPFpropmirror}
$$\nabla_{\m(L)}(t_1,\dots,t_r)=(-1)^{r-1}\cdot\nabla_L(t_1,\dots,t_r).$$
\item $\nabla_L(t_1,\dots,t_r)=(-1)^{r}\cdot\nabla_L(t_1^{-1},\dots,t_r^{-1}).$\label{CPFpropssymmetry}
\item If \(\rr(L,t_1)\) is obtained from \(L\) by reversing the orientation of the first strand, then \label{CPFpropRevOneStrand}
$$\nabla_{\rr(L,t_1)}(t_1,\dots,t_r)=-\nabla_L(t_1^{-1},t_2,\dots,t_r).
$$
\item If \(L_1,\dots,L_r\) denote the link components of \(L\), then \(\nabla_L(1,t_2,\dots,t_r)\) is equal to \label{CPFoneequalone}\vspace{\abovedisplayskip}\\
\textcolor{white}{\qedsymbol}\hfill$(t_2^{\lk(L_1,L_2)}\cdots t_r^{\lk(L_1,L_r)}-t_2^{-\lk(L_1,L_2)}\cdots t_r^{-\lk(L_1,L_r)})\nabla_{L\smallsetminus L_1}(t_2,\dots,t_r).$\hfill\qedsymbol

\end{enumerate}
\end{theorem}
Often, we want to glue tangle diagrams together along some parts of their boundary to obtain a new tangle diagram. We can easily compute the polynomial invariant of the new tangle from the invariants of the two glued components. (In fact, we have implicitly used the fact that $\nabla_T^s$ behaves well under glueing in the proof of theorem~\ref{thm:nablaisaninvariant} already.) The following result generalises the connected sum formula for knots and links.

\begin{proposition}[splitting/glueing formula]\label{prop:glueing}
Let \(T_1\) and \(T_2\) be two oriented tangles obtained by splitting an oriented tangle diagram \(T\) along some arc that does not meet any crossings, for example like so:
$$
\psset{unit=0.4}
\begin{pspicture}(-8.01,-3.1)(8.01,3.1)
\psellipse[linestyle=dotted](0,0)(5,3)

\psline(-2,0)(2,0)
\psline[linestyle=dotted](0,3)(0,-3)

\rput(-2,0){
\psline(0,0)(-3,0)
\psline(0,0)(2.6;125)
\psline(0,0)(2.6;-125)
}

\rput(2,0){
\psline(0,0)(3,0)
\psline(0,0)(2.6;55)
\psline(0,0)(2.6;-55)
}

\psecurve(-2,0)(-1.5,1)(1.5,1)(2,0)
\psecurve(-2,0)(-1.5,-1)(1.5,-1)(2,0)

\pscircle[fillcolor=white,fillstyle=solid](-2,0){1.5}
\pscircle[fillcolor=white,fillstyle=solid](2,0){1.5}
\rput(-6.5,0){$T=$}
\rput(-2,0){$T_1$}
\rput(2,0){$T_2$}
\end{pspicture}
$$
Let \(t_{1,1},\dots, t_{1,r_1}\), \(t_{2,1},\dots, t_{2,r_2}\) and \(t_{1},\dots, t_{r}\) be the colours of \(T_1\), \(T_2\) and \(T\), respectively. Glueing \(T_1\) and \(T_2\) back together induces an identification of these colours which gives rise to two homomorphisms 
$$\iota_i:\mathbb{Z}[t^{\pm 1/2}_{i,1},\dots, t^{\pm 1/2}_{i,r_i}]\rightarrow \mathbb{Z}[t^{\pm 1/2}_{1},\dots, t^{\pm 1/2}_{r}],\qquad i=1,2.$$
Then
$$\hat{\nabla}_T^s=\sum_{\substack{s_1\cap s_2=\emptyset\\ (s_1\cup s_2)\cap O(T)=s}}\iota_1(\hat{\nabla}_{T_1}^{s_1})\cdot \iota_2(\hat{\nabla}_{T_2}^{s_2}),$$
where \(O(T)\) denotes the set of open regions in the diagram \(T\).
\end{proposition}
\begin{proof}
This follows immediately from definition \ref{def:basic}, noting that the sum is over exactly those marker assignments that define generalised Kauffman states for $T$.
\end{proof}

Before moving on to the next section to study some properties of our tangle invariant, we state a generalisation of Kauffman's clock theorem \cite[theorem~2.5]{Kauffman} to tangles. For this we recall the following definition from \cite[figure~5]{Kauffman}. (Note that the theorem is false if we use the stronger definition given in the introduction of the same monograph.)

\begin{definition}
Suppose in a tangle diagram, there are two crossings which have two regions in common. 
%
Suppose further that in a Kauffman state $x$, the markers of the two crossings lie in these two regions. Then there also exists a Kauffman state $y$ obtained from $x$ by moving each of the markers of the two crossings in exactly the opposite region. We say that we go from $x$ to $y$ by a \textbf{transposition move}. A transposition move is called \textbf{clockwise} if the markers move clockwise around each crossing, as illustrated below:
\begin{center}\psset{unit=0.7}
\begin{pspicture}(-9,-1.3)(9,1.3)
\rput(-5,0){
\pscustom*[linecolor=lightgray]{
\psline(-2,0)(-2,1)
\psecurve(-4,1.3)(-2,1)(0,1.3)(2,1)(4,1.3)
\psline(2,1)(2,0)
\psline(2,0)(-2,0)
}
\pscustom*[linecolor=gray]{
\psline(-2,0)(-2,-1)
\psecurve(-4,-1.3)(-2,-1)(0,-1.3)(2,-1)(4,-1.3)
\psline(2,-1)(2,0)
\psline(2,0)(-2,0)
}
\psecurve*[linecolor=white](1,0)(0,0.8)(-1,0)(0,-0.8)(1,0)(0,0.8)(-1,0)

\psline(-2,1)(-2,-1)
\psline(2,1)(2,-1)
\psline(-3,0)(-1,0)
\psline[linestyle=dotted](-1,0)(1,0)
\psline(1,0)(3,0)
\uput{0.4}[45](-2,0){\psdot[dotsize=5pt](0,0)}
\uput{0.4}[-135](2,0){\psdot[dotsize=5pt](0,0)}
\psarc(-2,0){0.4}{-40}{45}
\uput{0.4}[-50](-2,-0.07){\rput{154}(0,0){\psline[arrowsize=2pt 3]{->}(0,-0.1)(0,0)}
}
\psarc(2,0){0.4}{140}{225}
\uput{0.4}[130](2,0.07){\rput{154}(0,0){\psline[arrowsize=2pt 3]{->}(0,0.1)(0,0)}
}}

\rput[c](0,0){$x~\longrightarrow~y$}




\rput(5,0){
\pscustom*[linecolor=lightgray]{
\psline(-2,0)(-2,1)
\psecurve(-4,1.3)(-2,1)(0,1.3)(2,1)(4,1.3)
\psline(2,1)(2,0)
\psline(2,0)(-2,0)
}
\pscustom*[linecolor=gray]{
\psline(-2,0)(-2,-1)
\psecurve(-4,-1.3)(-2,-1)(0,-1.3)(2,-1)(4,-1.3)
\psline(2,-1)(2,0)
\psline(2,0)(-2,0)
}
\psecurve*[linecolor=white](1,0)(0,0.8)(-1,0)(0,-0.8)(1,0)(0,0.8)(-1,0)

\psline(-2,1)(-2,-1)
\psline(2,1)(2,-1)
\psline(-3,0)(-1,0)
\psline[linestyle=dotted](-1,0)(1,0)
\psline(1,0)(3,0)
\uput{0.4}[-45](-2,0){\psdot[dotsize=5pt](0,0)}
\uput{0.4}[135](2,0){\psdot[dotsize=5pt](0,0)}
}
\end{pspicture}
\end{center}
The reverse is called a \textbf{anticlockwise} move. In appendix~\ref{proofgct}, we prove the following result.
\end{definition}

\begin{theorem}[generalised clock theorem]\label{geclockt}
Let \(T\) be a (not necessarily connected) tangle diagram and \(s\in\mathbb{S}(T)\). Then \(\mathbb{K}(T,s)\) is a finite lattice under the relation $$x>y \Leftrightarrow \exists \text{ sequence of clockwise transposition move from }x\text{ to }y.$$
\end{theorem}

%% file: sections/1_BasicProps.tex
\section{\texorpdfstring{Basic properties of $\nabla_T^s$}{Basic properties of ∇}}\label{sec:basicpropertiesofnabla}

In this section, we collect and prove some properties of the tangle invariants $\nabla_T^s$, guided by theorem \ref{thm:CPFprops}. We will see that with the exception of the symmetry property (\ref{CPFpropssymmetry}), the results generalise to the tangle case. Our generalisation of property (\ref{CPFoneequalone}) leaves room for improvement because it only applies to closed components. For open components, some better understanding of the relations between sites would probably be helpful. \\

The first proposition below corresponds to part (\ref{CPFpropmirror}) combined with (\ref{CPFpropssymmetry}) of theorem~\ref{thm:CPFprops}. 
\begin{proposition}\label{prop:mirrortangle}
Let \(T\) be an oriented tangle and \(\m(T)\) its mirror image. Then for all \(s\in\mathbb{S}(T)\), 
$$\hat{\nabla}^{s}_{\m(T)}(t_1,\dots,t_r,h)=\hat{\nabla}^{s}_T(t_1^{-1},\dots ,t_r^{-1}, h^{-1}).$$
\end{proposition}
\begin{proof}
Observe that the two Alexander codes in figure~\ref{figAlexCodesForNabla} are mirror images of one another after taking the reciprocals of all variables. 
\end{proof}
\begin{definition}\label{def:linkingnumber}
We define the \textbf{linking number} $\lk_T(p,q)$ for two components $p$ and $q$ of a tangle $T$ to be 
$$\lk_T(p,q):=\tfrac{1}{2}(\#\{\text{positive crossings of $p$ and $q$}\}-\#\{\text{negative crossings of $p$ and $q$}\}).$$
For a tangle with a component $t_j$, we also define 
$$\lk_T(t_j):=\sum\lk_T(t_i,t_j),$$ 
where the sum is over all components $t_i\neq t_j$. We sometimes omit the subscript when there is no risk of ambiguity.
\end{definition}
\begin{Remark} Note that for two-component links, $\lk(p,q)$ coincides with the usual linking number. Also, linking numbers are invariants of tangles.
\end{Remark}

\begin{lemma}\label{lem:Powerdistance2}
Given a tangle diagram \(T\), the powers of any colour in two Kauffman states of the same site \(s\in\mathbb{S}(T)\) differ by a multiple of 2. Furthermore, the exponent of a colour \(t\) in \(\hat{\nabla}_T^s\) is an integer iff \(\lk(t)\) is an integer.
\end{lemma}
\begin{proof}
By the generalised clock theorem (theorem \ref{geclockt}), any two Kauffman states with site $s$ are connected by a sequence of transposition moves. It is easy to see that two states connected by a single transposition move have either the same Alexander grading or the exponents of one colour (the one corresponding to the horizontal strand) changes by $\pm2$. 
The second statement follows directly from the definition of the linking number.
\end{proof}

The next proposition corresponds to part (\ref{CPFpropRevOneStrand}) of theorem~\ref{thm:CPFprops}. 
\begin{proposition}\label{prop:reverseorientI}
Let \(T\) be an oriented \(r\)-component tangle. If \(\rr(T,t_1)\) denotes the same tangle \(T\) with the orientation of the first strand reversed, then for all sites \(s\in\mathbb{S}(T)\), we have 
$$\hat{\nabla}^s_{\rr(T,t_1)}(t_1,\dots,t_r)=h^{\lk_T(t_1)} \hat{\nabla}^s_{T}(h^{-1}t_1^{-1},t_2,\dots,t_r).$$
\end{proposition}

\begin{proof}
This is easily seen by considering crossings separately: Modulo sign, the statement follows from observation~\ref{ObsAlexanderCode}. For the correct sign, note that after substituting $h^{-1}t_1^{-1}$ into the Alexander code of a positive (negative) crossing involving $t_1$ and some different colour, we obtain the Alexander code of the crossing with the orientation of the $t_1$-strand reversed multiplied by $h^{-\frac{1}{2}}$ (respectively $h^{\frac{1}{2}}$). For crossings involving only $t_1$, no additional factor is necessary, and for crossings not involving $t_1$ at all, there is nothing to show.
\end{proof}
Note that one has to be careful when setting $h=-1$ in the proposition above, because the $\hat{\nabla}^s_T$ are Laurent polynomials with half-integer powers. The same applies to corollary~\ref{coralloorientsrev} below. In corollary~\ref{cor:revorientrevisited}, we give a formula that is more convenient when working with $\nabla_T^s$ instead of $\hat{\nabla}_T^s$. 
\begin{corollary}\label{coralloorientsrev}
Let \(T\) be an oriented \(r\)-component tangle. If \(\rr(T)\) denotes the same tangle \(T\) with the orientation of all strands reversed, then for all sites \(s\in\mathbb{S}(T)\), we have $$\hat{\nabla}^s_{\rr(T)}(t_1,\dots,t_r)=\hat{\nabla}^s_{T}(h^{-1}t_1^{-1},\dots,h^{-1}t_r^{-1}).$$
\end{corollary}
\begin{proof}
We successively reverse the orientation of all strands, noting that each term $\lk_T(t_i,t_j)$ appears twice in the exponent of $h$, but with different signs, because the second time it appears, the orientation of one strand has been reversed.
\end{proof}
\begin{corollary}\label{cor:recoverlinkcase}
Let \(\rr(\cdot)\) denote the function which substitutes \(-t^{-1}\) for each colour~\(t\). Then, for an oriented link \(L\), we have the symmetry relation
$$\nabla_{L}=\nabla_{\rr(L)}=\rr(\nabla_{L}).$$
\end{corollary}
\begin{proof}
The first equality is theorem \ref{thm:CPFprops} (ii) and (iii), the second follows from theorem~\ref{thm:twoended} and the corollary above with $h=-1$.
\end{proof}
\begin{lemma}[one-colour skein relation]\label{onecolourskein}
Let \(T_+\), \(T_-\) and \(T_\circ\) denote the tangles 
\raisebox{-3pt}{\psset{unit=0.2}\begin{pspicture}(-1.05,-1.45)(1.05,1.05)
\psset{arrowsize=1.5pt 2}
\psline{->}(0.9,-0.9)(-0.9,0.9)
\pscircle*[linecolor=white](0,0){0.3}
\psline{->}(-0.9,-0.9)(0.9,0.9)
\end{pspicture}}, 
\raisebox{-3pt}{\psset{unit=0.2}\begin{pspicture}(-1.05,-1.45)(1.05,1.05)
\psset{arrowsize=1.5pt 2}
\psline{->}(-0.9,-0.9)(0.9,0.9)
\pscircle*[linecolor=white](0,0){0.3}
\psline{->}(0.9,-0.9)(-0.9,0.9)
\end{pspicture}} and
\raisebox{-3pt}{\psset{unit=0.2}\begin{pspicture}(-1.05,-1.45)(1.05,1.05)
\psset{arrowsize=1.5pt 2}
\psline{->}(-0.9,-0.9)(0.9,0.9)
\psline{->}(0.9,-0.9)(-0.9,0.9)
\pscircle*[linecolor=white](0,0){0.5}
\psarc(0.7071,0){0.5}{135}{225}
\psarc(-0.7071,0){0.5}{-45}{45}
\end{pspicture}}
respectively. Then for all sites \(s\),
$$\nabla_{T_+}^s(t,t)-\nabla_{T_-}^s(t,t)=(t-t^{-1})\cdot \nabla_{T_\circ}^s(t,t).$$
Thus, the single-variate polynomial tangle invariant \(\nabla_T^s(t,\dots,t)\) satisfies the same skein relation as the Alexander polynomial. 
\end{lemma}
\begin{proof}
Straightforward.
\end{proof}
\begin{corollary}\label{corknotpmoneequalone}
Let \(T\) be a 2-ended tangle representing a knot. Then \(\nabla_{T}^s(\pm 1)=1\).
\end{corollary}
\begin{proof}
Let $T'$ be the diagram obtained from $T$ by changing some crossings such that $T'$ represents the unknot. Then, by lemma~\ref{onecolourskein}, $\nabla_{T}^s(\pm 1)=\nabla_{T'}^s(\pm 1)$, and $\nabla_{T'}^s(t)\equiv 1$.
\end{proof}

The next proposition corresponds to part (\ref{CPFoneequalone}) of theorem~\ref{thm:CPFprops}.
\begin{proposition}\label{propsetpmone}
Let \(T\) be a tangle with a closed component \(K_1\). Then for all sites \(s\in\mathbb{S}\), \(\nabla^s_{T}(\pmr 1,t_2,\dots,t_r)\) is equal to
$$(\pmr1)^{\lk(t_1)+1}(t_2^{\lk(t_1,t_2)}\cdots t_r^{\lk(t_1,t_r)}-t_2^{-\lk(t_1,t_2)}\cdots t_r^{-\lk(t_1,t_r)})\cdot\nabla^s_{T\smallsetminus K_1}(t_2,\dots,t_r).$$
\end{proposition}
\begin{proof}
First we apply lemma~\ref{onecolourskein} to $t_1$-$t_1$-crossings in $T$. 
Setting $t_1=\pmr 1$ gives $$\nabla_{T}^s(\pmr 1,t_2,\dots,t_r)=\nabla_{T'}^s(\pmr 1,t_2,\dots,t_r),$$
where $T'$ denotes the tangle obtained by swapping over- and under-strands at some $t_1$-$t_1$-crossing. We can therefore assume without loss of generality that the $t_1$-component is the unknot, in particular that there are no $t_1$-$t_1$-crossings and near the $t_1$-component, the diagram looks as follows:
\begin{center}
\raisebox{-2.4cm}{\psset{unit=1.1}
\begin{pspicture}(-3,-2.75)(3,2.7)
\psarc(0,0){2}{110}{70}
\psarc[linestyle=dotted](0,0){2}{70}{110}
\uput{2}[40](0,0){\rput{40}(0,0){
\psarc(2,0){2}{150}{210}
\psarc[linestyle=dotted](2,0){2}{130}{150}
\psarc[linestyle=dotted](2,0){2}{210}{230}
\pscustom[fillstyle=solid,fillcolor=white]{\pspolygon(0.3,-0.6)(0.3,0.6)(-0.3,0.6)(-0.3,-0.6)}}\rput{40}(0,0){$T_{\pmt}$}}
\uput{2}[-25](0,0){\rput{-25}(0,0){
\psarc(2,0){2}{150}{210}
\psarc[linestyle=dotted](2,0){2}{130}{150}
\psarc[linestyle=dotted](2,0){2}{210}{230}
\pscustom[fillstyle=solid,fillcolor=white]{\pspolygon(0.3,-0.6)(0.3,0.6)(-0.3,0.6)(-0.3,-0.6)}}\rput{-25}(0,0){$T_{\pmt}$}}
\uput{2}[-90](0,0){\rput{-90}(0,0){
\psarc(2,0){2}{150}{210}
\psarc[linestyle=dotted](2,0){2}{130}{150}
\psarc[linestyle=dotted](2,0){2}{210}{230}
\pscustom[fillstyle=solid,fillcolor=white]{\pspolygon(0.3,-0.6)(0.3,0.6)(-0.3,0.6)(-0.3,-0.6)}}\rput{-90}(0,0){$T_{\pmt}$}}
\uput{2}[-155](0,0){\rput{-155}(0,0){
\psarc(2,0){2}{150}{210}
\psarc[linestyle=dotted](2,0){2}{130}{150}
\psarc[linestyle=dotted](2,0){2}{210}{230}
\pscustom[fillstyle=solid,fillcolor=white]{\pspolygon(0.3,-0.6)(0.3,0.6)(-0.3,0.6)(-0.3,-0.6)}}\rput{-155}(0,0){$T_{\pmt}$}}
\uput{2}[-220](0,0){\rput{-220}(0,0){
\psarc(2,0){2}{150}{210}
\psarc[linestyle=dotted](2,0){2}{130}{150}
\psarc[linestyle=dotted](2,0){2}{210}{230}
\pscustom[fillstyle=solid,fillcolor=white]{\pspolygon(0.3,-0.6)(0.3,0.6)(-0.3,0.6)(-0.3,-0.6)}}\rput{-220}(0,0){$T_{\pmt}$}}
\end{pspicture}}
where 
\raisebox{-0.91cm}{\begin{pspicture}(-0.35,-1)(0.35,1)
\pscustom[fillstyle=solid,fillcolor=white]{\pspolygon(0.3,-0.6)(0.3,0.6)(-0.3,0.6)(-0.3,-0.6)}\rput(0,0){$T_{\textcolor{blue}{+}}$}
\end{pspicture}}
$=$
\raisebox{-1.34cm}{\psset{unit=0.5}
\begin{pspicture}(-1,-2)(1,4)
\rput(0,1.8){
\psline{->}(0.9,-0.9)(-0.9,0.9)
\pscircle*[linecolor=white](0,0){0.3}
\psline{->}(-0.9,-0.9)(0.9,0.9)
}

\psline{->}(0.9,-0.9)(-0.9,0.9)
\pscircle*[linecolor=white](0,0){0.3}
\psline{->}(-0.9,-0.9)(0.9,0.9)

\rput(-0.6,3.2){$t_1$}
\rput(0.6,3.2){$t_i$}
\end{pspicture}}
and 
\raisebox{-0.91cm}{\begin{pspicture}(-0.35,-1)(0.35,1)
\pscustom[fillstyle=solid,fillcolor=white]{\pspolygon(0.3,-0.6)(0.3,0.6)(-0.3,0.6)(-0.3,-0.6)}\rput(0,0){$T_{\textcolor{blue}{-}}$}
\end{pspicture}}
$=$
\raisebox{-1.34cm}{\psset{unit=0.5}
\begin{pspicture}(-1,-2)(1,4)
\rput(0,1.8){
\psline{->}(0.9,-0.9)(-0.9,0.9)
\pscircle*[linecolor=white](0,0){0.3}
\psline{<-}(-0.9,-0.9)(0.9,0.9)
}

\psline{<-}(0.9,-0.9)(-0.9,0.9)
\pscircle*[linecolor=white](0,0){0.3}
\psline{->}(-0.9,-0.9)(0.9,0.9)

\rput(-0.6,3.2){$t_1$}
\rput(0.6,3.2){$t_i$}
\end{pspicture}} for $i\neq 1$.
\end{center}
\pagebreak[3]

\noindent
The following table gives the values of $\nabla^s_{T_{\pmt}}$, where $l, b, r, t$ denote the regions on the left, bottom, right and top of the diagrams $T_{\pmt}$, respectively.
\begin{center}
\begin{tabular}{ccc}
$s$ & $T_{\textcolor{blue}{+}}$ & $T_{\textcolor{blue}{-}}$\\ 
\hline
\rule{0pt}{14pt}$l$ & $(t_i-t_i^{-1})$ & $-(t_i-t_i^{-1})$ \\ 
$b$ & $t_1^{-1} t_i^{-1}$ & $t_1^{-1} t_i$ \\ 
$r$ & $(t_1-t_1^{-1})$ & $-(t_1-t_1^{-1})$ \\ 
$t$ & $t_1 t_i$ & $t_1 t_i^{-1}$ \\ 
\end{tabular} 
so after setting $t_1=\pm1$:
\begin{tabular}{ccc}
$s$ & $T_{\textcolor{blue}{+}}$ & $T_{\textcolor{blue}{-}}$ \\ 
\hline 
\rule{0pt}{14pt}$l$ & $(t_i-t_i^{-1})$ & $-(t_i-t_i^{-1})$ \\ 
$b$ & $\pmr t_i^{-1}$ & $\pmr t_i$ \\ 
$r$ & $0$ & $0$ \\ 
$t$ & $\pmr t_i$ & $\pmr t_i^{-1}$ \\ 
\end{tabular} 
\end{center}
Let the number of rectangular boxes $T_{\pmt}$ be denoted by $n$. The fact that $\nabla^{r}_{T_{\pmt}}(1,t_i)=0$ means that if we consider the picture above as a $2n$-ended tangle $T''$, there are at most $n$ sites such that the polynomial invariant becomes non-zero after substituting $t_1=\pm1$. Each of these sites is specified by the region which is not occupied by a marker and which is not a right region of any box $T_{\pmt}$. Fix such a region $u$. Then each Kauffman state of $T''$ is determined by a box and a marker in the left region of this box. For each box, there are two such Kauffman states, namely those that contribute to $\nabla^{l}$ of that box.\\ 
We introduce the following notation: We number the boxes in anticlockwise direction, starting from the fixed unoccupied region $u$. Let  $\alpha_j^{i\pmt}$ denote the number of $T_{\pmt}$s involving the $i^\text{th}$ strand in anticlockwise direction from $u$ to the $j^{\text{th}}$ box $T_j$. Similarly, let $\beta_j^{i\pmt}$ denote the number of $T_{\pmt}$s involving the $i^\text{th}$ strand in clockwise direction from $u$ to $T_j$. Let $i_j\neq1$ denote the index of the strand involved in $T_j$. The contribution of the boxes $T_{\pmt}\neq T_j$ to the labelling of the two Kauffman states corresponding to $T_j$ is
\begin{center}
\begin{tabular}{ccc}
$\alpha_j^{i\textcolor{blue}{+}}$&$\rightarrow$& $\pmr t_i$,\\
$\alpha_j^{i\textcolor{blue}{-}}$&$\rightarrow$& $\pmr t_i^{-1}$,\\
$\beta_j^{i\textcolor{blue}{+}}$&$\rightarrow$& $\pmr t_i^{-1}$,\\
$\beta_j^{i\textcolor{blue}{-}}$&$\rightarrow$& $\pmr t_i$,
\end{tabular}
\end{center}
so that the contribution of the two Kauffman states to $\nabla_{T''}$ of the site determined by $u$ is 
$$\pmt(t_{i_j}-t_{i_j}^{-1})\cdot \prod_{i\neq 1} (\pm t_i)^{(\alpha_j^{i\textcolor{blue}{+}}+\beta_j^{i\textcolor{blue}{-}})-(\alpha_j^{i\textcolor{blue}{-}}+\beta_j^{i\textcolor{blue}{+}})}.$$
Note that by definition
\begin{equation*}
\lk(t_1,t_i)= 
\begin{cases}
\alpha_j^{i\textcolor{blue}{+}}-\alpha_j^{i\textcolor{blue}{-}}+\beta_j^{i\textcolor{blue}{+}}-\beta_j^{i\textcolor{blue}{-}} &\text{ if $i\neq1, i_j$,}\\
\alpha_j^{i\textcolor{blue}{+}}-\alpha_j^{i\textcolor{blue}{-}}+\beta_j^{i\textcolor{blue}{+}}-\beta_j^{i\textcolor{blue}{-}}+1  &\text{ if $i=i_j$ and $T_j=T_{\textcolor{blue}{+}}$,}\\
\alpha_j^{i\textcolor{blue}{+}}-\alpha_j^{i\textcolor{blue}{-}}+\beta_j^{i\textcolor{blue}{+}}-\beta_j^{i\textcolor{blue}{-}}-1  &\text{ if $i=i_j$ and $T_j=T_{\textcolor{blue}{-}}$.}\\
\end{cases}
\end{equation*}
So the expression above becomes
\begin{gather*}
\pmt(t_{i_j}-t_{i_j}^{-1})\cdot (\pmr t_{i_j})^{\lk(t_1,t_{i_j})-2(\beta_{j}^{i_j+}-\beta_{j}^{i_j-})\mpt 1}\cdot \prod_{i\neq 1,i_j} (\pmr t_i)^{\lk(t_1,t_{i})-2(\beta_{j}^{i\textcolor{blue}{+}}-\beta_{j}^{i\textcolor{blue}{-}})}\\
=\pmr \left((\pmr t_{i_j})^{\lk(t_1,t_{i_j})-2(\beta_{j}^{i_j+}-\beta_{j}^{i_j-})}-(\pmr t_{i_j})^{\lk(t_1,t_{i_j})-2(\beta_{j}^{i_j+}-\beta_{j}^{i_j-}\pmt1)}\right)\\
\phantom{\pmt(t_{i_j}-t_{i_j}^{-1})\cdot (\pmr t_{i_j})^{\lk(t_1,t_{i_j})-2(\beta_{j}^{i_j+}-\beta_{j}^{i_j-})\mpt 1}}\cdot \prod_{i\neq 1,i_j} (\pmr t_i)^{\lk(t_1,t_{i})-2(\beta_{j}^{i\textcolor{blue}{+}}-\beta_{j}^{i\textcolor{blue}{-}})}.
\end{gather*}
Now, consider two consecutive boxes, say the $j^\text{th}$ and the $(j+1)^\text{st}$ box. Then
$$\beta_{j}^{i_j+}-\beta_{j}^{i_j-}=\beta_{j+1}^{i_j+}-\beta_{j+1}^{i_j-}\pmt1.$$
and for all other colours, the exponents remain the same. So we see that most terms cancel and the only surviving ones are the second one of the first box and the first one of the last box. But
\begin{equation*}
\beta_{1}^{i\textcolor{blue}{+}}-\beta_{1}^{i\textcolor{blue}{-}}= 
\begin{cases}
\lk(t_1,t_i) &\text{ if $i\neq1, i_1$,}\\
\lk(t_1,t_{i_1})\mpt1  &\text{ if $i=i_1$,}\\
\end{cases}
\end{equation*}
and\vspace{\abovedisplayshortskip}\\
\textcolor{white}{\qedsymbol}\hfill$\beta_{n}^{i\textcolor{blue}{+}}=\beta_{n}^{i\textcolor{blue}{-}}=0 \quad\text{for all $i\neq1$}.$\hfill\qedhere
\end{proof}

\begin{proposition}\label{propexponentsmodtwo}
Given a tangle \(T\), we can compute the powers modulo 2 appearing in the labellings of Kauffman states as follows. If \(t\) is the colour of a closed component in~\(T\), its exponents are equal to \(\lk(t)+1\mod 2\). If it is the colour of an open component, its exponents are equal to 
$$\lk(t)+1+\tfrac{1}{2}(s_i-s_o+r_o-r_u-1)\equiv\lk(t)+1+s_i+r_u\mod 2,$$
where, if we follow the outgoing \(t\)-strand in anticlockwise direction along the boundary of the disc to the other end, \(s_i\) is the number of ingoing strands, \(s_o\) is the number of outgoing strands, \(r_o\) is the number of occupied regions and \(r_u\) is the number of unoccupied regions that we meet along the way.
\end{proposition}
\begin{Remark}
It is easy to see that the formula is symmetric in the sense that it does not matter if we go clock- or anticlockwise. Nonetheless, it would be nice to find a more natural interpretation of this formula. Again, a better understanding of the relationship between sites would probably be useful for this. 
\end{Remark}
\begin{corollary}[reversing orientations revisited]\label{cor:revorientrevisited}
Let \(T\) be an oriented \(r\)-component tangle. If \(\rr(T,t_1)\) denotes the same tangle \(T\) with the orientation of the first component \(K_1\) reversed, then for all sites \(s\in\mathbb{S}(T)\), we have 
$$\nabla^s_{\rr(T,t_1)}(t_1,\dots,t_r)=
\begin{cases}
-\nabla^s_{T}(t_1^{-1},t_2,\dots,t_r)& \text{$K_1$ closed,}\\
(-1)^{2\lk_T(t_1)+1+s_i+r_u} \nabla^s_{T}(t_1^{-1},t_2,\dots,t_r)& \text{$K_1$ open,}\\
\end{cases}
$$
where we use the same notation as in the previous proposition. Similarly, if \(\rr(T)\) denotes the same tangle \(T\) with the orientation of all strands reversed, then for all sites \(s\in\mathbb{S}(T)\), we have $$\nabla^s_{\rr(T)}(t_1,\dots,t_r)=(-1)^{r+c+\sum s_i+\sum r_u}\nabla^s_{T}(t_1^{-1},\dots,t_r^{-1}),$$
where \(c\equiv\sum\lk(t_i)\mod 2\) is the number of interlinked pairs of endpoints of the same colour on the boundary of the disc.
\end{corollary}
\begin{proof}
The first part follows directly from proposition~\ref{prop:reverseorientI} and \ref{propexponentsmodtwo}. For the second part, we apply proposition~\ref{propexponentsmodtwo} to corollary~\ref{coralloorientsrev}. 
\end{proof}
Before we come to the proof of proposition \ref{propexponentsmodtwo}, we do some preparation first.
\begin{lemma}\label{noclosedthennonzero}
Let \(T\) be a tangle with no closed components. Then there exists a site \(s\in\mathbb{S}(T)\) such that \(\nabla_T^s\not\equiv0\). 
\end{lemma}
\begin{proof}
We extend the diagram $T$ to one of a knot by successively connecting two ends of strands with different orientations and colours until only two ends are left. (Note that this process might introduce some more crossings.) Then $\nabla^s$ of this knot will be a linear combination of the $\nabla_T^s$. Since it is also non-zero (say, by corollary~\ref{corknotpmoneequalone}), $\nabla_T^s$ cannot be identically zero for all $s$.
\end{proof}
\begin{lemma}\label{lem:Kauffmanstatesexist}
Let \(T\) be a (connected) tangle diagram (with at least one crossing). Then there exists a site \(s\) such that \(\mathbb{K}(T,s)\) is non-empty.
\end{lemma}
\begin{proof}
As in the previous proof, we close all but two strands of the diagram in some fashion. This new diagram represents a link and, by \cite[lemma~2.2, theorem~2.4]{Kauffman}, it has a Kauffman state. Hence its restriction to the original tangle diagram is also a Kauffman state for some site $s$.
\end{proof}
\begin{lemma}\label{neighbouringsites}
Given a tangle \(T\), let \(x\) and \(x'\) be two Kauffman states whose sites \(s\) and~\(s'\) differ in exactly two adjacent open regions, separated by a $t$-coloured strand.  Then the exponent of \(t\) in the labelling of \(x\) and \(x'\) differs by 1 modulo 2. Exponents of other colours agree modulo 2. 
\end{lemma}
\begin{proof}
As in the proof of lemma~\ref{lem:Powerdistance2}, we would like to have some simple move that we can perform to get from one site to another and that affects the labels in a predictable way and then appeal to some connectedness result. Here, the basic move is the following, we call it the boundary move:
\begin{center}
\begin{pspicture}(-7,-1.2)(7,1)

\rput(-3,0){
\psline[linestyle=dotted,linewidth=1.5pt](-1.5,-1)(-1,-1)
\psline[linestyle=dotted,linewidth=1.5pt](1.5,-1)(1,-1)
\pscustom[linewidth=1.5pt]{\psline[linewidth=1.5pt](-1,-1)(0,-1)
\psline[linewidth=1.5pt](0,-1)(1,-1)}

\psline[linestyle=dotted](0,0.5)(0,1)
\psline[linestyle=dotted](-1,0)(-0.5,0)
\psline[linestyle=dotted](0.5,0)(1,0)
\psline(0,-1)(0,0.5)
\psline(-0.5,0)(0.5,0)

\uput{0.4}[-135](0,0){\psdot[dotsize=6pt](0,0)}
\psarc(0,0){0.4}{225}{-50}
\uput{0.4}[-45](0.05,0){\rput{116}(0,0){\psline{->}(0,0.1)(0,0)}
}

\rput[cr](-2,-1){$\partial D^2$}}

\rput(0,0){$\longleftrightarrow$}

\rput(3,0){
\psline[linestyle=dotted,linewidth=1.5pt](-1.5,-1)(-1,-1)
\psline[linestyle=dotted,linewidth=1.5pt](1.5,-1)(1,-1)
\pscustom[linewidth=1.5pt]{\psline[linewidth=1.5pt](-1,-1)(0,-1)
\psline[linewidth=1.5pt](0,-1)(1,-1)}
\psline[linestyle=dotted](0,0.5)(0,1)
\psline[linestyle=dotted](-1,0)(-0.5,0)
\psline[linestyle=dotted](0.5,0)(1,0)
\psline(0,-1)(0,0.5)
\psline(-0.5,0)(0.5,0)

\uput{0.4}[-45](0,0){\psdot[dotsize=6pt](0,0)}
\psarc(0,0){0.4}{230}{-45}
\uput{0.4}[-135](-0.05,0){\rput{64}(0,0){\psline{->}(0,-0.1)(0,0)}
}

\rput[cl](2,-1){$\partial D^2$}}
\end{pspicture}
\end{center}
Consider two Kauffman states which are related by a single boundary move. The labellings of these two Kauffman states differ by the factor $t^{\pm 1}$. \\
However, such a boundary move might not always be possible. To fix this, we first modify the diagram slightly, namely we perform a Reidemeister II move as follows:
\begin{center}
\begin{pspicture}(-7,-1.2)(7,3)

\rput(-3,0){
\psline[linestyle=dotted,linewidth=1.5pt](-2.5,-1)(-2,-1)
\psline[linestyle=dotted,linewidth=1.5pt](2.5,-1)(2,-1)
\pscustom[linewidth=1.5pt]{\psline[linewidth=1.5pt](-2,-1)(0,-1)
\psline[linewidth=1.5pt](0,-1)(2,-1)}

\psecurve(-1.1,-1.4)(-1.2,-1)(-1.1,-0.6)(-0.5,0.2)(-1.1,0.9)(-1.2,1.2)
\psecurve(1.1,-1.4)(1.2,-1)(1.1,-0.6)(0.5,0.2)(1.1,0.9)(1.2,1.2)

\psecurve[linestyle=dotted](0.45,0.2)(-1.1,0.9)(-1.2,1.2)(-1.7,2.1)(-1.7,7.3)
\psecurve[linestyle=dotted](-0.45,0.2)(1.1,0.9)(1.2,1.2)(1.7,2.1)(1.7,7.3)

\psecurve[linestyle=dotted](-1.4,3.3)(-0.7,2.5)(0,2)(2.2,1.3)(3,1.26)
\psecurve[linestyle=dotted](1.4,3.3)(0.7,2.5)(0,2)(-2.2,1.3)(-3,1.26)
\uput{0.4}[-90](0,2){\psdot[dotsize=6pt](0,0)}

\rput[cr](-3,-1){$\partial D^2$}}

\rput(0,0){$\longrightarrow$}

\rput(3,0){
\psecurve(1.1,-1.4)(1.2,-1)(1.1,-0.6)(-0.45,0.2)(1.1,0.9)(1.2,1.2)
\psecurve[linewidth=4pt,linecolor=white](-1.1,-1.4)(-1.2,-1)(-1.1,-0.6)(0.45,0.2)(-1.1,0.9)(-1.2,1.2)
\psecurve(-1.1,-1.4)(-1.2,-1)(-1.1,-0.6)(0.45,0.2)(-1.1,0.9)(-1.2,1.2)

\psline[linestyle=dotted,linewidth=1.5pt](-2.5,-1)(-2,-1)
\psline[linestyle=dotted,linewidth=1.5pt](2.5,-1)(2,-1)
\pscustom[linewidth=1.5pt]{\psline[linewidth=1.5pt](-2,-1)(0,-1)
\psline[linewidth=1.5pt](0,-1)(2,-1)}

\psecurve[linestyle=dotted](0.45,0.2)(-1.1,0.9)(-1.2,1.2)(-1.7,2.1)(-1.7,7.3)
\psecurve[linestyle=dotted](-0.45,0.2)(1.1,0.9)(1.2,1.2)(1.7,2.1)(1.7,7.3)

\psecurve[linestyle=dotted](-1.4,3.3)(-0.7,2.5)(0,2)(2.2,1.3)(3,1.26)
\psecurve[linestyle=dotted](1.4,3.3)(0.7,2.5)(0,2)(-2.2,1.3)(-3,1.26)

\uput{0.4}[-90](0,2){\psdot[dotsize=6pt](0,0)}

\uput{0.4}[-90](0,-0.03){\psdot[dotsize=6pt](0,0)}

\uput{0.4}[-90](0.6,0.3){\pscircle(0,0){3pt}}

\psecurve[arrows=->](-0.5,0.2)(0.1,-0.43)(0.57,-0.16)(0.65,0.23)

\uput{0.4}[-90](0,0.7){\psdot[dotsize=6pt](0,0)}

\rput[cr](3,-1){$\partial D^2$}}
\end{pspicture}
\end{center}
The labelling of the new Kauffman state corresponding to $x$ as shown in the right picture above is the same as the labelling of $x$ itself. The same is true for the labelling of $x'$. (The additional two markers are then on the ``top'' of the new crossings.) We can now perform a boundary move as indicated by the arrow in the picture on the right. By lemma~\ref{lem:Powerdistance2}, the powers of each colour in the labellings of this Kauffman state and $x'$ agree modulo 2.
\end{proof}
\begin{lemma}\label{lotsofKauffmanstates}
For any (not necessarily connected) tangle diagram \(T\), we can find an equivalent diagram \(T'\) such that for any site \(s\in\mathbb{S}(T)=\mathbb{S}(T')\), \(T'\) has a Kauffman state. Furthermore, the labelling of any Kauffman state of \(T\) agrees with the labelling of any Kauffman state in \(T'\) belonging to the same site, considering any exponents modulo~2.
\end{lemma}
\begin{proof}
If there is no Kauffman state in $T$, we can modify the diagram such that the hypotheses of lemma~\ref{lem:Kauffmanstatesexist} are satisfied. So we may assume without loss of generality that there exists a Kauffman state in $T$ for some site. We now repeatedly apply the same method as in the proof above, noting that a Reidemeister II move merely enlarges the set of Kauffman states and the labellings of corresponding Kauffman states agree.
\end{proof}
\begin{proof}[Proof of proposition~\ref{propexponentsmodtwo}]
Let $s$ be the site of a Kauffman state $x$. We consider closed components first. Suppose that all linking numbers are non-zero. Then the claim follows from proposition~\ref{propsetpmone} and lemma~\ref{lem:Powerdistance2} if $\nabla^s$ of the tangle $T$ with all closed components removed is not identically zero. By lemma~\ref{noclosedthennonzero}, we can find at least one site $s'$ of this tangle, such that for any Kauffman state in $\mathbb{K}(T,s')$, the claim is true. By lemma~\ref{lotsofKauffmanstates}, we may assume that there is a Kauffman state for each site. Then by lemma~\ref{neighbouringsites}, we know that in particular the exponents of the colours of closed components agree for all sites. \\
The general case, where linking numbers may be zero, can be reduced to the first case. For this, we modify the diagram as follows: Consider two strands whose mutual linking number is zero. If they meet at a crossing, we can reverse the over- and under-strands at one such crossing. This changes both the linking number and the exponents of both colours by $\pm1$. If two components have no crossing in common, we perform some Reidemeister II moves to pull one strand across the other to get one. Note that this does not affect the labelling of the Kauffman states (as in the previous proof). This finishes the proof for closed components.\\
For open components, the idea is to add twists at the outgoing strand until it bounds the same open region as the incoming strand, so that for some sites, we can close this component and apply the result above. We use the convention that the $t$-strand always goes over its left neighbour. We can then distinguish the cases shown in figure~\ref{fig:twistcasesforprop}, depending on the orientation of the other strand involved and on whether the open region between these two strands is occupied by a marker or not.

\begin{figure}[t]
\psset{unit=1}
\begin{pspicture}(-7,-2.75)(7,1.4)
\rput(-5.1,0){
\pscustom{
\psline(-0.9,-2.1)(-0.9,-0.9)
\psline{->}(-0.9,-0.9)(0.9,0.9)}
\pscircle*[linecolor=white](0,0){0.3}
\pscustom{
\psline(0.9,-2.1)(0.9,-0.9)
\psline{->}(0.9,-0.9)(-0.9,0.9)}
\psecurve[linecolor=gray](-3.5,-2)(-1.3,-1.5)(1.3,-1.5)(3.5,-2)
\rput(1.4,-1.2){$\partial D^2$}
\uput{0.6}[90](0,0){\psdot[dotsize=6pt](0,0)}
\uput{1}[90](0,0){\rput(0,0){$t^{-\frac{1}{2}}$}}
\rput(0,-1.8){occupied}
\rput(-1.1,1.1){$t$}
\rput(0,-2.5){negative crossing}}

\rput(-1.7,0){
\pscustom{
\psline{->}(0.9,0.9)(-0.9,-0.9)
\psline(-0.9,-0.9)(-0.9,-2.1)}
\pscircle*[linecolor=white](0,0){0.3}
\pscustom{
\psline(0.9,-2.1)(0.9,-0.9)
\psline{->}(0.9,-0.9)(-0.9,0.9)}
\psecurve[linecolor=gray](-3.5,-2)(-1.3,-1.5)(1.3,-1.5)(3.5,-2)
\rput(1.4,-1.2){$\partial D^2$}
\uput{0.6}[90](0,0){\psdot[dotsize=6pt](0,0)}
\uput{1}[90](0,0){\rput(0,0){$t^{-\frac{1}{2}}$}}
\rput(0,-1.8){occupied}
\rput(-1.1,1.1){$t$}
\rput(0,-2.5){positive crossing}}

\rput(1.7,0){
\pscustom{
\psline(-0.9,-2.1)(-0.9,-0.9)
\psline{->}(-0.9,-0.9)(0.9,0.9)}
\pscircle*[linecolor=white](0,0){0.3}
\pscustom{
\psline(0.9,-2.1)(0.9,-0.9)
\psline{->}(0.9,-0.9)(-0.9,0.9)}
\psecurve[linecolor=gray](-3.5,-2)(-1.3,-1.5)(1.3,-1.5)(3.5,-2)
\rput(1.4,-1.2){$\partial D^2$}
\uput{0.6}[-90](0,0){\psdot[dotsize=6pt](0,0)}
\uput{1}[-90](0,0){\rput(0,0){$t^{\frac{1}{2}}$}}
\rput(0,-1.8){unocc.}
\rput(-1.1,1.1){$t$}
\rput(0,-2.5){negative crossing}}

\rput(5.1,0){
\pscustom{
\psline{->}(0.9,0.9)(-0.9,-0.9)
\psline(-0.9,-0.9)(-0.9,-2.1)}
\pscircle*[linecolor=white](0,0){0.3}
\pscustom{
\psline(0.9,-2.1)(0.9,-0.9)
\psline{->}(0.9,-0.9)(-0.9,0.9)}
\psecurve[linecolor=gray](-3.5,-2)(-1.3,-1.5)(1.3,-1.5)(3.5,-2)
\rput(1.4,-1.2){$\partial D^2$}
\uput{0.6}[-90](0,0){\psdot[dotsize=6pt](0,0)}
\uput{1}[-90](0,0){\rput(0,0){$t^{\frac{1}{2}}$}}
\rput(0,-1.8){unocc.}
\rput(-1.1,1.1){$t$}
\rput(0,-2.5){positive crossing}}

\end{pspicture}
\caption{Four cases distinguished in the proof of proposition~\ref{propexponentsmodtwo} }\label{fig:twistcasesforprop}
\end{figure}
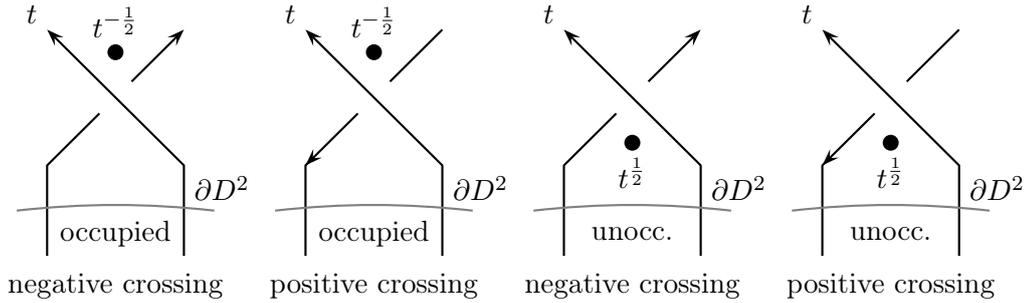
\noindent
Now, suppose the open region on the left of the incoming $t$-strand is occupied. Then we can close the $t$-component and apply the result above. The exponent of $t$ in the original diagram is equal to $\lk_\text{new}(t)+1+\tfrac{1}{2}((r_o-1)-r_u)$, where $\lk_\text{new}(t)$ denotes the linking number of the closed $t$-component in the new diagram. Suppose the open region on the left of the incoming $t$-strand is unoccupied, so we cannot close the $t$-component. One way to solve this problem is by ``pushing'' a marker of an open region into this region, using lemmas~\ref{lotsofKauffmanstates} and~\ref{neighbouringsites}. However, this does not work for $2$-ended tangles where we do not have any markers to ``push around''. 
Although one can, of course, look at this case separately, we give an alternative argument that works in general: We still close the $t$-component, but instead of glueing a single strand to the diagram, we attach a diagram that consists of two parallel strands that are twisted by a Reidemeister II move. By considering $\nabla$ for the attached two-crossing diagram, we see that the exponent of $t$ in the original Kauffman state is given by $\lk_\text{new}(t)+1+\tfrac{1}{2}(r_o-(r_u-1))+1\mod 2$, i.e. the same formula as in the first case.\\
So it remains to calculate $\lk_\text{new}(t)$. We have
$\lk_\text{new}(t)=\lk_T(t)+\tfrac{1}{2}(s_i-s_o)$. Substituting this into the formula above gives the first formula in the statement of the proposition. The equivalence to the second one is seen by substituting the obvious identity
$s_i+s_o=r_o+r_u-1$ into the first formula (and adding $2r_u$).
\end{proof}

%% file: sections/1_FourEndedTangles.tex
\section{4-ended tangles and mutation invariance}\label{sec:4endedandmutation}

\begin{proposition}\label{prop:fourended}
Let \(T\) be a 4-ended tangle. Then the endpoints of the same strands are either neighbours on \(\partial D^2\) (case I) or not (case II). We consider the following orientations for these two cases:
$$
\begin{pspicture}(-2.6,-1.5)(2.6,1.05)
\rput(-1.5,0){
\SpecialCoor
\psline[linecolor=violet]{->}(0.6;45)(1;45)
\psline[linecolor=violet]{<-}(0.6;-45)(1;-45)

\psline[linecolor=darkgreen]{->}(0.6;135)(1;135)
\psline[linecolor=darkgreen]{<-}(0.6;-135)(1;-135)

\uput{0.6}[0](0,0){$c$}
\uput{0.6}[90](0,0){$d$}
\uput{0.6}[180](0,0){$a$}
\uput{0.6}[270](0,0){$b$}

\uput{1.1}[135](0,0){$\textcolor{darkgreen}{p}$}
\uput{1.1}[-135](0,0){$\textcolor{darkgreen}{p}$}
\uput{1.1}[45](0,0){$\textcolor{violet}{q}$}
\uput{1.1}[-45](0,0){$\textcolor{violet}{q}$}
\pscircle[linestyle=dotted](0,0){1}

\rput(0,-1.3){\textit{case I}}
}

\rput(1.5,0){
\SpecialCoor
\psline[linecolor=violet]{->}(0.6;45)(1;45)
\psline[linecolor=violet]{<-}(0.6;-135)(1;-135)

\psline[linecolor=darkgreen]{->}(0.6;135)(1;135)
\psline[linecolor=darkgreen]{<-}(0.6;-45)(1;-45)

\uput{0.6}[0](0,0){$c$}
\uput{0.6}[90](0,0){$d$}
\uput{0.6}[180](0,0){$a$}
\uput{0.6}[270](0,0){$b$}

\uput{1.1}[135](0,0){$\textcolor{darkgreen}{p}$}
\uput{1.1}[-45](0,0){$\textcolor{darkgreen}{p}$}
\uput{1.1}[45](0,0){$\textcolor{violet}{q}$}
\uput{1.1}[-135](0,0){$\textcolor{violet}{q}$}
\pscircle[linestyle=dotted](0,0){1}
\rput(0,-1.3){\textit{case II}}
}

\end{pspicture}
$$
Then in both cases, 
\begin{equation}
\nabla_T^{b}=\nabla_{\rr(T)}^{d}.\label{eqn:bd}\tag{b-d}
\end{equation}
Furthermore, in case I
\begin{align}
(p-p^{-1})\cdot\nabla_T^{a}&=(q-q^{-1})\cdot\nabla_T^{c}, \tag{I a-c}\label{eqn:Iac}\\
\nabla_T^{c}&=\nabla_{\rr(T)}^{c} \text{ and }\tag{I c-c}\label{eqn:Icc}\\
(pq-p^{-1}q^{-1})\cdot\nabla_T^{c}&=
(p-p^{-1})\left(\nabla_T^{d}-\nabla_{\rr(T)}^{d}\right)\tag{I c-d}\label{eqn:Icd}
\end{align}
and in case II
\begin{align}
\nabla_T^{a}&=\nabla_{\rr(T)}^{c},\tag{II a-c}\label{eqn:IIac}\\
(pq-p^{-1}q^{-1}) \cdot\nabla_T^{c}
&= 
(p-p^{-1})\cdot\nabla_T^{d}-(q-q^{-1})\cdot\nabla_{\rr(T)}^{d}\text{ and }\tag{II c-d}\label{eqn:IIcd}\\
(p^{-1}q-pq^{-1}) \cdot\nabla_T^{d}
&= (q-q^{-1})\cdot\nabla_T^{a}-(p-p^{-1})\cdot\nabla_{\rr(T)}^{a}.\tag{II d-a}\label{eqn:IIda}
\end{align}
For orientations different from the above, we get similar relations, which we can easily compute using proposition~\ref{prop:reverseorientI} and \ref{propexponentsmodtwo}. 
\end{proposition}
\pagebreak[3]

\begin{corollary}\label{cor:fourendedonecolour}
Let $T$ be a 4-ended tangle. We distinguish between the following two cases:
$$
\begin{pspicture}(-4.05,-1.5)(4.05,1.05)
\rput(1.5,0){
\SpecialCoor
\psline{<-}(0.6;45)(1;45)
\psline{->}(0.6;-45)(1;-45)

\psline{->}(0.6;135)(1;135)
\psline{<-}(0.6;-135)(1;-135)

\uput{1.1}[135](0,0){$t$}
\uput{1.1}[-135](0,0){$t$}
\uput{1.1}[45](0,0){$t$}
\uput{1.1}[-45](0,0){$t$}

\uput{0.6}[0](0,0){$c$}
\uput{0.6}[90](0,0){$d$}
\uput{0.6}[180](0,0){$a$}
\uput{0.6}[270](0,0){$b$}

\pscircle[linestyle=dotted](0,0){1}

\rput(0,-1.3){\textit{case ii}}
}

\rput(-1.5,0){
\SpecialCoor
\psline{->}(0.6;45)(1;45)
\psline{<-}(0.6;-135)(1;-135)

\psline{->}(0.6;135)(1;135)
\psline{<-}(0.6;-45)(1;-45)

\uput{1.1}[135](0,0){$t$}
\uput{1.1}[-135](0,0){$t$}
\uput{1.1}[45](0,0){$t$}
\uput{1.1}[-45](0,0){$t$}

\uput{0.6}[0](0,0){$c$}
\uput{0.6}[90](0,0){$d$}
\uput{0.6}[180](0,0){$a$}
\uput{0.6}[270](0,0){$b$}
\pscircle[linestyle=dotted](0,0){1}
\rput(0,-1.3){\textit{case i}}
}
\end{pspicture}
$$
We write \(\nabla^{a}=\nabla^{a}_T(t,t,t_1,\dots,t_{r-2})\), \(\nabla^{a}_{\rr}=\nabla^{a}_{\rr(T)}(t,t,t_1,\dots,t_{r-2})\) and similarly for \(b\), \(c\) and~\(d\). Then in both cases, we have \(\nabla^{c}=\nabla^{a}=\nabla^{a}_{\rr}\), \(\nabla^{b}=\nabla^{d}_{\rr}\) and \((t+t^{-1})\nabla^{c}=\nabla^{d}-\nabla^{d}_{\rr}\), so
\begin{equation*}
t\nabla^{a}+\nabla^{b}+t^{-1}\nabla^{c}=\nabla^{d}.
\end{equation*}
In case ii, we additionally get \(\nabla^{b}=\nabla^{d}\).
\end{corollary}
\begin{proof}
The first identity is seen by closing $a$ or $c$. The second follows from the properties of the Conway potential function, theorem \ref{thm:CPFprops}(\ref{CPFpropssymmetry}) and (\ref{CPFpropRevOneStrand}). The third and fourth follow from proposition \ref{prop:fourended}. The last relation is seen again by closing $b$ or $d$.
\end{proof}

\begin{figure}[t]
\centering
\psset{unit=0.6, yunit=0.9}
\begin{pspicture}[showgrid=false](-4,-3.7)(2.6,3.7)

\rput(-0.7,0){
\pscircle[linestyle=dotted](0,0){3.3}
\rput(-2.8,0){$a$}
\rput(0,-2.8){$b$}
\rput(2.8,0){$c$}
\rput(0,2.8){$d$}
}

\psecurve[linecolor=darkgreen](-3,3.4)(-2.5,3)(-2,2)(-2,-2)(-2.5,-3)(-3,-3.4)

\pscircle*[linecolor=white](-2,2){0.2}
\pscircle*[linecolor=white](-2,-2){0.2}

\psecurve[linecolor=violet](-2,2)(0,2)(0.75,1)(-0.75,-1)(0,-2)(1.1,-3)(1.3,-5)

\psecurve[linecolor=violet](0,2)(-0.75,1)(0.75,-1)(0,-2)(-2,-2)(-3,0)(-2,2)(0,2)(0.75,1)

\pscircle*[linecolor=white](0,2){0.2}
\pscircle*[linecolor=white](0,0){0.2}
\pscircle*[linecolor=white](0,-2){0.2}

\psecurve[linecolor=violet](0,2)(0.75,1)(-0.75,-1)(0,-2)
\psecurve[linecolor=violet](1.3,5)(1.1,3)(-0.75,1)(0.75,-1)
\psecurve[linecolor=violet](-0.75,1)(0.75,-1)(0,-2)(-2,-2)(-3,0)

\end{pspicture}
\caption{The trefoil knot with an additional strand, illustrating remark~\ref{rem:nobetter4endedrelation}}\label{fig:nobetter4endedrelation}
\end{figure}
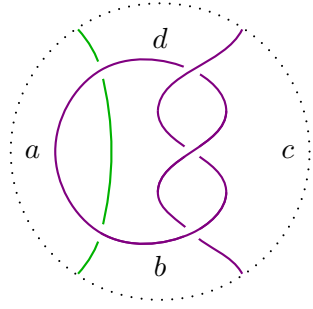
\begin{Remark}\label{rem:nobetter4endedrelation}
Proposition~\ref{prop:fourended} tells us that there is basically only one piece of information in the polynomial Alexander invariant of 4-ended tangles. One might ask whether all four invariants in case I contain the same information like in case II. However, this is not the case: For any knot $K$, we can consider the two-component tangle diagram obtained from $K$ by cutting it once to get one strand of the tangle, adding an unknotted strand to it and performing a Reidemeister II move to get rid of any multiple regions on the boundary, like in figure~\ref{fig:nobetter4endedrelation} for $K=\text{trefoil knot}$. 
Then both $a$ and $c$ are zero because closing either of these regions results in a link diagram with two unlinked components. However, for matching orientations of the strands, $b$ and $d$ both give the Alexander polynomial $\Delta_K$ of the knot $K$.
\end{Remark}
\begin{proof}
Relation \eqref{eqn:Iac} follows from theorem~\ref{thm:twoended} and the fact that the two diagrams obtained by closing either the $p$-component or the $q$-component both represent the same link. Next, using the notation from corollary \ref{cor:recoverlinkcase} and the corollary itself, we obtain
$$\rr\left(\frac{\nabla_T^{c}}{p-p^{-1}}\right)=\frac{\nabla_T^{c}}{p-p^{-1}}.$$
so we immediately get \eqref{eqn:Icc}. \pagebreak\\
Next, we compute the invariants for the following tangle~$U$:

\begin{center}
\psset{unit=0.7}
\begin{pspicture}(-8.5,-1.4)(8.5,3.2)
\rput[c](5.7,0.9){\parbox{4.5cm}{$\nabla_U^{A}= (q-q^{-1})$\\
$\nabla_U^{B}= p^{-1}q^{-1}$\\
$\nabla_U^{C}= (p-p^{-1})$\\
$\nabla_U^{D}= pq$}}

\rput(0,0){
\rput(0,1.8){
\psline[linecolor=darkgreen]{->}(0.9,-0.9)(-0.9,0.9)
\pscircle*[linecolor=white](0,0){0.3}
\psline[linecolor=violet]{->}(-0.9,-0.9)(0.9,0.9)
\uput{0.6}[90](0,0){$D$}
}
\psecurve[linestyle=dotted](-0.9,-0.9)(0.9,-0.9)(0.9,2.7)(-0.9,2.7)(-0.9,-0.9)(0.9,-0.9)(0.9,2.7)

\psline[linecolor=violet]{->}(0.9,-0.9)(-0.9,0.9)
\pscircle*[linecolor=white](0,0){0.3}
\psline[linecolor=darkgreen]{->}(-0.9,-0.9)(0.9,0.9)
\uput{0.6}[-90](0,0){$B$}

\rput(-1.2,0.9){$A$}
\rput(1.2,0.9){$C$}

\rput(0,1.8){
\uput{1.35}[135](0,0){$\textcolor{darkgreen}{p}$}
\uput{1.35}[45](0,0){$\textcolor{violet}{q}$}}
\uput{1.35}[-135](0,0){$\textcolor{darkgreen}{p}$}
\uput{1.35}[-45](0,0){$\textcolor{violet}{q}$}
}
\end{pspicture}
\end{center}
We observe that the diagrams obtained by glueing the above diagram either to the top or to the bottom  of $T$ and closing the $q$-component represent the same link. In the first case, the corresponding polynomial is given by $\nabla_U^{C}\cdot\nabla_T^{d}+\nabla_U^{B}\cdot\nabla_T^{c}$
and the second, by $\nabla_U^{C}\cdot\nabla_T^{b}+\nabla_U^{D}\cdot\nabla_T^{c}$. Again, corollary \ref{cor:recoverlinkcase} implies
\begin{equation}\label{eqncomb}
\rr\left(\frac{\nabla_U^{C}\,\nabla_T^{d}+\nabla_U^{B}\,\nabla_T^{c}}{p-p^{-1}}\right)=\frac{\nabla_U^{C}\,\nabla_T^{d}+\nabla_U^{B}\,\nabla_T^{c}}{p-p^{-1}}=\frac{\nabla_U^{C}\,\nabla_T^{b}+\nabla_U^{D}\,\nabla_T^{c}}{p-p^{-1}},
\end{equation}
After some simplification, we get
\begin{eqnarray*}
\rr(\nabla_U^{C})\cdot\nabla_{\rr(T)}^{d}+\rr(\nabla_U^{B})\cdot\nabla_{\rr(T)}^{c}&=&\nabla_U^{C}\cdot\nabla_T^{b}+\nabla_U^{D}\cdot\nabla_T^{c}
\end{eqnarray*}
Finally, using $\rr(\nabla_U^{B})=\nabla_U^{D}$ and $\rr(\nabla_U^{C})=\nabla_U^{C}$, we obtain the desired identities \eqref{eqn:bd} for case I. For \eqref{eqn:Icd}, we just use the first half of equations (\ref{eqncomb}) and, after simplification, we get
$$\rr(\nabla_U^{C})\cdot\nabla_{\rr(T)}^{d}+\rr(\nabla_U^{B})\cdot\nabla_{\rr(T)}^{c}=\nabla_U^{C}\cdot\nabla_T^{d}+\nabla_U^{B}\cdot\nabla_T^{c}.$$
Substituting $\nabla_{\rr(T)}^{c}$, $\rr(\nabla_U^{B})$ and $\rr(\nabla_U^{C})$ gives the desired identity.\\
By a similar method, we can derive the relations for case II. We first show the second relation. For this, we apply the corresponding statement from case I to the diagram obtained by glueing the following positive twist to the bottom of $T$:
$$
\begin{pspicture}(-1.45,-1.35)(1.45,1.35)
\psline[linecolor=violet]{->}(0.9,-0.9)(-0.9,0.9)
\pscircle*[linecolor=white](0,0){0.3}
\psline[linecolor=darkgreen]{->}(-0.9,-0.9)(0.9,0.9)
\uput{0.6}[90](0,0){$\textcolor{darkgreen}{p}^{\frac{1}{2}} \textcolor{violet}{q}^{\frac{1}{2}}$}
\uput{0.3}[180](0,0){$\textcolor{darkgreen}{p}^{\frac{1}{2}} \textcolor{violet}{q}^{-\frac{1}{2}}$}
\uput{0.75}[-90](0,0){$-\textcolor{darkgreen}{p}^{-\frac{1}{2}} \textcolor{violet}{q}^{-\frac{1}{2}}$\,\,}
\uput{0.3}[0](0,0){$\textcolor{darkgreen}{p}^{-\frac{1}{2}} \textcolor{violet}{q}^{\frac{1}{2}}$}
\end{pspicture}
$$
Then, we have 
$$-p^{-\frac{1}{2}}q^{-\frac{1}{2}}\nabla_T^{b}=\rr\left(p^{\frac{1}{2}}q^{\frac{1}{2}}\nabla_T^{d}\right),$$
which is \eqref{eqn:bd}. \eqref{eqn:IIcd} follows similarly; we have
 $$(pq-p^{-1}q^{-1})\cdot\left(p^{\frac{1}{2}}q^{\frac{1}{2}}\nabla_T^{c}+p^{-\frac{1}{2}}q^{\frac{1}{2}}\nabla_T^{b}\right)=
(p-p^{-1})\left(p^{\frac{1}{2}}q^{\frac{1}{2}}\nabla_T^{d}-\rr\left(p^{\frac{1}{2}}q^{\frac{1}{2}}\nabla_T^{d}\right)\right).$$
Substituting $\nabla_T^{b}$ yields
\begin{eqnarray*}
(pq-p^{-1}q^{-1})\nabla_T^{c}&=&
(p-p^{-1})\nabla_T^{d}+(q^{-1}-p^{-2}q^{-1})\cdot\nabla_{\rr(T)}^{d}-(q-p^{-2}q^{-1})\cdot\nabla_{\rr(T)}^{d}\\
&=&
(p-p^{-1})\nabla_T^{d}-(q-q^{-1})\nabla_{\rr(T)}^{d}.
\end{eqnarray*}
For \eqref{eqn:IIac}, we compare the two diagrams obtained by glueing a single positive twist either to the top or to the bottom of $T$ and closing the $q$-strand. \pagebreak[3]\\
We obtain 
$$\rr\left(\frac{-p^{-\frac{1}{2}}q^{-\frac{1}{2}}\nabla_T^{a}+p^{\frac{1}{2}}q^{-\frac{1}{2}}\nabla_T^{d}}{p-p^{-1}}\right)=\frac{p^{\frac{1}{2}}q^{\frac{1}{2}}\nabla_T^{c}+p^{-\frac{1}{2}}q^{\frac{1}{2}}\nabla_T^{b}}{p-p^{-1}},$$
This simplifies to
$$p^{\frac{1}{2}}q^{\frac{1}{2}}\nabla_{\rr(T)}^{a}+p^{-\frac{1}{2}}q^{\frac{1}{2}}\nabla_{\rr(T)}^{d}=p^{\frac{1}{2}}q^{\frac{1}{2}}\nabla_T^{c}+p^{-\frac{1}{2}}q^{\frac{1}{2}}\nabla_T^{b}$$
Using \eqref{eqn:bd}, the result follows. Finally, relation \eqref{eqn:IIda} is obtained from the previous one using proposition~\ref{prop:reverseorientI} and \ref{propexponentsmodtwo}. We reverse the $p$-strand, so we get
\begin{eqnarray*}
(pq^{-1}-p^{-1}q) \cdot \nabla_{\rr(T,p)}^{c}
&=&
(p-p^{-1})\cdot \nabla_{\rr(T,p)}^{d}-(q-q^{-1})\cdot\nabla_{\rr(\rr(T),p)}^{d}.
\end{eqnarray*}
We observe that $r(T,p)$ also belongs to case II; however, to get the same configuration as in the proposition, we have to rotate it by $90^\circ$ anticlockwise and switch $p$ and $q$. Doing this for the identity above gives us the required result.
\end{proof}
\begin{figure}[b]
\centering
\psset{unit=0.45}
\begin{pspicture}[showgrid=false](-7,-4.3)(7,5.4)

\psline[linestyle=dotted,linewidth=2pt,arrows=->](-4.5,0)(4.5,0)
\psline[linestyle=dotted,linewidth=2pt,arrows=->](0,-4.5)(0,4.5)

{\psset{linewidth=1pt,linecolor=gray}
\psline(1.2,1.2)(3.3,3.3)
\psline(-1.2,-1.2)(-3.3,-3.3)
\psline(1.2,-1.2)(3.3,-3.3)
\psline(-1.2,1.2)(-3.3,3.3)
}

\pscircle[linestyle=dotted](0,0){3}

\rput[l](4.6,0){$x$-axis}
\rput[b](0,4.6){$y$-axis}

\rput[br](-0.3,0.3){\psframebox[framesep=0pt,linecolor=white, fillstyle=solid,fillcolor=white]{$z$-axis}}

\rput(-0.8,-2){$T_\circ$}

\rput(-3.3,-1){$A$}
\rput(1,-3.3){$B$}
\rput(3.3,1){$C$}
\rput(-1,3.3){$D$}

\psdot[dotsize=11pt](0,0)
\psdot[dotsize=9pt,linecolor=white](0,0)
\psdot[dotsize=6pt](0,0)
\end{pspicture}
\caption{The three mutation axes}\label{fig:MutationAxes}
\end{figure}
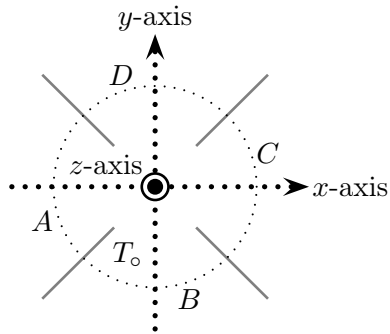
\begin{definition}
Let $T$ be a tangle and $T_\circ$ a 4-ended tangle obtained by intersecting $T$ with a closed 3-ball $B^3$. We may assume that all four tangle ends of $T_\circ$ lie equally spaced on a great circle on $\partial B^3$. 
Let $T'$ be the tangle obtained from~$T$ by rotation of $T_\circ$ by $\pi$ about one of the three axes that switch pairs of endpoints of~$T_\circ$ as shown in figure~\ref{fig:MutationAxes}. We say $T'$ is obtained from $T$ by \textbf{mutation} or $T'$ is a \textbf{mutant} of $T$. $T_\circ$ is called the \textbf{mutating tangle}. 
If $T$ is oriented, we choose an orientation of $T'$ that agrees with the one for $T$ outside of $B^3$. If this means that we need to reverse the orientation of the two open components of $T_\circ$ then we also reverse the orientation of all other components of~$T_\circ$ in $T'$; otherwise we do not change any orientation.
\end{definition}
\begin{Remark}
The definition above is equivalent to the one given in the introduction for links (definition~\ref{def:INTROmutation}). This can be easily seen by twisting the ends of a mutating tangle, see remark~\sref{Rem:mutationconj}.
\end{Remark}
\begin{theorem}\label{thm:mutation}
Let \(T\) be an oriented tangle and \(T'\) a mutant of \(T\). Suppose the colours of the two open strands of the mutating tangle agree. Then for all sites \(s\in\mathbb{S}(T)=\mathbb{S}(T')\),
$$\nabla_T^s=\nabla_{T'}^s.$$
\end{theorem}
\begin{corollary}
The multivariate Alexander polynomial is mutation invariant, provided that the open strands of the mutating tangle have the same colour.
\hfill$\square$
 \end{corollary}
\begin{Remark}
It is quite easy to find a counterexample for a potential symmetry relation $\nabla_{\rr(T,t)}^s=\pm\nabla_{T}^s,$ where $t$ is a closed strand. For example, let
$$
T=
\psset{unit=0.37}
\raisebox{-1.4cm}{
\begin{pspicture}(-4,-4)(4,4)

\pscircle[linestyle=dotted](0,0){4}

\psecurve[linecolor=darkgreen,arrows=<-](6;135)(4;135)(-2,1)(-1.85,0)(-2,-1)
\psecurve[linecolor=darkgreen,arrows=->](2,1)(1.85,0)(2,-1)(4;-45)(6;-45)
\psecurve(-2,1)(0,1.8)(2,1)(2.4,0)(2,-1)(0,-1.8)(-2,-1)(-2.4,0)(-2,1)(0,1.8)(2,1)(2.4,0)(2,-1)

\pscircle*[linecolor=white](2,1){0.4}
\pscircle*[linecolor=white](2,-1){0.4}
\pscircle*[linecolor=white](-2,-1){0.4}
\pscircle*[linecolor=white](-2,1){0.4}

\psecurve[linecolor=darkgreen](-2,1)(-1.85,0)(-2,-1)(4;-135)(6;-135)
\psecurve[linecolor=darkgreen](6;45)(4;45)(2,1)(1.85,0)(2,-1)
\psecurve{->}(2,1)(2.4,0)(2,-1)(0,-1.8)(-2,-1)
\psecurve{->}(-2,-1)(-2.4,0)(-2,1)(0,1.8)(2,1)

\rput[c](-3,0){$a$}
\rput[c](0,-3){$b$}
\rput[c](3,0){$c$}
\rput[c](0,3){$d$}

\rput[c](4.6;135){$\textcolor{darkgreen}{p}$}
\rput[c](4.6;-45){$\textcolor{darkgreen}{p}$}
\rput[c](0,1){$r$}
\end{pspicture}}
$$
Then $\nabla_T^b=\textcolor{darkgreen}{p}^{2} r  - (r-r^{-1})- \textcolor{darkgreen}{p}^{-2} r^{-1}$. So in general, the theorem and its corollary above become false if we do not reverse the orientation of \textit{all} strands in the tangle $T_\circ$ when the orientation of the open strands needs to be reversed for a mutation.
\end{Remark}

\begin{proof}[Proof of theorem \ref{thm:mutation}]
We consider the same two cases as in corollary \ref{cor:fourendedonecolour}, and also use its notation. Denote by $A$, $B$, $C$ and $D$ the Alexander polynomials of the corresponding counterparts of the sites $a$, $b$, $c$ and $d$ in $T\smallsetminus T_\circ$, such that
$$\nabla_T^{s}=A\nabla^{a}+B\nabla^{b}+C\nabla^{c}+D\nabla^{d}.$$  
If we rotate about the $x$-axis, we have to reverse orientations in both cases of corollary~\ref{cor:fourendedonecolour}, so
$$\nabla_{T'}^s=A\nabla^{a}_{\rr}+B\nabla^{d}_{\rr}+C\nabla^{c}_{\rr}+D\nabla^{b}_{\rr}=A\nabla^{a}+B\nabla^{b}+C\nabla^{c}+D\nabla^{d}=\nabla_T^{s}.$$
Next, let us consider rotations about the $y$-axis. In case i, we do not need to reverse orientations. We have
$$\nabla_{T'}^s=A\nabla^{c}+B\nabla^{b}+C\nabla^{a}+D\nabla^{d}=A\nabla^{a}+B\nabla^{b}+C\nabla^{c}+D\nabla^{d}=\nabla_T^{s}.$$
In case ii, we need to reverse orientations:
$$\nabla_{T'}^s=A\nabla^{c}_{\rr}+B\nabla^{b}_{\rr}+C\nabla^{a}_{\rr}+D\nabla^{d}_{\rr}
=A\nabla^{a}+B\nabla^{b}+C\nabla^{c}+D\nabla^{d}=\nabla_T^{s}$$
Finally, note that rotation about the $z$-axis is the same as rotation about both the $x$- and the $y$-axis (in any order), so we are done.
\end{proof}

%% file: sections/1_Geometric.tex
\section{\texorpdfstring{Geometric interpretation of $\nabla_T^s$}{Geometric interpretation of ∇}}\label{sec:geometricinterpretation}

Recall that the classical Alexander polynomial of a knot or link $L$ can be defined as follows: Let $\tilde{X}_L$ denote the maximal Abelian cover of the link complement $X_L=S^3\smallsetminus L$. We have an action of $H_1(X_L)$ on $\tilde{X}_L$. Thus, we can regard $H_1(\tilde{X}_L)$ as a module over the group ring of $H_1(X_L)$. Then, the Alexander polynomial is the determinant of any square presentation matrix of $H_1(\tilde{X}_L)$.\\
In this section, we will give a similar geometric interpretation of our polynomial tangle invariants. But first, we do some basic calculations.
\begin{lemma}
Let \(T\) be a tangle in \(B^3\) with \(n\) open and \(m\) closed components. Then \(H_1(B^3\smallsetminus T) \linebreak[2]\cong \mathbb{Z}^{n+m}\) is freely generated by the meridians of the tangle components and \(H_2(B^3\smallsetminus T)\cong \mathbb{Z}^{m}\) is freely generated by the boundaries of tubular neighbourhoods of the closed tangle components.
\end{lemma}
\begin{proof}
The Mayer-Vietoris sequence for the decomposition $B^3=(B^3\smallsetminus T)\cup \nu(T)$, $\nu(T)$ being a tubular neighbourhood of $T$, gives isomorphisms 
$$H_\ast((\amalg_n S^1)\amalg(\amalg_m T^2))\cong H_\ast(\partial\nu(T))\rightarrow H_\ast(B^3\smallsetminus T)\oplus H_\ast(\nu(T)) \text{\quad for }\ast>0.$$
Note that $\nu(T)\simeq (\amalg_n B^3)\amalg(\amalg_m S^1\times D^2)$  and the longitude of any torus $T^2$ goes to the generator of the corresponding $H_1(S^1\times D^2)$.
\end{proof}
Next, we explicitly calculate the cellular chain complex of the tangle complement $X_T=B^3\smallsetminus T$ from a fixed tangle diagram $T$ by considering the following handle decomposition: We start with two 0-handles, one sitting below and the other above the diagram. For each region, take a 1-handle from one 0-handle to the other. To simplify arguments below, we orient these 1-handles as follows: Choose a chequerboard colouring of the diagram and orient the 1-handles corresponding to one colour in one direction and the others in the opposite direction. Finally, for each crossing, attach a 2-handle to the 1-handlebody as illustrated in figure \ref{fig:attachingtwohandles}, where $a$, $b$, $c$ and $d$ denote 1-handles. 
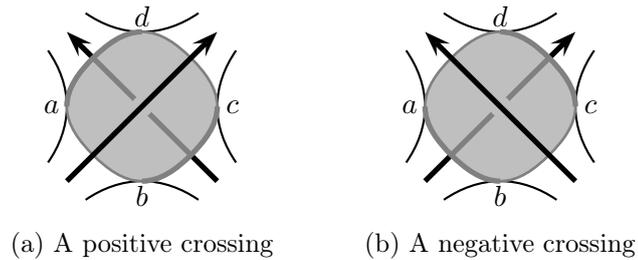
\begin{figure}[t]
	\psset{unit=0.7}
\centering
\begin{subfigure}[b]{0.3\textwidth}\centering
\begin{pspicture}(-3,-2)(3,2)
\rput{45}(0,0){

\psrotate(0,0){0}{\psecurve(0.6,3)(0.5,2)(1,1)(2,0.5)(3,0.6)}
\psrotate(0,0){90}{\psecurve(0.6,3)(0.5,2)(1,1)(2,0.5)(3,0.6)}
\psrotate(0,0){180}{\psecurve(0.6,3)(0.5,2)(1,1)(2,0.5)(3,0.6)}
\psrotate(0,0){-90}{\psecurve(0.6,3)(0.5,2)(1,1)(2,0.5)(3,0.6)}

\psline[linewidth=2pt]{->}(0,-2)(0,2)
\pscurve*[linecolor=lightgray,linewidth=2pt](0,1.3)(1,1)(1.3,0)(1,-1)(0,-1.3)(-1,-1)(-1.3,0)(-1,1)(0,1.3)
\psline[linewidth=2pt, linecolor=gray](0,-1.3)(0,1.3)
\pscircle*[linecolor=lightgray](0,0){0.2}

\psrotate(0,0){90}{\psecurve[linecolor=gray,linewidth=2pt](0,1.3)(1,1)(1.3,0)(1,-1)(0,-1.3)}
\psrotate(0,0){-90}{\psecurve[linecolor=gray,linewidth=2pt](0,1.3)(1,1)(1.3,0)(1,-1)(0,-1.3)}
\psrotate(0,0){0}{\psecurve[linecolor=gray,linewidth=1pt](0,1.3)(1,1)(1.3,0)(1,-1)(0,-1.3)}
\psrotate(0,0){180}{\psecurve[linecolor=gray,linewidth=1pt](0,1.3)(1,1)(1.3,0)(1,-1)(0,-1.3)}

\psline[linewidth=2pt]{->}(-2,0)(2,0)
}
\rput[c](-1.7,0){$a$}
\rput[c](0,-1.7){$b$}
\rput[c](1.7,0){$c$}
\rput[c](0,1.7){$d$}
\end{pspicture}
\caption{A positive crossing}
\end{subfigure}
\begin{subfigure}[b]{0.3\textwidth}\centering
\begin{pspicture}(-3,-2)(3,2)
\rput{-45}(0,0){

\psrotate(0,0){0}{\psecurve(0.6,3)(0.5,2)(1,1)(2,0.5)(3,0.6)}
\psrotate(0,0){90}{\psecurve(0.6,3)(0.5,2)(1,1)(2,0.5)(3,0.6)}
\psrotate(0,0){180}{\psecurve(0.6,3)(0.5,2)(1,1)(2,0.5)(3,0.6)}
\psrotate(0,0){-90}{\psecurve(0.6,3)(0.5,2)(1,1)(2,0.5)(3,0.6)}

\psline[linewidth=2pt, arrow=5 2pt]{->}(0,-2)(0,2)
\pscurve*[linecolor=lightgray,linewidth=2pt](0,1.3)(1,1)(1.3,0)(1,-1)(0,-1.3)(-1,-1)(-1.3,0)(-1,1)(0,1.3)
\psline[linewidth=2pt, linecolor=gray](0,-1.3)(0,1.3)
\pscircle*[linecolor=lightgray](0,0){0.2}

\psrotate(0,0){90}{\psecurve[linecolor=gray,linewidth=2pt](0,1.3)(1,1)(1.3,0)(1,-1)(0,-1.3)}
\psrotate(0,0){-90}{\psecurve[linecolor=gray,linewidth=2pt](0,1.3)(1,1)(1.3,0)(1,-1)(0,-1.3)}
\psrotate(0,0){0}{\psecurve[linecolor=gray,linewidth=1pt](0,1.3)(1,1)(1.3,0)(1,-1)(0,-1.3)}
\psrotate(0,0){180}{\psecurve[linecolor=gray,linewidth=1pt](0,1.3)(1,1)(1.3,0)(1,-1)(0,-1.3)}

\psline[linewidth=2pt, arrow=5 2pt]{<-}(-2,0)(2,0)
}
\rput[c](-1.7,0){$a$}
\rput[c](0,-1.7){$b$}
\rput[c](1.7,0){$c$}
\rput[c](0,1.7){$d$}
\end{pspicture}
\caption{A negative crossing}
\end{subfigure}
\caption{Attaching 2-handles at crossings}\label{fig:attachingtwohandles}
\end{figure}
Up to an overall sign, the attaching map is then given by $a+b+c+d$ in both cases. This gives rise to the following cellular chain complex of $X_T$:
\begin{equation}\label{cellccxofX}
\begin{tikzcd}  
\mathbb{Z}^c\arrow{r}{A} & \mathbb{Z}^{c+n+1}\arrow{r}&\mathbb{Z}^2,
\end{tikzcd}
\end{equation}
where $c$ is the number of crossings, $c+n+1$ the number of regions in the diagram and $A$ is the $(c+n+1)\times c$ matrix determined by the rules above. (Here, we assume that the diagram is connected and has at least one crossing.)
Let $R=\mathbb{Z}H_1(X_T)$, the group ring of $H_1(X_T)$. Then the cellular chain complex of the maximal Abelian cover is
\begin{equation}\label{cellccxoftildeX}
\begin{tikzcd}  
R^c\arrow{r}{\tilde{A}} & R^{c+n+1}\arrow{r} & R^2.
\end{tikzcd}
\end{equation}
The attaching maps of the 1-handles, which determine the matrix $\tilde{A}$, are given by $a+o^{-1}b+o^{-1}c+d$ for positive crossings and $a+b+o^{-1}c+o^{-1}d$ for negative crossings up to multiplication by a unit in $R$, where $o$ is the homology class of the meridian of the over-strand. Also note that we have used the right-hand rule to determine the orientation of these meridians. 
Each site $s$ of a tangle diagram gives rise to a subhandlebody of $X_T$ which we can regard as a subspace $B_s$ of $\partial B^3\smallsetminus \partial T$:
We associate with $s$ the subhandlebody of $X_T$ consisting of the two 0-handles and all 1-handles corresponding to those regions not in~$s$, 
i.\,e.\ unoccupied open regions. The cellular chain complex of $B_s$ is
$$\begin{tikzcd}  
\mathbb{Z}^{n+1}\arrow{r}&\mathbb{Z}^2,
\end{tikzcd}$$
which we can consider as a subcomplex of (\ref{cellccxofX}). Similarly, we can consider the preimage $\tilde{B}_s$ of $B_s$ in $\tilde{X}_T$ and regard its cellular chain complex as a subcomplex of (\ref{cellccxoftildeX}). Then the quotients of these chain complexes calculate $H_\ast(X_T,B_s)$ and $H_\ast(\tilde{X}_T,\tilde{B}_s)$, respectively, and they are given by\vspace*{-0.3cm}
$$
\begin{tikzcd}  
\mathbb{Z}^{c}\arrow{r}{A_s} & \mathbb{Z}^c
\end{tikzcd}
\text{ and }
\begin{tikzcd}  
 R^{c}\arrow{r}{\tilde{A}_s} &R^c,
\end{tikzcd}
$$
where $A_s$ and $\tilde{A}_s$ denote the matrices obtained from $A$ and $\tilde{A}$ by deleting those rows corresponding to $s$, respectively. Note that $H_1(\tilde{X}_T,\tilde{B}_s)$, considered as an $R$-module, is an invariant of $(\tilde{B}_s\subset\tilde{X}_T)$, and $\tilde{A}_s$ is just a presentation matrix of $H_1(\tilde{X}_T,\tilde{B}_s)$. 

\begin{proposition}\label{prop:geometricisthesame}
Let \(s\) be a site of a connected tangle diagram \(T\). Let \(\tilde{A}_s\) be a presentation matrix of \(H_1(\tilde{X}_T,\tilde{B}_s)\) as in the construction above. Then
$$\det\tilde{A}_s(t_1^2,\dots,t_r^2)~\dot{=}~\nabla_T^s(t_1,\dots,t_r),$$
where \(\dot{=}\) denotes equality up to multiplication by a unit.
\end{proposition}

\begin{observation}\label{obs:closedcomponentsgivefactor}
Let us verify this statement for knots and links. For a 2-ended tangle $T$ representing a knot or link $L$, define $X_L:=B^3\smallsetminus T\cong S^3\smallsetminus L$. Hence the first homology groups of the maximal Abelian covers are the same as $H_1(X_L)$-modules. $B_s$ is homotopic to the meridian of the open component, hence for knots, $\tilde{B}_s$ is homotopic to the real line, so $H_1(\tilde{X}_L,\tilde{B}_s)=H_1(\tilde{X}_L)$. The argument for links is slightly more complicated. This is to be expected because by remark \ref{rem:conwaypotfunctionandHFL}, we should see an additional factor $(c-1)$, where $c$ is the colour of the single open strand.\\
Let us consider a slightly more general situation, where we have an arbitrary tangle $T$ with a closed component $c$ and we want to compare $H_1(\tilde{X}_T,\tilde{B}_s)$ with $H_1(\tilde{X}_T,\tilde{B}_s\amalg \tilde{M}_c)$, where $M_c$ is a meridian of $c$. Given a diagram of $T$, consider the following handle-decomposition of $X_T$: Start with three 0-handles, one below the diagram on the boundary of the closed 3-ball ($e_{-}$), one above ($e_{+}$) and one on $M_c$ ($e_c$). Then add a 1-handle for each open region, connecting the two 0-handles above and below, a 1-handle $m_c$ along $M_c$ and a 1-handle $j$ joining $e_c$ to $e_{-}$, say. Finally, add a 2-handle along $m_j$, $j$ and the two 1-handles corresponding to the two open regions on either side of $M_c$ and additional 2-handles at the crossings as before. The cell decomposition near the meridian $M_c$ is illustrated in figure~\ref{fig:celldecompMeridian}. Let $R:=H_1(X_T)$. For the maximal Abelian cover, we obtain the following chain complex:
$$
\begin{tikzcd}  
R^a\arrow{r}{A} &R^{a+n}\oplus Rm_c\oplus Rj\arrow{r}{B} & Re_{-}\oplus Re_{+}\oplus Re_c.
\end{tikzcd}
$$
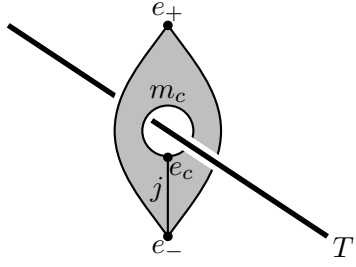
\begin{figure}[t]
	\centering
	\psset{unit=0.7}
	\begin{pspicture}(-3.5,-2.5)(3.5,2.5)
	\psline[linewidth=2pt](0,0)(-3,2)
	\psecurve[linewidth=5pt,linecolor=white](1,3)(0,2)(-1,0)(0,-2)(1,-3)
	\pscustom[fillstyle=solid, fillcolor=lightgray]{%
		\psecurve(1,3)(0,2)(-1,0)(0,-2)(1,-3)
		\psecurve(-1,-3)(0,-2)(1,0)(0,2)(-1,3)
	}
	\pscircle[fillstyle=solid, fillcolor=white](0,0){0.5}
	\psline[linewidth=6pt,linecolor=white](0,0)(3,-2)
	\psline[linewidth=2pt](-0.3,0.2)(3,-2)
	\rput[c](3.3,-2.2){$T$}
	\psline(0,-2)(0,-0.5)
	
	\psdots(0,2)(0,-2)(0,-0.5)
	
	\rput[c](0.25,-0.75){$e_c$}
	\rput[c](0,0.7){$m_c$}
	\rput[c](-0.2,-1.1){$j$}
	\rput[c](0,2.25){$e_+$}
	\rput[c](0,-2.25){$e_-$}
	\end{pspicture}
	\caption{The cell decomposition near the meridian $M_c$ used in observation~\ref{obs:closedcomponentsgivefactor}}\label{fig:celldecompMeridian}
\end{figure}
\pagebreak

\noindent
The matrix $B$ looks follows:
$$
\left(
\begin{matrix}
 \ast & 0 & 1 \\
 \ast & 0 & 0 \\
 0 & (c-1) & -1
\end{matrix}
\right).
$$
On the one hand, to get a presentation matrix $A_{B_s\amalg M_c}$ for $H_1(\tilde{X}_T,\tilde{B}_s\amalg \tilde{M}_c)$, we can simply take the quotient of this complex by 
$$
\begin{tikzcd}  
 R^{n+1}\oplus Rm_c  \arrow{r}& Re_{-}\oplus Re_{+}\oplus Re_c,
\end{tikzcd}
$$
where $R^{n+1}\subset R^{a+n}$ is generated by those 1-handles corresponding to the site $s$; so we just need to delete certain rows in $A$. On the other hand, if we quotient only by a subcomplex of the previous one, namely
$$
\begin{tikzcd}  
R^{n+1}\arrow{r}& Re_{-}\oplus Re_{+},
\end{tikzcd}
$$
then the kernel of the right-hand map in the quotient complex
\begin{equation*}
\begin{tikzcd}  
R^a\arrow{r}{A_q} &R^{a+n}/R^{n+1}\oplus Rm_c\oplus Rj\arrow{r}{B_q}&Re_c
\end{tikzcd}
\end{equation*}
is $R^{a+n}/R^{n+1}\oplus R(m_c+(c-1)j)$. By doing a simple change of basis that sends $m_c$ to $m_c+(c-1)j$ and fixes all other basis vectors, we see that the presentation matrix of $H_1(\tilde{X}_T,\tilde{B}_s)$ is obtained from $A_q$ by deleting the row for $j$. Now the fact that $B_qA_q=0$ implies that the row for $m_c$ in $A_q$ multiplied by $(c-1)$ is equal to the row for $j$ in $A_q$. Thus, by linearity of the determinant with respect to rows, we obtain
$$\det(A_{B_s\amalg M_c})=(c-1)\det(A_{B_s}).$$
We will use this argument later in the proof of theorem \sref{thm:HFTfromSFH}.
\end{observation}

The proof of proposition \ref{prop:geometricisthesame} follows from basically the same argument as the corresponding statement for the classical Alexander polynomial, see \cite[proposition~3.1]{Kauffman}. The crucial point is that the signs in the definition of the determinant of a matrix and the signs coming from $h=-1$ in the Alexander codes correspond to one another. In order to see this, we make use of the generalised clock theorem again. We also need the following lemma to match the Alexander codes -- and again, the generalised clock theorem is the key.

\begin{figure}[b]
\centering
\psset{unit=1.0}
\begin{subfigure}[b]{0.3\textwidth}\centering
\begin{pspicture}(-1.55,-1.05)(1.55,1.05)
\psline[linecolor=violet]{->}(0.9,-0.9)(-0.9,0.9)
\pscircle*[linecolor=white](0,0){0.3}
\psline[linecolor=darkgreen]{->}(-0.9,-0.9)(0.9,0.9)
\uput{0.5}[90](0,0){$\textcolor{darkgreen}{o}^{\frac{1}{2}}$}
\uput{0.5}[180](0,0){$\textcolor{darkgreen}{o}^{\frac{1}{2}}$}
\uput{0.5}[-90](0,0){$-\textcolor{darkgreen}{o}^{-\frac{1}{2}}$}
\uput{0.5}[0](0,0){$\textcolor{darkgreen}{o}^{-\frac{1}{2}}$}
\end{pspicture}
\caption{A positive crossing}
\end{subfigure}
\begin{subfigure}[b]{0.3\textwidth}\centering
\begin{pspicture}(-1.55,-1.05)(1.55,1.05)
\psline[linecolor=violet]{->}(-0.9,-0.9)(0.9,0.9)
\pscircle*[linecolor=white](0,0){0.3}
\psline[linecolor=darkgreen]{->}(0.9,-0.9)(-0.9,0.9)
\uput{0.5}[90](0,0){$\textcolor{darkgreen}{o}^{-\frac{1}{2}}$}
\uput{0.5}[180](0,0){$\textcolor{darkgreen}{o}^{\frac{1}{2}}$}
\uput{0.5}[-90](0,0){$-\textcolor{darkgreen}{o}^{\frac{1}{2}}$}
\uput{0.5}[0](0,0){$\textcolor{darkgreen}{o}^{-\frac{1}{2}}$}
\end{pspicture}
\caption{A negative crossing}
\end{subfigure}
\caption{The alternative Alexander codes from lemma~\ref{lem:AlexCodematching}. The variable $\textcolor{darkgreen}{o}$ is the colour of the over-strand.}\label{fig:altAlexCodes}
\end{figure}
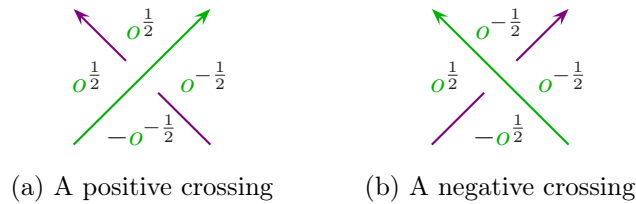

\begin{lemma}\label{lem:AlexCodematching}
In definition~\ref{def:basic}, we can replace the Alexander codes from figure~\ref{figAlexCodesForNabla} by those in figure~\ref{fig:altAlexCodes} to obtain new polynomials \(\nabla_{T,\textnormal{new}}^s\).
Then for all tangle diagrams \(T\) and sites~\(s\) 
$$\nabla_{T,\textnormal{new}}^s(t_1^2,\dots,t_r^2)~\dot{=}~\nabla_{T}^s(t_1,\dots,t_r).$$
\end{lemma}

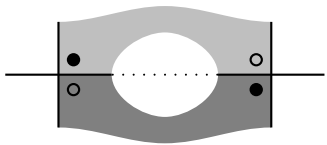
\begin{figure}[b]
\centering
\psset{unit=0.7}
\begin{pspicture}(-8,-1.3)(8,1.3)
\pscustom*[linecolor=lightgray]{
\psline(-2,0)(-2,1)
\psecurve(-4,1.3)(-2,1)(0,1.3)(2,1)(4,1.3)
\psline(2,1)(2,0)
\psline(2,0)(-2,0)
}
\pscustom*[linecolor=gray]{
\psline(-2,0)(-2,-1)
\psecurve(-4,-1.3)(-2,-1)(0,-1.3)(2,-1)(4,-1.3)
\psline(2,-1)(2,0)
\psline(2,0)(-2,0)
}
\psecurve*[linecolor=white](1,0)(0,0.8)(-1,0)(0,-0.8)(1,0)(0,0.8)(-1,0)

\psline(-2,1)(-2,-1)
\psline(2,1)(2,-1)
\psline(-3,0)(-1,0)
\psline[linestyle=dotted](-1,0)(1,0)
\psline(1,0)(3,0)
\uput{0.4}[45](-2,0){\pscircle*(0,0){2.5pt}}
\uput{0.4}[-135](2,0){\pscircle*(0,0){2.5pt}}

\uput{0.4}[-45](-2,0){\pscircle(0,0){2.5pt}}
\uput{0.4}[135](2,0){\pscircle(0,0){2.5pt}}
\end{pspicture}
\caption{A transposition move of Kauffman states}\label{fig:TrapoAltAlexCodes}
\end{figure}

\begin{proof}
Using observation~\ref{ObsAlexanderCode}, we consider the effect of a transposition move on the labelling of Kauffman states.
Let the colour of the horizontal strand in figure~\ref{fig:TrapoAltAlexCodes} be~$t$. Note that any other colours are not affected by a transposition move. Now, if the two vertical strands go either both over or both under the horizontal strand, the two Kauffman states will have the same labelling for both codes. In the other case, the exponent of $t$ will change by $\pm1$ for the code above, but by $\pm2$ for the original code. It is not hard to see that the signs are the same. Now apply the generalised clock theorem.
\end{proof}

\begin{definition}
We define the \textbf{sign \(\sgn(x)\) of a Kauffman state \(x\)} of a connected tangle diagram to be the sign of \(c(x)\) with \(h=-1\), see definition~\ref{def:basic}. 
\end{definition}
\begin{proof}[Proof of proposition~\ref{prop:geometricisthesame}]
Without loss of generality we may assume that at any crossing of the diagram, all four regions are pairwise distinct. For, once we have shown the above statement for this restricted case, it also holds true for any connected diagram, since both sides are invariants of $T$ up to isotopy. \\
We can regard the Kauffman states as a region-crossing assignment, so we can fix a map $P:\mathbb{K}(T,s)\rightarrow S_n$. If $x$ and $y$ are two Kauffman states in $\mathbb{K}(T,s), s\in\mathbb{S}(T)$, that are related by a transposition move, $x$ and $y$ have opposite signs. On the other hand, $P(x)$ and $P(y)$ also have opposite signs as elements of the permutation group. By the generalised clock theorem, we know that any two Kauffman states in $\mathbb{K}(T,s)$ are connected by a sequence of transposition moves. Hence $\sgn=\pm\sgn\circ P$. The proposition now follows from the definition of the determinant and lemma~\ref{lem:AlexCodematching}.
\end{proof}

\begin{corollary}
Let \(s\) be a site of a tangle \(T\). Then 
$$\nabla_T^s(1,\dots,1)=0 \Leftrightarrow H_2(B^3\smallsetminus T,B_s)\neq 0.$$
\end{corollary}
\begin{proof}
Observe that the matrix $A$ in \eqref{cellccxofX}  is obtained from the matrix $\tilde{A}$ in \eqref{cellccxoftildeX} by setting all colours equal to 1.
\end{proof}
\begin{Remark}
Given two sites $s$ and $s'$, $B_s$ can be obtained from $B_{s'}$ by applying some element of the mapping class group of $\partial B^3\smallsetminus \partial T$. This means that if we know $\nabla^s_T$ for a fixed site $s$ and all tangles $T$, we also know $\nabla^{s'}_T$ for all sites $s'$, up to normalisation. But why should we restrict ourselves to those subspaces $B$ of the punctured sphere $\partial B^3\smallsetminus \partial T$ that come from sites?
In principle, we can consider $H_1(\tilde{X}_L,\tilde{B})$ for any such $B$ and define $\nabla^B_T$ to be the generator of the smallest principal ideal containing the first elementary ideal of $H_1(\tilde{X}_L,\tilde{B})$. This gives rise to infinitely many polynomial invariants $\nabla^B_T$ for any given tangle $T$. However, if for example $H_1(\tilde{X}_L,\tilde{B})$ has a presentation matrix with more columns than rows, i.\,e.\ more generators than relations, then $\nabla^B_T$ will be zero. This is for example always the case if $\chi(B)>(1-n)=\chi(B^s)$. \pagebreak[3]\\
The above considerations give rise to the following proposition.
\end{Remark}
\begin{proposition}\label{prop:anyBisfine}
Let \(B\) be an essentially embedded subsurface of a \(2n\)-punctured sphere, i.\,e.\ the map \(\pi_1(B)\rightarrow\pi_1(\partial B^3\smallsetminus \partial T)\) induced by the inclusion is injective. Suppose further that \(\chi(B)=1-n\). Then there exists an element \(g\) in the mapping class group of \(\partial B^3\smallsetminus \partial T\) such that \(g(B)\) is homotopic to \(B_s\) for some site \(s\). 
\end{proposition} 
 
\begin{proof}
We fix some cellular structure on $X_T$ with a single 0-cell such that $B$ deformation retracts onto a 1-dimensional sub-cell-complex. Its 1-cells cannot be interleaved, so by sliding their endpoints along each other, we can arrange that they are attached to the 0-cell like petals. Since $B$ is essentially embedded, $(\partial B^3\smallsetminus \partial T)\smallsetminus B$ is a disjoint union of $(n+1)$ discs, each of which contains at least one puncture. Hence, we can apply an element of the mapping class group of the punctured sphere to obtain a ``standard'' surface in a punctured sphere which only depends on the partition of the punctures induced by $B$. Since any such partition can be achieved by $B_s$ for some site~$s$, the result follows.
\end{proof} 
\begin{questions}\label{que:MCGactionAndMore}
In the context of the proposition above, there are some open questions that we would like to answer; in particular:
\begin{itemize}
\item Are there any general relations between the polynomials \(\nabla_T^s\) for different sites \(s\), in particular those that cannot be twisted into one another by applying some element of the mapping class group?
\item Can we compute the action of the mapping class group of \(\partial B^3\smallsetminus \partial T\) on~\(\nabla_T^s\)?
\item What is the best way to normalise \(\nabla_{T,new}^s\)?
\item Can we describe higher elementary ideals of \(H_1(\tilde{X}_T,\tilde{B^s})\)? What about \(\nabla_T^B\) for the case \(\chi(B)<1-n\)?
\item Is there a geometric interpretation of the glueing formula from proposition \ref{prop:glueing}, for example via some Mayer-Vietoris argument?
\end{itemize}
\end{questions}

\subsection*{Comparison with other definitions of Alexander polynomials for tangles. }
In the introduction to this chapter, we mentioned several other generalisations of the Alexander polynomial to tangles. In the remainder of this section, we try to compare our invariant to some of the other definitions. I hope to come back to this at some point in the future to make some of those vague statements below more precise.

\subsection*{Archibald's invariant.}
In \cite{Archibald}, Archibald defines her polynomial Alexander invariant(s) for tangles directly via Alexander matrices. However, she uses a different cellular decomposition of the tangle complement to calculate the Alexander matrix; ours comes from the Dehn representation of the fundamental group of the complement, hers from the Wirtinger representation. For the latter, one fixes a single 0-cell above a tangle digram and adds a loop for each connected arc in the diagram, as illustrated below.
\begin{center}
\psset{unit=0.5, linewidth=1.1pt}
\begin{pspicture}(-4.5,-3.35)(3.5,3.35)
\psecurve[linewidth=0.7pt,linecolor=blue](4,-1)(3,0)(0.5,2)(0.5,2.5)(3,0)(4,-1)
\psecurve[linewidth=0.7pt,linecolor=blue](4,0)(3,0)(0.25,0.7)(0.75,1)(3,0)(4,-0.5)
\psecurve[linewidth=0.7pt,linecolor=blue](4,0)(3,0)(0.25,-0.7)(0.75,-1)(3,0)(4,0.5)
\psecurve[linewidth=0.7pt,linecolor=blue](4,1)(3,0)(0.5,-2)(0.5,-2.5)(3,0)(4,1)
\pscircle*[linecolor=white](0.4,0.34){0.2}
\pscircle*[linecolor=white](0.4,-0.34){0.2}
\pscircle*[linecolor=white](0.47,2.16){0.2}
\pscircle*[linecolor=white](0.47,-2.16){0.2}
\psecurve(-2,2)(0,2)(0.75,1)(-0.75,-1)(0,-2)(1,-2.2)(2,-2)
\psecurve{<-}(2,2)(1,2.2)(0,2)(-0.75,1)(0.75,-1)(0,-2)(-2,-2)(-3,0)(-2,2)(0,2)(0.75,1)
\pscircle*[linecolor=white](0,2){0.2}
\pscircle*[linecolor=white](0,0){0.2}
\pscircle*[linecolor=white](0,-2){0.2}
\psecurve(0,2)(0.75,1)(-0.75,-1)(0,-2)
\psecurve{<-}(2,2)(1,2.2)(0,2)(-0.75,1)(0.75,-1)
\psecurve(-0.75,1)(0.75,-1)(0,-2)(-2,-2)(-3,0)
\pscircle[linestyle=dotted](-1,0){3.05}
%
%
\pscircle*[linecolor=white](0.75,-1){0.13}
\pscircle*[linecolor=white](0.75,1){0.13}
\psecurve[linewidth=0.7pt,linecolor=blue](3,0)(0.25,0.7)(0.75,1)(3,0)(4,-0.5)
\psecurve[linewidth=0.7pt,linecolor=blue](3,0)(0.25,-0.7)(0.75,-1)(3,0)(4,0.5)

\psdot[linecolor=blue](3,0)
\end{pspicture}
\end{center}
These loops generate the fundamental group of the tangle complement. One then adds 2-cells, one for each crossing, which gives the relations between the generators. The columns of the Alexander matrix correspond to crossings as in our construction, but the rows correspond to arcs. Archibald then considers square matrices obtained by deleting some rows that correspond to arcs that meet the boundary, which we call ``marked'' arcs. A simple counting argument shows that we need to delete exactly $n$ rows (if there are no closed arcs). \\
Let us assume for simplicity that every arc meets the boundary at most once. (We can always find such a diagram for a given tangle.) Then, those square matrices define a presentation of $H_1(\tilde{X}_T,\tilde{B})$, where $B$ is the subspace of the boundary given by the 0-cell and those 1-cells corresponding to marked arcs.
This subspace $B$ does not necessarily have to be homotopic to one that comes from a site in our construction. However, after applying some element of the mapping class group of the boundary, it is, like in the last step of the proof of proposition~\ref{prop:anyBisfine}. Thus, we can calculate Archibald's invariant from our invariants and vice versa, up to normalisation.

\subsection*{Diagrammatic invariants. } Using the skein relation for the single-variate Alexander polynomial and the convention that any diagram with an unknot component is zero, one can reduce any given tangle to a linear combination of ``simpler'' tangles. For a suitable (minimal) choice of such ``elementary'' tangles, this also gives an invariant. (This can be easily seen by adapting the arguments from~\cite{LickorishMillett}.) Since $\nabla_T^s$ also satisfies the skein relation, we can compute it by substituting the elementary tangles by their polynomials~$\nabla_T^s$. Computations suggest that one can also go the other direction. Since Bigelow's \cite{Bigelow} as well as Polyak's \cite{Polyak} invariants also satisfy the skein relation, this would imply that  all of these polynomials contain basically the same information.\\
For the multivariate version, the same approach does not work. Computations suggest that $\nabla_T^s$ and Kennedy's multivariate version \cite{Kennedy} of Bigelow's invariant are closely related, but I have been unable to make this relationship precise. If we restrict ourselves to 4-ended tangles, we saw in proposition~\ref{prop:fourended} that there is basically only one piece of information in $\nabla_T^s$. One can show that Kennedy's invariant, consisting (\textit{a priori}) of nine different polynomials, contains at most two different pieces of information. It would be interesting to know if one can make this result as strong as for $\nabla_T^s$. 

\subsection*{Sartori's invariant. } In \cite{Sartori14}, Sartori interpreted the Alexander polynomial in terms of the representation theory of $\mathfrak{gl}(1\vert 1)$.  I assume that one can interpret the tangle Floer homology defined in the next chapter as a special case of Petkova and Vértesi's more general combinatorial tangle Floer homology \cite{cHFT}, namely their one-sided case. Since their construction can be seen as a categorification of Sartori's invariant \cite{DecatCTFH}, this would then imply a relationship between $\nabla_T^s$ and Sartori's one-sided invariant, taking the form of a map from a $U_q(\mathfrak{gl}(1\vert 1))$-representation $V$ to $\mathbb{C}(q)$. Under this map, $\nabla_T^s(q)$ should probably be the image of some element in $V$ that corresponds to the site $s$ in Petkova and Vértesi's construction. 

%% file: sections/2_HFTviaSFH.tex
\section{A categorification via sutured Heegaard Floer theory}\label{sec:HFTviaSFH}

In \cite{Juhasz}, Juhász generalised the hat version of Heegaard Floer homology of closed three manifolds and links to balanced sutured manifolds, certain manifolds with boundary together with some extra structure on the boundary. He used this sutured Floer homology, denoted by $\SFH$, to give short proofs of a number of known results, e.\,g.\ that link Floer homology detects the Thurston norm and fibredness. Juhász also proved a surface decomposition formula, which says that $\SFH$ behaves very nicely under splitting a balanced sutured manifold along certain embedded surfaces. For all basic definitions and properties of $\SFH$, we refer the reader to Juhász's original papers \cite{Juhasz,SurfaceDecomposition,polytope} and Altman's introductory article~\cite{Altman}.\\
In this section, we give a quick definition of a tangle Floer homology $\HFT$ in terms of $\SFH$ and show that its Euler characteristic agrees with the polynomial invariant~$\nabla_T^s$. Then, we use a version of Juhász's surface decomposition formula to prove symmetry relations for~$\HFT$.
\begin{definition}
	With an $r$-component tangle $T$ and a site $s$ of $T$, we associate a sutured 3-manifold $X_T^s$ defined as follows: The underlying 3-manifold with boundary is $X_T=B^3\smallsetminus \nu(T)$, the complement of a tubular neighbourhood of $T$ in $B^3$. The sutures on $\partial X_T^s$ are obtained by placing two oppositely oriented meridional circles around closed components of the tangle and meridional circles around the ends of the open components and performing surgery along the arcs in $s$, see figure~\sref{fig:HFTviaSFHsutmfd}. We orient the sutures such that one component of $R_-$ is contained in the boundary of the 3-ball.
\end{definition}
\begin{figure}[tb]
\centering
\psset{unit=0.22}
\begin{pspicture}(-10,-10)(10,10)
\pscircle(0,0){10}
\psarc[linecolor=red](0,0){7}{45}{135}
\psarc[linecolor=red,linestyle=dotted](0,0){7}{142}{215}
\psarc[linecolor=red,linestyle=dotted](0,0){7}{235}{305}
\psarc[linecolor=red,linestyle=dotted](0,0){7}{325}{398}

\pscircle[linecolor=lightgray](0,0){4}

\psline[linecolor=white,linewidth=4pt](3;45)(6.3;45)
\psline[linecolor=white,linewidth=4pt](3;135)(6.3;135)
\psline[linecolor=white,linewidth=4pt](3;-135)(6.3;-135)
\psline[linecolor=white,linewidth=4pt](3;-45)(6.3;-45)

\psline[linecolor=gray](3;45)(6.3;45)
\psline[linecolor=gray](3;135)(6.3;135)
\psline[linecolor=gray](3;-135)(6.3;-135)
\psline[linecolor=gray](3;-45)(6.3;-45)

\rput{-135}(7;-135){\psellipse[linecolor=darkgreen](0,0)(0.5,1)}
\rput{-45}(7;-45){\psellipse[linecolor=darkgreen](0,0)(0.5,1)}

\pscustom[linecolor=darkgreen]{
\psarc(0,0){6.5}{45}{135}
\psecurve(7;130)(6.5;135)(7;140)(7.5;135)(7;130)
\psarcn(0,0){7.5}{135}{45}
\psecurve(7;50)(7.5;45)(7;40)(6.5;45)(7;50)
}

\rput(-5.5,0){$a$}
\rput(0,-5.5){$b$}
\rput(5.5,0){$c$}
\rput(0,5){$d$}

\rput(0,0){\textcolor{darkgray}{$T$}}

\end{pspicture}
\caption{The set of sutures (\textcolor{darkgreen}{green} curves) on $X_T^s$ for a 4-ended tangle $T$ and $s=d$. Any closed components of $T$ get two meridional sutures as in the case of knots and links.}\label{fig:HFTviaSFHsutmfd}
\end{figure}
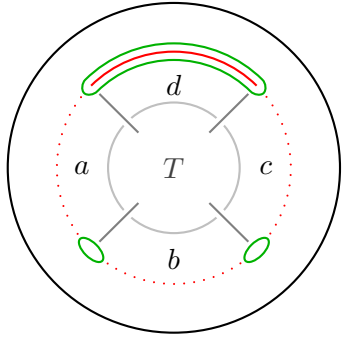
\begin{theorem}\label{thm:HFTfromSFH}
	The sutured Floer chain complex \(\SFC(X_T^s)\) is an invariant of the tangle~\(T\) and the site \(s\) up to chain homotopy equivalence. Furthermore
	\begin{equation}\label{eqn:decatSFTagreeswithnabla}
	\chi(\SFC(X_T^s))(t_1^2,\dots,t_r^2)~\dot{=}~\prod(c-c^{-1}) \cdot\nabla_T^s(t_1,\dots,t_r),
	\end{equation}
	where \(\dot{=}\) denotes equality up to multiplication by a unit and the product on the right is over all closed components of~\(T\) and their colours~\(c\). For a 2-ended tangle~\(T\), \(\SFC(X_T^{\emptyset})\) agrees with \(\CFL(L)\), where \(L\) denotes the link obtained by joining the two open ends of~\(T\).
\end{theorem}
\begin{definition}\label{def:HFTfromSFH}
	We denote the chain complex $\SFC(X_T^s)$ by $\CFT(T,s)$ and its homology by $\HFT(T,s)$, the \textbf{tangle Floer homology} of $T$ at $s$. 
\end{definition}
\begin{definition}
The sutured Floer homology of any sutured manifold $(M,\gamma)$ comes with two gradings: a grading by relative $\Spinc$-structures of $(M,\gamma)$ and, for each such $\Spinc$-structure, a relative \textbf{homological $\mathbb{Z}/2$-grading}. $\Spinc$-structures of $(M,\gamma)$ form an affine space of $H_1(M)$. So if $M=X_T^s$ for an $r$-component tangle $T$, an orientation on~$T$ induces a relative $\mathbb{Z}^{r}$-grading, which we call the \textbf{Alexander grading}. The sutured Floer chain complex of any $(M,\gamma)$ splits along $\Spinc$-structures. For our tangle Floer homology this means that we can write
$$\CFT(T,s)=\bigoplus_{a\in\mathbb{Z}^{r}}\CFT(T,s,a),$$
where $\CFT(T,s,a)$ denotes the component of $\CFT(T,s)$ in a fixed Alexander grading~$a$.
\end{definition}
\begin{Remark}
In the next section, when we study Heegaard diagrams of tangles, we see how to lift the homological $\mathbb{Z}/2$-grading to a relative $\mathbb{Z}$-grading and how to define the Alexander grading for all sites \emph{simultaneously}. To achieve this purely in terms of sutured Floer homology, one would have to relate relative $\Spinc$-gradings corresponding to different sets of sutures, see remark~\ref{rem:absolutegradings}. 
\end{Remark}
\begin{proof}[Proof of theorem \sref{thm:HFTfromSFH}]
We first check that $X_T^s$ is balanced. Say $T$ has $n$ open components and without loss of generality, we may assume that there are no closed components. The site $s$ consists of $(n-1)$ open regions, so there are $(n-1)$ arcs which we have performed surgery along. Hence, $R_-$ is a sphere with $(n+1)$ punctures, so it has Euler characteristic $(1-n)$. Each annulus around an open tangle component contributes 0 to the Euler characteristic, but each surgery decreases the Euler characteristic by 1. \\
Obviously, $X_T^s$ is an invariant of $T$, and so is its sutured Floer homology. Therefore, it only remains to check the identity (\sref{eqn:decatSFTagreeswithnabla}). For this, we use the explicit formula from \cite[proposition~5.1]{DecatSFH} for the graded Euler characteristic $\chi(\SFH(M,\gamma))$ of a sutured manifold $(M,\gamma)$: Consider the pair $(M,R_-)$ and its maximal Abelian cover $(\tilde{M},\tilde{R}_-)$. Let $A$ be a square presentation matrix of $H_1(\tilde{M},\tilde{R}_-)$ as a $H_1(M,R_-)$-module. Then $\chi(\SFH(M,\gamma))$ is equal to $\det(A)$. We essentially did all the work in the previous section. First, suppose $T$ does not have a closed component. Then, using the notation from proposition \sref{prop:geometricisthesame}, $H_1(X_T^s,R_-)\cong H_1(B^3\smallsetminus T,B_s)$ and the same holds for the maximal Abelian covers, so we are done by the same proposition. For the general case, we get a factor $(c-c^{-1})$ for each closed component by observation \sref{obs:closedcomponentsgivefactor}, noting that we need to take the square of all variables to match the convention used in the definition of $\nabla_{T}^s$.\\
Finally, specialising to 2-ended tangles $T$ representing a link $L$, we observe that $X_T^\emptyset$ is simply the complement of a tubular neighbourhood of $L$ in $S^3$ with two meridional sutures on each boundary component, so $\SFH(X_T^\emptyset)$ agrees with the ordinary link Floer homology $\HFL(L)$ by \cite[proposition~9.2]{Juhasz}.
\end{proof}
The following results can be regarded as the categorification of their counterparts in sections~\ref{sec:basicpropertiesofnabla} and~\ref{sec:4endedandmutation}.
\begin{proposition}\label{prop:HFTmirror}
Let \(\m(T)\) denote the mirror image of a tangle \(T\). Let \(\widehat{\operatorname{CFT}^\ast\!\!}\,\,(T,s)\) denote the dual chain complex of \(\CFT(T,s)\), with the usual convention that all gradings are reversed. Then 
$$\CFT(\m(T),s)~\dot{\cong}~\widehat{\operatorname{CFT}^\ast\!\!}\,\,(T,s),$$
where \(\dot{\cong}\) denotes graded chain homotopy equivalence up to an overall shift of the gradings.
\end{proposition}
\begin{proof}
This follows from the previous theorem and \cite[proposition~2.14]{DecatSFH}.
\end{proof}
\begin{proposition}\label{prop:HFTreverseorientI}
Let \(T\) be an oriented \(r\)-component tangle and \(s\) a site of \(T\). If \(\rr(T,t_1)\) denotes the same tangle \(T\) with the orientation of the first strand reversed, then for all Alexander gradings \(a=(a_1,\dots,a_r)\in\mathbb{Z}^r\),
$$\CFT(\rr(T,t_1),s,a)~\dot{\cong}~\CFT(T,s,(-a_1,a_2,\dots,a_r)).$$
Similarly, if \(\rr(T)\) denotes the tangle \(T\) with the orientation of all strands reversed, then 
$$\CFT(\rr(T),s,a)~\dot{\cong}~\CFT(T,s,-a).$$
\end{proposition}
\begin{proof}
The orientation of a tangle component is just a choice of an orientation of its meridian. This does not affect the relative homological grading.
\end{proof}
\begin{proposition}\label{prop:fourendedHFT}
Let \(T\) be a 4-ended tangle. We distinguish between the same two cases as in proposition~\sref{prop:fourended}:
$$
\begin{pspicture}(-2.6,-1.5)(2.6,1.05)
\rput(-1.5,0){
\SpecialCoor
\psline[linecolor=violet]{->}(0.6;45)(1;45)
\psline[linecolor=violet]{<-}(0.6;-45)(1;-45)

\psline[linecolor=darkgreen]{->}(0.6;135)(1;135)
\psline[linecolor=darkgreen]{<-}(0.6;-135)(1;-135)

\uput{0.6}[0](0,0){$c$}
\uput{0.6}[90](0,0){$d$}
\uput{0.6}[180](0,0){$a$}
\uput{0.6}[270](0,0){$b$}

\uput{1.1}[135](0,0){$\textcolor{darkgreen}{p}$}
\uput{1.1}[-135](0,0){$\textcolor{darkgreen}{p}$}
\uput{1.1}[45](0,0){$\textcolor{violet}{q}$}
\uput{1.1}[-45](0,0){$\textcolor{violet}{q}$}
\pscircle[linestyle=dotted](0,0){1}

\rput(0,-1.3){\textit{case I}}
}

\rput(1.5,0){
\SpecialCoor
\psline[linecolor=violet]{->}(0.6;45)(1;45)
\psline[linecolor=violet]{<-}(0.6;-135)(1;-135)

\psline[linecolor=darkgreen]{->}(0.6;135)(1;135)
\psline[linecolor=darkgreen]{<-}(0.6;-45)(1;-45)

\uput{0.6}[0](0,0){$c$}
\uput{0.6}[90](0,0){$d$}
\uput{0.6}[180](0,0){$a$}
\uput{0.6}[270](0,0){$b$}

\uput{1.1}[135](0,0){$\textcolor{darkgreen}{p}$}
\uput{1.1}[-45](0,0){$\textcolor{darkgreen}{p}$}
\uput{1.1}[45](0,0){$\textcolor{violet}{q}$}
\uput{1.1}[-135](0,0){$\textcolor{violet}{q}$}
\pscircle[linestyle=dotted](0,0){1}
\rput(0,-1.3){\textit{case II}}
}

\end{pspicture}
$$
In both cases, 
\begin{equation}
\CFT(T,b)~\dot{\cong}~\CFT(\rr(T),d).\label{eqn:BD}\tag{B-D}
\end{equation}
In case I, we also have
\begin{equation}
V_p\otimes\CFT(T,a)~\dot{\cong}~V_q\otimes\CFT(\rr(T),c), \tag{A-C}\label{eqn:AC}
\end{equation}
where \(V_t\) denotes a 2-dimensional vector space supported in consecutive Alexander and homological gradings.
In case II (and in case I with \(p=q\)), the second identity holds if we drop the tensor factors \(V_p\) and \(V_q\).
\end{proposition}

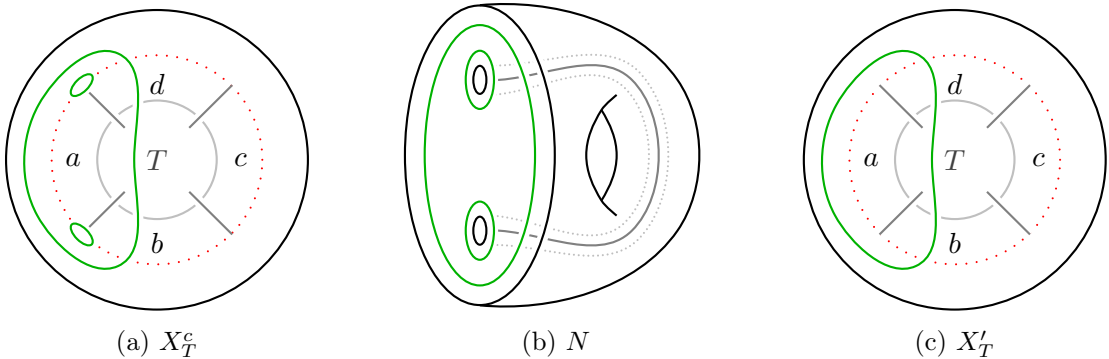
\begin{figure}[b]
\centering
\begin{subfigure}[b]{0.31\textwidth}
\centering
\psset{unit=0.2}
\begin{pspicture}(-10,-10)(10,10)
\pscircle(0,0){10}
\pscircle[linecolor=red,linestyle=dotted](0,0){7}

\pscircle[linecolor=lightgray](0,0){4}

\psline[linecolor=white,linewidth=4pt](3;45)(6.3;45)
\psline[linecolor=white,linewidth=4pt](3;135)(6.3;135)
\psline[linecolor=white,linewidth=4pt](3;-135)(6.3;-135)
\psline[linecolor=white,linewidth=4pt](3;-45)(6.3;-45)

\psline[linecolor=gray](3;45)(7;45)
\psline[linecolor=gray](3;135)(6.3;135)
\psline[linecolor=gray](3;-135)(6.3;-135)
\psline[linecolor=gray](3;-45)(7;-45)

\rput{-135}(7;-135){\psellipse[linecolor=darkgreen,fillcolor=white,fillstyle=solid](0,0)(0.5,1)}
\rput{135}(7;135){\psellipse[linecolor=darkgreen,fillcolor=white,fillstyle=solid](0,0)(0.5,1)}

\psecurve[linecolor=white,linewidth=4pt](7.5;110)(1.5;180)(7.5;-110)(8.7;-155)
(8.6;155)(7.5;110)(1.5;180)(7.5;-110)
\psecurve[linecolor=darkgreen](7.5;110)(1.5;180)(7.5;-110)(8.7;-155)
(8.6;155)(7.5;110)(1.5;180)(7.5;-110)

\rput(-5.5,0){$a$}
\rput(0,-5.5){$b$}
\rput(5.5,0){$c$}
\rput(0,5){$d$}

\rput(0,0){\textcolor{darkgray}{$T$}}

\end{pspicture}
\caption{$X_T^c$}\label{fig:SurfaceDecompBefore}
\end{subfigure}
\quad
\begin{subfigure}[b]{0.31\textwidth}
\centering
\psset{unit=0.2}
\begin{pspicture}(-10.3,-10.3)(10.3,10.3)

\psecurve[linecolor=gray](-10,6)(-5,5)(5,5)(5,-5)(-5,-5)(-10,-6)

\psellipse*[linecolor=white,linewidth=4pt](-5,5)(0.5,1)
\psellipse*[linecolor=white,linewidth=4pt](-5,-5)(0.5,1)

\psecurve[linecolor=lightgray,linestyle=dotted,dotsep=1pt](-10,5)(-5,4)(4.5,4.5)(4.5,-4.5)(-5,-4)(-10,-5)
\psecurve[linecolor=lightgray,linestyle=dotted,dotsep=1pt](-10,7)(-5,6)(5.5,5.5)(5.5,-5.5)(-5,-6)(-10,-7)

\psellipse[linecolor=white,linewidth=4pt](-5,0)(5.3,10.3)

\psecurve(-13,5)(-5,9.95)(8,5)(8,-5)(-5,-9.95)(-13,-5)
\psellipse(-5,0)(5,10)
\rput(-2,0){
\psecurve(4,4)(5,3)(6,0)(5,-3)(4,-4)
\pscurve(6,4)(5,3)(4,0)(5,-3)(6,-4)
}

\psellipse[linecolor=white,linewidth=4pt](-5,0)(4,9)
\psellipse[linecolor=darkgreen](-5,0)(3.7,8.7)

\psellipse[linecolor=white,linewidth=4pt](-5,5)(1.2,2.2)
\psellipse[linecolor=white,linewidth=4pt](-5,-5)(1.2,2.2)
\psellipse[linecolor=darkgreen,fillcolor=white,fillstyle=solid](-5,5)(1,2)
\psellipse[linecolor=darkgreen,fillcolor=white,fillstyle=solid](-5,-5)(1,2)

\psellipse(-5,5)(0.5,1)
\psellipse(-5,-5)(0.5,1)

\end{pspicture}
\caption{$N$}\label{fig:SurfaceDecompN}
\end{subfigure}
\quad
\begin{subfigure}[b]{0.31\textwidth}
\centering
\psset{unit=0.2}
\begin{pspicture}(-10,-10)(10,10)
\pscircle(0,0){10}
\pscircle[linecolor=red,linestyle=dotted](0,0){7}

\pscircle[linecolor=lightgray](0,0){4}

\psline[linecolor=white,linewidth=4pt](3;45)(7;45)
\psline[linecolor=white,linewidth=4pt](3;135)(7;135)
\psline[linecolor=white,linewidth=4pt](3;-135)(7;-135)
\psline[linecolor=white,linewidth=4pt](3;-45)(7;-45)

\psline[linecolor=gray](3;45)(7;45)
\psline[linecolor=gray](3;135)(7;135)
\psline[linecolor=gray](3;-135)(7;-135)
\psline[linecolor=gray](3;-45)(7;-45)


\psecurve[linecolor=white,linewidth=4pt](7.5;110)(1.5;180)(7.5;-110)(8.7;-155)
(8.6;155)(7.5;110)(1.5;180)(7.5;-110)
\psecurve[linecolor=darkgreen](7.5;110)(1.5;180)(7.5;-110)(8.7;-155)
(8.6;155)(7.5;110)(1.5;180)(7.5;-110)

\rput(-5.5,0){$a$}
\rput(0,-5.5){$b$}
\rput(5.5,0){$c$}
\rput(0,5){$d$}

\rput(0,0){\textcolor{darkgray}{$T$}}

\end{pspicture}
\caption{$X_T^\prime$}\label{fig:SurfaceDecompX}
\end{subfigure}
\caption{The surface decomposition used in the proof of proposition~\ref{prop:fourendedHFT}}\label{fig:SurfaceDecomp}
\end{figure}

\begin{proof}
Let us consider relation \seqref{eqn:BD} first. The underlying sutured manifolds are the same after switching the roles of $R_-$ and $R_+$ on one side. This can be easily seen by pushing the two meridional sutures of the open tangle components through the tangle. Then, by \cite[proposition~2.14]{DecatSFH}, the sutured Floer homologies are identical, except that the $\Spinc$-gradings are opposite to each other. Now apply proposition~\sref{prop:HFTreverseorientI}.\\
In case~II, \seqref{eqn:AC} (without the tensor factors) follows from the same arguments. 
In case~I, relation~\seqref{eqn:AC} is an exercise in applying the surface decomposition formula for sutured Floer homology, see figure~\ref{fig:SurfaceDecomp}. Consider $X_T^c$. Let $N$ be a tubular neighbourhood of the union of~$R_-$ and the component of~$R_+$ corresponding to the $q$-strand. Let $S$ be the surface obtained as the intersection of~$N$ with the closure of $X_T^\prime:=X_T^c\smallsetminus N$. Note that $X_T^\prime$ is diffeomorphic to~$X_T$. We turn it into a balanced sutured manifold by adding a single suture on the boundary of the 3-ball, separating the $p$-ends from the $q$-ends. We get a decomposition
$$X_T^c\rightsquigarrow^S X_T^\prime\cup N,$$
which satisfies the conditions of \cite[proposition~8.6]{SurfaceDecomposition}. Thus,
$$\SFH(X_T^c)=\SFH(X_T^\prime)\otimes \SFH(N).$$
It is now straightforward to calculate $\SFH(N)$ which gives $V_q$. Thus,
$$\HFT(T,c)=\SFH(X_T^\prime)\otimes V_q$$
and similarly
$$\HFT(T,a)=\SFH(X_T^{\prime\prime})\otimes V_p,$$
where $X_T^{\prime\prime}$ agrees with $X_T^\prime$, except that the roles of $R_-$ and $R_+$ are interchanged. 
To get relation~\seqref{eqn:AC}, we now argue just as before.
\end{proof}

%% file: sections/2_HDsForTangles.tex
\section{Heegaard diagrams for tangles}\label{sec:HDsForTangles}

Throughout this section, let $T$ be an oriented tangle with $n$ open and $m$ closed components and $s$~a~site of~$T$. The complement of a tubular neighbourhood of the tangle in the closed 3-ball is denoted by $X_T$. In the following, we adapt the Heegaard Floer construction for knots and links \cite{Jake,OSHFK,OSHFL} to tangles, which gives us another, but equivalent definition of $\HFT$ to the one defined in the previous section. 

\begin{definition}\label{def:HDsfortangles}
A \textbf{Heegaard diagram for a tangle} consists of the following data:
\begin{itemize}
\item An oriented surface $\Sigma_g$ of genus $g$ with $2(n+m)$ boundary components, denoted by $\Z$, which are partitioned into $(n+m)$ pairs,
\item a set $\Ac$ of $(g+m)$ pairwise disjoint circles $\alpha_1,\dots, \alpha_{g+m}$ on $\Sigma_g$,
\item a set $\Aa$ of $2n$ pairwise disjoint arcs $\aaa_1,\dots,\aaa_{2n}$ on $\Sigma_g$ which are disjoint from~$\Ac$ and whose endpoints lie on  $\Z$, and
\item a set $\B$ of $(g+m+n-1)$ pairwise disjoint circles $\beta_1,\dots, \beta_{g+m+n-1}$ on $\Sigma_g$.
\end{itemize}
We write $\A:=\Ac\cup\Aa$ and impose the following conditions on the data above:
\begin{itemize}
\item Contracting all boundary components turns $\Aa$ into a single circle. So in particular, on each component of $\Z$, there are either no or exactly two endpoints of~$\Aa$.
\item The surface $S_{\Ac}(\Sigma_g)$ obtained by surgery along the curves in $\Ac$ is a disjoint union of $m$ annuli, each of whose boundary is a pair in $\Z$, and a 2-sphere with $2n$ boundary components. We denote the set of circles in $\Z$ which meet the $\alpha$-curves by $\Z^\alpha$.
\item The surface $S_{\B}(\Sigma_g)$ obtained by surgery along the curves in $\B$ is a disjoint union of $(n+m)$ annuli, each of whose boundary is a pair in $\Z$. 
\end{itemize}
\end{definition}
\begin{Remark}\label{rem:conventions1}
We can recover the tangle complement $X_T$ from this data by attaching 2-handles to $\Sigma_g\times [0,1]$ along $\Ac\times\{0\}$ and $\B\times\{1\}$; the $\alpha$-arcs then correspond to $\red S^1$ from definition~\sref{def:tangle}. Conversely, we can pick a self-indexing Morse function $f$ on $B^3$ which is identical to $\frac{3}{2}$ on a neighbourhood of $\red S^1$ and which has a single minimum and a single maximum on each closed component of $T$ and a single minimum on each open tangle component. Then, the Heegaard surface is equal to $f^{-1}(\frac{3}{2})$ modulo punctures at the tangle ends; the $\alpha$- and $\beta$-circles are the loci of points on this surface flowing from/to the index~1 and~2 critical points, respectively. 
Our \textbf{convention on the orientation of the Heegaard surface} is that its normal vector field (using the right-hand rule) points in the positive direction of the Morse function, i.\,e.\ in the direction of the $\beta$-curves. However, we usually draw the Heegaard surfaces such that this normal vector field points into the plane.
\end{Remark}
\begin{Remark}\label{Rem:TanglesAsSpecialBSMflds}
Given a tangle $T$, we can endow $X_T$ with the structure of a bordered sutured manifold as follows: Each closed component gets two oppositely oriented meridional circles and around each tangle end, we have a single suture such that the boundary of $B^3$ minus a neighbourhood of the tangle ends lies in $R_-$. Furthermore, the arcs ${\red S^1}\smallsetminus \partial T$ together with small neighbourhoods of the endpoints on $\mathcal{Z}$ constitute the arc diagram. Similarly, Heegaard diagrams for tangles can be viewed as bordered sutured Heegaard diagrams, see section~\ref{sec:TFHviaBSFH}, in particular definition~\ref{def:HDforBorderedSuturedMfdls} and figure~\ref{fig:BDHDoriginal}.
\end{Remark}

\begin{lemma}\label{lem:HeegaardMoves}
Every tangle \(T\) has a Heegaard diagram. Moreover, any two diagrams for the same tangle can be obtained from one another by a sequence of the following moves:
\begin{itemize}
\item an isotopy of an \(\alpha\)- or \(\beta\)-circle or an isotopy of an \(\alpha\)-arc relative to its endpoints,
\item a handleslide of a \(\beta\)-circle over another \(\beta\)-circle,
\item a handleslide of an \(\alpha\)-curve over an \(\alpha\)-circle and
\item stabilisation.
\end{itemize} 
\end{lemma}
\begin{proof}
In view of remark~\ref{Rem:TanglesAsSpecialBSMflds}, this is a special case of \cite[proposition~4.5]{Zarev}. The first part also follows from the next example.
\end{proof}
\begin{figure}[b]
\centering
\begin{subfigure}[b]{0.31\textwidth}
\centering
\psset{unit=0.2}
\begin{pspicture}(-10.3,-10.3)(10.3,10.3)
\psline(6;-45)(6;135)
\pscircle*[linecolor=white](0,0){0.7}
\psline(6;-135)(6;45)
\pscircle[linestyle=dotted](0,0){6}
\end{pspicture}
\vspace*{9pt}
\caption{A single crossing tangle}\label{fig:1crossingTforHD}
\end{subfigure}
\quad
\begin{subfigure}[b]{0.31\textwidth}
\centering
\psset{unit=0.2}
\begin{pspicture}(-10.3,-10.3)(10.3,10.3)
\psrotate(0,0){90}{

\psecurve[linecolor=blue](8,-0.2)(7,0)(2,2)(0,7)(-0.2,8)
\psecurve[linecolor=blue](-4,4)(0,7)(-8.3,8.3)(-7,0)(-4,4)
\psecurve[linecolor=blue](-8,0.2)(-7,0)(-2,-2)(0,-7)(0.2,-8)
\psecurve[linecolor=blue](4,-4)(0,-7)(8.3,-8.3)(7,0)(4,-4)

\psarc[linecolor=red](0,0){7}{0}{360}

\rput(7;45){\pscircle*[linecolor=white]{1}\pscircle[linecolor=darkgreen]{1}}
\rput(7;135){\pscircle*[linecolor=white]{1}\pscircle[linecolor=darkgreen]{1}}
\rput(7;225){\pscircle*[linecolor=white]{1}\pscircle[linecolor=darkgreen]{1}}
\rput(7;315){\pscircle*[linecolor=white]{1}\pscircle[linecolor=darkgreen]{1}}

}
\psdot(0,7)
\psdot(7,0)
\psdot(0,-7)
\psdot(-7,0)

\rput[b](5.5;-145){\textcolor{red}{$\Ac$}}
\rput[b](2;-45){\textcolor{blue}{$\B$}}
\rput[br](8;135){$\Z$}

\rput[br](0,7.5){$D$}
\rput[tl](7.5,0){$C$}
\rput[tl](0,-7.5){$B$}
\rput[br](-7.5,0){$A$}
\end{pspicture}
\caption{A Heegaard diagram for the tangle on the left}\label{fig:HDfor1crossing}
\end{subfigure}
\quad
\begin{subfigure}[b]{0.31\textwidth}
\centering
\psset{unit=0.2}
\begin{pspicture}(-5.1,-10.1)(5.1,10.1)
\psellipticarc[linestyle=dotted,dotsep=1pt,linecolor=darkgreen](0,-8)(5,2){0}{180}
\psellipticarc[linestyle=dotted,dotsep=1pt,linecolor=blue](0,0)(5,2){0}{180}
\psline[linecolor=red](4,-9.2)(4,6.8)
\psline[linecolor=red](-4,-9.2)(-4,6.8)
\psellipticarc[linecolor=darkgreen](0,-8)(5,2){180}{0}
\psline(4.93,-8)(4.93,8)
\psline(-4.93,8)(-4.93,-8)
\psellipse[linecolor=darkgreen](0,8)(5,2)
\psellipticarc[linecolor=blue](0,0)(5,2){180}{0}

\psellipse[linecolor=darkgreen](0,0)(0.5,0.6)

\psellipse[linecolor=darkgreen](0,-4)(0.5,0.6)

\psellipse[linecolor=red](0,-2)(2,5)
\psdot(1.98,-1.82)
\psdot(-1.98,-1.82)
\psdot(4,-1.2)
\psdot(-4,-1.2)
\end{pspicture}
\caption{A ladybug, see also \cite[figure~8]{BaldwinLevine}}\label{fig:ladybug}
\end{subfigure}
\caption{Two building blocks of tangle Heegaard diagrams for example~\ref{exa:HDforonecrossing}}\label{fig:HDBuildingBlocks}
\end{figure}
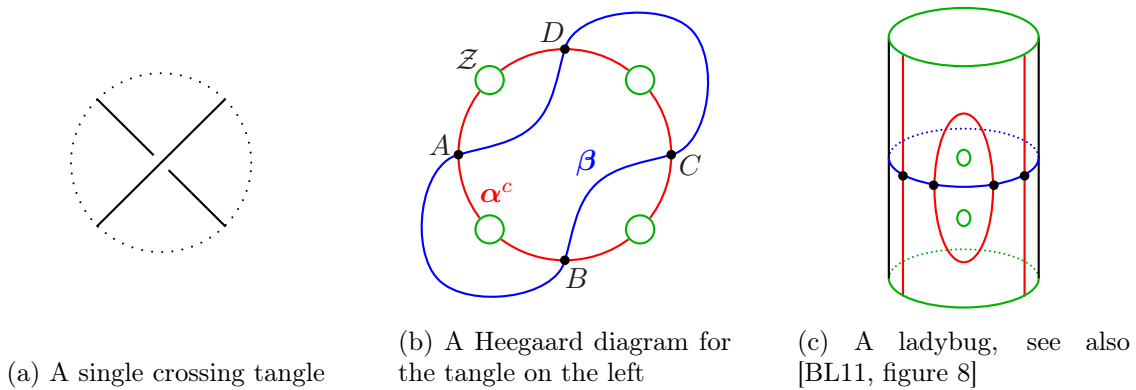
\begin{example}\label{exa:HDforonecrossing}
For a 1-crossing tangle, we draw the Heegaard diagram shown in figure~\ref{fig:HDfor1crossing}. From this, we can obtain a Heegaard diagram for any tangle without closed components as follows: We cut a tangle diagram of a given tangle up into 4-ended tangles with a single crossing each. Then, for each such component, we can use our Heegaard diagram from figure~\ref{fig:HDfor1crossing} and then glue these copies together along $\Z$ according to the tangle diagram. For tangles with closed components, we can do the same except that into each closed component, we insert a copy of the ``ladybug'' from figure~\ref{fig:ladybug}.
\end{example}
\begin{figure}[t]
\centering
\psset{unit=0.2}
\begin{subfigure}[b]{0.6\textwidth}
\centering
\begin{pspicture}(-2,-10.3)(34,10.3)
\psscalebox{1 -1}{
\psecurve[linestyle=dotted](32,3)(0,3)(0,-3)(32,-3)(32,3)(0,3)(0,-3)

\psecurve(-8,3)(0,-3)(8,3)(16,-3)
\psecurve(8,3)(16,-3)(24,3)(32,-3)

\psecurve(0,3)(8,-3)(16,3)(24,-3)
\psecurve(16,3)(24,-3)(32,3)(40,-3)

\pscircle*[linecolor=white](4,0){0.7}
\pscircle*[linecolor=white](12,0){0.7}
\pscircle*[linecolor=white](20,0){0.7}
\pscircle*[linecolor=white](28,0){0.7}

\psecurve(0,-3)(8,3)(16,-3)(24,3)
\psecurve(16,-3)(24,3)(32,-3)(40,3)

\psecurve(-8,-3)(0,3)(8,-3)(16,3)
\psecurve(8,-3)(16,3)(24,-3)(32,3)
}
\end{pspicture}
\caption{A rational tangle obtained from the 1-crossing tangle in figure~\ref{fig:1crossingTforHD} by three Dehn twists on the right}\label{fig:tangleforHDforRatTan}
\end{subfigure}
\quad
\begin{subfigure}[b]{0.33\textwidth}
\centering
\begin{pspicture}(-10.3,-10.3)(10.3,10.3)
\psscalebox{1 -1}{
\psarc[linecolor=red](0,0){8}{0}{360}

\SpecialCoor
\rput(8;45){\pscircle*[linecolor=white]{1}\pscircle[linecolor=darkgreen]{1}}
\rput(8;135){\pscircle*[linecolor=white]{1}\pscircle[linecolor=darkgreen]{1}}
\rput(8;225){\pscircle*[linecolor=white]{1}\pscircle[linecolor=darkgreen]{1}}
\rput(8;315){\pscircle*[linecolor=white]{1}\pscircle[linecolor=darkgreen]{1}}


\psecurve[linecolor=blue]%
(10;130)(10;140)(8;155)%
(8;-95)(11;-45)%
(10.5;45)(8;80)%
(8;-70)(10;-45)%
(9.5;45)(8;60)%
(9.5;-45)(8;-60)%
(8;70)(10;45)%
(10.5;-45)(8;-80)%
(5.5;-100)%
(8;120)(10;130)(10;140)(8;155)%

\psdot(8;155)
\psdot(8;120)
\psdot(8;80)
\psdot(8;70)
\psdot(8;60)
\psdot(8;-28)
\psdot(8;-60)
\psdot(8;-70)
\psdot(8;-80)
\psdot(8;-95)
}
\end{pspicture}
\caption{A genus 0 Heegaard diagram for the tangle in (a)}\label{fig:HDforHDforRatTan}
\end{subfigure}
\caption{Rational tangles have genus 0 Heegaard diagrams.}\label{fig:HDforRatTan}
\end{figure}
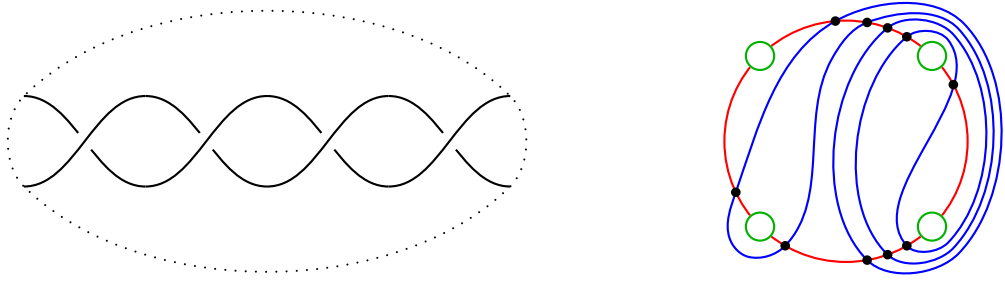
\begin{example}
For rational tangles, we can draw Heegaard diagrams on genus 0 surfaces. As illustrated in figure~\ref{fig:HDforRatTan}, this can be seen by performing Dehn twists on the Heegaard diagram for the 1-crossing tangle in figure~\ref{fig:HDfor1crossing}. In fact, a 4-ended tangle is rational iff it has a genus~0 Heegaard diagram. Indeed, a genus 0 Heegaard diagram for a 4-ended tangle has no $\alpha$-circles and just a single $\beta$-circle. By definition, we know that performing surgery along this $\beta$-circle gives us two cylinders, so it separates two punctures from the other two. 
\end{example}
In the following, let $\mathcal{H}=\mathcal{H}_T=(\Sigma_g,\Z,\A,\B)$ be a Heegaard diagram for $T$.
 
\begin{definition}\label{def:HDtoCFTbasics}
Let $\mathbb{T}=\mathbb{T}_\mathcal{H}$ be the set of tuples $\x=(x_1,\dots,x_{g+m+n-1})$ of points $x_1,\dots,x_{g+m+n-1}\in\A\cap\B$ such that there is exactly one point $x_i$ on each $\alpha$- and $\beta$-circle, and at most one point on each $\alpha$-arc. $\mathbb{T}$ will be the generating set of the chain module defined later on, so we call the elements in $\mathbb{T}$ \textbf{generators}.\\
Following the notation from definition~\sref{def:basic}, a \textbf{site} $s\in\mathbb{S}(T)$ corresponds to an $(n-1)$-element subset of $\Aa$. With each generator $\x\in\mathbb{T}$, we associate the site $s(\x)$ consisting of all those $\alpha$-arcs that are occupied by an intersection point in $\x$. We denote the set of all generators corresponding to a given site~$s$ by $\mathbb{T}^s$. Thus, we obtain a partition $$\mathbb{T}=\coprod_{s\in\mathbb{S}}\mathbb{T}^s.$$
We define $D$ to be the free Abelian group generated by the connected components of $\Sigma_g\smallsetminus(\A\cup\B\cup\Z )$, which we call \textbf{regions}. In other words,
$$D=D_{\mathcal{H}}:=H_2(\Sigma_g,\A\cup\B\cup\Z).$$
Elements of this group are called \textbf{domains}. Given two points $\x,\y\in\mathbb{T}$, we define $\pi_2(\x,\y)$ to be the subset of those domains $\phi$ which satisfy 
$$d(d\phi\cap\B)=\x-\y.$$
We call elements in $\pi_2(\x,\x)$ \textbf{periodic domains}. Note that this does not depend on the choice of $\x\in\mathbb{T}$.
Furthermore, let 
$$\pi^{\partial}_2(\x,\y):=\{\phi\in\pi_2(\x,\y)\vert\Z\cap\,\phi=\emptyset\}.$$
A Heegaard diagram is called \textbf{admissible} if 
every non-zero periodic domain in $\pi^{\partial}_2(\x,\y)$ has positive and negative multiplicities. 
\end{definition}
\begin{lemma}\label{lem:HeegaardMovesAdmissibility}
Every tangle diagram can be made admissible by isotopies of \(\B\). Furthermore, any two such diagrams for the same tangle can be transformed into one another by a sequence of Heegaard moves from lemma~\ref{lem:HeegaardMoves} through admissible diagrams.
\end{lemma}
\begin{proof}
In light of remark~\ref{Rem:TanglesAsSpecialBSMflds}, this is a special case of \cite[proposition~4.11 and corollary~4.12]{Zarev}. Note that our terminology differs slightly from Zarev's: He does not include regions near basepoints in $\pi_2(\x,\y)$, so in our special case, his definition of $\pi_2(\x,\y)$ coincides with our $\pi^{\partial}_2(\x,\y)$. Thus, the distinction in his terminology between periodic and provincial periodic domains becomes irrelevant. 
\end{proof}

\begin{lemma}\label{lem:noadmissibilityissues}
There is an isomorphism \(\pi_2(\x,\x)\cong H_2(X_T,\Z\cup\Aa;\mathbb{Z})\cong\mathbb{Z}^{n+2m+1}\). Furthermore, \(\pi_2^\partial(\x,\x)=\mathbb{Z}^{m}\); in particular, any Heegaard diagram for a tangle without closed components is admissible.
\end{lemma}
\begin{proof}
By attaching 2-handles to the $\alpha$- and $\beta$-circles we see that
$$\pi_2(\x,\x)\cong H_2(X_T,\Z\cup\Aa ;\mathbb{Z}).$$ 
This group is generated by annuli, one for each open and two for each closed component, and two discs, one at the front and the other at the back, the only relation being that the sum of the annuli is equal to the sum of the discs. This proves the first statement.\\
The second follows from the first: Let us write a given periodic domain in~$\pi_2^\partial(\x,\x)$ as a linear combination of the basic periodic domains described above such that the coefficient of one of the two discs, say, is zero. Then the coefficients of the two annuli of each closed component must be opposite and all other coefficients must be zero. In fact, the differences of the two annuli for each closed component form a basis of $\pi_2^\partial(\x,\x)$.
\end{proof}
\begin{lemma}\label{lem:pixynonempty}
\(\pi_2(\x,\y)\) is non-empty for all pairs \((\x,\y)\in\mathbb{T}^2\).
 \end{lemma}
\begin{proof}
For any pair $(\x,\y)\in\mathbb{T}^2$, there exists a 1-cycle $\gamma$ in $C_1(\A\cup\B\cup\Z)$ such that $d(\gamma\cap\B)=\x-\y$. Indeed: First, we choose a 1-chain on $C_1(\B)$ with this property. Then, since there is exactly one intersection point on every $\alpha$-circle for each $\x$ and $\y$, we can add 1-chains in $C_1(\Ac)$ such that the boundary of the new 1-chain lies on the $\alpha$-arcs only. But since $\gamma$ is allowed to have $\Z$-components, we can get rid of these intersection points, too, and obtain our cycle $\gamma$.\\
Next, we can add $\Z$-cycles to $\gamma$ such that the resulting 1-cycle is 0 in $H_1(X_T)$. Adding $\alpha$- and $\beta$-cycles gives us another 1-cycle $\gamma'$ which is 0 in $H_1(\Sigma_g)$ and also satisfies $d(\gamma'\cap\B)=\y-\x$. So we are done.
\end{proof}
Our next goal is to define a relative Alexander grading on generators. We do this by counting $\Z$-components of domains which connect two generators. 
\begin{definition}\label{def:AlexGradingOnDomains}
By definition, the endpoints of the $\alpha$-arcs divide each circle in $\Z^\alpha$ into two components, which we suggestively call front and back component. Let $\Sigma_g^f$, $(\A\cup\B\cup\Z)^f$ and $\A^f$ denote the spaces obtained from $\Sigma_g$, $(\A\cup\B\cup\Z)$ and $\A$ respectively by contracting the back component of each circle in $\Z^\alpha$ to a point. 
Note that $\B\cap\Z=\emptyset$, $\Ac\cap\Z=\emptyset$ and that the images of $\alpha$-arcs become a single circle in $\A^f$. Let $f:(\A\cup\B\cup\Z)\rightarrow(\A\cup\B\cup\Z)^f$ be the quotient map and $\partial: D\rightarrow H_1(\A\cup\B\cup\Z)$ the boundary map of the long exact sequence of the pair $(\Sigma_g,\A\cup\B\cup\Z)$. Now, consider the following diagram, where $\iota$, $\iota^f$, $i$ and $j$ are induced by inclusions, and $\text{pr}_2$ is the projection onto the second summand:
$$
\begin{tikzcd}[column sep=0.63cm,row sep=1.2cm]    
D\arrow{r}{\partial} &H_1(\A\cup\B\cup\Z) \arrow[color=white]{dr}[description]{\quad\textcolor{black}{\square}}
\arrow{r}{\iota}\arrow{d}{f} & H_1(\Sigma_g)\arrow[color=white,near start]{drr}[description]{\textcolor{black}{\square}}\arrow{d}{\cong}\arrow{rr}{j}& &H_1(\Sigma_g)/\langle\Ac,\B\rangle&\hspace{-0.74cm}\cong H_1(X_T)\\
& H_1((\A\cup\B\cup\Z)^f) \arrow{r}{\iota^f}\arrow[swap]{d}{\cong}& H_1(\Sigma_g^f)\arrow[swap,bend right=5]{urr}{j^{f}}\arrow[color=white, very near start]{drrr}[description]{\textcolor{black}{\square}}&&\phantom{X}\\
& H_1(\A^f\cup\B)\oplus H_1(\Z^f)\arrow{r}{\text{pr}_2}& H_1(\Z^f)\arrow[bend right=35,swap,start anchor=east,end anchor=south]{uurr}{c=j^f\circ i}\arrow{u}{i}&&&\phantom{X}
\end{tikzcd}
$$
Note that the orientation of~$T$ induces an orientation of the meridians using the right-hand rule, which gives rise to a canonical identification $H_1(X_T)\cong \mathbb{Z}^{n+m}$. We define 
$$A^f:D\rightarrow H_1(X_T)\cong\mathbb{Z}^{n+m}$$
as the composition $c\circ pr_2\circ f\circ\partial$. 
Similarly, by contracting the front components of $\Z^\alpha$, we obtain a homomorphism
$$A^b:D\rightarrow \mathbb{Z}^{n+m}.$$
\end{definition}
\begin{lemma}\label{lem:AlexFconstant}
\(A^f\) and \(A^b\) are constant on \(\pi_2(\x,\y)\) for all pairs \((\x,\y)\in\mathbb{T}^2\).
\end{lemma}
\begin{proof}
It suffices to show that $A^f$ and $A^b$ vanish on periodic domains. But this is obvious from the description of the periodic domains in (the proof of) lemma \ref{lem:noadmissibilityissues}. The annuli have cancelling $\Z$-components of the corresponding tangle component and the $\Z$-components of the two discs are the sums of all front/back $\Z$-components.
\end{proof}

Combining lemmas \ref{lem:pixynonempty} and \ref{lem:AlexFconstant} enables us to define a relative grading on $\mathbb{T}$.
\begin{definition} \label{def:AlexgradingonGenerators}
The homomorphism $A=A^f+A^b: D\rightarrow \mathbb{Z}^{n+m}$ induces a relative grading $A:\mathbb{T}\rightarrow \mathbb{Z}^{n+m}$ by setting $$A(\y)-A(\x)=A(\phi)\quad\text{for } \phi\in\pi_1(\x,\y).$$ 
We call $A$ the \textbf{Alexander grading}. Let $A^r$ be the composition of $A$ with the map $\mathbb{Z}^{n+m}\rightarrow\mathbb{Z}$ that adds all components. (``$r$'' stands for ``reduced''.)
We often introduce shifts by half-integers to achieve a certain symmetry which mimics the normalisation of~$\nabla_T^s$; thus the Alexander grading appears to be a relative $(\frac{1}{2}\mathbb{Z})^{n+m}$-grading. 
\end{definition}

\begin{lemma}\label{lem:pidxynonempty}
 \(\pi^\partial_2(\x,\y)\) is non-empty for all pairs \((\x,\y)\in\mathbb{T}^2\), iff \(\x\) and \(\y\) are in the same Alexander grading and belong to the same site.
 \end{lemma}
\begin{proof}
The only-if part is clear. The opposite direction follows from a refinement of the proof of \ref{lem:pixynonempty}: We can now get a 1-cycle~$\gamma$ in $C_1(\A\cup\B)$ such that $d(\gamma\cap\B)=\x-\y$, because the generators belong to the same site. This 1-cycle is already zero in $H_1(X_T)$, since the generators are in the same Alexander grading. Then we might have to add $\alpha$- and $\beta$-cycles as before and we are done.
\end{proof} 

To define the differentials in our chain complexes, we count the number of points in 1-dimensional moduli spaces of holomorphic curves modulo an $\mathbb{R}$-action. The formal dimension of these moduli spaces is called the \textbf{Maslov index}. It can be computed combinatorially, as shown in \cite[corollary~4.10]{LipshitzCyl}. We take this combinatorial formula as a definition.
\begin{definition}\label{def:homgrading}
Let $\phi\in\pi_2(\x,\y)$ for some $\x,\y\in\mathbb{T}$. We define the Maslov index by
$$
\mu(\phi)=e(\phi)+m_{\x}(\phi)+m_{\y}(\phi),
$$
where $e(\phi)$ is the Euler measure of $\phi$ and $m_{\x}(\phi)$ and $m_{\y}(\phi)$ are the multiplicities of $\phi$ at $\x$ and $\y$, respectively. More explicitly, given a region $\psi$ of the Heegaard diagram, let $m_\psi(\phi)$ denote the coefficient of $\psi$ in $\phi$. Then 
$$e(\phi)=\sum_{\text{regions~} \psi}m_\psi(\phi)(\chi(\psi)-\tfrac{1}{4}\#\{\text{acute corners of }\psi\}+\tfrac{1}{4}\#\{\text{obtuse corners of }\psi\}).$$
Furthermore, for any $\x\in\mathbb{T}$, let
$$m_{\x}(\phi)=\sum_{x_i\in\x}m_{x_i}(\phi),$$ 
where $m_x(\phi)$ is the average of the $m_{\psi_i}(\phi)$ in the four quadrants $\psi_1,\dots,\psi_4$ at $x$. 
\end{definition}
\begin{lemma}\label{lem:muisadditive}
Given \(\phi\in\pi_2(\x,\y)\) and \(\psi\in\pi_2(\y,\z)\), \(\mu(\phi)+\mu(\psi)=\mu(\phi+\psi)\).
\end{lemma}
\begin{proof}
This follows from basically the same arguments as \cite[theorems~3.1 and~3.3]{Sarkar06}. We give some details nonetheless. First of all, note that the Euler measure is additive. Hence, all we need to show is that 
$$m_{\x}(\phi)+m_{\y}(\phi)+m_{\y}(\psi)+m_{\z}(\psi)=m_{\x}(\phi+\psi)+m_{\z}(\phi+\psi).$$
This simplifies to 
$$m_{\y}(\phi)+m_{\y}(\psi)=m_{\x}(\psi)+m_{\z}(\phi).$$
Theorem 3.1 from \cite{Sarkar06} for $n=i=2$, $\eta^1=\A$ and $\eta^2=\B$ gives us
\begin{align*}
m_{\z}(\phi)-m_{\y}(\phi)&=\partial\phi \cdot\partial_{\B}(\psi)\quad\text{and similarly}\\
m_{\y}(\psi)-m_{\x}(\psi)&=\partial\psi \cdot\partial_{\B}(\phi),
\end{align*}
where the product $\cdot$ denotes the ``average'' intersection number from \cite{Sarkar06}. 
So we need to see that 
$$\partial\psi \cdot\partial_{\B}(\phi)+\partial_{\B}(\psi)\cdot\partial\phi=0. $$
The boundaries of the domains lie in $\A\cup\B\cup\Z$. However, $\B\cap\Z=\emptyset$, so the left-hand side equals $\partial_{\A}(\psi) \cdot\partial_{\B}(\phi)+\partial_{\B}(\psi)\cdot\partial_{\A}(\phi)=\partial_{\A\cup\B}(\psi) \cdot\partial_{\A\cup\B}(\phi)$. To see that this is zero, we modify the Heegaard surface by contracting all boundary components. Then the left-hand side is equal to $\partial(\psi) \cdot\partial(\phi)$, and this is indeed zero.
\end{proof}

\begin{lemma}\label{lem:homgradingconstant}
\(\mu\) is constant on \(\pi_2(\x,\y)\) for all pairs \((\x,\y)\in\mathbb{T}^2\). 
\end{lemma}
\begin{proof}
Applying the previous lemma to $\z=\y$, we see that all we need to show is that
$$\mu(\phi)=e(\phi)+2m_{\y}(\phi)$$
vanishes for all periodic domains $\phi\in\pi_2(\y,\y)$. In fact, it suffices to show this for every elementary periodic domain $\phi$ from the proof of lemma \ref{lem:noadmissibilityissues}. Consider $\phi$ as a subsurface of $\Sigma_g\times\{0\}$ or $\Sigma_g\times\{1\}$, depending on whether it corresponds to an annulus or disc from performing surgery along $\alpha$-circles or $\beta$-circles. The annulus or the disc is obtained from $\phi$ by performing surgery along the respective circles or attaching single discs to them. The former reduces the Euler measure of $\phi$ by~2, the latter by~1. For an annulus, this contribution is cancelled by the term $2m_{\y}(\phi)$, since every $\alpha$- and $\beta$-circle is occupied by exactly one intersection point of $\y$. So in this case, we have $\mu(\phi)=\mu(\text{annulus})=0$. For a disc, we also need to take into account that $\y$ occupies $(n-1)$ $\alpha$-arcs. In this case, $2m_{\y}(\phi)$ contributes $(n-1)$. On the other hand, the disc has $4n$ corners, so $e(\text{disc})=1-n$.
\end{proof}
Combining lemmas \ref{lem:muisadditive} and \ref{lem:homgradingconstant}, we can now define a relative grading on generators induced by the Maslov index $\mu$, just as for the Alexander grading.
\begin{definition}
The \textbf{$\delta$-grading} on generators is a relative $\tfrac{1}{2}\mathbb{Z}$-grading defined by
$$\delta(\y)-\delta(\x)=\mu(\phi), \quad\text{where }\phi\in\pi_2(\x,\y).$$
We also define a relative $\mathbb{Z}$-grading, the \textbf{homological grading}, by
$$h(\y)-h(\x)=h(\phi):=\tfrac{1}{2}A^r(\phi)-\mu(\phi), \quad\text{where }\phi\in\pi_2(\x,\y).$$
In short,
$$h=\tfrac{1}{2}A^r-\delta.$$
\end{definition}
\begin{lemma}
	The homological grading is well-defined.
\end{lemma}
\begin{proof}
The homological grading is \textit{a priori} a relative $\frac{1}{2}\mathbb{Z}$-grading. However, it is clear from the alternative formula for $\mu$ from~\cite[section~2]{Sarkar06}, that $\mu$ is an integer for a closed Heegaard surface. In our case, we can easily obtain a closed surface by contracting the boundary components as in the proof of lemma~\ref{lem:muisadditive}. Thus $2\mu(\phi)\equiv A(\phi)\mod 2$ for any $\phi\in\pi_2(\x,\y)$. 
\end{proof}
\begin{Remark}\label{rem:conventions2}
When comparing these gradings with those in link Floer homology, note that we are using a Heegaard surface with punctures instead of marked points. For two-ended tangles, our conventions agree with those in \cite[section~3.1, equations~(3.2)--(3.4)]{BaldwinLevine}, noting that tangle ends that point into the 3-ball are represented by $X$s and outgoing ones by $O$s. Although we have three different gradings, we call their union the \textbf{bigrading}, since any one of them is determined by the other two.
\end{Remark}

\begin{definition}\label{def:HFTiswelldefandinvariant}
Let $T$ be a tangle and $\mathcal{H}_T$ a Heegaard diagram for $T$. We define a chain module $\CFT(\mathcal{H}_T)$ as follows. The underlying $\mathbb{Z}$-module is freely generated by the elements in~$\mathbb{T}$. It carries three gradings, the Alexander grading and the homological grading and the $\delta$-grading, induced by the gradings on generators. We write $\CFT(\mathcal{H}_T,s)$ for the submodule generated by those elements in $\mathbb{T}^s$ and $\CFT(\mathcal{H}_T,s,a)$ for the submodule generated by those elements in $\mathbb{T}^s$ of Alexander grading $a\in\mathbb{Z}^{n+m}$. The differential on $\CFT(\mathcal{H}_T)$ is given by
$$\partial \x=\sum_{\y\in \mathbb{T}}\sum_{\substack{\phi\in\pi^{\partial}_2(\x,\y)\\ \mu(\phi)=1}}\#\widehat{\mathcal{M}}(\phi)\cdot\y,$$
where $\#\widehat{\mathcal{M}}(\phi)$ denotes the signed count of holomorphic curves associated with $\phi$, see for example \cite{LipshitzCyl}. Note that the sum is over domains in $\pi^{\partial}_2(\x,\y)$, i.\,e.\ those that avoid~$\Z$. Since there are no such domains between generators in distinct Alexander gradings or sites, the chain module $\CFT(\mathcal{H}_T)$ admits a splitting into summands $\CFT(\mathcal{H}_T,s,a)$. 
\end{definition}
\begin{theorem}\label{thm:HFTiswelldefandinvariant}
	The differential on \(\CFT(T,s)\) is well-defined. Furthermore, the bigraded chain homotopy type of \((\CFT(\mathcal{H}_T,s),\partial)\) is an invariant of the tangle \(T\) and it agrees with \(\CFT(T,s)\) from definition~\sref{def:HFTfromSFH}.
\end{theorem}

Furthermore, we can now easily re-prove the following result.

\begin{lemma}\label{lem:Eulercharagreeswithnabla}
The graded Euler characteristic of the chain module \(\CFT(\mathcal{H}_T,s)\) coincides with the polynomial invariant \(\nabla^s_T\) up to normalisation and additional factors for closed components as in theorem~\sref{thm:HFTfromSFH}.
\end{lemma}

\begin{Remark}\label{rem:absolutegradings}
All three gradings on $\HFT$ are relative. One can probably fix an absolute Alexander grading which agrees with the natural normalisation of the polynomial tangle invariants. For example, as long as $\nabla_T^s\not\equiv0$, we could simply ask for the Euler characteristic of $\CFT(\mathcal{H}_T,s)$ to agree with $\nabla_T^s$. For tangles in more general 3-manifolds, one could also use an absolute $\Spinc$-grading, or more explicitly, non-vanishing vector fields, see \cite{absolutegrading1} and \cite{absolutegrading2}. For 4-ended tangles, we can cheat like we usually do when computing gradings on $\HFK$: We simply use the symmetry relations for opposite sites (proposition~\ref{prop:fourendedHFT}) to fix an absolute Alexander grading such that the graded Euler characteristic agrees with $\nabla_T^s$ up to multiplication by $(\pm1)$.
\end{Remark}
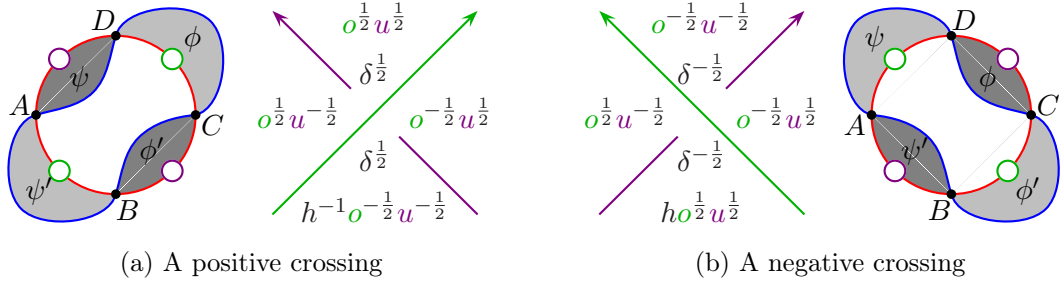
\begin{figure}[t]
\centering
\begin{subfigure}[b]{0.49\textwidth}\centering
{\psset{unit=0.15}
\begin{pspicture}(-10.3,-10.3)(10.3,10.3)

\pscustom*[linecolor=lightgray,linewidth=0pt]{\psecurve(4,4)(7,0)(8.3,8.3)(0,7)(4,4)}
\pscustom*[linecolor=white,linewidth=1pt]{
\psarc(0,0){7}{0}{45}
\psarc(0,0){7}{45}{90}
\psarc(7,7){7}{180}{-90}}

\pscustom*[linecolor=gray,linewidth=0pt]{
\psarc(0,0){7}{-90}{-45}
\psarc(0,0){7}{-45}{0}
\psecurve(-0.2,-8)(0,-7)(2,-2)(7,0)(8,-0.2)}

\pscustom*[linecolor=gray,linewidth=0pt]{
\psarc(0,0){7}{90}{135}
\psarc(0,0){7}{135}{180}
\psecurve(0.2,8)(0,7)(-2,2)(-7,0)(-8,-0.2)}

\pscustom*[linecolor=lightgray,linewidth=0pt]{\psecurve(-4,-4)(-7,0)(-8.3,-8.3)(0,-7)(-4,-4)}
\pscustom*[linecolor=white,linewidth=1pt]{
\psarc(0,0){7}{180}{-90}
\psarc(-7,-7){7}{0}{90}}


\psecurve[linecolor=blue](-0.2,-8)(0,-7)(2,-2)(7,0)(8,-0.2)
\psecurve[linecolor=blue](4,4)(7,0)(8.3,8.3)(0,7)(4,4)
\psecurve[linecolor=blue](0.2,8)(0,7)(-2,2)(-7,0)(-8,-0.2)
\psecurve[linecolor=blue](-4,-4)(-7,0)(-8.3,-8.3)(0,-7)(-4,-4)

\psarc[linecolor=red](0,0){7}{0}{360}

\SpecialCoor
\rput(7;45){\pscircle*[linecolor=white]{1}\pscircle[linecolor=darkgreen]{1}}
\rput(7;135){\pscircle*[linecolor=white]{1}\pscircle[linecolor=violet]{1}}
\rput(7;225){\pscircle*[linecolor=white]{1}\pscircle[linecolor=darkgreen]{1}}
\rput(7;315){\pscircle*[linecolor=white]{1}\pscircle[linecolor=violet]{1}

}
\psdot(0,7)
\psdot(7,0)
\psdot(0,-7)
\psdot(-7,0)

\SpecialCoor
\rput(9.5;45){$\phi$}
\rput(4.5;135){$\psi$}
\rput(4.5;-45){$\phi'$}
\rput(9.5;-135){$\psi'$}

\rput[br](0,7.5){$D$}
\rput[tl](7.5,0){$C$}
\rput[tl](0,-7.5){$B$}
\rput[br](-7.5,0){$A$}
\end{pspicture}
}
{\psset{unit=1.5}
\begin{pspicture}(-1.05,-1.05)(1.05,1.05)
\psline[linecolor=violet]{->}(0.9,-0.9)(-0.9,0.9)
\pscircle*[linecolor=white](0,0){0.3}
\psline[linecolor=darkgreen]{->}(-0.9,-0.9)(0.9,0.9)
\uput{0.7}[90](0,0){$\textcolor{darkgreen}{o}^{\frac{1}{2}} \textcolor{violet}{u}^{\frac{1}{2}}$}
\uput{0.3}[180](0,0){$\textcolor{darkgreen}{o}^{\frac{1}{2}} \textcolor{violet}{u}^{-\frac{1}{2}}$}
\uput{0.7}[-90](0,0){$h^{-1}\textcolor{darkgreen}{o}^{-\frac{1}{2}} \textcolor{violet}{u}^{-\frac{1}{2}}$}
\uput{0.3}[0](0,0){$\textcolor{darkgreen}{o}^{-\frac{1}{2}} \textcolor{violet}{u}^{\frac{1}{2}}$}
\uput{0.25}[90](0,0){$\delta^{\frac{1}{2}}$}
\uput{0.25}[-90](0,0){$\delta^{\frac{1}{2}}$}
\end{pspicture}
}
\caption{A positive crossing}
\end{subfigure}
\begin{subfigure}[b]{0.49\textwidth}\centering
{\psset{unit=1.5}
\begin{pspicture}(-1.05,-1.05)(1.05,1.05)
\psline[linecolor=violet]{->}(-0.9,-0.9)(0.9,0.9)
\pscircle*[linecolor=white](0,0){0.3}

\psline[linecolor=darkgreen]{->}(0.9,-0.9)(-0.9,0.9)
\uput{0.7}[90](0,0){$\textcolor{darkgreen}{o}^{-\frac{1}{2}} \textcolor{violet}{u}^{-\frac{1}{2}}$}
\uput{0.3}[180](0,0){$\textcolor{darkgreen}{o}^{\frac{1}{2}} \textcolor{violet}{u}^{-\frac{1}{2}}$}
\uput{0.7}[-90](0,0){$h\textcolor{darkgreen}{o}^{\frac{1}{2}} \textcolor{violet}{u}^{\frac{1}{2}}$}
\uput{0.3}[0](0,0){$\textcolor{darkgreen}{o}^{-\frac{1}{2}} \textcolor{violet}{u}^{\frac{1}{2}}$}
\uput{0.25}[90](0,0){$\delta^{-\frac{1}{2}}$}
\uput{0.25}[-90](0,0){$\delta^{-\frac{1}{2}}$}
\end{pspicture}}
{\psset{unit=0.15}
\begin{pspicture}(-10.3,-10.3)(10.3,10.3)

\pscustom*[linecolor=lightgray,linewidth=0pt]{\psecurve(-4,4)(-7,0)(-8.3,8.3)(0,7)(-4,4)}
\pscustom*[linecolor=white,linewidth=1pt]{
\psarcn(0,0){7}{180}{135}
\psarcn(0,0){7}{135}{90}
\psarc(-7,7){7}{-90}{0}}

\pscustom*[linecolor=gray,linewidth=0pt]{
\psarcn(0,0){7}{-90}{-135}
\psarcn(0,0){7}{-135}{-180}
\psecurve(0.2,-8)(0,-7)(-2,-2)(-7,0)(-8,-0.2)}

\pscustom*[linecolor=gray,linewidth=0pt]{
\psarcn(0,0){7}{90}{45}
\psarcn(0,0){7}{45}{0}
\psecurve(-0.2,8)(0,7)(2,2)(7,0)(8,-0.2)}

\pscustom*[linecolor=lightgray,linewidth=0pt]{\psecurve(4,-4)(7,0)(8.3,-8.3)(0,-7)(4,-4)}
\pscustom*[linecolor=white,linewidth=1pt]{
\psarcn(0,0){7}{0}{-45}
\psarcn(0,0){7}{-45}{-90}
\psarc(7,-7){7}{90}{180}}


\psecurve[linecolor=blue](0.2,-8)(0,-7)(-2,-2)(-7,0)(-8,-0.2)
\psecurve[linecolor=blue](-4,4)(-7,0)(-8.3,8.3)(0,7)(-4,4)
\psecurve[linecolor=blue](-0.2,8)(0,7)(2,2)(7,0)(8,-0.2)
\psecurve[linecolor=blue](4,-4)(7,0)(8.3,-8.3)(0,-7)(4,-4)

\psarc[linecolor=red](0,0){7}{0}{360}

\SpecialCoor
\rput(7;45){\pscircle*[linecolor=white]{1}\pscircle[linecolor=violet]{1}}
\rput(7;135){\pscircle*[linecolor=white]{1}\pscircle[linecolor=darkgreen]{1}}
\rput(7;225){\pscircle*[linecolor=white]{1}\pscircle[linecolor=violet]{1}}
\rput(7;315){\pscircle*[linecolor=white]{1}\pscircle[linecolor=darkgreen]{1}

}
\psdot(0,7)
\psdot(7,0)
\psdot(0,-7)
\psdot(-7,0)

\SpecialCoor
\rput(4.5;45){$\phi$}
\rput(9.5;135){$\psi$}
\rput(9.5;-45){$\phi'$}
\rput(4.5;-135){$\psi'$}

\rput[bl](0,7.5){$D$}
\rput[bl](7.5,0){$C$}
\rput[tr](0,-7.5){$B$}
\rput[tr](-7.5,0){$A$}

\end{pspicture}
}
\caption{A negative crossing}
\end{subfigure}
\caption{The gradings of the generators of the two 1-crossing tangles. The over-strand is coloured by $\textcolor{darkgreen}{o}$ and the under-strand by $\textcolor{violet}{u}$. Compare this to the Alexander codes from figure~\ref{figAlexCodesForNabla}. Our conventions agree with~\cite[figure~5 and~6]{OSrescube} and \cite[figure~10]{BaldwinLevine}.}\label{fig:AlexCodesForOneCrossingsWithDelta}
\end{figure}
\begin{proof}[Proof of lemma~\ref{lem:Eulercharagreeswithnabla}]
We first calculate the gradings of the generators for the 1-crossing diagrams, see figure~\ref{fig:AlexCodesForOneCrossingsWithDelta}.
In each case, we have four connecting domains $\psi$, $\phi$, $\psi'$ and~$\phi'$. The $\delta$-grading of all these domains is $+\frac{1}{2}$. This gives us the correct $\delta$-grading on generators, noting that the normal vector field of the Heegaard diagram, determined by the right-hand rule, points into the plane. (For example, for the positive crossing, $\psi$ is in~$\pi_2(A,D)$ and the $\delta$-grading increases along this domain by~$+\frac{1}{2}$.) Using the right-hand rule convention from definition~\ref{def:AlexGradingOnDomains}, we similarly obtain the correct Alexander gradings. This determines the homological grading. \pagebreak[3]\\
For a general tangle, we consider the Heegaard diagram induced by a tangle diagram as discussed in example~\ref{exa:HDforonecrossing}. Then, additivity of the gradings shows that the Alexander grading of a generator in the whole diagram is the sum of the gradings in the local diagrams at the crossings. For each closed component, we need to insert a ladybug into the Heegaard diagram, see figure~\ref{fig:ladybug}; this multiplies the number of generators by two, since there are two intersection points of the $\alpha$-circle in a ladybug. It is straightforward to compute the grading difference between corresponding generators of the two intersection points: the $\delta$-gradings agree and the Alexander gradings differ by~2. Hence, we get an extra factor $(t-t^{-1})$ in the decategorified invariant.
\end{proof}

\begin{Remark}
For tangles without closed components, the Mathematica program \cite{APT.m} explicitly computes the generators of the categorified tangle invariant from a standard Heegaard diagram as in the proof above. For tangles with closed components, we need to add the factor $(t-t^{-1})$ for each closed component. 
\end{Remark}

\begin{proof}[Proof of theorem \ref{thm:HFTiswelldefandinvariant}]
The theorem follows from the usual arguments in Heegaard Floer homology. In fact, the identification of every component $\CFT(\mathcal{H}_T,s)$ of $\CFT(\mathcal{H}_T)$ with $\CFT(T,s)$ implies most of the result. The main idea is to modify our Heegaard diagram~$\mathcal{H}_T$ in the following way, as illustrated in figure~\ref{fig:relationHFTandSFH}. \\
Let us consider a 2-torus with a fixed longitude and $2n$ disjoint meridians. Puncture the torus $2n$ times along the longitude such that any two meridians are no longer homotopic. We consider the remaining segments of the longitude as $\alpha$-arcs and the meridians as $\beta$-circles. Note that each $\beta$-circle intersects exactly one $\alpha$-arc in a single point and there are exactly $2n$ connected components in their complement on the punctured torus. We place a puncture in each of these components. Finally, we attach the (now $4n$-punctured) torus to $\Sigma$ in such a way that each $\alpha$-arc in the torus closes an $\alpha$-arc in the Heegaard surface $\Sigma_g$ for our tangle. This gives us a sutured Heegaard diagram $\overline{\mathcal{H}}$ consisting of a $2n$-punctured surface $\overline{\Sigma}$ with $(g+2n)$ $\alpha$-circles and $(g+2n+n-1)$ $\beta$-circles. However, $\overline{\mathcal{H}}$ is not balanced unless $n=1$. \\
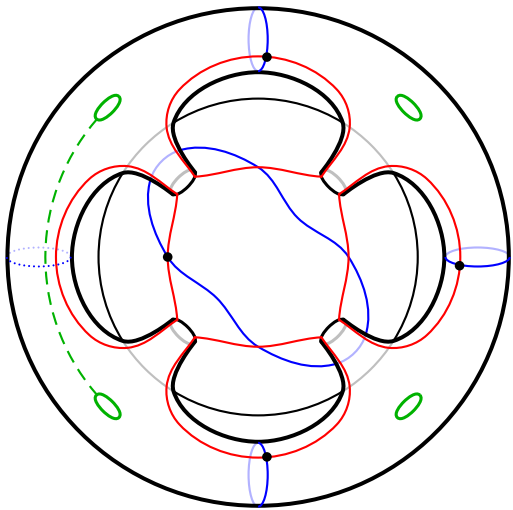
\begin{figure}[t]
\centering
\psset{unit=0.35}
\begin{pspicture}(-9.55,-9.55)(9.55,9.55)

\psecurve[linecolor=blue](5;-36)(5;-54)(3.4;-90)(-2.2;45)(3.4;180)(5;144)(5;126)
\psecurve[linecolor=blue](5;144)(5;126)(3.4;90)(2.1;45)(3.4;0)(5;-36)(5;-54)
\psecurve[linecolor=blue!30!white](3.4;180)(5;144)(5;126)(3.4;90)
\psecurve[linecolor=blue!30!white](3.4;0)(5;-36)(5;-54)(3.4;-90)


\rput{0}(0,0){\psarc(0,0){6}{-32}{32}}
\rput{90}(0,0){\psarc(0,0){6}{-32}{32}}
\rput{-90}(0,0){\psarc(0,0){6}{-32}{32}}
\rput{180}(0,0){\psarc(0,0){6}{-32}{32}}

\rput{45}(0,0){\psarc[linecolor=lightgray](0,0){6}{-13}{13}}
\rput{-45}(0,0){\psarc[linecolor=lightgray](0,0){6}{-13}{13}}
\rput{135}(0,0){\psarc[linecolor=lightgray](0,0){6}{-13}{13}}
\rput{-135}(0,0){\psarc[linecolor=lightgray](0,0){6}{-13}{13}}

\rput{45}(4;45){
\psecurve[linewidth=1.15pt,linecolor=lightgray](-0.25,0)(0,-0.6)(0.25,0)(0,0.6)(-0.25,0)
\psecurve[linewidth=1.15pt](0.25,0)(0,0.6)(-0.25,0)(0,-0.6)(0.25,0)
}
\rput{-45}(4;135){
\psecurve[linewidth=1.15pt](-0.25,0)(0,-0.6)(0.25,0)(0,0.6)(-0.25,0)
\psecurve[linewidth=1.15pt,linecolor=lightgray](0.25,0)(0,0.6)(-0.25,0)(0,-0.6)(0.25,0)}
\rput{45}(4;-135){
\psecurve[linewidth=1.15pt](-0.25,0)(0,-0.6)(0.25,0)(0,0.6)(-0.25,0)
\psecurve[linewidth=1.15pt,linecolor=lightgray](0.25,0)(0,0.6)(-0.25,0)(0,-0.6)(0.25,0)}
\rput{-45}(4;-45){
\psecurve[linewidth=1.15pt,linecolor=lightgray](-0.25,0)(0,-0.6)(0.25,0)(0,0.6)(-0.25,0)
\psecurve[linewidth=1.15pt](0.25,0)(0,0.6)(-0.25,0)(0,-0.6)(0.25,0)}

\psellipticarc[linecolor=blue!30!white](8.25,0)(1.25,0.4){0}{180}
\rput{90}(0,0){\psellipticarc[linecolor=blue!30!white](8.25,0)(1.25,0.4){0}{180}}
\rput{180}(0,0){\psellipticarc[linecolor=blue!30!white,linestyle=dotted,dotsep=1pt](8.25,0)(1.25,0.4){180}{360}} 
\rput{270}(0,0){\psellipticarc[linecolor=blue!30!white](8.25,0)(1.25,0.4){180}{360}}


\pscurve[linewidth=1.5pt,linecap=1](3.98;-36.5)(6;-32)(7;0)(6;32)(3.98;36.5)
\pscurve[linewidth=1.5pt,linecap=1](3.98;53.5)(6;58)(7;90)(6;122)(3.98;126.5)
\pscurve[linewidth=1.5pt,linecap=1](3.98;143.5)(6;148)(7;180)(6;212)(3.98;216.5)
\pscurve[linewidth=1.5pt,linecap=1](3.98;233.5)(6;238)(7;270)(6;302)(3.98;306.5)
\pscircle[linewidth=1.5pt](0,0){9.5}

\rput{0}(0,0){\pscustom[linecolor=red]{
\psecurve(5,-50)(3.86;-38)(3.4;0)(3.86;38)(5,50)
\psecurve(3;38)(3.86;38)(6;35)(7;25)(7.6;0)(7;-25)(6;-35)(3.86;-38)(3;-38)}
\psdots[linecolor=red,dotsize=1pt](3.86;-38)(3.86;38)
}
\rput{90}(0,0){\pscustom[linecolor=red]{
\psecurve(5,-50)(3.86;-38)(3.4;0)(3.86;38)(5,50)
\psecurve(3;38)(3.86;38)(6;35)(7;25)(7.6;0)(7;-25)(6;-35)(3.86;-38)(3;-38)}
\psdots[linecolor=red,dotsize=1pt](3.86;-38)(3.86;38)
}
\rput{180}(0,0){\pscustom[linecolor=red]{
\psecurve(5,-50)(3.86;-38)(3.4;0)(3.86;38)(5,50)
\psecurve(3;38)(3.86;38)(6;35)(7;25)(7.6;0)(7;-25)(6;-35)(3.86;-38)(3;-38)}
\psdots[linecolor=red,dotsize=1pt](3.86;-38)(3.86;38)
}
\rput{-90}(0,0){\pscustom[linecolor=red]{
\psecurve(5,-50)(3.86;-38)(3.4;0)(3.86;38)(5,50)
\psecurve(3;38)(3.86;38)(6;35)(7;25)(7.6;0)(7;-25)(6;-35)(3.86;-38)(3;-38)}
\psdots[linecolor=red,dotsize=1pt](3.86;-38)(3.86;38)
}

\rput{45}(8;45){
\psecurve[linecolor=darkgreen,linewidth=1.15pt](0.25,0)(0,0.6)(-0.25,0)(0,-0.6)(0.25,0)(0,0.6)(-0.25,0)}
\rput{-45}(8;135){
\psecurve[linecolor=darkgreen,linewidth=1.15pt](0.25,0)(0,0.6)(-0.25,0)(0,-0.6)(0.25,0)(0,0.6)(-0.25,0)}
\rput{45}(8;-135){
\psecurve[linecolor=darkgreen,linewidth=1.15pt](0.25,0)(0,0.6)(-0.25,0)(0,-0.6)(0.25,0)(0,0.6)(-0.25,0)}
\rput{-45}(8;-45){
\psecurve[linecolor=darkgreen,linewidth=1.15pt](0.25,0)(0,0.6)(-0.25,0)(0,-0.6)(0.25,0)(0,0.6)(-0.25,0)}

\psellipticarc[linecolor=blue](8.25,0)(1.25,0.4){180}{360}
\rput{90}(0,0){\psellipticarc[linecolor=blue](8.25,0)(1.25,0.4){180}{360}}
\rput{180}(0,0){\psellipticarc[linecolor=blue,linestyle=dotted,dotsep=1pt](8.25,0)(1.25,0.4){0}{180}}
\rput{270}(0,0){\psellipticarc[linecolor=blue](8.25,0)(1.25,0.4){0}{180}}

\psdot(3.4;180)
\psdot(7.58;87.5)
\psdot(7.58;-2.5)
\psdot(7.58;-87.5)

\psarc[linecolor=darkgreen,linestyle=dashed](0,0){8}{139}{-139}
\end{pspicture}
\caption{From a Heegaard diagram for a tangle to one for a sutured manifold, illustrating the proof of theorem~\ref{thm:HFTiswelldefandinvariant}. This example shows the Heegaard diagram for the 1-crossing tangle.}\label{fig:relationHFTandSFH}
\end{figure}

\noindent
So, let us fix a site $s$. By definition, $s$ is a set of $\alpha$-arcs in $\mathcal{H}$ that are occupied by generators in $\mathbb{T}^s$. An $\alpha$-arc in $\mathcal{H}$ corresponds to an $\alpha$-arc in the punctured torus which in turn corresponds to the $\beta$-circle that it intersects. Thus, a site $s$ gives rise to a collection of $\beta$-circles on the punctured torus. For each such circle $\beta_i$, we pick a path $\gamma_i$ between the two adjacent punctures (the dashed line in figure~\ref{fig:relationHFTandSFH}) which intersects no $\alpha$-circle and no $\beta$-circle except $\beta_i$. We delete $\beta_i$ and cut the surface along $\gamma_i$. This gives us a new sutured Heegaard diagram which \emph{is} balanced. Let us denote it by $\overline{\mathcal{H}}^s$.\\
Now observe that generators in $\mathbb{T}^s$ correspond to generators in $\overline{\mathcal{H}}^s$ and that domains in~$\mathcal{H}$ that avoid $\Z$ correspond to domains in $\overline{\mathcal{H}}^s$ that avoid the boundary. Furthermore, if we started with an admissible Heegaard diagram $\mathcal{H}$, then $\overline{\mathcal{H}}^s$ is also admissible. Hence, the sutured Floer homology $\SFH(\overline{\mathcal{H}}^s)$ is well-defined and identical to $\HFT(T,s)$. Furthermore, if we fix a site $s$, the homological grading on $\CFT(\mathcal{H}_T,s)$ agrees (by definition) with the one on $\HFT(T)$. Similarly, the Alexander grading agrees with the relative $H_1(X_T)$-grading on~$\HFT(T)$ induces by the $\Spinc$-grading (see for example~\cite[definition~4.6]{Juhasz}).
It now only remains to check that $\CFT(T)$ is an invariant as a relatively graded complex for all sites simultaneously. However, this follows from the observation that the Heegaard moves from lemma~\ref{lem:HeegaardMoves} do not change the Alexander nor the $\delta$-grading of the generators that correspond to one another under these moves. 
\end{proof}

\begin{figure}[t]
\centering
\begin{subfigure}[b]{0.33\textwidth}\centering
\psset{unit=0.64, linewidth=1.1pt}
\begin{pspicture}[showgrid=false](-4.2,-3.1)(2.2,3.1)
\psecurve[linecolor=violet]{->}(-2.5,1.5)(0,2)(0.75,1)(-0.75,-1)(0,-2)(0.97,-2.24)(2,-2)
\psecurve[linecolor=violet](2,2)(0.97,2.24)(0,2)(-0.75,1)(0.75,-1)(0,-2)(-2.5,-1.5)(-3.25,0)(-2.5,1.5)(0,2)(0.75,1)
\psecurve[linecolor=darkgreen]{<-}(-6,1.5)(-3.3,1.85)(-2.5,1.5)(-1.85,0)(-2.5,-1.5)(-3.3,-1.85)(-6,-1.5)
\pscircle*[linecolor=white](-2.5,1.5){0.2}

\psecurve[linecolor=violet](0.75,-1)(0,-2)(-2.5,-1.5)(-3.25,0)(-2.5,1.5)

\pscircle*[linecolor=white](0,2){0.2}
\pscircle*[linecolor=white](0,0){0.2}
\pscircle*[linecolor=white](0,-2){0.2}

\psecurve[linecolor=violet]{->}(0.75,1)(-0.75,-1)(0,-2)(0.97,-2.24)(2,-2)
\psecurve[linecolor=violet](0,2)(-0.75,1)(0.75,-1)(0,-2)
\psecurve[linecolor=violet](-2.5,-1.5)(-3.25,0)(-2.5,1.5)(0,2)(0.75,1)(-0.75,-1)

\pscircle*[linecolor=white](-2.5,-1.5){0.2}
\psecurve[linecolor=darkgreen](-2.5,1.5)(-1.85,0)(-2.5,-1.5)(-3.3,-1.85)(-6,-1.5)

\psline[linecolor=violet]{->}(-3.25,-0.1)(-3.25,0.1)
\pscircle[linestyle=dotted](-1,0){3.05}

\uput{0.2}[45](0.97,2.24){$\textcolor{violet}{q}$}
\uput{0.2}[-45](0.97,-2.24){$\textcolor{violet}{q}$}
\uput{0.2}[135](-3.3,1.85){$\textcolor{darkgreen}{p}$}
\uput{0.2}[-135](-3.3,-1.85){$\textcolor{darkgreen}{p}$}

\uput{2.5}[180](-1,0){$a$}
\uput{2.3}[-90](-1,0){$b$}
\uput{2.1}[0](-1,0){$c$}
\uput{2.3}[90](-1,0){$d$}

\end{pspicture}
\caption{The $(2,-3)$-pretzel tangle}\label{fig:mutationpretzeltangleT}
\end{subfigure}
\quad
\begin{subfigure}[b]{0.62\textwidth}\centering
\psset{unit=0.2}
\begin{pspicture}(-22,-11)(22,11)
\rput(-12,0){\psrotate(0,0){-90}{

\pscustom[fillstyle=solid,fillcolor=lightgray]{\psecurve[linewidth=0pt](10.5;45)(8;70)(8;-10)(10;-45)
\psarc[linewidth=0pt](0,0){8}{-10}{70}
}

\pscustom[fillstyle=solid,fillcolor=gray]{\psecurve[linewidth=0pt](8;-70)(8;120)(10;130)(10;140)(8;155)(8;-95)
\psarcn[linewidth=0pt](0,0){8}{155}{120}
}

\psarc[linecolor=red](0,0){8}{0}{360}

\SpecialCoor
\rput(8;45){\pscircle*[linecolor=white]{1}\pscircle[linecolor=violet]{1}}
\rput(8;135){\pscircle*[linecolor=white]{1}\pscircle[linecolor=violet]{1}}
\rput(8;225){\pscircle*[linecolor=white]{1}\pscircle[linecolor=darkgreen]{1}}
\rput(8;315){\pscircle*[linecolor=white]{1}\pscircle[linecolor=darkgreen]{1}}


\psecurve[linecolor=blue]%
(10;130)(10;140)(8;155)%
(8;-95)(11;-45)%
(10.5;45)(8;70)%
(8;-10)%
(10;-45)(8;-70)%
(8;120)(10;130)(10;140)(8;155)%

\psdot(8;155)%
\psdot(8;120)%
\psdot(8;70)%
\psdot(8;-10)%
\psdot(8;-70)%
\psdot(8;-95)%
}
\rput[b](8.7;65){$d$}
\rput[l](8;30){~$x_1$}
\rput[l](8;-20){~$x_2$}
\rput[b](7.3;-100){$b$}
\rput[l](8;-160){~$a_2$}
\rput[l](8;-185){~$a_1$}
}

\rput(12,0){\psrotate(0,0){-90}{

\pscustom[fillstyle=solid,fillcolor=lightgray]{\psecurve[linewidth=0pt](9.5;45)(8;20)(8;-60)(9.5;-45)
\psarc[linewidth=0pt](0,0){8}{-60}{20}
}

\pscustom[fillstyle=solid,fillcolor=gray]{\psecurve[linewidth=0pt](5.5;100)(8;-120)(10;-130)(10;-140)(8;-155)(8;95)
\psarc[linewidth=0pt](0,0){8}{-155}{-120}
}

\psarc[linecolor=red](0,0){8}{0}{360}

\SpecialCoor
\rput(8;45){\pscircle*[linecolor=white]{1}\pscircle[linecolor=violet]{1}}
\rput(8;135){\pscircle*[linecolor=white]{1}\pscircle[linecolor=violet]{1}}
\rput(8;225){\pscircle*[linecolor=white]{1}\pscircle[linecolor=violet]{1}}
\rput(8;315){\pscircle*[linecolor=white]{1}\pscircle[linecolor=violet]{1}}


\psecurve[linecolor=blue]%
(10;-130)(10;-140)(8;-155)%
(8;95)(11;45)%
(10.5;-45)(8;-80)%
(8;60)
(9.5;45)(8;20)
(8;-60)(9.5;-45)(10;45)
(8;75)%
(5.5;100)%
(8;-120)(10;-130)(10;-140)(8;-155)%

\psdot(8;-155)%
\psdot(8;-120)%
\psdot(8;-80)%
\psdot(8;-60)%
\psdot(8;20)%
\psdot(8;60)%
\psdot(8;75)%
\psdot(8;95)%
}
\rput[b](8.7;-245){$d'$}
\rput[r](8;-210){$y_1$~}
\rput[r](8;-170){$y_2$~}
\rput[b](7;-154){$y_3$}
\rput[b](7.3;-70){$b'$}
\rput[r](8;-29){$c_3$~}
\rput[r](8;-15){$c_2$~}
\rput[l](8;5){~$c_1$}
}

\psline[linestyle=dashed](-6.25,5.65685424949236)(6.25,5.65685424949236)

\psline[linestyle=dashed](-6.25,-5.65685424949236)(6.25,-5.65685424949236)

\end{pspicture}
\caption{A Heegaard diagram for the tangle on the left}
\label{fig:mutationpretzeltangleHD}
\end{subfigure}
\\
\vspace*{11pt}
\begin{subfigure}[b]{0.95\textwidth}\centering
\begin{tabular}{c|c|c|c}

\textbf{site a} & \textbf{site b} & \textbf{site c} & \textbf{site d} \\ 
$a_1y_1: \textcolor{darkgreen}{p}^{0} \textcolor{violet}{q}^{+3} \delta^{-\frac{1}{2}} $
&
$by_1: \textcolor{darkgreen}{p}^{-1} \textcolor{violet}{q}^{+1} \delta^{0} $ 
&
$x_1c_1: \textcolor{darkgreen}{p}^{+1} \textcolor{violet}{q}^{+2} \delta^{-\frac{1}{2}} $ 
&
\cancel{$dy_1: \textcolor{darkgreen}{p}^{-1} \textcolor{violet}{q}^{-3} \delta^{0}$} 
\\ 
$a_1y_2: \textcolor{darkgreen}{p}^{0} \textcolor{violet}{q}^{+1} \delta^{-\frac{1}{2}} $ 
&
$by_2: \textcolor{darkgreen}{p}^{-1} \textcolor{violet}{q}^{-1} \delta^{0} $ 
& 
$x_1c_2: \textcolor{darkgreen}{p}^{+1} \textcolor{violet}{q}^{0} \delta^{-\frac{1}{2}} $ 
& 
$dy_2: \textcolor{darkgreen}{p}^{+1} \textcolor{violet}{q}^{+1} \delta^{0} $ 
\\ 
$a_1y_3: \textcolor{darkgreen}{p}^{0} \textcolor{violet}{q}^{-1} \delta^{-\frac{1}{2}} $ 
& 
\cancel{$by_3: \textcolor{darkgreen}{p}^{-1} \textcolor{violet}{q}^{-3} \delta^{0}$} 
& 
$x_1c_3: \textcolor{darkgreen}{p}^{+1} \textcolor{violet}{q}^{-2} \delta^{-\frac{1}{2}} $ 
& 
$dy_3: \textcolor{darkgreen}{p}^{+1} \textcolor{violet}{q}^{-1} \delta^{0} $ 
\\ 
$a_2y_1: \textcolor{darkgreen}{p}^{0} \textcolor{violet}{q}^{+1} \delta^{-\frac{1}{2}} $ 
& 
$x_1b': \textcolor{darkgreen}{p}^{+1} \textcolor{violet}{q}^{-3} \delta^{-1} $  
& 
$x_2c_1: \textcolor{darkgreen}{p}^{-1} \textcolor{violet}{q}^{+2} \delta^{-\frac{1}{2}} $ 
& 
\cancel{$x_1d': \textcolor{darkgreen}{p}^{+1} \textcolor{violet}{q}^{+3} \delta^{-1}$} 
\\ 
$a_2y_2: \textcolor{darkgreen}{p}^{0} \textcolor{violet}{q}^{-1} \delta^{-\frac{1}{2}} $ 
& 
\cancel{$x_2b': \textcolor{darkgreen}{p}^{-1} \textcolor{violet}{q}^{-3} \delta^{-1}$}  
& 
$x_2c_2: \textcolor{darkgreen}{p}^{-1} \textcolor{violet}{q}^{0} \delta^{-\frac{1}{2}} $ 
& 
$x_2d': \textcolor{darkgreen}{p}^{-1} \textcolor{violet}{q}^{+3} \delta^{-1} $   
\\ 
$a_2y_3: \textcolor{darkgreen}{p}^{0} \textcolor{violet}{q}^{-3} \delta^{-\frac{1}{2}} $ 
& 
& 
$x_2c_3: \textcolor{darkgreen}{p}^{-1} \textcolor{violet}{q}^{-2} \delta^{-\frac{1}{2}} $ 
&
\\ 
\end{tabular} 
\caption{A table of the generators of $\HFT$ for the pretzel tangle above and their gradings. The generators that can be cancelled are crossed out.}\label{fig:mutationpretzeltangleTresult}
\end{subfigure}
\caption{The calculation of $\HFT$ for the $(2,-3)$-pretzel, see example~\ref{exa:pretzeltangle}}\label{fig:mutationexample}
\end{figure}

\begin{example}[the $(2,-3)$-pretzel tangle]\label{exa:pretzeltangle}
In figure~\ref{fig:mutationexample}, we compute our tangle Floer homology for the $(2,-3)$-pretzel tangle. The shaded regions in the Heegaard diagram in figure~\ref{fig:mutationpretzeltangleHD} show the only two domains that contribute to the differential. It is interesting to note that if we set $t:=\textcolor{darkgreen}{p}=\textcolor{violet}{q}$, then the result for the sites $a$ and $c$ are the same and for the sites $b$ and $d$ are the same after reversing the orientation $t\leftrightarrow t^{-1}$. For invariance under mutation by rotating the tangle by $\pi$ in the plane, we need however $b=d$, which is only true for the $\delta$-graded invariant; see theorem~\ref{thm:2m3pt} and example~\sref{exa:HFTdpretzeltangle}. 
\end{example}

%% file: sections/2_GlueingViaBSFH.tex
\section{Glueing via bordered sutured Floer theory}\label{sec:TFHviaBSFH}

The tangle Floer homology defined in the previous section does not satisfy any glueing formula because we only record domains away from the boundary. In this subsection, we use bordered sutured Floer theory, developed by Zarev in \cite{Zarev}, to add such a glueing structure to the invariant. We will assume some familiarity with \cite{Zarev} and only give a short review of the basic geometric objects. In particular, we will not give a complete definition of Zarev's invariants.

\begin{definition}
A \textbf{sutured surface} is a quadruple $(F, \Lambda, S_{+}, S_{-})$, where $F$ is a surface with boundary and no closed components and $\Lambda$ is a set of finitely many points on $\partial F$ that partition $\partial F$ into two subsets $S_+$ and $S_-$. If 
$$\pi_0(\Lambda)\rightarrow \pi_0(\partial F)$$
is surjective, we call a surface non-degenerate, otherwise degenerate. 
\end{definition}
\begin{definition}
An \textbf{arc diagram} $\mathcal{Z}$ is a triple $(Z, \mathbf{a},M)$, where $Z$ is a set of oriented line segments, $\mathbf{a}$ an even number of points on $Z$ and $M$ a matching of points in $\mathbf{a}$. The \textbf{graph $G(\mathcal{Z})$ of an arc diagram} $\mathcal{Z}$ is the graph obtained from the line segments $Z$ by adding  an edge between matched points in $\mathbf{a}$. A \textbf{parametrisation of a sutured surface} $F$ is an embedding of a graph of an arc diagram $\mathcal{Z}$ into $F$ such that the line segments~$Z$ are mapped onto $S_+$ and the image $G(\mathcal{Z})$ is a deformation retraction of $F$.
\end{definition}
\begin{definition}
A \textbf{bordered sutured manifold} is a quintuple $(Y,\Gamma, F,\mathcal{Z},\phi)$, where $Y$ is a sutured manifold with sutures $\Gamma$, $F$ is contained in the closure of $R_-$ and 
$$(F,\partial(\Gamma\cap\partial F), \Gamma\cap\partial F,R_-\cap\partial F)$$
 is a sutured surface parametrised by the arc diagram $\mathcal{Z}$ via an embedding $\phi:G(\mathcal{Z})\hookrightarrow F$ such that 
\begin{equation}\label{eqn:homlinind}
\pi_0(\Gamma\smallsetminus F)\rightarrow \pi_0(\partial Y\smallsetminus F)
\end{equation}
is surjective. 
\end{definition}
\begin{Remark}
Condition (\ref{eqn:homlinind}) is called \textbf{homological linear independence}. If we drop this condition, Zarev's invariants fail to be well-defined in general. Note that unlike Zarev, we allow the sutured surfaces of bordered sutured manifolds to be degenerate. This allows us to consider more general bordered sutured manifolds. If we restrict to non-degenerate sutured surfaces, homological linear independence is automatically satisfied, see \cite[proposition~3.6]{Zarev}. 
\end{Remark}

\subsection*{Bordered sutured invariants and glueing. }
Given a bordered sutured manifold~$Y$, Zarev defines two different invariants: a so-called type~A structure $\BSA(Y)$ and a type~D structure $\BSD(Y)$. The former is an $A_\infty$-module over an algebra $\mathcal{A}(\mathcal{Z})$ associated with the arc diagram parametrising the sutured surface on $\partial Y$; the latter is a type~D module over this algebra. For a discussion of type A and D structures as algebraic objects, we refer the reader to appendix~\ref{appendix:AlgStructFromGDCats}, in particular examples~\ref{exa:HighBrowDefTypeDoverI} and~\ref{exa:HighBrowDefTypeAoverI}. \\
Given a (balanced) sutured manifold $(Y,\Gamma)$ and a surface $F$ in $Y$ that has no closed components and splits $Y$ into two components $Y_1$ and $Y_2$ such that every component of $F$ intersects $\Gamma$ non-trivially, we can endow $Y_1$ and $Y_2$ with the structure of bordered sutured manifolds by choosing a parametrisation of $F$ by an arc diagram. Then by \cite[theorem~1]{Zarev}, the sutured Floer complex $\SFC$ can computed as a certain tensor product $\boxtimes$ between the type~A and type~D structures: 
\begin{equation}\label{eqn:BSpairing}
\SFC(Y)=\BSA(Y_1)\boxtimes \BSD(Y_2).
\end{equation}
Again, for details on the box tensor product $\boxtimes$, see appendix~\ref{appendix:AlgStructFromGDCats}, in particular definition~\ref{def:PairingTypeDandA}. Unsurprisingly, the two structures $\BSA(Y_1)$ and $\BSD(Y_2)$ are defined in terms of Heegaard diagrams for bordered sutured manifolds.
\begin{definition}\label{def:HDforBorderedSuturedMfdls}
A \textbf{Heegaard diagram of a bordered sutured manifold} is obtained from a Heegaard diagram of the underlying sutured manifold by adding the graph of the arc diagram to it. To be more precise, consider a Heegaard diagram of the underlying sutured manifold. Then we can embed the graph $G(\mathcal{Z})$ into $R_-$ in such a way that it misses the 2-handles corresponding to the $\alpha$-curves, simply by sliding them off those 2-handles. This gives us an embedding of $G(\mathcal{Z})$ into the Heegaard surface such that its image does not intersect the $\alpha$-curves.  We view the images of the edges connecting points in $\mathbf{a}$ as $\alpha$-arcs. The image of $Z$ lies on the boundary of the Heegaard diagram, 
i.\,e.\ the sutures, which we usually draw in \textcolor{darkgreen}{green}. We put a marked point, a \textbf{basepoint}, in every open component of the boundary minus the image of $Z$. 
\end{definition}

\subsection*{Bordered sutured manifolds for tangles. } 
We can view Heegaard diagrams for tangles (definition~\ref{def:HDsfortangles}) as bordered sutured Heegaard diagrams, as explained in remark~\ref{Rem:TanglesAsSpecialBSMflds}. In this case, each of the line segments $Z$ of the arc diagrams have only one marked point in $\mathbf{a}$, so the glueing surface is just a union of discs. However, when we pair two tangle complements together to obtain a link complement, we need to glue along the whole boundary of the 3-balls (minus the tangle ends), so we need to change the arc diagrams slightly. Depending on whether we want to compute a type~A or type~D invariant, we use slightly different bordered sutured structures on the tangle complements.

\begin{definition} \label{def:TypeAlphaAlphaBetaEnds}
Consider the local picture \ref{fig:BDHDoriginal} around a tangle Heegaard diagram from definition~\ref{def:HDsfortangles}. Note that we have added the two basepoints (the two dashes) on the original suture.
A \textbf{type~$\boldsymbol{\alpha}$ end} is obtained by removing one basepoint and adding an $\alpha$-arc on the opposite side as shown in figure~\ref{fig:BDHDtypeI}, which we call a \textbf{silly arc}. A \textbf{type~$\boldsymbol{\alpha\beta}$ end} is obtained from a type~$\alpha$ end by capping off the suture by a $\beta$-circle and puncturing the region enclosed by the silly $\alpha$-arc and the new $\beta$-circle, see figure~\ref{fig:BDHDtypeII}. 
\end{definition}

\begin{figure}[t]
\psset{unit=0.07}
\centering
{\psset{unit=0.8}
\begin{subfigure}{0.31\textwidth}\centering
\begin{pspicture}(-40,-30)(40,30)
\psframe(-40,-30)(40,30)
\psline[linecolor=red](40,0)(-40,0)
\pscircle[fillstyle=solid,fillcolor=white,linecolor=darkgreen](0,0){10}
\psline(0,12)(0,8)
\psline(0,-12)(0,-8)
\end{pspicture}
\caption{An original tangle end}\label{fig:BDHDoriginal}
\end{subfigure}
\quad
\begin{subfigure}{0.31\textwidth}\centering
\begin{pspicture}(-40,-30)(40,30)
\psframe(-40,-30)(40,30)

\psecurve[linecolor=red](-3,-20)(-3,-1)(-10,13.5)(0,22.55)(10,13.5)(3,-1)(3,-20)
\psecurve[linecolor=red](-60,0)(-40,0)(-5,-10)(0,-10)
\psecurve[linecolor=red](60,0)(40,0)(5,-10)(0,-10)

\pscircle[fillstyle=solid,fillcolor=white,linecolor=darkgreen](0,-10){10}
\psline(0,2)(0,-2)

\end{pspicture}
\caption{A type $\alpha$ end}\label{fig:BDHDtypeI}
\end{subfigure}
\quad
\begin{subfigure}{0.31\textwidth}\centering
\begin{pspicture}(-40,-30)(40,30)

\psframe(-40,-30)(40,30)
\pscircle[linecolor=darkgreen](0,12.5){5}

\psecurve[linecolor=red](-3,-30)(-3,-20)(-3,0)(-10,13.5)(0,22.55)(10,13.5)(3,0)(3,-20)(3,-30)
\psecurve[linecolor=red](-60,0)(-40,0)(-5,-10)(0,-10)
\psecurve[linecolor=red](60,0)(40,0)(5,-10)(0,-10)

\psframe*[linecolor=white](-20,-23)(20,-10)

\pscircle[fillstyle=solid,fillcolor=white,linecolor=darkgreen](0,-10){7}
\pscircle[linecolor=blue](0,-10){12}
\psline(0,-5)(0,-1)
\end{pspicture}
\caption{A type $\alpha\beta$ end}\label{fig:BDHDtypeII}
\end{subfigure}}
\vspace{0.3cm}
\psset{dimen=middle,viewpoint=1 1 0.75}
\begin{subfigure}[b]{\textwidth}\centering
\begin{pspicture}(-80,-61)(80,20)

\ThreeDput[embedangle=0, normal=0 1 0](60,0,0){%
\psline(5,-30)(-40,-30)(-40,30)(40,30)(40,-19)
\psline[linecolor=lightgray](5,-30)(40,-30)(40,-19)

\psecurve[linecolor=red](-3,-30)(-3,-20)(-3,0)(-10,13.5)(0,22.55)(10,13.5)(3,0)(3,-20)(3,-30)
\psecurve[linecolor=red](-60,0)(-40,0)(-9,0)(-9,-10)(-9,-20)
\psecurve[linecolor=red](60,0)(40,0)(9,0)(9,-10)(9,-20)

\psframe*[linecolor=white](-20,-20)(20,-5)

\psellipticarc[fillstyle=solid,fillcolor=white](0,-10)(20,10){15}{200}
\psellipticarc[linecolor=lightgray](0,-10)(20,10){200}{15}

\psellipticarc[linestyle=dashed,linecolor=darkgreen](0,-6)(29,18){-12}{74}
\psellipticarc[linestyle=dashed,linecolor=darkgreen](0,-6)(29,18){106}{-106}
\psellipticarc[linestyle=dashed,linecolor=darkgreen!30!white](0,-6)(29,18){-106}{-12}

\pscircle[linecolor=darkgreen](0,12.5){5}

\psline(0,6)(0,9)
\psline(0,16)(0,19)

}

\ThreeDput[embedangle=0, normal=1 0 0](0,60,0){%
\psline(-5,-30)(40,-30)(40,30)(-40,30)(-40,-19)
\psline[linecolor=lightgray](-5,-30)(-40,-30)(-40,-19)
\pscircle[linecolor=darkgreen](0,12.5){5}

\psecurve[linecolor=red](-3,-30)(-3,-20)(-3,0)(-10,13.5)(0,22.55)(10,13.5)(3,0)(3,-20)(3,-30)
\psecurve[linecolor=red](-60,0)(-40,0)(-9,0)(-9,-10)(-9,-20)
\psecurve[linecolor=red](60,0)(40,0)(9,0)(9,-10)(9,-20)

\psframe*[linecolor=white](-20,-20)(20,-5)

\psellipticarc[fillstyle=solid,fillcolor=white](0,-10)(20,10){-20}{165}
\psellipticarc[linecolor=lightgray](0,-10)(20,10){165}{-20}
}

\ThreeDput[embedangle=0, normal=1.732050 -1 0](30,51.96,0){%
\psellipticarc[linecolor=blue](0,-10)(19.8,9.9){-25}{160}
\psellipticarc[linecolor=lightblue](0,-10)(19.8,9.9){160}{-25}
}

\ThreeDput[normal=0 0 1](0,0,-5){
\psarc(0,0){42.7}{0}{90}
}
\ThreeDput[normal=0 0 1](0,0,-15.7){
\psarcn(0,0){76.5}{90}{0}
}

\ThreeDput[normal=0 0 1](0,0,-0.2){
\psarc[linecolor=red](0,0){63}{0}{90}
\psarc[linecolor=red](0,0){57}{0}{90}
}

\ThreeDput[normal=0 0 1](0,0,-1){
\psarc[linecolor=red](0,0){68.3}{0}{90}
\psarc[linecolor=red](0,0){51.7}{0}{90}
}

\end{pspicture}
\caption{A diagram before cutting/after glueing. It is a ladybug from figure~\ref{fig:ladybug}.}\label{fig:BDHDglued}
\end{subfigure}

\caption{The bordered sutured Heegaard diagram around the tangle ends. To get from (d) to (b) and (c), we cut along the dashed \textcolor{darkgreen}{green} line.}\label{fig:BDHD}
\end{figure}
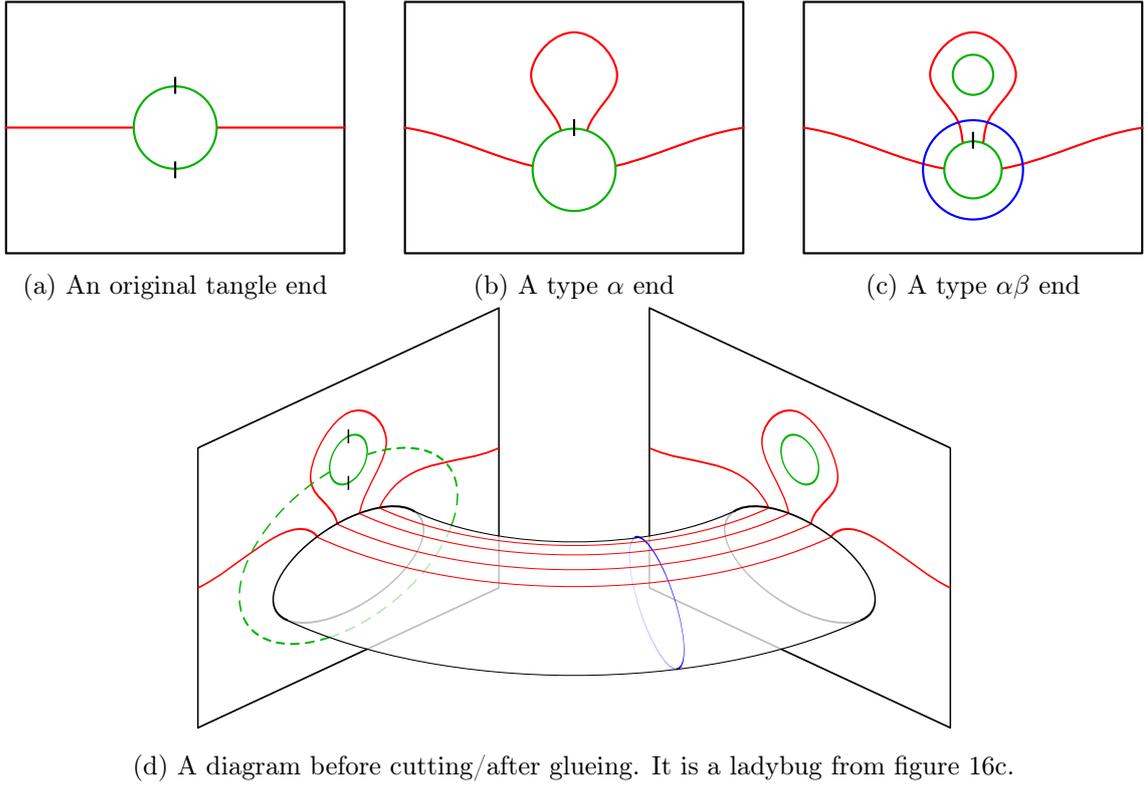

\subsection*{Glueing type $\boldsymbol{\alpha}$ and type $\boldsymbol{\alpha\beta}$ ends. } Figure~\ref{fig:BDHDglued} illustrates how a type~$\alpha$ and a type~$\alpha\beta$ end locally glue together along the arc diagram to form a Heegaard diagram for the glued manifold with two meridional sutures around the glued strand. Note that replacing the original tangle end by a type~$\alpha$ end does not change the generators of the Heegaard diagram. For a type $\alpha\beta$ end, the generators do change. However, when we glue a type $\alpha$ end to a type $\alpha\beta$ end, the new $\alpha$-circle only intersects the $\beta$-circle from the type $\alpha\beta$ end. So any generator of the type~$\alpha\beta$ side that does not occupy the silly $\alpha$-arc is killed in the pairing. If we calculate the type~A structure, we can indeed forget those generators and all arrows to and from it, and the bordered sutured pairing formula from equation (\ref{eqn:BSpairing}) still remains true. For the type~D side, this does not work in general, because in the box tensor product~$\boxtimes$, we also need to consider chains of arrows on the type~D side that could possibly go via some of the generators we want to forget (see definition~\ref{def:PairingTypeDandA}). We will avoid these complications by only taking $\alpha$ ends for the type~D structure and $\alpha\beta$ ends for the type~A structure. \pagebreak[3]\\
Unfortunately, when we glue two tangles together to obtain a knot or link, we cannot treat all tangle ends in the same way. This is for two reasons: First, the arcs need to parametrise the glueing surface, which in this case should be a $2n$-punctured 2-sphere. Second, the bordered sutured manifold should satisfy the homological linear independence condition~(\ref{eqn:homlinind}). Therefore, we have to remove some arcs and for this, we have lots of equally good or bad choices.

\subsection*{Glueing tangle complements. }
We introduce some notation for the next theorem: Let $s$ be a site of a tangle $T$ and $X_T$ the corresponding sutured manifold from section~\ref{sec:HFTviaSFH}. Let $X_{T_1}$ and $X_{T_2}$ be the complements of two tangles $T_1$ and $T_2$ obtained by splitting $X_T$ along some plane meeting $l$ tangle components, see figure \ref{fig:glueingCAT}. We can turn $X_{T_1}$ and $X_{T_2}$ into two bordered sutured manifolds as follows: First of all, the glueing surface is obviously the intersection of the cutting surface with $X_T$. It is parametrised by those arcs connecting the new tangle ends that we implicitly chose in defining $T_1$ and $T_2$. We replace the original tangle ends by type~$\alpha$ ends on $X_{T_2}$ and type~$\alpha\beta$ ends on $X_{T_1}$. Next, we choose the sutures on $X_{T_1}$ and $X_{T_2}$ to agree with those on $X_T$ on their respective common boundary.  Finally, if a suture cannot be homotoped away from the boundary of the cutting surface, we close and connect it with the remaining arcs from the parametrisation of the glueing surface, as illustrated in the upper part of figure~\ref{fig:glueingCAT}.\\
By construction, if we glue $X_{T_1}$ and $X_{T_2}$ together, we reobtain the original sutured manifold, except that there is now an extra pair of meridional sutures at those points of the tangles where we glue them together. We denote this new sutured manifold by $X_{T}'$. Let $A$ be the algebra corresponding to the glueing surface and $I$ the corresponding ring of idempotents. Let $I'\subset I$ be the subring generated by those idempotents $I(s)$ where $s$ ranges over those sets of arcs that contain the silly $\alpha$-arcs at the $\alpha$-ends. In other words, $I'$ consists of those idempotents that could belong to generators of $\BSD(X_{T_2})$. We are now ready to state the glueing theorem which can be viewed as a categorification of the glueing formula from proposition~\sref{prop:glueing}.

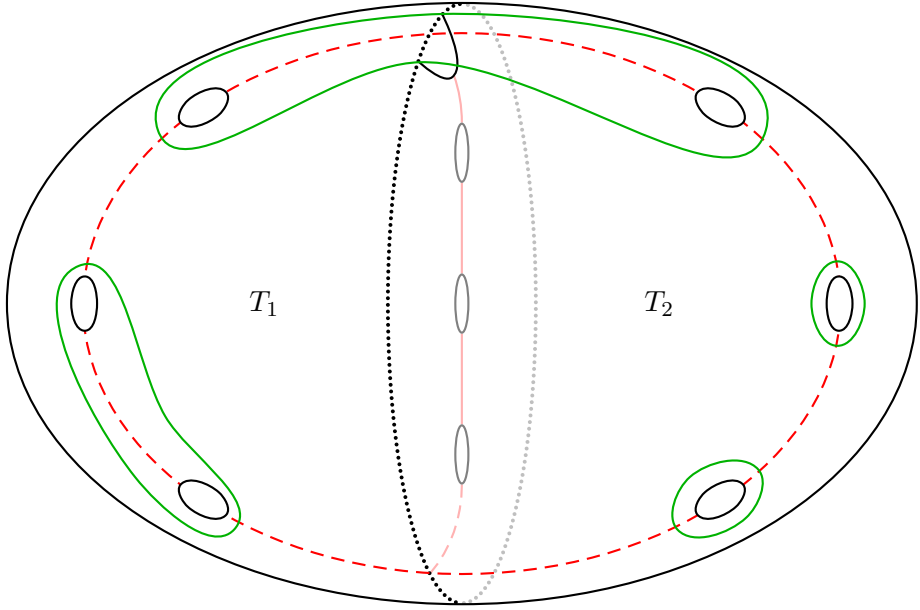
\begin{figure}[t]
\centering
\psset{unit=0.2}
\begin{pspicture}(-30,-20)(30,20)

\psellipticarc[linecolor=lightgray,linestyle=dotted,dotsep=1pt,linewidth=1.5pt](0,0)(5,20){-90}{90}

\psline[linecolor=lightred](0,2)(0,8)
\psline[linecolor=lightred](0,-2)(0,-8)
\psecurve[linecolor=lightred](-2.1,6.1)(0,12)(-2.1,17.9)(-6,20)
\psecurve[linecolor=lightred,linestyle=dashed](-2.1,-6.1)(0,-12)(-2.1,-17.9)(-6,-20)

\pscurve[fillcolor=white,fillstyle=solid,linecolor=lightdarkgreen](-1.3,19.2)(-0.65,15)(-2.8,16)
\pscircle*[linecolor=white](-2,18){0.5}


\psellipticarc[linestyle=dotted,dotsep=1pt,linewidth=1.5pt](0,0)(5,20){90}{-90}

\psellipticarc[linecolor=gray](0,10)(0.5,2){0}{360}
\psellipticarc[linecolor=gray](0,0)(0.5,2){0}{360}
\psellipticarc[linecolor=gray](0,-10)(0.5,2){0}{360}
\psellipticarc(0,0)(30,20){0}{360}

\psellipse[linecolor=red, linestyle=dashed](0,0)(25,18)

\psecurve[fillcolor=white,fillstyle=solid](-25.7,0)(-24.8,1.8)(-24,0)(-24.8,-1.8)(-25.7,0)(-24.8,1.8)(-24,0)
\psecurve[fillcolor=white,fillstyle=solid](-18.5,12)(-18,13.5)(-15.5,14)(-16,12.5)(-18.5,12)(-18,13.5)(-15.5,14)
\psecurve[fillcolor=white,fillstyle=solid](-18.5,-12)(-18,-13.5)(-15.5,-14)(-16,-12.5)(-18.5,-12)(-18,-13.5)(-15.5,-14)

\psecurve[fillcolor=white,fillstyle=solid](25.7,0)(24.8,1.8)(24,0)(24.8,-1.8)(25.7,0)(24.8,1.8)(24,0)
\psecurve[fillcolor=white,fillstyle=solid](18.5,12)(18,13.5)(15.5,14)(16,12.5)(18.5,12)(18,13.5)(15.5,14)
\psecurve[fillcolor=white,fillstyle=solid](18.5,-12)(18,-13.5)(15.5,-14)(16,-12.5)(18.5,-12)(18,-13.5)(15.5,-14)

\psecurve[linecolor=darkgreen](19.5,-11)(18.7,-14.2)(14.2,-15)(15,-11.5)(19.5,-11)(18.7,-14.2)(14.2,-15)
\psecurve[linecolor=darkgreen](26.5,0)(24.7,2.8)(23,0)(24.7,-2.8)(26.5,0)(24.7,2.8)(23,0)

\psecurve[linecolor=darkgreen](-1.3,19.2)(-20,11.5)(-2.8,16)(20,11.5)(-1.3,19.2)(-20,11.5)(-2.8,16)

\psecurve[linecolor=darkgreen](-25.2,2.5)(-19.5,-7.5)(-15,-15)(-21.5,-11.5)(-25.2,2.5)(-19.5,-7.5)(-15,-15)


\rput(-13,0){$T_1$}
\rput(13,0){$T_2$}

\end{pspicture}
\caption{An illustration for the glueing theorem \ref{thm:glueingCAT}}\label{fig:glueingCAT}
\end{figure}

\begin{theorem}\label{thm:glueingCAT}
With the notation from above, there exist vector space isomorphisms
\begin{eqnarray*}
\bigoplus_{s_2\cap O_2(T)= s\cap O_2(T)} \CFT(T_2,s_2)&\longrightarrow &\BSD(X_{T_2})\\
\bigoplus_{s_1\cap O_1(T)= s\cap O_1(T)} \CFT(T_1,s_1)\otimes V^{\otimes l}&\longrightarrow &\BSA(X_{T_1}).I'\\
\bigoplus_{\substack{s_1\cap s_2=\emptyset\\ (s_1\cup s_2)\cap O(T)=s}}\CFT(T_1,s_1)\otimes \CFT(T_2,s_2)\otimes V^{\otimes l}&\longrightarrow &\CFT(T,s)\otimes V^{\otimes (l-c)}
\end{eqnarray*}
where \(O(T)\) denotes the set of open regions of the tangle \(T\), \(O_1(T)=O(T)\smallsetminus O(T_2)\) and \(O_2(T)=O(T)\smallsetminus O(T_1)\), \(V\) is a vector space with two generators in grading \(t\) and \(h^{-1}t^{-1}\) and \(c\) is the difference between the number of closed tangle components before and after the splitting.
The first map is a chain isomorphism after setting all algebra elements with moving strands equal to zero. Furthermore, if one of \(\BSA(X_{T_1}).I'\) and \(\BSD(X_{T_2})\) is bounded,
\begin{equation}\label{eqn:BSGlueing}
\CFT(T,s)\otimes V^{\otimes(l-c)}= \BSA(X_{T_1}).I'\boxtimes \BSD(X_{T_2}).
\end{equation} 
\end{theorem}
\begin{Remark}
The bordered sutured invariants for the $(2,-3)$-pretzel tangle that we compute in theorem~\sref{thm:2m3pt} are both bounded, so we can apply the glueing statement above. Moreover, the invariants that we compute in proposition~\ref{prop:BSDrattangle} for the (positive) rational tangles are all bounded, except those for the trivial tangles. This suggests that we might always be able to find a parametrisation of the glueing surface such that the type~D side is bounded, simply by introducing Dehn twists of the tangles ends at the glueing surface. 
\end{Remark}
\begin{proof}
The first two identifications follow directly from a comparison of the Heegaard diagrams involved and the discussion following definition~\ref{def:TypeAlphaAlphaBetaEnds}. In particular, each tensor factor on the left of the second identification corresponds to an $\alpha\beta$-end, since the $\beta$-circle at such an end intersects the silly $\alpha$-arc twice. The difference between the contributions of those intersection points to the Alexander and homological gradings are straightforward to compute. For the third relation, we argue similarly, using 
$$\SFC(X_{T}')=\CFT(T,s)\otimes V^{\otimes (l-c)}.$$
This identity follows from a small adaptation of the arguments of the proof of proposition~\ref{prop:HFTreverseorientI}: 
Consider the surface obtained from $\partial X'_{T}$ by deleting the component of $R_-$ that is contained in $\partial B^3$. Take an open collar neighbourhood $N$ of this surface in $X'_T$, so we can write $X'_T$ as a union of the closure of $N$ and the complement of $N$ in $X'_T$ along their common boundary. We can now apply the surface decomposition formula \cite[proposition~8.6]{SurfaceDecomposition}. It is straightforward to see that the sutured Floer homology of $N$ is equal to $V^{\otimes (l-c)}$.\\
The final statement follows from Zarev's general pairing theorem in bordered sutured theory~\cite[theorem~7.16]{Zarev}. 
\end{proof}

\begin{example}[glueing two 4-ended tangles]
Suppose $T$ is a 2-ended tangle and we cut it into two 4-ended tangles $T_1$ and $T_2$. The glueing surface is a 3-punctured disc parametrised by five arcs, of which three are silly arcs and the remaining two correspond to two adjacent sites, say $s_1$ and $s_2$. Then the proposition above tells us that the generators of the type D module $\BSD(T_2)$ are in one-to-one correspondence with the generators of 
$$
\CFT(T_2,\{s_1\})\oplus\CFT(T_2,\{s_2\})
$$
and the honest differentials in the type $D$ structure correspond to the differentials in the tangle Floer complexes. 
The statement for the generators is also true on the type~A side for~$T_1$, except that we obtain additional tensor factors. Then, after glueing, we get 
$$\HFT(T,\emptyset)\otimes V^{\otimes n}=\HFL(L)\otimes V^{\otimes n},$$
where $n=2$ or $n=3$, depending on whether glueing closes a component or not.
\end{example}

\begin{proposition}\label{prop:BSDrattangle}
Let \(T_{p/q}\) be the positive rational tangle corresponding to the fraction~\(\frac{p}{q}\) with \(p>q\geq0\) and \(\gcd(p,q)=1\). Let \(X_{T_{p/q}}\) be the bordered sutured manifold with the parametrisation specified by figure~\ref{fig:BSDXTpqMan}. We label the two arcs which specify two sites of the tangle by \(p\) and \(q\). 

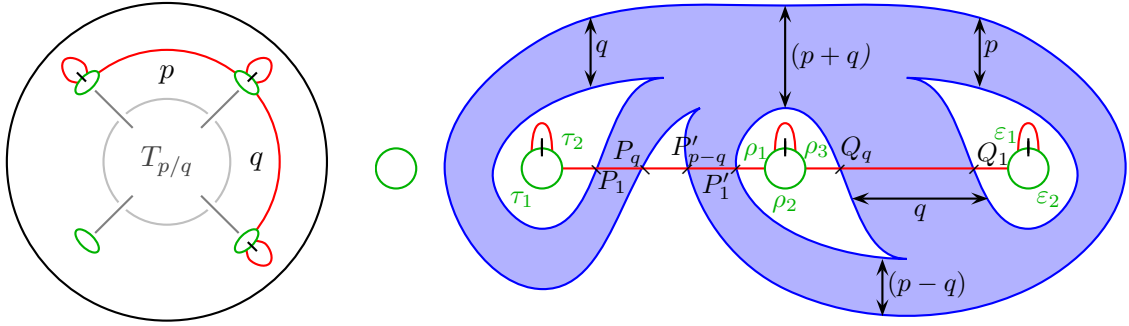
\begin{figure}[t]
\begin{subfigure}{0.3\textwidth}\centering
\psset{unit=0.212}
\begin{pspicture}(-10,-10)(10,10)
\pscircle(0,0){10}
\psarc[linecolor=red](0,0){7}{-45}{135}

\pscurve[linecolor=red](7;135)(8;130)(9;135)(8;140)(7;135)
\pscurve[linecolor=red](7;45)(8;50)(9;45)(8;40)(7;45)
\pscurve[linecolor=red](7;-45)(8;-40)(9;-45)(8;-50)(7;-45)

\pscircle[linecolor=lightgray](0,0){4}

\psline[linecolor=white,linewidth=4pt](3;45)(6.3;45)
\psline[linecolor=white,linewidth=4pt](3;135)(6.3;135)
\psline[linecolor=white,linewidth=4pt](3;-135)(6.3;-135)
\psline[linecolor=white,linewidth=4pt](3;-45)(6.3;-45)

\psline[linecolor=gray](3;45)(6.3;45)
\psline[linecolor=gray](3;135)(6.3;135)
\psline[linecolor=gray](3;-135)(6.3;-135)
\psline[linecolor=gray](3;-45)(6.3;-45)

{\psset{fillstyle=solid,fillcolor=white}
\rput{45}(7;45){\psellipse[linecolor=darkgreen](0,0)(0.5,1)}
\rput{135}(7;135){\psellipse[linecolor=darkgreen](0,0)(0.5,1)}
\rput{-135}(7;-135){\psellipse[linecolor=darkgreen](0,0)(0.5,1)}
\rput{-45}(7;-45){\psellipse[linecolor=darkgreen](0,0)(0.5,1)}
}

\psline(7.1;135)(7.9;135)
\psline(7.1;45)(7.9;45)
\psline(7.1;-45)(7.9;-45)

\rput[t](0,6){$p$}
\rput[r](6,0){$q$}

\rput(0,0){\textcolor{darkgray}{$T_{p/q}$}}

\end{pspicture}
\caption{The parametrisation of the boundary of the tangle complement $X_{T_{p/q}}$}\label{fig:BSDXTpqMan}
\end{subfigure}
\begin{subfigure}{0.69\textwidth}\centering
\psset{unit=0.4,xunit=0.8}
\begin{pspicture}(3.2,-5.1)(34.1,5.5)
\psecurve*[linewidth=0pt,linecolor=lightblue](30,0)(16.5,2)(10,-4)(6,0)(10,4.3)(20,5.4)(30,4.3)(34,0)(30,-4)(20,-4)(16,0)(16.5,2)(25,0)
{\psset{linecolor=blue}
\psecurve(30,0)(16.5,2)(10,-4)(6,0)(10,4.3)(20,5.4)(30,4.3)(34,0)(30,-4)(20,-4)(16,0)(16.5,2)(25,0)
\psecurve[fillcolor=white,fillstyle=solid](22,0)(15,3)(8,0)(10,-2)(15,3)(18,0)
\psecurve[fillcolor=white,fillstyle=solid](32,0)(25,-3)(18,0)(20,2)(25,-3)(28,0)
\psecurve[fillcolor=white,fillstyle=solid](18,0)(25,3)(32,0)(30,-2)(25,3)(22,0)
}
\small
\psline[linecolor=red](10,0)(30,0)

\psline{<->}(12,2.7)(12,5)
\psline{<->}(20,2)(20,5.4)
\psline{<->}(28,2.7)(28,5)

\psline{<->}(22.75,-1)(28.35,-1)
\psline{<->}(24,-3)(24,-4.9)

\rput[l](12.2,3.95){$q$}
\rput[l](20.2,3.7){$(p+q$)}
\rput[l](28.2,3.95){$p$}
\rput[t](25.55,-1.2){$q$}
\rput[l](24.2,-3.85){$(p-q)$}

\pscurve[linecolor=red](9.6,0)(10,1.5)(10.4,0)
\pscurve[linecolor=red](19.6,0)(20,1.5)(20.4,0)
\pscurve[linecolor=red](29.6,0)(30,1.5)(30.4,0)

{\psset{fillcolor=white, fillstyle=solid,linecolor=darkgreen}
\pscircle(4,0){0.7}
\pscircle(10,0){0.7}
\pscircle(20,0){0.7}
\pscircle(30,0){0.7}
}

\psline(10,1)(10,0.4)
\psline(20,1)(20,0.4)
\psline(30,1)(30,0.4)

{\psset{dotstyle=|}
\psdot[dotangle=45](12.2,0)\uput{0.1}[-45](12.2,0){$P_1$}
\psdot[dotangle=45](14.2,0)\uput{0.1}[135](14.2,0){$P_q$}
\psdot[dotangle=-45](16,0)\uput{0.1}[70](16,0){$P'_{p-q}$}
\psdot[dotangle=-45](18,0)\uput{0.15}[-130](18,0){$P'_1\,$}
\psdot[dotangle=-45](22.3,0)\uput{0.2}[55](22.3,0){$Q_q$}
\psdot[dotangle=-45](27.8,0)\uput{0.1}[45](27.8,0){$Q_1$}
}

\rput(10,0){
\uput{0.8}[-120](0,0){\textcolor{darkgreen}{$\tau_1$}}
\uput{0.9}[45](0,0){\textcolor{darkgreen}{$\tau_2$}}
\psline(0.5;90)(0.9;90)
}

\rput(20,0){
\uput{0.6}[145](0,0){\textcolor{darkgreen}{$\rho_1$}}
\uput{0.9}[-90](0,0){\textcolor{darkgreen}{$\rho_2$}}
\uput{0.7}[35](0,0){\textcolor{darkgreen}{$\rho_3$}}
\psline(0.5;90)(0.9;90)
}

\rput(30,0){
\uput{0.9}[120](0,0){\textcolor{darkgreen}{$\varepsilon_1$}}
\uput{0.8}[-60](0,0){\textcolor{darkgreen}{$\varepsilon_2$}}
\psline(0.5;90)(0.9;90)
}

\end{pspicture}
\caption{A genus 0 Heegaard diagram for $X_{T_{p/q}}$. The blue bands represent parallel $\beta$-curve segments, the number of which is specified by the label of the corresponding double arrow.}\label{fig:BSDXTpqHD}
\end{subfigure}
\caption{An illustration for proposition~\ref{prop:BSDrattangle}.}\label{fig:BSDXTpq}
\end{figure}

\noindent
Then the type~D structure for \(T_{p/q}\) has the form of a loop, which corresponds to the single \(\beta\)-curve in the Heegaard diagram in figure~\ref{fig:BSDXTpqHD} in the following sense: The generators are given by intersection points of the \(\beta\)-curve and the \(\alpha\)-arcs. Each generator of the type~D structure has exactly two incoming or outgoing arrows which correspond to paths on the \(\beta\)-curve from the corresponding intersection point to its two neighbours on the \(\beta\)-curve. The algebra elements picked up by the differentials correspond to those paths on the arc diagram after pulling the \(\beta\)-curve tight. \\
More explicitly, \(\BSD(X_{T_{p/q}})\) is obtained by glueing the following elementary ``puzzle pieces'' together in such a way that the integers at the puzzle ends match: 
\begin{center}
\begin{tikzpicture}[auto]
\node[left] (Px) at (5.5,0) {};
\node[left] (Px) at (-2,0) {$p+i$};
\node[right] (Py) at (2,0) {$p+1-i$};
\node[right] (range) at (-7,0) {$1\leq i\leq q$:};
\node (P) at (0,0) {$P_i$};
\draw[->] (-1.5,0) to node {$\rho_{13}\tau_{12}$} (P);
\draw[<-] (P) to node {$\rho_{13}$} (1.5,0);

\draw[xshift=-1.95cm, yshift=0.6cm] (0,0) -- (3.9,0);
\draw[xshift=-1.95cm, yshift=-0.6cm] (0,0) -- (3.9,0);
\draw[xshift=-1.8cm]  plot[smooth, tension=.5] coordinates {(-0.15,0.6) (-0.15,0.2) (0.05,0.2) (0.15,0.1) (0.15,-0.1) (0.05,-0.2) (-0.15,-0.2) (-0.15,-0.6)};
\draw[xshift=1.8cm]  plot[smooth, tension=.5] coordinates {(0.15,0.6) (0.15,0.2) (-0.05,0.2) (-0.15,0.1) (-0.15,-0.1) (-0.05,-0.2) (0.15,-0.2) (0.15,-0.6)};
\end{tikzpicture}
\vspace{5pt}\\
\begin{tikzpicture}[auto]
\node[left] (Px) at (5.5,0) {};
\node[left] (Px) at (-2,0) {$q+i$};
\node[right] (Py) at (2,0) {$q+1-i$};
\node[right] (range) at (-7,0) {$1\leq i\leq q$:};
\node (Q) at (0,0) {$Q_i$};
\draw[<-] (-1.5,0) to node {$\varepsilon_{12}$} (Q);
\draw[->] (P) to node {} (1.5,0);

\draw[xshift=-1.65cm, yshift=0.6cm] (0,0) -- (3.3,0);
\draw[xshift=-1.65cm, yshift=-0.6cm] (0,0) -- (3.3,0);
\draw[xshift=-1.8cm]  plot[smooth, tension=.5] coordinates {(0.15,0.6) (0.15,0.2) (-0.05,0.2) (-0.15,0.1) (-0.15,-0.1) (-0.05,-0.2) (0.15,-0.2) (0.15,-0.6)};
\draw[xshift=1.8cm]  plot[smooth, tension=.5] coordinates {(-0.15,0.6) (-0.15,0.2) (0.05,0.2) (0.15,0.1) (0.15,-0.1) (0.05,-0.2) (-0.15,-0.2) (-0.15,-0.6)};
\end{tikzpicture}
\vspace{5pt}\\
\begin{tikzpicture}[auto]
\node[left] (Px) at (5.5,0) {};
\node[left] (Px) at (-2,0) {$i+2q$};
\node[right] (Py) at (2,0) {$i$};
\node[right] (range) at (-7,0) {$1\leq i\leq p-q$:};
\node (P) at (0,0) {$P'_i$};
\draw[<-] (-1.5,0) to node {$\rho_2\varepsilon_{12}$} (P);
\draw[<-] (P) to node {$\rho_{13}$} (1.5,0);

\draw[xshift=-1.65cm, yshift=0.6cm] (0,0) -- (3.6,0);
\draw[xshift=-1.65cm, yshift=-0.6cm] (0,0) -- (3.6,0);
\draw[xshift=-1.8cm]  plot[smooth, tension=.5] coordinates {(0.15,0.6) (0.15,0.2) (-0.05,0.2) (-0.15,0.1) (-0.15,-0.1) (-0.05,-0.2) (0.15,-0.2) (0.15,-0.6)};
\draw[xshift=1.8cm]  plot[smooth, tension=.5] coordinates {(0.15,0.6) (0.15,0.2) (-0.05,0.2) (-0.15,0.1) (-0.15,-0.1) (-0.05,-0.2) (0.15,-0.2) (0.15,-0.6)};
\end{tikzpicture}
\end{center}
Our conventions for the algebra are explained in the proof, see also remark~\ref{rem:conventions3}.
If \(q>p\geq0\), the same holds, but for an explicit description in terms of such elementary pieces, we need to replace \(P'_i\) by 
\begin{center}
\begin{tikzpicture}[auto]
\node[left] (Pmax) at (5.5,0) {};
\node[left] (Px) at (-2,0) {$i+2p$};
\node[right] (Py) at (2,0) {$i$};
\node[right] (range) at (-7,0) {$1\leq i\leq q-p$:};
\node (P) at (0,0) {$Q'_i$};
\draw[<-] (1.5,0) to node[swap] {} (P);
\draw[<-] (P) to node[swap] {$\rho_{13}\tau_{12}\rho_2$} (-1.5,0);

\draw[xshift=-1.95cm, yshift=0.6cm] (0,0) -- (3.6,0);
\draw[xshift=-1.95cm, yshift=-0.6cm] (0,0) -- (3.6,0);
\draw[xshift=-1.8cm]  plot[smooth, tension=.5] coordinates {(-0.15,0.6) (-0.15,0.2) (0.05,0.2) (0.15,0.1) (0.15,-0.1) (0.05,-0.2) (-0.15,-0.2) (-0.15,-0.6)};
\draw[xshift=1.8cm]  plot[smooth, tension=.5] coordinates {(-0.15,0.6) (-0.15,0.2) (0.05,0.2) (0.15,0.1) (0.15,-0.1) (0.05,-0.2) (-0.15,-0.2) (-0.15,-0.6)};
\end{tikzpicture}
\end{center}
and to let the indices of the pieces \(P_i\) and \(Q_i\) run from \(1\) to \(p\). 
\end{proposition}
\begin{Remark}
The ``puzzle piece'' notation above is inspired by Hanselman and Watson's notation for their loop calculus in~\cite{HanselmanWatson}. 
\end{Remark}

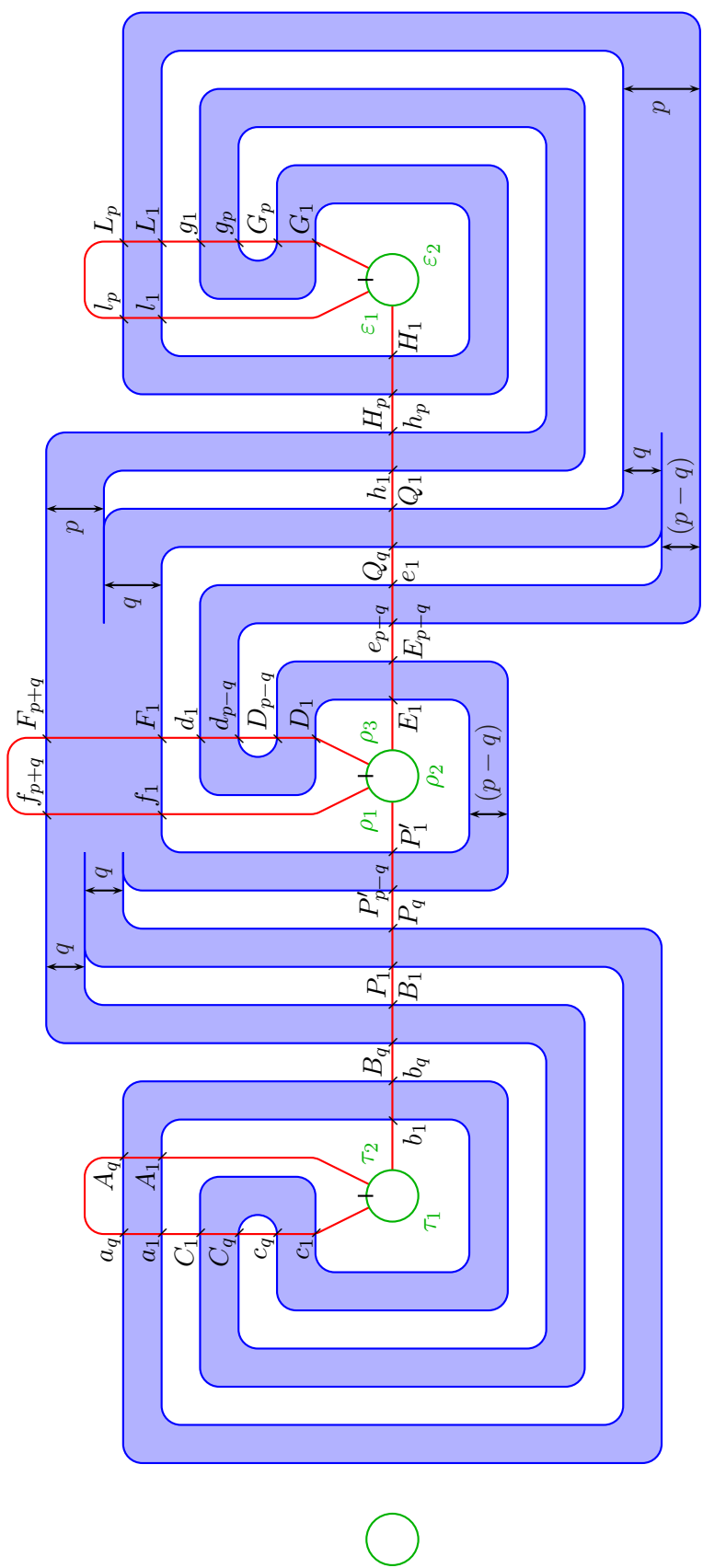
\begin{sidewaysfigure}
\vspace*{435pt}
\centering
\psset{unit=0.55,linearc=0.5}
\begin{pspicture}(-1.2,-8.2)(40.2,10.2)
\psline[linecolor=blue,fillcolor=lightblue,fillstyle=solid](18,7)(17,7)(17,-3)(23,-3)(23,3)(20.5,3)(20.5,4)(24,4)(24,-8)(40,-8)(40,7)(30,7)(30,-3)(36,-3)(36,3)(33.5,3)(33.5,4)(37,4)(37,-4)(29,-4)(29,9)(13,9)(13,-4)(5,-4)(5,4)(8.5,4)(8.5,3)(6,3)(6,-3)(12,-3)(12,7)(2,7)(2,-7)(16,-7)(16,7)(18,7)

\psline[linecolor=blue,fillcolor=white,fillstyle=solid](29,-7)(25,-7)(25,5)(19.5,5)(19.5,2)(22,2)(22,-2)(18,-2)(18,6)(26,6)(26,-7)(29,-7)

\psline[linecolor=blue,fillcolor=white,fillstyle=solid](18,8)(14,8)(14,-5)(4,-5)(4,5)(9.5,5)(9.5,2)(7,2)(7,-2)(11,-2)(11,6)(3,6)(3,-6)(15,-6)(15,8)(18,8)

\psline[linecolor=blue,fillcolor=white,fillstyle=solid](24,7.5)(28,7.5)(28,-5)(38,-5)(38,5)(32.5,5)(32.5,2)(35,2)(35,-2)(31,-2)(31,6)(39,6)(39,-6)(27,-6)(27,7.5)(24,7.5)

\psline[linecolor=red](9,0)(8,2)(8,8)(10,8)(10,2)(9,0)
\psline[linecolor=red](20,0)(19,2)(19,10)(21,10)(21,2)(20,0)
\psline[linecolor=red](33,0)(32,2)(32,8)(34,8)(34,2)(33,0)

\psline[linecolor=red](9,0)(33,0)

{\psset{fillcolor=white, fillstyle=solid, linecolor=darkgreen}
\pscircle(0,0){0.7}
\pscircle(9,0){0.7}
\pscircle(20,0){0.7}
\pscircle(33,0){0.7}
}

{\psset{dotstyle=|}
\psdot[dotangle=45](8,2)\uput{0.1}[135](8,2){$c_1$}
\psdot[dotangle=45](8,3)\uput{0.1}[135](8,3){$c_q$}
\psdot[dotangle=45](8,4)\uput{0.1}[135](8,4){$C_q$}
\psdot[dotangle=45](8,5)\uput{0.1}[135](8,5){$C_1$}
\psdot[dotangle=45](8,6)\uput{0.1}[135](8,6){$a_1$}
\psdot[dotangle=45](8,7)\uput{0.1}[135](8,7){$a_q$}
\psdot[dotangle=45](10,6)\uput{0.1}[135](10,6){$A_1$}
\psdot[dotangle=45](10,7)\uput{0.1}[135](10,7){$A_q$}

\psdot[dotangle=-45](19,6)\uput{0.1}[45](19,6){\,$f_1$}
\psdot[dotangle=-45](19,9)\uput{0.1}[45](19,9){\,$f_{p+q}$}
\psdot[dotangle=-45](21,9)\uput{0.1}[45](21,9){\,$F_{p+q}$}
\psdot[dotangle=-45](21,6)\uput{0.1}[45](21,6){\,$F_1$}
\psdot[dotangle=-45](21,5)\uput{0.1}[45](21,5){\,$d_1$}
\psdot[dotangle=-45](21,4)\uput{0.1}[45](21,4){\,$d_{p-q}$}
\psdot[dotangle=-45](21,3)\uput{0.1}[45](21,3){\,$D_{p-q}$}
\psdot[dotangle=-45](21,2)\uput{0.1}[45](21,2){\,$D_1$}
      
\psdot[dotangle=-45](32,6)\uput{0.1}[45](32,6){\,$l_1$}
\psdot[dotangle=-45](32,7)\uput{0.1}[45](32,7){\,$l_p$}
\psdot[dotangle=-45](34,7)\uput{0.1}[45](34,7){\,$L_p$}
\psdot[dotangle=-45](34,6)\uput{0.1}[45](34,6){\,$L_1$}
\psdot[dotangle=-45](34,5)\uput{0.1}[45](34,5){\,$g_1$}
\psdot[dotangle=-45](34,4)\uput{0.1}[45](34,4){\,$g_p$}
\psdot[dotangle=-45](34,3)\uput{0.1}[45](34,3){\,$G_p$}
\psdot[dotangle=-45](34,2)\uput{0.1}[45](34,2){\,$G_1$}

\psdot[dotangle=-45](11,0)\uput{0.1}[-135](11,-0.2){$b_1$}
\psdot[dotangle=45](12,0)\uput{0.1}[-45](12,-0.2){$b_q$}
\psdot[dotangle=45](13,0)\uput{0.1}[135](13,0){$B_q$}
\psdot[dotangle=45](14,0)\uput{0.1}[-45](14,-0.1){$B_1$}
\psdot[dotangle=45](15,0)\uput{0.1}[135](15,0){$P_1$}
\psdot[dotangle=45](16,0)\uput{0.1}[-45](16,-0.1){$P_q$}
\psdot[dotangle=45](17,0)\uput{0.1}[90](17,0){$P'_{p-q}$}
\psdot[dotangle=45](18,0)\uput{0.1}[-45](18,-0.1){$P'_1$}

\psdot[dotangle=-45](22,0)\uput{0.1}[-135](22,-0.1){$E_1$}
\psdot[dotangle=45](23,0)\uput{0.2}[-50](23,-0.1){$E_{p-q}$}
\psdot[dotangle=45](24,0)\uput{0.1}[100](24,0){$e_{p-q}$}
\psdot[dotangle=45](25,0)\uput{0.1}[-45](25,-0.2){$e_1$}
\psdot[dotangle=45](26,0)\uput{0.1}[135](26,0){$Q_q$}
\psdot[dotangle=45](27,0)\uput{0.1}[-45](27,-0.1){$Q_1$}
\psdot[dotangle=45](28,0)\uput{0.1}[135](28,0){$h_1$}
\psdot[dotangle=45](29,0)\uput{0.1}[-45](29,-0.2){$h_p$}
\psdot[dotangle=45](30,0)\uput{0.1}[135](30,0){$H_p$}
\psdot[dotangle=45](31,0)\uput{0.1}[-45](31,-0.1){$H_1$}
}


\rput(9,0){
\uput{0.9}[-120](0,0){\textcolor{darkgreen}{$\tau_1$}}
\uput{0.9}[30](0,0){\textcolor{darkgreen}{$\tau_2$}}
\psline(0.5;90)(0.9;90)
}

\rput(20,0){
\uput{0.9}[150](0,0){\textcolor{darkgreen}{$\rho_1$}}
\uput{0.9}[-90](0,0){\textcolor{darkgreen}{$\rho_2$}}
\uput{0.9}[30](0,0){\textcolor{darkgreen}{$\rho_3$}}
\psline(0.5;90)(0.9;90)
}

\rput(33,0){
\uput{0.9}[150](0,0){\textcolor{darkgreen}{$\varepsilon_1$}}
\uput{0.9}[-60](0,0){\textcolor{darkgreen}{$\varepsilon_2$}}
\psline(0.5;90)(0.9;90)
}

{\psset{arrowsize=1.5pt 2}
\psline{<->}(15,8)(15,9)\uput{0.3}[0](15,8.5){$q$}
\psline{<->}(17,8)(17,7)\uput{0.3}[0](17,7.5){$q$}
\psline{<->}(25,7.5)(25,6)\uput{0.3}[180](25,6.75){$q$}
\psline{<->}(27,7.5)(27,9)\uput{0.3}[180](27,8.25){$p$}
\psline{<->}(26,-8)(26,-7)\uput{0.3}[0](26,-7.5){$(p-q)$}
\psline{<->}(19,-3)(19,-2)\uput{0.3}[0](19,-2.5){$(p-q)$}
\psline{<->}(28,-6)(28,-7)\uput{0.3}[0](28,-6.5){$q$}
\psline{<->}(38,-8)(38,-6)\uput{0.3}[180](38,-7){$p$}
}

\end{pspicture}
\caption{A niceified Heegaard diagram for the $\frac{p}{q}$-rational tangle}\label{fig:BSDTpqNiceHD}
\end{sidewaysfigure}

\begin{Remark} \label{rem:conventions3}
In the following proof, and also in all other computations of bordered sutured invariants in this thesis, we actually compute the invariants of the mirrors of the underlying bordered sutured manifolds. This is due to a late change of conventions in an attempt to make them intrinsically consistent (we always use the right-hand rule to determine orientations) and to match those most commonly used in the literature, see remarks~\ref{rem:conventions1} and~\ref{rem:conventions2} and figure~\ref{fig:AlexCodesForOneCrossingsWithDelta}. It is quite natural to draw Heegaard diagrams of tangles such that the normal vector field points into the plane, see remark~\ref{rem:conventions1} and figure~\ref{fig:HDfor1crossing}. This is opposite to standard conventions.
However, the effect of reversing the orientation of the Heegaard surface is marginal. It simply corresponds to a reversal of all arrows and gradings as well as a reversal of the glueing algebra. 
\end{Remark}

\begin{proof}[Proof of proposition~\ref{prop:BSDrattangle}]
We can easily turn the Heegaard diagram from figure~\ref{fig:BSDXTpqHD} into a nice diagram, see figure~\ref{fig:BSDTpqNiceHD}. As noted in the previous remark, we now assume that the normal vector field points out of the plane, so we actually compute the invariant for the mirror image of $X_{T_{p/q}}$. \\
Since there is only one $\beta$-curve, the only domains that contribute to the type~D module are bigons with no or exactly one e-puncture, so we can compute the type~D module very easily from this nice diagram, using \cite[theorem~7.14]{Zarev}. The labelling of the additional generators is such that there is always an identity morphism from a capital letter to the corresponding lower-case letter with the same index. 

\begin{figure}[hb!]
\centering
\begin{subfigure}{0.45\textwidth}
$$
\begin{tikzcd}[row sep=0.65cm]
f_{p+i} & B_i\arrow[swap]{l}{\rho_1}\arrow{r}{\tau_1}\arrow{d}{1} & C_i\arrow{d}{1}\\
A_i\arrow{r}{\tau_2}\arrow{d}{1} & b_i\arrow{r}{\tau_1} & c_i\\
a_i & P_i \arrow{r}{\rho_1}\arrow[swap]{l}{\tau_1}  & f_{p+1-i}
\end{tikzcd}
$$
\vspace*{-0.4cm}
\caption{$1\leq i\leq q$}
\end{subfigure}
\quad
\begin{subfigure}{0.45\textwidth}
$$
\begin{tikzcd}[row sep=0.65cm] 
F_{q+1-i}\arrow{r}{\rho_3} & Q_i & L_i\arrow[swap]{l}{\varepsilon_2} \arrow{d}{1}\\
G_i\arrow{r}{\varepsilon_2}\arrow{d}{1} & H_i \arrow{r}{\varepsilon_1}\arrow{d}{1} & l_i\\
g_i\arrow{r}{\varepsilon_2} & h_i & F_{i+q}\arrow[swap]{l}{\rho_3}
\end{tikzcd}
$$
\vspace*{-0.4cm}
\caption{$1\leq i\leq q$}
\end{subfigure}
\\
\begin{subfigure}{0.9\textwidth}
$$
\begin{tikzcd}[row sep=0.65cm]
D_{i}\arrow{r}{\rho_3}\arrow[bend left,swap]{rr}{\rho_{23}}\arrow{d}{1} & E_{i}\arrow{r}{\rho_2}\arrow[bend left,swap]{rr}{\rho_{12}}\arrow{d}{1} & P'_{i}\arrow{r}{\rho_1} & f_{i}\\
d_{i}\arrow{r}{\rho_3} & e_{i} & L_{i+q} \arrow[swap]{l}{\varepsilon_2}\arrow{d}{1}\\
G_{i+q}\arrow{r}{\varepsilon_2}\arrow{d}{1} & H_{i+q} \arrow{r}{\varepsilon_1}\arrow{d}{1} & l_{i+q}\\
g_{i+q}\arrow{r}{\varepsilon_2} & h_{i+q} & F_{i+2q}\arrow[swap]{l}{\rho_3}
\end{tikzcd}
$$
\vspace*{-0.4cm}
\caption{$1\leq i\leq p-q$}
\end{subfigure}
\caption{The three basic components of the type D module for $T_{p/q}$}\label{fig:TpqComponents}
\end{figure}

\pagebreak[3]

\noindent
The regions near the sutures are labelled by Greek letters with indices. We also use them synonymously for the corresponding algebra elements. The arc diagram for the glueing surface agrees with that of figure~\ref{fig:2m3_pretzeltangleT_arc1}. There are also two algebra elements $\rho_{12}$ and $\rho_{23}$ which are picked up by rectangles with boundary regions ($\rho_1$ and $\rho_2$) and ($\rho_2$ and $\rho_3$), respectively. These, together with the idempotents are all generators of the algebra that appear in the type~D module. An illustration of these is shown in figure~\ref{fig:2m3_pretzeltangle_algebra1}, where the idempotents are represented by vertices and the algebra elements with moving strands by arrows connecting the corresponding starting and ending idempotents. However, the algebra itself is slightly larger, e.\,g., it contains an algebra element with a single moving strand from 8 to 10 whose differential is non-zero. For details, see~\cite[definitions~2.5 and~2.6]{Zarev}. \\
Our conventions about the order for the algebra multiplication are such that $\rho_2\rho_1$ has two moving strands, see also figure~\ref{fig:APPconventionsTypeD}. Note that $\tau_2\tau_1=0$ and similarly $\varepsilon_2\varepsilon_1=0$.
Finally,  
$$\tau_{12}:=\tau_1\tau_2,\quad\varepsilon_{12}:=\varepsilon_1\varepsilon_2\text{\quad and\quad}\rho_{13}:=\rho_1\rho_3.$$
Let us now calculate the type~D structure. For $1\leq i\leq p+q$, we have bigons from $F_i$ to $f_i$, which contribute honest differentials $F_i\rightarrow f_i$. If we ignore these for a moment, the (graph of the) type~D module breaks up into the connected components shown in figure~\ref{fig:TpqComponents}. 
By cancelling the identity morphisms (using lemma~\ref{lem:AbstractCancellation}), we can simplify the diagrams considerably and get the following:
$$
\begin{tikzcd}[row sep=tiny]  
(a)&1\leq i\leq q:&f_{p+i}&P_i\arrow{r}{\rho_1}\arrow[swap]{l}{\tau_1\tau_2\rho_1}& f_{p+1-i},\\
(b)&1\leq i\leq q:&F_{q+i}\arrow{r}{\rho_3\varepsilon_1\varepsilon_2}&Q_i& F_{q+1-i}\arrow[swap]{l}{\rho_3},\\
(c)&1\leq i\leq p-q:&F_{i+2q}\arrow{r}{\rho_3\varepsilon_1\varepsilon_2\rho_2} & P'_i\arrow{r}{\rho_1}& f_{i},
\end{tikzcd}
$$
noting that $\rho_3\varepsilon_1\varepsilon_2\rho_{12}=0$.
The complete type~D module can be obtained by connecting these pieces and cancelling all $F_i$ and $f_i$. From this, it is apparent that the type~D module is of the form described in the proposition, again, noting that we have actually computed the invariant of the mirror of $X_{T_{p/q}}$.\\
Suppose $q>p\geq0$. Then the corresponding Heegaard diagram is obtained by reflecting the Heegaard diagram from figure~\ref{fig:BSDTpqNiceHD} along a vertical line. This reverses the orientation of the diagram again. Hence the invariant is obtained from the rules above by relabelling the algebra ($\tau_1\leftrightarrow\varepsilon_1$, $\tau_2\leftrightarrow\varepsilon_2$, $\rho_1\leftrightarrow\rho_3$) and the generators ($P_i\leftrightarrow Q_i$, $P'_i\leftrightarrow Q'_i$). The pieces $P_i$ and $Q_i$ turn out to be identical to the ones in the proposition, but the index $i$ runs from 1 to $p$. Instead of the pieces $P'_i$, we get $Q'_i$.
\end{proof}
\begin{Remark}
$X_{T_{p/q}}$ is bounded, except when ($p=1$ and $q=0$) or ($p=0$ and $q=1$).
\end{Remark}
\begin{theorem}\label{thm:2m3pt}
Two knots or links that are related by mutation of the \((2,-3)\)-pretzel tangle, oriented as in figure~\ref{fig:2m3_pretzeltangleT_p1} have the same bigraded knot or link Floer homology, provided that the two open strands \(p\) and \(q\) of the \((2,-3)\)-pretzel tangle have the same colour. If the orientation of one of the two strands is reversed, then in general this statement only holds for \(\delta\)-graded knot or link Floer homology. 
\end{theorem}

\begin{proof}
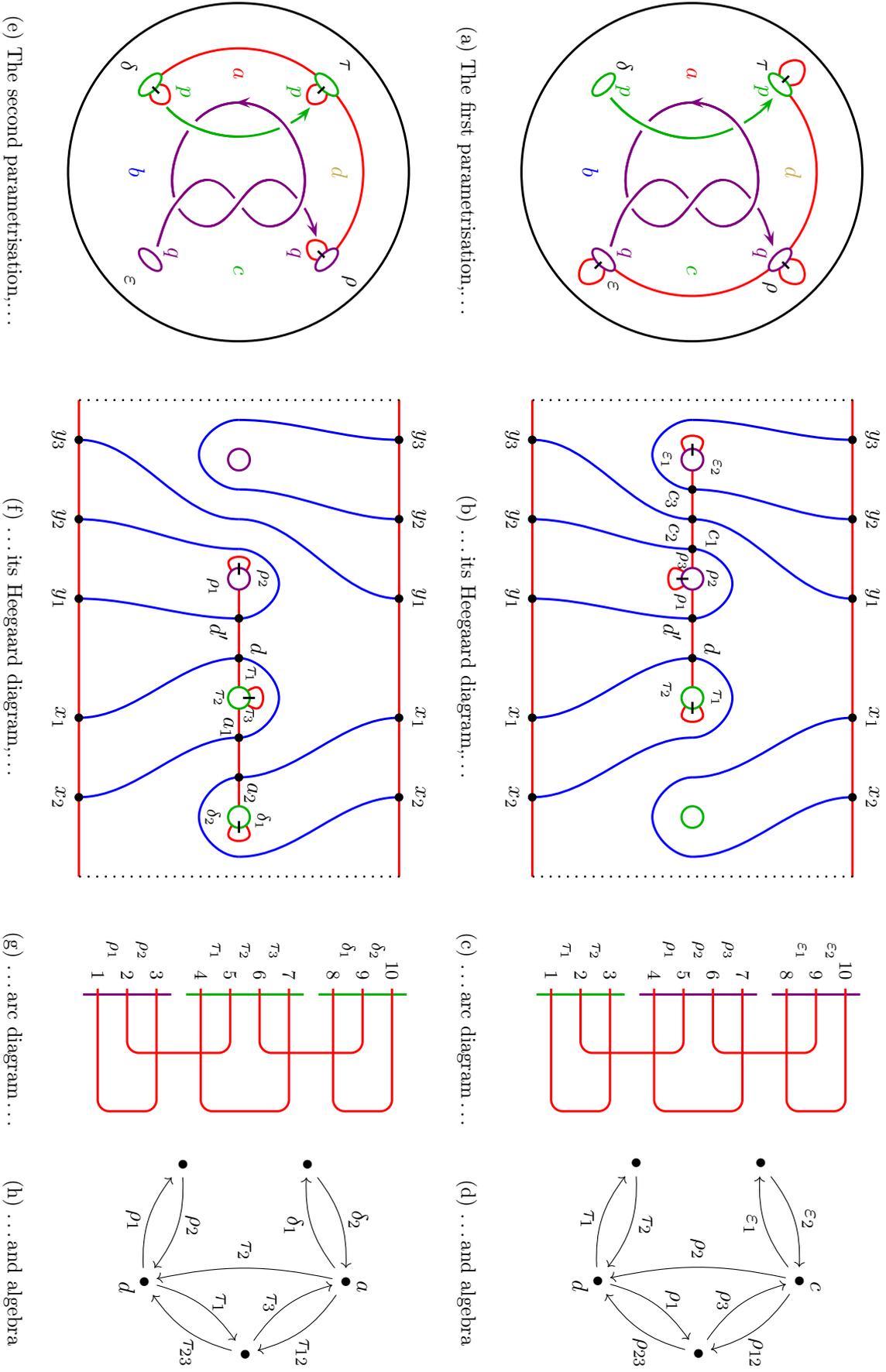
\begin{sidewaysfigure}[hp]
\vspace*{417pt}
\centering
\begin{subfigure}[b]{0.255\textwidth}\centering
\psset{unit=0.53, linewidth=1.1pt}
\begin{pspicture}[showgrid=false](-6.6,-6.4)(4.6,6.4)

\rput(-1,0){
\pscurve[linecolor=red](4;135)(4.5;130)(5;135)(4.5;140)(4;135)
\pscurve[linecolor=red](4;45)(4.5;50)(5;45)(4.5;40)(4;45)
\pscurve[linecolor=red](4;-45)(4.5;-40)(5;-45)(4.5;-50)(4;-45)
}

\psecurve[linecolor=violet](-2.5,1.5)(0,2)(0.75,1)(-0.75,-1)(0,-2)(1.4,-2.6)(6,-3.5)
\psecurve[linecolor=violet]{<-}(6,3.5)(1.4,2.6)(0,2)(-0.75,1)(0.75,-1)(0,-2)(-2.5,-1.5)(-3.25,0)(-2.5,1.5)(0,2)(0.75,1)
\psecurve[linecolor=darkgreen]{<-}(-10,6)(-3.4,2.6)(-2.5,1.5)(-2.1,0)(-2.5,-1.5)(-3.4,-2.6)(-10,-6)
\pscircle*[linecolor=white](-2.5,1.5){0.2}

\psecurve[linecolor=violet](0.75,-1)(0,-2)(-2.5,-1.5)(-3.25,0)(-2.5,1.5)

\pscircle*[linecolor=white](0,2){0.2}
\pscircle*[linecolor=white](0,0){0.2}
\pscircle*[linecolor=white](0,-2){0.2}

\psecurve[linecolor=violet](0.75,1)(-0.75,-1)(0,-2)(1.4,-2.6)(6,-3.5)
\psecurve[linecolor=violet](0,2)(-0.75,1)(0.75,-1)(0,-2)
\psecurve[linecolor=violet](-2.5,-1.5)(-3.25,0)(-2.5,1.5)(0,2)(0.75,1)(-0.75,-1)

\pscircle*[linecolor=white](-2.5,-1.5){0.2}
\psecurve[linecolor=darkgreen](-2.5,1.5)(-2.1,0)(-2.5,-1.5)(-3.4,-2.6)(-10,-6)

\psline[linecolor=violet]{<-}(-3.25,-0.1)(-3.25,0.1)

\rput(-1,0){
\uput{0.6}[135](4;-135){$\delta$}
\uput{0.6}[-135](4;135){$\tau$}
\uput{0.7}[-20](4;45){$\rho$}
\uput{0.7}[20](4;-45){$\varepsilon$}
}

\uput{0.2}[-90](1.6,2.6){$\textcolor{violet}{q}$}
\uput{0.2}[90](1.6,-2.6){$\textcolor{violet}{q}$}
\uput{0.2}[135](-3.3,1.85){$\textcolor{darkgreen}{p}$}
\uput{0.2}[-135](-3.3,-1.85){$\textcolor{darkgreen}{p}$}

\uput{3}[180](-1,0){$\red a$}
\uput{3}[-90](-1,0){$\blue b$}
\uput{3}[0](-1,0){$\darkgreen c$}
\uput{3}[90](-1,0){$\gold d$}

\psarc[linecolor=red](-1,0){4}{-45}{135}
\rput(-1,0){\psset{fillstyle=solid,fillcolor=white}
\rput{45}(4;45){\psellipse[linecolor=violet](0,0)(0.25,0.5)}
\rput{135}(4;135){\psellipse[linecolor=darkgreen](0,0)(0.25,0.5)}
\rput{-135}(4;-135){\psellipse[linecolor=darkgreen](0,0)(0.25,0.5)
}
\rput{-45}(4;-45){\psellipse[linecolor=violet](0,0)(0.25,0.5)}

\psline(4.1;135)(4.4;135)
\psline(4.1;45)(4.4;45)
\psline(4.1;-45)(4.4;-45)
}
\rput(-1,0){\pscircle(0,0){5.5}}

\end{pspicture}
\caption{The first parametrisation,\dots}\label{fig:2m3_pretzeltangleT_p1}
\end{subfigure}
\begin{subfigure}[b]{0.4\textwidth}\centering
\psset{unit=0.68, linewidth=1.1pt}
\begin{pspicture}[showgrid=false](-6.05,-5)(6.05,5)

\psline[linecolor=red](-6,-4)(6,-4)
\psline[linecolor=red](-6,4)(6,4)
\psline[linestyle=dotted](-6,-4)(-6,4)
\psline[linestyle=dotted](6,-4)(6,4)

\psline[linecolor=red](-4.5,0)(1.5,0)
\pscurve[linecolor=red](-1.5,0)(-1.75,-0.5)(-1.25,-0.5)(-1.5,0)
\pscurve[linecolor=red](-4.5,0)(-5,-0.25)(-5,0.25)(-4.5,0)
\pscurve[linecolor=red](1.5,0)(2,-0.25)(2,0.25)(1.5,0)

\pscircle[linecolor=violet,fillstyle=solid,fillcolor=white](-4.5,0){0.3}
\pscircle[linecolor=violet,fillstyle=solid,fillcolor=white](-1.5,0){0.3}
\pscircle[linecolor=darkgreen,fillstyle=solid,fillcolor=white](1.5,0){0.3}
\pscircle[linecolor=darkgreen,fillstyle=solid,fillcolor=white](4.5,0){0.3}

\psline(-4.6,0)(-4.9,0)
\psline(-1.5,-0.1)(-1.5,-0.4)
\psline(1.6,0)(1.9,0)

{
\psset{linecolor=blue}
\psecurve(-5,-4)(-5.5,0)(-5,4)(-5.5,8)
\psecurve(-4.5,1)(-5.5,0)(-4.625,-1)(-3.75,0)(-4.625,1)
\psecurve(-3,-4)(-3.75,0)(-3,4)(-3.75,8)

\psecurve(-3,-8)(-5,-4)(-3,0)(-5,4)
\psecurve(-1,-4)(-3,0)(-1,4)(-3,8)

\psecurve(-1,4)(-0.5,0)(-1,-4)(-0.5,-8)
\psecurve(-1.5,-1)(-0.5,0)(-1.375,1)(-2.25,0)(-1.375,-1)
\psecurve(-3,4)(-2.25,0)(-3,-4)(-2.25,-8)

\psecurve(0.5,-8)(2,-4)(0.5,0)(2,4)
\psecurve(1.5,-1)(2.5,0)(1.5,1)(0.5,0)(1.5,-1)
\psecurve(2.5,-8)(4,-4)(2.5,0)(4,4)

\psecurve(5.5,8)(4,4)(5.5,0)(4,-4)
\psecurve(4.5,1)(3.5,0)(4.5,-1)(5.5,0)(4.5,1)
\psecurve(3.5,8)(2,4)(3.5,0)(2,-4)
}

\psdots(-5,4)(-3,4)(-1,4)(2,4)(4,4)
\psdots(-3.75,0)(-3,0)(-2.25,0)(-0.5,0)(0.5,0)
\psdots(-5,-4)(-3,-4)(-1,-4)(2,-4)(4,-4)

\rput[c](-3.5,-0.5){$c_3$}
\rput[c](-2.65,-0.5){$c_2$}
\rput[c](-2.5,0.5){$c_1$}

\rput[c](-0.15,-0.5){$d'$}
\rput[c](0.25,0.5){$d$}

\rput[c](-5,4.5){$y_3$}
\rput[c](-3,4.5){$y_2$}
\rput[c](-1,4.5){$y_1$}
\rput[c](2,4.5){$x_1$}
\rput[c](4,4.5){$x_2$}

\rput[c](-5,-4.5){$y_3$}
\rput[c](-3,-4.5){$y_2$}
\rput[c](-1,-4.5){$y_1$}
\rput[c](2,-4.5){$x_1$}
\rput[c](4,-4.5){$x_2$}

\rput(-4.3,0.6){\footnotesize $\varepsilon_2$}
\rput(-4.5,-0.65){\footnotesize $\varepsilon_1$}

\rput(-1.95,-0.25){\footnotesize $\rho_3$}
\rput(-1.5,0.5){\footnotesize $\rho_2$}
\rput(-0.9,-0.3){\footnotesize $\rho_1$}

\rput(1.3,-0.65){\footnotesize $\tau_2$}
\rput(1.5,0.6){\footnotesize $\tau_1$}
\end{pspicture}
\caption{\dots its Heegaard diagram,\dots}
\label{fig:2m3_pretzeltangleHD_p1}
\end{subfigure}
\begin{subfigure}[b]{0.15\textwidth}\centering
\psset{unit=0.5, linewidth=1.1pt}
\begin{pspicture}[showgrid=false](-2,-7)(4.5,7)
\psline[linecolor=violet](0,5.5)(0,2.5)
\psline[linecolor=violet](0,2)(0,-2)
\psline[linecolor=darkgreen](0,-2.5)(0,-5.5)

\psline[linecolor=red,linearc=0.4](-0.2,5)(4,5)(4,3)(-0.2,3)
\psline[linecolor=red,linearc=0.4](-0.2,4)(2,4)(2,0.5)(-0.2,0.5)
\psline[linecolor=red,linearc=0.4](-0.2,1.5)(4,1.5)(4,-1.5)(-0.2,-1.5)
\psline[linecolor=red,linearc=0.4](-0.2,-4)(2,-4)(2,-0.5)(-0.2,-0.5)
\psline[linecolor=red,linearc=0.4](-0.2,-5)(4,-5)(4,-3)(-0.2,-3)

\footnotesize
\rput[r](-0.5,5){$10$}
\rput[r](-0.5,4){$9$}
\rput[r](-0.5,3){$8$}
\rput[r](-0.5,1.5){$7$}
\rput[r](-0.5,0.5){$6$}
\rput[r](-0.5,-0.5){$5$}
\rput[r](-0.5,-1.5){$4$}
\rput[r](-0.5,-3){$3$}
\rput[r](-0.5,-4){$2$}
\rput[r](-0.5,-5){$1$}

\rput[r](-1.2,4.5){$\varepsilon_2$}
\rput[r](-1.2,3.5){$\varepsilon_1$}
\rput[r](-1.2,1){$\rho_3$}
\rput[r](-1.2,0){$\rho_2$}
\rput[r](-1.2,-1){$\rho_1$}
\rput[r](-1.2,-3.5){$\tau_2$}
\rput[r](-1.2,-4.5){$\tau_1$}
\end{pspicture}
\caption{\dots arc diagram\,\dots}\label{fig:2m3_pretzeltangleT_arc1}
\end{subfigure}
\begin{subfigure}[b]{0.17\textwidth}\centering
\begin{tikzpicture}
\node (a) at (canvas polar cs:angle=144,radius=1.8cm) {$\bullet$};
\node (b) at (canvas polar cs:angle=72,radius=1.8cm) {$\bullet$};
\node (bl) at (canvas polar cs:angle=72,radius=2.1cm) {$c$};
\node (c) at (canvas polar cs:angle=0,radius=1.8cm) {$\bullet$};
\node (d) at (canvas polar cs:angle=-72,radius=1.8cm) {$\bullet$};
\node (dl) at (canvas polar cs:angle=-72,radius=2.1cm) {$d$};
\node (e) at (canvas polar cs:angle=-144,radius=1.8cm) {$\bullet$};
\node (white) at (0,-3.3) {\white$\bullet$};

\draw[->,below] (b) to[bend left=20] node{$\varepsilon_1$} (a);
\draw[<-,above] (b) to[bend right=20] node{$\varepsilon_2$} (a);

\draw[->,left,near start] (c) to[bend left=20] node{$\rho_3$} (b);
\draw[<-,right] (c) to[bend right=20] node{$\rho_{12}$} (b);

\draw[->,right] (c) to[bend left=20] node{$\rho_{23}$} (d);
\draw[<-,left,near start] (c) to[bend right=20] node{$\rho_1$} (d);

\draw[->,left] (b) to[bend right=12] node{$\rho_2$} (d);

\draw[->,below] (d) to[bend left=20] node{$\tau_1$} (e);
\draw[<-,above] (d) to[bend right=20] node{$\tau_2$} (e);
\end{tikzpicture}
\caption{\dots and algebra}\label{fig:2m3_pretzeltangle_algebra1}
\end{subfigure}
\\
\begin{subfigure}[b]{0.255\textwidth}\centering
\psset{unit=0.53, linewidth=1.1pt}
\begin{pspicture}[showgrid=false](-6.6,-6.4)(4.6,6.4)

\psecurve[linecolor=violet](-2.5,1.5)(0,2)(0.75,1)(-0.75,-1)(0,-2)(1.4,-2.6)(6,-3.5)
\psecurve[linecolor=violet]{<-}(6,3.5)(1.1,2.5)(0,2)(-0.75,1)(0.75,-1)(0,-2)(-2.5,-1.5)(-3.25,0)(-2.5,1.5)(0,2)(0.75,1)
\psecurve[linecolor=darkgreen]{<-}(-10,6)(-3.1,2.3)(-2.5,1.5)(-2.15,0)(-2.5,-1.5)(-3.1,-2.3)(-10,-6)
\pscircle*[linecolor=white](-2.5,1.5){0.2}

\psecurve[linecolor=violet](0.75,-1)(0,-2)(-2.5,-1.5)(-3.25,0)(-2.5,1.5)

\pscircle*[linecolor=white](0,2){0.2}
\pscircle*[linecolor=white](0,0){0.2}
\pscircle*[linecolor=white](0,-2){0.2}

\psecurve[linecolor=violet](0.75,1)(-0.75,-1)(0,-2)(1.4,-2.6)(6,-3.5)
\psecurve[linecolor=violet](0,2)(-0.75,1)(0.75,-1)(0,-2)
\psecurve[linecolor=violet](-2.5,-1.5)(-3.25,0)(-2.5,1.5)(0,2)(0.75,1)(-0.75,-1)

\pscircle*[linecolor=white](-2.5,-1.5){0.2}
\psecurve[linecolor=darkgreen](-2.5,1.5)(-2.15,0)(-2.5,-1.5)(-3.1,-2.3)(-10,-6)

\psline[linecolor=violet]{<-}(-3.25,-0.1)(-3.25,0.1)

\rput(-1,0){
\uput{0.6}[-135](4;-135){$\delta$}
\uput{0.6}[135](4;135){$\tau$}
\uput{0.7}[45](4;45){$\rho$}
\uput{0.7}[-45](4;-45){$\varepsilon$}
}

\uput{0.6}[-90](1.6,2.6){$\textcolor{violet}{q}$}
\uput{0.2}[90](1.6,-2.6){$\textcolor{violet}{q}$}
\uput{0.6}[-90](-3.6,2.6){$\textcolor{darkgreen}{p}$}
\uput{0.6}[90](-3.6,-2.6){$\textcolor{darkgreen}{p}$}

\uput{3}[180](-1,0){$\red a$}
\uput{3}[-90](-1,0){$\blue b$}
\uput{3}[0](-1,0){$\darkgreen c$}
\uput{3}[90](-1,0){$\gold d$}

\rput(-1,0){\pscircle(0,0){5.5}}

\rput(-1,0){
\pscurve[linecolor=red](4;135)(3.6;130)(3.2;135)(3.6;140)(4;135)
\pscurve[linecolor=red](4;45)(3.6;50)(3.2;45)(3.6;40)(4;45)
\pscurve[linecolor=red](4;-135)(3.6;-130)(3.2;-135)(3.6;-140)(4;-135)
}
\psarc[linecolor=red](-1,0){4}{45}{-135}
\rput(-1,0){\psset{fillstyle=solid,fillcolor=white}
\rput{45}(4;45){\psellipse[linecolor=violet](0,0)(0.25,0.5)}
\rput{135}(4;135){\psellipse[linecolor=darkgreen](0,0)(0.25,0.5)}
\rput{-135}(4;-135){\psellipse[linecolor=darkgreen](0,0)(0.25,0.5)
}
\rput{-45}(4;-45){\psellipse[linecolor=violet](0,0)(0.25,0.5)}

\psline(3.9;135)(3.6;135)
\psline(3.9;45)(3.6;45)
\psline(3.9;-135)(3.6;-135)
}

\end{pspicture}
\caption{The second parametrisation,\dots}\label{fig:2m3_pretzeltangleT_p2}
\end{subfigure}
\begin{subfigure}[b]{0.4\textwidth}\centering
\psset{unit=0.68, linewidth=1.1pt}
\begin{pspicture}[showgrid=false](-6.05,-5)(6.05,5)

\psline[linecolor=red](-6,-4)(6,-4)
\psline[linecolor=red](-6,4)(6,4)
\psline[linestyle=dotted](-6,-4)(-6,4)
\psline[linestyle=dotted](6,-4)(6,4)

\psline[linecolor=red](-1.5,0)(4.5,0)
\pscurve[linecolor=red](1.5,0)(1.75,0.5)(1.25,0.5)(1.5,0)
\pscurve[linecolor=red](-1.5,0)(-1.95,-0.25)(-1.95,0.25)(-1.5,0)
\pscurve[linecolor=red](4.5,0)(5,-0.25)(5,0.25)(4.5,0)

\pscircle[linecolor=violet,fillstyle=solid,fillcolor=white](-4.5,0){0.3}
\pscircle[linecolor=violet,fillstyle=solid,fillcolor=white](-1.5,0){0.3}
\pscircle[linecolor=darkgreen,fillstyle=solid,fillcolor=white](1.5,0){0.3}
\pscircle[linecolor=darkgreen,fillstyle=solid,fillcolor=white](4.5,0){0.3}

\psline(-1.6,0)(-1.9,0)
\psline(1.5,0.1)(1.5,0.4)
\psline(4.6,0)(4.9,0)

{
\psset{linecolor=blue}
\psecurve(-5,-4)(-5.5,0)(-5,4)(-5.5,8)
\psecurve(-4.5,1)(-5.5,0)(-4.625,-1)(-3.75,0)(-4.625,1)
\psecurve(-3,-4)(-3.75,0)(-3,4)(-3.75,8)

\psecurve(-3,-8)(-5,-4)(-3,0)(-5,4)
\psecurve(-1,-4)(-3,0)(-1,4)(-3,8)

\psecurve(-1,4)(-0.5,0)(-1,-4)(-0.5,-8)
\psecurve(-1.5,-1)(-0.5,0)(-1.375,1)(-2.25,0)(-1.375,-1)
\psecurve(-3,4)(-2.25,0)(-3,-4)(-2.25,-8)

\psecurve(0.5,-8)(2,-4)(0.5,0)(2,4)
\psecurve(1.5,-1)(2.5,0)(1.5,1)(0.5,0)(1.5,-1)
\psecurve(2.5,-8)(4,-4)(2.5,0)(4,4)

\psecurve(5.5,8)(4,4)(5.5,0)(4,-4)
\psecurve(4.5,1)(3.5,0)(4.5,-1)(5.5,0)(4.5,1)
\psecurve(3.5,8)(2,4)(3.5,0)(2,-4)
}

\psdots(-5,4)(-3,4)(-1,4)(2,4)(4,4)
\psdots(-0.5,0)(0.5,0)(2.5)(3.5)
\psdots(-5,-4)(-3,-4)(-1,-4)(2,-4)(4,-4)

\rput[c](3.9,0.3){$a_2$}
\rput[c](2.25,-0.3){$a_1$}
\rput[c](-0.15,-0.5){$d'$}
\rput[c](0.25,0.5){$d$}

\rput[c](-5,4.5){$y_3$}
\rput[c](-3,4.5){$y_2$}
\rput[c](-1,4.5){$y_1$}
\rput[c](2,4.5){$x_1$}
\rput[c](4,4.5){$x_2$}

\rput[c](-5,-4.5){$y_3$}
\rput[c](-3,-4.5){$y_2$}
\rput[c](-1,-4.5){$y_1$}
\rput[c](2,-4.5){$x_1$}
\rput[c](4,-4.5){$x_2$}

\rput(4.6,0.6){\footnotesize $\delta_1$}
\rput(4.5,-0.65){\footnotesize $\delta_2$}

\rput(1.95,0.25){\footnotesize $\tau_3$}
\rput(1.5,-0.5){\footnotesize $\tau_2$}
\rput(0.9,0.3){\footnotesize $\tau_1$}

\rput(-1.3,-0.65){\footnotesize $\rho_1$}
\rput(-1.5,0.6){\footnotesize $\rho_2$}

\end{pspicture}
\caption{\dots its Heegaard diagram,\dots}
\label{fig:2m3_pretzeltangleHD_p2}
\end{subfigure}
\begin{subfigure}[b]{0.15\textwidth}\centering
\psset{unit=0.5, linewidth=1.1pt}
\begin{pspicture}[showgrid=false](-2,-7)(4.5,7)
\psline[linecolor=darkgreen](0,5.5)(0,2.5)
\psline[linecolor=darkgreen](0,2)(0,-2)
\psline[linecolor=violet](0,-2.5)(0,-5.5)

\psline[linecolor=red,linearc=0.4](-0.2,5)(4,5)(4,3)(-0.2,3)
\psline[linecolor=red,linearc=0.4](-0.2,4)(2,4)(2,0.5)(-0.2,0.5)
\psline[linecolor=red,linearc=0.4](-0.2,1.5)(4,1.5)(4,-1.5)(-0.2,-1.5)
\psline[linecolor=red,linearc=0.4](-0.2,-4)(2,-4)(2,-0.5)(-0.2,-0.5)
\psline[linecolor=red,linearc=0.4](-0.2,-5)(4,-5)(4,-3)(-0.2,-3)

\footnotesize
\rput[r](-0.5,5){$10$}
\rput[r](-0.5,4){$9$}
\rput[r](-0.5,3){$8$}
\rput[r](-0.5,1.5){$7$}
\rput[r](-0.5,0.5){$6$}
\rput[r](-0.5,-0.5){$5$}
\rput[r](-0.5,-1.5){$4$}
\rput[r](-0.5,-3){$3$}
\rput[r](-0.5,-4){$2$}
\rput[r](-0.5,-5){$1$}

\rput[r](-1.2,4.5){$\delta_2$}
\rput[r](-1.2,3.5){$\delta_1$}
\rput[r](-1.2,1){$\tau_3$}
\rput[r](-1.2,0){$\tau_2$}
\rput[r](-1.2,-1){$\tau_1$}
\rput[r](-1.2,-3.5){$\rho_2$}
\rput[r](-1.2,-4.5){$\rho_1$}
\end{pspicture}
\caption{\dots arc diagram\,\dots}
\label{fig:2m3_pretzeltangleHD_arc2}
\end{subfigure}
\begin{subfigure}[b]{0.17\textwidth}\centering
\begin{tikzpicture}
\node (a) at (canvas polar cs:angle=144,radius=1.8cm) {$\bullet$};
\node (b) at (canvas polar cs:angle=72,radius=1.8cm) {$\bullet$};
\node (bl) at (canvas polar cs:angle=72,radius=2.1cm) {$a$};
\node (c) at (canvas polar cs:angle=0,radius=1.8cm) {$\bullet$};
\node (d) at (canvas polar cs:angle=-72,radius=1.8cm) {$\bullet$};
\node (dl) at (canvas polar cs:angle=-72,radius=2.1cm) {$d$};
\node (e) at (canvas polar cs:angle=-144,radius=1.8cm) {$\bullet$};
\node (white) at (0,-3.3) {\white$\bullet$};

\draw[->,below] (b) to[bend left=20] node{$\delta_1$} (a);
\draw[<-,above] (b) to[bend right=20] node{$\delta_2$} (a);

\draw[->,left,near start] (c) to[bend left=20] node{$\tau_3$} (b);
\draw[<-,right] (c) to[bend right=20] node{$\tau_{12}$} (b);

\draw[->,right] (c) to[bend left=20] node{$\tau_{23}$} (d);
\draw[<-,left,near start] (c) to[bend right=20] node{$\tau_1$} (d);

\draw[->,left] (b) to[bend right=12] node{$\tau_2$} (d);

\draw[->,below] (d) to[bend left=20] node{$\rho_1$} (e);
\draw[<-,above] (d) to[bend right=20] node{$\rho_2$} (e);
\end{tikzpicture}
\caption{\dots and algebra}\label{fig:2m3_pretzeltangle_algebra2}
\end{subfigure}
\caption{The first (a-d) and second (e-h) parametrisation for the $(2,-3)$-pretzel tangle. The Heegaard diagrams, arc diagrams and algebras are actually the ones for the mirror of the tangle, i.\,e.\ the $(-2,3)$-pretzel tangle, see remark~\ref{rem:conventions3}.}\label{fig:2m3_pretzeltangle_p2}
\end{sidewaysfigure}
We prove this result by calculating the type~D invariant for the $(2,-3)$-pretzel tangle with two different parametrisations. This computation is essentially the same as for rational tangles -- except that the nice diagrams are quite large: In one case, there are over 400 generators, in the other nearly 3000. Therefore, we use the  Mathematica package~\cite{BSFH.m} to compute the invariants, see notebooks~\cite{notebook1} and~\cite{notebook2}. The program first computes the type~D modules over a given strands algebra from nice diagrams. Then, it cancels generators in a certain order, namely such that the generators from the non-niceified Heegaard diagrams in Figures~\ref{fig:2m3_pretzeltangleHD_p1} and~\ref{fig:2m3_pretzeltangleHD_p2} survive. For more details, see appendix~\ref{app:manualBSFH}, where we explain how to use the package~\cite{BSFH.m}. Also note that in both calculations, we have performed some homotopies to simplify the results, using the clean-up lemma~\ref{lem:AbstractCleanUp}; for details, see the two notebooks. \\
Figure~\ref{fig:ProofOf2m3prop} shows the results of these computations (after reversing arrows, gradings and the algebras, see remark~\ref{rem:conventions3}): In both cases, we have arranged the generators in a grid according to their Alexander bigrading, where $\odot$ marks the origin $(0,0)$. The first Alexander grading (corresponding to the $\darkgreen p$-strand) increases from top to bottom and the second grading (corresponding to the $\textcolor{violet}{q}$-strand) increases from left to right. Also, the  $\delta$-grading of each generator is shown in the exponent of $\delta$ in front of each generator, compare with example~\ref{exa:pretzeltangle}. Furthermore, each generator is coloured corresponding to its site, see figures~\ref{fig:2m3_pretzeltangleT_p1} and~\ref{fig:2m3_pretzeltangleT_p2}. The labelled arrows connecting these generators show the differential of the type~D structures and the algebra elements picked up along those differentials. Here, we use the same conventions about the algebra as in the previous proposition.\\
Now, mutation about the $y$-axis corresponds to interchanging $\rho\leftrightarrow\tau$ and $\varepsilon\leftrightarrow\delta$ for algebra elements and the sites $\red a$ and $\textcolor{darkgreen}{c}$. Let us ignore the Alexander grading for a moment and just consider the directed graphs consisting of the vertices and the solid arrows. Both of them have four connected components: Two solitary generators, a loop of three generators and a loop with four generators. We can choose an isomorphism between them that interchanges generators of the sites $\red a$ and $\textcolor{darkgreen}{c}$, preserves the $\delta$-grading and changes the labelling of the arrows exactly as specified by mutation about the $y$-axis. In other words, the $\delta$-graded invariants are homotopic to each other. \\
\begin{figure}[t]
\begin{subfigure}[b]{\textwidth}\centering
$$
\begin{tikzcd}[column sep=0.8cm,row sep=0.6cm]   
%
%
\delta^{-2}\gold{x_2d'}
&
%
\delta^{-\frac{3}{2}}\darkgreen{x_2c_1}
\arrow[leftarrow,swap]{l}{\rho_2}
&
&
%
\delta^{-\frac{3}{2}}\darkgreen{x_2c_2} 
&
&
%
\delta^{-\frac{3}{2}}\darkgreen{x_2c_3} 
\\
%
%
&
&
&
\odot 
\\
%
%
&
%
\delta^{-\frac{3}{2}}\darkgreen{x_1c_1} 
\arrow[leftarrow]{luu}{\tau_{12}\rho_2}
&
%
\delta^{-1}\gold{dy_2} 
\arrow[leftarrow,swap]{luu}{\rho_{13}\tau_{12}}
\arrow[leftarrow,swap]{l}{\rho_{13}}
&
%
\delta^{-\frac{3}{2}}\darkgreen{x_1c_2} 
\arrow[leftarrow,bend left]{uu}{\rho_{13}\tau_{12}\rho_2\varepsilon_{12}}
&
%
\delta^{-1}\gold{dy_3}
\arrow[leftarrow,swap]{luu}{\rho_{13}\tau_{12}} 
\arrow[leftarrow,swap]{l}{\rho_{13}}
&
%
\delta^{-\frac{3}{2}}\darkgreen{x_1c_3} 
\end{tikzcd}
$$
\caption{The result from \cite{notebook1} for the first parametrisation}\label{fig:2m3ptTypeDResult}
\end{subfigure}
\begin{subfigure}[b]{\textwidth}\centering

$$
\begin{tikzcd}[column sep=0.54cm,row sep=0.6cm]   
%
%
\delta^{-2}\gold{x_2d'}
 \\
%
%
\delta^{-\frac{3}{2}}{\red a_1y_1}
\arrow[leftarrow]{u}{\tau_2}
 &
 &
%
\delta^{-\frac{3}{2}}{\red a_1y_2}~ \delta^{-\frac{3}{2}}{\red a_2y_1} 
\arrow[leftarrow,start anchor=150,swap]{ull}{\rho_{12}\tau_2}
&
\odot &
%
\delta^{-\frac{3}{2}}{\red a_1y_3}~ \delta^{-\frac{3}{2}}{\red a_2y_2} 
\arrow[leftarrow,bend right,swap]{ll}{\tau_{13}\rho_{12}\tau_2\delta_{12}}
&
&
%
\delta^{-\frac{3}{2}}{\red a_2y_3} 
\\
%
%
&
&
%
\delta^{-1}\gold{dy_2} 
\arrow[leftarrow,end anchor=-130,swap,pos=0.4]{u}{\tau_{13}}
\arrow[leftarrow]{ull}{\tau_{13}\rho_{12}}
&
&
%
\delta^{-1}\gold{dy_3} 
\arrow[leftarrow,end anchor=-140,swap,pos=0.4]{u}{\tau_{13}}
\arrow[leftarrow,end anchor=-20]{ull}{\tau_{13}\rho_{12}}
\end{tikzcd}
$$
\caption{The result from \cite{notebook2} for the second parametrisation}\label{fig:2m3ptmutTypeDResult}
\end{subfigure}
\caption{Two type~D modules for the $(2,-3)$-pretzel tangle from theorem~\ref{thm:2m3pt}. Note that we have reversed all arrows, gradings and algebras since the notebooks compute the invariants of the mirrors, see remark~\ref{rem:conventions3}.}\label{fig:ProofOf2m3prop}
\end{figure}
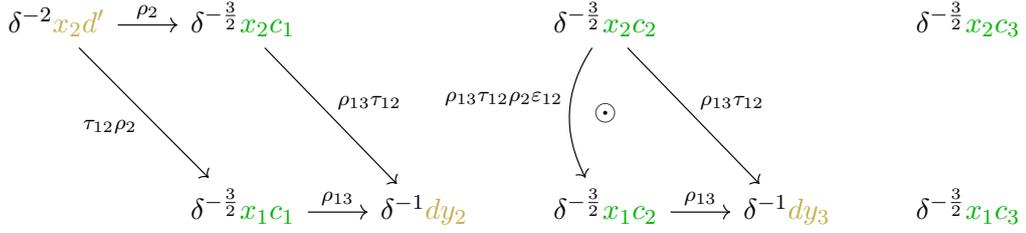
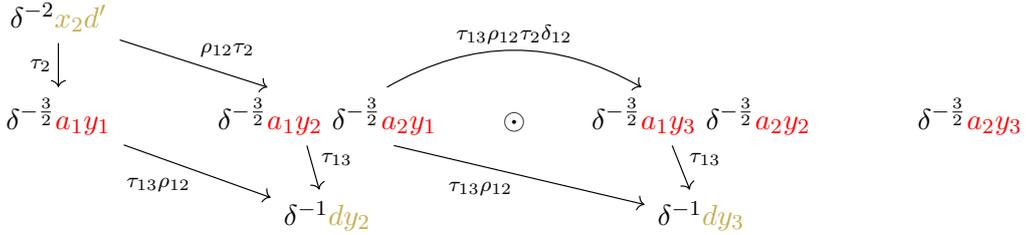
\noindent
Let us now consider the Alexander grading. Obviously, the bigraded invariants are different, already on the level of generators. This is to be expected, since the multivariate Alexander polynomial is only mutation invariant if the colours of the two open stands in the mutating tangle are the same, see theorem~\sref{thm:mutation}. So we want to collapse Alexander gradings, i.\,e.\ add the two Alexander gradings to a single $\mathbb{Z}$-grading. This means that those generators on the diagonals going from left to right and from bottom to top live in the same Alexander grading. It is easy to see that the identification above can be chosen such that it also preserves this single-variate Alexander grading. \\
Together with theorem~\ref{thm:glueingCAT} and the observation that mutation about the $x$-axis does not change the knot or link, this finishes the proof of theorem~\ref{thm:2m3pt}.
\end{proof}

\begin{example}\label{exa:counterexampleMUT}
Let us reverse the orientation of a single strand, say the $\darkgreen p$-strand, and see why bigraded mutation invariance fails. For example, this is the case for the Conway and Kinoshita-Terasaka knots used in the counterexample from \cite[theorems~1.1 and~1.2]{OSmutation}, see figure~\ref{fig:KTCcounterexample}.\\
Changing the orientation of the $\darkgreen p$-strand has the effect of reversing the Alexander grading ${\darkgreen p}\leftrightarrow {\darkgreen p}^{-1}$. In the two type~D structures of the proof above, this means that the first Alexander grading now increases from bottom to top. If we collapse the Alexander bigrading to a single $\mathbb{Z}$-grading, generators on the diagonals going from left to right and top to bottom now live in the same Alexander grading. Obviously, these two invariants are not identical.
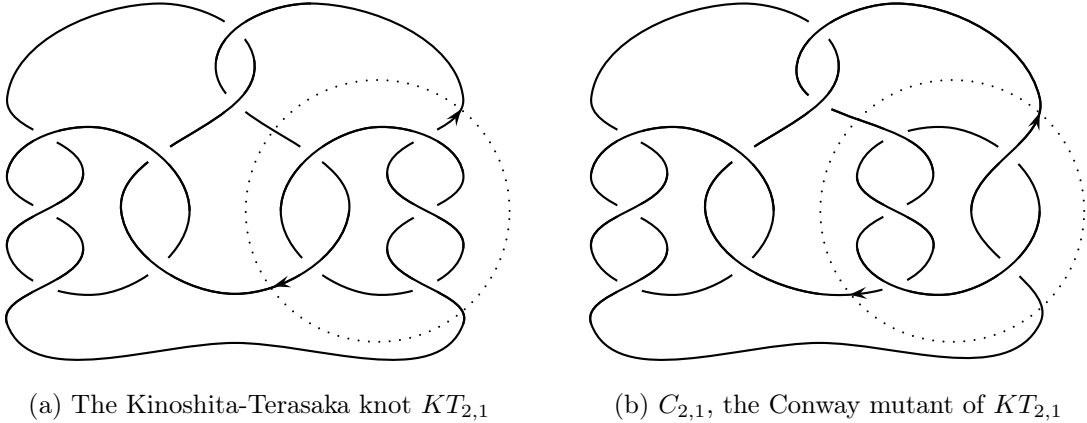
\begin{figure}[h]
\psset{unit=0.5}\centering
\begin{subfigure}[b]{0.47\textwidth}\centering
\begin{pspicture}[showgrid=false](-6.2,-4.3)(7.47,5.7)
\psecurve(-6,-1)(-4,1)(-6,3)(-2,5.5)(0.5,4)(-3,0)(0,-2.2)(3,0)(-0.5,4)(2,5.5)(6,3)(4,1)(6,-1)(3.5,-2.2)(1.2,0)(3.5,2.2)(6,1)(4,-1)(6,-3)(0,-3.5)(-6,-3)(-4,-1)(-6,1)(-3.5,2.2)(-1.2,0)(-3.5,-2.2)(-6,-1)(-4,1)(-6,3)
\psdots[dotsize=10pt,linecolor=white](-5,2)(-5,0)(-5,-2)(-2,1.5)(-2,-1.5)(2,1.5)(2,-1.5)(5,2)(5,0)(5,-2)(0,2.85)(0,4.8)
\psecurve(3,0)(-0.5,4)(2,5.5)(6,3)
\psecurve(0,-2.2)(-3,0)(0.5,4)(-2,5.5)
\psdot[dotsize=10pt,linecolor=white](-2,1.5)
\psecurve(0,-3.5)(-6,-3)(-4,-1)(-6,1)
\psecurve(-3.5,-2.2)(-6,-1)(-4,1)(-6,3)
\psecurve(-4,-1)(-6,1)(-3.5,2.2)(-1.2,0)(-3.5,-2.2)
\psecurve(-0.5,4)(3,0)(0,-2.2)(-3,0)(0.5,4)
\psecurve(3.5,-2.2)(1.2,0)(3.5,2.2)(6,1)(4,-1)
\psecurve(6,3)(4,1)(6,-1)(3.5,-2.2)
\psecurve(6,1)(4,-1)(6,-3)(0,-3.5)
\pscircle[linestyle=dotted,linecolor=black](3.75,0){3.5}
\psline[linecolor=black]{<-}(1,-2)(1.2,-1.94)
\psline[linecolor=black]{<-}(5.95,2.7)(5.805,2.5)
\end{pspicture}
\caption{The Kinoshita-Terasaka knot $KT_{2,1}$}\label{fig:KTknot}
\end{subfigure}
\quad
\begin{subfigure}[b]{0.46\textwidth}\centering
\begin{pspicture}[showgrid=false](-6.2,-4.3)(7.2,5.7)
\psecurve(-6,-1)(-4,1)(-6,3)(-2,5.5)(0.5,4)(-3,0)(0,-2.2)(3,-1)(1,1)(3.5,2.2)(5.8,0)(3.5,-2.2)(1,-1)(3,1)(-0.6,3.75)(2,5.5)(5.8,3)(4,0)(5.8,-3)(0,-3.5)(-6,-3)(-4,-1)(-6,1)(-3.5,2.2)(-1.2,0)(-3.5,-2.2)(-6,-1)(-4,1)(-6,3)
\psdots[dotsize=10pt,linecolor=white](-5,2)(-5,0)(-5,-2)(-2,1.5)(-2,-1.5)(5,1.5)(5,-1.5)(2,2)(2,0)(2,-2)(0,2.85)(0,4.8)
\psecurve(1,-1)(3,1)(-0.6,3.75)(2,5.5)
\psdot[dotsize=10pt,linecolor=white](0,2.85)
\psecurve(0,-2.2)(-3,0)(0.5,4)(-2,5.5)
\psdot[dotsize=10pt,linecolor=white](-2,1.5)
\psecurve(0,-3.5)(-6,-3)(-4,-1)(-6,1)
\psecurve(-3.5,-2.2)(-6,-1)(-4,1)(-6,3)
\psecurve(-4,-1)(-6,1)(-3.5,2.2)(-1.2,0)(-3.5,-2.2)
\psecurve(0.5,4)(-3,0)(0,-2.2)(3,-1)
\psecurve(3,1)(-0.6,3.75)(2,5.5)(5.8,3)(4,0)(5.8,-3)
\psecurve(3.5,2.2)(5.8,0)(3.5,-2.2)(1,-1)(3,1)
\psecurve(0,-2.2)(3,-1)(1,1)(3.5,2.2)
\pscircle[linestyle=dotted,linecolor=black](3.5,0){3.5}
\psline[linecolor=black]{<-}(0.8,-2.22)(1,-2.2)
\psline[linecolor=black]{<-}(5.8,2.6)(5.725,2.4)
\end{pspicture}
\caption{$C_{2,1}$, the Conway mutant of $KT_{2,1}$}\label{fig:Cknot}
\end{subfigure}
\caption{Ozsv\'{a}th and Szab\'{o}'s counterexample from \cite{OSmutation} for bigraded mutation invariance of $\HFK$}\label{fig:KTCcounterexample}
\end{figure} 
\end{example}
\begin{example}\label{exa:exampleMUT}
As a non-trivial example to which the proposition can be applied to give the same bigraded link Floer homologies, we can consider the mutant pair in figure~\ref{fig:BigradedMutInv}. Indeed, their link Floer homologies both live in two consecutive $\delta$-gradings and in each of them, the  Poincaré polynomial is
$$t^{-6}+3t^{-4}+3t^{-2}+3t^{\pm0}+3t^{+2}+3t^{+4}+t^{+6}.$$
See for example \cite[section~5.4]{OSHFLThurston} and \cite[section~3]{OSHDsandFloer} for explicit calculations. 
%
%
%
%
%
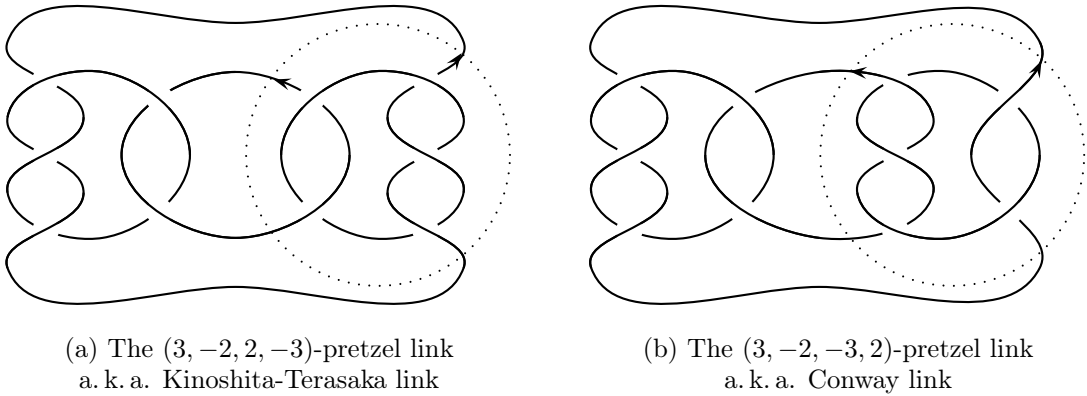
\begin{figure}[h]
\captionsetup{justification=centering}
\psset{unit=0.5}\centering
\begin{subfigure}[b]{0.47\textwidth}\centering
\begin{pspicture}[showgrid=false](-6.2,-4.3)(7.45,4.2)
\psecurve(-6,-1)(-4,1)(-6,3)(0,3.5)(6,3)(4,1)(6,-1)(3.5,-2.2)(1.2,0)(3.5,2.2)(6,1)(4,-1)(6,-3)(0,-3.5)(-6,-3)(-4,-1)(-6,1)(-3.5,2.2)(-1.2,0)(-3.5,-2.2)(-6,-1)(-4,1)(-6,3)
\psecurve(0,2.2)(-3,0)(0,-2.2)(3,0)(0,2.2)(-3,0)(0,-2.2)
\psdots[dotsize=10pt,linecolor=white](-5,2)(-5,0)(-5,-2)(-2,1.5)(-2,-1.5)(2,1.5)(2,-1.5)(5,2)(5,0)(5,-2)(0,2.85)(0,4.8)
\psecurve(0,-2.2)(-3,0)(0,2.2)(3,0)
\psdot[dotsize=10pt,linecolor=white](-2,1.5)
\psecurve(0,-3.5)(-6,-3)(-4,-1)(-6,1)
\psecurve(-3.5,-2.2)(-6,-1)(-4,1)(-6,3)
\psecurve(-4,-1)(-6,1)(-3.5,2.2)(-1.2,0)(-3.5,-2.2)
\psecurve(0,2.2)(3,0)(0,-2.2)(-3,0)(0,2.2)
\psecurve(3.5,-2.2)(1.2,0)(3.5,2.2)(6,1)(4,-1)
\psecurve(6,3)(4,1)(6,-1)(3.5,-2.2)
\psecurve(6,1)(4,-1)(6,-3)(0,-3.5)
\pscircle[linestyle=dotted,linecolor=black](3.75,0){3.5}
\psline[linecolor=black]{<-}(1,2)(1.2,1.94)
\psline[linecolor=black]{<-}(5.95,2.7)(5.805,2.5)
\end{pspicture}
\caption{The $(3,-2,2,-3)$-pretzel link\\ a.\,k.\,a. Kinoshita-Terasaka link}\label{fig:BigradedMutInv1}
\end{subfigure}
\quad
\begin{subfigure}[b]{0.46\textwidth}\centering
\begin{pspicture}[showgrid=false](-6.2,-4.3)(7.2,4.2)
\psecurve(-6,-1)(-4,1)(-6,3)(0,3.5)(5.8,3)(4,0)(5.8,-3)(0,-3.5)(-6,-3)(-4,-1)(-6,1)(-3.5,2.2)(-1.2,0)(-3.5,-2.2)(-6,-1)(-4,1)(-6,3)
\psecurve(-3,0)(0,-2.2)(3,-1)(1,1)(3.5,2.2)(5.8,0)(3.5,-2.2)(1,-1)(3,1)(0,2.2)(-3,0)(0,-2.2)(3,-1)
\psdots[dotsize=10pt,linecolor=white](-5,2)(-5,0)(-5,-2)(-2,1.5)(-2,-1.5)(5,1.5)(5,-1.5)(2,2)(2,0)(2,-2)(0,2.85)(0,4.8)
\psecurve(1,-1)(3,1)(0,2.2)(-3.0)
\psecurve(0,-3.5)(-6,-3)(-4,-1)(-6,1)
\psecurve(-3.5,-2.2)(-6,-1)(-4,1)(-6,3)
\psecurve(-4,-1)(-6,1)(-3.5,2.2)(-1.2,0)(-3.5,-2.2)
\psecurve(0,2.2)(-3,0)(0,-2.2)(3,-1)
\psecurve(0,3.5)(5.8,3)(4,0)(5.8,-3)
\psecurve(3.5,2.2)(5.8,0)(3.5,-2.2)(1,-1)(3,1)
\psecurve(0,-2.2)(3,-1)(1,1)(3.5,2.2)
\pscircle[linestyle=dotted,linecolor=black](3.5,0){3.5}
\psline[linecolor=black]{<-}(0.8,2.22)(1,2.2)
\psline[linecolor=black]{<-}(5.85,2.5)(5.755,2.3)
\end{pspicture}
\caption{The $(3,-2,-3,2)$-pretzel link\\ a.\,k.\,a. Conway link}\label{fig:BigradedMutInv2}
\end{subfigure}
\caption{An example for bigraded mutation invariance of $\HFL$}\label{fig:BigradedMutInv}
\end{figure}
\end{example}

%% file: sections/3_HFTd.tex

\section{\texorpdfstring{A glueing structure on $\HFT$ for 4-ended tangles}{A glueing structure on HFT for 4-ended tangles}}\label{sec:CFTd}

\begin{definition}
Let $T$ be a 4-ended tangle. A \textbf{peculiar Heegaard diagram} for $T$ is obtained from a tangle Heegaard diagram for $T$ by a local modification around the punctures, as illustrated in figure~\ref{fig:HD4ended}: We collapse the four boundary components of $\Sigma$ which meet the $\alpha$-arcs, thereby joining the four $\alpha$-arcs to a single $\alpha$-circle $\red S^1$. Then we add a marked point for each tangle end on either side of $\red S^1$, $p_i$ on the front and $q_i$ on the back, and connect these two points by an arc which intersects $\red S^1$ exactly once and no other curve. We also contract all other boundary components to points $z_j$, so we get a closed Heegaard surface. We call the union of the $p_i$, $q_i$ and $z_j$ \textbf{basepoints} of the Heegaard diagram.
Again, we need to restrict ourselves to \textbf{admissible} diagrams, i.\,e.\ diagrams whose non-zero periodic domains avoiding all basepoints have both negative and positive multiplicities. 
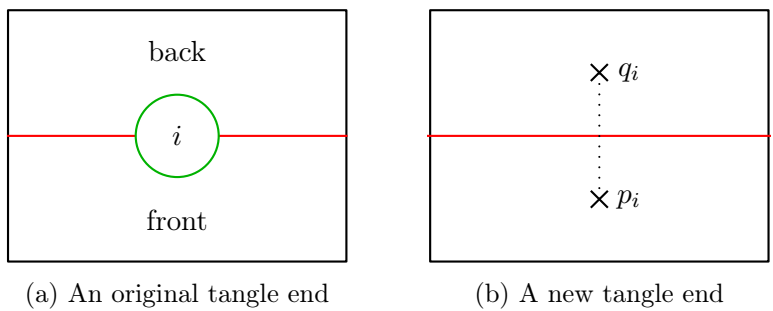
\begin{figure}[ht]
\psset{unit=0.07}
\centering
{\psset{unit=0.8}
\begin{subfigure}{0.33\textwidth}\centering
\begin{pspicture}(-40,-30)(40,30)
\psframe(-40,-30)(40,30)
\psline[linecolor=red](40,0)(-40,0)
\pscircle[fillstyle=solid,fillcolor=white,linecolor=darkgreen](0,0){10}
\rput(0,0){$i$}
\rput(0,20){back}
\rput(0,-20){front}
\end{pspicture}
\caption{An original tangle end}\label{fig:HD4endedoriginal}
\end{subfigure}
\quad
\begin{subfigure}{0.33\textwidth}\centering
\begin{pspicture}(-40,-30)(40,30)
\psframe(-40,-30)(40,30)
\psline[linecolor=red](-60,0)(60,0)
\psline(-2,17)(2,13)
\psline(-2,13)(2,17)
\psline(-2,-17)(2,-13)
\psline(-2,-13)(2,-17)
\psline[linestyle=dotted](0,15)(0,-15)
\rput(7,15){$q_i$}
\rput(7,-15){$p_i$}
\end{pspicture}
\caption{A new tangle end}\label{fig:HD4endednew}
\end{subfigure}
}
\caption{The lazy way to count holomorphic curves near the boundary: Peculiar Heegaard diagrams are obtained from those for tangles by local changes near the boundary of the Heegaard surface.}\label{fig:HD4ended}
\end{figure}
\end{definition}
\begin{Remark}\label{rem:PeculiarHDmoves}
It is obvious that we can go from a peculiar Heegaard diagram back to an ordinary tangle Heegaard diagram. The only reason for introducing peculiar Heegaard diagrams is to avoid any bordered Heegaard Floer theory, so the proof of invariance of the algebraic structures we are about to define is a minor adaptation of the one for link Floer homology. In particular, note that the number of $\alpha$-circles and $\beta$-circles in a peculiar Heegaard diagram is the same.\\
Obviously, the Heegaard moves from lemma~\sref{lem:HeegaardMoves} are equivalent to the following moves for peculiar Heegaard diagrams:
\begin{itemize}
\item isotopies of the $\alpha$- and $\beta$-curves away from the marked points and the arcs connecting these,
\item handleslides of $\B$-curves over $\B$-curves and handleslides of $\A$-curves over $\A$-curves other than $\red S^1$, and
\item stabilisation.
\end{itemize}
\end{Remark}
\begin{Remark}
The attribute ``peculiar'' should be considered as a homophone of ``$p$-$q$-liar'', a reference to the labels $p_i$ and $q_i$ we have chosen for the marked points, which is the notation used in \cite{Abouzaid} in the context of the wrapped Fukaya category of the 4-punctured sphere, see section~\ref{sec:LoopsAreLoops}.
\end{Remark}
In the following, we add differentials to $\CFT$ that count the number of holomorphic curves that have possibly non-zero multiplicities at the $p_i$ or $q_i$. Unfortunately, the resulting complex does not satisfy the relation $\partial^2=0$ any more. However, it satisfies a slightly modified $\partial^2$-relation which enables us to promote $\CFT$ to a more sophisticated homological invariant which we call a \emph{curved} type D module, see example~\ref{exa:HighBrowDefTypeDoverI}. Let us recall this definition here in slightly more down-to-earth terms.

\begin{definition}\label{def:curvedTypeDStructure}
Let $I$ be a ring of idempotents and $A$ a $\mathbb{Z}$-graded algebra over $I$. Also fix a central element $a_c\in Z(A)$ of degree $-2$. A \textbf{curved type D structure} over $A$ is a $\mathbb{Z}$-graded $I$-module $M$ together with an $I$-module homomorphism $\delta:~M\rightarrow A\otimes_I M$ of degree $-1$ satisfying 
$$(\mu\otimes 1_M)\circ(1_A\otimes \delta)\circ\delta=a_c\otimes 1_M,$$
where $\mu$ denotes composition in $A$. We call $a_c$ the \textbf{curvature} of $M$. A morphism between two curved type D structures $(M,\delta_M)$ and $(N,\delta_N)$ is an $I$-module homomorphism \linebreak$M\rightarrow A\otimes_I N$. For two such morphisms $f$ and $g$, their composition is defined as 
$$(g\circ f)=(\mu_2\otimes 1_M)\circ(1_A\otimes g)\circ f.$$
We endow the space of morphisms $\Mor(M,N)$ with a differential~$D$ defined by
$$D(f)= \delta_N\circ f+f\circ\delta_M.$$
Then indeed $D^2=0$, since we have chosen $a_c$ to be central. This gives us an enriched category over $\Com$, the category of ordinary chain complexes over $\mathbb{F}_2$. The underlying ordinary category is obtained by restricting the morphism spaces to degree 0 elements in the kernel of $D$, giving us the usual notions of chain homotopy and homotopy equivalence, see definition~\ref{def:UnderlyingOrdinaryCat} and example~\ref{exa:UnderlyingOrdCat}.
\end{definition}
\begin{Remark}\label{rem:MatFakAsCurvedComplexes}
It is interesting to compare curved type D structures to matrix factorisations as studied by Khovanov-Rozansky \cite{KhRoz}. Given an algebra~$A$ over some field $k$, a matrix factorisation of a potential $w\in A$ consists of two free $A$-modules $M_0$ and $M_1$ with two maps 
\begin{equation}\label{eqn:matrixfac}
\begin{tikzcd}[row sep=0.5cm, column sep=1cm]
M_0
\arrow[bend left=10]{r}[inner sep=1pt]{d_0}
&
M_1
\arrow[bend left=10]{l}[inner sep=1pt]{d_1}
\end{tikzcd}
\end{equation}
such that $d_1d_0=w.\id_{M_0}$ and $d_0d_1=w.\id_{M_1}$. If $\overline{M_0}$ and $\overline{M_1}$ denote the $k$-vector spaces generated by an $A$-basis of $M_0$ and $M_1$, respectively, we can regard $d_0$ and $d_1$ as maps 
$$\overline{d_0}:\overline{M_0}\rightarrow A\otimes_I\overline{M_1}\quad\text{ and }\quad\overline{d_1}:\overline{M_1}\rightarrow A\otimes_I\overline{M_0}.$$ 
Then $(\overline{M_0}\oplus \overline{M_1},\overline{d_0}+\overline{d_1})$ defines a curved type D structure over the $k$-algebra $A$. In general, we cannot go in the other direction. For example curved complexes associated to manifolds with torus boundary do, in general, not admit a splitting of the form~\eqref{eqn:matrixfac}. This is for the simple reason that the total number of generators can be odd, see for example~\cite[p.\,9 fig.\,6]{HRW}. However, for peculiar modules, such splittings exist, see remark~\ref{rem:MatFakForPecMod}.
\end{Remark}
\begin{definition}\label{def:Ad}
Let $\Ad$ be the path algebra of the quiver
\vspace{-9pt}
$$
\begin{tikzcd}[row sep=0.5cm, column sep=0.5cm]
&
\overset{4}{\bullet}
\arrow[bend left=10,leftarrow]{dr}[inner sep=1pt]{q_4}
\arrow[bend right=10,swap]{dr}[inner sep=1pt]{p_4}
&
\\
\!\!\!\raisebox{7pt}{$\underset{1}{~}$}\,\bullet
\arrow[bend left=10,leftarrow]{ur}[inner sep=1pt]{q_1}
\arrow[bend right=10,swap]{ur}[inner sep=1pt]{p_1}
&
&
\bullet\,\raisebox{7pt}{$\underset{3}{~}$}\!\!\!
\arrow[bend left=10,leftarrow]{dl}[inner sep=1pt]{q_3}
\arrow[bend right=10,swap]{dl}[inner sep=1pt]{p_3}
\\
&
\underset{2}{\bullet}
\arrow[bend left=10,leftarrow]{ul}[inner sep=1pt]{q_2}
\arrow[bend right=10,swap]{ul}[inner sep=1pt]{p_2}
\end{tikzcd}
\vspace{-9pt}
$$
with relations $p_iq_i=0=q_ip_i$ and $\Id$ the corresponding ring of idempotents. For an illustration, see figure~\ref{fig:nutshell}. We call $\Ad$ the \textbf{peculiar algebra}. \\
More explicitly, $\Ad$ can be described as follows. For $i=1,2,3,4$, let $I_i=\mathbb{F}_2[\iota_i]/(\iota_i^2=\iota_i)$ and $\mathcal{I}^\partial=I_1\times I_2\times I_3\times I_4$ the ring of idempotents, with one idempotent for each of the four sites $a$, $b$, $c$ and $d$ of a $4$-ended tangle. 
Let $\Ad$ be the following algebra: As an additive group, it is generated by closed oriented intervals on $\mathbb{R}$ with integer boundaries modulo overall shifts of multiples of 4, i.e.
$$[t_0,t_1]\sim[t_0+4n,t_1+4n]\quad \text{ for all } t_0, t_1,n\in \mathbb{Z}.$$
Given two algebra elements $a=[t_0,t_1]$ and $a'=[t'_0,t'_1]$, we define
\begin{equation*}
a\cdot a'=\begin{cases}
[t_0,t_1+t'_1-t'_0] & \text{if } (t_1\equiv t'_0 \mod 4) \quad\text{and}\quad(t_1-t_0)(t'_1-t'_0)\geq0\\
0 & \text{otherwise.}
\end{cases}
\end{equation*}
The second condition in the first case means that the composition of two oppositely oriented intervals is always zero. Furthermore, we define a bimodule action of $\mathcal{I}^\partial$ on $\Ad$ by
\begin{equation*}
\iota_i[t_0,t_1]\iota_j=\begin{cases}
[t_0,t_1] & \text{if } (t_0\equiv i \mod 4) \quad\text{and}\quad (t_1\equiv j \mod 4)\\
0 & \text{otherwise.}
\end{cases}
\end{equation*}
Thus, $[t,t]=\iota_{\overline{t}}$ for $t\equiv \overline{t}\in\{1,2,3,4\}\mod 4$.
For convenience, we usually write the elements $[t_0,t_1]$ with $t_0\neq t_1$ as $q_{(t_0+1)\cdots (t_1-1) t_1}$ or $p_{t_0 (t_0-1)\cdots (t_1+1)}$, depending on the orientation of the interval, where we take the indices modulo 4 with an offset of 1. For example, we write \linebreak[3] $q_{41}$ for $[3,5]$ and $p_4$ for $[0,-1]$. Furthermore, to simplify notation, we set 
$$
p=p_1+p_2+p_3+p_4
\quad\text{ and }\quad
q=q_1+q_2+q_3+q_4,
$$
so we can write for example $p^4=p_4p_3p_2p_1+p_3p_2p_1p_4+p_2p_1p_4p_3+p_1p_4p_3p_2$. 
\begin{figure}[t]
\centering
\psset{unit=1.7}
\begin{pspicture}(-1.52,-1.52)(1.52,1.52)

\psline[linecolor=red](-1,-1)(-1,1)(1,1)(1,-1)(-1,-1)

\rput(-1,0){\fcolorbox{white}{white}{\red $a$}}
\rput(0,-1){\fcolorbox{white}{white}{\red $b$}}
\rput(1,0){\fcolorbox{white}{white}{\red $c$}}
\rput(0,1){\fcolorbox{white}{white}{\red $d$}}

\rput(-1.25,0){$\iota_1$}
\rput(0,-1.25){$\iota_2$}
\rput(1.25,0){$\iota_3$}
\rput(0,1.25){$\iota_4$}

\pscircle[fillstyle=solid, fillcolor=white,linecolor=darkgreen](1,1){0.25}
\pscircle[fillstyle=solid, fillcolor=white,linecolor=darkgreen](1,-1){0.25}
\pscircle[fillstyle=solid, fillcolor=white,linecolor=darkgreen](-1,1){0.25}
\pscircle[fillstyle=solid, fillcolor=white,linecolor=darkgreen](-1,-1){0.25}

\rput[c](-1,1){1}
\rput[c](-1,-1){2}
\rput[c](1,-1){3}
\rput[c](1,1){4}

\rput(-1,1){
\psarcn{<-}(0,0){0.5}{-5}{-85}
\psarcn{<-}(0,0){0.5}{-95}{5}
\pscircle*[linecolor=white](0.5;-45){8pt}
\rput(0.5;-45){$p_1$}
\pscircle*[linecolor=white](0.5;135){8pt}
\rput(0.5;135){$q_1$}
}

\rput(-1,-1){
\psarcn{<-}(0,0){0.5}{85}{5}
\psarcn{<-}(0,0){0.5}{-5}{95}
\pscircle*[linecolor=white](0.5;45){8pt}
\rput(0.5;45){$p_2$}
\pscircle*[linecolor=white](0.5;-135){8pt}
\rput(0.5;-135){$q_2$}
}

\rput(1,-1){
\psarcn{<-}(0,0){0.5}{175}{95}
\psarcn{<-}(0,0){0.5}{85}{-175}
\pscircle*[linecolor=white](0.5;135){8pt}
\rput(0.5;135){$p_3$}
\pscircle*[linecolor=white](0.5;-45){8pt}
\rput(0.5;-45){$q_3$}
}

\rput(1,1){
\psarcn{<-}(0,0){0.5}{-95}{-175}
\psarcn{<-}(0,0){0.5}{175}{-85}
\pscircle*[linecolor=white](0.5;-135){8pt}
\rput(0.5;-135){$p_4$}
\pscircle*[linecolor=white](0.5;45){8pt}
\rput(0.5;45){$q_4$}
}

\end{pspicture}
\caption{An illustration of the algebra $\Ad$. An algebra element can ``walk'' along the front (the $p_i$s) or back (the $q_i$s) as far as it likes, but not around a puncture (the $p_i$s \textit{and} $q_i$s).}\label{fig:nutshell}
\end{figure}
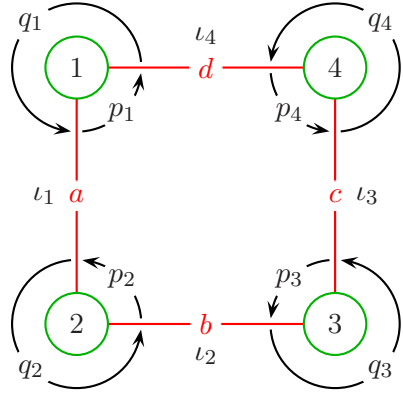
\end{definition}
\begin{definition}\label{def:AlexGradingOnAd}
We can define a \textbf{$\boldsymbol{\delta}$-grading} on $\Ad$ given by half the absolute length of the interval:
$$\delta([t_0,t_1]):=\tfrac{1}{2}\vert t_1-t_0\vert.$$
We can also define an \textbf{Alexander grading} $A$ on $\mathcal{A}$. However, some additional choices are necessary for this. An oriented 4-ended tangle induces an orientation (i.\,e.\ labelling by ``in''/``out'') and colouring of the four punctures in figure~\ref{fig:nutshell}. Note that the numbers of labels of the same colour labelled ``in'' and ``out'' agree. Then for any such choice, we obtain an Alexander grading on $\mathcal{A}$ as follows: Let $\vert p_i\vert=\vert q_i\vert= a^{\pm 1}$, where $a$ is the colour of the $i^\text{th}$ puncture, and choose ``$+$'' iff the $i^\text{th}$ puncture is labelled ``in''. Then we can extend this grading to the whole algebra such that $\vert a\cdot a'\vert=\vert a\vert \cdot\vert a'\vert$ unless $a\cdot a'=0$.\\
We sometimes denote the Alexander grading on generators by a superscript list of integers (or half-integers), like $a^\Red{+1}$ for the single-variate or $a^{(\frac{3}{2},-\frac{1}{2})}$ for the multivariate Alexander grading. Again, we can define a reduced Alexander grading $A^{r}$ by identifying both colours and a \textbf{homological grading} $h$ as the difference 
$$\tfrac{1}{2}A^{r}-\delta.$$
\end{definition}
\begin{definition}
	Let $\pqMod$ be the category of curved complexes over $\Ad$ with curvature $a_c:=p^4+q^4$. We call the objects of this category \textbf{peculiar modules}.
\end{definition}
\begin{Remark}\label{rem:MatFakForPecMod}
	Any homologically graded peculiar module $(M, \partial)$ admits a splitting like a matrix factorisation (see remark~\ref{rem:MatFakAsCurvedComplexes}), i.\,e.\ a splitting of $M$ into two summands $M_0$ and $M_1$ such that $\partial\vert_{M_0\rightarrow M_0}$ and $\partial\vert_{M_1\rightarrow M_1}$ vanish. To see this, let us write $M=\bigoplus \iota_i.M$ and pick a basis for each direct summand, so we get a basis for $M$. Consider two sequences of arrows in $\partial$ between two elements $x$ and $y$ of this basis. Suppose the labels of the arrows are given by basic algebra elements $a_1,\dots,a_n$ and $b_1,\dots,b_m$, respectively. 
	Since the homological grading decreases along each arrow by 1, we have
	$$
	h(x)-n=\sum_{i=1}^n h(a_i)+h(y)\quad\text{and}\quad h(x)-m=\sum_{j=1}^m h(b_j)+h(y),
	$$
	so
	\begin{equation}\label{eqn:MatFakForPecMod}
	m-n=\sum_{i=1}^n h(a_i)-\sum_{j=1}^m h(b_j).
	\end{equation}
	We want to show that this difference is even. The crucial observation is that the homological grading difference between any two algebra elements in $\iota_i.\Ad.\iota_j$ for fixed $i,j\in\{1,2,3,4\}$ is an even integer. Thus, in the two sequences $(a_1,\dots,a_n)$ and $(b_1,\dots,b_m)$, we can replace any algebra element which is a product of $q_i$s by a product of $p_i$s without changing the value of \eqref{eqn:MatFakForPecMod} modulo 2. But since the homological grading is compatible with the algebra multiplication, \eqref{eqn:MatFakForPecMod} is equal modulo 2 to 
	$$h(a_1\cdots a_n)-h(b_1\cdots b_m).$$
	Now we can use the same argument again to see that this difference is even, hence so is $m-n$. So we can split each connected component of the graph of $(M,\partial)$ by picking a base vertex and taking the union of all vertices of even and respectively odd distance to the base vertex.
\end{Remark}
\begin{definition}
Following the notation from chapter~\ref{chapter:categorification}, in particular definition~\sref{def:HDtoCFTbasics}, let $\phi\in\pi_2(\x,\y)$ be a domain between two generators $\x$ and $\y$ in $\mathbb{T}$. Let $n_{x}(\phi)$ denote the multiplicity of $\phi$ at a basepoint $x=p_i,q_i,z_j$. Let $\pi_2^+(\x,\y)$ denote the subset of domains in $\pi_2(\x,\y)$ for which $n_{x}(\phi)\geq0$ for $x=p_i, q_i$ and $n_{z_j}(\phi)=0$ for all $j$. Furthermore, set 
$$
n_p(\phi)=\sum_{i=1,2,3,4}n_{p_i}(\phi)
\quad\text{ and }\quad
n_q(\phi)=\sum_{i=1,2,3,4}n_{q_i}(\phi).
$$
\end{definition} 
\begin{definition}\label{defthm:CFTd}
Given a 4-ended tangle $T$ and an (admissible) peculiar Heegaard diagram for $T$, let $\CFTd(T)$ be the bigraded $\mathbb{F}_2$-module freely generated by elements in $\mathbb{T}(T)$, just as for $\CFT$, see definition~\sref{def:HFTiswelldefandinvariant}. We can turn it into a left $\mathcal{I}^\partial$-module as follows:
\begin{equation*}
\iota_i.\x =\begin{cases}
\x & \text{if $s(\x)=i$}\\
0 & \text{otherwise}
\end{cases}
\quad\text{for $\x\in \mathbb{T}$.}
\end{equation*}
We define an endomorphism $\partial$ on $\CFTd(T)$ by 
\begin{equation}\label{eqn:differentialOnCFTd}
\partial \x=\sum_{y\in\mathbb{T}}\sum_{\substack{\phi\in\pi_2^+(\x,\y)\\ \mu(\phi)=1}}\#\widehat{\mathcal{M}}(\phi) \cdot \iota_{s(\x)}.p^{n_p(\phi)}q^{n_q(\phi)}.\iota_{s(\y)}\otimes_{\Id} \y.
\end{equation}
As usual, $\#\widehat{\mathcal{M}}(\phi)$ denotes the modulo 2 count of holomorphic curves in the moduli space $\widehat{\mathcal{M}}(\phi)$ associated with~$\phi$. Note that we choose the Alexander grading on $\Ad$ that is induced by $T$. We call $(\CFTd(T),\partial)$ the \textbf{peculiar module of~$T$}.
\end{definition}
\begin{theorem}
	\(\CFTd(T)\) is indeed a well-defined peculiar module. Furthermore, its bigraded chain homotopy type is an invariant of the tangle~\(T\).
\end{theorem}

\begin{lemma}\label{lem:CFTdGradings}
\(\partial\) increases the \(\delta\)-grading by 1 and preserves the Alexander grading. (As~usual, the grading on a tensor product is given by the sum of the gradings of the tensor factors.)
\end{lemma}
\begin{proof}
Recall that the relative $\delta$-grading on generators is defined via the Maslov index~$\mu$ of connecting domains, see definition~\sref{def:homgrading}. When using peculiar Heegaard diagrams, we replace punctures by basepoints. Of course, this does not change the Maslov index formula. However, it \emph{does} change the relationship between the Maslov index and the $\delta$-grading. If $\phi\in\pi_1(\x,\y)$ is a domain in a peculiar Heegaard diagram, let $\varphi$ be the corresponding domain in the usual tangle Heegaard diagram from section~\ref{sec:HDsForTangles}. Then 
$$\delta(\y)-\delta(\x)=\mu(\varphi)=\mu(\phi)-\tfrac{1}{2}\#\{p_i,q_i\}-\#\{z_j\}.$$
The holomorphic discs contributing to the differential  in (\ref{eqn:differentialOnCFTd}) have Maslov index 1 and have multiplicity 0 at the basepoints $z_j$. So, the first claim follows from the definition of the $\delta$-grading on $\Ad$. The second statement follows directly from the definitions of the Alexander gradings of generators and algebra elements. 
\end{proof}
\begin{lemma}
For each \(\x\in\CFTd\), the sum on the right-hand side of  \eqref{eqn:differentialOnCFTd} is finite.
\end{lemma}
\begin{proof}
The proof is essentially the same as in link Floer homology, see \cite[lemma~4.2]{OSHFL}. Since $\CFTd$ is finitely generated, it is sufficient to show that the coefficient of each $\y\in\CFTd$ is a finite sum. Note that by the previous lemma, the difference of the $\delta$-gradings of $\x$ and $\y$ determines the $\delta$-grading of $a(\phi)$. Thus, there are at most two choices for the coefficients $a(\phi)$. So let us also fix the multiplicities of $\phi$ at the basepoints. We can now argue as in the proof of \cite[lemma~4.13]{OSHF3mfds}, using admissibility of the underlying Heegaard diagram.
\end{proof}
In the following, we need two analytical facts from \cite{OSHFL}.
\begin{fact}\label{fact:OSHFL_lemma_5_4}
\cite[lemma~5.4]{OSHFL} Given a homology class \(\phi\) of \(\beta\)-injective boundary degenerations, write \(\phi\) as a linear combination of connected components of \(\Sigma\smallsetminus \B\). Then its Maslov index \(\mu(\phi)\) is equal to twice the sum of the coefficients. The same holds for \(\A\)-injective boundary degenerations.\hfill \qedsymbol
\end{fact}
\begin{fact}\label{fact:OSHFL_lemma_5_5} \cite[theorem~5.5]{OSHFL} Given a surface \(\Sigma\) of genus \(g\), equipped with a set \(\A\) of \((g+1)\) attaching circles and a pseudo-holomorphic curve \(\psi\) of Maslov index \(\mu(\psi)=2\) and with non-negative multiplicities, then \(\psi\) is equal to one of the two connected components of \(\Sigma\smallsetminus \A\) and the number of pseudo-holomorphic boundary degenerations in the homology class of \(\psi\) is odd.\hfill\qedsymbol
\end{fact}
\begin{proof}[Proof of theorem~\ref{defthm:CFTd}]
Checking the $\partial^2$-identity is analogous to the link case; we can follow \cite[proof of lemma~4.3]{OSHFL} and count ends of moduli spaces of Maslov index 2 curves. We fix two generators $\x$ and $\z$ and consider the disjoint union of moduli spaces $\widehat{\mathcal{M}}(\phi)$, where $\phi$ varies over those curves in $\pi_2(\x,\z)$ with $\mu(\phi)=2$ and $a(\phi)=a$ for some fixed $a\in \Ad$. (In particular, this fixes the multiplicities of $\phi$ at the $p_i$ and $q_i$.)
If there are no boundary degenerations, there is an even number of ends, so the $a\otimes \z$-component of $\partial^2 \x$ vanishes. If there are boundary degenerations, then by fact~\ref{fact:OSHFL_lemma_5_4} above, they contribute at least 2 to the Maslov index, so the remaining curve has to be constant, hence $\x=\z$. By fact~\ref{fact:OSHFL_lemma_5_5}, we get a boundary degeneration for each component of $\Sigma\smallsetminus\A$ and $\Sigma\smallsetminus\B$. But since $p_iq_j=0$ for all $i,j$, $\beta$-injective boundary degenerations do not appear. $\alpha$-injective boundary degenerations do appear and they contribute exactly the terms $p^4+q^4$. All other ends appear in pairs again, so their contributions cancel. \\
It remains to show that the peculiar module is an invariant of the tangle $T$. Again, we adapt the arguments for link Floer homology, more precisely the proofs of \cite[theorems~4.4 and~4.7]{OSHFL}, which rely on the general machinery developed in \cite{OSHF3mfds}. We need to show  that the three Heegaard moves from remark~\ref{rem:PeculiarHDmoves} do not change the homotopy type of the peculiar module of $T$. 
\begin{itemize}
\item isotopies: Firstly, we need to study isotopies that do not change transversality of $\A$ and $\B$. In the case for closed manifolds, this is done in \cite[theorem~6.1]{OSHF3mfds}. To adapt the proof to our setting, we only need to replace the basepoint~$z$ by the marked points $p_i$ and $q_i$, and count multiplicities of holomorphic discs at those points in the chain maps. Boundary degenerations do not appear because of fact~\ref{fact:OSHFL_lemma_5_4}. In the case of closed components in $T$, admissibility ensures that all sums are finite. \\
Secondly, we need to check that exact Hamiltonian isotopies do not change the homotopy type of $\CFTd(T)$. However, we can apply the same arguments as above to adapt the proof for closed manifolds \cite[theorem~7.3]{OSHF3mfds} to ours.
\item handleslides: The main ingredient for proving handleslide invariance for Heegaard Floer homology of closed manifolds and links therein are holomorphic triangles. Only small modifications are necessary to adapt the proof from \cite[section~6.3]{OSHFL}, based on the arguments in \cite[section~9]{OSHF3mfds} to our setting.\\
Let us consider a Heegaard diagram $(\Sigma,\A,\B)$ and perform a handleslide of one curve in $\B$ over another, away from all basepoints. Let $\GG$ denote the union of the new $\beta$-curve with a small Hamiltonian translate of all unchanged $\beta$-curves. Also let $\D$ denote a small Hamiltonian translate of all $\beta$-curves. Then we consider the maps 
\begin{equation*}
F_{\A\B\GG}:~\CFTd(\A,\B)\otimes \CFTd(\B,\GG)\rightarrow \CFTd(\A,\GG),
\end{equation*}
defined by using a holomorphic triangle count, recording multiplicities of the triangles at the basepoints in the same way as in $\CFTd$. Note that such triangle maps preserve both Alexander and $\delta$-gradings. Since the basepoint $p_i$ is in the same region as $q_i$ for all $i$, $\CFTd(\B,\GG)=\HF(\B,\GG,\{p_i=q_i,z_j\})$. By the proof of \cite[proposition~6.10]{OSHFL}, the latter is equal to $\Lambda^\ast V$, where $V$ is some finite-dimensional vector space, with a distinguished element $\Theta_{\B\GG}$. Then $F_{\A\B\GG}$ induces a map
$$\Psi_{\A\B\GG}:~\CFTd(\A,\B)\rightarrow \CFTd(\A,\GG),$$
by sending a generator $x$ to $F_{\A\B\GG}(x\otimes\Theta_{\B\GG})$.\\
In the same way, we can define the maps $\Psi_{\A\GG\D}$ and $\Psi_{\A\B\D}$. Associativity of the triangle maps show that actually 
$$\Psi_{\A\GG\D}\circ\Psi_{\A\B\GG}=\Psi_{\A\B\D}.$$
Hence, it only remains to show that the map on the right is an isomorphism. For this we adapt the arguments from \cite[proof of proposition~9.8]{OSHF3mfds}. Let us consider generators in a fixed Alexander and $\delta$-grading and a fixed site. If this set is non-empty, let us fix a generator $\x_0$ in this set. Then, we can define a filtration as follows. Choose an area function on $\Sigma$ such that periodic domains avoiding basepoints have signed area~0. Then, we define $\mathcal{F}(\x):=-\text{Area}(\phi)$, where $\phi\in\pi_{2}^\partial(\x_0,\x)$. Note that for generators in the same site and Alexander grading, there is always such a domain $\phi$ by lemma~\sref{lem:pidxynonempty}. By the choice of our area function, $\mathcal{F}$ is independent of $\phi$. Now we can use the arguments from \cite[proof of proposition~9.8]{OSHF3mfds} to see that the restriction of $\Psi_{\A\B\D}$ to our fixed set of generators can be written as the sum of the identity map and a part which strictly lowers the filtration level. \\
We can now put all filtrations together to obtain a filtration on all generators. Indeed, note that in a fixed $\delta$-grading, there are only arrows between generators of the same Alexander grading and the same site. Furthermore, because the $\delta$-grading on $\Ad$ is non-negative, there are no arrows in $\Psi_{\A\B\D}$ along which the $\delta$-grading increases. Since there are only finitely many generators, we can argue inductively, starting at the lowest supported filtration level to see that $\Psi_{\A\B\D}$ is an isomorphism, as in \cite[lemma~9.10]{OSHF3mfds}.\\
Next, we look at handleslides of $\alpha$-curves. Let us consider a Heegaard diagram $(\Sigma,\A,\B)$ and perform a handleslide of one curve in $\A$ over another, away from all basepoints; as before let $\GG$ denote the union of the new $\alpha$-curve with a small Hamiltonian translate of all unchanged $\alpha$-curves. Also, let $\D$ denote a small Hamiltonian translate of all $\alpha$-curves. Then we consider the maps 
\begin{equation*}
F_{\GG\A\B}:~\CFTd(\GG,\A)\otimes \CFTd(\A,\B)\rightarrow \CFTd(\GG,\B),
\end{equation*}
as above. Since all $p_i$ lie in the same component of $\Sigma\smallsetminus(\A\cup\GG)$, and similarly the $q_i$,  $\CFTd(\GG,\A)/(p_i=0=q_i)=\HF(\Sigma,\GG,\A,\{p_1,q_1,z_j\})=\Lambda^\ast V$ as above. So we can restrict this chain map in the second component to the unique element in highest homological degree in $\Lambda^\ast V$, i.\,e.\ lowest in $\delta$-grading. Together with the algebra grading, this means that this restriction is still a chain map. Now we can argue as above.
\item stabilisation: For closed 3-manifolds, this is established in \cite[section~10]{OSHF3mfds}. The main ingredient is a glueing theorem for holomorphic discs which identifies the moduli spaces before and after the glueing. Thus, stabilisation not only leaves the generators unchanged, but also the structure maps, so the arguments carry over to our case without any alteration. \qedhere
\end{itemize}
\end{proof}

\begin{example}[rational tangles]\label{exa:CFTdRatTang}
Figure~\ref{fig:CFTdForSomeRatTangles} shows the peculiar modules of some very simple 4-ended tangles. As shown in example~\sref{exa:HDforonecrossing}, every rational tangle $T$ has a tangle Heegaard diagram with just a single $\beta$-curve. Thus, we only count bigons in the differential of the peculiar invariant, and only those that do not occupy both~$p_i$ and~$q_j$. By tightening the $\beta$-curve, we can assume that there are no honest differentials in $\CFTd(T)$, i.\,e.\ that every bigon covers some $p_i$ or $q_i$. Then $\CFTd(T)$ is a loop, and its representative in the 4-punctured sphere is exactly the $\beta$-curve in the Heegaard diagram; compare this to proposition~\sref{prop:BSDrattangle}. 
\end{example}

\begin{figure}
\centering
\psset{unit=0.15}
\begin{subfigure}[b]{0.26\textwidth}\centering
\psset{unit=7}
\begin{pspicture}(-1,-1)(1,1)
\psarc[linecolor=violet]{<-}(2,0){1.41432}{-225}{-135}
\psarc[linecolor=darkgreen]{->}(-2,0){1.41432}{-45}{45}
\end{pspicture}
\end{subfigure}
\quad
\begin{subfigure}[b]{0.32\textwidth}\centering
\psset{unit=7}
\begin{pspicture}(-1,-1)(1,1)
\psline[linecolor=violet]{->}(0.9,-0.9)(-0.9,0.9)
\pscircle*[linecolor=white](0,0){0.3}
\psline[linecolor=darkgreen]{->}(-0.9,-0.9)(0.9,0.9)
\end{pspicture}
\end{subfigure}
\quad
\begin{subfigure}[b]{0.34\textwidth}\centering
\psset{unit=7}
\begin{pspicture}(-1,-1)(1,1)
\psline[linecolor=violet]{->}(-0.9,-0.9)(0.9,0.9)
\pscircle*[linecolor=white](0,0){0.3}
\psline[linecolor=darkgreen]{->}(0.9,-0.9)(-0.9,0.9)
\end{pspicture}
\end{subfigure}
\\
\begin{subfigure}[b]{0.26\textwidth}\centering
\begin{pspicture}(-10.3,-10.3)(10.3,10.3)
\psrotate(0,0){90}{
\psecurve[linecolor=blue](7,0)(0,0)(-7,0)(-8.3,8.3)(0,10)(8.3,8.3)(7,0)(0,0)(-7,0)
}

\psarc[linecolor=red](0,0){7}{0}{360}

\rput(7;45){\pscircle*[linecolor=white]{1}\pscircle[linecolor=violet]{1}}
\rput(7;135){\pscircle*[linecolor=white]{1}\pscircle[linecolor=darkgreen]{1}}
\rput(7;225){\pscircle*[linecolor=white]{1}\pscircle[linecolor=darkgreen]{1}}
\rput(7;315){\pscircle*[linecolor=white]{1}\pscircle[linecolor=violet]{1}}

\psdot(0,7)
\psdot(0,-7)

\rput[bl](0,7.5){$d$}
\rput[tl](0,-7.5){$b$}
\end{pspicture}
\end{subfigure}
\quad
\begin{subfigure}[b]{0.32\textwidth}\centering
\begin{pspicture}(-10.3,-10.3)(10.3,10.3)
\psrotate(0,0){90}{

\psecurve[linecolor=blue](8,-0.2)(7,0)(2,2)(0,7)(-0.2,8)
\psecurve[linecolor=blue](-4,4)(0,7)(-8.3,8.3)(-7,0)(-4,4)
\psecurve[linecolor=blue](-8,0.2)(-7,0)(-2,-2)(0,-7)(0.2,-8)
\psecurve[linecolor=blue](4,-4)(0,-7)(8.3,-8.3)(7,0)(4,-4)

\psarc[linecolor=red](0,0){7}{0}{360}

\rput(7;45){\pscircle*[linecolor=white]{1}\pscircle[linecolor=violet]{1}}
\rput(7;135){\pscircle*[linecolor=white]{1}\pscircle[linecolor=darkgreen]{1}}
\rput(7;225){\pscircle*[linecolor=white]{1}\pscircle[linecolor=violet]{1}}
\rput(7;315){\pscircle*[linecolor=white]{1}\pscircle[linecolor=darkgreen]{1}}

}
\psdot(0,7)
\psdot(7,0)
\psdot(0,-7)
\psdot(-7,0)

\rput[br](0,7.5){$d$}
\rput[tl](7.5,0){$c$}
\rput[tl](0,-7.5){$b$}
\rput[br](-7.5,0){$a$}
\end{pspicture}
\end{subfigure}
\quad
\begin{subfigure}[b]{0.34\textwidth}\centering
\begin{pspicture}(-10.3,-10.3)(10.3,10.3)
\psecurve[linecolor=blue](8,-0.2)(7,0)(2,2)(0,7)(-0.2,8)
\psecurve[linecolor=blue](-4,4)(0,7)(-8.3,8.3)(-7,0)(-4,4)
\psecurve[linecolor=blue](-8,0.2)(-7,0)(-2,-2)(0,-7)(0.2,-8)
\psecurve[linecolor=blue](4,-4)(0,-7)(8.3,-8.3)(7,0)(4,-4)

\psarc[linecolor=red](0,0){7}{0}{360}

\rput(7;45){\pscircle*[linecolor=white]{1}\pscircle[linecolor=violet]{1}}
\rput(7;135){\pscircle*[linecolor=white]{1}\pscircle[linecolor=darkgreen]{1}}
\rput(7;225){\pscircle*[linecolor=white]{1}\pscircle[linecolor=violet]{1}}
\rput(7;315){\pscircle*[linecolor=white]{1}\pscircle[linecolor=darkgreen]{1}}

\psdot(0,7)
\psdot(7,0)
\psdot(0,-7)
\psdot(-7,0)

\rput[bl](0,7.5){$d$}
\rput[bl](7.5,0){$c$}
\rput[tr](0,-7.5){$b$}
\rput[tr](-7.5,0){$a$}
\end{pspicture}
\end{subfigure}
\\
\begin{subfigure}[b]{0.26\textwidth}\centering
$\begin{tikzcd}[row sep=1.4cm, column sep=0.5cm]
\delta^{0}b^{(\textcolor{darkgreen}{0},\textcolor{violet}{0})}
\arrow[leftarrow,bend left=12]{r}{p_{43}+q_{12}}
& 
\delta^{0}d^{(\textcolor{darkgreen}{0},\textcolor{violet}{0})}
\arrow[leftarrow,bend left=10]{l}{p_{21}+q_{34}}
\end{tikzcd}$
\end{subfigure}
\quad
\begin{subfigure}[b]{0.32\textwidth}\centering
$\begin{tikzcd}[row sep=1.4cm, column sep=0.5cm]
\delta^{0}a^{(\textcolor{darkgreen}{\frac{1}{2}},\textcolor{violet}{-\frac{1}{2}})}
\arrow[leftarrow,bend right=7,swap]{r}{p_{432}}
\arrow[leftarrow,bend left=12]{d}{q_{341}}
& 
\delta^{\frac{1}{2}}d^{(\textcolor{darkgreen}{\frac{1}{2}},\textcolor{violet}{\frac{1}{2}})}
\arrow[leftarrow,bend right=7,swap]{l}{p_1}
\arrow[leftarrow,bend left=12]{d}{q_4}
\\
\delta^{\frac{1}{2}}b^{(\textcolor{darkgreen}{-\frac{1}{2}},\textcolor{violet}{-\frac{1}{2}})}
\arrow[leftarrow,bend right=7,swap]{r}{p_3}
\arrow[leftarrow,bend left=12]{u}{q_2}
&
\delta^{0}c^{(\textcolor{darkgreen}{-\frac{1}{2}},\textcolor{violet}{\frac{1}{2}})}
\arrow[leftarrow,bend right=7,swap]{l}{p_{214}}
\arrow[leftarrow,bend left=12]{u}{q_{123}}
\end{tikzcd}$
\end{subfigure}
\quad
\begin{subfigure}[b]{0.34\textwidth}\centering
$\begin{tikzcd}[row sep=1.4cm, column sep=0.35cm]
\delta^{0}a^{(\textcolor{darkgreen}{\frac{1}{2}},\textcolor{violet}{-\frac{1}{2}})}
\arrow[leftarrow,bend left=5,pos=0.8]{r}{q_1}
\arrow[leftarrow,bend right=12,swap]{d}{p_2}
& 
\delta^{-\frac{1}{2}}d^{(\textcolor{darkgreen}{-\frac{1}{2}},\textcolor{violet}{-\frac{1}{2}})}
\arrow[leftarrow,bend left=5,pos=0.2]{l}{q_{234}}
\arrow[leftarrow,bend right=12,swap]{d}{p_{321}}
\\
\delta^{-\frac{1}{2}}b^{(\textcolor{darkgreen}{\frac{1}{2}},\textcolor{violet}{\frac{1}{2}})}
\arrow[leftarrow,bend left=5]{r}{q_{412}}
\arrow[leftarrow,bend right=12,swap]{u}{p_{143}}
&
\delta^{0}c^{(\textcolor{darkgreen}{-\frac{1}{2}},\textcolor{violet}{\frac{1}{2}})}
\arrow[leftarrow,bend left=7]{l}{q_3}
\arrow[leftarrow,bend right=12,swap]{u}{p_4}
\end{tikzcd}$
\end{subfigure}
\caption{Basic rational tangles, their Heegaard diagrams and peculiar modules. The superscripts of the generators specify the Alexander grading. }\label{fig:CFTdForSomeRatTangles}
\end{figure}
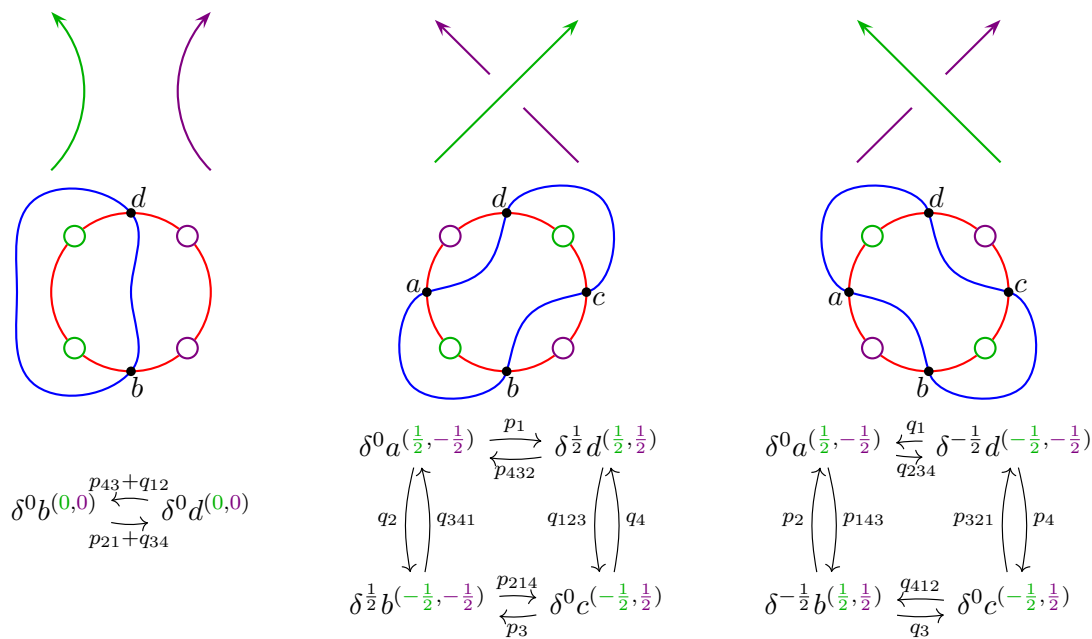

\begin{sidewaysfigure}[p]
\vspace*{435pt}
\centering
\begin{subfigure}[b]{0.3\textwidth}\centering
\psset{unit=0.8, linewidth=1.1pt}

\begin{pspicture}[showgrid=false](-5.2,-3.1)(3.2,3.1)
\psecurve[linecolor=violet](-2.5,1.5)(0,2)(0.75,1)(-0.75,-1)(0,-2)(0.97,-2.24)(2,-2)
\psecurve[linecolor=violet]{<-}(2,2)(0.97,2.24)(0,2)(-0.75,1)(0.75,-1)(0,-2)(-2.5,-1.5)(-3.25,0)(-2.5,1.5)(0,2)(0.75,1)
\psecurve[linecolor=darkgreen]{<-}(-6,1.5)(-3.3,1.85)(-2.5,1.5)(-1.85,0)(-2.5,-1.5)(-3.3,-1.85)(-6,-1.5)
\pscircle*[linecolor=white](-2.5,1.5){0.2}

\psecurve[linecolor=violet](0.75,-1)(0,-2)(-2.5,-1.5)(-3.25,0)(-2.5,1.5)

\pscircle*[linecolor=white](0,2){0.2}
\pscircle*[linecolor=white](0,0){0.2}
\pscircle*[linecolor=white](0,-2){0.2}

\psecurve[linecolor=violet](0.75,1)(-0.75,-1)(0,-2)(0.97,-2.24)(2,-2)
\psecurve[linecolor=violet](0,2)(-0.75,1)(0.75,-1)(0,-2)
\psecurve[linecolor=violet](-2.5,-1.5)(-3.25,0)(-2.5,1.5)(0,2)(0.75,1)(-0.75,-1)

\pscircle*[linecolor=white](-2.5,-1.5){0.2}
\psecurve[linecolor=darkgreen](-2.5,1.5)(-1.85,0)(-2.5,-1.5)(-3.3,-1.85)(-6,-1.5)

\psline[linecolor=violet]{<-}(-3.25,-0.1)(-3.25,0.1)
\pscircle[linestyle=dotted](-1,0){3.05}

\uput{0.2}[45](0.97,2.24){$\textcolor{violet}{q}$}
\uput{0.2}[-45](0.97,-2.24){$\textcolor{violet}{q}$}
\uput{0.2}[135](-3.3,1.85){$\textcolor{darkgreen}{p}$}
\uput{0.2}[-135](-3.3,-1.85){$\textcolor{darkgreen}{p}$}

\uput{2.5}[180](-1,0){$\textcolor{red}{a}$}
\uput{2.3}[-90](-1,0){$\textcolor{blue}{b}$}
\uput{2.1}[0](-1,0){$\textcolor{darkgreen}{c}$}
\uput{2.3}[90](-1,0){$\textcolor{gold}{d}$}

\end{pspicture}
\caption{The $(2,-3)$-pretzel tangle}\label{fig:HFTdmutationpretzeltangleT}
\end{subfigure}
\quad
\begin{subfigure}[b]{0.6\textwidth}\centering
\psset{unit=0.25}
\begin{pspicture}(-24,-11)(24,11)
\rput(-12,0){\psrotate(0,0){-90}{

\psarc[linecolor=red](0,0){8}{0}{360}

\SpecialCoor
\rput(8;45){\pscircle*[linecolor=white]{1}\pscircle[linecolor=violet]{1}}
\rput(8;135){\pscircle*[linecolor=white]{1}\pscircle[linecolor=violet]{1}}
\rput(8;225){\pscircle*[linecolor=white]{1}\pscircle[linecolor=darkgreen]{1}}
\rput(8;315){\pscircle*[linecolor=white]{1}\pscircle[linecolor=darkgreen]{1}}


\psecurve[linecolor=blue]%
(10;130)(10;140)(8;155)%
(8;-95)(11;-45)%
(10.5;45)(8;70)%
(8;-10)%
(10;-45)(8;-70)%
(8;120)(10;130)(10;140)(8;155)%

\psdot(8;155)%
\psdot(8;120)%
\psdot(8;70)%
\psdot(8;-10)%
\psdot(8;-70)%
\psdot(8;-95)%
}
\rput[b](8.7;65){$d$}
\rput[l](8;30){~$x_1$}
\rput[l](8;-20){~$x_2$}
\rput[b](7.3;-100){$b$}
\rput[l](8;-160){~$a_2$}
\rput[l](8;-185){~$a_1$}

\footnotesize
\rput[t](7.1;110){$p_1$}
\rput[t](9.8;110){$q_1$}
\rput[c](8.8;-123){$q_2$}
\rput[l](2;-90){$p_2$}
\rput[c](3.5;90){\texttt{4}}
\rput[c](7;-60){\texttt{2}}
\rput[c](9;-80){\texttt{1}}
\rput[c](9.5;42){\texttt{3}}
}

\rput(12,0){\psrotate(0,0){-90}{

\psarc[linecolor=red](0,0){8}{0}{360}

\SpecialCoor
\rput(8;45){\pscircle*[linecolor=white]{1}\pscircle[linecolor=violet]{1}}
\rput(8;135){\pscircle*[linecolor=white]{1}\pscircle[linecolor=violet]{1}}
\rput(8;225){\pscircle*[linecolor=white]{1}\pscircle[linecolor=violet]{1}}
\rput(8;315){\pscircle*[linecolor=white]{1}\pscircle[linecolor=violet]{1}}


\psecurve[linecolor=blue]%
(10;-130)(10;-140)(8;-155)%
(8;95)(11;45)%
(10.5;-45)(8;-80)%
(8;60)
(9.5;45)(8;20)
(8;-60)(9.5;-45)(10;45)
(8;75)%
(5.5;100)%
(8;-120)(10;-130)(10;-140)(8;-155)%

\psdot(8;-155)%
\psdot(8;-120)%
\psdot(8;-80)%
\psdot(8;-60)%
\psdot(8;20)%
\psdot(8;60)%
\psdot(8;75)%
\psdot(8;95)%
}
\rput[b](8.7;-245){$d'$}
\rput[r](8;-210){$y_1$~}
\rput[r](8;-170){$y_2$~}
\rput[b](7;-154){$y_3$}
\rput[b](7.3;-70){$b'$}
\rput[r](8;-29){$c_3$~}
\rput[r](8;-15){$c_2$~}
\rput[l](8;5){~$c_1$}

\footnotesize
\rput[t](7.5;70){$p_4$}
\rput[t](9.8;70){$q_4$}
\rput[t](3;-90){$p_3$}
\rput[l](9.5,-6){~$q_3$}
\rput[l](9.5,-7.5){~\texttt{1}}
\rput[l](9.5,-9){~\texttt{5}}
\psline[linestyle=dotted,dotsep=1pt](9.1;-45)(9.5,-6)
\psline[linestyle=dotted,dotsep=1pt](8.7;-80)(9.5,-7.5)
\psline[linestyle=dotted,dotsep=1pt](9.9;-80)(9.5,-9)
\rput[c](4;90){\texttt{4}}
\rput[c](1;90){\texttt{6}}
\rput[c](7;-120){\texttt{2}}
\rput[c](9.5;138){\texttt{3}}
}

\psline[linestyle=dashed](-6.25,5.65685424949236)(6.25,5.65685424949236)

\psline[linestyle=dashed](-6.25,-5.65685424949236)(6.25,-5.65685424949236)

\end{pspicture}
\caption{A Heegaard diagram for the tangle on the left}
\label{fig:HFTdmutationpretzeltangleHD}
\end{subfigure}
\\
\begin{subfigure}[b]{\textwidth}
$$
\begin{tikzcd}[column sep=0.99cm,row sep=2cm]   
%
%
\delta^{-2}\gold{x_2d'}
\arrow[leftarrow,swap,bend right=10]{r}
\arrow[leftarrow,swap,bend right=10]{d}{q_{234}(15)}&
%
\delta^{-\frac{3}{2}}\darkgreen{x_2c_1}
\arrow[leftarrow,swap,bend right=10]{l}{p_4} 
\arrow[leftarrow,swap,bend right=10,end anchor=170]{dr}{q_{23}(15)} &
%
\delta^{-1}{\blue by_1} 
\arrow[leftarrow,swap,bend right=10,end anchor=40]{d}{q_2} 
\arrow[leftarrow,swap,bend right=10,end anchor=140,dotted,pos=0.2]{d}{q_2(126)} 
\arrow[leftarrow,swap,bend right=10]{r}{p_3(26)} &
%
\delta^{-\frac{3}{2}}\darkgreen{x_2c_2} 
\arrow[leftarrow,swap,bend right=10,end anchor=170]{dr}{q_{23}(1)} 
\arrow[leftarrow,bend left,end anchor=80,dotted,pos=0.05]{dr}{q_{23}(1345)} 
\arrow[leftarrow,swap,bend right=10]{l}{p_{214}(4)} &
%
\delta^{-1}{\blue by_2} 
\arrow[leftarrow,swap,bend right=10,end anchor=40]{d}{q_2} 
\arrow[leftarrow,swap,bend right=10]{r}{p_3(2)} &
%
\delta^{-\frac{3}{2}}\darkgreen{x_2c_3} 
\arrow[leftarrow,swap,bend right=10]{l}{p_{214}(46)} 
\arrow[leftarrow,swap,bend right=5,pos=0.7]{r}{q_3} &
%
\delta^{-2}{\blue x_2b'}~~\delta^{-1}{\blue by_3} 
\arrow[leftarrow,swap,bend right=10,start anchor=-40]{d}{q_2}
\arrow[leftarrow,swap,bend right=5]{l}{}
\arrow[leftarrow,start anchor=-60,end anchor=-120,bend left]{lr}{1(2)}
\arrow[leftarrow,start anchor=140,end anchor=40,bend left]{lr}{p_{2143}(46)}
 \\
%
%
\delta^{-\frac{3}{2}}{\red a_1y_1}
\arrow[leftarrow,swap,bend right=10]{u}{q_1(3)}
\arrow[leftarrow,swap,bend right=10,end anchor=150]{d}{}
\arrow{dr}{p_{32}(2346)}
 &
 &
%
\delta^{-\frac{3}{2}}{\red a_1y_2}~~ \delta^{-\frac{3}{2}}{\red a_2y_1} 
\arrow[leftarrow,swap,bend right=10,start anchor=160,pos=0.2]{ul}{q_{41}(3)} 
\arrow[leftarrow,start anchor=-30,pos=0.85, end anchor=60,dotted]{d}{p_{432}(122466)} 
\arrow[leftarrow,swap,bend right=10,start anchor=-150]{d}{} 
\arrow[leftarrow,swap,bend right=10,start anchor=30]{u}{} 
\arrow[leftarrow,swap, bend right=10, start anchor=-20]{dr}{p_{32}(26)}
\arrow[leftarrow,swap,bend right=10,start anchor=50,dotted,pos=0.9]{ul}{q_{41}(1236)} 
&
\odot &
%
\delta^{-\frac{3}{2}}{\red a_1y_3}~~\delta^{-\frac{3}{2}}{\red a_2y_2} 
\arrow[leftarrow,swap,bend right=10,start anchor=160,pos=0.35]{ul}{q_{41}(35)} 
\arrow[leftarrow,swap,bend right=10,start anchor=-150]{d}{} 
\arrow[leftarrow,swap,bend right=10,start anchor=30]{u}{} 
\arrow[leftarrow,swap, bend right=10, start anchor=-20,pos=0.6]{dr}{p_{32}(2)} 
\arrow[leftarrow,bend right=10, start anchor=150, end anchor=-110,pos=0.8,dotted]{u}{q_{341}(133455)}
\arrow[leftarrow,swap,bend right=10, start anchor=-100, end anchor=180,pos=0.9,dotted]{dr}{p_{32}(2345)}
&
&
%
\delta^{-\frac{3}{2}}{\red a_2y_3} 
\arrow[leftarrow,swap,bend right=10,end anchor=-30]{u}{} 
\arrow[leftarrow,swap,bend right=10]{d}{p_2(2)}
\arrow{ul}{} 
 \\
%
%
\delta^{-1}\gold{dy_1}~~\delta^{-2}\gold{x_1d'} 
\arrow[leftarrow,swap,bend right=10,start anchor=140]{u}{p_1}
\arrow[leftarrow,swap,bend right=5]{r}{} 
\arrow[start anchor=60,end anchor=120,bend right,swap]{lr}{1(3)}
\arrow[start anchor=-140,end anchor=-40,bend right,swap]{lr}{q_{1234}(15)}
&
%
\delta^{-\frac{3}{2}}\darkgreen{x_1c_1} 
\arrow[leftarrow,swap,bend right=10]{r}{q_{123}(15)}
\arrow[leftarrow,swap,bend right=5,pos=0.7]{l}{p_4} 
&
%
\delta^{-1}\gold{dy_2} 
\arrow[leftarrow,swap,bend right=10,end anchor=-140]{u}{p_1} 
\arrow[leftarrow,swap,bend right=10]{l}{q_4(3)} 
&
%
\delta^{-\frac{3}{2}}\darkgreen{x_1c_2} 
\arrow[leftarrow,swap,bend right=10]{r}{q_{123}(1)} 
\arrow[leftarrow,swap, bend right=10, end anchor=-10]{ul}{p_{14}(4)}
\arrow[leftarrow,end anchor=-100,start anchor=170,dotted,pos=0.1]{ul}{p_{14}(1246)}
&
%
\delta^{-1}\gold{dy_3} 
\arrow[leftarrow,swap,bend right=10,end anchor=-140,pos=0.6]{u}{p_1} 
\arrow[leftarrow,swap,bend right=10]{l}{q_4(35)} 
\arrow[leftarrow,swap,start anchor=60,end anchor=-40,pos=0.2,dotted]{u}{p_1(345)}
&
%
\delta^{-\frac{3}{2}}\darkgreen{x_1c_3} 
\arrow[leftarrow,swap, bend right=10]{r}{q_3} 
\arrow[leftarrow,swap, bend right=10, end anchor=-10]{ul}{p_{14}(46)} &
%
\delta^{-2}{\blue x_1b'} 
\arrow[leftarrow,swap, bend right=10]{l}{} 
\arrow[leftarrow,swap, bend right=10]{u}{p_{143}(46)}
\end{tikzcd}
$$
\caption{The peculiar invariant for the pretzel tangle above}\label{fig:firstcomplex}
\end{subfigure}
\caption{A computation of a peculiar invariant for a non-rational tangle, see example~\ref{exa:HFTdpretzeltangle}}\label{fig:HFTdmutationexample}
\end{sidewaysfigure}
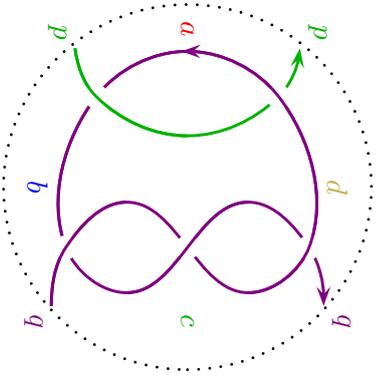
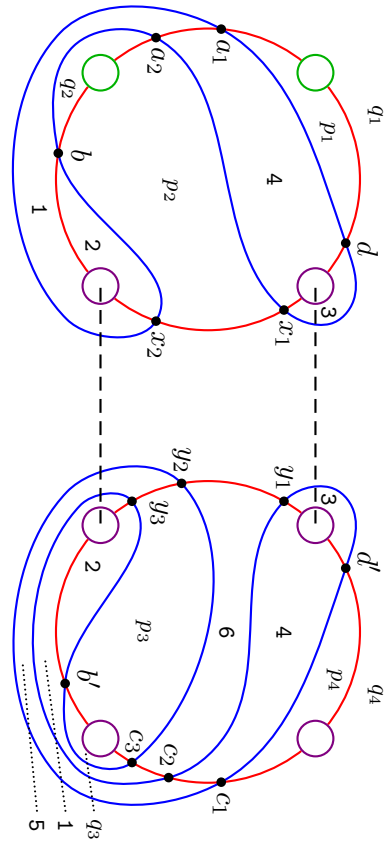
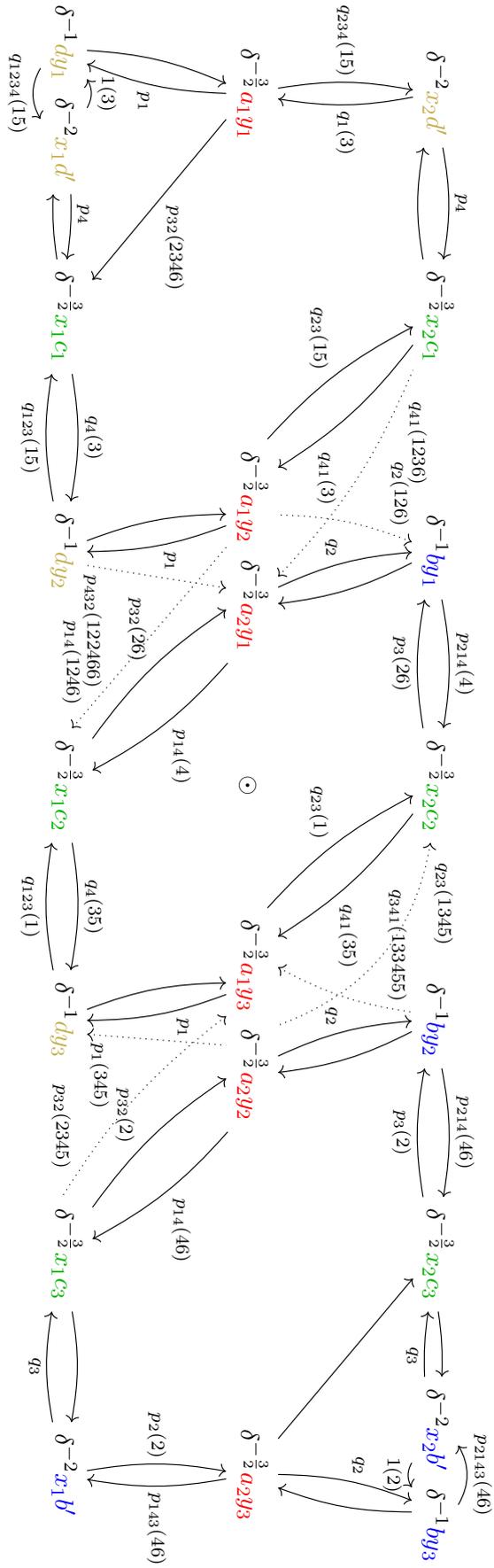

\begin{example}[the $(2,-3)$-pretzel tangle]\label{exa:HFTdpretzeltangle}
Figure~\ref{fig:HFTdmutationexample} shows the computation of $\CFTd(T)$ for the pretzel tangle from example~\sref{exa:pretzeltangle} and theorem~\sref{thm:2m3pt}. First, we only compute bigons and squares. Those are the labelled arrows in figure~\ref{fig:firstcomplex}. But there are also other contributing domains. Grading constraints tell us that we can only get additional morphisms between those generators which are connected by the other arrows, but in principle, those could point in both directions. However, in each case, the connecting domains in one direction either have negative multiplicities or occupy both $p_i$s and $q_i$s, so we can only get arrows in one direction. From this and the $\partial^2$-relation, we can deduce that all solid arrows contribute. There are only eight remaining arrows (the dotted ones) and they can only appear in pairs. But it is easy to see that we can homotope those dotted arrows away (using the clean-up lemma \ref{lem:AbstractCleanUp} for curved type D modules), so in any case, the complex is homotopic to the invariant consisting of the solid black arrows only. 
We can then apply the cancellation lemma~\ref{lem:AbstractCancellation}. We obtain a complex in which every arrow is paired with another one going in the opposite direction and every generator is connected along the arrows to exactly two other generators. A schematic picture of this complex is shown in figure~\ref{fig:examplesimplifiedgraph}, where these arrow pairs have been replaced by single unoriented edges, such that we obtain a collection of loops. This motivates definition~\ref{def:loop} below. Together with a glueing formula like the one in theorem~\ref{thm:CFTdGeneralGlueing}, mutation invariance can now be seen as the rotational symmetry of these loops. Moreover, if we take into account the bigrading, we recover theorem~\sref{thm:2m3pt}.\\
In figures~\ref{fig:mutationexamplefinalresult}b-d, the loops have been transferred onto separate 4-punctured spheres in such a way that the vertices lie on the four arcs that connect the punctures and the unoriented edges lie on the front or back of the spheres depending on whether they correspond to arrow pairs labelled by $p_i$s or $q_i$s. The meaning of these loops will be discussed briefly in the next section, see remark~\ref{rem:CFTdClosing}, and in section~\ref{sec:LoopsAreLoops} -- for the moment, they are just a convenient way to see the rotational symmetry.
\end{example}

\begin{figure}[t]\centering
\begin{subfigure}[b]{\textwidth}\centering
\psset{unit=0.47}
\begin{pspicture}(-16,-6.5)(16,6.5)

{\psset{dotsep=1pt,linestyle=dotted}
\psline(-15,5)(-10,5)
\psline(-5,5)(0,5)
\psline(5,5)(10,5)

\psline(-15,0)(-10,-5)
\psline(-4.5,0)(0,-5)
\psline(5.5,0)(10,-5)

\psline(-5.5,0)(-5,-5)
\psline(4.5,0)(5,-5)
\psline(15,0)(15,-5)
}

\psline(-15,5)(-15,0)
\psline(-10,5)(-5.5,0)
\psline(-10,-5)(-5,-5)
\psline(-5,5)(-4.5,0)
\psline(0,5)(4.5,0)
\psline(0,-5)(5,-5)
\psline(5,5)(5.5,0)
\psline(10,5)(15,0)
\psline(10,-5)(15,-5)

\rput(0,0){$\odot$}

\Mydot(90,-15,5,gold,$-2$)
\Mydot(90,-10,5,darkgreen,$-\frac{3}{2}$)
\Mydot(90,-5,5,blue,$-1$)
\Mydot(90,0,5,darkgreen,$-\frac{3}{2}$)
\Mydot(90,5,5,blue,$-1$)
\Mydot(90,10,5,darkgreen,$-\frac{3}{2}$)

\Mydot(30,-15,0,red,$-\frac{3}{2}$)
\Mydot(210,-5.5,0,red,$-\frac{3}{2}$)
\Mydot(30,-4.5,0,red,$-\frac{3}{2}$)
\Mydot(210,4.5,0,red,$-\frac{3}{2}$)
\Mydot(30,5.5,0,red,$-\frac{3}{2}$)
\Mydot(210,15,0,red,$-\frac{3}{2}$)

\Mydot(-90,-10,-5,darkgreen,$-\frac{3}{2}$)
\Mydot(-90,-5,-5,gold,$-1$)
\Mydot(-90,0,-5,darkgreen,$-\frac{3}{2}$)
\Mydot(-90,5,-5,gold,$-1$)
\Mydot(-90,10,-5,darkgreen,$-\frac{3}{2}$)
\Mydot(-90,15,-5,blue,$-2$)

\end{pspicture}
\caption{A schematic picture of the result. Generators correspond to vertices, arranged according to their Alexander grading and labelled by their $\delta$-grading. The dotted edges correspond to pairs of arrows labelled by $p_i$s, the solid ones to pairs of arrows labelled by $q_i$s.}\label{fig:examplesimplifiedgraph}
\end{subfigure}
\\
\psset{unit=1.3}
\begin{subfigure}[b]{0.3\textwidth}\centering
\begin{pspicture}(-1.5,-1.5)(1.5,1.5)
\psecurve(1.2,1)(0,1.4)(-1.2,1)(1.4,1)(0,0.6)(-1.4,1)(1.2,1)(0,1.4)(-1.2,1)

\psline*[linecolor=white](1,1)(1,-1)(-1,-1)(-1,1)
\psline[linestyle=dashed](1,1)(1,-1)
\psline[linestyle=dashed](1,-1)(-1,-1)
\psline[linestyle=dashed](-1,-1)(-1,1)
\psline[linestyle=dashed](-1,1)(1,1)

\psecurve[dotsep=1pt,linestyle=dotted](1.2,1)(0,1.4)(-1.2,1)(1.4,1)(0,0.6)(-1.4,1)(1.2,1)(0,1.4)(-1.2,1)

\psset{dotsize=5pt}

\psdot[linecolor=red](-1,0.8)
\psdot[linecolor=red](-1,0.663)

\psdot[linecolor=gold](0.125,1)
\psdot[linecolor=gold](-0.125,1)

\psdot[linecolor=darkgreen](1,0.8)
\psdot[linecolor=darkgreen](1,0.663)

\pscircle[fillstyle=solid, fillcolor=white](1,1){0.08}
\pscircle[fillstyle=solid, fillcolor=white](-1,1){0.08}
\pscircle[fillstyle=solid, fillcolor=white](1,-1){0.08}
\pscircle[fillstyle=solid, fillcolor=white](-1,-1){0.08}

\end{pspicture}
\caption{}\label{fig:loopbottom}
\end{subfigure}
\quad
\begin{subfigure}[b]{0.3\textwidth}\centering
\begin{pspicture}(-1.5,-1.5)(1.5,1.5)

\psecurve(-0.8,-1.2)(-1.2,-1.2)(1.2,1.1)(-1.2,0.9)(0,1.4)(1.4,1.2)(-0.8,-1.2)(-1.2,-1.2)(1.2,1.1)

\psline*[linecolor=white](1,1)(1,-1)(-1,-1)(-1,1)
\psline[linestyle=dashed](1,1)(1,-1)
\psline[linestyle=dashed](1,-1)(-1,-1)
\psline[linestyle=dashed](-1,-1)(-1,1)
\psline[linestyle=dashed](-1,1)(1,1)

\psecurve[dotsep=1pt,linestyle=dotted](-0.8,-1.2)(-1.2,-1.2)(1.2,1.1)(-1.2,0.9)(0,1.4)(1.4,1.2)(-0.8,-1.2)(-1.2,-1.2)(1.2,1.1)

\psset{dotsize=5pt}

\psdot[linecolor=red](-1,0.724)
\psdot[linecolor=red](-1,-0.635)

\psdot[linecolor=gold](-0.027,1)

\psdot[linecolor=darkgreen](1,0.51)
\psdot[linecolor=darkgreen](1,-0.03)

\psdot[linecolor=blue](-0.38,-1)

\pscircle[fillstyle=solid, fillcolor=white](1,1){0.08}
\pscircle[fillstyle=solid, fillcolor=white](-1,1){0.08}
\pscircle[fillstyle=solid, fillcolor=white](1,-1){0.08}
\pscircle[fillstyle=solid, fillcolor=white](-1,-1){0.08}

\end{pspicture}
\caption{}\label{fig:loopmiddle}
\end{subfigure}
\quad
\begin{subfigure}[b]{0.3\textwidth}\centering
\begin{pspicture}(-1.5,-1.5)(1.5,1.5)

\psrotate(0,0){180}{
\psecurve(1.2,1)(0,1.4)(-1.2,1)(1.4,1)(0,0.6)(-1.4,1)(1.2,1)(0,1.4)(-1.2,1)

\psline*[linecolor=white](1,1)(1,-1)(-1,-1)(-1,1)
\psline[linestyle=dashed](1,1)(1,-1)
\psline[linestyle=dashed](1,-1)(-1,-1)
\psline[linestyle=dashed](-1,-1)(-1,1)
\psline[linestyle=dashed](-1,1)(1,1)

\psecurve[dotsep=1pt,linestyle=dotted](1.2,1)(0,1.4)(-1.2,1)(1.4,1)(0,0.6)(-1.4,1)(1.2,1)(0,1.4)(-1.2,1)

\psset{dotsize=5pt}

\psdot[linecolor=darkgreen](-1,0.8)
\psdot[linecolor=darkgreen](-1,0.663)

\psdot[linecolor=blue](0.125,1)
\psdot[linecolor=blue](-0.125,1)

\psdot[linecolor=red](1,0.8)
\psdot[linecolor=red](1,0.663)

\pscircle[fillstyle=solid, fillcolor=white](1,1){0.08}
\pscircle[fillstyle=solid, fillcolor=white](-1,1){0.08}
\pscircle[fillstyle=solid, fillcolor=white](1,-1){0.08}
\pscircle[fillstyle=solid, fillcolor=white](-1,-1){0.08}
}
\end{pspicture}
\caption{}\label{fig:looptop}
\end{subfigure}
\caption{The final result of the computation from example~\ref{exa:HFTdpretzeltangle} and figure~\ref{fig:HFTdmutationexample}. (b)--(d) show the three loops from (a) separately on 4-punctured spheres.}\label{fig:mutationexamplefinalresult}
\end{figure}

\begin{proposition}\label{prop:CFTdCancelAll}
For any 4-ended tangle \(T\), \(\CFTd(T)\) is homotopic to a peculiar module without any identity components in the differential. 
\end{proposition}
\begin{proof}
The identity is the only algebra element of $\Ad$ with $\delta$-grading 0 and the only algebra elements with $\delta$-grading 1 are $p_ip_{i-1}$ and $q_iq_{i+1}$, so there cannot be any arrow from a generator to itself. So we can repeatedly apply the cancellation lemma~\ref{lem:AbstractCancellation}.
\end{proof}

\begin{definition}\label{def:loop}
Let $M$ be a peculiar module without any identity components. As usual, we can think of $M$ as a labelled directed graph whose vertices are the generators and whose arrows correspond to morphisms, labelled by $p_i$s and $q_i$s. $M$ is called a \textbf{loop}, if this graph is connected and every vertex is 4-valent. Note that we can represent any such loop as an immersed curve on the 4-punctured sphere in the same way as in the previous two examples.
In analogy to the torus boundary case from~\cite{HanselmanWatson}, we call $M$ \textbf{loop-type} if $M$ is homotopic to a collection of loops.
\end{definition}
\begin{questions}\label{que:IsEverythingALoop}
Is \(\CFTd(T)\) loop-type for every tangle \(T\)? Do homotopy classes of loop-type peculiar modules correspond to homotopy classes of loops on the 4-punctured sphere? If so, what does the number of components tell us about \(T\)?
\end{questions}
\begin{Remark}
We do not expect every peculiar module to be loop-type. Consider the following example:
$$
\begin{tikzcd}
b
\arrow[leftarrow,bend left=12]{r}{p_{43}+q_{12}}
\arrow[swap]{d}{p_{21}}
& 
d
\arrow[leftarrow,bend left=10]{l}{p_{21}+q_{34}}
\arrow{d}{p_{43}}
\\
d
\arrow[bend left=12]{r}{p_{43}+q_{12}}
& 
b
\arrow[bend left=10]{l}{p_{21}+q_{34}}
\end{tikzcd}
$$
In the Fukaya category setting, i.\,e.\ under the functor $\mathcal{L}$ from theorem~\ref{thm:TwFukpqModEquivalent}, this corresponds to the cone of the non-identity morphism between the loop for the trivial tangle from figure~\ref{fig:CFTdForSomeRatTangles} and a Hamiltonian push-off, which is a prototypical example of a local system from~\cite{Kontsevich}, see remark~\ref{rem:FUKlooptype}.
\end{Remark}
In proposition~\ref{prop:CFTdCancelAll}, we have seen that one can eliminate all identity components in a peculiar module. 
In view of questions~\ref{que:IsEverythingALoop}, the following question might also be interesting.
\begin{question}
Is every peculiar module homotopic to one with \(p\)-, \(p^2\)-, \(p^3\)-, \(q\)-, \(q^2\)- and \(q^3\)-components only? In other words, can we always eliminate all identity components and ``long'' algebra elements?
\end{question}
\begin{proposition}\label{prop:LazyClosing}
Let \(T\) be a 4-ended tangle and \(L\) the link obtained by closing \(T\) at the sites \(a\) and \(c\) like so:
$$
\psset{unit=0.3}
\begin{pspicture}(-8.01,-3.01)(8.01,3.01)
\pscircle[linestyle=dotted](0,0){3}
\pscircle[fillcolor=white,fillstyle=solid](0,0){1.5}
\rput(0,0){$T$}
\rput(2.25;180){$a$}
\rput(2.25;0){$c$}
\psecurve(0,0)(1.5;135)(4;135)(-4.5,0)(4;-135)(1.5;-135)(0,0)
\psecurve(0,0)(1.5;45)(4;45)(4.5,0)(4;-45)(1.5;-45)(0,0)
\end{pspicture}
$$
Then
\begin{align*}\SwapAboveDisplaySkip
\CFL(L)\otimes V^{i}&\cong\CFTd(T)/(p_1=p_2=q_3=q_4=1,q_1=q_2=p_3=p_4=0)\\
&\cong\CFTd(T)/(p_1=p_2=q_3=q_4=0,q_1=q_2=p_3=p_4=1),
\end{align*}
where \(i=0\)~or~\(1\), depending on whether \(L\) has two or one more closed component(s) than~\(T\), respectively. For the other two opposite sites, we obtain a similar formula by a cyclic permutation of the indices.
\end{proposition}
\begin{proof}
If we delete those basepoints in a peculiar Heegaard diagram of $T$ that correspond to the variables that we set equal to 1, we obtain a Heegaard diagram for $L$. In $\CFL(L)$, we only count those holomorphic curves that stay away from the remaining basepoints, so we need to set those algebra elements equal to 0. 
\end{proof}

\begin{Remark}\label{rem:LazyClosing}
One might hope to get the other flavours of link Floer homology from the same kind of idea as in the proof of the previous proposition. However, this does not work quite as well because of the relation $p_iq_i=0$. As soon as we do not set at least one of the two marked points at each tangle end equal to 0, we would have to count more holomorphic curves than we do for $\CFTd(T)$. \\
Suppose, the first two of questions~\ref{que:IsEverythingALoop} have a positive answer. Then, we might still hope to reconstruct from $\CFTd(T)$  the other flavours of the link Floer homology of the closure of~$T$ as follows: Consider the embedded loop(s) corresponding to $\CFTd(T)$ as a $\beta$-circle and the four $\alpha$-arcs as a single $\alpha$-circle as before. Then we can calculate a new complex from this by also counting those bigons which occupy both $p_i$ and $q_i$.
\end{Remark}
We end this section with some simple observations about $\CFTd$.

\begin{observation}\label{obs:AlexGradingOfCFTdLoops}
Let $T\subset B^3$ be a 4-ended tangle and suppose its peculiar module is a collection of loops. As in example~\ref{exa:HFTdpretzeltangle}/figure~\ref{fig:mutationexamplefinalresult}, we can transfer them onto the 4-punctured sphere $\partial B^3\smallsetminus \partial T$. By considering the Alexander grading, we see that every loop is in the kernel of 
$$\pi_1(\partial B^3\smallsetminus \partial T)\rightarrow\pi_1(X_{T})\rightarrow H_1(X_{T}),$$ 
where $X_T=B^3\smallsetminus \nu(T)$ is the complement of a tubular neighbourhood of the tangle~$T$.
\end{observation}
\begin{observation}
By definition, the Alexander grading corresponding to each closed component vanishes on $\Ad$. Also, the differential of a peculiar module preserves the Alexander grading by lemma~\ref{lem:CFTdGradings}. Thus, $\CFT(T)$ decomposes into the direct sum over the Alexander gradings of closed components. 
\end{observation}
\begin{observation}
In \cite{OSHFKalt}, Ozsv\'{a}th and Szab\'{o} showed that the link Floer homology of alternating links is completely determined by the Alexander polynomial (and, to be precise, the signature, but this is only needed to fix the absolute grading). The proof generalises immediately to $\HFT$, using the generalised clock theorem~\sref{geclockt}. 
\end{observation}
\begin{question}
Given an alternating tangle \(T\), is the bigraded chain homotopy type of \(\CFTd(T)\) determined by \(\nabla_T^s\)?
\end{question}

%% file: sections/3_Pairing.tex

\section{Pairing 4-ended tangles}\label{sec:Pairing}
In this subsection, we prove a glueing formula for $\CFTd$: Given the peculiar modules of two 4-ended tangles $T_1$ and $T_2$, we compute the Heegaard Floer homology of the link $L$ obtained by glueing $T_1$ to $T_2$, up to at most three stabilisations. The proof is essentially a calculation of a bordered sutured type AA bimodule, for which we use a computer since the number of generators in the bimodule is rather large. As a warm-up, we prove a special case of the glueing formula, namely for $T_2$ being a trivial tangle, where we can do this calculation by hand. This corresponds to computing $\HFL$ of the closure of a tangle. Of course, we already know how to do this, see proposition~\ref{prop:LazyClosing}. There are three reasons for looking at this again. Firstly, it illustrates how the proof of the general glueing formula works. Secondly, the result that we get in this special case is as nice as one might hope, which, unfortunately, we cannot say about the general glueing formula in its present state, see conjecture~\ref{conj:CFTdBetterGlueing}. Thirdly, a comparison of the two methods for computing $\HFL$ of the closure might give us interesting symmetry relations on $\CFTd$.

\begin{figure}[b]
\centering
\begin{subfigure}[b]{\textwidth}
$$\begin{tikzcd}[row sep=2cm, column sep=3cm]
b_1
\arrow[leftarrow]{r}[description]{1+(p_{21},q_{12})}
\arrow[leftarrow]{d}[description]{(q_2)}
\arrow[leftarrow,pos=0.2]{rd}[description]{(p_1,q_{12})}
& 
b_2
\\
a_1
\arrow[leftarrow]{r}[description]{(p_1,q_1)}
\arrow[leftarrow,pos=0.2]{ur}[description]{(p_{21},q_1)}
&
a_2
\arrow[leftarrow]{u}[description]{(p_2)}
\end{tikzcd}$$
\vspace{-0.4cm}
\caption{The type A structure $\mathcal{C}(s)$ for closing a 4-ended tangle at the site $s=a$. The identity action is implicit. We similarly define $\mathcal{C}(s)$ for any of the other three sites of a 4-ended tangle by cyclic permutations of the sites and indices corresponding to rotations of figure~\ref{fig:nutshell}.}\label{fig:ClosingGraph}
\end{subfigure}
\begin{subfigure}[b]{\textwidth}\centering
\psset{unit=0.4}
\begin{pspicture}(-11,-11)(11,11)

\pscustom[fillstyle=solid,fillcolor=lightgray,linewidth=0pt,linecolor=white]{
\psline(0,-2.5)(0,-7.5)
\psline[liftpen=1,linearc=0.8](0,-7.5)(2.5,-7.5)(2.5,-2.5)(0,-2.5)
}

\pscustom[fillstyle=solid,fillcolor=lightgray,linewidth=0pt,linecolor=white]{
\psline[liftpen=1,linearc=1](0,10)(-10,10)(-10,-10)(0,-10)
\psline(0,-10)(0,-7.5)
\psline[liftpen=1,linearc=0.8](0,-7.5)(-7.5,-7.5)(-7.5,7.5)(-2.5,7.5)(-2.5,-2.5)(0,-2.5)
\psline(0,-2.5)(0,10)
}

\psframe[linecolor=blue,linearc=1,cornersize=absolute](-10,-10)(10,10)
\psline[linecolor=blue](0,10)(0,-10)
\psline[linecolor=blue](10,0)(-10,0)

\psframe[linecolor=red,linestyle=dotted,dotsep=1pt](-5,-5)(5,5)
\pscircle[fillstyle=solid](-5,-5){0.7}
\pscircle[fillstyle=solid](5,-5){0.7}
\pscircle[fillstyle=solid](-5,5){0.7}
\pscircle[fillstyle=solid](5,5){0.7}

\rput(5,5){\blue $4$}
\rput(5,-5){\blue $3$}
\rput(-5,-5){\blue $2$}
\rput(-5,5){\blue $1$}

\pscustom[linecolor=violet,linearc=0.8,cornersize=absolute]{
\psline(-7.5,0)(-7.5,7.5)(-2.5,7.5)(-2.5,-2.5)(2.5,-2.5)(2.5,-7.5)(-7.5,-7.5)(-7.5,0)
}

\pscircle[fillstyle=solid,linecolor=blue, fillcolor=white](0,0){0.7}
\rput(0,0){\blue $p$}
\rput(1.3;135){\blue $1$}
\rput(1.3;-135){\blue $2$}
\rput(1.3;-45){\blue $3$}
\rput(1.3;45){\blue $4$}

\rput(-10.3,10.3){\blue $q$}
\rput(10.3,-10.3){\blue $q$}
\rput(10.3,10.3){\blue $q$}
\rput(-10.3,-10.3){\blue $q$}
\rput(9.2,9.2){\blue $4$}
\rput(9.2,-9.2){\blue $3$}
\rput(-9.2,-9.2){\blue $2$}
\rput(-9.2,9.2){\blue $1$}

\uput{0.5}[135](-5,0){\blue $a$}
\uput{0.5}[-135](0,-5){\blue $b$}
\uput{0.5}[45](5,0){\blue $c$}
\uput{0.5}[45](0,5){\blue $d$}

\psdot(-7.5,0)
\uput{0.4}[-45](-7.5,0){$a_1$}
\psdot(-2.5,0)
\uput{0.4}[-135](-2.5,0){$a_2$}
\psdot(0,-2.5)
\uput{0.4}[-135](0,-2.5){$b_2$}
\psdot(0,-7.5)
\uput{0.4}[135](0,-7.5){$b_1$}

\end{pspicture}
\caption{A combinatorial model for pairing any Lagrangian in the wrapped Fukaya category of the 4-punctured sphere with the \textcolor{violet}{violet} loop respresenting a trivial tangle. The boundary of the picture is identified to a point. The {\blue blue} lines denote a 1-skeleton on which the Lagrangian lives.}\label{fig:ClosingFUK}
\end{subfigure}
\caption{Closing a 4-ended tangles at the site $s=a$}\label{fig:ClosingCAT}
\end{figure}
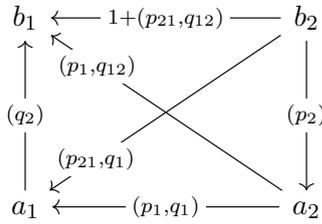
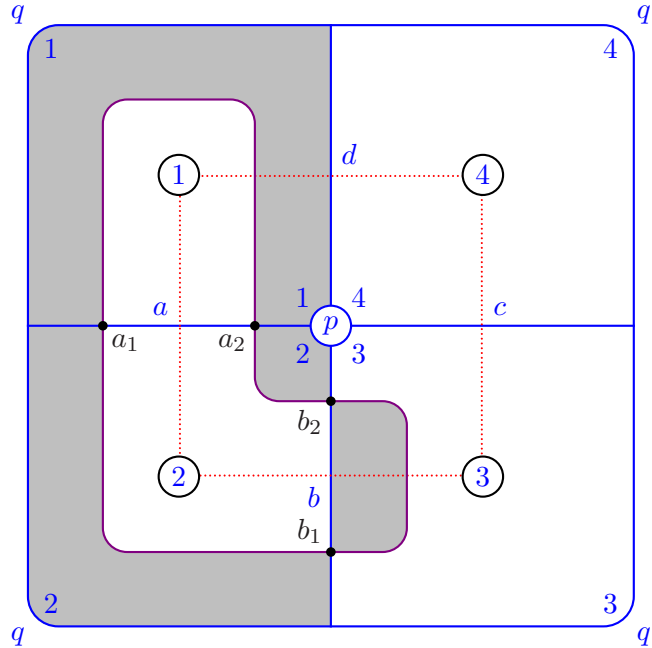

\begin{theorem}\label{thm:CFTdGlueingTrivial}
Let \(T\) be a 4-ended tangle and \(L\) the link obtained by closing \(T\) at the site \(a\) and \(c\) like so:
$$
\psset{unit=0.3}
\begin{pspicture}(-8.01,-3.01)(8.01,3.01)
\pscircle[linestyle=dotted](0,0){3}
\pscircle[fillcolor=white,fillstyle=solid](0,0){1.5}
\rput(0,0){$T$}
\rput(2.25;180){$a$}
\rput(2.25;0){$c$}
\psecurve(0,0)(1.5;135)(4;135)(-4.5,0)(4;-135)(1.5;-135)(0,0)
\psecurve(0,0)(1.5;45)(4;45)(4.5,0)(4;-45)(1.5;-45)(0,0)
\end{pspicture}
$$
For a site \(s\), let \(\mathcal{C}(s)\) be the strictly unital type A structure over \(\Ad\) defined by the labelled graph shown in figure~\ref{fig:ClosingGraph}. Then 
$$\CFL(L)\otimes V^{i}\cong\CFTd(T)\boxtimes\,\mathcal{C}(a)\cong\CFTd(T)\boxtimes\,\mathcal{C}(c)$$
where \(i=0\)~or~\(1\), depending on whether \(L\) has two or one more closed component(s) than~\(T\), respectively. For the other two opposite sites, we obtain similar formulas by cyclic permutations of the indices.
\end{theorem}
\begin{Remark}\label{rem:CFTdClosing}
If $\CFTd$ is loop-type, we can interpret this pairing theorem in terms of Lagrangian intersection homology. Consider figure~\ref{fig:ClosingFUK}. The blue lines denote a 1-skeleton~$\blue \mathcal{R}$ with two 0-cells labelled~$\blue p$ and~$\blue q$ and four 1-cells $\blue a$, $\blue b$, $\blue c$ and $\blue d$, each of them connecting~$\blue p$ and~$\blue q$, such that $\blue \mathcal{R}$ is a deformation retract of the 4-punctured sphere. 
Given a collection of loops representing a loop-type peculiar module of a tangle~$T$, we write it as a cycle on $\blue \mathcal{R}$. By shortening this cycle as much as possible, we get a unique one. 
Now consider the violet loop~$\textcolor{violet}{\mathcal{C}}$, which is a Hamiltonian (=area-preserving) push-off of a loop on $\blue \mathcal{R}$ representing the trivial tangle. Then we see four intersection points of~$\textcolor{violet}{\mathcal{C}}$ with~$\blue \mathcal{R}$, which we label $a_1$, $a_2$, $b_1$ and $b_2$. 
These correspond to the generators in~$\mathcal{C}(a)$. The differentials in $\mathcal{C}(a)$ can also be computed from this picture; they come from counting convex bigons that connect the generators and avoid all punctures. The labelling can be computed from the multiplicities of the bigons in the regions near the 0-cells~$\blue p$ and~$\blue q$. 
Our \textbf{conventions} are such that the curve on \textcolor{blue}{$\mathcal{R}$} plays the role of a $\beta$-curve and \textcolor{violet}{$\mathcal{C}$} the one of an $\alpha$-curve. 
Furthermore, the normal vector field points into the plane, i.\,e.\ the sphere. For example, we see two bigons connecting $b_2$ to $b_1$: One of them stays away from the 0-cells~$\blue p$ and~$\blue q$, so it contributes~1, the other picks up $p_2$, $p_1$, $q_1$ and $q_2$.
\end{Remark}
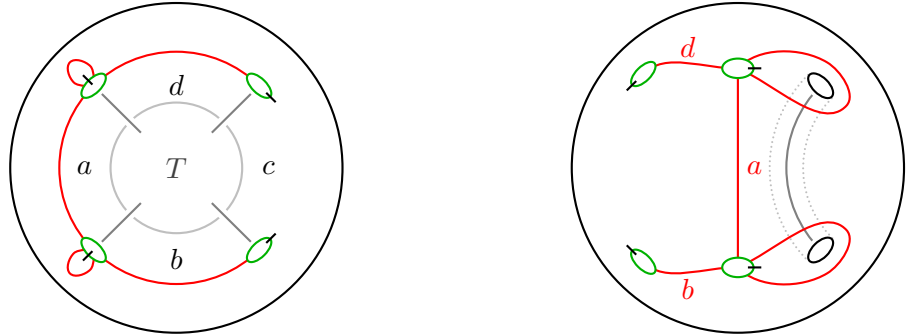
\begin{figure}[!ht]
\centering
\psset{unit=0.22}
\begin{subfigure}[b]{0.45\textwidth}\centering
\begin{pspicture}(-10,-10)(10,10)
\pscircle(0,0){10}
\pscircle[linecolor=lightgray](0,0){4}

\psline[linecolor=white,linewidth=4pt](3;45)(7;45)
\psline[linecolor=white,linewidth=4pt](3;135)(7;135)
\psline[linecolor=white,linewidth=4pt](3;-135)(7;-135)
\psline[linecolor=white,linewidth=4pt](3;-45)(7;-45)

\psline[linecolor=gray](3;45)(6.3;45)
\psline[linecolor=gray](3;135)(6.3;135)
\psline[linecolor=gray](3;-135)(6.3;-135)
\psline[linecolor=gray](3;-45)(6.3;-45)

\psarc[linecolor=red](0,0){7}{45}{-45}

\pscurve[linecolor=red](7;135)(8;130)(9;135)(8;140)(7;135)
\pscurve[linecolor=red](7;-135)(8;-130)(9;-135)(8;-140)(7;-135)

{\psset{fillstyle=solid,fillcolor=white}
\rput{45}(7;45){\psellipse[linecolor=darkgreen](0,0)(0.5,1)\psline(0,-0.6)(0,-1.4)}
\rput{135}(7;135){\psellipse[linecolor=darkgreen](0,0)(0.5,1)}
\rput{-135}(7;-135){\psellipse[linecolor=darkgreen](0,0)(0.5,1)}
\rput{-45}(7;-45){\psellipse[linecolor=darkgreen](0,0)(0.5,1)\psline(0,0.6)(0,1.4)}
}

\psline(7.1;135)(7.9;135)
\psline(7.1;-135)(7.9;-135)

\rput(-5.5,0){$a$}
\rput(0,-5.5){$b$}
\rput(5.5,0){$c$}
\rput(0,5){$d$}
\rput(0,0){\textcolor{darkgray}{$T$}}

\end{pspicture}
\caption{The bordered sutured structure on $X_T$}\label{fig:ClosingTangleT}
\end{subfigure}
\quad
\begin{subfigure}[b]{0.45\textwidth}\centering
\begin{pspicture}(-10,-10)(10,10)
\pscircle(0,0){10}

\psecurve[linecolor=red](8;-135)(8;135)(0,6)(2,6)
\psecurve[linecolor=red](8;135)(8;-135)(0,-6)(2,-6)
\psline[linecolor=red](0,-6)(0,6)

\psarc[linecolor=gray](9.89949493661167,0){7}{135}{-135}

\rput{45}(7;45){\psellipse*[linecolor=white](0,0)(0.75,1.2)}
\rput{-45}(7;-45){\psellipse*[linecolor=white](0,0)(0.75,1.2)}

\psarc[linecolor=lightgray,linestyle=dotted,dotsep=1pt](9.89949493661167,0){8}{135}{-135}
\psarc[linecolor=lightgray,linestyle=dotted,dotsep=1pt](9.89949493661167,0){6}{135}{-135}

\psecurve[linecolor=red](0,4)(0,6)(8.3;50)(7.5;30)(0,6)(-2,7)
\psecurve[linecolor=red](0,-4)(0,-6)(8.3;-50)(7.5;-30)(0,-6)(-2,-7)

{\psset{fillstyle=solid,fillcolor=white}
\rput{45}(7;45){\psellipse[linecolor=black](0,0)(0.55,1)}
\rput{-45}(7;-45){\psellipse[linecolor=black](0,0)(0.55,1)}
\rput{135}(8;135){\psellipse[linecolor=darkgreen](0,0)(0.5,1)\psline(0,0.6)(0,1.4)}
\rput{-135}(8;-135){\psellipse[linecolor=darkgreen](0,0)(0.5,1)\psline(0,-0.6)(0,-1.4)}
\rput{90}(6;90){\psellipse[linecolor=darkgreen](0,0)(0.65,1)}
\rput{-90}(6;-90){\psellipse[linecolor=darkgreen](0,0)(0.65,1)}
}

\psline(0.6,6)(1.4,6)
\psline(0.6,-6)(1.4,-6)

\rput(-3,-7.3){\red $b$}
\rput(1,0){\red $a$}
\rput(-3,7.3){\red $d$}

\end{pspicture}
\caption{The bordered sutured manifold $C(a)$}\label{fig:ClosingTangleClosure}
\end{subfigure}
\caption{An illustration for the proof of theorem~\ref{thm:CFTdGlueingTrivial}. Pairing these two bordered sutured manifolds gives the sutured manifold of the link obtained by closing $T$ at the site $a$, plus an extra pair of meridional sutures, if the two tangle ends adjacent to the arc $a$ do not belong to the same component of $T$.}\label{fig:ClosingTangle}
\end{figure}

\begin{proof}[Proof of theorem~\ref{thm:CFTdGlueingTrivial}]
The proof is essentially a computation of a bordered sutured type~A structure corresponding to closing a tangle, namely the one for the bordered sutured manifold $C(a)$ shown in figure~\ref{fig:ClosingTangleClosure}. A Heegaard diagram for it is shown in figure~\ref{fig:ClosingTangleHD}. We only compute the full subcomplex of the corresponding type A structure consisting of those generators which occupy exactly two arcs of $\{a,b,d\}$; this part is shown in figure~\ref{fig:ClosingTangleTypeARaw}. Its computation is more or less straightforward, noting that the multiplicities of all regions are either 0 or 1. The only domain which could contribute and which is not a polygon is $A+B$. To solve this problem, we can argue as follows: Pairing the type A structure with the type D structure consisting of a single generator in idempotent $\{d\}$ and no differential gives~0. By symmetry of $C(a)$, pairing with a single generator in idempotent $\{b\}$ should give 0 as well. However, if $A+B$ did not contribute, the homology of the paired complex would be 2-dimensional.\\
We can simplify this type A structure by cancelling the identity map between the generators $b_2y_1z_2$ and $b_1x_2z_1$ and then do some homotopies using lemma~\ref{lem:AbstractCleanUp} to obtain the following complex, which we denote by $\mathcal{C}$: 
$$\begin{tikzcd}[row sep=2cm, column sep=3cm]
b_1
\arrow[leftarrow]{r}[description]{1+(\delta_2\tau_2,\tau_{13}\delta_{13})}
\arrow[leftarrow]{d}[description]{(\delta_{13})}
\arrow[leftarrow,pos=0.3]{rd}[description]{(\tau_2,\tau_{13}\delta_{13})}
& 
b_2
\\
a_1
\arrow[leftarrow]{r}[description]{(\tau_2,\tau_{13})}
\arrow[leftarrow,pos=0.3]{ur}[description]{(\delta_2\tau_2,\tau_{13})}
&
a_2
\arrow[leftarrow]{u}[description]{(\delta_2)}
\end{tikzcd}$$\pagebreak[3]

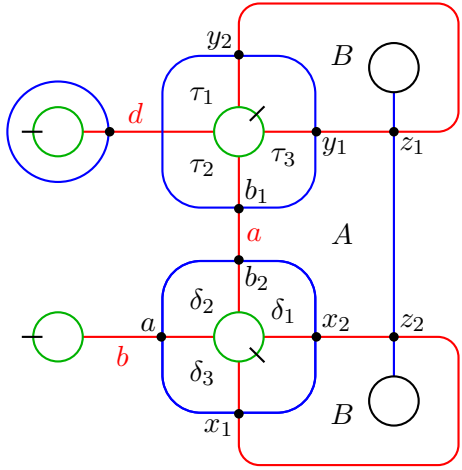
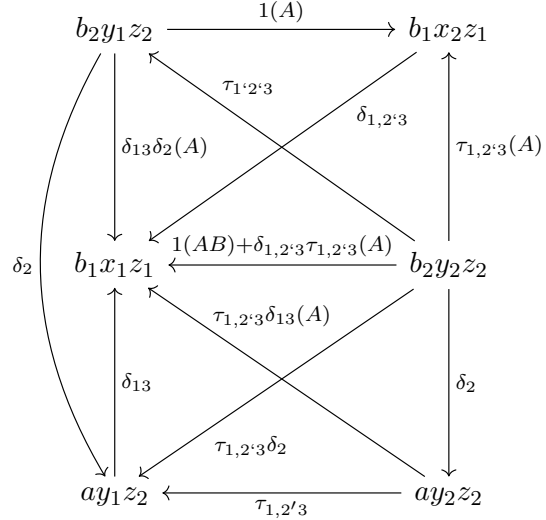
\begin{figure}[t]
\centering
\psset{unit=0.34}
\begin{subfigure}[b]{0.45\textwidth}\centering
\begin{pspicture}(-9,-10)(9,10)

\psline[linecolor=red,linearc=0.8](-7,4)(8.5,4)(8.5,9)(0,9)(0,-9)(8.5,-9)(8.5,-4)(-7,-4)

\psframe[linecolor=blue,framearc=0.5](-3,1)(3,7)
\psframe[linecolor=blue,framearc=0.5](-3,-1)(3,-7)
\psframe[linecolor=blue,framearc=0.5](-3,-1)(3,-7)
\psline[linecolor=blue](6,6.5)(6,-6.5)
\pscircle[linecolor=blue](-7,4){2}

\pscircle[fillstyle=solid,fillcolor=white,linecolor=darkgreen](0,4){1}
\pscircle[fillstyle=solid,fillcolor=white,linecolor=darkgreen](0,-4){1}
\pscircle[fillstyle=solid,fillcolor=white,linecolor=darkgreen](-7,4){1}
\pscircle[fillstyle=solid,fillcolor=white,linecolor=darkgreen](-7,-4){1}
\pscircle[fillstyle=solid,fillcolor=white](6,6.5){1}
\pscircle[fillstyle=solid,fillcolor=white](6,-6.5){1}

\psdot(-5,4)
\rput(0,1){\psdot(0,0)\uput{0.3}[45](0,0){$b_1$}}
\rput(0,7){\psdot(0,0)\uput{0.3}[135](0,0){$y_2$}}
\rput(3,4){\psdot(0,0)\uput{0.3}[-45](0,0){$y_1$}}
\rput(0,-1){\psdot(0,0)\uput{0.3}[-45](0,0){$b_2$}}
\rput(0,-7){\psdot(0,0)\uput{0.3}[-135](0,0){$x_1$}}
\rput(3,-4){\psdot(0,0)\uput{0.3}[45](0,0){$x_2$}}
\rput(-3,-4){\psdot(0,0)\uput{0.3}[135](0,0){$a$}}
\rput(6,4){\psdot(0,0)\uput{0.3}[-45](0,0){$z_1$}}
\rput(6,-4){\psdot(0,0)\uput{0.3}[45](0,0){$z_2$}}
\rput(4,0){$A$}
\rput(4,7){$B$}
\rput(4,-7){$B$}
\rput(0,4){
\rput(2;-30){$\tau_3$}
\rput(2;-135){$\tau_2$}
\rput(2;135){$\tau_1$}
\psline(0.6;45)(1.4;45)
}
\rput(0,-4){
\rput(2;-135){$\delta_3$}
\rput(2;135){$\delta_2$}
\rput(2;30){$\delta_1$}
\psline(0.6;-45)(1.4;-45)
}
\psline(-7.6,4)(-8.4,4)
\psline(-7.6,-4)(-8.4,-4)

\rput(-4.5,-4.75){\red $b$}
\rput(0.6,0){\red $a$}
\rput(-4,4.75){\red $d$}

\end{pspicture}
\caption{A Heegaard diagram for the bordered sutured manifold in figure~\ref{fig:ClosingTangleClosure}. The normal vector field points into the plane.}\label{fig:ClosingTangleHD}
\end{subfigure}
\quad
\begin{subfigure}[b]{0.49\textwidth}
$$\begin{tikzcd}[row sep=2.5cm, column sep=3cm]
b_2y_1z_2
\arrow[leftarrow,pos=0.25]{rd}{\tau_{1`2`3}}
&
b_1x_2z_1 
\arrow[leftarrow,swap]{l}{1(A)}
\arrow[leftarrow]{d}{\tau_{1,2`3}(A)}
\\
b_1x_1z_1 
\arrow[leftarrow]{r}{1(AB)+\delta_{1,2`3}\tau_{1,2`3}(A)}
\arrow[leftarrow]{d}{\delta_{13}}
\arrow[leftarrow,pos=0.25]{dr}{\!\!\!\tau_{1,2`3}\delta_{13}(A)}
\arrow[leftarrow,swap]{u}{\delta_{13}\delta_{2}(A)}
\arrow[leftarrow,swap,pos=0.75]{ur}{\delta_{1,2`3}}
& 
b_2y_2z_2
\\
ay_1z_2
\arrow[leftarrow,swap]{r}{\tau_{1,2'3}}
\arrow[leftarrow,pos=0.25,swap]{ru}{\tau_{1,2`3}\delta_{2}}
\arrow[leftarrow,bend left]{uu}{\delta_{2}}
& 
ay_2z_2
\arrow[leftarrow,swap]{u}{\delta_{2}}
\end{tikzcd}$$
\vspace{-0.4cm}
\caption{The type A structure computed from the Heegaard diagram on the left}\label{fig:ClosingTangleTypeARaw}
\end{subfigure}
\caption{The computation of the type A structure for the bordered sutured manifold in figure~\ref{fig:ClosingTangleClosure}. There are only two interior regions, labelled by $A$, $B$. If a domain containing one or both of these contributes to the structure map, we record it in brackets in the type A structure. A $`$ ($,$) in a subscript means that the corresponding $\alpha$-arc is (not) occupied. For more details, see \cite{BSAx2.nb}.
}\label{fig:ClosingTangleComputation}
\end{figure}

\noindent
Let $\Ad_{12}$ be the quotient of $\Ad$ by $\iota_3$ and all elements~$p_I$ and~$q_I$ where $I\cap\{3,4\}\neq\emptyset$. Let $\pi_{12}$ denote the quotient map. Note that $\Ad_{12}$ is generated by $\{p_1,p_2,p_{12},q_1,q_2,q_{12}\}$ as an $\Id/\iota_3$-module. Furthermore, if $\mathcal{B}$ denotes the dg algebra underlying $\mathcal{C}$, we can define an $\Id/\iota_3$-algebra homomorphism 
$$\pi:\mathcal{B}\rightarrow\Ad_{12}$$
by sending
$\tau_2\mapsto p_1$, 
$\delta_2\mapsto p_2$, 
$\delta_2\tau_2\mapsto p_{21}$, 
$\tau_{13}\mapsto q_1$, 
$\delta_{13}\mapsto q_2$, 
$\tau_{13}\delta_{13}\mapsto q_{12}$ and all other basic algebra elements to 0.
Note that this is a homomorphism of dg algebras, where we define the differential on $\Ad_{12}$ to vanish. Hence, by remark~\ref{rem:FunctorsBetweenDGCats}, $\pi$ induces a functor $\mathcal{F}_\pi$ from type A structures over $\mathcal{B}$ to those over $\Ad_{12}$, that respects chain homotopies. Note that $\mathcal{F}_\pi(\mathcal{C})=\mathcal{C}(a)$.\\
Next, let $\mathcal{D}(T)$ be the type D structure computed from a Heegaard diagram, where the silly $\alpha$-arcs do not intersect any $\beta$-curves. Then
$$\SFC(X_T\cup C(a))=\mathcal{D}(T)\boxtimes\mathcal{C}=\mathcal{F}_\pi(\mathcal{D}(T))\boxtimes\mathcal{C}(a).$$
By construction, $X_T\cup C(a)$ is equal to the link complements of $L$ with some meridional sutures. If the two open components of~$T$ are not joined up, each link component of~$L$ carries a single pair of meridional sutures, so $\SFC(X_T\cup C(a))$ agrees with $\CFL(L)$. Otherwise, there is an extra pair of meridional sutures on the single new closed component, which contributes the tensor factor $V$.\\
Since applying $(-\boxtimes\mathcal{C}(a))$ to $\mathcal{F}_{\pi_{12}}(\CFTd(T))$ and $\CFTd(T)$ gives the same result, it only remains to compare $\mathcal{F}_\pi(\mathcal{D}(T))$ to $\mathcal{F}_{\pi_{12}}(\CFTd(T))$. We claim that they are identical, provided the underlying Heegaard diagram for $\mathcal{D}(T)$ is obtained from the one for $\CFTd(T)$ by removing the arc for site $c$ and performing the obvious modifications near the tangle ends. Then, the underlying generator sets in sites $a$ and $b$ are the same. Furthermore, the moduli spaces for the algebra elements corresponding to the basic elements in $\Ad_{12}$ are the same by the next lemma, since the underlying domains are obtained from one another by adding/removing the silly $\alpha$-arcs. 
\end{proof}

\begin{lemma}
Let \(D\) and \(D'\) be two domains which only differ in a small region of multiplicity 1 as follows:
$$
{\raisebox{-0.9cm}{
\psset{unit=0.5}
\begin{pspicture}(-4,-0.1)(4,4.6)
\pscustom[fillstyle=solid,fillcolor=lightgray,linewidth=0pt,linecolor=white]{
\psline(-3,4)(-3,0)
\psline(-3,0)(3,0)
\psline(3,0)(3,4)
\psecurve(10,4.5)(3,4)(0,4.5)(-3,4)(-10,4.5)
}
\pscustom[fillstyle=solid,fillcolor=white,linewidth=0pt,linecolor=white]{
\psecurve(0,-1)(-1,0)(0,1.5)(1,0)(0,-1)
\psline(1,0)(1,-1)(-1,-1)(-1,0)
}
\rput(0,2.75){$D$}
\psline[linecolor=red](3,0)(3,4)
\psline[linecolor=red](-3,0)(-3,4)
\psecurve[linecolor=red](0,-1)(-1,0)(0,1.5)(1,0)(0,-1)
\psline[linecolor=darkgreen](-3.5,0)(-0.5,0)
\psline[linecolor=darkgreen](0.5,0)(3.5,0)
\end{pspicture}
}
\longleftrightarrow
\raisebox{-0.9cm}{
\psset{unit=0.5}
\begin{pspicture}(-4,-0.1)(4,4.6)
\pscustom[fillstyle=solid,fillcolor=lightgray,linewidth=0pt,linecolor=white]{
\psline(-3,4)(-3,0)
\psline(-3,0)(3,0)
\psline(3,0)(3,4)
\psecurve(10,4.5)(3,4)(0,4.5)(-3,4)(-10,4.5)
}
\rput(0,2){$D'$}
\psline[linecolor=red](3,0)(3,4)
\psline[linecolor=red](-3,0)(-3,4)
\psline[linecolor=darkgreen](-3.5,0)(3.5,0)
\end{pspicture}
}
}
$$
Then, for a suitable choice of complex structure, \(D\) contributes iff \(D'\) does. 
\end{lemma}
\begin{proof}
This follows from the same arguments as \cite[proposition~2.7]{Hanselman}.
\end{proof}

\begin{Remark}\label{rem:BSAx2}
In the proof above, we could alternatively have also added two meridional sutures to the manifold, thereby making the Heegaard diagram planar and the computation purely combinatorial. This comes at the expense of getting an extra tensor factor after the pairing. However, this factor can be removed, since the type A structure is homotopic to a direct sum of two identical copies of $\mathcal{C}(s)$. This computation is done in the  notebook~\cite{BSAx2.nb}, see also appendix~\ref{app:manualBSFH}, in particular figure~\ref{fig:pictures}. For the general glueing formula, one might hope something similar would happen. However, this appears not to be the case. 
\end{Remark}

\begin{theorem}\label{thm:CFTdGeneralGlueing}
	Let \(T_1\) and \(T_2\) be two 4-ended tangles and \(L\) the link obtained by glueing them together according to the following picture
	$$
	\psset{unit=0.7}
	\begin{pspicture}(-2.5,-1.6)(2.5,1.6)
	\pscurve(-1.5,0)(-0.7,1)(2.3,1)(1.5,0)
	\pscircle*[linecolor=white](0,1.2){0.25}
	\pscurve(-1.5,0)(-2.3,1)(0.7,1)(1.5,0)
	
	\pscurve(-1.5,0)(-0.7,-1)(2.3,-1)(1.5,0)
	\pscircle*[linecolor=white](0,-1.2){0.25}
	\pscurve(-1.5,0)(-2.3,-1)(0.7,-1)(1.5,0)
	
	\pscircle*[linecolor=white](-1.5,0){1}
	\pscircle(-1.5,0){1}
	\rput[c](-1.5,0){$T_1$}
	
	\pscircle*[linecolor=white](1.5,0){1}
	\pscircle(1.5,0){1}
	\rput[c](1.5,0){$T_2$}
	\end{pspicture}
	$$
	Then there exists a strictly unital type AA structure \(\mathcal{P}\) over \(\Ad\) such that
	$$\CFL(L)\otimes V^{i}=\CFTd(T_1)\boxtimes\,\mathcal{P}\boxtimes\,\CFTd(T_2)$$
	where \(i=\vert T_1\vert+\vert T_2\vert-\vert L\vert\in\{2,3\}\). 
\end{theorem}
\begin{Remark}\label{rem:CFTdGlueing}
	The type AA structure $\mathcal{P}$ looks almost like the direct sum of four copies of the complex $\mathcal{Q}$ shown in figure~\ref{fig:TypeAAGlueingOneBlock}. However, between those four identical blocks, there are structure maps which I do not seem to be able to get rid of. They also have a certain symmetry. Schematically, $\mathcal{P}$ looks as follows:
	$$\begin{tikzcd}[row sep=1.2cm, column sep=1.5cm]
	\mathcal{Q}
	\arrow[leftrightarrow]{r}[description]{g}
	\arrow[leftrightarrow]{d}[description]{f}
	\arrow[leftrightarrow]{rd}[description]{h}
	& 
	\mathcal{Q}
	\\
	\mathcal{Q}
	\arrow[leftrightarrow]{r}[description]{g}
	&
	\mathcal{Q}
	\arrow[leftrightarrow]{u}[description]{f}
	\end{tikzcd}$$
	
	We will not define the maps $f$, $g$ or $h$ since they are not particularly enlightening; the interested reader should look at the computation in~\cite{BSAAx4e.nb} for details. The complex $\mathcal{Q}$ itself does not look particularly simple either, meaning that I am at the moment unable to find an interpretation of $\mathcal{Q}$ in terms of the Lagrangian intersection homology as in the case of closing a tangle, see remark~\ref{rem:CFTdClosing} and figure~\ref{fig:ClosingFUK}. Instead, I had expected the type AA structure~$\mathcal{P}'$ from figure~\ref{fig:GlueingCAT}, which \emph{does} admit such an interpretation. Calculations for loop-type peculiar modules suggest that we can also use $\mathcal{P}'$, giving rise to the following conjecture; also compare this to conjecture~\sref{conj:GlueingCFTdFUK}.

	\begin{figure}[t]
		\centering
		\includegraphics[scale=1.45]{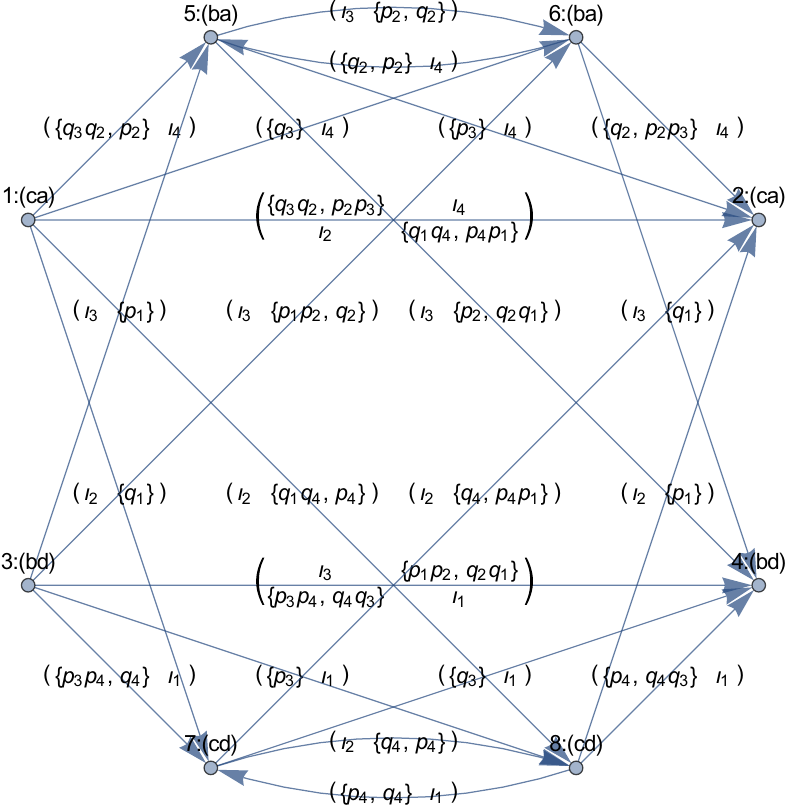}
		\caption{The complex $\mathcal{Q}$, one of four identical components of the type~AA structure~$\mathcal{P}$ from theorem~\ref{thm:CFTdGeneralGlueing}. This is an output of the notebook~\cite{BSAAx4e.nb}, so by remark~II.\ref{rem:conventions3}, we actually need to reverse all arrows and the algebra.}\label{fig:TypeAAGlueingOneBlock}
	\end{figure}
\end{Remark}

\begin{conjecture}\label{conj:CFTdBetterGlueing}
	Theorem~\ref{thm:CFTdGeneralGlueing} remains true, if we replace~\(\mathcal{P}\) by the type~AA structure~\(\mathcal{P}'\) defined by the labelled graph in figure~\ref{fig:GlueingCAT} and set \(i=\vert T_1\vert+\vert T_2\vert-\vert L\vert-2\in\{0,1\}\).
\end{conjecture}

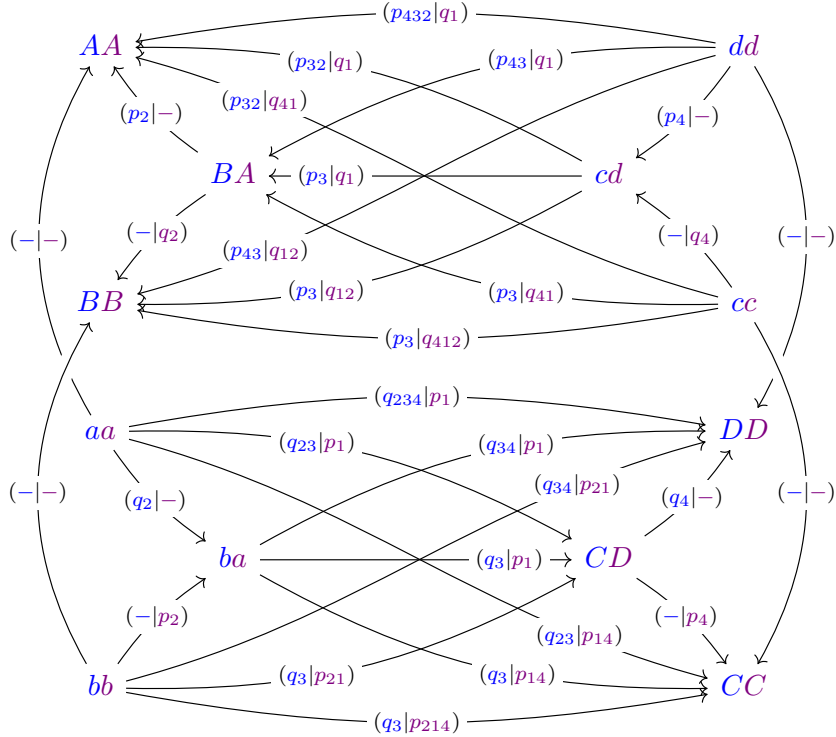
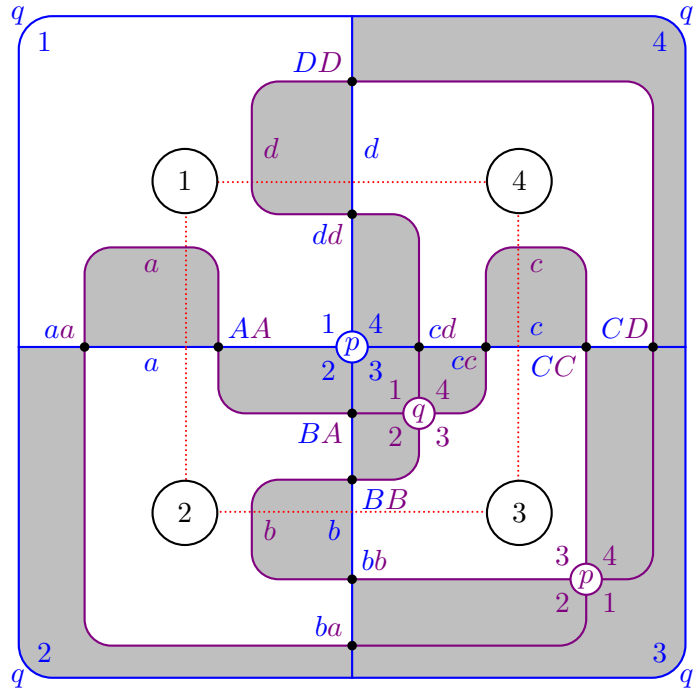
\begin{figure}[p]
	\centering
	{
		\begin{subfigure}[b]{\textwidth}\centering
			$$\begin{tikzcd}[row sep=1.15cm, column sep=0.8cm]
			\blue A\textcolor{violet}{A}
			\arrow[leftarrow,bend left=10]{rrrrrrr}[description]{\pq{432}{1}}
			\arrow[leftarrow,bend left=15,pos=0.4]{rrrrrrd}[description]{\pq{32}{1}}
			\arrow[leftarrow,bend right=10]{rd}[description]{\pn{2}}
			\arrow[leftarrow,bend right=30]{ddd}[description]{\nn}
			&
			&
			&
			&
			&
			&
			&
			\blue d\textcolor{violet}{d}
			\\
			&
			\blue B\textcolor{violet}{A}
			\arrow[leftarrow,bend left=15,pos=0.6]{rrrrrru}[description]{\pq{43}{1}}
			\arrow[leftarrow,pos=0.2]{rrrrr}[description]{\pq{3}{1}}
			\arrow[leftarrow,bend right=15,pos=0.6]{rrrrrrd}[description]{\pq{3}{41}}
			&
			&
			&
			&
			&
			\blue c\textcolor{violet}{d}
			\arrow[leftarrow,bend right=10]{ru}[description]{\pn{4}}
			\arrow[leftarrow,bend left=10]{rd}[description]{\nq{4}}
			\\
			\blue B\textcolor{violet}{B}
			\arrow[leftarrow,bend left=10]{ru}[description]{\nq{2}}
			\arrow[leftarrow,in=-165,out=15,pos=0.2]{rrrrrrruu}[description]{\pq{43}{12}}
			\arrow[leftarrow,bend right=15,pos=0.4]{rrrrrru}[description]{\pq{3}{12}}
			\arrow[leftarrow,bend right=10]{rrrrrrr}[description]{\pq{3}{412}}
			\arrow[leftarrow,crossing over, bend right=30]{ddd}[description]{\nn}
			&
			&
			&
			&
			&
			&
			&
			\blue c\textcolor{violet}{c}
			\arrow[out=165,in=-15,pos=0.8]{uulllllll}[description]{\pq{32}{41}}
			\\
			\blue a\textcolor{violet}{a}
			&
			&
			&
			&
			&
			&
			&
			\blue D\textcolor{violet}{D}
			\arrow[leftarrow,bend right=30]{uuu}[description]{\nn}
			\arrow[leftarrow,bend right=10]{lllllll}[description]{\qp{234}{1}}
			\arrow[leftarrow,bend right=15,pos=0.4]{lllllld}[description]{\qp{34}{1}}
			\arrow[leftarrow,bend left=10]{ld}[description]{\qn{4}}
			\\
			&
			\blue b\textcolor{violet}{a}
			\arrow[leftarrow,bend left=10]{lu}[description]{\qn{2}}
			\arrow[leftarrow,bend right=10]{ld}[description]{\np{2}}
			&
			&
			&
			&
			&
			\blue C\textcolor{violet}{D}
			\arrow[leftarrow,bend right=15,pos=0.6]{llllllu}[description]{\qp{23}{1}}
			\arrow[leftarrow,pos=0.2]{lllll}[description]{\qp{3}{1}}
			\arrow[leftarrow,bend left=15,pos=0.6]{lllllld}[description]{\qp{3}{21}}
			\\
			\blue b\textcolor{violet}{b}
			\arrow[in=-165,out=15,pos=0.8]{rrrrrrruu}[description]{\qp{34}{21}}
			&
			&
			&
			&
			&
			&
			&
			\blue C\textcolor{violet}{C}
			\arrow[leftarrow,crossing over, bend right=30]{uuu}[description]{\nn}
			\arrow[leftarrow,bend right=10]{lu}[description]{\np{4}}
			\arrow[leftarrow,out=165,in=-15,pos=0.2]{llllllluu}[description]{\qp{23}{14}}
			\arrow[leftarrow,bend left=15,pos=0.4]{llllllu}[description]{\qp{3}{14}}
			\arrow[leftarrow,bend left=10]{lllllll}[description]{\qp{3}{214}}
			\end{tikzcd}$$
			\vspace{-0.4cm}
			\caption{The expected type AA structure $\mathcal{P}'$. The identity action is implicit. Part of the actual type~AA structure used in theorem~\ref{thm:CFTdGeneralGlueing} is shown in figure~\ref{fig:TypeAAGlueingOneBlock}, see remark~\ref{rem:CFTdGlueing}.}\label{fig:CombGlueingGraph}
		\end{subfigure}
		\begin{subfigure}[b]{\textwidth}\centering
			\psset{unit=0.44}
			\begin{pspicture}(-10.5,-10.5)(10.5,10.5)
			
			\pscustom[fillstyle=solid,fillcolor=lightgray,linewidth=0pt,linecolor=white]{
				\psline(-8,0)(-4,0)
				\psline[liftpen=1,linearc=0.8](-4,0)(-4,3)(-8,3)(-8,0)
			}
			\pscustom[fillstyle=solid,fillcolor=lightgray,linewidth=0pt,linecolor=white]{
				\psline(0,8)(0,4)
				\psline[liftpen=1,linearc=0.8](0,4)(-3,4)(-3,8)(0,8)
			}
			\pscustom[fillstyle=solid,fillcolor=lightgray,linewidth=0pt,linecolor=white]{
				\psline(-4,0)(4,0)
				\psline[liftpen=1,linearc=0.8](4,0)(4,-2)(-4,-2)(-4,0)
			}
			\pscustom[fillstyle=solid,fillcolor=lightgray,linewidth=0pt,linecolor=white]{
				\psline(0,-4)(0,4)
				\psline[liftpen=1,linearc=0.8](0,4)(2,4)(2,-4)(0,-4)
			}
			\pscustom[fillstyle=solid,fillcolor=lightgray,linewidth=0pt,linecolor=white]{
				\psline(4,0)(7,0)
				\psline[liftpen=1,linearc=0.8](4,0)(4,3)(7,3)(7,0)
			}
			\pscustom[fillstyle=solid,fillcolor=lightgray,linewidth=0pt,linecolor=white]{
				\psline(0,-7)(0,-4)
				\psline[liftpen=1,linearc=0.8](0,-4)(-3,-4)(-3,-7)(0,-7)
			}
			\pscustom[fillstyle=solid,fillcolor=lightgray,linewidth=0pt,linecolor=white]{
				\psline[liftpen=1,linearc=1](-10,0)(-10,-10)(10,-10)(10,10)(0,10)
				\psline(0,10)(0,8)
				\psline[liftpen=1,linearc=0.8](0,8)(9,8)(9,0)
				\psline(9,0)(7,0)(7,-7)(0,-7)(0,-9)
				\psline[liftpen=1,linearc=0.8](0,-9)(-8,-9)(-8,0)
				\psline(-8,0)(-10,0)
			}
			
			\psframe[linecolor=blue,linearc=1,cornersize=absolute](-10,-10)(10,10)
			\psline[linecolor=blue](0,10)(0,-10)
			\psline[linecolor=blue](10,0)(-10,0)
			
			\psframe[linecolor=red,linestyle=dotted,dotsep=1pt](-5,-5)(5,5)
			\pscircle[fillstyle=solid](-5,-5){1}
			\pscircle[fillstyle=solid](5,-5){1}
			\pscircle[fillstyle=solid](-5,5){1}
			\pscircle[fillstyle=solid](5,5){1}
			
			\rput(-5,5){$1$}
			\rput(-5,-5){$2$}
			\rput(5,-5){$3$}
			\rput(5,5){$4$}
			
			\psline[linecolor=violet,linearc=0.8,cornersize=absolute]
			(7,-7)(-3,-7)(-3,-4)(2,-4)(2,4)(-3,4)(-3,8)(9,8)(9,-7)(7,-7)
			\psline[linecolor=violet,linearc=0.8,cornersize=absolute]
			(7,-7)(7,3)(4,3)(4,-2)(-4,-2)(-4,3)(-8,3)(-8,-9)(7,-9)(7,-7)
			
			\pscircle[fillstyle=solid,linecolor=blue, fillcolor=white](0,0){0.5}
			\rput(0,0){\blue $p$}
			\rput(1;135){\blue $1$}
			\rput(1;-135){\blue $2$}
			\rput(1;-45){\blue $3$}
			\rput(1;45){\blue $4$}
			
			\rput(-10,10){\blue $q$}
			\rput(10,-10){\blue $q$}
			\rput(10,10){\blue $q$}
			\rput(-10,-10){\blue $q$}
			\rput(9.2,9.2){\blue $4$}
			\rput(9.2,-9.2){\blue $3$}
			\rput(-9.2,-9.2){\blue $2$}
			\rput(-9.2,9.2){\blue $1$}
			
			\pscircle[fillstyle=solid,linecolor=violet, fillcolor=white](2,-2){0.5}
			\rput(2,-2){
				\rput(0,0){\textcolor{violet}{$q$}}
				\rput(1;135){\textcolor{violet}{$1$}}
				\rput(1;-135){\textcolor{violet}{$2$}}
				\rput(1;-45){\textcolor{violet}{$3$}}
				\rput(1;45){\textcolor{violet}{$4$}}}
			
			\pscircle[fillstyle=solid,linecolor=violet, fillcolor=white](7,-7){0.5}
			\rput(7,-7){
				\rput(0,0){\textcolor{violet}{$p$}}
				\rput(1;135){\textcolor{violet}{$3$}}
				\rput(1;-135){\textcolor{violet}{$2$}}
				\rput(1;-45){\textcolor{violet}{$1$}}
				\rput(1;45){\textcolor{violet}{$4$}}}
			
			\uput{0.35}[-90](-6,0){\blue $a$}
			\uput{0.35}[180](0,-5.5){\blue $b$}
			\uput{0.35}[90](5.5,0){\blue $c$}
			\uput{0.35}[0](0,6){\blue $d$}
			
			\uput{0.35}[-90](-6,3){\textcolor{violet}{$a$}}
			\uput{0.35}[0](-3,-5.5){\textcolor{violet}{$b$}}
			\uput{0.35}[-90](5.5,3){\textcolor{violet}{$c$}}
			\uput{0.35}[0](-3,6){\textcolor{violet}{$d$}}
			
			\psdot(-4,0)
			\uput{0.4}[45](-4,0){\blue $A$\textcolor{violet}{$A$}}
			\psdot(-8,0)
			\uput{0.4}[135](-8,0){\blue $a$\textcolor{violet}{$a$}}
			
			\psdot(0,-4)
			\uput{0.4}[-45](0,-4){\blue $B$\textcolor{violet}{$B$}}
			\psdot(0,-7)
			\uput{0.4}[45](0,-7){\blue $b$\textcolor{violet}{$b$}}
			\psdot(0,-2)
			\uput{0.4}[-135](0,-2){\blue $B$\textcolor{violet}{$A$}}
			\psdot(0,-9)
			\uput{0.4}[135](0,-9){\blue $b$\textcolor{violet}{$a$}}
			
			\psdot(7,0)
			\uput{0.4}[-135](7,0){\blue $C$\textcolor{violet}{$C$}}
			\psdot(4,0)
			\uput{0.4}[-135](4,0){\blue $c$\textcolor{violet}{$c$}}
			\psdot(9,0)
			\uput{0.3}[130](9,0){\blue $C$\textcolor{violet}{$D$}}
			\psdot(2,0)
			\uput{0.4}[45](2,0){\blue $c$\textcolor{violet}{$d$}}
			
			\psdot(0,8)
			\uput{0.4}[135](0,8){\blue $D$\textcolor{violet}{$D$}}
			\psdot(0,4)
			\uput{0.4}[-135](0,4){\blue $d$\textcolor{violet}{$d$}}
			
			\end{pspicture}
			\caption{The expected combinatorial model for pairing in the wrapped Fukaya category of the 4-punctured sphere; compare this to figure~\ref{fig:ClosingFUK}. Again, the boundary of the picture is identified to a point. The {\blue blue} curves denote a 1-skeleton and the \textcolor{violet}{violet} ones a Hamiltonian translate thereof, with the labelling $\textcolor{violet}{p}$ and $\textcolor{violet}{q}$ reversed.}\label{fig:GlueingFUK}
		\end{subfigure}
		\caption{The expected glueing structure in conjecture~\ref{conj:CFTdBetterGlueing}}\label{fig:GlueingCAT}
	}
\end{figure} 

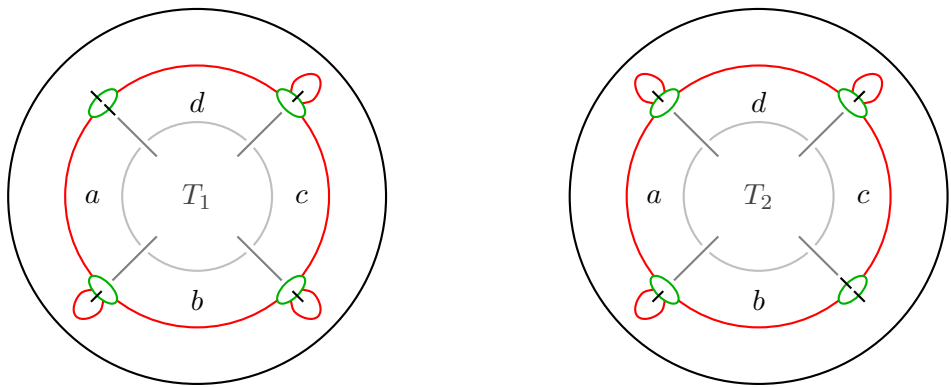
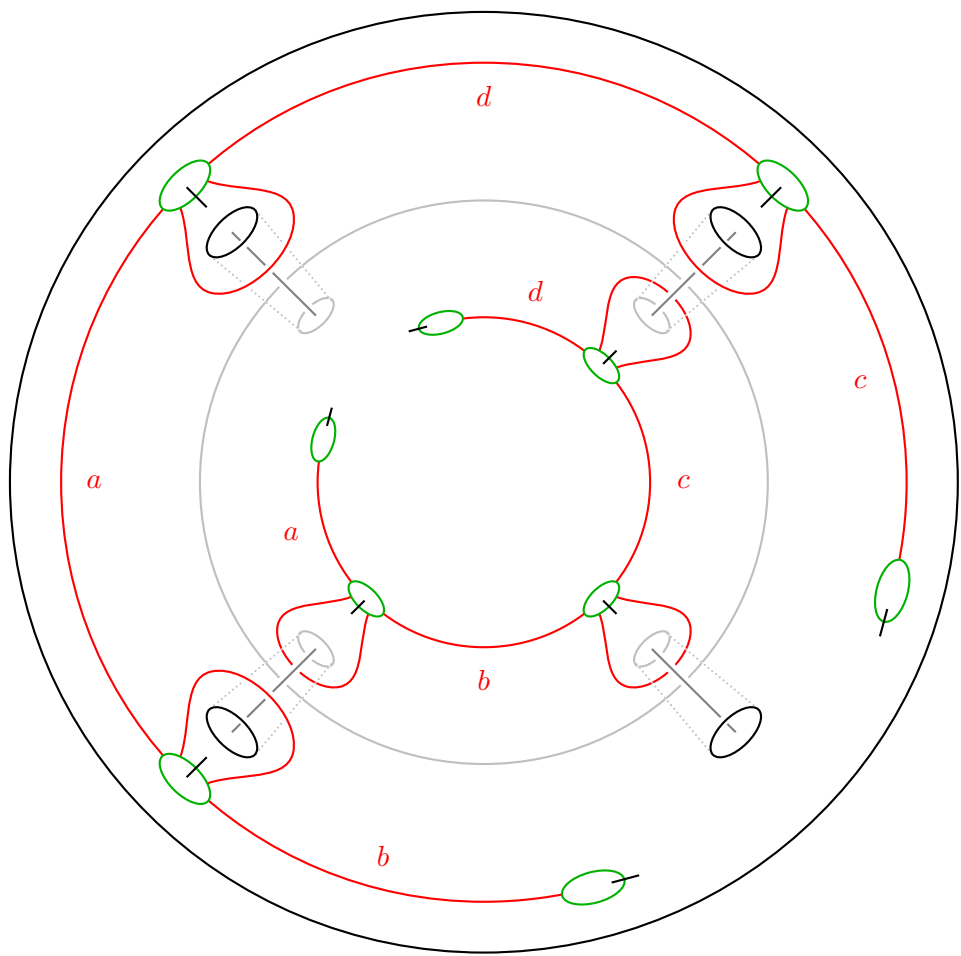
\begin{figure}[p]
\centering
\psset{unit=0.25}
\begin{subfigure}[b]{0.45\textwidth}\centering
\begin{pspicture}(-10,-10)(10,10)
\pscircle(0,0){10}
\pscircle[linecolor=lightgray](0,0){4}

\psline[linecolor=white,linewidth=4pt](3;45)(7;45)
\psline[linecolor=white,linewidth=4pt](3;135)(7;135)
\psline[linecolor=white,linewidth=4pt](3;-135)(7;-135)
\psline[linecolor=white,linewidth=4pt](3;-45)(7;-45)

\psline[linecolor=gray](3;45)(6.3;45)
\psline[linecolor=gray](3;135)(5.9;135)
\psline[linecolor=gray](3;-135)(6.3;-135)
\psline[linecolor=gray](3;-45)(6.3;-45)

\pscircle[linecolor=red](0,0){7}

\pscurve[linecolor=red](7;-135)(8;-130)(9;-135)(8;-140)(7;-135)
\pscurve[linecolor=red](7;45)(8;40)(9;45)(8;50)(7;45)
\pscurve[linecolor=red](7;-45)(8;-40)(9;-45)(8;-50)(7;-45)

{\psset{fillstyle=solid,fillcolor=white}
\rput{45}(7;45){\psellipse[linecolor=darkgreen](0,0)(0.5,1)}
\rput{135}(7;135){\psellipse[linecolor=darkgreen](0,0)(0.5,1)}
\rput{-135}(7;-135){\psellipse[linecolor=darkgreen](0,0)(0.5,1)}
\rput{-45}(7;-45){\psellipse[linecolor=darkgreen](0,0)(0.5,1)}
}

\psline(7.1;45)(7.9;45)
\psline(7.1;-45)(7.9;-45)
\psline(7.1;-135)(7.9;-135)
\psline(7.1;135)(7.9;135)
\psline(6.9;135)(6.1;135)

\rput(-5.5,0){$a$}
\rput(0,-5.5){$b$}
\rput(5.5,0){$c$}
\rput(0,5){$d$}
\rput(0,0){\textcolor{darkgray}{$T_1$}}

\end{pspicture}
\caption{The bordered sutured structure on $X_{T_1}$}\label{fig:GlueingTangleT1}
\end{subfigure}
\quad
\begin{subfigure}[b]{0.45\textwidth}\centering
\begin{pspicture}(-10,-10)(10,10)
\pscircle(0,0){10}
\pscircle[linecolor=lightgray](0,0){4}

\psline[linecolor=white,linewidth=4pt](3;45)(7;45)
\psline[linecolor=white,linewidth=4pt](3;135)(7;135)
\psline[linecolor=white,linewidth=4pt](3;-135)(7;-135)
\psline[linecolor=white,linewidth=4pt](3;-45)(7;-45)

\psline[linecolor=gray](3;45)(6.3;45)
\psline[linecolor=gray](3;135)(6.3;135)
\psline[linecolor=gray](3;-135)(6.3;-135)
\psline[linecolor=gray](3;-45)(5.9;-45)

\pscircle[linecolor=red](0,0){7}

\pscurve[linecolor=red](7;135)(8;130)(9;135)(8;140)(7;135)
\pscurve[linecolor=red](7;-135)(8;-130)(9;-135)(8;-140)(7;-135)
\pscurve[linecolor=red](7;45)(8;40)(9;45)(8;50)(7;45)

{\psset{fillstyle=solid,fillcolor=white}
\rput{45}(7;45){\psellipse[linecolor=darkgreen](0,0)(0.5,1)}
\rput{-45}(7;-45){\psellipse[linecolor=darkgreen](0,0)(0.5,1)}
\rput{-135}(7;-135){\psellipse[linecolor=darkgreen](0,0)(0.5,1)}
\rput{135}(7;135){\psellipse[linecolor=darkgreen](0,0)(0.5,1)}
}
\psline(6.9;-45)(6.1;-45)
\psline(7.1;-45)(7.9;-45)
\psline(7.1;45)(7.9;45)
\psline(7.1;135)(7.9;135)
\psline(7.1;-135)(7.9;-135)

\rput(-5.5,0){$a$}
\rput(0,-5.5){$b$}
\rput(5.5,0){$c$}
\rput(0,5){$d$}
\rput(0,0){\textcolor{darkgray}{$T_2$}}

\end{pspicture}
\caption{The bordered sutured structure on $X_{T_2}$}\label{fig:GlueingTangleT2}
\end{subfigure}
\\
\begin{subfigure}[b]{0.9\textwidth}\centering
\psset{unit=2.5}
\begin{pspicture}(-10.2,-10.2)(10.2,10.4)
\pscircle(0,0){10}

\pscircle[linecolor=lightgray](0,0){6}

\psecurve[linecolor=red](3;48)(3.5;48)(5.3;55)(5.3;35)(3.5;42)(3;42)
\psecurve[linecolor=red](3;-48)(3.5;-48)(5.3;-55)(5.3;-35)(3.5;-42)(3;-42)
\psecurve[linecolor=red](3;-138)(3.5;-138)(5.3;-145)(5.3;-125)(3.5;-132)(3;-132)

\psarc[linecolor=red](0,0){3.5}{165}{105}

\psarc[linecolor=red](0,0){8.9}{-15}{-75}

{\psset{fillstyle=solid,fillcolor=white}

\rput{-135}(3.5;-135){\psellipse[linecolor=darkgreen](0,0)(0.25,0.5)}
\rput{105}(3.5;105){\psellipse[linecolor=darkgreen](0,0)(0.25,0.5)\psline(0,0.3)(0,0.7)}
\rput{165}(3.5;165){\psellipse[linecolor=darkgreen](0,0)(0.25,0.5)\psline(0,-0.3)(0,-0.7)}
\rput{45}(3.5;45){\psellipse[linecolor=darkgreen](0,0)(0.25,0.5)}
\rput{-45}(3.5;-45){\psellipse[linecolor=darkgreen](0,0)(0.25,0.5)}

\rput{-135}(5;-135){\psellipse[linecolor=lightgray](0,0)(0.25,0.5)}
\rput{135}(5;135){\psellipse[linecolor=lightgray](0,0)(0.25,0.5)}
\rput{45}(5;45){\psellipse[linecolor=lightgray](0,0)(0.25,0.5)}
\rput{-45}(5;-45){\psellipse[linecolor=lightgray](0,0)(0.25,0.5)}

\psline[linecolor=white,linewidth=4pt](5;45)(7.5;45)
\psline[linecolor=gray](5;45)(7.05;45)
\psline[linecolor=lightgray,linestyle=dotted,dotsep=1pt](5;50.2)(7.5;50.2)
\psline[linecolor=lightgray,linestyle=dotted,dotsep=1pt](5;39.8)(7.5;39.8)

\psline[linecolor=white,linewidth=4pt](5;-45)(7.5;-45)
\psline[linecolor=gray](5;-45)(7.05;-45)
\psline[linecolor=lightgray,linestyle=dotted,dotsep=1pt](5;-50.2)(7.5;-50.2)
\psline[linecolor=lightgray,linestyle=dotted,dotsep=1pt](5;-39.8)(7.5;-39.8)

\psline[linecolor=white,linewidth=4pt](5;135)(7.5;135)
\psline[linecolor=gray](5;135)(7.05;135)
\psline[linecolor=lightgray,linestyle=dotted,dotsep=1pt](5;140.2)(7.5;140.2)
\psline[linecolor=lightgray,linestyle=dotted,dotsep=1pt](5;129.8)(7.5;129.8)

\psline[linecolor=white,linewidth=4pt](5;-135)(7.5;-135)
\psline[linecolor=gray](5;-135)(7.05;-135)
\psline[linecolor=lightgray,linestyle=dotted,dotsep=1pt](5;-140.2)(7.5;-140.2)
\psline[linecolor=lightgray,linestyle=dotted,dotsep=1pt](5;-129.8)(7.5;-129.8)

}

\psecurve[linecolor=white,linewidth=4pt](9.4;47)(8.9;47)(7;55)(7;35)(8.9;43)(9.4;43)
\psecurve[linecolor=white,linewidth=4pt](9.4;137)(8.9;137)(7;145)(7;125)(8.9;133)(9.4;133)
\psecurve[linecolor=white,linewidth=4pt](9.4;-137)(8.9;-137)(7;-145)(7;-125)(8.9;-133)(9.4;-133)

\psecurve[linecolor=red](9.4;47)(8.9;47)(7;55)(7;35)(8.9;43)(9.4;43)
\psecurve[linecolor=red](9.4;137)(8.9;137)(7;145)(7;125)(8.9;133)(9.4;133)
\psecurve[linecolor=red](9.4;-137)(8.9;-137)(7;-145)(7;-125)(8.9;-133)(9.4;-133)

{\psset{fillstyle=solid,fillcolor=white}

\rput{-135}(7.5;-135){\psellipse(0,0)(0.35,0.7)}
\rput{135}(7.5;135){\psellipse(0,0)(0.35,0.7)}
\rput{45}(7.5;45){\psellipse(0,0)(0.35,0.7)}
\rput{-45}(7.5;-45){\psellipse(0,0)(0.35,0.7)}

\rput{45}(8.9;45){\psellipse[linecolor=darkgreen](0,0)(0.35,0.7)}
\rput{-15}(8.9;-15){\psellipse[linecolor=darkgreen](0,0)(0.35,0.7)\psline(0,-0.4)(0,-1)}
\rput{-75}(8.9;-75){\psellipse[linecolor=darkgreen](0,0)(0.35,0.7)\psline(0,0.4)(0,1)}
\rput{-135}(8.9;-135){\psellipse[linecolor=darkgreen](0,0)(0.35,0.7)}
\rput{135}(8.9;135){\psellipse[linecolor=darkgreen](0,0)(0.35,0.7)}

}

\psline[linecolor=gray](7.25;-45)(7.5;-45)
\psline[linecolor=gray](7.25;45)(7.5;45)
\psline[linecolor=gray](7.25;135)(7.5;135)
\psline[linecolor=gray](7.25;-135)(7.5;-135)

\psline(3.55;45)(3.95;45)
\psline(3.55;-135)(3.95;-135)
\psline(3.55;-45)(3.95;-45)

\psline(8.85;45)(8.25;45)
\psline(8.85;135)(8.25;135)
\psline(8.85;-135)(8.25;-135)

\rput(4.2;195){$\red a$}
\rput(0,-4.2){$\red b$}
\rput(4.2,0){$\red c$}
\rput(4.2;75){$\red d$}

\rput(-8.2,0){$\red a$}
\rput(8.2;-105){$\red b$}
\rput(8.2;15){$\red c$}
\rput(0,8.2){$\red d$}

\end{pspicture}
\caption{The bordered sutured manifold $P$}\label{fig:GlueingAAP}
\end{subfigure}
\caption{Pairing the bordered sutured manifold for~$T_1$ to the inside of~$P$ and the one for~$T_2$ to the outside gives the sutured manifold of the link obtained by pairing~$T_1$ and~$T_2$ according to the picture in theorem~\ref{thm:CFTdGeneralGlueing}, plus some extra pairs of meridional sutures. There are three additional pairs of sutures if the four open components of~$T_1$ and~$T_2$ glue up to a single closed component; otherwise, there are only two superfluous meridional suture pairs.}\label{fig:GlueingTangleP}
\end{figure} 

\begin{proof}[Proof of theorem~\ref{thm:CFTdGeneralGlueing}]
	We follow the same line of argument as in the proof of theorem~\ref{thm:CFTdGlueingTrivial}, except that we let the computer do the main calculation, using the Mathematica notebook~\cite{BSAAx4e.nb}. We glue two bordered sutured manifolds associated with the tangles~$T_1$ and~$T_2$ together along a third bordered sutured manifold~$P$, which is topologically just a thickened 4-punctured sphere as illustrated in figure~\ref{fig:GlueingTangleP}. The entire calculation of the bordered sutured type AA structure for $P$, including cancellations and homotopies, are done in~\cite{BSAAx4e.nb} -- most of them in some automated fashion.\\
	Just as for closing a tangle, the main point of this proof is that we can homotope the type AA structure for~$P$, which is defined over some complicated bordered sutured algebras~$\mathcal{B}_1$ and~$\mathcal{B}_2$, to one where we can replace~$\mathcal{B}_1$ and~$\mathcal{B}_2$ by some suitable quotient of $\Ad$ without changing the result under any glueing. We quotient out by~$p_1$ and~$q_1$ for the first glueing surface, i.\,e.\ the one for~$T_1$ on the ``inside'' of~$P$ in figure~\ref{fig:GlueingAAP}. For the second glueing surface, we quotient out by~$p_3$ and~$q_3$. 
	Then, by using Heegaard diagrams for~$X_{T_1}$ and~$X_{T_2}$, where the silly arcs do not intersect any $\beta$-curves, the same argument as in the proof of theorem~\ref{thm:CFTdGlueingTrivial} tells us that
	$$
	\SFC(X_{T_1}\cup P\cup X_{T_2})=\CFTd(T_1)\boxtimes \mathcal{P}\boxtimes \CFTd(T_2).
	$$
	Now we just need to count the number of meridional sutures in $X_{T_1}\cup P\cup X_{T_2}$ to determine how often we need to stabilise the link Floer homology of $L$.
\end{proof}

\begin{theorem}\label{thm:CFTdDetectsRatTan}
A 4-ended tangle \(T\) is rational iff \(\CFTd(T)\) is homotopic to a single loop that corresponds to an embedded loop on the 4-punctured sphere.
\end{theorem}
\begin{proof}
The only-if direction is simply a calculation that we did in example~\ref{exa:CFTdRatTang}. The opposite direction follows essentially from the fact that link Floer homology detects the knot genus and a cut-and-paste argument. \\
Suppose $\CFTd(T)$ is a single loop which corresponds to an embedded loop on the 4-punctured sphere. It divides the sphere into two disc components, each of which has at least one puncture, since the loop is not nullhomotopic. By observation~\ref{obs:AlexGradingOfCFTdLoops}, there are exactly two punctures in each disc, so $\CFTd(T)$ agrees with $\CFTd(T')$ for some rational tangle $T'$. Then also $\CFTd(\m(T))$ agrees with $\CFTd(\m(T'))$. Let $L$ be the link obtained by pairing $T$ with its mirror $\m(T)$. If we glue $T'$ to $\m(T')$, we obtain the 2-component unlink $\bigcirc\amalg\bigcirc$, so by theorem~\ref{thm:CFTdGeneralGlueing}
\begin{align*}
	\HFL(L)\otimes V^{\otimes 2}&=\CFTd(T)\boxtimes\mathcal{P}\boxtimes\CFTd(\m(T))\\
	&=\CFTd(T')\boxtimes\mathcal{P}\boxtimes\CFTd(\m(T'))=\HFL(\bigcirc\amalg\bigcirc)\otimes V^{\otimes 2}.
\end{align*}
We now apply the fact that link Floer homology detects the link genus \cite{OSHFLThurston}, so $L$ is the 2-component unlink. The following lemma finishes the proof.
\end{proof}
\begin{lemma}\label{lem:RatTanDet}
Let \(T_1\) and \(T_2\) be two 4-ended tangles that glue together to the 2-component unlink. Then either \(T_1\) or \(T_2\) is a rational tangle. 
\end{lemma}
\begin{proof}
Let $S$ be the 4-punctured sphere along which we glue $T_1$ and $T_2$. Let $U$ be the sphere that separates the two unknot components and assume that $S$ and $U$ intersect transversely in a disjoint union of circles. We now proceed by induction on the number of circles in $S\cap U$. First of all, this intersection is non-empty, since $U$ is separating. So we can always find a curve $\gamma$ that bounds a disc $D$ in $U$. If $\gamma$ also bounds a disc $D'$ in $S$, $D\cup D'$ bounds a 3-ball, which we can use as a homotopy for $U$ to remove $\gamma$, so we are done by the induction hypothesis. If $\gamma$ does not bound a disc in $S$, it separates two punctures from the other two. So $D$ separates the two strands in $T_1$ or $T_2$. They must obviously be unknotted, since the connected sum of two knots is the unknot iff both knots are unknots. Thus either $T_1$ or $T_2$ is rational.
\end{proof}

%% file: sections/3_SkeinRelations.tex

\section{Skein relations}\label{sec:SkeinRelations}
We start with a slight generalisation of Ozsv\'{a}th and Szab\'{o}'s exact triangle \cite{OSHFK} which categorifies the oriented skein relation for the Alexander polynomial (lemma~\sref{onecolourskein}). Note, however, that we only get a twice stabilised version, due to the shortcomings of our glueing theorem~\ref{thm:CFTdGeneralGlueing}. 
\begin{theorem}[$n$-twist skein exact triangle]\label{thm:nTwistSkeinRelation}
Let \(T_n\) be the positive \(n\)-twist tangle, \(T_{-n}\) the negative \(n\)-twist tangle and \(T_0\) the trivial tangle, see figure~\ref{fig:OrientedSkeinRelation}. Furthermore, let \(V\) be a 2-dimensional vector space supported in degrees \(\delta^0 t^{n}\) and \(\delta^0 t^{-n}\). Then there is an exact triangle
$$\begin{tikzcd}[row sep=0.7cm, column sep=-0.5cm]
\delta^{-\frac{n}{2}}\CFTd(T_{n})
\arrow{rr}{\varphi_n}
& &
\delta^{\frac{n}{2}}\CFTd(T_{-n})
\arrow{dl}
\\
&
\CFTd(T_{0})\otimes V
\arrow{lu}
\end{tikzcd}$$
\(\varphi_n\) preserves the (single-variate) Alexander grading and changes \(\delta\)- and homological gradings by \(+1\) and \(-1\), respectively; the other two maps preserve the all three gradings. Moreover, given three links \(L_{n}\), \(L_{-n}\) and \(L_0\) in \(S^3\), which agree outside a closed 3-ball and in this closed 3-ball agree with the 4-ended tangles \(T_{n}\), \(T_{-n}\) and \(T_{0}\), respectively, then the above triangle together with the glueing theorem induces an exact triangle 
$$\begin{tikzcd}[row sep=0.7cm, column sep=-.6cm]
\HFL(L_{n})\otimes V^{l_{n}}
\arrow{rr}
& &
\HFL(L_{-n})\otimes V^{l_{-n}}
\arrow{dl}
\\
&
\HFL(L_{0})\otimes V^{l_0+1}
\arrow{lu}
\end{tikzcd}$$
where for \(i\in\{n,-n,0\}\), \(l_i\) is either \(2\) or \(3\), depending on whether the two strands in \(T_i\) belong to different or the same components in \(L_i\), respectively. 
\end{theorem}
\begin{figure}[b]
\centering
\psset{unit=0.3}
\begin{subfigure}[b]{0.25\textwidth}\centering
$n\left\{\raisebox{-0.85cm}{
\begin{pspicture}(-1.1,-3)(1.1,3.4)
\psline{->}(-1,3)(-1.05,3.3)
\psline{->}(1,3)(1.05,3.3)
\psecurve(1,5)(-1,3)(1,1)(-1,-1)
\pscircle*[linecolor=white](0,2){0.3}
\psecurve(-1,5)(1,3)(-1,1)(1,-1)
\rput(0,0.3){$\vdots$}
\psecurve(-1,-5)(1,-3)(-1,-1)(1,1)
\pscircle*[linecolor=white](0,-2){0.3}
\psecurve(1,-5)(-1,-3)(1,-1)(-1,1)
\end{pspicture}}\right.\quad
$
\caption{$T_{n}$}\label{fig:OrientedSkeinRelationTn}
\end{subfigure}
\quad
\begin{subfigure}[b]{0.25\textwidth}\centering
$n\left\{\raisebox{-0.85cm}{
\begin{pspicture}(-1.1,-3)(1.1,3.4)
\psline{->}(-1,3)(-1.05,3.3)
\psline{->}(1,3)(1.05,3.3)
\psecurve(-1,5)(1,3)(-1,1)(1,-1)
\pscircle*[linecolor=white](0,2){0.3}
\psecurve(1,5)(-1,3)(1,1)(-1,-1)
\rput(0,0.3){$\vdots$}
\psecurve(1,-5)(-1,-3)(1,-1)(-1,1)
\pscircle*[linecolor=white](0,-2){0.3}
\psecurve(-1,-5)(1,-3)(-1,-1)(1,1)
\end{pspicture}}\right.\quad
$
\caption{$T_{-n}$}\label{fig:OrientedSkeinRelationTmn}
\end{subfigure}
\quad
\begin{subfigure}[b]{0.25\textwidth}\centering
\begin{pspicture}(-5,-3.5)(5,3.5)
\psecurve{<-}(-2,6)(-1,3)(-1,-3)(-2,-6)
\psecurve{<-}(2,6)(1,3)(1,-3)(2,-6)
\end{pspicture}
\caption{$T_0$}\label{fig:OrientedSkeinRelationT0}
\end{subfigure}
\caption{The basic tangles for the skein exact sequence from theorem~\ref{thm:nTwistSkeinRelation}}\label{fig:OrientedSkeinRelation}
\end{figure}
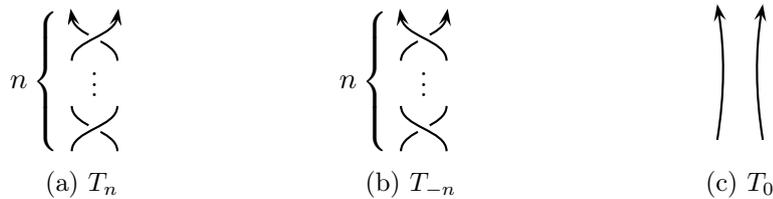
\begin{Remark}\label{rem:nTwistSkeinRelation}
Similar results hold for other orientations; for $n$ even, also multivariate Alexander gradings are preserved.
\end{Remark}

\begin{proof}
It is straightforward to compute $\CFTd(T_{n})$ and $\CFTd(T_{-n})$ from genus 0 Heegaard diagrams, they are both shown in figure~\ref{fig:OrientSkeinBoundaryMap}, the former on the left, the latter on the right. If $n$ is odd, the dashed lines denote a sequence of alternating generators in sites $a$ and $c$, connected by pairs of morphisms, labelled alternatingly by $p_i$s and $q_i$s. For even $n$, the two components are connected by similar sequences along the dotted lines. The horizontal arrows in figure~\ref{fig:OrientSkeinBoundaryMap} describe $\varphi_n$. \\

\begin{figure}[bh!]
\centering
$\begin{tikzcd}[row sep=0.3cm, column sep=1cm]
&
\delta^{-\frac{1}{2}}c^\Red{1-n}
\arrow[pos=0.35]{rrrr}{1}
\arrow[-, dotted,in=90,out=-90]{dddddddl}
\arrow[-, dashed,in=90,out=-90]{dddd}
&
&
&
&
\delta^{\frac{1}{2}}c^\Red{1-n}
\arrow[-, dashed,in=90,out=-90]{dddd}
\\
\delta^0 b^\Red{-n}
\arrow[crossing over,pos=0.65]{rrrr}{p_{21}+q_{34}}
\arrow[bend left=10,leftarrow]{ru}{p_3}
\arrow[bend right=10,swap,pos=0.55]{ru}{p_{214}}
\arrow[bend left=10,leftarrow]{dd}{q_{2}}
\arrow[bend right=10,swap]{dd}{q_{341}}
&
&
&
&
\delta^0 d^\Red{-n}
\arrow[bend left=10,leftarrow,pos=0.2]{ru}{p_{321}}
\arrow[bend right=10,swap]{ru}{p_{4}}
\arrow[bend left=10,leftarrow]{dd}{q_{234}}
\arrow[bend right=10,swap]{dd}{q_{1}}
\\
&
\phantom{\vdots}
&
&
&
&
\phantom{\vdots}
\\
\delta^{-\frac{1}{2}}a^\Red{1-n}
\arrow[-, dashed,in=90,out=-90]{dddd}
\arrow[crossing over,pos=0.65]{rrrr}{1}
\arrow[-, dotted]{dr}
&
&
&
&
\delta^{\frac{1}{2}}a^\Red{1-n}
\arrow[-, dotted]{dr}
\\
&
\delta^{-\frac{1}{2}}c^\Red{n-1}
\arrow[pos=0.35]{rrrr}{1}
&
&
&
&
\delta^{\frac{1}{2}}c^\Red{n-1}
\\
\phantom{\vdots}
&
&
&
&
\phantom{\vdots}
\\
&
\delta^0 d^\Red{n}
\arrow[bend left=10,leftarrow,pos=0.3]{ld}{p_1}
\arrow[bend right=10,swap,pos=0.3]{ld}{p_{432}}
\arrow[bend left=10,leftarrow]{uu}{q_{4}}
\arrow[bend right=10,swap]{uu}{q_{123}}
\arrow[pos=0.35]{rrrr}{p_{43}+q_{12}}
&
&
&
&
\delta^0 b^\Red{n}
\arrow[bend left=10,leftarrow]{ld}{p_{143}}
\arrow[bend right=10,swap,pos=0.7]{ld}{p_{2}}
\arrow[bend left=10,leftarrow]{uu}{q_{412}}
\arrow[bend right=10,swap]{uu}{q_{3}}
\\
\delta^{-\frac{1}{2}}a^\Red{n-1}
\arrow[pos=0.65]{rrrr}{1}
&
&
&
&
\delta^{\frac{1}{2}}a^\Red{n-1}
\arrow[-, dashed,in=-90,out=90,crossing over]{uuuu}
\arrow[-, dotted,in=-90,out=90]{uuuuuuur}
\end{tikzcd}$
\caption{The morphism $\varphi_n:\delta^{-\frac{n}{2}}\CFTd(T_{n})\rightarrow\delta^{\frac{n}{2}}\CFTd(T_{-n})$}\label{fig:OrientSkeinBoundaryMap}
\end{figure}
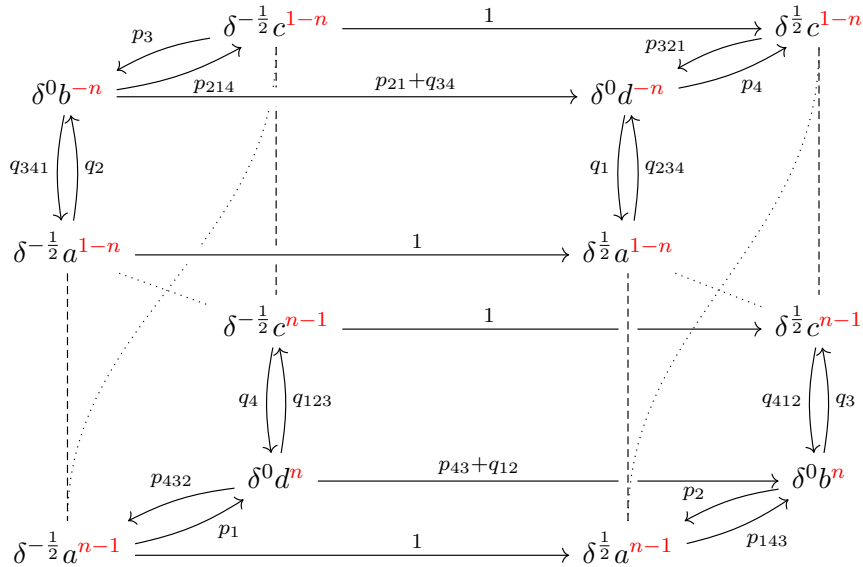

\noindent
By cancelling all identity components of the mapping cone of $\varphi_n$, we get two copies of 
$$\begin{tikzcd}[row sep=0.3cm, column sep=1cm]
\nmathphantom{\CFTd(T_0)=}\CFTd(T_0)=\delta^0b^\Red{0}
\arrow[bend left=10,leftarrow]{r}{p_{43}+q_{12}}
\arrow[bend right=10,swap,pos=0.55]{r}{p_{21}+q_{34}}
&
\delta^0d^\Red{0}.
\end{tikzcd}$$
Thus, we can write $\CFTd(T_0)\otimes V$ as a cone of $\CFTd(T_n)$ and $\CFTd(T_{-n})$, which gives rise to the exact triangle of the required form.
\end{proof}
Next, we give a new proof of a theorem by Manolescu \cite[theorem~1]{Manolescu}. Again, we only get a twice stabilised version. Like Manolescu's triangle, ours does not preserve any gradings. Note that Manolescu probably uses slightly different conventions from ours, because the two triangles only look the same after reversing the direction of the three arrows.
\begin{theorem}[resolution skein exact triangle]\label{thm:ResolutionExactTriangle}
There is an exact triangle
$$\begin{tikzcd}[row sep=0.7cm, column sep=0cm]
\CFTd(\!\!
\raisebox{-5pt}{
\psset{unit=0.2}
\begin{pspicture}(-1.05,-1.45)(1.05,1.05)
\psline(-0.9,-0.9)(0.9,0.9)
\psline(0.9,-0.9)(-0.9,0.9)
\pscircle*[linecolor=white](0,0){0.5}
\psarc(0.7071,0){0.5}{135}{225}
\psarc(-0.7071,0){0.5}{-45}{45}
\end{pspicture}
}
\!\!)
\arrow{rr}{\varphi}
& &
\CFTd(\!\!
\raisebox{-5pt}{
\psset{unit=0.2}
\begin{pspicture}(-1.05,-1.45)(1.05,1.05)
\psline(-0.9,-0.9)(0.9,0.9)
\psline(0.9,-0.9)(-0.9,0.9)
\pscircle*[linecolor=white](0,0){0.5}
\psarc(0,0.7071){0.5}{-135}{-45}
\psarc(0,-0.7071){0.5}{45}{135}
\end{pspicture}
}
\!\!)
\arrow{dl}
\\
&
\CFTd(\!\!
\raisebox{-5pt}{
\psset{unit=0.2}
\begin{pspicture}(-1.05,-1.45)(1.05,1.05)
\psline(-0.9,-0.9)(0.9,0.9)
\pscircle*[linecolor=white](0,0){0.3}
\psline(0.9,-0.9)(-0.9,0.9)
\end{pspicture}
}
\!\!)
\arrow{lu}
\end{tikzcd}$$
Moreover, given three links \(L_{0}\), \(L_{1}\) and \(L_X\) in \(S^3\), which agree outside a closed 3-ball and in this closed 3-ball agree with the 4-ended tangles 
\(T_0=\!\!\raisebox{-5pt}{
\psset{unit=0.2}
\begin{pspicture}(-1.05,-1.45)(1.05,1.05)
\psline(-0.9,-0.9)(0.9,0.9)
\psline(0.9,-0.9)(-0.9,0.9)
\pscircle*[linecolor=white](0,0){0.5}
\psarc(0.7071,0){0.5}{135}{225}
\psarc(-0.7071,0){0.5}{-45}{45}
\end{pspicture}
}\!\!,\) 
\(T_1=\!\!\raisebox{-5pt}{
\psset{unit=0.2}
\begin{pspicture}(-1.05,-1.45)(1.05,1.05)
\psline(-0.9,-0.9)(0.9,0.9)
\psline(0.9,-0.9)(-0.9,0.9)
\pscircle*[linecolor=white](0,0){0.5}
\psarc(0,0.7071){0.5}{-135}{-45}
\psarc(0,-0.7071){0.5}{45}{135}
\end{pspicture}
}\)\!\!
 and  
 \(T_X=
\!\!\raisebox{-5pt}{
\psset{unit=0.2}
\begin{pspicture}(-1.05,-1.45)(1.05,1.05)
\psline(-0.9,-0.9)(0.9,0.9)
\pscircle*[linecolor=white](0,0){0.3}
\psline(0.9,-0.9)(-0.9,0.9)
\end{pspicture}
}\!\!,\)
respectively, then the above triangle, together with the glueing theorem induces an exact triangle 
$$\begin{tikzcd}[row sep=0.7cm, column sep=-.6cm]
\HFL(L_{0})\otimes V^{l_{0}}
\arrow{rr}
& &
\HFL(L_{1})\otimes V^{l_{1}}
\arrow{dl}
\\
&
\HFL(L_{X})\otimes V^{l_X}
\arrow{lu}
\end{tikzcd}$$
where for \(i\in\{0,1,X\}\), \(l_i\) is either \(2\) or \(3\), depending on whether the two strands in \(T_i\) belong to different or the same components in \(L_i\), respectively. 
\end{theorem}

\begin{proof}
The map $\varphi$ is given by (the horizontal arrows in) the following diagram on the left:
$$\begin{tikzcd}[row sep=2cm, column sep=2.8cm]
b
\arrow[bend left=10,leftarrow]{d}{p_{43}+q_{12}}
\arrow[bend right=10,swap]{d}{p_{21}+q_{34}}
\arrow{r}{q_3}
\arrow[pos=0.3]{rd}{p_2}
&
c
\arrow[bend left=10,leftarrow]{d}{p_{14}+q_{23}}
\arrow[bend right=10,swap]{d}{p_{32}+q_{41}}
\\
d
\arrow{r}{q_1}
\arrow[pos=0.3,swap]{ru}{p_4}
&
a
\end{tikzcd}
\quad\cong\quad
\begin{tikzcd}[row sep=2cm, column sep=2.8cm]
b
\arrow[bend left=10,leftarrow,pos=0.8]{dr}{p_{143}}
\arrow[bend right=10,swap,pos=0.2]{dr}{p_{2}}
\arrow[bend left=10,leftarrow]{r}{q_{412}}
\arrow[bend right=10,swap]{r}{q_{3}}
&
c
\\
d
\arrow[bend right=10,swap,leftarrow,pos=0.8]{ru}{p_{321}}
\arrow[bend left=10,pos=0.2]{ru}{p_{4}}
\arrow[bend right=10,swap,leftarrow]{r}{q_{234}}
\arrow[bend left=10]{r}{q_{1}}
&
a
\end{tikzcd}$$
Using lemma~\ref{lem:AbstractCleanUp}, we see that it is homotopic to the diagram on the right,
which is $\CFTd(\!\!\raisebox{-5pt}{
\psset{unit=0.2}
\begin{pspicture}(-1.05,-1.45)(1.05,1.05)
\psline(-0.9,-0.9)(0.9,0.9)
\pscircle*[linecolor=white](0,0){0.3}
\psline(0.9,-0.9)(-0.9,0.9)
\end{pspicture}
}\!\!)$. Now apply the same arguments as in the proof of theorem~\ref{thm:nTwistSkeinRelation}.
\end{proof}
\pagebreak[3]

\begin{proposition}\label{prop:singularcrossing}
There are two morphisms
$$
\delta^{-\frac{1}{2}}t^{\pm1}\CFTd\left(\raisebox{-5pt}{
\psset{unit=0.3}
\begin{pspicture}(-0.91,-0.91)(0.91,0.91)
\psline{->}(-0.9,-0.9)(0.9,0.9)
\psline{->}(0.9,-0.9)(-0.9,0.9)
\pscircle*[linecolor=white](0,0){0.5}
\psarc(0.7071,0){0.5}{135}{225}
\psarc(-0.7071,0){0.5}{-45}{45}
\end{pspicture}
}\right)\rightarrow
\CFTd\left(\raisebox{-5pt}{
\psset{unit=0.3}
\begin{pspicture}(-0.91,-0.91)(0.91,0.91)
\psline{->}(0.9,-0.9)(-0.9,0.9)
\pscircle*[linecolor=white](0,0){0.3}
\psline{->}(-0.9,-0.9)(0.9,0.9)
\end{pspicture}
}\right)
$$
whose mapping cones are homotopic to loop-type complexes, representing the following ``figure-8'' loops:\vspace*{-0.5em}
$$
{\psset{unit=1.5}
\raisebox{-2.2cm}{
\begin{pspicture}(-1.5,-1.5)(1.5,1.5)
\psrotate(0,0){180}{
\psecurve(-1.2,0.6)(-1.2,1.2)(1.2,0.6)(1.2,1.2)(-1.2,0.6)(-1.2,1.2)(1.2,0.6)

\psline*[linecolor=white](1,1)(1,-1)(-1,-1)(-1,1)
\psline[linestyle=dashed](1,1)(1,-1)
\psline[linestyle=dashed](1,-1)(-1,-1)
\psline[linestyle=dashed](-1,-1)(-1,1)
\psline[linestyle=dashed](-1,1)(1,1)

\psecurve[dotsep=1pt,linestyle=dotted](-1.2,0.6)(-1.2,1.2)(1.2,0.6)(1.2,1.2)(-1.2,0.6)(-1.2,1.2)(1.2,0.6)

\psset{dotsize=5pt}

\psdot[linecolor=darkgreen](-1,0.507)

\psdot[linecolor=blue](0.155,1)
\psdot[linecolor=blue](-0.155,1)

\psdot[linecolor=red](1,0.507)

\pscircle[fillstyle=solid, fillcolor=white](1,1){0.08}
\pscircle[fillstyle=solid, fillcolor=white](-1,1){0.08}
\pscircle[fillstyle=solid, fillcolor=white](1,-1){0.08}
\pscircle[fillstyle=solid, fillcolor=white](-1,-1){0.08}
}
\end{pspicture}
}
\text{ and }
\raisebox{-2.2cm}{
\begin{pspicture}(-1.5,-1.5)(1.5,1.5)

\psecurve(-1.2,0.6)(-1.2,1.2)(1.2,0.6)(1.2,1.2)(-1.2,0.6)(-1.2,1.2)(1.2,0.6)

\psline*[linecolor=white](1,1)(1,-1)(-1,-1)(-1,1)
\psline[linestyle=dashed](1,1)(1,-1)
\psline[linestyle=dashed](1,-1)(-1,-1)
\psline[linestyle=dashed](-1,-1)(-1,1)
\psline[linestyle=dashed](-1,1)(1,1)

\psecurve[dotsep=1pt,linestyle=dotted](-1.2,0.6)(-1.2,1.2)(1.2,0.6)(1.2,1.2)(-1.2,0.6)(-1.2,1.2)(1.2,0.6)

\psset{dotsize=5pt}

\psdot[linecolor=red](-1,0.507)

\psdot[linecolor=gold](0.155,1)
\psdot[linecolor=gold](-0.155,1)

\psdot[linecolor=darkgreen](1,0.507)

\pscircle[fillstyle=solid, fillcolor=white](1,1){0.08}
\pscircle[fillstyle=solid, fillcolor=white](-1,1){0.08}
\pscircle[fillstyle=solid, fillcolor=white](1,-1){0.08}
\pscircle[fillstyle=solid, fillcolor=white](-1,-1){0.08}
\end{pspicture}
}
}
$$
By the symmetry of these loops, taking the mirror gives us these loops as the mapping cone of maps from the negative crossing to the trivial tangle. 
\end{proposition}
\begin{proof}
The two maps between the peculiar invariants of the two tangles look as follows:
$$
\begin{tikzcd}[row sep=0.5cm, column sep=0.22cm,crossing over clearance=3pt]
&
\delta^{-\frac{1}{2}}d^{\red +1}
\arrow[out=0,in=180]{rrrr}{1}
&&&&
\delta^{\frac{1}{2}}d^{\red +1}
\arrow[bend left=10,leftarrow,pos=0.5]{dd}{q_4}
\arrow[bend right=10,swap,pos=0.6]{dd}{q_{123}}
\\
&&&&
\delta^{0}a^{\red 0}
\arrow[crossing over, bend left=10,leftarrow,pos=0.35]{ur}{p_{432}}
\arrow[crossing over, bend right=10,swap,pos=0.4]{ur}{p_1}
\\
&&&&&
\delta^{0}c^{\red 0}
\\
\delta^{-\frac{1}{2}}b^{\red +1}
\arrow[out=30,in=180,pos=0.3]{rrrruu}{p_2}
\arrow[swap,out=20,in=180,pos=0.5]{rrrrru}{q_3}
\arrow[bend left=10,leftarrow,pos=0.5]{uuur}{p_{43}+q_{12}}
\arrow[bend right=10,swap,pos=0.85]{uuur}{p_{21}+q_{34}}
&&&&
\delta^{\frac{1}{2}}b^{\red -1}
\arrow[bend left=10,leftarrow,pos=0.65]{ur}{p_3}
\arrow[bend right=10,swap,pos=0.65]{ur}{p_{214}}
\arrow[crossing over, bend left=10,leftarrow,pos=0.55]{uu}{q_2}
\arrow[crossing over, bend right=10,swap,pos=0.75]{uu}{q_{341}}
\end{tikzcd}
\text{ and }
\begin{tikzcd}[row sep=0.5cm, column sep=0.22cm,crossing over clearance=3pt]
&
\delta^{-\frac{1}{2}}d^{\red -1}
\arrow[out=-20,in=180,pos=0.2]{rrrrdd}{p_4}
\arrow[crossing over, out=-30,in=180,swap,pos=0.3]{rrrd}{q_1}
&&&&
\delta^{\frac{1}{2}}d^{\red +1}
\arrow[bend left=10,leftarrow,pos=0.5]{dd}{p_4}
\arrow[bend right=10,swap,pos=0.6]{dd}{q_{123}}
\\
&&&&
\delta^{0}a^{\red 0}
\arrow[crossing over, bend left=10,leftarrow,pos=0.35]{ur}{p_{432}}
\arrow[crossing over, bend right=10,swap,pos=0.4]{ur}{p_1}
\\
&&&&&
\delta^{0}c^{\red 0}
\\
\delta^{-\frac{1}{2}}b^{\red -1}
\arrow[out=0,in=180,swap]{rrrr}{1}
\arrow[bend left=10,leftarrow,pos=0.5]{uuur}{p_{43}+q_{12}}
\arrow[bend right=10,swap,pos=0.5]{uuur}{p_{21}+q_{34}}
&&&&
\delta^{\frac{1}{2}}b^{\red -1}
\arrow[bend left=10,leftarrow,pos=0.65]{ur}{p_3}
\arrow[bend right=10,swap,pos=0.65]{ur}{p_{214}}
\arrow[crossing over, bend left=10,leftarrow,pos=0.55]{uu}{q_2}
\arrow[crossing over, bend right=10,swap,pos=0.75]{uu}{q_{341}}
\end{tikzcd}
$$
We now cancel the identity arrows in both mapping cones and respectively get
$$
\begin{tikzcd}[row sep=1.5cm, column sep=1.2cm,crossing over clearance=3pt]
\delta^{0}a^{\red 0}
\arrow[bend left=10,leftarrow,pos=0.5]{d}{q_{341}}
\arrow[bend right=10,swap,pos=0.5]{d}{q_2}
\arrow[bend left=10,leftarrow,pos=0.5]{r}{p_2}
\arrow[bend right=10,swap,pos=0.5]{r}{p_{143}}
&
\delta^{-\frac{1}{2}}b^{\red +1}
\\
\delta^{\frac{1}{2}}b^{\red -1}
&
\delta^{0}c^{\red 0}
\arrow[bend left=10,leftarrow,pos=0.5]{u}{q_{3}}
\arrow[bend right=10,swap,pos=0.5]{u}{q_{412}}
\arrow[bend left=10,leftarrow,pos=0.5]{l}{p_{214}}
\arrow[bend right=10,swap,pos=0.5]{l}{p_3}
\end{tikzcd}
\quad\text{ and }\quad
\begin{tikzcd}[row sep=1.5cm, column sep=1.2cm,crossing over clearance=3pt]
\delta^{0}a^{\red 0}
\arrow[bend left=10,leftarrow,pos=0.5]{d}{q_1}
\arrow[bend right=10,swap,pos=0.5]{d}{q_{234}}
\arrow[bend left=10,leftarrow,pos=0.5]{r}{p_{432}}
\arrow[bend right=10,swap,pos=0.5]{r}{p_1}
&
\delta^{\frac{1}{2}}d^{\red +1}
\\
\delta^{-\frac{1}{2}}d^{\red -1}
&
\delta^{}c^{\red 0}
\arrow[bend left=10,leftarrow,pos=0.5]{u}{q_{123}}
\arrow[bend right=10,swap,pos=0.5]{u}{q_4}
\arrow[bend left=10,leftarrow,pos=0.5]{l}{p_4}
\arrow[bend right=10,swap,pos=0.5]{l}{p_{321}}
\end{tikzcd}
$$
Obviously, these peculiar modules are loop-type. Also, reversing all arrows, swapping $p_i$s and $q_i$s and reversing the Alexander grading leaves both of them invariant. Doing this to the mapping cone gives us maps between the mirrors of the two tangles, but in the opposite direction.
\end{proof}
\begin{Remark}\label{rem:singularcrossings}
It is interesting to compare the ``figure-8'' curve to the local Heegaard diagram for a singular crossing $\!\!\raisebox{-5pt}{
\psset{unit=0.2}
\psset{arrowsize=1.5pt 2}
\begin{pspicture}(-1.05,-1.45)(1.05,1.05)
\psline{->}(0.9,-0.9)(-0.9,0.9)
\psline{->}(-0.9,-0.9)(0.9,0.9)
\psdot(0,0)
\end{pspicture}
}\!\!$ in \cite{OSSHFS}, as the number of generators agree for the second loop above, up to an additional tensor factor. Also note that the proposition above gives rise to an exact triangle similar to the one in \cite{OSrescube}. Moreover, we can write the $n$-twist tangle $T_{n}$ from figure~\ref{fig:OrientedSkeinRelationTn}, with both strands oriented upwards, as a complex in the objects $\!\!\raisebox{-5pt}{
\psset{unit=0.2}
\psset{arrowsize=1.5pt 2}
\begin{pspicture}(-1.05,-1.45)(1.05,1.05)
\psline{->}(0.9,-0.9)(-0.9,0.9)
\psline{->}(-0.9,-0.9)(0.9,0.9)
\psdot(0,0)
\end{pspicture}
}\!\!$ 
and
$
\!\!\raisebox{-5pt}{
\psset{unit=0.2}
\psset{arrowsize=1.5pt 2}
\begin{pspicture}(-1.05,-1.45)(1.05,1.05)
\psline{->}(-0.9,-0.9)(0.9,0.9)
\psline{->}(0.9,-0.9)(-0.9,0.9)
\pscircle*[linecolor=white](0,0){0.5}
\psarc(0.7071,0){0.5}{135}{225}
\psarc(-0.7071,0){0.5}{-45}{45}
\end{pspicture}
}\!\!
$. Indeed, cancelling the identity components in the following complex gives us a loop representing $T_{n}$.
$$\begin{tikzcd}[row sep=0.5cm, column sep=0.45cm,crossing over clearance=3pt]
&
b^{\red 2-n}
\arrow[bend left=10,leftarrow,pos=0.35]{dd}{q_{412}}
\arrow[bend right=10,swap,pos=0.75]{dd}{q_3}
&&&&
b^{\red 4-n}
\arrow[bend left=10,leftarrow,pos=0.35]{dd}{q_{412}}
\arrow[bend right=10,swap,pos=0.75]{dd}{q_3}
&&
\!\!\!\cdots\!\!\!
&&
b^{\red n}
\arrow[bend left=10,leftarrow,pos=0.65]{dd}{q_{412}}
\arrow[bend right=10,swap,pos=0.75]{dd}{q_3}
&&&&
d^{\red n}
\\
a^{\red 1-n}
\arrow[bend left=10,leftarrow]{ru}{p_2}
\arrow[bend right=10,swap,pos=0.15]{ru}{p_{143}}
\arrow[crossing over,in=180,out=0,swap,pos=0.1]{rrrrrd}{p_{14}}
&&&&
a^{\red 3-n}
\arrow[bend left=10,leftarrow]{ru}{p_2}
\arrow[bend right=10,swap,pos=0.15]{ru}{p_{143}}
\arrow[dotted,crossing over,in=180,out=0,swap,pos=0.1]{rrrrrd}{p_{14}}
&&
\!\!\!\cdots\!\!\!
&&
a^{\red n-1}
\arrow[bend left=10,leftarrow]{ru}{p_2}
\arrow[bend right=10,swap,pos=0.15]{ru}{p_{143}}
\arrow[crossing over,pos=0.9,in=180,out=0]{urrrrr}{p_1}
\\
&
c^{\red 1-n}
\arrow[bend left=10,leftarrow]{ld}{p_{214}}
\arrow[bend right=10,swap,pos=0.3]{ld}{p_3}
\arrow[in=180,out=0,pos=0.03,swap]{rrru}{q_{41}}
&&&&
c^{\red 3-n}
\arrow[bend left=10,leftarrow]{ld}{p_{214}}
\arrow[bend right=10,swap,pos=0.3]{ld}{p_3}
\arrow[dotted,in=180,out=0,pos=0.03,swap]{rrru}{q_{41}}
&&
\!\!\!\cdots\!\!\!
&&
c^{\red n-1}
\arrow[bend left=10,leftarrow]{ld}{p_{214}}
\arrow[bend right=10,swap,pos=0.4]{ld}{p_3}
\arrow[pos=0.65,in=-150,out=0,pos=0.6]{uurrrr}{q_4}
\\
b^{\red -n}
\arrow[crossing over,bend left=10,leftarrow,pos=0.5]{uu}{q_2}
\arrow[crossing over,bend right=10,swap,pos=0.55]{uu}{q_{341}}
&&&&
b^{\red 2-n}
\arrow[crossing over,bend left=10,leftarrow,pos=0.65]{uu}{q_2}
\arrow[crossing over,bend right=10,swap,pos=0.65]{uu}{q_{341}}
\arrow[swap,in=0,out=180,pos=0.8,leftarrow]{llluuu}{1}
&&
\!\!\!\cdots\!\!\!
&&
b^{\red n-2}
\arrow[crossing over,bend left=10,leftarrow,pos=0.65]{uu}{q_2}
\arrow[crossing over,bend right=10,swap,pos=0.65]{uu}{q_{341}}
\arrow[swap,dotted,in=0,out=180,pos=0.8,leftarrow]{llluuu}{1}
&&&&
b^{\red n}
\arrow[swap,in=0,out=180,pos=0.8,leftarrow]{llluuu}{1}
\arrow[bend left=10,leftarrow,pos=0.2]{uuur}{p_{43}+q_{12}}
\arrow[bend right=10,swap,pos=0.3]{uuur}{p_{21}+q_{34}}
\end{tikzcd}$$
Similarly, we can obtain a complex for $T_{-n}$ by applying the mirror operation. Furthermore, it is also easy to find such complexes for other orientations of $T_{-n}$ and $T_n$. Note that these complexes look very much like the ones we get in Bar-Natan's Khovanov homology of tangles \cite{BarNatanKhT}. We assume that every tangle can be written as a complex in the two objects $\!\!\raisebox{-5pt}{
\psset{unit=0.2}
\psset{arrowsize=1.5pt 2}
\begin{pspicture}(-1.05,-1.45)(1.05,1.05)
\psline{->}(0.9,-0.9)(-0.9,0.9)
\psline{->}(-0.9,-0.9)(0.9,0.9)
\psdot(0,0)
\end{pspicture}
}\!\!$ 
and
$
\!\!\raisebox{-5pt}{
\psset{unit=0.2}
\psset{arrowsize=1.5pt 2}
\begin{pspicture}(-1.05,-1.45)(1.05,1.05)
\psline{->}(-0.9,-0.9)(0.9,0.9)
\psline{->}(0.9,-0.9)(-0.9,0.9)
\pscircle*[linecolor=white](0,0){0.5}
\psarc(0.7071,0){0.5}{135}{225}
\psarc(-0.7071,0){0.5}{-45}{45}
\end{pspicture}
}\!\!
$, or the two objects 
$
\!\!\raisebox{-5pt}{
\psset{unit=0.2}
\psset{arrowsize=1.5pt 2}
\begin{pspicture}(-1.05,-1.45)(1.05,1.05)
\psline{-}(-0.9,-0.9)(0.9,0.9)
\psline{<->}(0.9,-0.9)(-0.9,0.9)
\pscircle*[linecolor=white](0,0){0.5}
\psarc(0.7071,0){0.5}{135}{225}
\psarc(-0.7071,0){0.5}{-45}{45}
\end{pspicture}
}\!\!
$
and 
$
\!\!\raisebox{-5pt}{
\psset{unit=0.2}
\psset{arrowsize=1.5pt 2}
\begin{pspicture}(-1.05,-1.45)(1.05,1.05)
\psrotate(0,0){90}{
\psline{<->}(-0.9,-0.9)(0.9,0.9)
\psline{-}(0.9,-0.9)(-0.9,0.9)
\pscircle*[linecolor=white](0,0){0.5}
\psarc(0.7071,0){0.5}{135}{225}
\psarc(-0.7071,0){0.5}{-45}{45}
}
\end{pspicture}
}\!\!
$, depending on the orientation. In fact, we can iteratively use the type AA glueing structure from theorem~\sref{thm:CFTdGeneralGlueing} together with the skein exact sequence from theorem~\sref{thm:nTwistSkeinRelation} to locally modify tangles until we obtain a complex of peculiar modules of trivial and 1-crossing tangles, up to a large number of tensor factors from glueing. Then, one (only) needs to get rid of these extra factors. For example, in the case of the $(2,-3)$-pretzel tangle, this is indeed possible; we may write it as a complex of elementary tangles. This is actually how I originally found the curves in figure~\ref{fig:mutationexamplefinalresult} in the Fukaya category setting (see section~\ref{sec:LoopsAreLoops}) before recovering them from tangle Heegaard diagrams.\\
It is also interesting to compare our ``figure-8'' curve to the curve that Hedden, Herald, Kirk associate with a trivial tangle in \cite[figure~10]{HHK13} in the context of instanton knot Floer homology in the pillowcase, which is the following:
$$
{\psset{unit=1.9}
\begin{pspicture}(-1.1,-1.1)(1.1,1.1)

\psecurve[linecolor=blue](1,6)(0.4,0.81)(0,0)(-0.4,-0.81)(-1,-6)
\psecurve[linecolor=blue](6,1)(0.825,0.4)(0,0)(-0.82,-0.4)(-6,-1)
\psecurve[linecolor=blue,linestyle=dotted,dotsep=1pt](0.5,0.5)(0.815,0.4)(0.4,0.815)(0.5,0.5)
\psecurve[linecolor=blue,linestyle=dotted,dotsep=1pt](-0.5,-0.5)(-0.815,-0.4)(-0.4,-0.815)(-0.5,-0.5)

{\psset{linewidth=1.5pt}
\psecurve(2,2)(1,1)(0.8,0)(1,-1)(2,-2)
\psrotate(0,0){90}{\psecurve(2,2)(1,1)(0.8,0)(1,-1)(2,-2)}
\psrotate(0,0){-90}{\psecurve(2,2)(1,1)(0.8,0)(1,-1)(2,-2)}
\psrotate(0,0){180}{\psecurve(2,2)(1,1)(0.8,0)(1,-1)(2,-2)}
}
\end{pspicture}
}$$
\end{Remark}

%% file: sections/3_MoreRelations.tex
\section{\texorpdfstring{Symmetry relations for $\CFTd$}{Symmetry relations for CFTᵈ}}\label{sec:SymRels}

We start with a result which can be viewed as a categorification of the 4-term relations between the Alexander polynomials for different sites from corollary~\sref{cor:fourendedonecolour}.
\begin{proposition}\label{prop:CFTdAll4sites}
Let \(T\) be a 4-ended tangle and consider \(\CFTd(T)\). After passing to the quotient algebra \(\Ad/(q_1=q_2=q_3=q_4=p_1=0)\) and taking homology, we obtain the following chain complex
$$\begin{tikzcd}[row sep=1.5cm, column sep=2cm]
\HFT(T,d)
\arrow[swap]{d}{\Phi_{p_4}}
\arrow[pos=0.2,swap]{dr}{\Phi_{p_{43}\!\!\!\!\!\!}}
\arrow{r}{\Phi_{p_{432}}}
&
\HFT(T,a)
\\
\HFT(T,c)
\arrow{r}{\Phi_{p_3}}
\arrow[pos=0.8,swap]{ur}{\Phi_{p_{32}}}
&
\HFT(T,b)
\arrow[swap]{u}{\Phi_{p_2}}
\end{tikzcd}$$
where we regard the maps \(\Phi_{\cdot}\) as homomorphisms between the four vector spaces \(\HFT(T,s)\), \(s\in\{a,b,c,d\}\). This chain complex is null-homotopic. By symmetry, the corresponding statements also hold when we cyclically permute sites and algebra elements or swap the roles of the \(p_i\)s and \(q_i\)s.
\end{proposition}
\begin{proof}
The above complex can be identified with the type D module of the following bordered sutured structure on the tangle complement.
\begin{center}
\psset{unit=0.22}
\begin{pspicture}(-10,-10)(10,10)
\pscircle(0,0){10}
\pscircle[linecolor=lightgray](0,0){4}

\psline[linecolor=white,linewidth=4pt](3;45)(6.3;45)
\psline[linecolor=white,linewidth=4pt](3;135)(6.3;135)
\psline[linecolor=white,linewidth=4pt](3;-135)(6.3;-135)
\psline[linecolor=white,linewidth=4pt](3;-45)(6.3;-45)

\psline[linecolor=gray](3;45)(6.3;45)
\psline[linecolor=gray](3;135)(5.9;135)
\psline[linecolor=gray](3;-135)(6.3;-135)
\psline[linecolor=gray](3;-45)(6.3;-45)

\psarc[linecolor=red](0,0){7}{45}{44}

{\psset{fillstyle=solid,fillcolor=white}
\rput{45}(7;45){\psellipse[linecolor=darkgreen](0,0)(0.5,1)}
\rput{135}(7;135){\psellipse[linecolor=darkgreen](0,0)(0.5,1)}
\rput{-135}(7;-135){\psellipse[linecolor=darkgreen](0,0)(0.5,1)}
\rput{-45}(7;-45){\psellipse[linecolor=darkgreen](0,0)(0.5,1)}
}

\psline(6.1;135)(6.9;135)
\psline(7.1;135)(7.9;135)
\psline(7.1;-135)(7.9;-135)
\psline(7.1;45)(7.9;45)
\psline(7.1;-45)(7.9;-45)

\rput(-5.5,0){\red $a$}
\rput(0,-5.5){\red $b$}
\rput(5.5,0){\red $c$}
\rput(0,5.5){\red $d$}
\rput(0,0){\textcolor{darkgray}{$T$}}

\end{pspicture}
\end{center}
Taking homology of this complex is the same as pairing this type D module with the type A-module computed from the following diagram.
\begin{center}
\psset{unit=0.3}
\begin{pspicture}(-10,-3.5)(10,3.5)
\psline[linecolor=red](-9,0)(9,0)
\psframe[linecolor=darkgreen](-9,-3)(9,3)

\pscircle[fillstyle=solid,fillcolor=white,linecolor=darkgreen](-5,0){1}
\pscircle[fillstyle=solid,fillcolor=white,linecolor=darkgreen](0,0){1}
\pscircle[fillstyle=solid,fillcolor=white,linecolor=darkgreen](5,0){1}
\pscircle[linecolor=blue](-5,0){2}
\pscircle[linecolor=blue](0,0){2}
\pscircle[linecolor=blue](5,0){2}

\rput[b](-8,0.5){\red $a$}
\rput[b](-2.5,0.5){\red $b$}
\rput[b](2.5,0.5){\red $c$}
\rput[b](8,0.5){\red $d$}

\psline(0,0.7)(0,1.3)
\psline(5,0.7)(5,1.3)
\psline(-5,0.7)(-5,1.3)

\psline(0,2.7)(0,3.3)
\psline(0,-2.7)(0,-3.3)

\end{pspicture}
\end{center}
However, if we glue these two bordered sutured manifolds together, we obtain a sutured manifold which is not taut, so its sutured Floer homology vanishes by \cite[proposition~9.18]{Juhasz}.
\end{proof}
The maps $\Phi_{p_i}$ and $\Phi_{q_i}$ are, by definition, invariants of the tangle $T$. We do not consider any naturality issues here, so this essentially just means that the ranks of all bigraded parts of the maps are invariants. For loop type $\CFTd(T)$, those ranks are simply the number of arrows labelled by the corresponding $p_i$s or $q_i$s. 
\begin{proposition}\label{prop:CFTdPeculiarRanks}
With the notation as in the previous proposition, let
$$\rk_\delta(p_1):=\rk\left(\Phi_{p_1}:\HFT_\delta(T,a)\rightarrow\HFT_{\delta+\frac{1}{2}}(T,d)\right)$$ 
and similarly for the other \(p_i\) and all \(q_i\). Then
$$
\rk_\delta(p_1)=\rk_\delta(q_2)=\rk_\delta(p_3)=\rk_\delta(q_4)
\quad\text{and}\quad
\rk_\delta(q_1)=\rk_\delta(p_2)=\rk_\delta(q_3)=\rk_\delta(p_4).
$$
\end{proposition}
This result, together with the symmetry relations for generators, tempts us to conjecture the following.
\begin{conjecture}[$\delta$-graded mutation invariance]\label{conj:MutInvCFTd}
Let \(T\) be a 4-ended tangle and \(T'\) obtained from \(T\) by switching two or all four opposite sites of \(T\). Then,
$$\CFTd(T)\cong \CFTd(T')$$
as \(\delta\)-graded invariants.
\end{conjecture}
Together with the glueing theorem, this would imply $\delta$-graded mutation invariance of link Floer homology (conjecture~\ref{conj:MutInvHFL}).
\begin{proof}[Proof of proposition~\ref{prop:CFTdPeculiarRanks}]
The proof is very similar to the previous one. By symmetry, we only need to prove the first set of equalities. Let us consider the identity $\rk_\delta(p_1)=\rk_\delta(p_3)$ first. The maps $\Phi_{p_1}$ and $\Phi_{p_3}$ can be computed from the bordered sutured structure on the tangle complement shown on the left; we either use the dotted $\alpha$-arcs or the solid ones:
\begin{center}
\psset{unit=0.22}
\begin{pspicture}(-23.5,-10)(23.5,10)
\rput(-12.5,0){
\pscircle(0,0){10}
\pscircle[linecolor=lightgray](0,0){4}

\psline[linecolor=white,linewidth=4pt](3;45)(5.8;45)
\psline[linecolor=white,linewidth=4pt](3;135)(6.3;135)
\psline[linecolor=white,linewidth=4pt](3;-135)(5.8;-135)
\psline[linecolor=white,linewidth=4pt](3;-45)(6.3;-45)

\psline[linecolor=gray](3;45)(5.8;45)
\psline[linecolor=gray](3;135)(6.3;135)
\psline[linecolor=gray](3;-135)(5.8;-135)
\psline[linecolor=gray](3;-45)(6.3;-45)

\psarc[linecolor=red](0,0){7}{-135}{45}
\psarc[linecolor=red,linestyle=dotted,dotsep=1pt](0,0){7}{45}{-135}

{\psset{fillstyle=solid,fillcolor=white}
\rput{45}(7;45){\psellipse[linecolor=darkgreen](0,0)(0.5,1)}
\rput{135}(7;135){\psellipse[linecolor=darkgreen](0,0)(0.5,1)}
\rput{-135}(7;-135){\psellipse[linecolor=darkgreen](0,0)(0.5,1)}
\rput{-45}(7;-45){\psellipse[linecolor=darkgreen](0,0)(0.5,1)}
}

\psline(6.1;-135)(6.9;-135)
\psline(6.1;45)(6.9;45)
\psline(7.1;135)(7.9;135)
\psline(7.1;-135)(7.9;-135)
\psline(7.1;45)(7.9;45)
\psline(7.1;-45)(7.9;-45)

\rput(-5.5,0){\red $a$}
\rput(0,-5.5){\red $b$}
\rput(5.5,0){\red $c$}
\rput(0,5.5){\red $d$}
\rput(0,0){\textcolor{darkgray}{$T$}}
}

\rput(0,0){$\longrightarrow$}

\rput(12.5,0){
\pscircle(0,0){10}
\pscircle[linecolor=lightgray](0,0){4}

\psline[linecolor=white,linewidth=4pt](3;45)(7;45)
\psline[linecolor=white,linewidth=4pt](3;135)(6.3;135)
\psline[linecolor=white,linewidth=4pt](3;-135)(7;-135)
\psline[linecolor=white,linewidth=4pt](3;-45)(6.3;-45)

\psline[linecolor=gray](3;45)(7;45)
\psline[linecolor=gray](3;135)(6.3;135)
\psline[linecolor=gray](3;-135)(7;-135)
\psline[linecolor=gray](3;-45)(6.3;-45)

\psecurve[linecolor=darkgreen](8;45)(1.5;135)(8;-135)(1.5;-45)(8;45)(1.5;135)(8;-135)

{\psset{fillstyle=solid,fillcolor=white}
\rput{135}(7;135){\psellipse[linecolor=darkgreen](0,0)(0.5,1)}
\rput{-45}(7;-45){\psellipse[linecolor=darkgreen](0,0)(0.5,1)}
}

\rput(0,0){\textcolor{darkgray}{$T$}}
}
\end{pspicture}
\end{center}
Computing the homology of $\Phi_{p_1}$, resp.\ $\Phi_{p_3}$, i.\,e.\ computing sum of the dimensions of the kernel and cokernel in each (bi)grading, corresponds to pairing these two bordered sutured manifolds with the one below. In both cases, the resulting sutured manifold is the same, so the (bi)graded sutured Floer homologies agree.
\begin{center}
\psset{unit=0.3}
\begin{pspicture}(-10,-3.5)(10,3.5)
\psline[linecolor=red](-9,0)(9,0)
\psframe[linecolor=darkgreen](-9,-3)(9,3)
\pscircle[fillstyle=solid,fillcolor=white,linecolor=darkgreen](0,0){1}
\pscircle[linecolor=blue](0,0){2}
\rput[b](-5,0.5){{\red $b$} resp. {\red $d$}}
\rput[b](5,0.5){{\red $c$} resp. {\red $a$}}
\psline(0,0.7)(0,1.3)
\psline(0,2.7)(0,3.3)
\psline(0,-2.7)(0,-3.3)
\end{pspicture}
\end{center}
We now use proposition~\sref{prop:fourendedHFT} to see that the dimensions of the $\delta$-graded parts of the domain and codomain agree, so the ranks are also the same. \\
Similarly, we can show $\rk_\delta(q_2)=\rk_\delta(q_4)$. Note that the two sutured manifolds for this identity and the previous one are very similar: In one case, we can push the meridional sutures through the tangle to see that we have the same set of sutures up to reversal of orientation. Then, the $\delta$-graded sutured Floer homologies agree by \cite[proposition~2.14]{DecatSFH}. In the other case, the two manifolds contain one pair of meridional sutures on the open tangle components and are obtained from one another by deleting one and adding the other. In this case, we can consider exactly the same decomposing surface as in the proof of proposition~\sref{prop:fourendedHFT} and use \cite[proposition~8.6]{SurfaceDecomposition} to see that, again, the $\delta$-graded sutured Floer homologies are the same. Now we can argue as before, by using proposition~\sref{prop:fourendedHFT}.
\end{proof}
\begin{Remark}
The symmetries in the previous proposition do not hold for bigraded ranks. (As a counterexample, consider the computation for the $(2,-3)$-pretzel tangle in example~\sref{exa:HFTdpretzeltangle}.) The homologies of the maps, considered as differentials, agree as bigraded vector spaces, but in order to identify the tangle Floer homologies of opposite sites, we need to reverse the Alexander grading. 
\end{Remark}
\begin{Remark}\label{Rem:mutationconj}
In view of the calculations for the $(2,-3)$-pretzel tangle (theorem~\sref{thm:2m3pt} and example~\sref{exa:HFTdpretzeltangle}) and the relation between the non-glueable tangle Floer homology for opposite sites \sref{prop:fourendedHFT}, one might be tempted to conjecture the following, bigraded version of conjecture~\ref{conj:MutInvCFTd}: 
\begin{quote}
\textit{Let \(T\) be a 4-ended tangle and \(T'\) obtained from \(T\) by switching opposite sites of \(T\) and reversing the orientation of all tangle strands. Then,
$$\CFTd(T)\cong \CFTd(T')$$
as bigraded invariants, where we identify the Alexander gradings of the two open tangle strands.}
\end{quote}
However, it turns out, this version is too strong. For this, consider a 4-ended tangle~$T$ oriented such that the outward pointing ends are next to one another. Then the conjecture above, together with a glueing theorem, would imply that mutation in the plane ($\text{Mut}_z$) about any such tangle $T$ leaves \textit{bigraded} knot Floer homology invariant. This is because for this particular orientation and mutation axis, the orientation of the open tangle strands of the mutating tangle needs to be reversed before we can glue it back in. However, we can write mutation about one axis as mutation about another, by introducing twists: 
\begin{equation*}\psset{unit=0.6}
\text{Mut}_z\left(
\raisebox{-1.1cm}{\begin{pspicture}(-2,-2)(2,2)
\psecurve{->}(4;165)(2;135)(0.5;-135)(4;-180)
\psecurve{->}(3;-135)(2;-135)(1.5;-135)(0;-135)
\psecurve{->}(0,0)(1.1;-110)(2;45)(5;50)
\psdot[dotsize=5pt,linecolor=white](1.02;128)
\psdot[dotsize=5pt,linecolor=white](0.75;8)
\psecurve{->}(0.5;-90)(1.1;-161.5)(1.05;95)(2;-45)(4;-40)
\pscircle[linestyle=dotted](0,0){2}
\pscircle[fillstyle=solid, fillcolor=white](1;-135){0.5}
\rput(1;-135){$\Gamma$}
\end{pspicture}}
\right)=
\raisebox{-1.1cm}{\begin{pspicture}(-2,-2)(2,2)
\psrotate(0,0){180}{
\psecurve{<-}(4;165)(2;135)(0.5;-135)(4;-180)
\psecurve{<-}(3;-135)(2;-135)(1.5;-135)(0;-135)
\psecurve{<-}(0,0)(1.1;-110)(2;45)(5;50)
\psdot[dotsize=5pt,linecolor=white](1.02;128)
\psdot[dotsize=5pt,linecolor=white](0.75;8)
\psecurve{<-}(0.5;-90)(1.1;-161.5)(1.05;95)(2;-45)(4;-40)
\pscircle[linestyle=dotted](0,0){2}
\pscircle[fillstyle=solid, fillcolor=white](1;-135){0.5}
\rput(1;-135){$\Gamma$}
}
\end{pspicture}}
=
\raisebox{-1.1cm}{
\begin{pspicture}(-2,-2)(2,2)
\psecurve{->}(4;165)(2;135)(0.5;-135)(4;-180)
\psecurve{->}(3;-135)(2;-135)(1.5;-135)(0;-135)
\psecurve{->}(0,0)(1.1;-110)(2;45)(5;50)
\psdot[dotsize=5pt,linecolor=white](1.02;128)
\psdot[dotsize=5pt,linecolor=white](0.75;8)
\psecurve{->}(0.5;-90)(1.1;-161.5)(1.05;95)(2;-45)(4;-40)
\pscircle[linestyle=dotted](0,0){2}
\pscircle[fillstyle=solid, fillcolor=white](1;-135){0.5}
\rput(1;-135){\reflectbox{$\Gamma$}}
\end{pspicture}}
\end{equation*}
So we deduce that mutation about the vertical axis also leaves bigraded knot Floer homology invariant if the orientation of the mutating tangle is such that the outward pointing ends are opposite each another. This, however, cannot be true, as the Kinoshita-Terasaka/Conway pair (example~\sref{exa:counterexampleMUT}) shows, so the conjecture above is indeed too strong.
\end{Remark}
\begin{question}
Let \(T\) be a 4-ended tangle which is symmetric about one mutation axis. Can we always find an orientation such that \(\CFTd\) is bigraded mutation invariant with respect to the other two axes?
\end{question}

%% file: sections/3_LoopsAreLoops.tex

\section{\texorpdfstring{$\CFTd$ and the wrapped Fukaya category of the 4-punctured sphere}{CFTᵈ and the wrapped Fukaya category of the 4-punctured sphere}}
\label{sec:LoopsAreLoops}
In this last section, we explore the relationship between the peculiar invariants of 4-ended tangles and the wrapped Fukaya category $\Fuk$ of the 4-punctured sphere. Much of this is still work in progress, so there are still plenty of open questions, some of which we discuss in an outlook at the end.
\subsection*{What is $\Fuk$?}
The category $\Fuk$ is an $A_\infty$-category which, roughly speaking, encodes the Lagrangian intersection theory of (exact) Lagrangians on a $4$-punctured sphere (equipped with a certain 1-form). We will not discuss the technical details of its construction here, but refer the interested reader to~\cite[section~4]{Abouzaid} instead. The objects in $\Fuk$ are the Lagrangians and the morphisms essentially correspond to intersection points of these. The $A_\infty$-composition maps are defined using certain holomorphic polygons connecting those intersection points. As usual in Lagrangian intersection theory, we need to perturb the Lagrangians by some Hamiltonian flow when computing morphisms. Here, we restrict ourselves to certain Hamiltonians which are ``quadratic at infinity'', which effectively means that a Lagrangian approaching a puncture is wrapped infinitely many times around the puncture by such a Hamiltonian perturbation -- hence the name \emph{wrapped} Fukaya category. This is illustrated in figure~\ref{fig:ML1}. \\
In \cite{Abouzaid}, Abouzaid et al. study the wrapped Fukaya category $\Fuk(S^2,n)$ of the $n$-punctured sphere for $n\geq 3$. They introduce a finite auxiliary $A_\infty$-category~$\mathcal{A}$ which generates $\Fuk(S^2,n)$ and describe some structure maps of~$\mathcal{A}$, namely all higher compositions up to length~$n$. Then, using a Hochschild homology argument, they show that these compositions uniquely characterise $\mathcal{A}$, and hence $\Fuk(S^2,n)$, up to homotopy, see~\cite[theorem~4.1]{Abouzaid}. 
\subsection*{Relationship with peculiar modules. }
For the case $n=4$, we will show that $\Fuk:=\Fuk(S^2,4)$ and the category of peculiar modules $\pqMod$ are closely related. In theorem~\ref{thm:TwFukpqModEquivalent}, we construct $A_\infty$-functors~$\mathcal{M}$ and~$\mathcal{L}$ between $\pqMod$ and the category $\TwFuk$, the triangulated enlargement of $\Fuk$, also known as the category of twisted complexes. For more background on the category $\TwFuk$, twisted complexes and $A_\infty$-categories in general, see~\cite[section~I.1a--d and~I.3l]{Seidel}, \cite[remark~4.2]{Abouzaid} and example~\ref{exa:twistedcomplexes}.\\
Given a 4-ended tangle $T$, $\mathcal{L}(\CFTd(T))$ is an object in $\TwFuk$ which is an invariant of~$T$ up to homotopy. Given also a site $s$ of $T$, we can interpret $\HFT(T,s)$ as the Lagrangian intersection homology of $\mathcal{L}(\CFTd(T))$ and a certain generator $L_s$ of $\TwFuk$ associated with $s$. Moreover, computations suggest that the two functors~$\mathcal{M}$ and~$\mathcal{L}$ actually set up an equivalence of categories.\\
For the construction of~$\mathcal{M}$ and~$\mathcal{L}$, it is essential that we understand the composition maps in $\TwFuk$. This is equivalent to understanding the finite category $\mathcal{A}$, since by the generation result from~\cite{Abouzaid}, $\TwFuk$ agrees with $\Tw\mathcal{A}$. 
\subsection*{The auxiliary category~$\mathcal{A}$. }
Following~\cite[section~2]{Abouzaid}, we build the $A_\infty$-category~$\mathcal{A}$ from an ordinary category $A$, which we define now. 
\begin{definition}
Let $A$ be a category with $n$ objects $L_1,\dots,L_n$. The morphism spaces between these objects are as follows:
\begin{equation*}
\Mor(L_i,L_j):=
\begin{cases}
k[x_i,y_i]/x_iy_i & \text{for }i=j\\
k[x_{i+1}]u_{i,i+1}=u_{i,i+1}k[y_i] & \text{for }j=i+1\\
k[y_{i-1}]v_{i,i-1}=v_{i,i-1}k[x_i] & \text{for }j=i-1\\
0 & \text{otherwise}
\end{cases},
\end{equation*}
where indices are taken modulo $n$. Each morphism space comes with a natural basis, namely 
$$\{1,x_i^l,y_i^l\vert l\geq1\},\quad \{x_{i+1}^lu_{i,i+1}=u_{i,i+1}y_i^l\vert l\geq0\}\quad\text{and}\quad\{y_{i-1}^lv_{i,i-1}=v_{i,i-1}x_i^l\vert l\geq0\},$$ 
respectively. In the following, when we talk about morphisms, we mean those basis elements. By sequences of morphisms of length $k$, we mean elements of the induced basis of the tensor space 
$$\Mor(L_{i_{k-1}},L_{i_k})\otimes\dots\otimes\Mor(L_{i_0},L_{i_1}).$$
For $k=2$, this is the domain of the composition map $\mu_2$ for $A$, which we describe later; for arbitrary $k$, this will be the domain for higher compositions in $\mathcal{A}$. We sometimes drop the index when it is clear from the context. We follow the convention in~\cite{Seidel} to read sequences of morphisms from right to left. 
\end{definition}

\subsection*{Graphical representation of morphisms in $A$ and $\mathcal{A}$. } For explicit calculations, it is helpful to depict morphisms and sequences thereof graphically. The basic objects $L_i$ are represented by infinitely many levels (i.\,e.\ horizontal lines, but we never actually draw them), numbered cyclically from 1 to $n$, each of them at a constant positive distance from its two neighbours -- like on a stave. A morphism from $L_i$ to $L_j$ is represented by a straight line from a level corresponding to $L_i$ to the closest one that corresponds to~$L_j$. So by definition, such a line either moves up by one, stays constant or moves down a level. Next, we label the line by the morphism, but the advantage now is that we can drop the indices, without loosing information.  Furthermore, if the level changes, we only need to record the exponent $l$ in $x^lu$ or $v x^l$. For simplicity, we often drop the label completely if this exponent is 0. For sequences of morphisms, we just concatenate the lines representing the morphisms in the sequence in the specified order. In this graphical representation, we read sequences from left to right.\\
We indicate the numbering of the levels by specifying the level of the starting point of the first morphism and use the convention that the numbering increases from bottom to top. Often, we even drop this information, since the multiplication maps are symmetric in the~$L_i$.
\begin{example}\label{exa:sequences}
The following table shows some morphisms and their graphical notation as described above. All morphisms start at the index 1.
\begin{center}
\begin{tabular}{c|c|c|c|c|c|c}
$u_{1,2}$ & 
$x_2u_{1,2}$ & 
$y_1^2$ & 
$1_1$ & 
$x_1^4$ & 
$v_{1,n}$ & 
$y_1^2v_{1,n}$ \\ 
\hline 
{\psset{unit=0.6}
\begin{pspicture}(-0.3,-0.3)(1.3,1.3)
\Wu(0,0)
\end{pspicture}
}
 &
{\psset{unit=0.6}
\begin{pspicture}(-0.3,-0.3)(1.3,1.3)
\wu{1}(0,0)
\end{pspicture}
} 
 & 
{\psset{unit=0.6}
\begin{pspicture}(-0.3,-0.3)(1.3,1.3)
\wi{$y^2$}(0,0.5)
\end{pspicture}
} 
  & 
{\psset{unit=0.6}
\begin{pspicture}(-0.3,-0.3)(1.3,1.3)
\wi{$1$}(0,0.5)
\end{pspicture}
}   
   & 
{\psset{unit=0.6}
\begin{pspicture}(-0.3,-0.3)(1.3,1.3)
\wi{$x^4$}(0,0.5)
\end{pspicture}
}      
    & 
 {\psset{unit=0.6}
\begin{pspicture}(-0.3,-0.3)(1.3,1.3)
\Wv(0,1)
\end{pspicture}
}    
     & 
   {\psset{unit=0.6}
\begin{pspicture}(-0.3,-0.3)(1.3,1.3)
\wv{$2$}(0,1)
\end{pspicture}
}     
      \\ 
\end{tabular} 
\end{center}
\pagebreak[3]
\noindent
The sequence $s=(x_3u_{2,3},u_{1,2},x_1,v_{2,1},v_{3,2},y_3)$ is represented by
$$
{\psset{unit=0.6}
\begin{pspicture}(-0.5,-0.3)(6.3,2.3)
\rput[r](-0.2,2){\scriptsize 3}
\wi{$y$}(0,2)
\Wv(1,2)
\Wv(2,1)
\wi{$x$}(3,0)
\Wu(4,0)
\wu{$1$}(5,1)
\end{pspicture}
}  
$$
In the case $n=4$, we will see below that $\mu$ of this sequence is the morphism $x_3$.
\end{example}

\begin{definition}[composition in $A$]
Most compositions are already implicit in the definition of the morphism spaces as rings and bimodules. The only remaining composition are
\begin{center}
\psset{unit=0.6}
$\mu\left(\!\!
\raisebox{-0.6cm}{
\begin{pspicture}(-0.3,-0.3)(2.3,2.3)
\wu{$\cdot$}(0,0)
\wu{$\cdot$}(1,1)
\end{pspicture}}
\right)=0$, 
$\mu\left(\!\!
\raisebox{-0.6cm}{
\begin{pspicture}(-0.3,-0.3)(2.3,2.3)
\wv{$\cdot$}(0,2)
\wv{$\cdot$}(1,1)
\end{pspicture}}
\right)=0$, \\
$\mu\left(\!\!
\raisebox{-0.3cm}{
\begin{pspicture}(-0.3,-0.3)(2.3,1.3)
\wu{$k$}(0,0)
\wv{$l$}(1,1)
\end{pspicture}}
\right)=y^{k+l+1}$\quad and \quad
$\mu\left(\!\!
\raisebox{-0.3cm}{
\begin{pspicture}(-0.3,-0.3)(2.3,1.3)
\wv{$k$}(0,1)
\wu{$l$}(1,0)
\end{pspicture}}
\right)=x^{k+l+1}$ \quad for $k,l\geq0$.
\end{center}
For completeness, we also describe the other obvious compositions graphically:
\begin{center}
\psset{unit=0.6}
$\mu\left(\!\!
\raisebox{-0.3cm}{
\begin{pspicture}(-0.3,-0.3)(2.3,1.3)
\wu{$k$}(0,0)
\wi{$x^l$}(1,1)
\end{pspicture}}
\right)=
\raisebox{-0.3cm}{
\begin{pspicture}(-0.3,-0.3)(1.3,1.3)
\wu{$k+l$}(0,0)
\end{pspicture}}$ 
\quad and \quad
$\mu\left(\!\!
\raisebox{-0.3cm}{
\begin{pspicture}(-0.3,-0.3)(2.3,1.3)
\wi{$x^l$}(0,1)
\wv{$k$}(1,1)
\end{pspicture}}
\right)=
\raisebox{-0.3cm}{
\begin{pspicture}(-0.3,-0.3)(1.3,1.3)
\wv{$k+l$}(0,1)
\end{pspicture}}
$ \quad
for $k,l\geq0$,\\
$\mu\left(\!\!
\raisebox{-0.3cm}{
\begin{pspicture}(-0.3,-0.3)(2.3,1.3)
\wi{$y^l$}(0,0)
\wu{$k$}(1,0)
\end{pspicture}}
\right)=
\raisebox{-0.3cm}{
\begin{pspicture}(-0.3,-0.3)(1.3,1.3)
\wu{$k+l$}(0,0)
\end{pspicture}}$ 
\quad and \quad
$\mu\left(\!\!
\raisebox{-0.3cm}{
\begin{pspicture}(-0.3,-0.3)(2.3,1.3)
\wv{$k$}(0,1)
\wi{$y^l$}(1,0)
\end{pspicture}}
\right)=
\raisebox{-0.3cm}{
\begin{pspicture}(-0.3,-0.3)(1.3,1.3)
\wv{$k+l$}(0,1)
\end{pspicture}}
$ \quad
for $k,l\geq0$.
\end{center}
\end{definition}
\begin{definition}[higher compositions in $\mathcal{A}$]
The $A_\infty$-category $\mathcal{A}$ is built from $A$ by adding higher composition maps. They are recursively defined by
\begin{equation*}
\begin{array}{rl}
\mu\left(\dots,\mu_2\left(b,
\raisebox{-0.3cm}{\psset{unit=0.6}
\begin{pspicture}(-0.1,-0.1)(1.1,1.1)
\Wu(0,0)
\end{pspicture}
}  
\right),\underbrace{
\raisebox{-0.3cm}{\psset{unit=0.6}
\begin{pspicture}(-0.1,-0.1)(1.1,1.1)
\Wu(0,0)
\end{pspicture}
}
,\dots,
\raisebox{-0.3cm}{\psset{unit=0.6}
\begin{pspicture}(-0.1,-0.1)(1.1,1.1)
\Wu(0,0)
\end{pspicture}
}}_{n-2}
,\mu_2\left(
\raisebox{-0.3cm}{\psset{unit=0.6}
\begin{pspicture}(-0.1,-0.1)(1.1,1.1)
\Wu(0,0)
\end{pspicture}
}  
,a\right),\dots\right)&:=\mu(\dots,b,a,\dots) \text{ and}\\
\mu\left(\dots,\mu_2\left(b,
\raisebox{-0.3cm}{\psset{unit=0.6}
\begin{pspicture}(-0.1,-0.1)(1.1,1.1)
\Wv(0,1)
\end{pspicture}
}  
\right),\underbrace{
\raisebox{-0.3cm}{\psset{unit=0.6}
\begin{pspicture}(-0.1,-0.1)(1.1,1.1)
\Wv(0,1)
\end{pspicture}
}
,\dots,
\raisebox{-0.3cm}{\psset{unit=0.6}
\begin{pspicture}(-0.1,-0.1)(1.1,1.1)
\Wv(0,1)
\end{pspicture}
}}_{n-2}
,\mu_2\left(
\raisebox{-0.3cm}{\psset{unit=0.6}
\begin{pspicture}(-0.1,-0.1)(1.1,1.1)
\Wv(0,1)
\end{pspicture}
}  
,a\right),\dots\right)&:=\mu(\dots,b,a,\dots),
\end{array}
\end{equation*}
where $a$ and $b$ are some morphisms, all sequences are composable and the two $\mu_2$-terms on the left hand sides are all non-zero. In other words, all non-zero $A_\infty$-terms are obtained by ``expanding'' non-zero 2-term sequences by $n$-term sequences of $u$s and $n$-term sequences of $v$s. In particular, the differential on $\mathcal{A}$ vanishes.
\end{definition}
\begin{proposition}\label{prop:muwelldefined}
The maps \(\mu\) above define a unital \(A_\infty\)-structure on \(\mathcal{A}\).
\end{proposition}
\begin{proof}
To my knowledge, this was first proven in~\cite[proposition~A.11]{Bocklandt}. For a more concise and elegant proof, which simply exhibits cancelling pairs in the $A_\infty$-relations, see \cite[proposition~3.1]{Kontsevich}.
\end{proof}

\subsection*{Gradings on~$\mathcal{A}$. }
In \cite{Abouzaid}, the morphisms in $\mathcal{A}$ carry a $\mathbb{Z}$-grading $\deg$ defined by 
$$\deg(u_{i,i+1})=\deg\left(
\raisebox{-0.3cm}{\psset{unit=0.6}
\begin{pspicture}(-0.5,-0.3)(1.3,1.3)
\rput[r](-0.2,0){\scriptsize i}
\Wu(0,0)
\end{pspicture}
}  
\right) = p_{i+1}\quad\text{and}\quad
\deg(v_{i,i-1})=\deg\left(
\raisebox{-0.3cm}{\psset{unit=0.6}
\begin{pspicture}(-0.5,-0.3)(1.3,1.3)
\rput[r](-0.2,1){\scriptsize i}
\Wv(0,1)
\end{pspicture}
}  
\right) = q_i,$$ 
where $p_1,\dots,p_n$ and $q_1,\dots,q_n$ are some odd integers, satisfying the relations 
\begin{equation}\label{eqn:sumallpandq}
p_1+\dots+p_n=(n-2)=q_1+\dots+q_n.
\end{equation}
It is easy to see that this extends to a well-defined grading on all basic morphisms using
$$\deg(\mu(b,a))=\deg(a)+\deg(b).$$
The relations (\ref{eqn:sumallpandq}) imply that $\mu_k$ decreases the grading by $(k-2)$. As usual, we extend this grading to $\Tw\mathcal{A}$ by setting 
\begin{equation}\label{eqn:GradingOnMor}
\deg(f:L_i[x]\rightarrow L_j[y])=\deg(f)+y-x.
\end{equation}
for homogeneous $f$.\\
However, here, we replace $\deg$ by several, slightly different gradings which correspond to the gradings on our Heegaard Floer invariants. 
For this, let us treat $p_1, p_2, p_3, p_4, q_1, q_2, q_3$ and $q_4$ as formal variables in some free Abelian monoid.

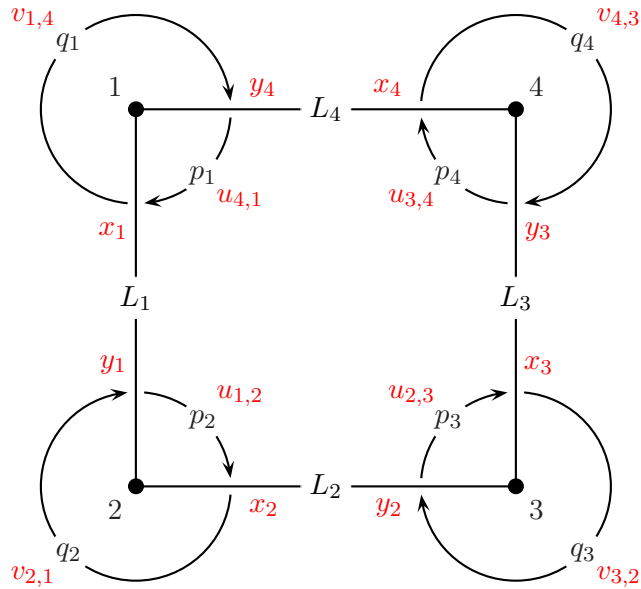
\begin{figure}[h]
\centering
\psset{unit=2.5}
\begin{pspicture}(-1.7,-1.7)(1.7,1.7)

\psline(-1,-1)(-1,1)(1,1)(1,-1)(-1,-1)

\psdots[dotsize=6pt](-1,1)(1,1)(1,-1)(-1,-1)
\rput(-1,0){\fcolorbox{white}{white}{$L_1$}}
\rput(0,-1){\fcolorbox{white}{white}{$L_2$}}
\rput(1,0){\fcolorbox{white}{white}{$L_3$}}
\rput(0,1){\fcolorbox{white}{white}{$L_4$}}
\uput{0.1}[135](-1,1){1}
\uput{0.1}[-135](-1,-1){2}
\uput{0.1}[-45](1,-1){3}
\uput{0.1}[45](1,1){4}

\rput(-1,1){
\psarcn{->}(0,0){0.5}{-5}{-85}
\psarcn{->}(0,0){0.5}{-95}{5}
\pscircle*[linecolor=white](0.5;-45){8pt}
\rput(0.5;-45){$p_1$}
\pscircle*[linecolor=white](0.5;135){8pt}
\rput(0.5;135){$q_1$}
}

\rput(-1,-1){
\psarcn{->}(0,0){0.5}{85}{5}
\psarcn{->}(0,0){0.5}{-5}{95}
\pscircle*[linecolor=white](0.5;45){8pt}
\rput(0.5;45){$p_2$}
\pscircle*[linecolor=white](0.5;-135){8pt}
\rput(0.5;-135){$q_2$}
}

\rput(1,-1){
\psarcn{->}(0,0){0.5}{175}{95}
\psarcn{->}(0,0){0.5}{85}{-175}
\pscircle*[linecolor=white](0.5;135){8pt}
\rput(0.5;135){$p_3$}
\pscircle*[linecolor=white](0.5;-45){8pt}
\rput(0.5;-45){$q_3$}
}

\rput(1,1){
\psarcn{->}(0,0){0.5}{-95}{-175}
\psarcn{->}(0,0){0.5}{175}{-85}
\pscircle*[linecolor=white](0.5;-135){8pt}
\rput(0.5;-135){$p_4$}
\pscircle*[linecolor=white](0.5;45){8pt}
\rput(0.5;45){$q_4$}
}

\red
\rput(-1,1){%
\uput{0.6}[-45](0,0){$u_{4,1}$}
\uput{0.6}[135](0,0){$v_{1,4}$}
\uput{0.6}[-100](0,0){$x_{1}$}
\uput{0.6}[10](0,0){$y_{4}$}
}

\rput(-1,-1){%
\uput{0.6}[45](0,0){$u_{1,2}$}
\uput{0.6}[-135](0,0){$v_{2,1}$}
\uput{0.6}[-10](0,0){$x_{2}$}
\uput{0.6}[100](0,0){$y_{1}$}
}

\rput(1,-1){%
\uput{0.6}[135](0,0){$u_{2,3}$}
\uput{0.6}[-45](0,0){$v_{3,2}$}
\uput{0.6}[80](0,0){$x_{3}$}
\uput{0.6}[190](0,0){$y_{2}$}
}

\rput(1,1){%
\uput{0.6}[-135](0,0){$u_{3,4}$}
\uput{0.6}[45](0,0){$v_{4,3}$}
\uput{0.6}[170](0,0){$x_{4}$}
\uput{0.6}[-80](0,0){$y_{3}$}
}
\end{pspicture}
\caption{The $A_\infty$-category of the $4$-punctured sphere in a nutshell. Compare this to figure~\ref{fig:nutshell}.}\label{fig:Ainftynutshell}
\end{figure}

\begin{definition}
Fix an orientation of the four punctures, as if they were endpoints of a 4-ended tangle; in other words, two punctures are labelled ``in'' and the other two ``out'', compare with definition~\sref{def:AlexGradingOnAd}. Define the \textbf{Alexander grading} by setting $A=\deg$ with $p_i=q_i=+1$ if the $i^\text{th}$ puncture is labelled ``out'' and $p_i=q_i=-1$ otherwise. As for $\CFTd$, we sometimes denote the Alexander grading of a basic object by $L_j^\Red{A(\ast)}$. Similarly, we can define a multivariate Alexander grading by separating the four punctures into two pairs of oppositely oriented punctures.
Note that
\begin{equation}\label{eqn:sumallpandqAlex}
p_1+p_2+p_3+p_4=0=q_1+q_2+q_3+q_4,
\end{equation}
so the structure maps $\mu_k$ preserve the Alexander grading.
\end{definition}

\begin{definition}
We define a $\frac{1}{2}\mathbb{Z}$-grading on morphisms in $\mathcal{A}$, which we call \textbf{$\delta$-grading}, by setting $\delta=\deg$ with $p_i=q_i=\frac{1}{2}$. In this case, 
\begin{equation}\label{eqn:sumallpandqDelta}
p_1+p_2+p_3+p_4=2=q_1+q_2+q_3+q_4,
\end{equation}
so the structure maps $\mu_k$ decrease grading by $(k-2)$. Given an orientation on the punctures, the homological grading $h$ is then defined by
$h=\frac{1}{2}A-\delta$. This is equivalent to setting $p_i=q_i=0$ if the $i^\text{th}$ strand is labelled ``out'' and $-1$ otherwise. 
\end{definition}
\begin{Remark}
We can interpret morphisms and sequences thereof in terms of paths in the 4-punctured sphere, see figure~\ref{fig:Ainftynutshell}. In fact, this is where the proofs of proposition~\ref{prop:muwelldefined} come from. 
\end{Remark}
\begin{definition}
Consider the \textbf{opposite algebra} $\Adop$ of $\Ad$, i.\,e.\ the algebra defined by
$$\Adop:=\{(a)^{-1}\vert a\in\Ad\},\quad (a)^{-1}.(b)^{-1}:=(b.a)^{-1},\quad \iota_1.(a)^{-1}.\iota_2:=(\iota_2.a.\iota_1)^{-1},$$
where $\iota_1,\iota_2\in\Id$. If $\deg$ is a grading on $\Ad$, it induces a grading on $\Adop$ by 
$$\deg((a)^{-1})=-\deg(a)$$
for homogeneous $a\in\Ad$. 
We also define two $\Id$-algebra homomorphisms
$$\Adop\times\Ad\rightarrow\Adop\quad\text{and}\quad\Ad\times\Adop\rightarrow\Adop.$$
If $a, b\in\Ad$ are paths in the quiver algebra, define
$$
((a)^{-1},b)
\mapsto
(a)^{-1}.b:=
\begin{cases}
(c)^{-1} &\text{if }\exists \text{path }c: bc=a\\
0	& \text{otherwise}
\end{cases}
\qquad
\raisebox{-12pt}{
	\begin{pspicture}(0,-0.4)(3,0.6)
	\psline{->}(0,0)(1,0)
	\psline{->}(1,0)(3,0)
	\psline{->}(0,0.2)(3,0.2)
	\rput(0.5,-0.2){$b$}
	\rput(2,-0.2){$c$}
	\rput(1.5,0.4){$a$}
	\end{pspicture}
}
$$
Note that if $c$ exists, it is unique. 
Similarly, let
$$
(b,(a)^{-1})
\mapsto
b.(a)^{-1}:=
\begin{cases}
(c)^{-1} &\text{if }\exists \text{path }c: cb=a\\
0	& \text{otherwise}
\end{cases}
\qquad
\raisebox{-12pt}{
	\begin{pspicture}(0,-0.4)(3,0.6)
	\psline{->}(0,0)(2,0)
	\psline{->}(2,0)(3,0)
	\psline{->}(0,0.2)(3,0.2)
	\rput(1,-0.2){$c$}
	\rput(2.5,-0.2){$b$}
	\rput(1.5,0.4){$a$}
	\end{pspicture}
}
$$
Then extend both maps linearly. Obviously, they define $\Id$-algebra homomorphisms.
\end{definition}
\begin{theorem}\label{thm:TwFukpqModEquivalent}
Let \(\TwFuk\) be the triangulated enlargement of the Fukaya category of the 4-punctured sphere. As discussed in the introduction of this section, we can identify \(\TwFuk\) with the \(A_\infty\)-category \(\Tw\mathcal{A}\) of twisted complexes of \(\mathcal{A}\), see example~\ref{exa:twistedcomplexes}. Let \(\pqMod\) be the dg category of peculiar modules, considered as an \(A_\infty\)-category. 
We define a covariant functor 
$$\mathcal{M}:\TwFuk\rightarrow \pqMod$$
as follows. Given a twisted complex \(L\in\ob\TwFuk\), \(\mathcal{M}(L)\) is defined as an \(\Id\)-module by
$$\mathcal{M}(L):= \bigoplus_{i=1,2,3,4} \iota_i.\Mor(L_i,L).$$
The structure map \(\partial_{\mathcal{M}(L)}:\mathcal{M}(L)\rightarrow\Ad\otimes \mathcal{M}(L)\), given by
$$
f\mapsto 1\otimes\mutw_1(f)+\sum_{n\geq 1}p^{n}\otimes\mutw_{n+1}(f,\underbrace{u,\dots,u}_{n})+q^{n}\otimes\mutw_{n+1}(f,\underbrace{v,\dots,v}_{n}),
$$
turns \((\mathcal{M}(L),\partial_{\mathcal{M}(L)})\) into a well-defined peculiar module. \pagebreak\\
Moreover, given a sequence of morphisms of twisted complexes 
\((a_r,\dots,a_1): L^{(0)}\!\rightarrow\!\cdots\!\rightarrow\! L^{(r)}\), we define a map 
\begin{gather*}
\SwapAboveDisplaySkip
\mathcal{M}_r(a_r,\dots,a_1):\mathcal{M}(L^{(0)})\rightarrow \Ad\otimes\mathcal{M}(L^{(r)}),\\
\begin{split}
f\mapsto
& 1\otimes\mutw_{r+1}(a_r,\dots,a_1,f)
+\sum_{n\geq1}p^{n}\otimes\mutw_{n+r+1}(a_r,\dots,a_1,f,\underbrace{u,\dots,u}_{n})\\
& \phantom{1\otimes\mutw_{r+1}(a_r,\dots,a_1,f)}
+\sum_{n\geq1}q^{n}\otimes\mutw_{n+r+1}(a_r,\dots,a_1,f,\underbrace{v,\dots,v}_{n}).
\end{split}
\end{gather*}
We also have a covariant functor 
$$\mathcal{L}: \pqMod\rightarrow\TwFuk,$$
defined as follows. Given a peculiar module \((M, \partial_M)\in\ob\pqMod\), we define a twisted complex
$$
\mathcal{L}(M):=\bigoplus_{j=1,2,3,4}\left(\iota_j.\Adop\otimes_{\Id}M\right)\otimes L_j
$$
with structure map \(\partial_{\mathcal{L}(M)}:\mathcal{L}(M)\rightarrow\mathcal{L}(M)\) given by
$$
b\otimes m\otimes L_j\mapsto
b.\partial_M(m)\otimes \id_{L_j}
+p.b\otimes m\otimes u_{j,j+1}
+q.b\otimes m\otimes v_{j,j-1}.
$$
Furthermore, if \(a:M\rightarrow \Ad\otimes M'\) is a morphism of peculiar modules, we define a morphism \(\mathcal{L}(a):\mathcal{L}(M)\rightarrow\mathcal{L}(M')\) by
$$
b\otimes m\otimes L_j\mapsto b.a(m)\otimes \id_{L_j}.
$$
We set \(\mathcal{L}(a_r,\dots,a_1)=0\) for any sequence of morphisms \((a_1,\dots,a_r)\) with \(r>1\).
\end{theorem}
\begin{Remark}
The two functors are compatible with the various gradings on both sides. As usual, the gradings on a tensor product are given by the sum of the gradings of all factors; the gradings on morphism spaces is as defined in equation~(\ref{eqn:GradingOnMor}).
\end{Remark}
\begin{example}
The image of $L_1$ under $\mathcal{M}$ is given by 
$$
\begin{tikzcd}[column sep=1.1cm]
\cdots
\arrow[dotted,bend left=10,leftarrow,pos=0.5]{r}{q_1}
\arrow[dotted,bend right=10,swap,pos=0.5]{r}{q_{234}}
&
d
\arrow[bend left=10,leftarrow,pos=0.5]{r}{p_1}
\arrow[bend right=10,swap,pos=0.5]{r}{p_{432}}
&
a
\arrow[bend left=10,leftarrow,pos=0.5]{r}{q_1}
\arrow[bend right=10,swap,pos=0.5]{r}{q_{234}}
&
d
\arrow[bend left=10,leftarrow,pos=0.5]{r}{p_1}
\arrow[bend right=10,swap,pos=0.5]{r}{p_{432}}
&
a
\arrow[bend left=10,leftarrow,pos=0.5]{r}{q_{341}}
\arrow[bend right=10,swap,pos=0.5]{r}{q_2}
&
b
\arrow[bend left=10,leftarrow,pos=0.5]{r}{p_{143}}
\arrow[bend right=10,swap,pos=0.5]{r}{p_{2}}
&
a
\arrow[bend left=10,leftarrow,pos=0.5]{r}{q_{341}}
\arrow[bend right=10,swap,pos=0.5]{r}{q_2}
&
b
\arrow[dotted,bend left=10,leftarrow,pos=0.5]{r}{p_{143}}
\arrow[dotted,bend right=10,swap,pos=0.5]{r}{p_{2}}
&
\cdots
\end{tikzcd}
$$
which corresponds to the curve in figure~\ref{fig:ML1}. Similarly, if we (formally) apply $\mathcal{L}$ to $(M, \partial)=(\langle a\rangle,0)$, we obtain
$$
\begin{tikzcd}[column sep=0.6cm]
\cdots
\arrow{r}{u}
&
L_1
\arrow{r}{u}
&
L_2
\arrow{r}{u}
&
L_3
\arrow{r}{u}
&
L_4
\arrow{r}{u}
&
L_1
&
L_2
\arrow[swap]{l}{v}
&
L_3
\arrow[swap]{l}{v}
&
L_4
\arrow[swap]{l}{v}
&
L_1
\arrow[swap]{l}{v}
&
\cdots
\arrow[swap]{l}{v}
\end{tikzcd}
$$
$\mutw_1$ of this complex is non-zero, but note that $(\langle a\rangle,0)$ is not a peculiar module either.

\begin{figure}[t]
\centering
\psset{unit=1.5}
\begin{pspicture}(-2,-1.55)(2,1.55)

\rput(-1,1){
\psrotate(0,0){90}{
\parametricplot[linecolor=red,plotstyle=curve,plotpoints=200]{360}{7200}{%
360 t div t 0.5 mul cos mul 1 mul
360 t div t 0.5 mul sin mul 1 mul
}
}}
\rput(-1,-1){
\psrotate(0,0){-90}{
\parametricplot[linecolor=red,plotstyle=curve,plotpoints=200]{360}{7200}{%
360 t div t 0.5 mul cos mul 1 mul
360 t div t 0.5 mul sin mul 1 mul
}}}

\psline(-1,-1)(-1,1)(1,1)(1,-1)(-1,-1)

\psdots[dotsize=6pt](-1,1)(1,1)(1,-1)(-1,-1)
\rput(-1.55,-0.4){\fcolorbox{white}{white}{$\red \mathcal{M}(L_1)$}}
\rput(-1,0.4){\fcolorbox{white}{white}{$L_1$}}
\rput(0,-1){\fcolorbox{white}{white}{$L_2$}}
\rput(1,0){\fcolorbox{white}{white}{$L_3$}}
\rput(0,1){\fcolorbox{white}{white}{$L_4$}}
\end{pspicture}
\caption{The curve representing the peculiar module $\mathcal{M}(L_1)$}\label{fig:ML1}
\end{figure}
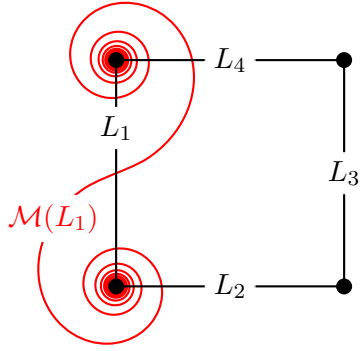

\end{example}

\begin{proof}[Proof of theorem~\ref{thm:TwFukpqModEquivalent}]
The infinite sums in the definition of $\mathcal{M}$ are actually finite, since by definition of the $A_\infty$-structure all terms for $n>3$ vanish. (However, we stick to the notation above as it is easier.) Let us check that $\partial_{\mathcal{M}(L)}$ satisfies the required $\partial^2$-relation. The coefficient of $1$ vanishes, since $\mutw_1$ is a differential. The coefficient of $p^n$ is given by
\begin{equation*}
\begin{split}
&\sum_{\substack{i+j=n\\i,j\geq0}}\mutw_{j+1}(\mutw_{i+1}(f,\underbrace{u,\dots,u}_i),\underbrace{u,\dots,u}_j)\\
=&
\sum_{\substack{i+j+k+1=n\\i,j,k\geq0}}\mutw_{i+k+2}(f,\underbrace{u,\dots,u}_i,\mutw_{j+1}(\underbrace{u,\dots,u}_{j+1}),\underbrace{u,\dots,u}_k).
\end{split}
\end{equation*}
The term $\mutw_{j+1}(u,\dots,u)=\mu_{j+1}(u,\dots,u)$ vanishes for all $j$ except $j=3$, in which case it is the identity. Since $\mu$ and hence also $\mutw$ is unital, this implies that the term above vanishes for all $n$ except $n=4$, in which case there is just one summand, namely 
$$\mutw_{2}(f,\mutw_{4}(u,u,u,u))=\mutw_{2}(f,1)=\mu_2(f,1)=f.$$
We can argue similarly for the coefficient of $q^n$. So $\mathcal{M}$ is well-defined on objects. \\
Next we need to check that $\mathcal{M}$ satisfies the $A_\infty$-relations
\begin{equation*}
\begin{split}
D(\mathcal{M}(a_r,\dots,a_1))
+\sum_{0<l<r}\mathcal{M}(a_{r},\dots,a_{l+1})\circ \mathcal{M}(a_l,\dots,a_1)\\
=
\sum_{1\leq i<j\leq r}\mathcal{M}(a_r,\dots,\mutw_{j-i+1}(a_j,\dots,a_i),\dots,a_1)
\end{split}
\end{equation*}
for all sequences of morphisms $(a_1,\dots,a_r)$ and $r\geq1$. This follows in the same way as above, except that there is no contributing term as
$$\mutw_{r+2}(a_r,\dots,a_1,g,\mutw_{4}(u,u,u,u))=\mutw_{r+2}(a_r,\dots,a_1,g,1)=0.$$
Next, consider $\mathcal{L}$. The only non-zero terms in $\mutw_1(\partial_{\mathcal{L}(M)})$ correspond to $\mu_2(\id,\id)$, $\mu_4(u,u,u,u)$, $\mu_4(v,v,v,v)$, $\mu_2(\id,u)$, $\mu_2(u,\id)$, $\mu_2(\id,v)$, $\mu_2(v,\id)$. The first three cases constitute the identity component $\id_{L_i}$, which vanishes as $\partial_M^2=p^4+q^4$. The terms in the fourth and fifth case cancel each other, and so do those in the last two. The $\delta$-grading defines a filtration on $\mathcal{L}(M)$ which ensures that the structure maps are upper-triangular.\\
Since $\pqMod$ has no higher multiplications and $\mathcal{L}$ vanishes on morphism sequences of length greater than 1, there are only two non-trivial $A_\infty$-relations to check:
$$
\mutw_1(\mathcal{L}(a))=\mathcal{L}(D(a))
\quad\text{and}\quad
\mutw_2(\mathcal{L}(a_1),\mathcal{L}(a_2))=\mathcal{L}(a_1\circ a_2),
$$
both of which are straightforward to check. 
\end{proof}

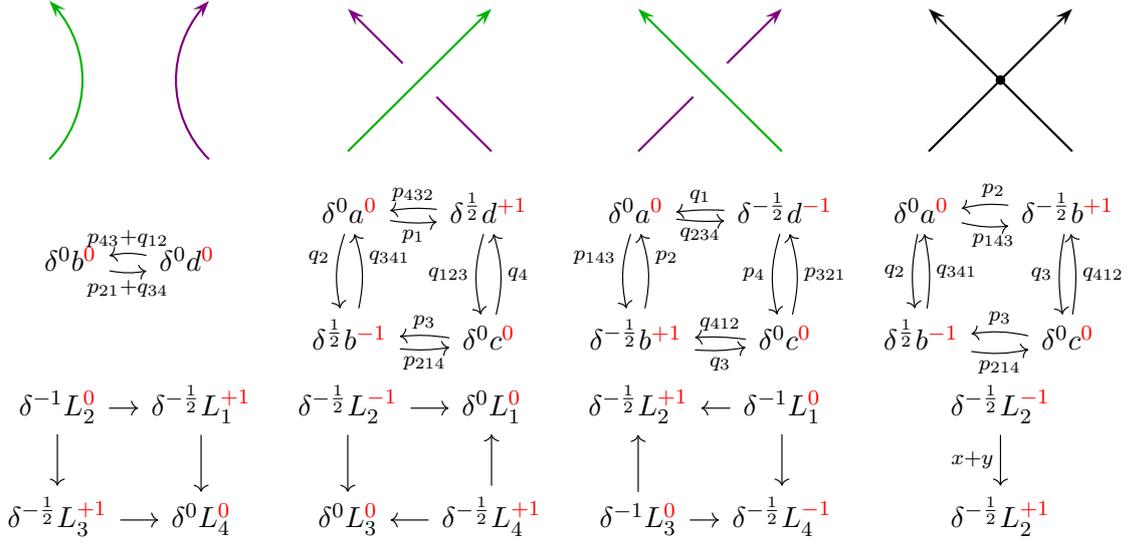
\begin{figure}[ht]
\centering
\psset{unit=0.15}
\begin{subfigure}[b]{0.25\textwidth}\centering
\psset{unit=7}
\begin{pspicture}(-1,-1)(1,1)
\psarc[linecolor=violet]{<-}(2,0){1.41432}{-225}{-135}
\psarc[linecolor=darkgreen]{->}(-2,0){1.41432}{-45}{45}
\end{pspicture}
\end{subfigure}\!\!
\begin{subfigure}[b]{0.25\textwidth}\centering
\psset{unit=7}
\begin{pspicture}(-1,-1)(1,1)
\psline[linecolor=violet]{->}(0.9,-0.9)(-0.9,0.9)
\pscircle*[linecolor=white](0,0){0.3}
\psline[linecolor=darkgreen]{->}(-0.9,-0.9)(0.9,0.9)
\end{pspicture}
\end{subfigure}\!\!
\begin{subfigure}[b]{0.25\textwidth}\centering
\psset{unit=7}
\begin{pspicture}(-1,-1)(1,1)
\psline[linecolor=violet]{->}(-0.9,-0.9)(0.9,0.9)
\pscircle*[linecolor=white](0,0){0.3}
\psline[linecolor=darkgreen]{->}(0.9,-0.9)(-0.9,0.9)
\end{pspicture}
\end{subfigure}\!\!
\begin{subfigure}[b]{0.25\textwidth}\centering
\psset{unit=7}
\begin{pspicture}(-1,-1)(1,1)
\psline{->}(-0.9,-0.9)(0.9,0.9)
\psline{->}(0.9,-0.9)(-0.9,0.9)
\psdot(0,0)
\end{pspicture}
\end{subfigure}
\\
\begin{subfigure}[b]{0.25\textwidth}
$$
\begin{tikzcd}[row sep=1.0cm, column sep=0.5cm]
\delta^{0}b^\Red{0}
\arrow[leftarrow,bend left=12]{r}{p_{43}+q_{12}}
& 
\delta^{0}d^\Red{0}
\arrow[leftarrow,bend left=10]{l}{p_{21}+q_{34}}
\end{tikzcd}
$$
\end{subfigure}\!\!
\begin{subfigure}[b]{0.25\textwidth}
$$
\begin{tikzcd}[row sep=1.0cm, column sep=0.5cm]
\delta^{0}a^\Red{0}
\arrow[leftarrow,bend left=7]{r}{p_{432}}
\arrow[leftarrow,bend left=15]{d}{q_{341}}
& 
\delta^{\frac{1}{2}}d^\Red{+1}
\arrow[leftarrow,bend left=7]{l}{p_1}
\arrow[leftarrow,bend left=15]{d}{q_4}
\\
\delta^{\frac{1}{2}}b^\Red{-1}
\arrow[leftarrow,bend left=7]{r}{p_3}
\arrow[leftarrow,bend left=15]{u}{q_2}
&
\delta^{0}c^\Red{0}
\arrow[leftarrow,bend left=7]{l}{p_{214}}
\arrow[leftarrow,bend left=15]{u}{q_{123}}
\end{tikzcd}
$$
\end{subfigure}\!\!
\begin{subfigure}[b]{0.25\textwidth}
$$
\begin{tikzcd}[row sep=1.0cm, column sep=0.4cm]
\delta^{0}a^\Red{0}
\arrow[leftarrow,bend left=5]{r}{q_1}
\arrow[leftarrow,bend left=15]{d}{p_2}
& 
\delta^{-\frac{1}{2}}d^\Red{-1}
\arrow[leftarrow,bend left=5]{l}{q_{234}}
\arrow[leftarrow,bend left=15]{d}{p_{321}}
\\
\delta^{-\frac{1}{2}}b^\Red{+1}
\arrow[leftarrow,bend left=5]{r}{q_{412}}
\arrow[leftarrow,bend left=15]{u}{p_{143}}
&
\delta^{0}c^\Red{0}
\arrow[leftarrow,bend left=7]{l}{q_3}
\arrow[leftarrow,bend left=15]{u}{p_4}
\end{tikzcd}
$$
\end{subfigure}\!\!
\begin{subfigure}[b]{0.25\textwidth}
$$
\begin{tikzcd}[row sep=1.0cm, column sep=0.5cm]
\delta^{0}a^{\red 0}
\arrow[bend left=10,leftarrow,pos=0.5]{d}{q_{341}}
\arrow[bend right=10,swap,pos=0.5]{d}{q_2}
\arrow[bend left=10,leftarrow,pos=0.7]{r}{p_2}
\arrow[bend right=10,swap,pos=0.7]{r}{p_{143}}
&
\delta^{-\frac{1}{2}}b^{\red +1}
\\
\delta^{\frac{1}{2}}b^{\red -1}
&
\delta^{0}c^{\red 0}
\arrow[bend left=10,leftarrow,pos=0.5]{u}{q_{3}}
\arrow[bend right=10,swap,pos=0.5]{u}{q_{412}}
\arrow[bend left=10,leftarrow,pos=0.5]{l}{p_{214}}
\arrow[bend right=10,swap,pos=0.5]{l}{p_3}
\end{tikzcd}
$$
\end{subfigure}
\\\vspace*{-0.3cm}
\begin{subfigure}[b]{0.25\textwidth}
$$
\begin{tikzcd}[column sep=0.25cm]    
\delta^{-1}L_2^\Red{0} \arrow{d}\arrow{r} & \delta^{-\frac{1}{2}}L_1^\Red{+1} \arrow{d}\\
\delta^{-\frac{1}{2}}L_3^\Red{+1} \arrow{r} & \delta^{0}L_4^\Red{0}
\end{tikzcd}
$$
\end{subfigure}\!\!
\begin{subfigure}[b]{0.25\textwidth}
$$
\begin{tikzcd}[column sep=0.25cm]    
\delta^{-\frac{1}{2}}L_2^\Red{-1} \arrow{r}\arrow{d} & \delta^0 L_1^\Red{0} \\
\delta^{0}L_3^\Red{0} & \delta^{-\frac{1}{2}}L_4^\Red{+1} \arrow{u}\arrow{l}
\end{tikzcd}
$$
\end{subfigure}\!\!
\begin{subfigure}[b]{0.25\textwidth}
$$
\begin{tikzcd}[column sep=0.25cm]    
\delta^{-\frac{1}{2}} L_2^\Red{+1} & \delta^{-1}L_1^\Red{0} \arrow{l}\arrow{d}\\
\delta^{-1} L_3^\Red{0} \arrow{u}\arrow{r} & \delta^{-\frac{1}{2}} L_4^\Red{-1}
\end{tikzcd}
$$
\end{subfigure}\!\!
\begin{subfigure}[b]{0.25\textwidth}
$$
\begin{tikzcd}[column sep=0.25cm]  
\delta^{-\frac{1}{2}}L_2^\Red{-1} \arrow[swap]{d}{x+y}\\
\delta^{-\frac{1}{2}}L_2^\Red{+1}
\end{tikzcd}
$$
\end{subfigure}
\caption{Some basic rational tangles, their peculiar modules and the images thereof under~$\mathcal{L}$. The superscripts of the generators specify the Alexander grading. The last column shows one of the two possible invariants for the singular crossing from proposition~\ref{prop:singularcrossing}.}\label{fig:MOY4basicobjects}
\end{figure}

\begin{example}
Figure~\ref{fig:MOY4basicobjects} shows the images of the peculiar modules corresponding to some basic tangles under the functor $\mathcal{L}$ after doing some cancellation. Conversely, we can recover the original peculiar modules by applying $\mathcal{M}$ to these twisted complexes and, again, doing some cancellation. This seems to be sufficient computational evidence to justify the following conjecture.
\end{example}

\begin{conjecture}\label{conj:TwFukpqModEquivalent}
The functors \(\mathcal{M}\) and \(\mathcal{L}\) give rise to an equivalence of \(A_\infty\)-categories.
\end{conjecture}

As a first step towards a proof of this conjecture, we offer the following proposition. Its proof is an exercise in cancellation. The conjecture should follow from studying functoriality properties of cancellation and a reformulation of these in the $A_\infty$-context.

\begin{proposition}\label{prop:CFTsiteFromFukCFTd}
Let \(T\) be a 4-ended tangle. Then the tangle Floer complex \(\CFT(T,a)\) is bigraded chain homotopic to 
$$\left(\Mor\left(L_1,\mathcal{L}\left(\CFTd(T)\right)\right),\mutw_1\right).$$
By symmetry, the corresponding statement holds for all other three sites. 
\end{proposition}
\begin{proof}
Let $M:=\CFTd(T)$. On the one hand, $\CFT(T,a)$ is obtained from $\iota_1.M$ by setting $p_i=q_i=0$ for $i=1,2,3,4$. On the other hand, we can view $\iota_1.M$ as a subspace of 
\begin{equation}\label{eqn:CFTsiteFromFukCFTdModule}
\Mor(L_1,\mathcal{L}(M))=\bigoplus_{j=1,2,3,4}\iota_j.\Adop\otimes_{\Id}M\otimes\Mor(L_1,L_j)
\end{equation}
via
$$\iota_1.m\mapsto (1)^{-1}\otimes m\otimes 1_{L_1}.$$
The differential on this complex is given by
\begin{align*}
\mutw_1:~a\otimes m\otimes g\mapsto & ~a.\partial_M(m)\otimes g+ p.a\otimes m\otimes \mu_2(u,g)+p^3.a\otimes m\otimes \mu_4(u,u,u,g)\\
& \phantom{~a.\partial(m)\otimes g}+ q.a\otimes m\otimes \mu_2(v,g)+q^3.a\otimes m\otimes \mu_4(v,v,v,g)
\end{align*}
The idea is to cancel all other generators in this chain complex using the cancellation lemma~\ref{lem:AbstractCancellation}. For this, we split the vector space from (\ref{eqn:CFTsiteFromFukCFTdModule}) into the following three components:
\begin{align*}
Z =& \left\langle\left.(1)^{-1}\otimes m\otimes 1_{L_1}\right\vert m\in\iota_1.M\right\rangle
\end{align*}
\begin{eqnarray*}
Y_1\!\!\!\! &=&\!\!\!\! \left\langle\iota_1 .p^{-i}\otimes M\otimes 
\left.\raisebox{-0.36cm}{\psset{unit=0.6}
\begin{pspicture}(-0.3,-0.3)(1.3,1.3)
\wi{$y^j$}(0,0.5)
\end{pspicture}}
\right\vert i>0,j\geq0\right\rangle\oplus \left\langle\iota_4 .p^{-i}\otimes M\otimes 
\left.\raisebox{-0.36cm}{\psset{unit=0.6}
\begin{pspicture}(-0.3,-0.3)(1.3,1.3)
\wv{$j$}(0,1)
\end{pspicture}}
\right\vert i>0,j\geq0\right\rangle\\
&&\!\!\!\! \oplus \left\langle\iota_1 .q^{-i}\otimes M\otimes 
\left.\raisebox{-0.36cm}{\psset{unit=0.6}
\begin{pspicture}(-0.3,-0.3)(1.3,1.3)
\wi{$x^j$}(0,0.5)
\end{pspicture}}
\right\vert i>0,j\geq0\right\rangle\oplus \left\langle\iota_2 .q^{-i}\otimes M\otimes 
\left.\raisebox{-0.36cm}{\psset{unit=0.6}
\begin{pspicture}(-0.3,-0.3)(1.3,1.3)
\wu{$j$}(0,0)
\end{pspicture}}
\right\vert i>0,j\geq0\right\rangle\\
Y_2\!\!\!\! &=& \!\!\!\!\left\langle\iota_2 .p^{-i}\otimes M\otimes 
\left.\raisebox{-0.36cm}{\psset{unit=0.6}
\begin{pspicture}(-0.3,-0.3)(1.3,1.3)
\wu{$j$}(0,0)
\end{pspicture}}
\right\vert i\geq0,j\geq0\right\rangle\oplus \left\langle\iota_1 .p^{-i}\otimes M\otimes 
\left.\raisebox{-0.36cm}{\psset{unit=0.6}
\begin{pspicture}(-0.3,-0.3)(1.3,1.3)
\wi{$x^j$}(0,0.5)
\end{pspicture}}
\right\vert i\geq0,j>0\right\rangle\\
&&\!\!\!\! \oplus \left\langle\iota_4 .q^{-i}\otimes M\otimes 
\left.\raisebox{-0.36cm}{\psset{unit=0.6}
\begin{pspicture}(-0.3,-0.3)(1.3,1.3)
\wv{$j$}(0,1)
\end{pspicture}}
\right\vert i\geq0,j\geq0\right\rangle\oplus \left\langle\iota_1 .q^{-i}\otimes M\otimes 
\left.\raisebox{-0.36cm}{\psset{unit=0.6}
\begin{pspicture}(-0.3,-0.3)(1.3,1.3)
\wi{$y^j$}(0,0.5)
\end{pspicture}}
\right\vert i\geq0,j>0\right\rangle
\end{eqnarray*}
$\mutw_1(Z)\subseteq Z$, so in particular the map $c:Z\rightarrow Y_2$ vanishes. In fact, the restriction map $\mutw_1\vert_Z$ only counts the identity components of $\partial_M$, so it agrees with the differential on $\CFT(T,a)$. So if we can show that $f=\mutw_1\vert_{Y_1\rightarrow Y_2}$ is an isomorphism, we are done.\\
Let us write $f=f_0+f_1$, where $f_1$ is the first component of $\mutw_1$. We claim that $f_0$ sets up an identification of $Y_1$ with $Y_2$ as vector spaces. Indeed: the second component of $\mutw_1$ sets up a bijection on the level of generators between the first two summands of $Y_1$ and $Y_2$; similarly, the fourth component of $\mutw_1$ identifies the other two summands; finally, the third and fifth components of $\mutw_1$ vanish on $Y_1$.\\
Let $g_0$ be the inverse of $f_0$ and define
$$
g:=g_0+g_0f_1g_0+g_0f_1g_0f_1g_0+g_0f_1g_0f_1g_0f_1g_0+\dots
$$
$g$ is a well-defined homomorphism, since $g_0$ lowers the $\delta$-grading of the last tensor factor, $f_1$ preserves it and the $\delta$-grading on $\Mor(L_1,L_j)$ is bounded below. Obviously, $g$ is an inverse of $f$.
\end{proof}

\subsection*{Outlook. }
Ideally, we would like to reformulate the pairing theorem~\sref{thm:CFTdGeneralGlueing} for our peculiar invariants in terms of Lagrangian intersection homology as follows.
\begin{conjecture}\label{conj:GlueingCFTdFUK}
Let \(L\), \(T_1\) and \(T_2\) be as in theorem~\sref{thm:CFTdGeneralGlueing}. Then \(\CFL(L)\) is bigraded chain homotopic to
$$\left(\Mor\left(\mathcal{L}\left(\CFTd(\m(T_1))\right),\mathcal{L}\left(\CFTd(T_2)\right)\right),\mutw_1\right).$$
\end{conjecture}
It would be interesting to see if one can define a functor similar to $\mathcal{L}$ for type AA structures and what the image of the glueing structure $\mathcal{P}$ from theorem~\sref{thm:CFTdGeneralGlueing} would look like under this functor.
\begin{question}\label{que:FUKlooptype}
Does \(\mathcal{L}\) send loop-type peculiar modules to twisted complexes representing closed curves and vice versa?
\end{question}
\begin{Remark}\label{rem:FUKlooptype}
According to a classification result from \cite[theorem~4.3]{Kontsevich}, any twisted complex in a (partially) wrapped Fukaya category of a punctured surface either represents a closed curve on this surface or an arc connecting two punctures (up to the issue of local systems). If such a classification is also true for our fully wrapped Fukaya category $\Fuk$ we might be able to show that all peculiar modules of tangles are loop-type (question~\sref{que:IsEverythingALoop}), using the following argument:\pagebreak[3] If a peculiar module $M$ corresponds to an arc, the intersection homology with some $L_i$ is an infinite-dimensional vector-space. If at the same time $M$ were homotopic to $\CFTd(T)$ for some 4-ended tangle, then, by proposition~\ref{prop:CFTsiteFromFukCFTd}, this vector space would be isomorphic to $\HFT(T,s_i)$ for the site $s_i$ corresponding to $L_i$. But the tangle Floer homology is always finite-dimensional, so we have a contradiction.
\end{Remark}

%% file: sections/A_AlgebraicStructures.tex
\chapter{Algebraic structures from dg categories}\label{appendix:AlgStructFromGDCats}
In chapters~\ref{chapter:categorification} and~\ref{chapter:HFTd}, we often work in categories of various algebraic structures, namely type A, type AA, type D and curved type D structures. In all four settings, we often want to simplify these structures by replacing them by homotopy equivalent ones. The main goal of this appendix is to develop some tools for dealing with this problem, namely the cancellation lemma (\ref{lem:AbstractCancellation}) and the clean-up lemma (\ref{lem:AbstractCleanUp}). The former can be used to reduce the number of generators of an algebraic structure, the latter for making the structure maps ``look nicer'', essentially by changing the basis. \\
In the category of ordinary chain complexes, both tools will be familiar to the reader as easy exercises in linear algebra. So it might not be too surprising that they also work in quite general settings. We will spend the first part of this appendix explaining a general construction which turns any differential graded category into another such category in which the lemmas hold in some generality sufficient for our purposes, see definitions~\ref{def:CatOfMatrices} and~\ref{def:CatOfComplexes}. Next, we show that the various different algebraic structures mentioned above arise naturally from this general construction. Finally, we state and prove the cancellation and clean-up lemmas in this general framework.\\
For simplicity, we only work over the field $\mathbb{F}_2=\mathbb{Z}/2$, so we do not need to keep track of signs. However, with the correct sign conventions, all statements should also hold over fields of arbitrary characteristic.

\begin{definition}
Let $\Com$ be the category of $\mathbb{Z}$-graded chain complexes over $\mathbb{F}_2$ and grading preserving chain maps between them. A \textbf{differential graded (dg) category}~$\mathcal{C}$ over~$\mathbb{F}_2$ is an enriched category over $\Com$. To spell this out more explicitly, the hom-objects are $\mathbb{Z}$-graded $\mathbb{F}_2$-vector spaces,
$$\Mor(A,B)=\bigoplus_{i\in\mathbb{Z}}\Mor_i(A,B)$$
endowed with differentials
$$\partial_i:\Mor_i(A,B)\rightarrow \Mor_{i-1}(A,B),$$
i.\,e.\ vector space homomorphisms satisfying 
$\partial_{i-1}\partial_i=0$
and
\begin{equation}\label{eqn:CompatibleWithComposition}
\partial\circ m=m\circ(\partial\otimes \id+\id\otimes \partial),
\end{equation}
where
$$m: \Mor_i(A,B)\otimes\Mor_j(B,C)\rightarrow\Mor_{i+j}(A,C)$$
denotes composition in $\mathcal{C}$, which is associative and unital. For more details on enriched categories, see for example \cite{cathtpy}. Note that the identity morphisms have degree zero and lie in the kernel of~$\partial$.
\end{definition}
\begin{definition}\label{def:UnderlyingOrdinaryCat}\cite[definition~3.4.5]{cathtpy}. 
Given an enriched category $\mathcal{C}$ over some monoidal category $\mathcal{V}$, the \textbf{underlying ordinary category} $\mathcal{C}_0$ of $\mathcal{C}$ has the same objects as $\mathcal{C}$ and its hom-sets are defined by 
$$\mathcal{C}_0(A,B):=\Mor_\mathcal{V}(1_\mathcal{V},\mathcal{C}(A,B)).$$
\end{definition}
\begin{example}\label{exa:UnderlyingOrdCat}
Let $\mathcal{C}$ be a dg category. The unit in $\Com$ is the complex $0\rightarrow\mathbb{F}_2\rightarrow0$, supported in homological degree 0, and the morphisms in $\Com$ are grading preserving. Hence, the hom-sets of $\mathcal{C}_0$ are those elements in the kernel of $\partial_0$.\\
Consider the enriched category $H_\ast(\mathcal{C})$ over the category of graded vector spaces and grading preserving morphisms between them, obtained from $\mathcal{C}$ by replacing the hom-objects by their homologies with respect to the differential $\partial$. By passing to the underlying ordinary category, we pick out the degree~0 morphisms in $H_\ast(\mathcal{C})$. Therefore, we denote this category by $H_0(\mathcal{C})$.\\
Since the hom-sets in $H_0(\mathcal{C})$ are just quotients of those in $\mathcal{C}_0$, we now get the usual notions of chain homotopies between morphisms and objects. The reason why we need to pass to the underlying category is that otherwise, two objects could be (chain) isomorphic through grading shifting morphisms. 
\end{example}
\begin{definition}\label{def:CatOfMatrices}
(cp.~\cite[section~6]{BarNatanKhT})
Given a dg category $\mathcal{C}$, we define another dg category $\mathfrak{Mat}(\mathcal{C})$ as follows. Its objects are formal direct sums
$$\bigoplus_{i\in I}O_i[n_i],$$
where $I$ is some index set and $O_i[n_i]$ denotes the object $O_i\in\obj(\mathcal{C})$ with a formal grading shift by an integer $n_i$. Morphisms are given by
$$\Mor_n(\bigoplus_{i\in I}O_i[n_i],\bigoplus_{j\in J}O_j[n_j]):=\bigoplus_{(i,j)\in I\times J}\Mor_{n+n_i-n_j}(O_i,O_j).$$
Compositions and differentials in $\mathfrak{Mat}(\mathcal{C})$ are induced by those in $\mathcal{C}$.
\end{definition}
\begin{definition}\label{def:CatOfComplexes}
Given a differential graded category $\mathcal{C}$, we define an auxiliary category $\mathfrak{Cx}^{pre}(\mathcal{C})$, \textbf{the category of pre-complexes}, which is an enriched category over the category of $\mathbb{Z}$-graded vector spaces and grading preserving morphisms between them. Its objects are pairs $(O,d_O)$, where $O\in\obj(\mathcal{C})$ and $d_O\in\Mor_{-1}(O,O)$. 
The hom-objects are the same as in $\mathcal{C}$,\vspace*{-0.3cm}
$$\Mor((O,d_O),(O',d_{O'}))=\Mor(O,O'),$$
viewed as $\mathbb{Z}$-graded vector spaces. On these, we can define a map
$$D:\Mor_{i}((O,d_O),(O',d_{O'}))\rightarrow \Mor_{i-1}((O,d_O),(O',d_{O'}))$$
by setting 
$$D(f):=d_{O'}\circ f+f \circ d_O+\partial(f).$$
We would like $D$ to be a differential in order to turn $\mathfrak{Cx}^{pre}(\mathcal{C})$ into a dg category. However, this only works in general if we restrict ourselves to a full subcategory of $\mathfrak{Cx}^{pre}(\mathcal{C})$: \\
It is easy to check that $D$ is always compatible with multiplication in the sense of (\ref{eqn:CompatibleWithComposition}). So $D$ is a differential iff
$$D^2(f)=(d_{O'}^2+\partial(d_{O'}))\circ f
+f\circ (d_O^2+\partial(d_O))$$
vanishes. This is, of course, the case for the full subcategory $\mathfrak{Cx}^{0}(\mathcal{C})$ of $\mathfrak{Cx}^{pre}(\mathcal{C})$ consisting of those objects $(O,d_O)$ for which
\begin{equation} 
d^2_O+\partial(d_O) \tag{$\ast$}\label{eqn:d2term}
\end{equation}
vanishes. However, in some situations, other conditions on (\ref{eqn:d2term}) also  work. For example, if we replace $\Com$ by the category of $\mathbb{Z}/2$-graded chain complexes, we can restrict to those objects $(O,d_O)$ for which (\ref{eqn:d2term}) is equal to the identity. Also, if the hom-objects are bimodules over an algebra~$\mathcal{A}$, we can ask (\ref{eqn:d2term}) to be equal to $a.\id_O$ for a fixed central algebra element $a$ (of degree $-2$) which commutes with all morphisms~$f$. In both cases, $D$ will be a differential.\\
In any of these cases, the resulting category is a differential graded category again and we call any such full subcategory \textbf{a category of complexes}, denoted by $\mathfrak{Cx}^{\ast}(\mathcal{C})$, where $\ast\in\{0,1,a\}$ is the value of~(\ref{eqn:d2term}). 
\end{definition}

\pagebreak[3]

\begin{Remark}\label{exa:GraphsForAlgebraicStructures}
As usual, we can associate a directed graph to a category, where objects correspond to vertices and arrows to  morphisms. In the same way, we can think of complexes in $\mathfrak{Cx}^{\ast}(\mathfrak{Mat}(\mathcal{C}))$ as graphs. We often label the arrows by the morphisms. 
\end{Remark}
The point of the construction above is that after choosing a basis, we can interpret the categories of type~D, type~A, type~AA and curved type D structures as instances of $\mathfrak{Cx}^{\ast}(\mathfrak{Mat}(\mathcal{C}))$ for suitable choices of relatively simple differential graded categories $\mathcal{C}$. But let us start with an even simpler example: ordinary chain complexes.
\subsection*{Note of warning.} In the following examples, our definitions only coincide with the usual ones after passing to the underlying ordinary categories, see example~\ref{exa:UnderlyingOrdCat}. The advantage of our point of view is that the conditions we usually impose on morphism and chain homotopies for various algebraic structures arise naturally by viewing those morphisms as elements of chain complexes.  
\begin{example}[ordinary chain complexes over $\mathbb{F}_2$] \label{exa:HighBrowDefChainCxs}
Let $\mathcal{C}$ be the category with a single object $\bullet$ in grading 0, $\Mor(\bullet,\bullet)=\mathbb{F}_2$ and vanishing differential. Then (the underlying ordinary category of) $\mathfrak{Cx}^{0}(\mathfrak{Mat}(\mathcal{C}))$ is $\mathfrak{Com}$.
\end{example}
\begin{example}[type D modules over dg $\mathbb{F}_2$-algebras]\label{exa:HighBrowDefTypeDoverF2}
Let $\mathcal{A}$ be a differential graded algebra over $\mathbb{F}_2$. Let $\mathcal{C}$ be the category with a single object~$\bullet$ and morphisms being elements in $\mathcal{A}$. Composition is multiplication in $\mathcal{A}$ and the differential~$\partial$ is induced by the differential on $\mathcal{A}$. We define the category of type D modules by $\mathfrak{Cx}^{0}(\mathfrak{Mat}(\mathcal{C}))$. (Again, note that we need to pass to the underlying ordinary category to obtain the definitions in \cite{Zarev} and \cite{LOT}.)
\end{example}

\begin{example}[type D modules over dg $\mathcal{I}$-algebras]\label{exa:HighBrowDefTypeDoverI}
Let us assume that $\mathcal{A}$ is an algebra over some ring $\mathcal{I}\subseteq\mathcal{A}$ of idempotents and fix a basis $\{i_j\}_{j\in J}$ of idempotents of $\mathcal{I}$, where $J$ is some index set. Let $\mathcal{C}$ be the category with one object for each basis element of $\mathcal{I}$, and for any two such elements $i_1$ and $i_2$, let $\Mor(i_1,i_2):=i_1.\mathcal{A}.i_2$, viewed as a quotient of $\mathcal{A}$. 
Again, composition is multiplication in $\mathcal{A}$ and the differential $\partial$ is induced by the differential on $\mathcal{A}$. We define the category of type D modules over dg $\mathcal{I}$-algebras by $\mathfrak{Cx}^{0}(\mathfrak{Mat}(\mathcal{C}))$.  
\end{example}
\begin{Remark}
\textit{A priori}, the definition in the previous example depends on a choice of basis for~$\mathcal{I}$. In the examples that we see in chapters~\ref{chapter:categorification} and~\ref{chapter:HFTd}, there is a natural choice of such a basis, so this is not an issue.\\
However, we can replace $\mathcal{C}$ above by the enlarged category $\mathcal{C}_\mathcal{I}$, where there is an object for \textit{every} element in $\mathcal{I}$. 	Then $\mathcal{C}$ is a full subcategory of $\mathcal{C}_\mathcal{I}$ and it is not hard to see that $\mathfrak{Mat}(\mathcal{C})$ and $\mathfrak{Mat}(\mathcal{C}_\mathcal{I})$ are equivalent. Now, the construction of the category of complexes is functorial (in the category of dg categories), so after all, the definition above does not depend on a basis for $\mathcal{I}$. 
\end{Remark}
\begin{example}[curved type D modules over dg $\mathcal{I}$-algebras]\label{exa:HighBrowDefcurvedTypeD}
We start with the same category $\mathcal{C}$ as in the previous example, but we fix a central element $a_c\in Z(\mathcal{A})$, the curvature, and define the category of curved type D modules with curvature $a_c$ as $\mathfrak{Cx}^{a_c}(\mathfrak{Mat}(\mathcal{C}))$. For a more explicit, but less concise definition, see definition~\sref{def:curvedTypeDStructure}.
\end{example}
\begin{Remark}
In the Heegaard Floer community, the term ``curved'' seems to be the accepted attribute for algebraic structures for which some differential is non-vanishing; however, the first written reference (that I am aware of) in which this terminology is used is of very recent date \cite{Zemke}.
\end{Remark}
\begin{example}[type A structures over an $A_\infty$-algebra over $\mathcal{I}$]\label{exa:HighBrowDefTypeAoverI}
Let $\mathcal{A}$ be an $A_\infty$-algebra over a ring of idempotents $\mathcal{I}$ over $\mathbb{F}_2$. As in example~\ref{exa:HighBrowDefTypeDoverI}, fix a basis $\{i_k\}_{k\in I}$ of idempotents of $\mathcal{I}$, where $I$ is some index set. Let $\mathcal{C}$ be the category with one object for each basis element in $\mathcal{I}$, just as for type D structures. However, a morphism in a hom-object $\Mor(i_1,i_2)$ of $\mathcal{C}$ is given by a sequence of vector space homomorphisms 
$$(f_i: i_1.\mathcal{A}^{\otimes (i-1)}.i_2\rightarrow \mathbb{F}_2)_{i\geq1}$$
where composition is defined by 
$$(f\circ g)_i:= \sum_{j+k=i+1}f_j\circ(g_k\otimes\id_{\mathcal{A}^{\otimes (j-1)}}).$$
The differential $\partial$ is given by
$$(\partial(f))_i(a_1\otimes\cdots\otimes a_{i-1}):=\sum_{j+k=i+1}\sum_{l=1}^{i-k} f_j(a_1\otimes\dots\otimes \mu_k(a_l\otimes \dots \otimes a_{l+k-1})\otimes \dots \otimes a_{i-1}).$$
We define the category of type A structures by $\mathfrak{Cx}^{0}(\mathfrak{Mat}(\mathcal{C}))$. 
We define the category of strictly unital type A structures by restricting to those objects $(O,d_O)$ such that 
$$d_O(\cdot,1)=id_O$$
and 
$$d_O(\cdot,a_1\otimes\dots\otimes a_{i-1})=0\text{ if $i>2$ and $a_j=1$ for some $j=1,\dots,i-1$}$$
and morphisms to those satisfying
$$f(\cdot,a_1\otimes\dots\otimes a_{i-1})=0\text{ if $i>1$ and $a_j=1$ for some $j=1,\dots,i-1$}.$$
We say a type A structure $(O,d_O)$ is bounded if $(d_O)_i=0$ for sufficiently large $i$.
\end{example}

\pagebreak[3]

\begin{Remark}\label{exa:GraphsForTypeAstructures}
When we describe type A structures as directed graphs, it is useful to fix a basis of the algebra $\mathcal{A}$. Then, we label an arrow corresponding to a morphism $f$ by the formal sum of those tuples/tensor products of basis elements of the algebra $\mathcal{A}$ on which $f$ is non-zero. In this language, composition of two morphisms $f$ and $g$ can be described as the sum of all concatenations of labels for $f$ and $g$ (modulo~2).\\
To describe the differential in these terms, we introduce the following notation. For a tuple of basic algebra elements $a=(a_1,\dots,a_{i-1})$, define
$$M_i(a):=\sum_{j+k=i+1}\sum_{l=1}^{i-k} (a_1,\dots, \mu_k(a_l\otimes \dots \otimes a_{l+k-1}),\dots, a_{i-1}).$$
Now consider a morphism $f_a$ whose only label is $a$. Then the arrow of $\partial(f_a)$ is labelled by all tuples of basis algebra elements $b=(b_1,\dots,b_{j-1})$ for which the $a$-component of $M_j(b)$ is~1.
\end{Remark}

\begin{example}[type AA bimodules]\label{exa:HighBrowDefTypeAAoverI}
Let $\mathcal{A}$ and $\mathcal{B}$ be two $A_\infty$-algebra over rings of idempotents $\mathcal{I}$ and $\mathcal{J}$ over $\mathbb{F}_2$, respectively. Fix a basis $\{i_k\}_{k\in I}$ of $\mathcal{I}$ and $\{j_l\}_{l\in J}$ of $\mathcal{J}$, where $I$ and $J$ are some index sets. Let the objects in $\mathcal{C}$ be of the form $(i_k,j_l)$ for some $(i,j)\in I\times J$. A morphism in $\Mor((i_k,j_l),(i_{k'},j_{l'}))$ is given by a sequence of vector space homomorphisms 
$$(f_{m,n}: i_{k'}.\mathcal{A}^{\otimes (m-1)}.i_k\times j_l.\mathcal{B}^{\otimes (n-1)}.j_{l'}\rightarrow \mathbb{F}_2)_{m,n\geq1}$$
where composition is given by 
$$(f\circ g)_{m,n}:= \sum_{\substack{i+k=m+1\\j+l=n+1}}f_{i,j}\circ(\id_{\mathcal{A}^{\otimes (i-1)}}\otimes g_{k,l}\otimes\id_{\mathcal{B}^{\otimes (j-1)}}).$$
For $a=a_1\otimes\dots\otimes a_{m-1}$ and $b=b_1\otimes\dots\otimes b_{m-1}$, the differential $\partial$ is given by
\begin{eqnarray*}
(\partial(f))_{m,n}(a,b)&:=&\sum_{j+k=m+1}\sum_{l=1}^{i-k} f_{j,n}(a_1\otimes\dots\otimes \mu_k(a_l\otimes \dots \otimes a_{l+k-1})\otimes \dots \otimes a_{i-1},b)\\
&& +\sum_{j+k=n+1}\sum_{l=1}^{i-k} f_{m,j}(a,b_1\otimes\dots\otimes \mu_k(b_l\otimes \dots \otimes b_{l+k-1})\otimes \dots \otimes b_{i-1}).
\end{eqnarray*}
We define the category of type AA $\mathcal{A}$-$\mathcal{B}$-bimodules by $\mathfrak{Cx}^{0}(\mathfrak{Mat}(\mathcal{C}))$. We define the category of strictly unital type AA structures by restricting to those objects $(O,d_O)$ such that 
$$d_O(\cdot,1)=id_O=d_O(1,\cdot)$$
and 
$$d_O(a_1\otimes\dots\otimes a_{i-1},\cdot,b_1\otimes\dots\otimes b_{j-1})=0\text{ if $i+j>3$ and ($a_j=1$ or $b_j=1$) for some $j$}$$
and morphisms to those satisfying
$$f(a_1\otimes\dots\otimes a_{i-1},\cdot,b_1\otimes\dots\otimes b_{j-1})=0\text{ if $i+j>2$ and ($a_j=1$ or $b_j=1$) for some $j$}.$$
As in remark~\ref{exa:GraphsForAlgebraicStructures}, we can easily translate all of this into the language of labelled graphs after fixing a basis of $\mathcal{A}$ and $\mathcal{B}$.
\end{example}
\pagebreak[3]
\begin{Remark}\label{rem:FunctorsBetweenDGCats}
In the proofs of the glueing results in section~\ref{sec:Pairing} and also in section~\ref{sec:SymRels}, we sometimes need to change the underlying algebra of the algebraic structures that we are working with by some algebra homomorphism $\pi: A\rightarrow B$ (which, in most cases, is a quotient map). In terms of the categorical description above, such a homomorphism corresponds to a functor 
$$\tilde{\mathcal{F}}_\pi:\mathcal{C}_A\rightarrow\mathcal{C}_B,$$
where $\mathcal{C}_A$ and $\mathcal{C}_B$ are the two categories corresponding to the algebras $A$ and $B$. If $\pi$ respects the differentials on $A$ and $B$, then so does $\tilde{\mathcal{F}}_\pi$. 
Then $\tilde{\mathcal{F}}_\pi$ in turn induces a functor 
$$\mathcal{F}_\pi:\mathfrak{Cx}^{\ast}(\mathfrak{Mat}(\mathcal{C}))\rightarrow\mathfrak{Cx}^{\ast}(\mathfrak{Mat}(\mathcal{D}))$$
that likewise respects the differentials on both sides. In particular it sends homotopic complexes to homotopic ones. Also note that a curved type D structure might become an ordinary type D structure, if $\pi(a_{tw})=0$.
\end{Remark}
\begin{definition}[pairing type D and type A structures]\label{def:PairingTypeDandA}
Let $M$ be a type A structure and $N$ a type D structure over the same dg algebra $\mathcal{A}$ over a ring $\mathcal{I}$ of idempotents, together with a fixed basis of $\mathcal{I}$ and $\mathcal{A}$. We now reformulate the definition of the chain complex $(M\boxtimes N,\partial^\boxtimes)$ from \cite[definition~7.4]{Zarev} and \cite[section~2.4]{LOT} in terms of the graphs associated to $M$ and $N$. The generators of $M\boxtimes N$ are defined by pairs of vertices in $M$ and $N$ labelled by the same idempotents. Given two such pairs $(v_1,w_1)$ and $(v_2,w_2)$, the $(v_2,w_2)$-component of $\partial^\boxtimes((v_1,w_1))$ is equal to the number of pairs $(s_1,s_2)$ (modulo 2), where $s_2$ is a sequence of labels of consecutive arrows along a path from~$w_1$ to~$w_2$ in~$N$, $s_1$ is a label on an arrow from $v_1$ to $v_2$ and $s_1=s_2$. \\
If we start with a type AA module and a type D module, the pairing is defined in the same way, except that we compare the sequences of labels in the type D structure only to one component of the labels of the type AA structure and record the other component in the output, which is a type A structure. 
\end{definition}
\begin{example}[twisted complexes]\label{exa:twistedcomplexes}
Let $(\mathcal{C},\mu)$ be an $A_\infty$-category. Then we can define $\mathfrak{Mat}(\mathcal{C})$ just as in~\ref{def:CatOfMatrices}, with higher multiplications defined in the obvious way. We can also generalise the definition of $\mathfrak{Cx}^0(\mathcal{C})$ from~\ref{def:CatOfComplexes} to the $A_\infty$-setting by modifying the differential $D$ such that we use all $A_\infty$-operations, 
$$D(f):=\sum_{k,l\geq0} \mu_{k+l+1}(\underbrace{d_{O'},\dots,d_{O'}}_{k},f,\underbrace{d_{O},\dots,d_{O}}_{l}),$$
and similarly generalise higher multiplications, 
while restricting to upper triangular precomplexes to make sure that all sums are finite. We denote the $A_\infty$-category $\mathfrak{Cx}^0(\mathfrak{Mat}(\mathcal{C}))$ by $\Tw\mathcal{C}$ and call it the category of twisted complexes over $\mathcal{C}$. The differential $D$ above is usually denoted by $\mutw_1$, multiplication by $\mutw_2$, and higher multiplications by $\mutw_i$, $i>2$; for more details, see \cite[section~I.3l]{Seidel}. In this setting, the cancellation lemma and clean-up lemma only hold under certain assumptions that, again, are needed to ensure that all sums are finite.
\end{example}

\pagebreak
We now state and prove the two central lemmas mentioned in the introduction.
\begin{lemma}[Cancellation Lemma]\label{lem:AbstractCancellation}
Let \((X,\delta)\) be an object of \(\mathfrak{Cx}^{\ast}(\mathfrak{Mat}(\mathcal{C}))\) for some differential graded category \(\mathcal{C}\) and suppose it has the form
$$\begin{tikzcd}[row sep=1.5cm, column sep=0.6cm]
& (Z,\zeta) \arrow[bend left=12]{ld}{a} \arrow[bend left=12]{rd}{c} \\
(Y_1,\varepsilon_1)\arrow[bend left=12]{rr}{f} \arrow[bend left=12]{ru}{b}
& &
(Y_2,\varepsilon_2) \arrow[bend left=12]{ll}{e}\arrow[bend left=12]{lu}{d}
\end{tikzcd}$$
where \((Y_1,\varepsilon_1)\), \((Y_2,\varepsilon_2)\), \((Z,\zeta)\in\obj\mathfrak{Cx}^{pre}(\mathfrak{Mat}(\mathcal{C}))\), \(f\) is an isomorphism with inverse \(g\). Then \((X,\delta)\) is chain homotopic to \((Z,\zeta+bgc)\).
\end{lemma}
\begin{Remark}
We usually apply this lemma to the case where $(Y_1,\varepsilon_1)=(Y_2,\varepsilon_2)$ and $f$ is the identity map.
\end{Remark}
\begin{proof}
First of all, let us check that $(Z,\zeta+bgc)$ is indeed an object of $\mathfrak{Cx}^{\ast}(\mathfrak{Mat}(\mathcal{C}))$:
\begin{align*}
(\zeta+bgc)^2+\partial(\zeta+bgc)=&~\zeta\zeta+\zeta bgc+bgc \zeta +bgcbgc+\partial(\zeta)+\partial(b)gc+b\partial(g)c+bg\partial(c)\\
=&~\zeta\zeta+(\zeta b +\partial(b))gc+bg(c \zeta +\partial(c)) +\partial(\zeta)+b\partial(g)c+bgcbgc\\
=&~\zeta\zeta+(b\varepsilon_1 +df)gc+bg(\varepsilon_2 c +fa) +\partial(\zeta)+b\partial(g)c+bgcbgc\\
=&~\zeta\zeta+dc+ba+\partial(\zeta)+\cancel{b\left(D(g)+gcbg\right)c}=\ast.
\end{align*}
For the last step, we observe that $$gcbg=gD(f)g=D(gfg)+D(g)fg+gfD(g)=D(g).$$
Next, we define two chain maps 
$$F:(Z,\zeta+bgc)\rightarrow(X,\delta)\qquad\text{and}\qquad G:(X,\delta)\rightarrow(Z,\zeta+bgc)$$
by
\begin{center}
$\begin{tikzcd}[row sep=1cm, column sep=-0.2cm]
(Z,\zeta+bgc)\arrow{rrr}{1}\arrow[bend right=12]{rrdd}{gc} &~~~~~~~ && (Z,\zeta) \arrow[bend left=12,pos=0.3]{ldd}{a} \arrow[bend left=12]{rrd}{c}\\
& &&&~& (Y_2,\varepsilon_2) \arrow[bend left=12]{llld}{e}\arrow[bend left=12]{llu}{d}\\
& &(Y_1,\varepsilon_1) \arrow[bend left=10,pos=0.7]{rrru}{f} \arrow[bend left=12]{ruu}{b}
\end{tikzcd}$
\!\!\! and \!\!\!\!\!\!\!\!\!
$\begin{tikzcd}[row sep=1cm, column sep=-0.2cm]
& (Z,\zeta) \arrow[bend left=12,pos=0.3]{ldd}{a} \arrow[bend left=12]{rrd}{c}\arrow{rrrr}{1}&&&~~~~~ &(Z,\zeta+bgc)\\
&&~& (Y_2,\varepsilon_2) \arrow[bend left=12]{llld}{e}\arrow[bend left=12]{llu}{d}\arrow[bend right=12]{rru}{bg}\\
(Y_1,\varepsilon_1) \arrow[bend left=10,pos=0.7]{rrru}{f} \arrow[bend left=12]{ruu}{b}
\end{tikzcd}$
\end{center}
respectively. One easily checks that indeed $D(F)=0$ and $D(G)=0$. Indeed, the only non-trivial terms we need to compute are
\begin{align*}
gc(\zeta+bgc)+\varepsilon_1gc+\partial(gc)+a=& ~g(c\zeta+\partial(c))+(\varepsilon_1g+\partial(g))c+a+gcbgc\\
=&~g(fa+\varepsilon_2c)+(\varepsilon_1g+\partial(g))c+a+gcbgc\\
=&~\left(D(g)+gcbg\right)c=0
\end{align*}
for the first identity and similarly
\begin{align*}
(\zeta+bgc)bg+bg\varepsilon_2+\partial(bg)+d=& ~(\zeta b+\partial(b))g+b(g\varepsilon_2+\partial(g))+d+bgcbg\\
=&~(df+b\varepsilon_1)g+b(g\varepsilon_2+\partial(g))+d+bgcbg\\
=&~b\left(D(g)+gcbg\right)=0.
\end{align*}
for the second identity. Now, $GF=1_{Z}$ and conversely, it is not hard to check that 
$$FG=1_{Z}+D(H),$$
where $H$ is the homotopy given by the dashed line in the following diagram:
$$\begin{tikzcd}[row sep=1cm, column sep=0cm]
& (Z,\zeta) \arrow[bend left=12,pos=0.3]{ldd}{a} \arrow[bend left=12]{rrd}{c}\arrow{rrrr}{1}&&&~~~~~ &(Z,\zeta+bgc)
\arrow{rrr}{1}\arrow[bend right=12]{rrdd}{gc} &~~~~~~~ && (Z,\zeta) \arrow[bend left=12,pos=0.3]{ldd}{a} \arrow[bend left=12]{rrd}{c}
\\
&&~& (Y_2,\varepsilon_2) \arrow[bend left=12]{llld}{e}\arrow[bend left=12]{llu}{d}\arrow[bend right=12]{rru}{bg}\arrow[dashed,swap, bend right=12]{rrrrd}{g}&&
& &&&~& (Y_2,\varepsilon_2) \arrow[bend left=12]{llld}{e}\arrow[bend left=12]{llu}{d}
\\
(Y_1,\varepsilon_1) \arrow[bend left=12,pos=0.7]{rrru}{f} \arrow[bend left=12]{ruu}{b}&&&&&
& &(Y_1,\varepsilon_1) \arrow[bend left=12,pos=0.7]{rrru}{f} \arrow[bend left=12]{ruu}{b}
\end{tikzcd}$$
\vspace*{-1cm}~\\
\end{proof}
\begin{lemma}[Clean-up Lemma]\label{lem:AbstractCleanUp}
Let \((O,d_O)\) be an object in \(\mathfrak{Cx}^{\ast}(\mathcal{C})\) for some differential graded category \(\mathcal{C}\).
Then for any morphism \(h\in\Mor_0((O,d_O),(O,d_O))\) for which 
$$h^2, \quad hD(h) \quad\text{ and } \quad D(h)h$$
vanish, \((O,d_O)\) is chain homotopic to \((O,d_O+D(h))\).
\end{lemma}
\begin{proof}
We can easily check that $(O,d_O+D(h))$ is an object in $\mathfrak{Cx}^{\ast}(\mathfrak{Mat}(\mathcal{C}))$:
\begin{align*}
(d_O+D(h))^2+\partial(d_O+D(h)) = &\left(d_O^2+\partial(d_O)\right)+\left(d_OD(h)+D(h)d_O^{\phantom{2}\!\!}+\partial(D(h))\right)\\
&+D(h)D(h)
\end{align*}
The first term on the right gives ($\ast$) and the last term vanishes, which can be seen by applying the differential $D$ to $h D(h)=0$. The middle term also vanishes, which can be seen by expanding $D(h)=d_Oh+hd_O+\partial(h)$ and using the fact that a term ($\ast$) commutes with any morphism.
The chain isomorphisms between the two objects are given by 
\begin{center}
$\begin{tikzcd}[row sep=1.5cm, column sep=0.6cm]
(O,d_O) \arrow{rr}{1+h} && (O,d_O+D(h))
\end{tikzcd}$
\quad
and
\quad
$\begin{tikzcd}[row sep=1.5cm, column sep=0.6cm]
(O,d_O+D(h)) \arrow{rr}{1+h} && (O,d_O).
\end{tikzcd}$
\end{center}
Indeed, these two morphisms lie in the kernel of $D$, since $hD(h)$ and $D(h)h$ vanish. Their composition is equal to $1+h^2=1$.
\end{proof}

%% file: sections/A_GCTproof.tex
\chapter{Proof of the generalised clock theorem}\label{proofgct}

First of all, we introduce some terminology, some of which is inspired by \cite{Kauffman}: 
\begin{definition}
A \textbf{D-graph} is a planar graph embedded in the closed disc $D^2$ such that at least one vertex lies on $\partial D^2$. 
A \textbf{universe} is a D-graph such that all vertices in the interior of $D^2$ are 4-valent and the ones on $\partial D^2$ are 1-valent. A \textbf{face} of a graph embedded into $D^2$ is a connected component of the complement of this graph in $D^2$. 
\end{definition}
\begin{Remark} We define generalised Kauffman states of universes just as for tangles. In fact, a universe is obtained from a tangle diagram by forgetting the under/over information at each crossing, and conversely, any universe gives rise to some tangle diagram. 
\end{Remark}
\begin{definition}
A \textbf{clocked state} of a tangle/universe is a generalised Kauffman state with the property that one cannot perform any anticlockwise transposition move.
\end{definition}
\begin{proposition}\label{noninfmov}
 In any universe one cannot perform an infinite number of clockwise transposition moves in sequence.
\end{proposition}
\begin{proposition}\label{clockunique}
The set of Kauffman states of a fixed site of a universe has a unique clocked state if it is not empty.
\end{proposition}
\begin{proof}[Proof of the generalised clock theorem~\sref{geclockt}]
If the set of Kauffman states for a given site is empty, there is nothing to show. Otherwise, given any two Kauffman states $x$ and $x'$ of a fixed site, proposition \ref{noninfmov} gives us two clocked states $y$ and $y'$ with $x\leq y$ and $x'\leq y'$. Proposition \ref{clockunique} tells us that $y=y'$, so the set of Kauffman states of a fixed site is a (finite) join-semilattice.\\
The duals of propositions \ref{noninfmov} and \ref{clockunique} follow from considering mirror diagrams and observing that the mirror of a clockwise transposition move is a anticlockwise transposition move and vice versa. Thus we see that the set of Kauffman states of a fixed site is also a (finite) meet-semilattice and the result follows.
\end{proof}
\begin{proof}[Proof of proposition~\ref{noninfmov}]
The argument from 
\cite[lemma~4]{ouka} carries over to the tangle case; for completeness, we recall it here: First of all, note that transposition moves do not change sites. Hence, the marker of an outermost vertex, i.\,e.\ one which meets an open unoccupied region, cannot make a complete cycle. A marker at a vertex which has a common edge with an outermost vertex cannot make two cycles because in the course of each cycle it has to interact with the marker at the outermost vertex and that one cannot make a full cycle. Similarly, a marker at a vertex which is $n=2$ edges away from an outermost vertex cannot make $2^n=2^2=4$ full cycles, and so on. 
\end{proof}
For proposition~\ref{clockunique}, we first generalise the correspondence used in \cite{ouka} between spanning trees and states, see also \cite[theorem~2.4]{Kauffman}. \\
Given a connected oriented universe $U$, partition the set of faces of $U$ into two subsets by shading each component which is separated from an arbitrarily fixed region
by an odd number of edges. Form a new D-graph $G(U)=G$ with one vertex
for each shaded region and one edge for each crossing shared by these shaded
regions, such that the vertices lie in their corresponding regions and those vertices of all open regions lie on $\partial D^2$. In a similar way, we obtain a D-graph $H$ from the unshaded regions. Note that $H$ is the dual graph of $G$. 
\begin{definition}
Given a D-graph $G$ and its dual $G^\ast$, a \textbf{D-forest} $F$ is the union of a maximal forest $F_G$ in $G$ and its dual $F_{G}^\ast$ in $G^\ast$ such that each forest has at least one point on $\partial D^2$, together with a specification of a root on $\partial D^2$ for each tree in $F$.
\end{definition}

\begin{lemma}\label{CorrespondenceKauffForest}
There is a one-to-one correspondence between D-forests in \(G(U)\) and Kauffman states of a (connected) universe \(U\).
\end{lemma}
\begin{proof}
The argument from \cite[lemma~2.4]{Kauffman} carries over, but we spell this out explicitly nonetheless. 
Given a D-forest $F$, we orient its edges in the canonical way, that is arrows point away from the roots. For each edge of $F$, place a marker into the region that the arrow of the edge is pointing towards. We claim that this defines a valid Kauffman state. Indeed, for each vertex in $F$ except the roots there is exactly one arrow pointing towards this vertex. Since vertices in $F$ correspond one-to-one to faces of $U$, we are done. \\
Going from Kauffman states back to forests in $G$ is now straightforward: For each marker at a crossing, we draw an edge between the corresponding vertices, according to the rule used above. Note that at each crossing, there is exactly one edge, so we get two disjoint subgraphs, one in $G(U)$ and one in $G(U)^\ast$. There is no simple cycle in either of these subgraphs. Otherwise, there would be a subdiagram $D'$ traced out by the cycle in $D^2$. A simple Euler characteristic argument leads to a contradiction: Say, there are $n$ vertices in this cycle (each of which becomes a 2-valent vertex in $D'$) and $m$ crossings of $U$ in $D'$. Then we have a total of $(m+n)$ vertices, $2(m+n)$ edges and therefore $1+2(m+n)-(m+n)=(m+n+1)$ faces in $D'$. Only the inner faces need to be occupied by markers, which means $m$ markers have to occupy $(m+1)$ faces. Contradiction. \\
Finally, every unoccupied region of $U$ becomes a root in $F$. It is clear that every tree in $F$ has exactly one root.
\end{proof}

Next, we need to describe what a transposition move looks like in the language of D-forests. The following picture illustrates the clockwise transposition move. The dotted edges denote those edges in the graph $G$ and its dual $G^\ast$ that do not belong to the forest. The dashed edges denote edges in the graph that may or may not be in the forest.
\begin{center}
\psset{unit=0.8}
\begin{pspicture}(-7.5,-2.8)(7.5,2.8)
\rput(-4,1.5){
\pspolygon*[linecolor=lightgray](-3,-1)(-2,-1)(-2,0)(-3,0)
\pspolygon*[linecolor=lightgray](3,-1)(2,-1)(2,0)(3,0)
\pspolygon*[linecolor=lightgray](-2,0)(2,0)(2,1)(-2,1)
\psline(-2,1)(-2,-1)
\psline(2,1)(2,-1)
\psline(-3,0)(3,0)
\psline[linewidth=1pt](-3,-1)(-2,0)
\psecurve[linewidth=1pt]{->}(-3,-2)(-2,0)(0,1)(2,0)
\psline[linestyle=dotted,linewidth=1pt](-3,1)(-2,0)
\psecurve[linestyle=dotted,linewidth=1pt](-3,2)(-2,0)(0,-1)(2,0)
\psline[linewidth=1pt](3,1)(2,0)
\psecurve[linewidth=1pt]{->}(3,2)(2,0)(0,-1)(-2,0)
\psline[linestyle=dotted,linewidth=1pt](3,-1)(2,0)
\psecurve[linestyle=dotted,linewidth=1pt](3,-2)(2,0)(0,1)(-2,0)

\psecurve*[linecolor=white](-1,0)(0,0.5)(1,0)(0,-0.5)(-1,0)(0,0.5)(1,0)
\psline[linestyle=dotted](-1,0)(1,0)

\uput{0.4}[45](-2,0){\psdot[dotsize=5pt](0,0)}
\uput{0.4}[-135](2,0){\psdot[dotsize=5pt](0,0)}
\psarc[arrowsize=1.5pt 2]{<-}(-2,0){0.4}{-50}{45}
\psarc[arrowsize=1.5pt 2]{<-}(2,0){0.4}{130}{225}
}

\rput(0,1.5){$\longrightarrow$}

\rput(4,1.5){
\pspolygon*[linecolor=lightgray](-3,-1)(-2,-1)(-2,0)(-3,0)
\pspolygon*[linecolor=lightgray](3,-1)(2,-1)(2,0)(3,0)
\pspolygon*[linecolor=lightgray](-2,0)(2,0)(2,1)(-2,1)
\psline(-2,-1)(-2,1)
\psline(2,-1)(2,1)
\psline(-3,0)(3,0)
\psline[linewidth=1pt](-3,1)(-2,0)
\psecurve[linewidth=1pt]{->}(-3,2)(-2,0)(0,-1)(2,0)
\psline[linestyle=dotted,linewidth=1pt](-3,-1)(-2,0)
\psecurve[linestyle=dotted,linewidth=1pt](-3,-2)(-2,0)(0,1)(2,0)
\psline[linewidth=1pt](3,-1)(2,0)
\psecurve[linewidth=1pt]{->}(3,-2)(2,0)(0,1)(-2,0)
\psline[linestyle=dotted,linewidth=1pt](3,1)(2,0)
\psecurve[linestyle=dotted,linewidth=1pt](3,2)(2,0)(0,-1)(-2,0)

\psecurve*[linecolor=white](-1,0)(0,0.5)(1,0)(0,-0.5)(-1,0)(0,0.5)(1,0)
\psline[linestyle=dotted](-1,0)(1,0)

\uput{0.4}[135](2,0){\psdot[dotsize=5pt](0,0)}
\uput{0.4}[-45](-2,0){\psdot[dotsize=5pt](0,0)}

}

\rput(-4,-1.5){

\psline[linewidth=1pt](-3,-1)(-2,0)
\psecurve[linewidth=1pt]{->}(-3,-2)(-2,0)(0,1)(2,0)
\psline[linestyle=dotted,linewidth=1pt](-3,1)(-2,0)
\psecurve[linestyle=dotted,linewidth=1pt](-3,2)(-2,0)(0,-1)(2,0)
\psline[linewidth=1pt](3,1)(2,0)
\psecurve[linewidth=1pt]{->}(3,2)(2,0)(0,-1)(-2,0)
\psline[linestyle=dotted,linewidth=1pt](3,-1)(2,0)
\psecurve[linestyle=dotted,linewidth=1pt](3,-2)(2,0)(0,1)(-2,0)

\pscurve[linestyle=dashed](0,1)(0.8,0)(0,-1)
\pscurve[linestyle=dashed](0,1)(0.25,0)(0,-1)
\pscurve[linestyle=dashed](0,1)(-0.25,0)(0,-1)
\pscurve[linestyle=dashed](0,1)(-0.8,0)(0,-1)

\psecurve*[linecolor=white](-1,0)(0,0.3)(1,0)(0,-0.3)(-1,0)(0,0.3)(1,0)
}

\rput(0,-1.5){$\longrightarrow$}

\rput(4,-1.5){
\psline[linewidth=1pt](-3,1)(-2,0)
\psecurve[linewidth=1pt]{->}(-3,2)(-2,0)(0,-1)(2,0)
\psline[linestyle=dotted,linewidth=1pt](-3,-1)(-2,0)
\psecurve[linestyle=dotted,linewidth=1pt](-3,-2)(-2,0)(0,1)(2,0)
\psline[linewidth=1pt](3,-1)(2,0)
\psecurve[linewidth=1pt]{->}(3,-2)(2,0)(0,1)(-2,0)
\psline[linestyle=dotted,linewidth=1pt](3,1)(2,0)
\psecurve[linestyle=dotted,linewidth=1pt](3,2)(2,0)(0,-1)(-2,0)

\pscurve[linestyle=dashed](0,1)(0.8,0)(0,-1)
\pscurve[linestyle=dashed](0,1)(0.25,0)(0,-1)
\pscurve[linestyle=dashed](0,1)(-0.25,0)(0,-1)
\pscurve[linestyle=dashed](0,1)(-0.8,0)(0,-1)

\psecurve*[linecolor=white](-1,0)(0,0.3)(1,0)(0,-0.3)(-1,0)(0,0.3)(1,0)
}
\end{pspicture}
\end{center}

\noindent
We are now ready to give an outline of the proof of proposition \ref{clockunique}. The proof goes by induction on the number of crossings in a tangle diagram. The induction step relies on the fact that a certain edge of the graph $G$ belongs to any clocked forest. This is the content of the next proposition.

\begin{proposition}\label{whatssospecial}
For any site of a (connected) universe, there exists a pair of two adjacent roots with the following property: For all \(r\geq0\), the number of roots in the next \(2r\) boundary faces (moving anticlockwise along the boundary, that is to the right of the picture below, see also figure \ref{FigEulerCharForTreeNotOnBoundaryLemma}) contain at least \(r\) roots.\\
Furthermore, suppose we have a site of a connected universe together with two such roots. Since the diagram is connected, we can consider the first crossing that one reaches along their common edge from the boundary of the disc. Suppose the two other regions at this crossing are no roots, as in the following picture.
\begin{center}
\begin{pspicture}(-6,-1.2)(6,1)

\psline[linestyle=dotted,linewidth=1pt](-2.5,-1)(-2,-1)
\psline[linestyle=dotted,linewidth=1pt](2.5,-1)(2,-1)
\pscustom[linewidth=1pt]{\psline[linewidth=1pt](-2,-1)(0,-1)
\psline[linewidth=1pt](0,-1)(2,-1)}

\psline[linestyle=dotted](0,0.5)(0,1)
\psline[linestyle=dotted](-1,0)(-0.5,0)
\psline[linestyle=dotted](0.5,0)(1,0)
\psline(0,-1)(0,0.5)
\psline(-0.5,0)(0.5,0)

\rput[b](-1,-0.9){root}
\rput[b](1,-0.9){root}
\rput[r](-0.2,0.3){no root}
\rput[l](0.2,0.3){no root}

\rput[cr](-3,-1){$\partial D^2$}
\end{pspicture}
\end{center}
Then the marker of the crossing in the picture above sits in the upper left region for any clocked Kauffman state of this universe with respect to the fixed site.
\end{proposition}

We postpone the proof of proposition \ref{whatssospecial} and first see how it implies proposition~\ref{clockunique}.

\begin{proof}[Proof of proposition~\ref{clockunique}]
The start of the induction is given by all those tangle diagrams that have exactly one Kauffman state. For these, the statement is trivially true. \\
For the induction step, consider a universe $U$ and a clocked Kaufman state $x$ of $U$. Without loss of generality, we may assume that the diagram is connected. Indeed: Otherwise we can consider its connected components by splitting the diagram at a region with more than one component on $\partial D^2$ and repeating this process as often as necessary. Observe that the restrictions $x_i$ of $x$ to the connected diagrams are also clocked and that the sites of the restrictions of any other (clocked) Kauffman state of the same site as $x$ agree with the sites of the $x_i$. Hence, if we can show the proposition for the $x_i$, then it also holds for $x$. \\
Now, consider two adjacent roots and the first crossing that one reaches along their common edge from the boundary of the disc. Obviously, not all four regions of this crossing can be roots. If there is exactly one additional root at this crossing, one can split the diagram into two parts like so: 
\begin{center}
\begin{pspicture}(-7,-1.2)(7,1.2)
\rput(-3,0){
\psline[linestyle=dotted,linewidth=1pt](-2.5,-1)(-2,-1)
\psline[linestyle=dotted,linewidth=1pt](2.5,-1)(2,-1)
\pscustom[linewidth=1pt]{\psline[linewidth=1pt](-2,-1)(0,-1)
\psline[linewidth=1pt](0,-1)(2,-1)}

\psline[linestyle=dotted](0,0.5)(0,1)
\psline[linestyle=dotted](-1,0)(-0.5,0)
\psline[linestyle=dotted](0.5,0)(1,0)
\psline(0,-1)(0,0.5)
\psline(-0.5,0)(0.5,0)

\psdot(0.2,0.2)
\psline[linestyle=dotted,linewidth=1pt](-2.5,1)(-2,1)
\psline[linestyle=dotted,linewidth=1pt](-0.5,1)(0,1)
\pscustom[linewidth=1pt]{\psline[linewidth=1pt](-2,1)(-1,1)
\psline[linewidth=1pt](-1,1)(-0.5,1)}

\rput(-0.05,0){
\psset{linestyle=dashed,linewidth=0.5pt}
\psecurve(-1,-1)(0,0)(-1,1)(-0.5,1.5)
\psline(0,0)(0,-1)
}
\rput(0.05,0){
\psset{linestyle=dashed,linewidth=0.5pt}
\psecurve(-1,-1)(0,0)(-1,1)(-0.5,1.5)
\psline(0,0)(0,-1)
}

\psset{linestyle=dashed}
\rput[b](-1,-0.9){root}
\rput[b](1,-0.9){root}
\rput[t](-1.5,0.9){root}

\rput[cr](-3,1){$\partial D^2$}
\rput[cr](-3,-1){$\partial D^2$}
}
\rput(0,0){$\longrightarrow$}

\rput(3,0){
\psline[linestyle=dotted,linewidth=1pt](-2.5,-1)(-2,-1)
\pscustom[linewidth=1pt]{
\psline[linewidth=1pt](-2,-1)(-0.5,-1)
\psarc[linewidth=1pt](-0.5,-0.5){0.5}{-90}{0}
\psline[linewidth=1pt](0,-0.5)(0,0.5)
\psarc[linewidth=1pt](-0.5,0.5){0.5}{0}{90}
\psline[linewidth=1pt](-2,1)(-0.5,1)}
\psline[linestyle=dotted,linewidth=1pt](-2.5,1)(-2,1)
\rput[b](-1.1,-0.9){root}
\rput[t](-1.1,0.9){root}

\psline[linestyle=dotted](-1,0)(-0.5,0)
\psline(-0.5,0)(0,0)

\rput(0.5,0){
\psline[linestyle=dotted,linewidth=1pt](2.5,-1)(2,-1)
\pscustom[linewidth=1pt]{
\psline[linewidth=1pt](2,-1)(0.5,-1)
\psarcn[linewidth=1pt](0.5,-0.5){0.5}{-90}{-180}
\psline[linewidth=1pt](0,-0.5)(0,0.5)
\psarcn[linewidth=1pt](0.5,0.5){0.5}{180}{90}
\psline[linewidth=1pt](2,1)(0.5,1)}
\psline[linestyle=dotted,linewidth=1pt](2.5,1)(2,1)
\rput[b](1.1,-0.9){root}
\rput[t](1.1,0.9){root}

\pscustom[linestyle=dotted]{
\psarc(2.167,0.666){0.333}{-90}{0}
\psline(2.5,0.666)(2.5,1)}
\psline(2.167,0.333)(0,0.333)
\psline[linestyle=dotted](1,-0.333)(0.5,-0.333)
\psline(0.5,-0.333)(0,-0.333)
\rput[l](0.2,0){new root}
}
}
\end{pspicture}
\end{center}
Note that the marker at the crossing has to be where it is for any Kauffman state. Also note that this splitting reduces the number of crossings by 1, so we can apply the induction hypothesis to both diagrams and we are done. So without loss of generality, we may now also assume that all adjacent roots locally look as in proposition \ref{whatssospecial}.\\
We consider the subdiagram obtained by removing a small neighbourhood of the edge from the boundary to this crossing. The clocked Kauffman state $x$ restricts to a clocked Kauffman state $x\vert$ of the subdiagram. 
$$
\begin{pspicture}(-6,-1.2)(6,1.5)

\pscustom*[linecolor=lightgray]{\psecurve(-1.4,-0.9)(-1,-1)(-0.6,-0.9)(-0.25,0.5)(0.25,0.5)(0.6,-0.9)(1,-1)(1.4,-0.9)}\psecurve[linestyle=dashed](-1.4,-0.9)(-1,-1)(-0.6,-0.9)(-0.25,0.5)(0.25,0.5)(0.6,-0.9)(1,-1)(1.4,-0.9)

\psline[linestyle=dotted,linewidth=1.5pt](-2.5,-1)(-2,-1)
\psline[linestyle=dotted,linewidth=1.5pt](2.5,-1)(2,-1)
\pscustom[linewidth=1.5pt]{\psline[linewidth=1.5pt](-2,-1)(0,-1)
\psline[linewidth=1.5pt](0,-1)(2,-1)}

\psline[linestyle=dotted](0,1)(0,1.5)
\psline[linestyle=dotted](-1,0)(-1.5,0)
\psline[linestyle=dotted](1,0)(1.5,0)
\psline(0,-1)(0,1)
\psline(-1,0)(1,0)

\rput[b](-1.3,-0.9){root}
\rput[b](1.3,-0.9){root}

\psdot(-0.2,0.2)
\rput[cr](-3,-1){$\partial D^2$}
\end{pspicture}
$$
The subdiagram has one crossing fewer and we can apply the induction hypothesis, i.\,e.\ $x\vert$ is the unique clocked Kauffman state of the subdiagram with respect to the site of $x$. Since this site is uniquely determined by the site of $x$ and the marker in the picture above, $x$ is the unique clocked Kauffman state of the universe $U$ with respect to the site of $x$.
\end{proof}
\begin{proof}[Proof of proposition~\ref{whatssospecial}]
The existence of a pair of roots with the special property required above follows from a standard argument: Consider the $2n$ line segments on $\partial D^2$, choose one as a starting point and walk along $\partial D^2$ in anticlockwise direction. Define a function from the set of segments to $\mathbb{Z}$ by adding $+1$ for each root segment and $-1$ for an occupied segment, see the illustration below. The value of the function at the very last segment will be $2$. Then consider the rightmost segment where the function takes its minimum. This will be a non-root, followed by two successive roots. By construction, these have the required property.

$${\psset{unit=0.4}
\begin{pspicture}(-1,-4.1)(30.5,4.1)

\psline{->}(-1,0)(22,0)
\psline{->}(0,-4)(0,4)
\psdot(0,0)
\psdot(21,2)
\psdot(11,-3)

{\psset{linewidth=1.5pt}
\psline(0,0)(1,1)(2,0)(3,-1)(4,0)
\psline[linestyle=dotted, dots=2pt](4,0)(5,-1)
\psline[linestyle=dotted](7,-1)(8,-2)
\psline(8,-2)(9,-3)(10,-2)(11,-3)(12,-2)(13,-1)(14,-2)
\psline[linestyle=dotted](14,-2)(15,-1)
\psline[linestyle=dotted](17,2)(18,3)
\psline(18,3)(19,2)(20,1)(21,2)
}
\psline(1,-0.3)(1,0.3)
\psline(2,-0.3)(2,0.3)
\psline(3,-0.3)(3,0.3)
\psline(4,-0.3)(4,0.3)
\psline(8,-0.3)(8,0.3)
\psline(9,-0.3)(9,0.3)
\psline(10,-0.3)(10,0.3)
\psline(11,-0.3)(11,0.3)
\psline(12,-0.3)(12,0.3)
\psline(13,-0.3)(13,0.3)
\psline(14,-0.3)(14,0.3)
\psline(18,-0.3)(18,0.3)
\psline(19,-0.3)(19,0.3)
\psline(20,-0.3)(20,0.3)
\psline(21,-0.3)(21,0.3)

\psline(-0.3,1)(0.3,1)
\psline(-0.3,2)(0.3,2)
\psline(-0.3,3)(0.3,3)
\psline(-0.3,-1)(0.3,-1)
\psline(-0.3,-2)(0.3,-2)
\psline(-0.3,-3)(0.3,-3)

\psline[linestyle=dashed](0,-3)(21,-3)

\psline{->}(12.5,0.8)(12.1,0.4)
\psline{->}(12.5,0.8)(12.9,0.4)
\psline(12.5,0.8)(12.5,2)
\rput[c](12.5,2.5){root pair}

\rput[c](1,-1){$1$}
\rput[c](21,-1){$2n$}

\rput[l](22.4,0){segments on $\partial D^2$}
\end{pspicture}}
$$
Let us denote the vertex corresponding to the upper right (resp. left) region of the picture in proposition \ref{whatssospecial} by $p_0$ (resp. $p_1$). Suppose the edge from the right root to $p_1$ does not belong to the $D$-forest $F$ corresponding to a clocked Kauffman state. We want to show that somewhere in the diagram, we can perform a anticlockwise transposition move, so $x$ cannot be a clocked state. Let us call its dual (which is in the D-forest) $e_0$. Then $e_0$ points away from the left root:
$$
\begin{pspicture}(-6,-1.2)(6,1.5)

\psline[linestyle=dotted,linewidth=1.5pt](-2.5,-1)(-2,-1)
\psline[linestyle=dotted,linewidth=1.5pt](2.5,-1)(2,-1)
\pscustom[linewidth=1.5pt]{\psline[linewidth=1.5pt](-2,-1)(0,-1)
\psline[linewidth=1.5pt](0,-1)(2,-1)}

\psline[linestyle=dotted](0,0.5)(0,1)
\psline[linestyle=dotted](-1,0)(-0.5,0)
\psline[linestyle=dotted](0.5,0)(1,0)
\psline(0,-1)(0,0.5)
\psline(-0.5,0)(0.5,0)

\rput[b](-1.5,-0.9){root}
\rput[b](1.5,-0.9){root}

\psline[linestyle=dotted](1,-1)(-0.5,0.5)
\psdot(-0.5,0.5)
\psdot(0.5,0.5)
\psdot(1,-1)
\psdot(-1,-1)
\rput[r](-0.6,0.5){$p_1$}
\rput[l](0.6,0.5){$p_0$}
\rput[l](-0.4,1){$e_1$}
\rput[r](-0.6,-0.5){$e_0$}
\psline{->}(-1,-1)(0.5,0.5)
\psline{->}(-0.5,1.5)(-0.5,0.5)

\rput[cr](-3,-1){$\partial D^2$}
\end{pspicture}
$$

\noindent
Note that $p_1$ has (exactly) one incoming edge $e_1$ of $F$, since $p_1$ is not a root. Consider the edges in $F$ that the dual $e_0^\ast$ of $e_0$ meets on its way to $e_1$ during a anticlockwise rotation around $p_1$, see the picture below. (Note that there might be no such edge.)
$$
\begin{pspicture}(-4,-1.5)(4,1.7)
\SpecialCoor
\pscustom*[linecolor=lightgray]{
\psline(0,0)(1.7;-45)
\psarc(0,0){1.7}{-45}{90}}
\rput{-45}(0,0){\psline[linestyle=dotted](0,0)(1.5,0)}

\rput{-5}(0,0){\psline{->}(0,0)(1.5,0)}
\rput{35}(0,0){\psline{->}(0,0)(1.5,0)}
\rput{55}(0,0){\psline[linestyle=dotted](0,0)(1.5,0)}
\rput{75}(0,0){\psline[linestyle=dotted](0,0)(1.5,0)}
\rput{15}(0,0){\psline[linestyle=dotted](0,0)(1.5,0)}
\rput{-25}(0,0){\psline[linestyle=dotted](0,0)(1.5,0)}

\psarc[linestyle=dashed]{->}(0,0){1}{-45}{90}

\psdot(0,0)

\rput[r](-0.1,0){$p_1$}
\rput[r](-0.1,1.4){$e_1$}
\rput[l](0.75,-1.3){$e_0^\ast$}
\psline{->}(0,1.5)(0,0)
\end{pspicture}
$$

\noindent
Let $F'$ be the tree consisting of $p_1$, these edges and all vertices and edges in $F$ that can be reached from these edges when following the arrows. Then the special property of our chosen pair of adjacent edges enables us to prove the following:
\begin{lemma}\label{EulerCharArgTreeNotOnBoundary}
Any vertex \(q\) that belongs to \(F'\) does not lie on \(\partial D^2\).
\end{lemma}
\begin{proof}
Suppose there is such a vertex $q$. Note that $p_1=q$ is also allowed. We have the following situation:

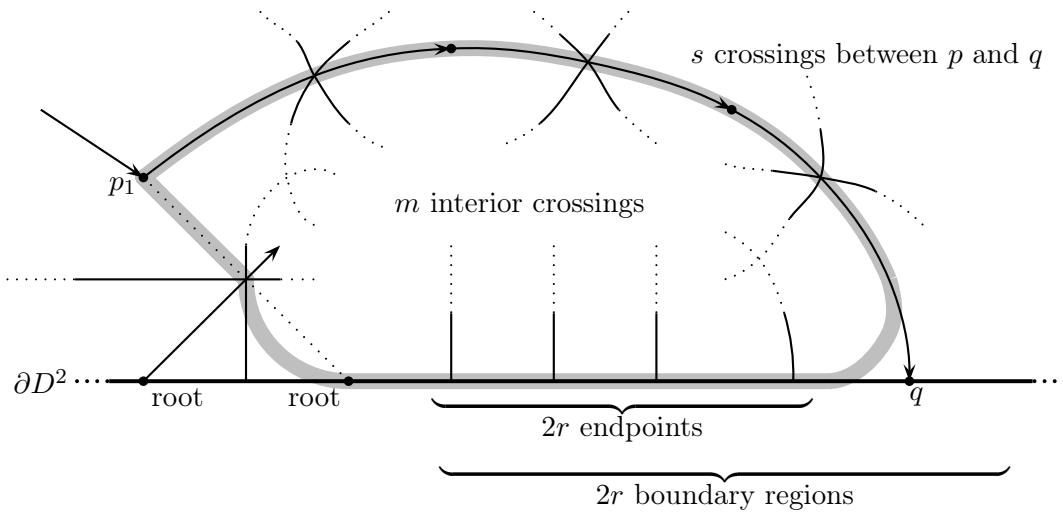
\begin{figure}[ht]\centering
\psset{unit=0.9}
\begin{pspicture}(-3.5,-3.7)(12.1,4)
\pscustom[linewidth=6pt, linecolor=lightgray, linejoin=1]{
\psecurve(7.6,-0.8)(8.5,-1.5)(9.4,-0.8)(9.4,0)(8.4,1.5)
\psecurve(9.7,-1.5)(9.4,0)(8.4,1.5)(7.1,2.5)(5,3.2)(3,3.4)(1,3)(-1.5,1.5)(-2,1)
\psline[linestyle=dotted](0,0)(-1.5,1.5)
\psarc(1.5,0){1.5}{180}{270}
\psline(1.5,-1.5)(8.5,-1.5)
}




\psecurve[linestyle=dotted](0.5,-0.75)(0,0.5)(0.6,1.5)(1.3,1.55)(1.5,1)
\psecurve[linestyle=dotted](1,3)(0.66,2.3)(0.6,1.5)(1,0.8)(1.5,1)

\psecurve(2.1,3.95)(1.5,3.5)(1,3)(0.66,2.3)(0.6,1.5)
\psecurve[linestyle=dotted](2.4,4.3)(2.1,3.95)(1.5,3.5)(1,3)
\psecurve(0.2,3.95)(0.75,3.5)(1,3)(1.5,2.3)(2,1.95)
\psecurve[linestyle=dotted](-0.4,5)(0.2,3.95)(0.75,3.5)(1,3)
\psecurve[linestyle=dotted](1,3)(1.5,2.3)(2,1.95)(12.5,1.7)

\psecurve(3.8,2)(4.3,2.3)(5,3.2)(5.3,3.6)(5.6,3.95)
\psecurve[linestyle=dotted](3.3,1.9)(3.8,2)(4.3,2.3)(5,3.2)
\psecurve[linestyle=dotted](5,3.2)(5.3,3.6)(5.6,3.95)(6,4.2)

\psecurve(4.2,3.95)(4.6,3.6)(5,3.2)(5.45,2.5)(6,2)
\psecurve[linestyle=dotted](3.7,4.2)(4.2,3.95)(4.6,3.6)(5,3.2)
\psecurve[linestyle=dotted](5,3.2)(5.45,2.5)(6,2)(6.5,1.6)

\psecurve(8.2,3)(8.4,2.2)(8.4,1.5)(7.95,0.9)(7.5,0.3)
\psecurve[linestyle=dotted](8,3.4)(8.2,3)(8.4,2.2)(8.4,1.5)
\psecurve[linestyle=dotted](8.4,1.5)(7.95,0.9)(7.5,0.3)(7,0)(6.5,0)
\psecurve(7,1.7)(7.7,1.6)(8.4,1.5)(9.2,1.3)(9.9,0.8)
\psecurve[linestyle=dotted](6.5,2)(7,1.7)(7.7,1.6)(8.4,1.5)
\psecurve[linestyle=dotted](8.4,1.5)(9.2,1.3)(9.9,0.8)(10.5,0)

\psecurve[linestyle=dotted](6,1)(7,0.7)(7.5,0.3)(7.87,-0.5)(8,-1.5)
\psecurve(7.5,0.3)(7.87,-0.5)(8,-1.5)(7.8,-2.5)


\psline(6,-1.5)(6,-0.5)
\psline[linestyle=dotted](6,-0.5)(6,0.5)

\psline(4.5,-1.5)(4.5,-0.5)
\psline[linestyle=dotted](4.5,-0.5)(4.5,0.5)

\psline(3,-1.5)(3,-0.5)
\psline[linestyle=dotted](3,-0.5)(3,0.5)
\psline[linestyle=dotted,linewidth=1.5pt](-2.5,-1.5)(-2,-1.5)
\psline[linestyle=dotted,linewidth=1.5pt](12,-1.5)(11.5,-1.5)
\pscustom[linewidth=1.5pt]{\psline[linewidth=1.5pt](-2,-1.5)(0,-1.5)
\psline[linewidth=1.5pt](0,-1.5)(11.5,-1.5)}

\psline[linestyle=dotted](-2.5,0)(-3.5,0)
\psline[linestyle=dotted](0.5,0)(1,0)
\psline(0,-1.5)(0,0.5)
\psline(-2.5,0)(0.5,0)

\rput[t](-1,-1.6){root}
\rput[t](1,-1.6){root}

\rput[c](4,1.1){$m$ interior crossings}


\rput(5.5,-2){$\underbrace{\hspace{5cm}}_{\phantom{x}}$}
\rput(5.5,-2.2){$2r$ endpoints}

\rput(7,-3){$\underbrace{\hspace{7.5cm}}_{\phantom{x}}$}
\rput(7,-3.2){$2r$ boundary regions}

\psdot(1.5,-1.5)
\psdot(-1.5,-1.5)
\psline[linestyle=dotted](1.5,-1.5)(-1.5,1.5)
\psdot(-1.5,1.5)
\rput[tr](-1.6,1.5){$p_1$}
\psdot(9.7,-1.5)
\rput[tl](9.7,-1.6){$q$}    
\psline{->}(-1.5,-1.5)(0.5,0.5)
\psline{->}(-3,2.5)(-1.5,1.5)

\psecurve{->}(-2,1)(-1.5,1.5)(1,3)(3,3.4)(5,3.2)
\psecurve{->}(1,3)(3,3.4)(5,3.2)(7.1,2.5)(8.4,1.5)
\psecurve{->}(5,3.2)(7.1,2.5)(8.4,1.5)(9.4,0)(9.7,-1.5)(9.4,-3)

\psdots(3,3.4)(7.1,2.5)

\rput[l](6.5,3.3){$s$ crossings between $p$ and $q$}
\rput[cr](-3,-1.5){$\partial D^2$}
\end{pspicture}
\caption{The Euler characteristic argument for the proof of lemma \ref{EulerCharArgTreeNotOnBoundary}} \label{FigEulerCharForTreeNotOnBoundaryLemma}
\end{figure}
\noindent
We compute the number of faces of the disc enclosed by the grey curve, using an Euler characteristic argument. There are $m$ interior crossings, say, $2r$ endpoints on the boundary between q and the right root, and $s$ crossings between $p_1$ and $q$. (Again, $s=0$ is allowed.) Then we have
\begin{eqnarray*}
\#\{\text{vertices}\} & = & 2r+s+2+m\\
\#\{\text{edges}\} & = & (2r-1)+s+3+\tfrac{1}{2}(4m+2r+2s+2) \\
& = & 3r+2s+3+2m\\
\Rightarrow\quad \#\{\text{faces}\}&=&1+\#\{\text{edges}\}-\#\{\text{vertices}\}=r+s+2+m
\end{eqnarray*}
Hence the number of occupied faces is
$$\#\{\text{faces}\}-(s+1)-(\geq r+1)\leq m.$$
However,
$$\#\{\text{markers}\}=m+1.$$
Contradiction!
\end{proof}
\begin{corollary}\label{CorEulerArg}
The vertex \(p_2\) that we reach by following the dual edge of \(e_1\) to the left of \(e_1\) is in the full subtree starting at \(p_0\), see figure \ref{FigClockedItertionStep}.
\end{corollary}
\begin{proof}
Choose a small contractible neighbourhood of the tree $F'$. Then its boundary gives us a path in $D^2$ from the region corresponding to the vertex $p_0$ to the region containing $p_2$ that is disjoint from the $D$-forest $F$. 
\end{proof}
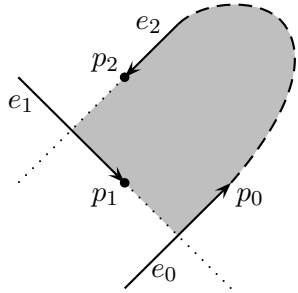
\begin{figure}[h]
\centering
{\psset{unit=0.7}
\begin{pspicture}(-2.5,-2)(4.5,3.5)
\pscustom*[linecolor=lightgray]{
\psecurve(1,1)(1,3)(3,3)(2,0)(-2,-2)
\psline(2,0)(1,-1)
\psline(-1,1)(1,-1)
\psline(-1,1)(1,3)
}
\psdot(0,0)
\psline{->}(-2,2)(0,0)
\psline[linestyle=dotted](0,0)(2,-2)
\psline{->}(0,-2)(2,0)
\psline[linestyle=dotted](-2,0)(0,2)
\psline{->}(1,3)(0,2)
\psecurve[linestyle=dashed](1,1)(1,3)(3,3)(2,0)(-2,-2)
\psdot(0,2)
\rput[tl](2.1,-0.1){$p_0$}
\rput[tr](-0.1,-0.1){$p_1$}
\rput[br](-0.1,2.1){$p_2$}
\rput[l](0.5,-1.7){$e_0$}
\rput[r](-1.7,1.5){$e_1$}
\rput[r](0.7,3){$e_2$}
\end{pspicture}}
\caption{An illustration of corollary~\ref{CorEulerArg}}\label{FigClockedItertionStep}
\end{figure}
\noindent 
We can now finish our proof of proposition \ref{whatssospecial}. The corollary above tells us in particular that $p_2$ is not a root, so there is exactly one incoming edge $e_2$. Note that if $p_0=p_2$ (or equivalently $e_0=e_2$), we are done, because we can perform a anticlockwise transposition move. If $p_0\neq p_2$, we can repeat the argument for $p_2$ in place of $p_1$. It is now clear that $p_3$ is in the subtree starting at $p_1$ and therefore has an incoming edge. This is where we needed the lemma in the first iteration step. \\
This algorithm terminates because there are only finitely many vertices in the shaded region in figure~\ref{FigClockedItertionStep} and the number of vertices in the corresponding shaded area in each iteration step strictly decreases unless the algorithm terminates during this step.
\end{proof}

%% file: sections/A_APT-Manual.tex
\chapter{Manual for \texttt{APT.m}}\label{app:manualAPT}
\vspace{-12pt}
\begin{center}
	version 1.1, written in Mathematica 10.3.1.0 for Linux x86 (64-bit)
\end{center}
\vspace{36pt}

\noindent
The Mathematica package \cite{APT.m} provides a simple tool for calculating the polynomial tangle invariants~$\nabla_T^s$ from definition~\sref{def:basic}. It can also be used to compute the generators of $\CFT$ from the ``standard'' Heegaard diagrams in example~\sref{exa:HDforonecrossing}. This manual should be read alongside the notebook \cite{APT.nb}. APT stands for ``Alexander Polynomial for Tangles''.

\subsection*{Input preparation. }  We start by loading the package \cite{APT.m} as follows:
$$\texttt{<\relax< APT.m}$$
Given an oriented $2n$-ended tangle $T$, we choose a connected tangle diagram with at least one crossing. To translate this data into computer-speak, we enumerate all crossings and give a unique name to each region. Usually, we use letters \texttt{a}, \texttt{b}, \texttt{c},\dots for this and label the open regions first, in anticlockwise order, as illustrated in the example in figure~\ref{fig:labelling}. Finally, we choose colours of the tangle components, for which we usually use the letters \texttt{p}, \texttt{q}, \texttt{r},\dots

\begin{figure}[bh]
\centering
\psset{unit=0.8, linewidth=1.1pt, arrowsize=2pt 3}
\begin{pspicture}[showgrid=false](-5.2,-3.1)(3.2,3.1)
\psecurve[linecolor=violet](-2.5,1.5)(0,2)(0.75,1)(-0.75,-1)(0,-2)(0.97,-2.24)(2,-2)
\psecurve[linecolor=violet]{<-}(2,2)(0.97,2.24)(0,2)(-0.75,1)(0.75,-1)(0,-2)(-2.5,-1.5)(-3.25,0)(-2.5,1.5)(0,2)(0.75,1)
\psecurve[linecolor=darkgreen](-3.3,1.85)(-2.5,1.5)(-1.85,0)(-2.5,-1.5)(-3.3,-1.85)(-6,-1.5)
\pscircle*[linecolor=white](-2.5,-1.5){0.2}

\psecurve[linecolor=violet](0.75,-1)(0,-2)(-2.5,-1.5)(-3.25,0)(-2.5,1.5)

\pscircle*[linecolor=white](0,2){0.2}
\pscircle*[linecolor=white](0,0){0.2}
\pscircle*[linecolor=white](0,-2){0.2}

\psecurve[linecolor=violet](0.75,1)(-0.75,-1)(0,-2)(0.97,-2.24)(2,-2)
\psecurve[linecolor=violet](0,2)(-0.75,1)(0.75,-1)(0,-2)
\psecurve[linecolor=violet](-3.25,0)(-2.5,1.5)(0,2)(0.75,1)(-0.75,-1)

\pscircle*[linecolor=white](-2.5,1.5){0.2}
\psecurve[linecolor=darkgreen]{<-}(-6,1.5)(-3.3,1.85)(-2.5,1.5)(-1.85,0)(-2.5,-1.5)

\psline[linecolor=violet]{->}(0.75,1.1)(0.75,1.2)
\psline[linecolor=violet]{->}(-0.75,1.11)(-0.75,1.2)
\psline[linecolor=violet]{->}(-0.75,-0.9)(-0.75,-0.8)
\psline[linecolor=violet]{->}(0.75,-0.9)(0.75,-0.8)

\pscircle[linestyle=dotted](-1,0){3.05}

\psline[linecolor=violet]{->}(-0.9,2.12)(-1.05,2.11)

\psline[linecolor=darkgreen]{->}(-1.85,0.05)(-1.85,0.15)
\psline[linecolor=violet]{->}(-3.25,-0.05)(-3.25,-0.15)

\psline[linecolor=violet]{->}(-1.5,-2.015)(-1.4,-2.05)

\rput[c](0,-1.65){\texttt{1}}
\rput[c](0.5,0){\texttt{2}}
\rput[c](0,1.65){\texttt{3}}

\rput[c](-2.5,-1){\texttt{5}}
\rput[c](-1.9,1.4){\texttt{4}}

\rput(-1,0){
\rput[c](3.2;145){$\textcolor{darkgreen}{p}$}
\rput[c](3.2;50){$\textcolor{violet}{q}$}
\rput[c](2.7;180){\texttt{a}}
\rput[c](2.7;270){\texttt{b}}
\rput[c](2.7;0){\texttt{c}}
\rput[c](2.7;90){\texttt{d}}

\rput[c](1.3;180){\texttt{e}}
\rput[c](0,0){\texttt{f}}
\rput[c](1,1){\texttt{g}}
\rput[c](1,-1){\texttt{h}}
}

\end{pspicture}
\caption{A labelling of regions and crossings of a tangle diagram}\label{fig:labelling}
\end{figure}


\subsection*{Input data. }
The data for a tangle diagram is entered into a Mathematica notebook as a list
$$\texttt{v=\{v$_0$,v$_1$,v$_2$,\dots,v$_m$\}},$$
where \texttt{v$_0$} is a list of open regions of the tangle (in anticlockwise order),\pagebreak[3] $m$ is the number of crossings and, for $1\leq i\leq m$, \texttt{v$_i$} is a list with the following seven entries:
\begin{itemize}
\item The first four entries \texttt{(v$_i$)$_1$} to \texttt{(v$_i$)$_4$} are the labels of the regions in the four quadrants of the $i^\text{th}$ crossing. The first entry is the label of the region between the two outward-pointing arrows and then we go in anticlockwise direction, as illustrated in figure~\ref{fig:APPcrossing}. 
\item The fifth entry \texttt{(v$_i$)$_5$} is either \texttt{L} or \texttt{R}, depending on whether $i$ is the label of a positive (left) or negative (right) crossing, see figure \ref{fig:APPcrossingL} and \ref{fig:APPcrossingR}.
\item The entries \texttt{(v$_i$)$_6$} and \texttt{(v$_i$)$_7$} are the colours of the over- and under-strands, respectively.  
\end{itemize}
It is not hard to see that the tangle is uniquely determined by \texttt{v}. 

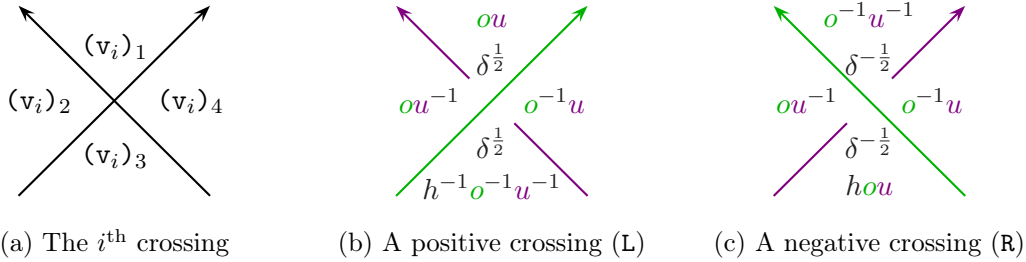
\begin{figure}[t]
\centering
\psset{unit=1.4, arrowsize=3pt 2}
\begin{subfigure}[b]{0.3\textwidth}\centering
\begin{pspicture}(-1.05,-1.05)(1.05,1.05)
\psline{->}(-0.9,-0.9)(0.9,0.9)
\psline{->}(0.9,-0.9)(-0.9,0.9)
\uput{0.4}[90](0,0){\texttt{(v$_i$)$_1$}}
\uput{0.4}[180](0,0){\texttt{(v$_i$)$_2$}}
\uput{0.4}[-90](0,0){\texttt{(v$_i$)$_3$}}
\uput{0.4}[0](0,0){\texttt{(v$_i$)$_4$}}
\end{pspicture}
\caption{The $i^\text{th}$ crossing}\label{fig:APPcrossing}
\end{subfigure}\quad
\begin{subfigure}[b]{0.3\textwidth}\centering
\begin{pspicture}(-1.05,-1.05)(1.05,1.05)
\psline[linecolor=violet]{->}(0.9,-0.9)(-0.9,0.9)
\pscircle*[linecolor=white](0,0){0.3}
\psline[linecolor=darkgreen]{->}(-0.9,-0.9)(0.9,0.9)
\uput{0.7}[90](0,0){$\textcolor{darkgreen}{o} \textcolor{violet}{u}$}
\uput{0.3}[180](0,0){$\textcolor{darkgreen}{o} \textcolor{violet}{u}^{-1}$}
\uput{0.7}[-90](0,0){$h^{-1}\textcolor{darkgreen}{o}^{-1} \textcolor{violet}{u}^{-1}$}
\uput{0.3}[0](0,0){$\textcolor{darkgreen}{o}^{-1} \textcolor{violet}{u}$}
\uput{0.25}[90](0,0){$\delta^{\frac{1}{2}}$}
\uput{0.25}[-90](0,0){$\delta^{\frac{1}{2}}$}
\end{pspicture}
\caption{A positive crossing (\texttt{L})}\label{fig:APPcrossingL}
\end{subfigure}\quad
\begin{subfigure}[b]{0.3\textwidth}\centering
\begin{pspicture}(-1.05,-1.05)(1.05,1.05)
\psline[linecolor=violet]{->}(-0.9,-0.9)(0.9,0.9)
\pscircle*[linecolor=white](0,0){0.3}

\psline[linecolor=darkgreen]{->}(0.9,-0.9)(-0.9,0.9)
\uput{0.7}[90](0,0){$\textcolor{darkgreen}{o}^{-1} \textcolor{violet}{u}^{-1}$}
\uput{0.3}[180](0,0){$\textcolor{darkgreen}{o} \textcolor{violet}{u}^{-1}$}
\uput{0.7}[-90](0,0){$h\textcolor{darkgreen}{o} \textcolor{violet}{u}$}
\uput{0.3}[0](0,0){$\textcolor{darkgreen}{o}^{-1} \textcolor{violet}{u}$}
\uput{0.25}[90](0,0){$\delta^{-\frac{1}{2}}$}
\uput{0.25}[-90](0,0){$\delta^{-\frac{1}{2}}$}
\end{pspicture}
\caption{A negative crossing (\texttt{R})}\label{fig:APPcrossingR}
\end{subfigure}
\caption{(a) shows the order of the four quadrants \texttt{v$_i$} at an oriented crossing. (b) and (c) show the Alexander codes used in \cite{APT.m}. They coincide with those from figure~\ref{figAlexCodesForNabla} after substituting $\textcolor{darkgreen}{o}\mapsto \textcolor{darkgreen}{o}^{1/2}$ and $\textcolor{violet}{u}\mapsto \textcolor{violet}{u}^{1/2}$. The over-strand is coloured by~$\textcolor{darkgreen}{o}$ and the under-strand by~$\textcolor{violet}{u}$.}\label{fig:AlexCodesForNabla}
\end{figure}

\begin{example}\label{exa:APT}
Let $T$ be the tangle diagram shown in figure \ref{fig:labelling}. Then the list $\texttt{v}=\texttt{v}(T)$ is given by
\begin{center}
\begin{minipage}{127pt}
$\texttt{v}=$ \texttt{\{\{a,b,c,d\},\\
\phantom{$\texttt{v}=$ }\{h,f,b,c,R,q,q\},\\
\phantom{$\texttt{v}=$ }\{g,f,h,c,R,q,q\},\\
\phantom{$\texttt{v}=$ }\{d,f,g,c,R,q,q\},\\
\phantom{$\texttt{v}=$ }\{a,e,f,d,L,p,q\},\\
\phantom{$\texttt{v}=$ }\{f,e,a,b,L,q,p\}\}}.
\end{minipage}
\end{center}
\end{example}

\pagebreak[2]

\subsection*{Calculation of $\nabla_T^s$. }
To explain the main function \texttt{AlexpolyGdh[v]} of this package, we introduce the following notation: As in definition~\sref{def:basic}, let $c(x)$ be the labelling of a Kauffman state~$x$ using the Alexander codes from figure~\ref{fig:AlexCodesForNabla}. For a site $s$ of $T$, let $\texttt{s}(s)$ be the word obtained by concatenating all labels of regions in $s$ in alphabetical order; we set $\texttt{s}(\emptyset)=1$. Similarly, let $\texttt{g}(x)$ be the concatenation of all labels of regions occupied by $x$ in the order of the crossings. Then \texttt{AlexpolyGdh[v]} computes
$$\sum_{\substack{s\in\mathbb{S}(T)\\x\in\mathbb{K}(T,s)}}c(x)\cdot\texttt{s}(s)\cdot\texttt{g}(x).$$
The terms $\texttt{g}(x)$ uniquely determine the Kauffman states and thus the generators of the tangle Floer complex $\CFT$ from the ``standard'' Heegaard diagrams in example~\sref{exa:HDforonecrossing}.\\
There is also a function \texttt{AlexpolyPdh[v]}, which is the same as \texttt{AlexpolyGdh[v]} except that the terms $\texttt{g}(s)$ are omitted. There are six other variants of \texttt{AlexpolyPdh} and \texttt{AlexpolyGdh} which ``forget'' the homological grading and/or the $\delta$-grading by setting $h=-1$ and/or $\delta=1$. The names of the functions are obtained from \texttt{AlexpolyPdh} and \texttt{AlexpolyGdh} simply by dropping the letters \texttt{h} and/or~\texttt{d}, respectively. 
In particular, the polynomial invariants $\nabla_T^s$ are equal to the coefficients of~$\texttt{s}(s)$ in \texttt{AlexpolyP[v]} -- up to substitution of all colours of $T$ by their square roots, see figure~\ref{fig:AlexCodesForNabla}. 

\begin{example}\label{exa:APPpolys}
In the example from above, we obtain
\begin{align*}
\texttt{AlexpolyPdh[v]}=& ~
\texttt{d}~ \delta^{-1} q^{-6} p^{-2} h^{-1} + 
\texttt{b}~ p^2 \delta^{-1} q^{-6} + 
\texttt{d}~ p^2 \delta^{-1} q^{-6} + 
\texttt{d}~ \delta^{-1} q^{-2} p^{-2}\\
& + 
\texttt{b}~ h p^2 \delta^{-1} q^{-2} + 
\texttt{d}~ h q^2 \delta^{-1} p^{-2} + 
\texttt{b}~ h^2 p^2 q^2 \delta^{-1} + 
\texttt{b}~ h^2 q^6 \delta^{-1} p^{-2}\\
& + 
\texttt{d}~ h^2 q^6 \delta^{-1} p^{-2} + 
\texttt{b}~ h^3 p^2 q^6 \delta^{-1} + 
\texttt{c}~ p^{-2} \delta^{-\frac{1}{2}} + 
\texttt{c}~ h p^2 \delta^{-\frac{1}{2}}\\
& + 
\texttt{a}~ q^{-6} h^{-1} \delta^{-\frac{1}{2}} + 
\texttt{c}~ q^{-4} p^{-2} h^{-1} \delta^{-\frac{1}{2}} + 
\texttt{c}~ p^2 q^{-4} \delta^{-\frac{1}{2}} + 
2 \texttt{a}~ q^{-2} \delta^{-\frac{1}{2}}\\
& + 
2 \texttt{a}~ h q^2 \delta^{-\frac{1}{2}} + 
\texttt{c}~ h q^4 p^{-2} \delta^{-\frac{1}{2}} + 
\texttt{c}~ h^2 p^2 q^4 \delta^{-\frac{1}{2}} + 
\texttt{a}~ h^2 q^6 \delta^{-\frac{1}{2}}
\end{align*}
and
\begin{align*}
\texttt{AlexpolyGh[v]}=& ~
\texttt{c}~ \texttt{hcgfe}~ p^{-2} + 
\texttt{c}~ h~ \texttt{hcgef}~ p^2 + 
\texttt{a}~ \texttt{hgfea}~ q^{-6} h^{-1}+
\texttt{d}~ \texttt{hgdfe}~ q^{-6} p^{-2} h^{-1}\\
&  +  
\texttt{d}~ \texttt{hgdef}~ p^2 q^{-6} + 
\texttt{b}~ \texttt{hgfeb}~ p^2 q^{-6} +
\texttt{c}~ \texttt{hgcfe}~ q^{-4} p^{-2} h^{-1}\\
&  + 
\texttt{c}~ \texttt{hgcef}~ p^2 q^{-4} + 
\texttt{a}~ \texttt{hfgea}~ q^{-2} +
\texttt{a}~ \texttt{hgfae}~ q^{-2} + 
\texttt{d}~ \texttt{hgfde}~ q^{-2} p^{-2}\\
&  + 
\texttt{b}~ h~ \texttt{hfgeb}~ p^2 q^{-2} +
\texttt{a}~ \texttt{fhgea}~ h q^2 + 
\texttt{a}~ h~ \texttt{hfgae}~ q^2 + 
\texttt{d}~ h~ \texttt{hfgde}~ q^2 p^{-2}\\
&  + 
\texttt{b}~ \texttt{fhgeb}~ h^2 p^2 q^2 + 
\texttt{c}~ \texttt{chgfe}~ h q^4 p^{-2} + 
\texttt{c}~ \texttt{chgef}~ h^2 p^2 q^4+ 
\texttt{a}~ \texttt{fhgae}~ h^2 q^6\\
&  + 
\texttt{b}~ \texttt{bhgfe}~ h^2 q^6 p^{-2} + 
\texttt{d}~ \texttt{fhgde}~ h^2 q^6 p^{-2} + 
\texttt{b}~ \texttt{bhgef}~ h^3 p^2 q^6
\end{align*}

\end{example}

\begin{figure}[t]
\centering
\includegraphics[scale=1.2]{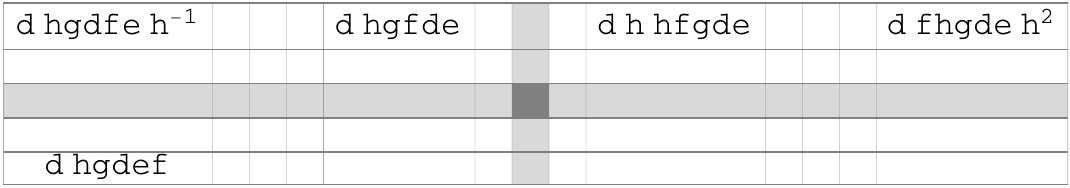}
\caption{The output of \texttt{SimpleGrid} applied to the polynomial \texttt{AlexpolyGh[v]} from example \ref{exa:APPpolys}, after setting the variables \texttt{a}, \texttt{b} and \texttt{c} equal to 0. The heavily shaded entry has Alexander grading~$p^0q^0$. The lightly shaded entries indicate the axes $p^0q^n$ (horizontal, pointing right) and $p^nq^0$ (vertical, pointing down).}\label{fig:GraphicaloutputExample}
\end{figure}


\subsection*{Basic graphical output.}
As example~\ref{exa:APPpolys} illustrates, it can be quite hard to understand polynomials with many terms just from their algebraic expressions. Fortunately, the package \cite{APT.m} also provides some graphical tools for dealing with this.
\begin{description}
\item[\texttt{SimpleGrid[poly]}] rearranges the monomials of a Laurent polynomial \texttt{poly} in a grid with two axes, one for the colour $p$ and one for $q$; see figure~\ref{fig:GraphicaloutputExample} for an example. 
\item[\texttt{\texttt{AlexpolytableGdh[v,poly]}}] is based on \texttt{SimpleGrid[poly]}. Its input is a tangle \texttt{v} and a Laurent polynomial \texttt{poly} which is usually the output of the function \texttt{AlexpolyGdh} or one of its variants. The output of \texttt{\texttt{AlexpolytableGdh}} is a list of all sites~$s$ of~\texttt{v} and the corresponding outputs of \texttt{SG} applied to the coefficient of $\texttt{s}(s)$ in \texttt{poly}. 
\item[\texttt{AlexpolytablePdh[v,poly]}] is similar to the previous function, but it sets all terms~$\texttt{g}(x)$ in~\texttt{poly} equal to~1.
\end{description}
Just as for \texttt{AlexpolyGdh} and \texttt{AlexpolyPdh}, there are variants of the functions \texttt{AlexpolytableGdh} and \texttt{AlexpolytablePdh} which ``forget'' the homological and/or $\delta$-grading in the output. Again, the names of these functions are obtained from \texttt{AlexpolytableGdh} and \texttt{AlexpolytablePdh} by dropping the letters~\texttt{h} and/or~\texttt{d}. 

\subsection*{More functions.}
\begin{description}
\item[\texttt{DConfig[v]}] graphically represents the boundary configuration of the tangle~\texttt{v}, i.\,e.\ the open regions of the tangle diagram and the open tangle components together with their orientations and colours, see figure~\ref{fig:DConfig}.
\begin{figure}[t]
\centering
\includegraphics[scale=0.9]{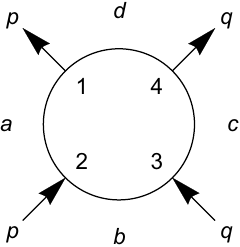}
\caption{The output of \texttt{DConfig[v]} for example~\ref{exa:APT}}\label{fig:DConfig}
\end{figure}
\item[\texttt{Twist[v,openregion,twist]}] computes the new tangle obtained from \texttt{v} by adding a single crossing at the region \texttt{openregion}, where \texttt{twist} determines if the new crossing is a positive (\texttt{L}) or a negative (\texttt{R}) crossing. (Note: This function only works correctly if there are fewer than 26 regions in the new diagram -- otherwise we run out of letters to label the new region with.)
\item[\texttt{Closure[v,openregion]}] computes the Alexander polynomial of the link obtained from the tangle~\texttt{v} by pairing it with a tangle of parallel strands which caps off the region \texttt{openregion}. 
\item[\texttt{DotGrid[v]}] is a function for 4-ended tangles \texttt{v} whose open regions are labelled (according to our conventions) by $a$, $b$, $c$ and $d$. It returns a convenient graphical output, which represents Kauffman states of all sites by dots in a single 2-dimensional grid according to their Alexander gradings, where we use the same conventions as for the function~\texttt{SimpleGrid}, see figure~\ref{fig:QDots}. The dots are labelled by their $\delta$-grading. They are also coloured according to the site they belong to  with the following colour conventions: 
$${\red a}, {\blue b}, \textcolor{darkgreen}{c}, \textcolor{gold}{d}.$$ 
\end{description}

\begin{figure}[b]
\centering
\includegraphics[scale=0.9]{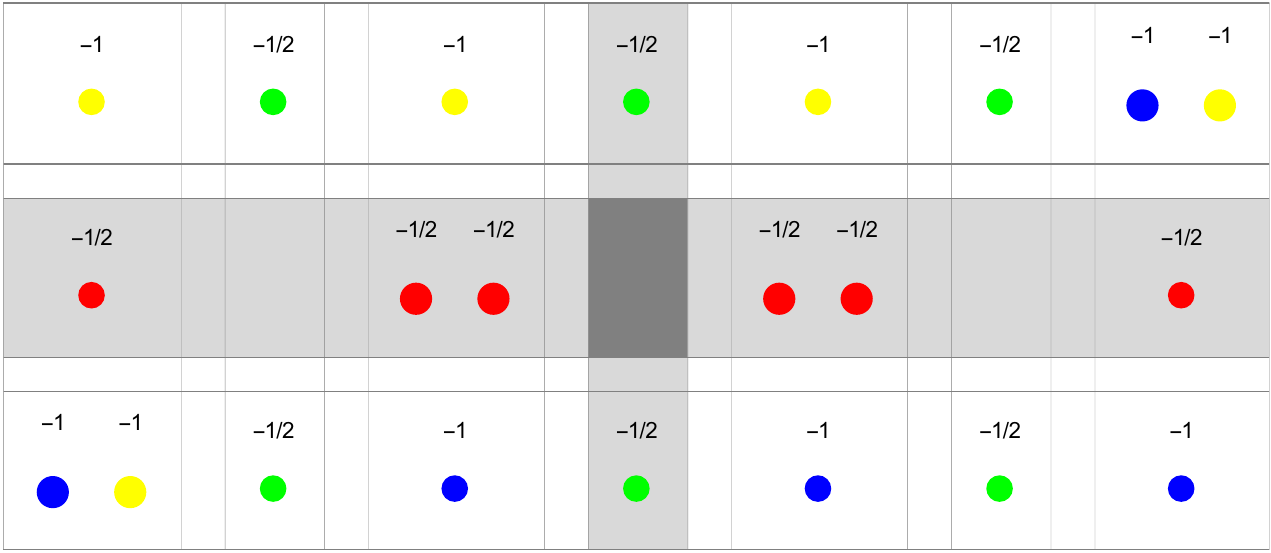}
\caption{The output of \texttt{DotGrid[v]} for example~\ref{exa:APT}}\label{fig:QDots}
\end{figure}

\subsection{Change log: versions $\mathbf{1.0\rightarrow 1.1}$}
\begin{itemize}
\item Changed $\delta$-grading in Alexander codes to match new conventions. 
\item Defined new functions \texttt{Twist}, \texttt{Closure}, \texttt{DConfig} and \texttt{DotGrid}. \item Moved code to a Mathematica package file \texttt{APT.m}.
\end{itemize}

%% file: sections/A_BSFH-Manual.tex
\chapter{Manual for \texttt{BSFH.m}}\label{app:manualBSFH}
\vspace{-12pt}
\begin{center}
	version 1.0, written in Mathematica 10.3.1.0 for Linux x86 (64-bit)
\end{center}
\vspace{36pt}

\noindent
The Mathematica package \cite{BSFH.m} provides a tool for calculating Zarev's bordered sutured Floer homology for any bordered sutured manifold~\cite{Zarev}. Both type~A and type~D structures can be computed with this program, as well as the various bimodule invariants. Tools for manipulating the invariants are also provided, namely for cancellation and for performing basic homotopies. This manual only covers the main functionality of the package \cite{BSFH.m} and should be read alongside the notebook~\cite{BSAx2.nb}, where we compute the bordered sutured type A structure from remark~\sref{rem:BSAx2} and explain some more advanced features.

\subsection*{Input preparation.}
Our program \cite{BSFH.m} implements an algorithm due to Zarev~\cite[theorems~7.14 and~7.15]{Zarev}, which allows us to compute bordered sutured invariants combinatorially from nice Heegaard diagrams. So, given a Heegaard diagram for a bordered sutured manifold, we first need to niceify it, using \cite[proposition~4.17]{Zarev}. How we do this can have a significant impact on the run time of the computation for large diagrams, so one should keep the number of generators as low as possible. Figure~\ref{fig:HD} shows the niceified Heegaard diagram that we use for the computation in~\cite{BSAx2.nb}. We now explain how to translate a picture like this into the input data for our program. \\
First, we label the intersection points of $\alpha$-curves and $\beta$-curves by integers ({\footnotesize 1}, {\footnotesize 2}, {\footnotesize 3}, {\footnotesize\dots}). Similarly, we label the regions in the Heegaard diagram that are not adjacent to any basepoints (\texttt{1}, \texttt{2}, \texttt{3}, \dots). Furthermore, we enumerate the $\alpha$-arcs and $\alpha$-curves and, separately, all $\beta$-curves. We also draw an arc diagram for each glueing surface. In~\cite{BSAx2.nb}, we only have one arc diagram, which is shown in figure~\ref{fig:AD}. The orientation of the arc diagram is as on the Heegaard surface and we draw it such that the $\alpha$-arcs are to the right of the line segments. We enumerate all the endpoints of the $\alpha$-arcs from bottom to top. We also label each component of the line segments between two endpoints by the (index of the) region adjacent to it and its corresponding algebra element.

\begin{figure}[t]
\centering
\psset{unit=0.4}
\begin{subfigure}[b]{0.55\textwidth}\centering
\begin{pspicture}(-9,-10)(11,10)

\psline[linecolor=red,linearc=0.8](-7,5)(10.5,5)(10.5,10)(0,10)(0,1)(11,1)(11,-1)(0,-1)(0,-10)(10.5,-10)(10.5,-5)(-7,-5)

\psframe[linecolor=blue,framearc=0.5](-3,2)(3,8)
\psframe[linecolor=blue,framearc=0.5](-3,-2)(3,-8)
\psframe[linecolor=blue,framearc=0.5](-3,-2)(3,-8)

\psframe[linecolor=blue,framearc=0.5](5,-9)(9,9)

\pscircle[linecolor=blue](-7,5){2}

\pscircle[fillstyle=solid,fillcolor=white,linecolor=darkgreen](0,5){1}
\pscircle[fillstyle=solid,fillcolor=white,linecolor=darkgreen](0,-5){1}
\pscircle[fillstyle=solid,fillcolor=white,linecolor=darkgreen](-7,5){1}
\pscircle[fillstyle=solid,fillcolor=white,linecolor=darkgreen](-7,-5){1}
\pscircle[fillstyle=solid,fillcolor=white,linecolor=darkgreen](7,7){1}
\pscircle[fillstyle=solid,fillcolor=white,linecolor=darkgreen](7,-7){1}

\rput(-7,5){$\rho$}
\rput(0,5){$\tau$}
\rput(0,-5){$\delta$}
\rput(-7,-5){$\varepsilon$}

{\scriptsize
\rput(-5,5){\psdot(0,0)\uput{0.3}[-45](0,0){1}}
\rput(-3,5){\psdot(0,0)\uput{0.3}[-135](0,0){2}}
\rput(0,2){\psdot(0,0)\uput{0.3}[-135](0,0){3}}
\rput(5,1){\psdot(0,0)\uput{0.3}[-135](0,0){4}}
\rput(9,1){\psdot(0,0)\uput{0.3}[45](0,0){5}}
\rput(9,-1){\psdot(0,0)\uput{0.3}[-45](0,0){6}}
\rput(5,-1){\psdot(0,0)\uput{0.3}[135](0,0){7}}
\rput(0,-2){\psdot(0,0)\uput{0.3}[135](0,0){8}}
\rput(-3,-5){\psdot(0,0)\uput{0.3}[135](0,0){9}}
\rput(0,8){\psdot(0,0)\uput{0.3}[-45](0,0){10}}
\rput(9,5){\psdot(0,0)\uput{0.3}[135](0,0){11}}
\rput(5,5){\psdot(0,0)\uput{0.3}[45](0,0){12}}
\rput(3,5){\psdot(0,0)\uput{0.3}[135](0,0){13}}
\rput(3,-5){\psdot(0,0)\uput{0.3}[-135](0,0){14}}
\rput(5,-5){\psdot(0,0)\uput{0.3}[-45](0,0){15}}
\rput(9,-5){\psdot(0,0)\uput{0.3}[-135](0,0){16}}
\rput(0,-8){\psdot(0,0)\uput{0.3}[45](0,0){17}}
}

\rput(0,5){
\rput(2;-45){\texttt{1}($\tau_1$)}
\rput(2;-135){\texttt{2}($\tau_2$)}
\rput(2;135){\texttt{3}($\tau_3$)}
\psline(0.6;45)(1.4;45)
}
\rput(0,-5){
\rput(2;-135){\texttt{4}$(\delta_1$)}
\rput(2;135){\texttt{5}$(\delta_2$)}
\rput(2;45){\texttt{6}$(\delta_3$)}
\psline(0.6;-45)(1.4;-45)
}
\rput(4,8){\texttt{7}}
\rput(4,3){\texttt{8}}
\rput(7,3){\texttt{9}}
\rput(7,0){\texttt{10}}
\rput(7,-3){\texttt{11}}
\rput(4,-3){\texttt{12}}
\rput(4,-8){\texttt{13}}
\rput(10,0){\texttt{14}}

\psline(-7.6,5)(-8.4,5)
\psline(-7.6,-5)(-8.4,-5)

\rput(-4,5.5){\red $\alpha_1$}
\rput(-4.25,6.5){\red $(d)$}
\rput(2,0){\red $\alpha_2 (a)$}
\rput(-4.5,-5.75){\red $\alpha_3 (b)$}
\rput(2,9.25){\red $\alpha_4$}
\rput(2,-9.25){\red $\alpha_5$}

\rput(-6,2.5){\blue $\beta_1$}
\rput(-3,2){\blue $\beta_2$}
\rput(-3,-2){\blue $\beta_3$}
\rput(9.75,3){\blue $\beta_4$}

\end{pspicture}
\caption{A nice Heegaard diagram with many labels}\label{fig:HD}
\end{subfigure}
\quad
\begin{subfigure}[b]{0.3\textwidth}\centering
\psset{unit=1.5}
\begin{pspicture}[showgrid=false](-2,-1)(4,11.5)

\psline[linecolor=darkgreen](0,10)(0,11)
\psline[linecolor=darkgreen](0,5.5)(0,9.5)
\psline[linecolor=darkgreen](0,1)(0,5)
\psline[linecolor=darkgreen](0,-0.5)(0,0.5)

\psline[linecolor=red,linearc=0.4](-0.2,0)(2,0)(2,2.5)(-0.2,2.5)
\psline[linecolor=red,linearc=0.4](-0.2,1.5)(3,1.5)(3,4.5)(-0.2,4.5)
\psline[linecolor=red,linearc=0.4](-0.2,3.5)(2,3.5)(2,7)(-0.2,7)
\psline[linecolor=red,linearc=0.4](-0.2,6)(3,6)(3,9)(-0.2,9)
\psline[linecolor=red,linearc=0.4](-0.2,8)(2,8)(2,10.5)(-0.2,10.5)

{\footnotesize
\rput[r](-0.5,0){$1$}
\rput[r](-0.5,1.5){$2$}
\rput[r](-0.5,2.5){$3$}
\rput[r](-0.5,3.5){$4$}
\rput[r](-0.5,4.5){$5$}
\rput[r](-0.5,6){$6$}
\rput[r](-0.5,7){$7$}
\rput[r](-0.5,8){$8$}
\rput[r](-0.5,9){$9$}
\rput[r](-0.5,10.5){$10$}
}

\rput[r](-1.2,8.5){$\delta_1$}
\rput[r](-1.2,7.5){$\delta_2$}
\rput[r](-1.2,6.5){$\delta_3$}
\rput[r](-1.2,4){$\tau_1$}
\rput[r](-1.2,3){$\tau_2$}
\rput[r](-1.2,2){$\tau_3$}

\rput[r](0.75,8.5){\texttt{4}}
\rput[r](0.75,7.5){\texttt{5}}
\rput[r](0.75,6.5){\texttt{6}}
\rput[r](0.75,4){\texttt{1}}
\rput[r](0.75,3){\texttt{2}}
\rput[r](0.75,2){\texttt{3}}

\rput[l](2.25,0.75){\red $\alpha_1 (d)$}
\rput[l](2.25,5.25){\red $\alpha_2 (a)$}
\rput[l](2.25,9.75){\red $\alpha_3 (b)$}
\rput[l](3.25,3){\red $\alpha_4$}
\rput[l](3.25,7.5){\red $\alpha_5$}

\end{pspicture}
\caption{An arc diagram for the Heegaard diagram on the left}\label{fig:AD}
\end{subfigure}
\caption{The starting point of any calculation is a nice Heegaard diagram and an arc diagram. Here, we have drawn them for the computation in \cite{BSAx2.nb}.}\label{fig:pictures}
\end{figure}
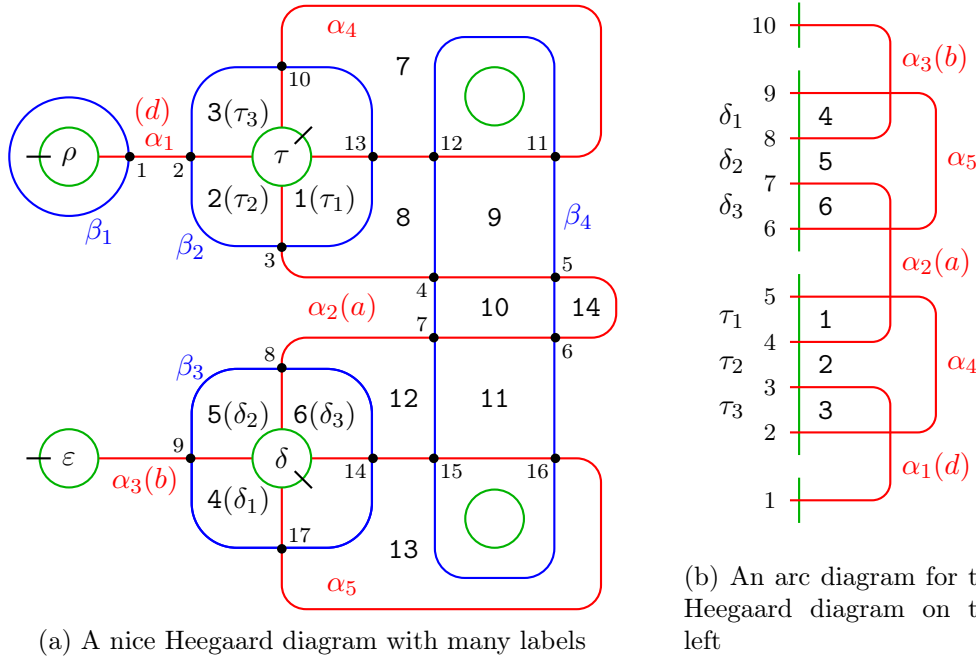

\subsection*{Basic input data. } 
We start by loading the package \cite{BSFH.m} as follows: 
$$\texttt{<\relax< BSFH.m}$$
Next, we enter the following data into the notebook. Note that for each calculation, a~separate notebook should be used.
\begin{description}
\item[\texttt{regionsInput}] is a table whose rows are indexed by the regions of our Heegaard diagram. The first entry of the $i^\text{th}$ row is equal to $i$, then follow the vertices of that region. These are ordered by the boundary orientation of the region, i.\,e.\ anticlockwise, starting with a vertex which is the start of a segment of an $\alpha$-curve or $\alpha$-arc in the boundary of the region. For bigons, the last two entries remain empty, i.\,e.\ occupied by the symbol~$\lightgray\square$.
\item[\texttt{alphaarcsInput}]  is a table whose rows are indexed by the $\alpha$-arcs. Again, the first entry should be the index of the arc; the second is a list of (the indices of) all intersection points on this arc.
\texttt{alphacurvesInput} and \texttt{betacurvesInput} are the corresponding tables for $\alpha$- and $\beta$-curves.
\item[\texttt{CancellationSortListInput}] is a table of (indices of) intersection points with two columns, so each row corresponds to a pair of intersection points. This table has an effect on the order of the generators and is important for the initial cancellation after computing the invariant. \\
Usually, we start with a Heegaard diagram which is not necessarily nice and then niceify it using finger moves, thereby creating lots of new intersection points and generators. In the initial cancellation step, the program attempts to cancel such generators, thereby reversing the effect of niceification. The order in which this is done is determined by the order of the generators, and this, in turn, is determined by \texttt{CancellationSortListInput}.\\
So the pairs of intersection points should be those that can be removed by a reversed finger move. The order of the intersection points for each pair is the same as for \texttt{regionsInput}, using the bigon which connects these two points and which is removed by the reversed finger move. This bigon accounts for the identity component in the differential which the program will attempt to cancel.
\item[\texttt{RelevantGeneratorQ[v]}] is a function used in the generation of generators. It takes a generator $\texttt{v}$ of the Heegaard diagram and decides if it should be used in the calculation of the invariant. For type A structures or bimodules, it can make sense to forget some generators in certain idempotents (see for example theorem~\sref{thm:glueingCAT}), but this functionality should be used with care. As a default, set equal to \texttt{True}. 
\item[\texttt{RelevantGeneratorGenerationQ[v]}] also plays an important role in the generation of generators. It takes a list of $\alpha$-curves and decides if generators occupying these curves should be used in the calculation of the invariant. This is a quick check used for speeding up the generator generation. Like \texttt{RelevantGeneratorQ}, it should be used with care. As a default, set equal to \texttt{True}.
\item[\texttt{S1BoundaryInput}] is used for the calculation of the bordered sutured algebra: It is a list of (indices of) regions for each line segment in the arc diagram for the first glueing surface S1, ordered from bottom to top. If S1 plays the role of a type A side, this ordering is reversed internally, compare with figure~\ref{fig:APPconventions}. This is opposite to Zarev's conventions. Similarly for \texttt{S2BoundaryInput}.
\item[\texttt{TypeS1}] should be set to either \texttt{TypeA} or \texttt{TypeD}, depending on which structure we would like to compute for the first glueing surface \texttt{S1}. Similar for \texttt{S2}; if there is no second glueing surface \texttt{S2}, either inputs are fine.
\item[\texttt{ArcDiagram1MatchingsInput}] is a list of the lists of (the indices of) endpoints of the $\alpha$-arcs in the arc diagram for \texttt{S1}, ordered by the indices of the $\alpha$-arcs. Each entry itself should be ordered. Same for~\texttt{S2}.
\item[\texttt{S1AlwaysOccupiedAlphaArcs}] is a list of indices of $\alpha$-arcs for~\texttt{S1} that are occupied by all generators in the set of relevant generators determined by \texttt{RelevantGeneratorQ[v]} above. Similarly, \texttt{S1NeverOccupiedAlphaArcs} is a list of indices of $\alpha$-arcs for~\texttt{S1} that are cannot be occupied by any relevant generator. As a default, set both equal to~\texttt{\{\}}. 
\texttt{S1PotentiallyOccupiedAlphaArcs} is the list of remaining indices of $\alpha$-arcs for~\texttt{S1}. Same for~\texttt{S2}.
\item[\texttt{NS1oaa}] is equal to the total number of occupied $\alpha$-arcs for~\texttt{S1}. Same for~\texttt{S2}. 
\item[\texttt{NR}] and \texttt{NI} are the total number of regions and intersection points, respectively.
\end{description}

\subsection*{Running a calculation.} 
We run the program in several separate steps and substeps by evaluating the variables called 
$$
\texttt{Step<$i$><}\text{name of step $i$}\texttt{>}\quad\text{and}\quad\texttt{Step<$i$>x<$j$><}\text{name of step $i$}\texttt{>}.
$$
They are modules defined in the package \cite{BSFH.m} which group together blocks of code that compute the various components of the invariant, like for example the generators, the domains, parts of the algebra, etc. 

\subsection*{Step 0: Double-checking basic input data.}
Before we start the actual calculation, we make sure that the input data that we have entered so far is consistent and does indeed correspond to our Heegaard diagram. At the end of the cell containing the basic input data above, we call \texttt{Step0Preparation}. This returns a list of results of automated sanity checks and some additional outputs, which one should check manually. Perhaps the most useful of these are the lists named \texttt{0 vertices}, \texttt{+1 vertices}, etc., which partition the set of intersection points into five subsets as follows. An intersection point is in \texttt{<$n$> vertices} if the sum of unlabelled regions (i.\,e. those adjacent to basepoints) in the four quadrants around it equals $n$, where each region contributes as follows: Walking around an intersection point in anticlockwise direction, an unlabelled region that we enter through a $\beta$-curve counts as~$+1$, one that we exit through a $\beta$-curve as~$-1$.

\subsection*{Backups.} At various stages throughout the program, backups of preliminary results are made. For this, we need to specify a path to the directory where the backups should be stored. They can then be recalled for example like this:
$$
\texttt{<\relax< (BackupFilePath <> "\_BeforeCancellation.mx");}
$$

\subsection*{Step 1: The algebra. }
In the first step of the calculation, we compute the glueing algebra. This is done in substeps \texttt{Step1x1Algebra} to \texttt{Step1x5Algebra}. After each of the first two substep, a few more inputs are needed. \texttt{Step1x1Algebra} returns a list of idempotents for each glueing surface. We can then choose names for these in \texttt{S1IdemNames} and \texttt{S2IdemNames}.\\
Furthermore, after evaluating \texttt{Step1x2Algebra} in \cite{BSAx2.nb}, there two inputs called \texttt{S1SubalgebraGensInput} and \texttt{S2SubalgebraGensInput}. Here, we can specify names of generators for a suitable subalgebra of each glueing algebra. The point of this is the following: algebra elements are internally represented by integers. If a generator lies in the specified subalgebra, it will later be represented in any graphical output by the corresponding product of names for the generators of the subalgebra. As a default, one might want to give a name to each algebra element with a single moving strand between consecutive starting and ending points. Also, if this functionality is not required, one can simply set both variables equal to~\texttt{\{\}}. \\
\texttt{S1SubalgebraGensInput} and \texttt{S2SubalgebraGensInput} are two matrices where each row corresponds to a generator of the respective subalgebra. The first entry is a string that represents this generator. The second entry is a \LaTeX-friendly version of this string. The third and fifth entries are the starting and ending idempotent indices; the fourth a list of start- and endpoints of the moving strands.\\
After this, we evaluate the remaining substeps \texttt{Step1x3Algebra} to \texttt{Step1x5Algebra}.
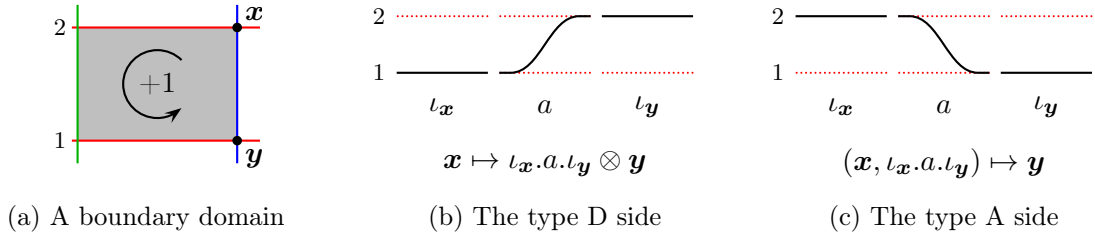
\begin{figure}[t]
\centering
\psset{unit=1.5}
\begin{subfigure}[b]{0.31\textwidth}\centering
\begin{pspicture}(-0.4,-0.4)(1.6,1.2)
\psframe*[linecolor=lightgray,linewidth=0pt](1.4,1)

\psline[linecolor=red](-0.05,0)(1.6,0)
\psline[linecolor=red](-0.05,1)(1.6,1)

\psline[linecolor=darkgreen](0,-0.2)(0,1.2)
\psline[linecolor=blue](1.4,-0.2)(1.4,1.2)

\psdot(1.4,0)\uput{0.1}[45](1.4,1){$\x$}
\psdot(1.4,1)\uput{0.1}[-45](1.4,0){$\y$}

{\footnotesize
\uput{0.1}[180](0,0){$1$}
\uput{0.1}[180](0,1){$2$}
}

\psarc{->}(0.7,0.5){0.3}{45}{-45}
\rput(0.7,0.5){+1}
\end{pspicture}
\caption{A boundary domain}\label{fig:APPconventionsdomain}
\end{subfigure}
\quad
\begin{subfigure}[b]{0.31\textwidth}\centering
\begin{pspicture}(-1.5,-1)(1.5,0.6)

\rput(0,-0.8){$\x\mapsto \iota_{\x}.a.\iota_{\y}\otimes\y$}

\rput(-0.9,0){
\psline[linestyle=dotted,dotsep=1pt,linecolor=red](-0.4,0.5)(0.4,0.5)
\psline(-0.4,0)(0.4,0)
\rput(0,-0.3){$\iota_{\x}$}
}

\psline[linestyle=dotted,dotsep=1pt,linecolor=red](-0.4,0.5)(0.4,0.5)
\psline[linestyle=dotted,dotsep=1pt,linecolor=red](-0.4,0)(0.4,0)
\pscustom{
\psline(-0.4,0)(-0.3,0)
\psecurve(-0.9,0.5)(-0.3,0)(0.3,0.5)(0.9,0)
\psline(0.3,0.5)(0.4,0.5)
}
\rput(0,-0.3){$a$}

\rput(0.9,0){
\psline(-0.4,0.5)(0.4,0.5)
\psline[linestyle=dotted,dotsep=1pt,linecolor=red](-0.4,0)(0.4,0)
\rput(0,-0.3){$\iota_{\y}$}
}

{\footnotesize
\uput{0.2}[180](-1.2,0){$1$}
\uput{0.2}[180](-1.2,0.5){$2$}
}

\end{pspicture}
\caption{The type D side}\label{fig:APPconventionsTypeD}
\end{subfigure}
\quad
\begin{subfigure}[b]{0.31\textwidth}\centering
\begin{pspicture}(-1.5,-1)(1.5,0.6)

\rput(0,-0.8){$(\x,\iota_{\x}.a.\iota_{\y})\mapsto\y$}

\rput(0.9,0){
\psline[linestyle=dotted,dotsep=1pt,linecolor=red](-0.4,0.5)(0.4,0.5)
\psline(-0.4,0)(0.4,0)
\rput(0,-0.3){$\iota_{\y}$}
}

\psline[linestyle=dotted,dotsep=1pt,linecolor=red](-0.4,0.5)(0.4,0.5)
\psline[linestyle=dotted,dotsep=1pt,linecolor=red](-0.4,0)(0.4,0)
\pscustom{
\psline(0.4,0)(0.3,0)
\psecurve(0.9,0.5)(0.3,0)(-0.3,0.5)(-0.9,0)
\psline(-0.3,0.5)(-0.4,0.5)
}
\rput(0,-0.3){$a$}

\rput(-0.9,0){
\psline(-0.4,0.5)(0.4,0.5)
\psline[linestyle=dotted,dotsep=1pt,linecolor=red](-0.4,0)(0.4,0)
\rput(0,-0.3){$\iota_{\x}$}
}

{\footnotesize
\uput{0.2}[180](-1.2,0){$1$}
\uput{0.2}[180](-1.2,0.5){$2$}
}

\end{pspicture}
\caption{The type A side}\label{fig:APPconventionsTypeA}
\end{subfigure}
\caption{Our conventions for counting domains and the algebras of moving strands for type D and type A sides. For type A sides, the program internally reverses the labelling of the endpoints of the $\alpha$-arcs. This is opposite to Zarev's conventions. }\label{fig:APPconventions}
\end{figure}

\subsection*{Steps 2--3: Generators and domains. }
In step~\texttt{Step2Generators} and substeps \texttt{Step3x1Domains} to \texttt{Step3x4Domains}, we compute the generators and the structure map of the invariant, respectively. Our conventions for counting domains are shown in figure~\ref{fig:APPconventions}. Some of the steps offer intermediate results which can be used as additional sanity checks. Basically, the invariant is fully computed after evaluating substep \texttt{Step3x4Domains}; the generators are stored in \texttt{generators} and the structure map in \texttt{StructureMapsSparse}. However, usually there are quite a lot of generators, so the result at this stage is not particularly useful. Therefore, as mentioned earlier, we offer some tools for manipulating the result. \enlargethispage{0.2cm}

\subsection*{Step 4: Cancellation.}
\texttt{Step4CancellationPreparation} provides the tools for cancellation. For the initial cancellation, one should always use
$$
\texttt{CancellationSparse[StructureMapsSparse, generators]}.
$$
The output of this function is a list 
$$
\texttt{\{<}\text{cancelled structure map}\texttt{>,<}\text{remaining generators}\texttt{>\}}.
$$
The number of remaining generators should be the same as for the initial non-niceified (admissible) Heegaard diagram -- provided we have set up \texttt{CancellationSortListInput} correctly. Note that \texttt{CancellationSparse} uses sparse arrays for the structure maps, but converts them into their normal form for the output. \\
One might want to perform some further cancellation. There are two tools for this:
\begin{description}
\item[\texttt{Cancellation[StructureMaps, gens, cancel]}] works very similar to the function \texttt{CancellationSparse}. \texttt{gens} is the set of generators underlying \texttt{StructureMaps}. \texttt{cancel} is an optional third argument
where one can specify a list of generator pairs 
$$
\texttt{\{<}\text{end of cancelling arrow}\texttt{>,<}\text{start of cancelling arrow}\texttt{>\}}.
$$
The iteration will stop at the first instance the conditions for cancellation are not satisfied, or if it has run through the complete list. If the third argument is missing, any identity morphism will be cancelled (if possible). In this case, there is a third entry in the output, which is a list of all cancelled generator pairs.
\item[\texttt{CarefulCancellation[StructureMaps, generators]}] is very similar to the function \texttt{Cancellation}, but it only performs those cancellations which do not result in a complex with non-trivial maps from a generator to itself (loops) or differentials between two generators in either directions (double arrows). 
\end{description}

\subsection*{Step 5: Post-Processing. } The last step \texttt{Step5PostProcessing} loads two main functionalities of the package, namely tools for displaying the result graphically and for performing basic homotopies. 

\subsection*{Basic graphical outputs.}
There are two main output formats for the whole invariant, which are provided in \texttt{Step5OutputPreparation}.
\begin{description}
\item[\texttt{PlotGraph[StructureMaps]}] is a very basic graphical representation. It shows a schematic picture of the matrix \texttt{StructureMaps}, where each entry is represented by a coloured box in a grid. If the colour is green, the entry is empty. 
If the entry is non-empty, the colour is a shade of grey, where white represents a (relatively) short entry, black a (relatively) long one.
\item[\texttt{ShowGraph[StructureMaps, GeneratorSet]}] displays the full invariant as a labelled graph. If one only wants to see  a portion of the graph, one can use the variant 
$$
\texttt{ShowSubGraph[StructureMaps, GeneratorSet, Subgraph]},
$$
where \texttt{Subgraph} is a list of (indices of) those generators forming the vertices of the  subgraph. An output of this is shown in figure~\ref{fig:TypeAAGlueingOneBlock}.
\end{description}
There are also functions for a graphical representation of algebra elements, the most useful one being \texttt{S1H} for the glueing algebra \texttt{S1} and \texttt{S2H} for \texttt{S2}. Given an (index of an) algebra element for \texttt{S1}, \texttt{S1H} returns the sum of its basic algebra elements in the (larger) moving strands algebra. Each such element is represented by a grid with a row for each moving strand displaying its starting point on the left and its endpoint on the right. For type A sides, note that the indices of the start- and endpoints are reversed, see figure~\ref{fig:APPconventions}.

\subsection*{Performing basic homotopies. }
There are several functions for ``adding homotopies'' in the sense of the clean-up lemma (lemma~\ref{lem:AbstractCleanUp}). The easiest to use is the following:
\begin{description}
\item[\texttt{AutoCleanPosns[StructureMaps, gens, pos, max]}] tries to simplify the matrix \texttt{StructureMaps} at positions \texttt{pos} by a sequence of at most \texttt{max} homotopies. The output is of the form 
$$
\texttt{\{<}\text{new structure map}\texttt{>,<}\text{record of added homotopies}\texttt{>\}}.
$$
Note that this function does not always give the same output, because it uses the pseudo-random Mathematica function \texttt{RandomSample}. 
\end{description}
We also mention two other, more basic functions which are quite useful if the previous automated function does not give quite the desired result, see for example the notebook~\cite{BSAAx4e.nb}.
\begin{description}
\item[\texttt{Homotopy[StructureMaps, gens, homotopy, start, end]}] calculates the structure map obtained by adding the homotopy to \texttt{StructureMaps} which is labelled by an algebra element \texttt{homotopy} and goes from the generator with index \texttt{start} to the generator with index \texttt{end} in the underlying set of generators specified by \texttt{gens}.
\item[\texttt{AllHomotopies[StructureMaps, pos]}] computes a list of all homotopies that can remove an arrow label specified by \texttt{pos} in \texttt{StructureMaps}. Here, a homotopy has the following form
$$
\texttt{\{<}\text{algebra element}\texttt{>,<}\text{start}\texttt{>,<}\text{end}\texttt{>\}}.
$$
\end{description}

%% file: thesis.bbl
\begin{thebibliography}{BSFH.m}
\footnotesize\raggedright\setlength{\itemsep}{0mm}
\bibitem[AAEKO]{Abouzaid} M.\,Abouzaid, D.\,Auroux, A.\,I.\,Efimov, L.\,Katzarkov, D.\, Orlov, \href{http://arxiv.org/abs/1103.4322v2}{\textit{Homological mirror symmetry for punctured spheres}}, J. Amer. Math. Soc. 26 (2013), 1051-1083 (arXiv:~1103.4322)
\bibitem[Ada94]{Adams} C.\,C.\,Adams, \textit{The Knot Book: An Elementary Introduction to the Mathematical Theory of Knots}, W.\,H.\,Freeman \& Company (1994) 
\bibitem[Ale28]{Alexander} J.\,W.\,Alexander, \href{http://www.maths.ed.ac.uk/~aar/papers/alex.pdf}{\textit{Topological invariants of knots and links}}, Trans. Amer. Math. Soc. 30 (1928), 275--306
\bibitem[Alt13]{Altman} I.\,Altman, \href{https://arxiv.org/abs/1304.2606v1}{\textit{Introduction to sutured Floer homology}}, arXiv:~1304.2606v1
\bibitem[Arc10]{Archibald} J.\,Archibald, \textit{The multivariable Alexander polynomial on tangles}, PhD thesis (2010), University of Toronto, available at \url{http://www.math.toronto.edu/jfa/jana_thesis.pdf}
\bibitem[Aur10]{Auroux}  D.\,Auroux, 
\href{http://arxiv.org/abs/1001.4323v3}{\textit{Fukaya categories of symmetric products and bordered Heegaard-Floer homology}}, J. Gökova Geom. Topol. 4 (2010), 1--54 (arXiv:~1001.4323v3)

\bibitem[Bar02]{BarNatanIntroKh} D.\,Bar-Natan, \href{http://arxiv.org/abs/math/0201043v3}{\textit{On Khovanov's categorification of the Jones polynomial}}, Algebr. Geom. Topol. 2 (2002), 337--370 (arXiv:~0201043v3)
\bibitem[Bar04]{BarNatanKhT} \rule{1cm}{0.7pt}, \href{http://arxiv.org/abs/math/0410495v2}{\textit{Khovanov's homology for tangles and cobordisms}}, Geom. Topol. 9 (2005), 1443--1499 (arXiv:~0410495v2)
\bibitem[BL11]{BaldwinLevine} J.\,A.\,Baldwin, A.\,S.\,Levine, \href{http://arxiv.org/abs/1105.5199v2}{\textit{A combinatorial spanning tree model for knot Floer homology}}, Adv. Math. 231 (2012), 1886--1939 (arXiv:~1105.5199v2)
\bibitem[Big12]{Bigelow} S.\,Bigelow, \href{http://arxiv.org/abs/1203.5457v1}{\textit{A diagrammatic Alexander invariant of tangles}}, arXiv:~1203.5457v1
\bibitem[BCF12]{Florens} S.\,Bigelow, A.\,Cattabriga, V.\,Florens, \href{http://arxiv.org/abs/1203.4590v2}{\textit{Alexander representation of tangles}}, arXiv:~1203.4590v2
\bibitem[Boc07]{Bocklandt} R.\,Bocklandt, \href{http://arxiv.org/abs/1111.3392v2}{\textit{Noncommutative mirror symmetry for punctured surfaces}}, arXiv:~1111.3392v2 \textit{(brought to my attention by M.\,Abouzaid)}
\bibitem[DV16]{Damiani} C.\,Damiani, V.\,Florens, \href{http://arxiv.org/abs/1602.06191v1}{\textit{Alexander invariants of ribbon tangles and planar algebras}}, arXiv:~1602.06191v1

\bibitem[EPV15]{DecatCTFH} A.\,P.\,Ellis, I\,Petkova, V.\,Vértesi: \href{http://arxiv.org/abs/1510.03483v1}{\textit{Quantum $\mathfrak{gl}(1|1)$ and tangle Floer homology}}, arXiv:~1510.03483v1
\bibitem[FJR09]{DecatSFH} S.\,Friedl, A.\,Juh\'{a}sz, J.\,A.\,Rasmussen, \href{http://arxiv.org/abs/0903.5287v4}{\textit{The decategorification of sutured Floer homology}}, J.~Topol. 4 (2011), no.~2, 431--478 (arXiv:~0903.5287v4)

\bibitem[GL86]{ouka} P.\,M.\,Gilmer, R.\,A.\,Litherland, \textit{The duality conjecture in formal knot theory}, Osaka J.~Math.~23 (1986), 229--247

\bibitem[HKK14]{Kontsevich} F.\,Haiden, L.\,Katzarkov, M.\,Kontsevich, \href{http://arxiv.org/abs/1409.8611v2}{\textit{Flat surfaces and stability structures}}, arXiv:~1409.8611v2
\bibitem[Han13]{Hanselman} J.\,Hanselman,  \href{https://arxiv.org/abs/1310.6696}{\textit{Bordered Heegaard Floer homology and graph manifolds}}, arXiv:~1310.6696
\bibitem[HRW16]{HRW} J.\,Hanselman, J.\,A.\,Rasmussen, L.\,Watson, \href{http://arxiv.org/abs/1604.03466v1}{\textit{Bordered Floer homology for manifolds with torus boundary via immersed curves}}, arXiv:~1604.03466v1
\bibitem[HW15]{HanselmanWatson} J.\,Hanselman, L.\,Watson, \href{http://arxiv.org/abs/1508.05445v1}{\textit{A calculus for bordered Floer homology}}, arxiv:~1508.05445v1
\bibitem[Har83]{hartley} R.\,Hartley, \textit{The Conway potential function for links}, Comment. Math. Helv. 58 (1983), 365--378
\bibitem[HHK13]{HHK13} M.\,Hedden, C.\,Herald, P.\,Kirk, \href{http://arxiv.org/abs/1301.0164v1}{\textit{The pillowcase and perturbations of traceless representations of knot groups}}, Geom. Topol. 18 (2014), 211--287 (arXiv:~1301.0164v1)
\bibitem[HHK15]{HHK15} \rule{1cm}{0.7pt}, \href{http://arxiv.org/abs/1501.00028v1}{\textit{The pillowcase and traceless representations of knot groups II: a Lagrangian-Floer theory in the pillowcase}}, arXiv:~1501.00028v1
\bibitem[HR11]{absolutegrading1} Y.\,Huang, V.\,G.\,B.\,Ramos, \href{http://arxiv.org/abs/1112.0290v2}{\textit{An absolute grading on Heegaard Floer homology by homotopy classes of oriented 2-plane fields}}, arXiv:~1112.0290v2
\bibitem[HR12]{absolutegrading2} \rule{1cm}{0.7pt}, \href{http://arxiv.org/abs/1211.7367v2}{\textit{A topological grading on bordered Heegaard Floer homology}}, Quantum Topol. 6 (2015), 403--449, (arXiv:~1211.7367v2)

\bibitem[Jia14]{jiang14} B.\ Jiang, \href{http://arxiv.org/abs/1407.3081v2}{\textit{On Conway's potential function for colored links}}, Acta Math. Sinica (English Series) 32 (2016), no.\,1, 25--39 (arXiv:~1407.3081v2)
\bibitem[Juh06a]{Juhasz} A.\,Juhász, \href{http://arxiv.org/abs/math/0601443v3}{\textit{Holomorphic discs and sutured manifolds}}, Algebr. Geom. Topol. 6 (2006), 1429--1457 (arXiv:~0601443v3)
\bibitem[Juh06b]{SurfaceDecomposition} \rule{1cm}{0.7pt}, \href{http://arxiv.org/abs/math/0609779}{\textit{Floer homology and surface decompositions}}, Geom. Topol. 12 (2008), 299--350 (arXiv:~0601443v3)
\bibitem[Juh08]{polytope} \rule{1cm}{0.7pt}, \href{http://arxiv.org/abs/0802.3415v3}{\textit{ The sutured Floer homology polytope}}, Geom. Topol. 14 (2010), 1303--1354 (arXiv:~0802.3415v3)

\bibitem[Kau83]{Kauffman} L.\ Kauffman, \textit{Formal Knot Theory}, Princeton University Press (1983)
\bibitem[Ken12]{Kennedy} K.\,G.\,Kennedy, \href{http://arxiv.org/abs/1205.5781v2}{\textit{A diagrammatic multivariate Alexander invariant of tangles}}, arXiv:~1205.5781v2
\bibitem[Kho99]{Khovanov} M.\,Khovanov, \href{https://arxiv.org/abs/math/9908171v2}{\textit{A categorification of the Jones polynomial}}, Duke Math. J. 101 (2000), no.\,3, 359--426 (arXiv:~9908171v2)
\bibitem[KM10]{KM} P.\,Kronheimer, T.\,Mrowka, \href{http://www.math.harvard.edu/~kronheim/alexander}{\textit{Instanton Floer homology and the Alexander polynomial}}, Algebr. Geom. Topol.~10 (2010), no.~3, 1715--1738
\bibitem[KR04]{KhRoz} M.\,Khovanov, L.\,Rozansky, \href{http://arxiv.org/abs/math/0401268v2}{\textit{Matrix factorizations and link homology}}, arXiv:~0401268v2

\bibitem[L16]{LambertCole1} P.\,Lambert-Cole, \href{https://arxiv.org/abs/1608.02011}{\textit{Twisting, mutation and knot Floer homology}}, arXiv:~1608.02011
\bibitem[L17]{LambertCole2} \rule{1cm}{0.7pt}, \href{https://arxiv.org/abs/1701.00880}{\textit{On Conway mutation and link homology}}, arXiv:~1701.00880
\bibitem[Lic97]{Lickorish} W.\,B.\,R.\,Lickorish, \textit{An Introduction to Knot Theory}, Springer (1997)
\bibitem[LM87]{LickorishMillett} W.\,B.\,R.\,Lickorish, K.\,C.\,Millett, \href{http://www.math.ucsb.edu/~millett/Papers/1987Millett9LickorishTopology.pdf}{\textit{A polynomial invariant of oriented links}}, Topol. 26 (1987), 107--141
\bibitem[Lip05]{LipshitzCyl} R.\,Lipshitz, \href{http://arxiv.org/abs/math/0502404v2}{\textit{A cylindrical reformulation of Heegaard Floer homology}}, Geom. Topol. 10 (2006), 955--1096 (arXiv:~0502404v2)
\bibitem[LOT08]{LOT} R.\,Lipshitz, P.\,Ozsv\'{a}th, D.\,Thurston, \href{http://arxiv.org/abs/0810.0687v5}{\textit{Bordered Heegaard Floer homology: Invariance and pairing}}, arXiv:~0810.0687v5
\bibitem[LOT10]{LOTMor} \rule{1cm}{0.7pt}, \href{http://arxiv.org/abs/1005.1248v2}{\textit{Heegaard Floer homology as morphism spaces}}, Quantum Topol. 2 (2011), no.~4, 381--449, (arXiv:~1005.1248v2)

\bibitem[Man06]{Manolescu} C.\,Manolescu, \href{https://arxiv.org/abs/math/0609531v3}{\textit{An unoriented skein exact triangle for knot Floer homology}}, Math. Res. Lett.~14 (2007), 839--852 (arXiv:~0609.531v3)

\bibitem[OS01]{OSHF3mfds} P.\,Ozsv\'{a}th, Z.\,Szab\'{o}, \href{https://arxiv.org/abs/math/0101206v4}{\textit{Holomorphic discs and topological invariants of closed 3-manifolds}}, Ann. Math. 159 (2004), 1027--1158 (arXiv:~0101206v4)
\bibitem[OS02]{OSHFKalt} \rule{1cm}{0.7pt}, \href{https://arxiv.org/abs/math/0209149v3}{\textit{Heegaard Floer homology and alternating knots}}, Geom. Topol.~7 (2003), 225--254 (arXiv:~0209149v3)
\bibitem[OS03a]{OSHFK} \rule{1cm}{0.7pt}, \href{http://arxiv.org/abs/math/0209056v4}{\textit{Holomorphic disks and knot invariants}}, Adv. Math. 186 (2004), no.\,1, 58--116 (arXiv:~0209056v4)
\bibitem[OS03b]{OSmutation} \rule{1cm}{0.7pt}, \href{http://arxiv.org/abs/math/0303225v2}{\textit{Knot Floer homology, genus bounds, and mutation}}, arXiv:~0303225v2
\bibitem[OS05]{OSHFL} \rule{1cm}{0.7pt}, \href{http://arxiv.org/abs/math/0512286v2}{\textit{Holomorphic disks and link invariants}}, Algebr. Geom. Topol. 8 (2008), 615--692 (arXiv:~0512286v2)
\bibitem[OS06a]{OSHFLThurston}\rule{1cm}{0.7pt}, \href{http://arxiv.org/abs/math/0601618v3}{\textit{Link Floer homology and the Thurston norm}}, arXiv:~0601618v3
\bibitem[OS06b]{OSHDsandFloer}\rule{1cm}{0.7pt}, \href{http://arxiv.org/abs/math/0602232v1}{\textit{Heegaard diagrams and Floer homology}}, Proc. ICM 2 (2006), no.\,51, 1083--1099 (arXiv:~0602232v1)
\bibitem[OS07]{OSrescube}\rule{1cm}{0.7pt}, \href{http://arxiv.org/abs/0705.3852v1}{\textit{A cube of resolutions for knot Floer homology}},  J.~Topol. 2 (2009), no.\,4, 865--910 (arXiv:~0705.3852v1)
\bibitem[OS16]{OSKauffmanStates}\rule{1cm}{0.7pt}, \href{http://arxiv.org/abs/1603.06559v1}{\textit{Kauffman states, bordered algebras, and a bigraded knot invariant}}, arXiv:~1603.06559v1

\bibitem[OSS07]{OSSHFS} P.\,Ozsv\'{a}th, A.\,I.\,Stipsicz, Z.\,Szab\'{o}, \href{http://arxiv.org/abs/0705.2661v3}{\textit{Floer homology and singular knots}}, J.~Topol.~2 (2009), 380--404 (arXiv:~0705.2661v3)

\bibitem[Pol10]{Polyak} M.\,Polyak, \href{http://arxiv.org/abs/1011.6200v1}{\textit{Alexander-Conway invariants of tangles}}, arXiv:~1011.6200v1
\bibitem[PV14]{cHFT} I.\,Petkova, V.\,V\'{e}rtesi, \href{http://arxiv.org/abs/1410.2161v2}{\textit{Combinatorial tangle Floer homology}}, arXiv:~1410.2161v2

\bibitem[Ras03]{Jake} J.\,A.\,Rasmussen, \href{http://arxiv.org/abs/math/0306378v1}{\textit{Floer homology and knot complements}}, PhD thesis (2003), Harvard (arXiv:~0306378v1)
\bibitem[Ras05]{JakeComparisonKhHFK} \rule{1cm}{0.7pt}, \href{http://arxiv.org/abs/math/0504045v1}{\textit{Knot polynomials and knot homologies}}, arXiv:~0504045v1
\bibitem[Rie14]{cathtpy} E.\,Riehl, \href{http://www.math.jhu.edu/~eriehl/cathtpy}{\textit{Categorical homotopy theory}}, Cambridge University Press (2014)

\bibitem[Srk06]{Sarkar06} S.\,Sarkar, \href{http://arxiv.org/abs/math/0609673v4}{\textit{Maslov index formulas for Whitney $n$-gons}}, J.~Sympl. Geom. 9 (2011), no.\,2, 251--270 (arXiv:~0609673v4)
\bibitem[Srt13]{Sartori14} A.\,Sartori, \href{http://arxiv.org/abs/1308.2047v2}{\textit{The Alexander polynomial as quantum invariant of links}}, Ark. Mat. 53 (2015), 177--202 (arXiv:~1308.2047v2)
\bibitem[Sei08]{Seidel} P.\,Seidel, \textit{Fukaya Categories and Picard-Lefschetz Theory}, Zürich Lectures in Advanced Mathematics, European Math. Soc. (2008)

\bibitem[Weh09]{WehrliKhMutInv} S.\,M.\,Wehrli, \href{http://arxiv.org/abs/0904.3401v1}{\textit{Mutation invariance of Khovanov homology over $\mathbb{F}_2$}}, arXiv:~0904.3401v1

\bibitem[Zar09]{Zarev} R.\,Zarev, \href{http://arxiv.org/abs/0908.1106v2}{\textit{Bordered Floer homology for sutured manifolds}}, arXiv:~0908.1106v2
\bibitem[Zem16]{Zemke} I.\,Zemke, \href{https://arxiv.org/abs/1610.05207v1}{\textit{ Link cobordisms and functoriality in link Floer homology}}, arXiv:~1610.05207v1
\bibitem[Zib16]{essay} C.\,B.\,Zibrowius, \href{https://arxiv.org/abs/1601.04915v1}{\textit{On a polynomial Alexander invariant for tangles and its categorification}}, arXiv:~1601.04915v1

\bibitem[nLab2]{nLabCurved} entry for curved dg-algebra in nLab: \url{https://ncatlab.org/nlab/show/curved+dg-algebra}
\bibitem[nLab1]{nLabEnrichedCat} entry for enriched categories in nLab: \url{https://ncatlab.org/nlab/show/enriched+category}

\vspace*{\fill}
\noindent\hfil\rule{0.5\textwidth}{.4pt}\hfil\\
\medskip
This thesis is accompanied by the following ancillary files:\bigskip
\bibitem[APT.m]{APT.m} Mathematica package \texttt{APT.m} for computing~$\nabla_T^s$
\bibitem[BSFH.m]{BSFH.m} Mathematica package \texttt{BSFH.m} for computing bordered sutured Floer homology
\bibitem[APT.nb]{APT.nb} Mathematica notebook \texttt{APT.nb} 
\bibitem[nb1]{notebook1} Mathematica notebook \texttt{2m3ptBSD.nb}
\bibitem[nb2]{notebook2} Mathematica notebook \texttt{2m3ptmutBSD.nb}
\bibitem[nb3]{BSAx2.nb} Mathematica notebook \texttt{ClosingBSAx2.nb}
\bibitem[nb4]{BSAAx4e.nb} Mathematica notebook \texttt{GlueingBSAAx4e.nb}\bigskip\\

\noindent
Furthermore, scans of hand-drawn Heegaard diagrams used for the computations in the Mathematica notebooks \cite{notebook1}, \cite{notebook2} and \cite{BSAAx4e.nb} can be found under the address
\url{http://www.dpmms.cam.ac.uk/~cbz20/documents/research/HDs.pdf}.
\bigskip\\

This is a printer-friendly version of my thesis. The original can be found at 
\url{http://www.dpmms.cam.ac.uk/~cbz20/documents/research/thesis.pdf}.
\\
\noindent\hfil\rule{0.5\textwidth}{.4pt}\hfil

\vspace*{\fill}

\end{thebibliography}
